\documentclass{book}
\usepackage[lang = american]{ems-book} 

\usepackage{tikz-cd}
\usepackage{color,mathrsfs}
\usepackage[makeroom]{cancel}
\usepackage{mathtools}
\mathtoolsset{showonlyrefs}

\newcommand{\beq}{\begin{equation}}
\newcommand{\eeq}{\end{equation}}
\newcommand{\bea}{\begin{eqnarray}}
\newcommand{\para}{\mathbin{\!/\mkern-5mu/\!}}
\newcommand{\eea}{\end{eqnarray}}
\newcommand{\beas}{\begin{eqnarray*}}
\newcommand{\eeas}{\end{eqnarray*}}

\newcommand{\jam}{P_{\tau}^{r,s,\beta}}
\newtheorem{theorem}{Theorem}[section]
\newtheorem{teor}[theorem]{Theorem}

\newtheorem{definition}[theorem]{Definition}
\newtheorem{proposition}[theorem]{Proposition}
\newtheorem{prop}[theorem]{Proposition}
\newtheorem{corollary}[theorem]{Corollary}
\newtheorem{lemma}[theorem]{Lemma}
\newtheorem{remark}[theorem]{Remark}
\newtheorem{example}[theorem]{Example} 
\newtheorem{examples}[theorem]{Examples}
 
\newtheorem{foo}[theorem]{Remarks}

\newenvironment{Example}{\begin{example}\rm}{\end{example}}

\DeclareMathOperator\arctanh{arctanh}
\DeclareMathOperator\arcosh{arcosh}

\newcommand{\CH}{\mathbb{C}H^n}
\newcommand{\C}{\mathbb{C}}
\newcommand{\A}{\mathfrak a}
\newcommand{\HP}{\mathbb{H}P}
\newcommand{\U}{\mathbf{U}}

\newcommand{\G}{\mathbf G}

\newcommand{\ee}{\ell}
\newcommand{\Hh}{\mathcal{H}}

\newcommand{\bM}{\mathbb M}
\newcommand{\bB}{\mathbf B}
\newcommand{\tr}{\mathrm{tr}}

\newcommand{\rL}{\overline{\Delta}_{\mathcal{H}} }
\newcommand{\Dh}{{\Delta}_{\mathcal{H}}}
\newcommand{\Dv}{{\Delta}_{\mathcal{V}}}
\newcommand{\Ho}{\mathcal H}
\newcommand{\V}{\mathcal V}

\newcommand{\M}{\mathbb M}
\newcommand{\B}{\mathbb B}
\newcommand{\bH}{\mathbb H}
\newcommand{\HH}{\mathbb H H}
\newcommand{\rp}{\overline p}
\newcommand{\bS}{\mathbb S}

\newcommand{\R}{\mathbb R}

\newcommand{\ch}{\mathcal H}

\newcommand{\Z}{\mathbb Z}

\newcommand{\bw}{\overline{w}}

\newcommand{\ep}{\varepsilon}
\newcommand{\bpartial}{\overline{\partial}}
\newcommand{\X}{{V}}
\newcommand{\cf}{{\theta}}
\begin{document}
\frontmatter


\title{Stochastic areas, Horizontal Brownian Motions, and  Hypoelliptic Heat Kernels}
\author{Fabrice Baudoin, Nizar Demni, Jing Wang}

\maketitle

\chapter*{Preface}
The monograph is devoted to the study of  stochastic area functionals of Brownian motions and of the associated heat kernels on Lie groups and Riemannian manifolds. It is essentially self-contained and as such can serve as a textbook on the theory of Brownian motions and horizontal Brownian motions on manifolds. Emphasis is put on concrete examples which allows us to  illustrate the rich and deep interactions between stochastic calculus, Riemannian and sub-Riemannian geometry, the theory of complex and quaternionic symmetric spaces and random matrices.

\tableofcontents


\mainmatter

\chapter{Introduction}

\section{The L\'evy's area formula and its generalizations}

By its simplicity, the number of its far reaching applications, and its connections to many areas of mathematics,  the Paul L\'evy's stochastic area formula is undoubtedly among the most important and beautiful formulas in stochastic calculus.

Let $Z(t)=B_1(t)+iB_2(t)$, $t\ge0 $, be a  Brownian motion in the complex plane such that $Z(0)=0$. The algebraic area swept out by the path of $Z$ up to time $t$ is given by half of the value
\begin{equation*}
S(t)=\int_{Z[0,t]} (xdy-ydx)=\int_0^t \left(B_1(s) dB_2(s) -B_2(s)dB_1(s)\right),
\end{equation*}
where the stochastic integral is an It\^o integral, or equivalently a Stratonovich integral since the quadratic covariation between $B_1$ and $B_2$ is 0. The L\'evy's area formula
\begin{equation}\label{LA}
\mathbb{E}\left( e^{i\lambda S(t)} | Z(t)=z\right)=\frac{\lambda t}{\sinh \lambda t} e^{-\frac{|z|^2}{2t}(\lambda t \coth \lambda t -1) }
\end{equation}
was originally proved in \cite{Levy} by using a series expansion of $Z$. The formula nowadays admits many different proofs. A particularly elegant probabilistic approach is due to Marc Yor \cite{YorArea} (see also \cite{MR0461687}). The first observation is that, due to the invariance by rotations of $Z$, one has for every $\lambda \in \R$,
\[
\mathbb{E}\left( e^{i\lambda S(t)} | Z(t)=z\right)=\mathbb{E}\left( \left. e^{-\frac{\lambda^2}{2} \int_0^t |Z(s)|^2 ds} \right| |Z(t)|=|z|  \right).
\] 
One considers then the new probability measure
\[
\mathbb{P}_{/ \mathcal{F}_t}^\lambda = \exp \left( \frac{\lambda}{2}(|Z(t)|^2 -2t) -\frac{\lambda^2}{2} \int_0^t |Z(s)|^2 ds \right)\mathbb{P}_{/ \mathcal{F}_t}
\]
under which, thanks to Girsanov theorem, $(Z(t))_{t \ge 0}$ is a Gaussian process (an Ornstein-Uhlenbeck process). Formula \eqref{LA} then easily follows from standard computations on Gaussian measures.

\

Somewhat surprisingly formula \eqref{LA} and the stochastic area process $(S(t))_{t \ge 0}$ appear in many different contexts, for instance: 

\begin{itemize}
\item \hspace{.1in} S. Watanabe points out in \cite{MR1439539} the connection between Formula \eqref{LA} and the differential of the exponential map in Lie groups;
\item \hspace{.1in} Formula \eqref{LA} also appears in the works by J.M. Bismut \cite{MR744920,MR756173} where probability methods are used to prove index theorems. Variations of those methods allow to construct explicit parametrices for the heat equation on vector bundles, see \cite{MR2376573};
\item \hspace{.1in} The Mellin transform of $S(t)$ is closely related to analytic number thoery and in particular to the Riemann zeta function, see the survey paper \cite{BianeYor} by P. Biane and M. Yor. We also point out some connections with algebraic geometry as explained in the paper  \cite{MR1869989} by K. Hara and N. Ikeda.
\item \hspace{.1in} The stochastic area process is a central character in the Terry Lyons' rough paths theory construction, see the monograph  \cite{MR2604669}) by N. Victoir and P. Friz. 
\item \hspace{.1in} Formula \eqref{LA} also appears as an important tool in Malliavin calculus, see in particular the paper \cite{MR877589} by S. Watanabe.
\item \hspace{.1in} The stochastic area process is also intimately connected to sub-Riemannian geometry. More precisely, in his paper \cite{Gaveau},  B. Gaveau actually observed that the $3$-dimensional process $(B_1(t),B_2(t), S(t))_{t\ge0}$ is a horizontal Brownian motion on  Heisenberg group. As a consequence, formula \eqref{LA} yields an expression for the heat kernel of the  sub-Laplacian on the Heisenberg group; see also the paper \cite{MR418257} by A. Hulanicki.
\item \hspace{.1in} Formula \eqref{LA} is also relevant in mathematical physics as pointed out in the papers by B. Duplantier \cite{MR1007231} and M. Yor \cite{MR1007232} (see also the references therein).

\end{itemize}

The stochastic area process or the L\'evy area formula can be generalized in many different directions. For instance there exist analogues  for Gaussian processes as in the paper  \cite{MR1266245} by N. Ikeda, S. Kusuoka, and S. Manabe (see also L. Coutin and N. Victoir \cite{MR2528082}) or for the free Brownian motion as in the paper \cite{MR1807256} by M. Capitaine and C. Donati-Martin or the paper \cite{MR2034293} by N. Victoir.

In the present monograph we present generalizations of both the area process and the L\'evy area formula  for Brownian motions on manifolds. As we will show,  natural generalizations of the stochastic area for a Brownian motion $(X(t))_{t \ge 0}$ on a manifold $\M$ are functionals that write
\begin{align}\label{area bundle}
S(t)=\int_{X[0,t]} \alpha
\end{align}
where $\alpha$ is a one-form with some geometric significance taking values in a Lie algebra (or one of its quotients). To ensure the existence of an explicit formula like \eqref{LA} we will  need that $\M$ is a Riemannian homogeneous space and that $\alpha$ is a  form  on $\M$ coming  from  the connection form of a homogeneous bundle over $\M$;  for a somewhat general setting see  Example \ref{BB fibration} and Section \ref{horizontal BM bundle}. For instance, in the simplest case where $\alpha$ is $\mathbb R$-valued, the \textit{area one-form} will arise from a K\"ahler structure on $\M$.  By definition, a K\"ahler form on a complex manifold $\M$ is a closed 2-form $\omega$ that induces the metric on $\M$ in the sense that $\omega (X,JY)$ is the Riemannian  metric on $\M$ where $J$ is the almost complex structure. It is a classical result in complex analysis that such 2-form can (at least locally) be written as $\omega =i\partial \bar{\partial} \Phi$ where $\partial$ and $\bar{\partial}$ are the Dolbeault operators and $\Phi$ is a smooth function. The (at least locally defined) real one-form 
\[
\alpha=\frac{1}{2i}(\partial -\bar{\partial}) \Phi
\]
is then a natural \textit{area one-form} on $\M$ which satisfies $d \alpha=(\partial +\bar{\partial}) \alpha=\omega$. When $\M$ is compact the homogeneous bundle over $\M$ we mentioned before is  a  $\mathbb S^1$-bundle referred to as the Boothby-Wang fibration (see the discussion in Section \ref{Sasakian manifold}). It is worth noting that the pull-back to that bundle of the form $\alpha$ then yields a \textit{winding one-form}: Integrating this form against a path describes the $\mathbb S^1$ fiber component of that path. 

We will also consider more general one-forms $\alpha$  like $\mathfrak{su}(2)$-valued one-forms and associated bundles, however we stress that our goal in the monograph is  not to develop a general and abstract theory of stochastic area type functionals associated with homogeneous bundles. We will rather  focus on specific relevant examples  for which very concrete calculations can be done. Covering in great details specific examples  will give us the opportunity to explore several topics of independent interest related to the study of stochastic area functionals. We will in particular focus our attention on connections with the theory of Riemannian submersions and associated horizontal Brownian motions, the theory of complex and quaternionic projective and hyperbolic spaces, the theory of hypoelliptic heat kernels and the  theory of random matrices.

\section{Organization of the monograph}

\textbf{Chapter 2.} In this chapter we  present the stochastic area functionals for Euclidean Brownian motions and the associated  L\'evy area formulas. We show how the geometry of such functionals is intimately related to the theory of nilpotent Lie groups. We also study some Brownian functionals which belong to the same family as the stochastic areas: the Brownian winding functionals.

\

\textbf{Chapter 3.} The chapter is an introduction to the theory of Brownian motions on Lie groups and Riemannian manifolds as is needed in the monograph.

\

\textbf{Chapter 4.} It presents the general theory of Riemannian submersions and associated horizontal Brownian motions. A special emphasis is put on examples and on the situation where the submersion is the projection map of a principal bundle; It is then shown in that case that stochastic areas functionals appear as the fiber motion of the horizontal Brownian motion of that bundle.

\

\textbf{Chapter 5.} It is devoted to the study of stochastic areas and of their distributions on complex projective spaces and complex hyperbolic spaces.

\

\textbf{Chapter 6.} It is devoted to the study of stochastic areas and of their distributions on the quaternionic projective spaces.

\

\textbf{Chapter 7.} It is devoted to the study of stochastic areas and of their distributions on the quaternionic hyperbolic spaces.

\

\textbf{Chapter 8.} It is devoted to the study of stochastic areas and of their distributions on the octonionic projective and hyperbolic spaces.

\

\textbf{Chapter 9.} In this chapter, using the techniques from random matrix theory,  we study Brownian motions and associated eigenvalue processes on complex Grasmannian spaces. We then define stochastic area functionals in that setting by means of the Stiefel fibration. The distribution of those area functionals is computed and limit theorems are proved.

\

\textbf{Chapter 10.} In this chapter we study Brownian motions and associated eigenvalue processes on the complex hyperbolic Grasmannian spaces. Stochastic area functionals in that setting can be defined  by means of the hyperbolic Stiefel fibration. The distribution of those area functionals is computed and limit theorems are proved.

\

At the end of the monograph, we included three appendices. The first one is a short summary of stochastic calculus results used throughout the monograph. In particular, we present in Theorem \ref{Yor transform section} the Yor's transform method, which is the powerful method due to M. Yor  and used throughout the monograph to compute conditional Laplace transforms of additive functionals of diffusion processes. The second appendix gives some useful formulas about the special diffusion operators appearing in the monograph. Finally, the last appendix lists the formulas for the radial parts of the Laplace-Beltrami operator on some rank-one symmetric spaces.

 \chapter{Stochastic areas and Brownian winding functionals}
%
%
 \section{The  L\'evy's area formula and some of its generalizations}
 
 \subsection{Complex L\'evy area formula}
 Consider the  L\'evy area process
 \[
 S(t)=\int_0^t \left(B_1(s) dB_2(s)-B_2(s) dB_1(s)\right),\quad t\ge0
 \]
 where $Z(t)=(B_1(t),B_2(t))$, $t \ge 0$, is a two-dimensional Brownian motion started at 0. We can write
 \[
 S(t)=\int_{Z[0,t]} \alpha
 \]
 where $\alpha$ is the 1-form $\alpha=xdy-ydx$. Since $d\alpha =2 dx \wedge dy$, we interpret $S(t)$ as (twice) the algebraic area swept out in the plane by the Brownian curve up to time $t$. The process $(S(t))_{t \ge 0}$ is not a Markov process in its own natural filtration. However, if we consider the 3-dimensional process
 \[
 X(t)=(B_1(t),B_2(t),S(t)),\quad t\ge0
 \]
 then $(X(t))_{t \geq 0}$ is solution of the following stochastic differential system:
 \begin{align*}
 \begin{cases}
 dX_1(t)& =dB_1(t), \\
 dX_2(t)& =dB_2(t), \\
 dX_3(t)& =-X_2(t) dB_1(t)+X_1(t) dB_2(t).
 \end{cases}
 \end{align*}
 As a matter of fact, it is a Markov process with generator:
 \begin{align}\label{eq-intro-L}
 L&=\frac{1}{2}(\mathbb{X}^2+\mathbb{Y}^2)\notag\\
  &=\frac{1}{2} \left( \frac{\partial^2}{\partial x^2}+\frac{\partial^2}{\partial y^2} \right)+\left(x  \frac{\partial}{\partial y} - y \frac{\partial }{\partial x}\right)\frac{\partial}{\partial z} +\frac{1}{2}(x^2+y^2) \frac{\partial^2}{\partial z^2}
 \end{align}
 where $\mathbb{X},\mathbb{Y}$ are the following vector fields:
 \[
 \mathbb{X}=\frac{\partial}{\partial x}-y \frac{\partial}{\partial z}
 \]
 \[
 \mathbb{Y}=\frac{\partial}{\partial y}+x \frac{\partial}{\partial z}.
 \]
 Observe that the Lie bracket
 \[
[\mathbb{X},\mathbb{Y}]=\mathbb{X}\mathbb{Y}-\mathbb{Y}\mathbb{X} =2 \frac{\partial}{\partial z}.
\]
Thus, for every $x \in \mathbb{R}^3$, $(\mathbb{X}(x),\mathbb{Y}(x), [\mathbb{X},\mathbb{Y}](x))$ is a basis of $\R^3$. From  H\"ormander's theorem (see Theorem \ref{Hormander theorem}), it follows that for every $t>0$, the random variable $X(t)$ has a smooth density with respect to the Lebesgue measure of $\R^3$. In particular, so does $S(t)$ and its density may be derived as follows. Firstly, one reduces the complexity of the random variable $X(t)$ by making use of invariance by rotational symmetries.

\begin{lemma}\label{rad heisenberg}
Let $r(t)=| B(t) |=\sqrt{ B_1(t)^2 +B_2(t)^2 }$, $t \ge 0$. Then, the couple
\[
(r(t) , S(t))_{t \ge 0}
\]
is a Markov process with generator
\[
\mathcal{L}=\frac{1}{2}  \frac{\partial^2}{\partial r^2}+\frac{1}{2r} \frac{\partial}{\partial r} + \frac{1}{2}r^2  \frac{\partial^2}{\partial s^2}.
\]
Consequently, the following equality in distribution holds:
 \[
\left( r(t) ,S(t) \right)_{t \ge 0} \overset{d}{=}  
\left( r(t),\gamma\left(\int_0^t  r^2(s)ds \right)\right)_{t \ge 0},
\]
where $(r(t))_{t \ge 0}$  is a 2-dimensional Bessel process and  $(\gamma(t))_{t \ge 0}$ is a standard  Brownian motion independent of $r$.

\end{lemma}

\begin{proof}
From It\^o's formula, we have
\begin{align*}
dr(t) &=\frac{dt}{2 r(t)}+\frac{B_1(t) dB_1(t)+B_2(t) dB_2(t)}{ \sqrt{ B_1(t)^2 +B_2(t)^2 }} \\
dS(t)&=r(t) \frac{B_1(t) dB_2(t)-B_2(t) dB_1(t)}{ \sqrt{ B_1(t)^2 +B_2(t)^2 }}.
\end{align*}
From Proposition \ref{carac Levy} in Appendix $1$ the two processes
\begin{align*}
\beta(t)=\int_0^t \frac{B_1(s) dB_1(s)+B_2(s) dB_2(s)}{ \sqrt{ B_1(s)^2 +B_2(s)^2 }} \\
\gamma (t)=\int_0^t \frac{B_1(s) dB_2(s)-B_2(s) dB_1(s)}{ \sqrt{ B_1(s)^2 +B_2(s)^2 }},
\end{align*}
are independent Brownian motions, the conclusion therefore follows.
\end{proof}

Secondly, one proves the celebrated L\'evy's area formula.
\begin{theorem}[L\'evy area formula]\label{levy area}
For $t>0$ and $x \in \R^2$, and $\lambda >0$
\begin{equation}\label{eq-leavy-area}
    \mathbb{E}\left( e^{i\lambda S(t)} \mid B(t)=x\right)=\frac{\lambda t}{\sinh \lambda t} e^{-\frac{|x|^2}{2t}(\lambda t \coth \lambda t -1) }.
\end{equation}
\end{theorem}

\begin{proof}
From the rotational symmetry of the Brownian motion $(B(t))_{t \ge 0}$, we have
\[
\mathbb{E}\left( e^{i\lambda S(t)} \mid B(t)=x\right)=\mathbb{E}\left( e^{i\lambda S(t)} \mid | B(t) | =| x | \right).
\] 
Then, according to the previous lemma,
\[
\mathbb{E}\left( e^{i\lambda S(t)} \mid | B(t) | =| x | \right)=\mathbb{E}\left( e^{i\lambda \gamma_{\int_0^t r(s)^2 ds}} \mid r(t) =| x | \right),
\]
where $\gamma(t)$ is a Brownian motion independent from $r$. We deduce
\[
\mathbb{E}\left( e^{i\lambda S(t)} \mid B(t)=x\right)=\mathbb{E}\left( e^{-\frac{\lambda^2}{2} \int_0^t r(s)^2 ds} \mid r(t) =| x | \right).
\]
As we have seen in the proof of Lemma \ref{rad heisenberg}, $r(t)$ solves a stochastic differential equation
\[
dr(t) =\frac{dt}{2 r(t)}+d\beta(t),
\]
where $\beta$ is a one-dimensional Brownian motion. In order to compute $$\mathbb{E}\left( e^{-\frac{\lambda^2}{2} \int_0^t r(s)^2 ds} \mid r(t) =| x | \right)$$ we will use the Yor's transform method which is explained in Theorem \ref{Yor transform section}.

More precisely, let us consider the new probability measure defined by:
\begin{equation}\label{eq-intro-Girsa}
    \mathbb{P}_{/ \mathcal{F}_t}^\lambda = \exp \left(- \lambda \int_0^t r(s) d\beta(s) -\frac{\lambda^2}{2} \int_0^t r(s)^2 ds \right)\mathbb{P}_{/ \mathcal{F}_t},
\end{equation}
where $\mathcal{F}$ is the natural filtration of $\beta$. Observe that
\begin{align}\label{eq-intro-int-rbeta}
\int_0^t r(s) d\beta(s) &=\int_0^t r(s) dr(s) -\frac{t}{2} \notag \\
 & =\frac{1}{2} r(t)^2 -t.
\end{align}
Therefore
\[
\exp \left( -\lambda \int_0^t r(s) d\beta(s) -\frac{\lambda^2}{2} \int_0^t r(s)^2 ds \right)=e^{\lambda t} \exp \left( -\frac{\lambda}{2} r(t)^2 -\frac{\lambda^2}{2} \int_0^t r(s)^2 ds \right).
\]
In particular, one deduces that
\[
\exp \left( -\lambda \int_0^t r(s) d\beta(s) -\frac{\lambda^2}{2} \int_0^t r(s)^2 ds \right) \le e^{\lambda t},
\]
which proves that 
\begin{equation*}
\exp \left( -\lambda \int_0^t r(s) d\beta(s) -\frac{\lambda^2}{2} \int_0^t r(s)^2 ds \right)_{t \geq 0} 
\end{equation*}
is a martingale. By using the change of probability \eqref{eq-intro-Girsa} and plugging in \eqref{eq-intro-int-rbeta}, if $f$ is a bounded Borel function, we have
\begin{align*}
\mathbb{E}\left( f(r(t)) e^{-\frac{\lambda^2}{2} \int_0^t r(s)^2 ds} \right)& =e^{-\lambda t} \mathbb{E}^\lambda \left(f(r(t)) \exp \left( \frac{\lambda}{2} r(t)^2 +\frac{\lambda^2}{2} \int_0^t r(s)^2 ds \right) e^{-\frac{\lambda^2}{2} \int_0^t r(s)^2 ds}  \right) \\
&=e^{-\lambda t} \mathbb{E}^\lambda \left(f(r(t))  \exp \left( \frac{\lambda}{2} r(t)^2 \right) \right).
\end{align*}
We are then left with the computation of the distribution of $r(t)$ under the probability $\mathbb{P}^\lambda$. To this end, Girsanov's theorem implies that the process
\[
\beta^\lambda(t)=\beta(t) +\lambda \int_0^t r(s) ds
\]
is a Brownian motion under the probability $\mathbb{P}^\lambda$. Thus
\[
dr(t) =\left( \frac{1}{2 r(t)} -\lambda r(t) \right) dt+d\beta^\lambda(t).
\]
In law, this is the stochastic differential equation solved by $| Y(t) |$ where $(Y(t))_{t \ge 0}$ solves the following stochastic differential equation
\[
dY(t) =-\lambda Y(t) dt +dB^\lambda (t), \quad Y(0)=0.
\]
We deduce that $r(t)$ is distributed as $| Y(t) |$, the norm of a two-dimensional Ornstein Uhlenbeck process with parameter $-\lambda$. Since $Y(t)$ is a Gaussian random variable with mean 0 and variance $(1-e^{-2\lambda t})/(2\lambda) I_2$, the conclusion follows from standard computations about the Gaussian distribution.
\end{proof}

\begin{remark}
This formula \eqref{eq-leavy-area} is due to Paul L\'evy \cite{Levy} who originally used a series expansion of the Brownian motion. The proof we present here is due to Marc Yor \cite{YorArea}. We also refer to Gaveau \cite{Gaveau} and Hulanicki \cite{MR418257} who uncovered the connection between the stochastic area and the Heisenberg group (as described below) and to \cite{MR898500} and \cite{MR2723056} for yet other approaches.
\end{remark}

The L\'evy's area formula has several interesting consequences. The first one is a representation of the heat kernel of the Markov process $(X(t))_{t \ge 0}$:

\begin{corollary}
Let $X(t)=(B_1(t),B_2(t),S(t))$, $t\ge0$. Then for $t >0$, the density of $X(t)$ with respect to the Lebesgue measure of $\mathbb R^3$ is given by:
\begin{align}\label{heat kernel Heisenberg}
p_t(x,y,z)=\frac{1}{(2\pi)^2} \int_{\R} e^{i \lambda z} \frac{\lambda t}{\sinh \lambda t} e^{-\frac{(x^2+y^2)}{2}\lambda  \coth \lambda t  } d\lambda.
\end{align}
\end{corollary}
 
 \begin{proof}
 The result easily follows from the L\'evy's formula Theorem \ref{levy area} using the inverse Fourier transform. For sake of completeness, we provide an alternative proof based on the fact the heat kernel $p_t(x,y,z)$ solves the heat equation
 \[
 \frac{\partial p}{\partial t}=Lp
 \]
 where $L$ is as given in \eqref{eq-intro-L}.
 The idea is to use a Fourier transform in $z$ and look for a solution of the above equation in the form of
 \[
 p_t(x,y,z)=\rp_t (r,z)
 \]
 where $r=\sqrt{x^2+y^2}$. Then we see that
\[
\rp_t (r,z)=\frac{1}{2\pi} \int_\R e^{i \lambda z} \Phi_t (r,\lambda) d\lambda,
\]
where $ \Phi_t (r,z,\lambda)$ is the fundamental solution at 0 of the parabolic partial differential equation
\[
\frac{\partial \Phi}{ \partial t}=\frac{1}{2}\frac{\partial^2 \Phi }{\partial r^2}+\frac{1}{2r} \frac{\partial \Phi }{\partial r}  -\frac{1}{2}\lambda^2 r^2 \Phi .
\]
Therefore we are reduced to obtaining a formula for the semigroup generated by the Sch\"rodinger operator
\[
\mathcal{L}_\lambda=\frac{\partial^2  }{\partial r^2}+\frac{1}{r} \frac{\partial  }{\partial r}  -\lambda^2 r^2.
\]
The trick (called the ground state method) is to observe that for every $f$,
\[
\mathcal{L}_\lambda \left( e^{\frac{\lambda r^2}{2}} f \right)= e^{\frac{\lambda r^2}{2}} \left( 2 \lambda +\mathcal{G}_\lambda \right)f,
\]
where
\[
\mathcal{G}_\lambda=\frac{\partial^2   }{\partial r^2}+\left( 2 \lambda r+\frac{1}{r} \right)\frac{\partial   }{\partial r}.
\]
The operator $\mathcal{G}_\lambda$ turns out to be the radial part of the Ornstein-Uhlenbeck operator $\Delta_{\mathbb{R}^{2}} +2 \lambda \langle x , \nabla_{\mathbb{R}^{2}} \rangle$ whose heat kernel at 0 is a  Gaussian density with mean 0 and variance $\frac{1}{2\lambda}(e^{4\lambda t}-1)$. This implies that the heat kernel issued from 0 of $\mathcal{G}_\lambda$ is given by
\[
q_t (r)=\frac{1}{2\pi} \left( \frac{2\lambda}{e^{4\lambda t}-1} \right) e^{-\frac{\lambda r^2}{ e^{4\lambda t}-1}}.
\]
As a result, 
\[
 \Phi_{2t} (r,z,\lambda)=\frac{e^{2\lambda t}}{2\pi} \left( \frac{2\lambda}{e^{4\lambda t}-1} \right) e^{-\frac{\lambda r^2}{2}}  e^{-\frac{\lambda r^2}{ e^{4\lambda t}-1}}= \frac{1}{2\pi} \left( \frac{\lambda}{\sinh 2\lambda t} \right)  e^{-\frac{\lambda r^2}{2}\coth 2\lambda t}
 \]
 and the conclusion follows. 
\end{proof}

\begin{remark}
At first, the formula \eqref{heat kernel Heisenberg} appears difficult to directly manipulate since it involves the integral of an oscillatory function. However, it is indeed possible to derive precise estimates and small-time asymptotics from it, see \cite{BBBC,BGG,Gaveau,HQLi}.
\end{remark}

When $x=0$, we deduce from the L\'evy area formula Theorem \ref{levy area} that
\[
\mathbb{E}\left( e^{i\lambda S(t)} | B(t)=0\right)=\frac{\lambda t}{\sinh \lambda t}.
\] 
This gives a formula for the characteristic function of the algebraic stochastic area within a complex Brownian loop with length $t$. Inverting this Fourier transform yields the density of this stochastic area
\[
\mathbb{P} \left( S(t) \in ds | B(t)=0 \right)=\frac{\pi}{2t} \frac{1}{\cosh^2 \left( \frac{\pi s}{t}\right)} ds,
\]
 see \cite{MR1007231, MR1007232} for interesting variations on that formula and connections to mathematical physics. 

Next, integrating the L\'evy's area formula \eqref{eq-leavy-area} with respect to the distribution of $B(t)$ yields the characteristic function of $S(t)$:
\[
\mathbb{E}\left( e^{i\lambda S(t)} \right)=\frac{1}{\cosh (\lambda t)}
\]
This Fourier transform is also easily inverted and gives the density of $S(t)$:
\[
\mathbb{P} \left( S(t) \in ds \right)=\frac{\pi}{t}\frac{1}{\cosh \left( \frac{\pi s }{2t} \right)}.
\]

One may deduce from it the following formula (due to Biane-Yor \cite{BianeYor}): For $\alpha>0$,
\[
\mathbb{E} (|S(t)|^\alpha)=\frac{2^{\alpha+2} \Gamma(1+\alpha)}{\pi^{1+\alpha}} L(1+\alpha) t^\alpha,
\]
where 
\begin{equation*}
L(s)=\sum_{n=0}^{+\infty} \frac{(-1)^n}{(2n+1)^s}    
\end{equation*} is the Dirichlet L-function associated with the character of $\mathbb{Z}/4\mathbb{Z}$. Doing so provides an unexpected and fascinating connection with the Riemann zeta function.

\subsection{Geometric structure of the L\'evy area}
As mentioned above, the L\'evy area process
 \[
 S(t)=\int_0^t\left( B_1(s) dB_2(s)-B_2(s) dB_1(s)\right),
 \]
 is closely related to the sub-Riemannian structure of the Heisenberg group. Indeed, recall that the 3-dimensional process
 \[
 X(t)=(B_1(t),B_2(t),S(t)), \quad t \geq 0,
 \]
  is a Markov process with generator
 \[
 L=\frac{1}{2}(\mathbb{X}^2+\mathbb{Y}^2)
 \]
 where $\mathbb{X},\mathbb{Y}$ are the following vector fields
 \[
 \mathbb{X}=\frac{\partial}{\partial x}-y \frac{\partial}{\partial z}
 \]
 \[
 \mathbb{Y}=\frac{\partial}{\partial y}+x \frac{\partial}{\partial z}.
 \]
 Let 
 \[
 \mathbb{Z}=\frac{\partial}{\partial z}.
 \]
 We have then Lie brackets commutation relations
 \[
[\mathbb{X},\mathbb{Y}]=2\mathbb{Z},\quad [\mathbb{X},\mathbb{Z}]=[\mathbb{Y},\mathbb{Z}]=0.
\]
Consequently, $(\mathbb{X},\mathbb{Y},\mathbb{Z})$ generate a 3-dimensional nilpotent Lie algebra of (complete) vector fields. Actually, this is the Lie algebra of the Heisenberg group
\[
\mathbf{H}^3=\{ (x,y,z) \in \mathbb{R}^3 \}
\]
where the non-commutative group law is given by
\[
(x_1,y_1,z_1) \star (x_2,y_2,z_2) =(x_1+x_2,y_2+y_2,z_1+z_2+x_1y_2-x_2y_1).
\]

For $s \le t$, one has
\begin{align*}
X(s)^{-1} \star X(t) & =(-B_1(s),-B_2(s),-S(s)) \star (B_1(t),B_2(t),S(t)) \\
& =(B_1(t)-B_1(s),B_2(t)-B_2(s),S(t)-S(s)-B_1(s)B_2(t)+B_2(s)B_1(t)).
\end{align*}
Observe now that
\[
S(t)-S(s)-B_1(s)B_2(t)+B_2(s)B_1(t)=\int_s^t \left[(B_1(u)-B_1(s))dB_2(u)-(B_2(u)-B_2(s))dB_1(u)\right].
\]
Therefore, $X(s)^{-1} \star X(t) $ is independent from the sigma field $\sigma(X(u),  u \le s)$ and is distributed as $X(t-s)$. Therefore it is  natural to call $(X(t))_{t\ge 0}$ a Brownian motion in the Heisenberg group. Observe also that while $(\mathbb{X},\mathbb{Y},\mathbb{Z})$ form a basis of the Lie algebra of $\mathbf{H}^3$, the generator of $(X(t))_{t \geq 0}$ only involves the vector fields $(\mathbb{X},\mathbb{Y})$. In other words, the \textit{direction} $\mathbb{Z}$ is missing. Calling $\mathbf{span}(\mathbb{X},\mathbb{Y})$ the set of {\bf horizontal} directions, we then refer to $(X(t))_{t \geq 0}$ as a \textbf{horizontal Brownian motion} in $\mathbf{H}^3$.

\subsection{Extension to Higher dimensions}

The L\'evy area formula is easily extended to higher dimensional  Heisenberg groups as follows.

Let $B(t): = B_1(t) +i B_2(t)$, $t\ge0$ be a Brownian motion in $\mathbb{C}^n$ started at 0. This means that $B_1$ and $B_2$ are two independent Brownian motions in $\mathbb{R}^n$. We consider the one-form
\[
\alpha=\sum_{j=1}^n (x_j dy_j -y_j dx_j)
\]
and define the generalized L\'evy area process by:
\begin{align}\label{area heisenberg multi}
S(t) =\int_{B[0,t]} \alpha =\sum_{j=1}^n 
\int_0^t \left(B_{1,j}(s) dB_{2,j}(s) -B_{2,j}(s)dB_{1,j}(s)\right).
\end{align}
We note that 
\[
d\alpha=2\sum_{j=1}^n dx_j \wedge dy_j
\]
which is twice the canonical symplectic form on $\mathbb C^n$. It is therefore natural to interpret $S(t)$ as a stochastic area type process as we did when $n=1$.

The process
\[
X(t)=(B(t),S(t)),\quad t\ge0
\]
is a diffusion process in $\mathbb{C}^n \times \mathbb{R}$ with generator
\[
L=\frac{1}{2} \sum_{j=1}^n \left(\mathbb{X}_j^2+\mathbb{Y}_j^2\right)
\]
where we set
\[
\mathbb{X}_j=\frac{\partial}{\partial x_j} -y_j \frac{\partial}{\partial z}, \quad 1 \leq j \leq n,
\]
\[
\mathbb{Y}_j=\frac{\partial}{\partial y_i} +x_j \frac{\partial}{\partial z}, \quad 1 \leq j \leq n,
\] 
with $(x_1, y_1,\dots, x_n, y_n, z)$ being the real coordinates in $\mathbb{C}^n \times \mathbb{R}$. Setting further
 \begin{equation*}
 \mathbb{Z}= \frac{\partial}{\partial z},  
 \end{equation*}
 we have the following Lie brackets relations:
 \[
 [\mathbb{X}_m,\mathbb{X}_j]=[\mathbb{Y}_m,\mathbb{Y}_j]
 =[\mathbb{X}_j,\mathbb{Z}]=[\mathbb{Y}_j,\mathbb{Z}]=0
 \]
 and
 \[
 [\mathbb{X}_m,\mathbb{Y}_j]=2 \delta_{mj} \mathbb{Z}, \quad 
 \quad 1 \leq m,j \leq n,.
 \]
In particular, the H\"ormander's condition is satisfied and $(\mathbb{X}_j,\mathbb{Y}_j,\mathbb{Z})_{1 \leq j \leq n} $ generate the Lie algebra of the $2n+1$-dimensional Heisenberg group
\[
\mathbf{H}^{2n+1}=\{ (x,y,z) \in \mathbb{R}^n\times \mathbb{R}^n\times \mathbb{R} \}
\]
with the group law 
\[
(x,y,z) \star (x',y',z') =\left(x+x',y+y',z+z'+\sum_{j=1}^n(x_{j}y'_{j}-x'_{j}y_{j})\right).
\]
As before, we can interpret $(X(t))_{t \ge 0}$ as a horizontal Brownian motion on $\mathbf{H}^{2n+1}$. Lemma \ref{rad heisenberg} and Theorem \ref{levy area} are then easily generalized as follows using almost identical proofs.

\begin{lemma}
Let $r(t)=| B(t) |=\sqrt{ | B_1(t)|^2 +| B_2(t)|^2 }$, $t \ge 0$. Then, the couple
\[
(r(t) , S(t))_{t \ge 0}
\]
is a Markov process with generator
\[
\mathcal{L}=\frac{2n-1}{2r} \frac{\partial}{\partial r}+\frac{1}{2}  \frac{\partial^2}{\partial r^2}+\frac{1}{2}r^2  \frac{\partial^2}{\partial s^2}.
\]
Consequently, the following equality in distribution holds:
 \[
\left( r(t) ,S(t) \right)_{t \ge 0} \overset{d}{=}  
\left( r(t),\gamma\left(\int_0^t  r^2(s)ds \right)\right)_{t \ge 0},
\]
where $(r(t))_{t \ge 0}$  is a $2n$-dimensional Bessel process and  $(\gamma(t))_{t \ge 0}$ is a standard  Brownian motion independent from $(r(t))_{t \ge 0}$

\end{lemma}

\begin{theorem}\label{levy area 2}
For any $t>0, x \in \mathbb C^n$ and any $\lambda >0$,
\[
\mathbb{E}\left( e^{i\lambda S(t)} | B(t)=x\right)=\left(\frac{\lambda t}{\sinh \lambda t} \right)^n e^{-\frac{|x|^2}{2t}(\lambda t \coth \lambda t -1) }.
\] 
\end{theorem}

On can further generalize Theorem \ref{levy area 2}  using the diagonalization theory of skew-symmetric matrices.

\begin{theorem}\label{general levy area}
Let $A$ be a $2n \times 2n$ skew-symmetric real matrix and $(B(t))_{t \ge 0}$ be a $2n$-dimensional real Brownian motion. Then for $x \in \mathbb R^{2n}$ and $t>0$,
\[
\mathbb{E} \left( e^{i \int_0^t \left\langle AB(s) ,dB(s) \right\rangle } \mid B(t)=x \right)=\mathrm{det} \left(\frac{ t A}{\sinh  tA} \right)^{1/2} e^{-\frac{1}{2t} \left\langle(  t A \coth ( t A) -I)x,x\right\rangle }.
\]
where $\langle\cdot,\cdot\rangle$ is the Euclidean inner product of $\mathbb{R}^{2n}$. 
\end{theorem}

\begin{proof}
Using orthonormal invariance of the Brownian motion $(B(t))_{t \ge 0}$ and diagonalization theory of skew-symmetric real matrices, one can assume that $A$ is of the form
\[
A=  \begin{pmatrix}
    \Lambda_1 & 0 & \dots & 0 \\
    0 & \Lambda_2 & \dots & 0 \\
    \vdots & \vdots & \ddots & \vdots \\
    0 & 0 & \dots & \Lambda_n
  \end{pmatrix}
\]
where
$
\Lambda_j= \begin{pmatrix}
0 & -\lambda_j \\
\lambda_j & 0
\end{pmatrix}
$ with $\lambda_j \in \mathbb R, 1 \leq j \leq n$.
As a result, 
\begin{align*}
 & \mathbb{E} \left( e^{i \int_0^t \left\langle AB(s) ,dB(s) \right\rangle } \mid B(t) =x \right) \\
 =&\mathbb{E} \left( e^{i \sum_{j=1}^n \lambda_j
 \int_0^t B_{2j-1}(s)dB_{2j}(s)-B_{2j}(s)dB_{2j-1}(s) } \mid B(t)=x \right)\\
  =&\prod_{j=1}^n\mathbb{E} \left( e^{i  \lambda_j 
  \int_0^t B_{2j-1}(s)dB_{2j}(s)-B_{2j}(s)dB_{2j-1}(s) } 
  \mid B_{2j-1}(t)=x_{2j-1},B_{2j}(t)=x_{2j}  \right).
\end{align*}
From \eqref{eq-leavy-area}, we end up with:
\[
\mathbb{E} \left( e^{i \int_0^t \left\langle AB(s) ,dB(s) \right\rangle } \mid B(t) =x \right)
=\prod_{j=1}^n \left(\frac{\lambda_j t}{\sinh \lambda_j t}\right) e^{-\frac{x_{2j-1}^2+x_{2j}^2}{2t}(\lambda_j t \coth \lambda_j t -1) }.
\]
\end{proof}

\begin{remark}
In particular, one obtains:
\[
\mathbb{E} \left( e^{i \int_0^t \left\langle AB(s) ,dB(s) \right\rangle } \mid B(t)=0 \right)=\mathrm{det} \left(\frac{ t A}{\sinh  tA} \right)^{1/2} 
\]
This formula is closely related to the differential of the exponential map on  compact Lie groups, see \cite{MR1439539}, and plays an important role in the probabilistic approach to Atiyah-Singer index theorems, see \cite{MR2376573} and \cite{MR744920,MR756173}.
\end{remark}

\subsection{Quaternionic L\'evy area formula}

Theorem \ref{general levy area} provides the general form of possible L\'evy area formulas associated with Euclidean Brownian motions. The formula takes a particularly nice form in geometric situations that we will explore in the next two sections.

Let $\mathbb{H}$ be the non-commutative field of quaternions
\[
\mathbb{H}=\{q=u+xI+yJ+zK, (u,x,y,z)\in\R^4\},
\]
where  $I,J,K$ are defined as 
\[
I=\left(
\begin{array}{ll}
i & 0 \\
0&-i 
\end{array}
\right), \quad 
J= \left(
\begin{array}{ll}
0 & 1 \\
-1 &0 
\end{array}
\right), \quad 
K= \left(
\begin{array}{ll}
0 & i \\
i &0 
\end{array}
\right).
\]
We note that $I,J,K$ satisfy the Hamilton's relations $I^2=J^2=K^2=IJK=-1$. 
For $q=u+xI+yJ+zK \in \mathbb{H}$, we denote by $\overline q= u -xI-yJ-zK$ its conjugate, $|q|^2=u^2+x^2+y^2+z^2$ its squared norm and $\mathrm{Im}(q)=(x,y,z) \in \mathbb{R}^3$ its imaginary part. 
For $q = (q_1, \dots, q_n), q' = (q'_1, \dots, q'_n) \in \bH^{n}$, we will set\footnote{Note that the different convention that $\mathrm{Im}\langle q,q'\rangle=\mathrm{Im}\sum_{j=1}^n \overline{q_j} q'_j$ is also sometimes used in literature.} 
\begin{equation*}
\mathrm{Im}\langle q,q'\rangle=\mathrm{Im}\sum_{j=1}^nq_j\overline{q'_j}.
\end{equation*}The quaternionic stochastic area process is defined by:
\begin{align}\label{quaternionic levy area}
S(t)=\sum_{j=1}^n \int_0^t\mathrm{Im}\langle dB_j(s), B_j(s) \rangle
\end{align}
where $(B_j(t))_{t \geq 0}, 1 \leq i \leq n$, are independent $\bH$-valued Brownian motions. Note that unlike the previous cases, this process takes values in $\mathbb R^3$. The Markov process
\begin{equation*}
\left( X(t)\right)_{t \geq 0} = \left(B_1(t), \dots, B_n(t), \sum_{j=1}^n\int_0^t\mathrm{Im}\langle dB_j(s), B_j(s) \rangle\right)_{t \geq 0},
\end{equation*}
can then be interpreted as follows, as a horizontal Brownian motion on the so-called \emph{quaternionic Heisenberg group}. The latter is  defined as the product space $\bH^{n}\times\mathrm{Im}(\bH)$ equipped with the group law
\[
(q,\phi )*(q', \phi')=\left(q+q',\phi+\phi'+\mathrm{Im}\langle q,q'\rangle \right),
\]
 and is one of the basic examples of H-type groups discussed in the next section. Applying the same arguments as before, we readily obtain the following results:
\begin{lemma}
Let $r(t)=| B(t) |=\sqrt{ \sum_{j=1}^n | B_j(t)|^2  }$, $t \ge 0$. Then, the couple
\[
(r(t) , S(t))_{t \ge 0}
\]
is a Markov process with generator
\[
\mathcal{L}=\frac{1}{2}\frac{\partial^2}{\partial r^2} + \frac{4n-1}{2r} \frac{\partial}{\partial r} +\frac{1}{2}r^2 \Delta_{\R^3}  .
\]
Consequently, the following equality in distribution holds:
 \[
\left( r(t) ,S(t) \right)_{t \ge 0} \overset{d}{=}  
\left( r(t),\beta\left(\int_0^t  r^2(s)ds \right)\right)_{t \ge 0},
\]
where $(r(t))_{t \ge 0}$ is a $4n$-dimensional Bessel process  and $(\beta(t))_{t \ge 0}$ is a standard $3$-dimensional Brownian motion independent from $(r(t))_{t \ge 0}$.
\end{lemma}

\begin{theorem}
For $t>0$ and $x \in \mathbb H^n$, and $\lambda \in \mathbb R^3$
\[
\mathbb{E}\left( e^{i \lambda\cdot S(t)} | B(t)=x\right)=\left(\frac{|\lambda| t}{\sinh |\lambda | t} \right)^{2n} e^{-\frac{ | x |^2}{2t}(|\lambda | t \coth |\lambda| t -1) },
\] 
where $\lambda\cdot S(t)=\sum_{j=1}^3\lambda_j S_j(t) $ and $|\lambda|=\sqrt{ \lambda_1^2+\lambda_2^2+\lambda_3^2}$.
\end{theorem}

Using Fourier inversion formula, one further derives the following expression for the density of $S(t)$.

\begin{proposition}
For any $t \geq 0$, the density of $S(t)$ with respect to the Lebesgue measure of $\mathbb R^3$ is given by
\begin{equation*}
h_t(\phi) = \frac{2^{2n-1}}{16\pi t^3} \int_{-1}^1 du \int_0^1 dv\, (v(1-v))^{n-1} \left(\frac{v}{1-v}\right)^{iu|\phi|/2t}\ln^2\left[\frac{v}{1-v}\right] .
\end{equation*}
\end{proposition}

\begin{proof}
From the last theorem, it follows that: 
\begin{equation}\label{eq-lim-He}
\mathbb{E}\left(e^{i \lambda\cdot S(t)}\right)
=\left(\cosh |\lambda|t\right)^{-2n}
\end{equation}
Using  Fourier inversion formula, the density of $S(t)$ is given up to a normalizing constant by: 
\begin{equation*}
h_t(\phi) := \int_{\mathbb{R}^3}\left(\cosh |\lambda|t \right)^{-2n}e^{-i \lambda\cdot \phi} d\lambda,
\end{equation*}
which reads in polar coordinates as
\begin{equation*}
h_t(\phi) = \int_{0}^{\infty}r^2\left(\cosh(rt)\right)^{-2n}\left\{\int_{S^2} e^{-i r \theta \cdot \phi} d\theta\right\} dr, \quad \phi \in \mathbb{R}^3.
\end{equation*}
By rotation invariance, the inner integral may be written as: 
\begin{equation*}
\int_{-1}^1e^{-i r u |\phi|} du = \int_{-1}^1\cos(r u |\phi|)du,
\end{equation*}
and as such, Fubini Theorem entails: 
\begin{equation*}
h_t(\phi) = \int_{-1}^1 \int_{0}^{\infty}r^2\cos(r u |\phi|) \left(\cosh rt \right)^{-2n} dr du.
\end{equation*}
But, we know from Table 3.985 in \cite{Gra-Ryz} that for any $z \in \mathbb{R}$: 
\begin{align*}
\int_{0}^{\infty} \cos(r z) \left(\cosh rt \right)^{-2n} dr & = \frac{2^{2n-1}}{2t(2n-1)!}\left|
\Gamma\left(n+i\frac{z}{2t}\right)\right|^2 
\\& = \frac{2^{2n-1}}{2t}\int_0^1 [v(1-v)]^{n-1} \left[\frac{v}{1-v}\right]^{iz/2t}dv,
\end{align*}
whence
\begin{align*}
\int_{0}^{\infty}r^2\cos(rz) \left(\cosh rt \right)^{-2n} dr  & = -\frac{2^{2n-1}}{2t(2n-1)!} \frac{d^2}{dz^2}\left|
\Gamma\left(n+i\frac{z}{2t}\right)
\right|^2
\\& = \frac{2^{2n-1}}{8t^3}\int_0^1 [v(1-v)]^{n-1} \left[\frac{v}{1-v}\right]^{iz/2t}\ln^2\left[\frac{v}{1-v}\right] dv.
\end{align*}
Note that the last integral is absolutely convergent (uniformly in $z$) which may be easily seen after performing there the variable change $y = v/(1-v)$: 
\begin{equation*}
\int_0^1 [v(1-v)]^{n-1} \left[\frac{v}{1-v}\right]^{iz/2t}\ln^2\left[\frac{v}{1-v}\right] dv = \int_0^{\infty} y^{iz/2t} \frac{y^{n-1}}{(1+y)^{2n}} \ln^2(y) dy.
\end{equation*}
Substituting $z = u|\phi|$ and taking into account the factor $1/(2\pi)$ present in Fourier inversion formula, the density follows.  
\end{proof}

\subsection{Octonionic L\'evy area formula}\label{Octonions}

We consider the non-associative division algebra of octonions which is abstractly described by the set
\begin{equation*}
\mathbb{O}=\left\{x=\sum_{j=0}^{7}x_j e_j, x_j\in \mathbb{R} \right\},
\end{equation*}
endowed with the multiplication rules are given by: 
\begin{equation*}
e_0 e_j=e_j, \quad e_j e_0=e_j, \quad 1 \leq j \leq 7,
\end{equation*}
and 
\begin{equation*}
e_m e_j=-\delta_{mj}e_0+\epsilon_{mjk}e_k, 1 \leq m,j,k \leq 7,
\end{equation*}
where $\delta_{mj}$ is the Kronecker delta and $\epsilon_{mjk}$ is the completely skew-symmetric tensor with value $1$ when $mjk = 123, 145, 176, 246, 257, 347, 365$. In other words,  simply denoting $e_0=1$, the multiplication table is given by:

{
\begin{center}
\begin{tabular}{|l|l|l|l|l|l|l|l|}
		\hline
		& $e_1$ & $e_2$ & $e_3$ & $e_4$ & $e_5$ & $e_6$ & $e_7$  \\ \hline
		$e_1$ & $-1$ & $e_4$ & $e_7$ & $-e_2$ & $e_6$ & $-e_5$ & $-e_3$ \\ \hline
		$e_2$ & $-e_4$ & $-1$ & $e_5$ & $e_1$ & $-e_3$ & $e_7$ & $-e_6$ \\ \hline
		$e_3$ & $-e_7$ & $-e_5$ & $-1$ & $e_6$ & $e_2$ & $-e_4$ & $e_1$  \\ \hline
		$e_4$ & $e_2$ & $-e_1$ & $-e_6$ & $-1$ & $e_7$ & $e_3$ & $-e_5$ \\ \hline
		$e_5$ & $-e_6$ & $e_3$ & $-e_2$ & $-e_7$ & $-1$ & $e_1$ & $e_4$  \\ \hline
		$e_6$ & $e_5$ & $-e_7$ & $e_4$ & $-e_3$ & $-e_1$ & $-1$ & $e_2$  \\ \hline
		$e_7$ & $e_3$ & $e_6$  & $-e_1$ & $e_5$  & $-e_4$ & $-e_2$ & $-1$    \\ \hline
\end{tabular}
\end{center}
}

The octonionic multiplication is not associative, but it is alternative; This means that the real algebra generated by two arbitrary octonions is associative. More than that, it is possible to prove that the algebra generated by two arbitrary octonions is always isomorphic to the field $\mathbb R, \mathbb C$ or $\mathbb H$.

If $x=\sum_{j=0}^{7}x_j e_j,\in \mathbb{O}$, its conjugate and the octonionic norm are defined by:
\[
\bar{x}=x_0 -\sum_{j=1}^{7}x_j e_j, \quad
|x|^2=\sum_{j=0}^{7}x^2_j.
\]
From the multiplication table, one readily check that
\[
x \bar{x}=\bar{x} x= |x|^2, 
\]

and that the multiplicative property holds: for $x,y \in \mathbb O$,
\[
| x y |=|x| \, |y|.
\]
One also defines the imaginary part of $x \in \mathbb O$ as
\[
\mathrm{Im} (x)=\frac{x-\bar{x}}{2}=\sum_{j=1}^{7}x_j e_j.
\]
For further details about octonions, we refer the interested reader to \cite{MR1886087}. 

Coming to the octonionic  stochastic area process, it is defined by
\begin{align}\label{octonionic levy area}
S(t)= \int_0^t\mathrm{Im}( dB(s) \overline{B(s)})
=\frac{1}{2}\int_0^t \left(dB(s) \overline{B(s)} -B(s) \overline{dB(s)}\right)
\end{align}
where $(B(t))_{t \geq 0}$, is a $\mathbb{O}$-valued Brownian motion, that is
\[
B(t)=\sum_{j=0}^7 B_j(t) e_j
\]
where the $B_j(t)$'s are independent real Brownian motions. In particular, the  octonionic L\'evy area is seven-dimensional and the 15-dimensional Markov process
\begin{equation*}
\left( X(t)\right)_{t \geq 0} = \left(B(t), S(t)\right)_{t \geq 0},
\end{equation*}
can then be interpreted as  a horizontal Brownian motion on the so-called octonionic Heisenberg  group. Indeed, the latter is nothing else but the product space $\mathbb{O} \times\mathrm{Im}(\mathbb O)$ equipped with the group law
\[
(x,\phi )*(x', \phi')=\left(x+x',\phi+\phi'+\mathrm{Im}(x\overline{x'}) \right)
\]
 and is an example of an H-type group (see the Section \ref{sec-H-group}). 
%
%
 In that framework, one obtains analogous results to those derived for the complex and the quaternionic algebras:

\begin{lemma}
Let $r(t)=| B(t)|$, $t \ge 0$. Then, the couple
\[
(r(t) , S(t))_{t \ge 0}
\]
is a Markov process with generator
\[
\mathcal{L}=\frac{1}{2}  \frac{\partial^2}{\partial r^2}+ 
\frac{7}{2r} \frac{\partial}{\partial r}+\frac{1}{2}r^2 \Delta_{\R^7}  .
\]
Consequently, the following equality in distribution holds:
 \[
\left( r(t) ,S(t) \right)_{t \ge 0} \overset{d}{=}  \left( r(t),\beta\left(\int_0^t  r^2(s)ds\right)\right)_{t \ge 0},
\]
where $(r(t))_{t \ge 0}$ is a $8$-dimensional Bessel process  and $(\beta(t))_{t \ge 0}$ is a standard $7$-dimensional Brownian motion independent from $(r(t))_{t \ge 0}$.
\end{lemma}

\begin{theorem}
For any $t>0, x \in \mathbb O$, and any $\lambda \in \mathbb R^7$, one has: 
\[
\mathbb{E}\left( e^{i\langle \lambda, S(t)\rangle} | B(t)=x\right)=\left(\frac{|\lambda| t}{\sinh |\lambda | t} \right)^{4} e^{-\frac{|x|^2}{2t}(|\lambda | t \coth |\lambda | t -1) }.
\] 
\end{theorem}

\subsection{Carnot and H-type groups}

Natural generalizations of complex, quaternionic, and octonionic stochastic processes, as well as their corresponding L\'evy's formulas, are possible by utilizing algebraic structures known as H-type groups.
 To provide perspective, we present these generalizations within the context of the broader class of Carnot groups. 

A Carnot group of step (or depth) $N$ is a simply connected Lie
group $\G$ whose Lie algebra can be decomposed as
\[
\mathfrak{g}=\mathcal{V}_{1}\oplus\cdots\oplus \mathcal{V}_{N},
\]
where
\[
\lbrack \mathcal{V}_{i},\mathcal{V}_{j}]=\mathcal{V}_{i+j}
\]
and
\[
\mathcal{V}_{s}=0,\text{ for }s>N.
\]
Such decomposition of $\mathfrak{g}$ is called a Lie algebra stratification and the above bracket relations show that Carnot groups are nilpotent: any Lie  bracket of length more than $N$ has to be zero.

The number
\[
Q=\sum_{i=1}^N i \dim \mathcal{V}_{i}
\]
is called the homogeneous dimension of $\G$ and any element 
$\mathbb{X} \in \mathfrak{g}$ can be identified with  a left invariant vector field on $\G$ through the following action on smooth functions:
\[
df(\mathbb{X} (g)):=\lim_{t \to 0} \frac{ f\left( g e^{t\mathbb{X}} \right)-f(g)}{t}, \quad g\in \G.
\]

One can then define horizontal Brownian motions in Carnot groups using stochastic differential equations (see also Section \ref{BM Lie group} about Brownian motions on Lie groups).

\begin{definition}
Let $\G$ be a Carnot group with unit $\mathbf{e}$ and with Lie algebra stratification:
\[
\mathfrak{g}=\mathcal{V}_{1}\oplus\cdots\oplus \mathcal{V}_{N}.
\]
If $\mathbb{X}_1,\dots,\mathbb{X}_d$ is a basis of $\mathcal{V}_{1}$ and if $(B(t))_{t \ge 0}$ is a $d$-dimensional Brownian motion, then the $\G$-valued process $(Y(t))_{t \ge 0}$ solving the Stratonovich stochastic differential equation
\[
dY(t) =\sum_{i=1}^d \mathbb{X}_i (Y(t)) \circ dB_{i}(t) , \quad Y_0=\mathbf{e},
\]
is called a horizontal Brownian motion on $\G$.
\end{definition}

The Lie algebra structure of $\mathfrak{g}$ implies that H\"ormander's Lie bracket generating condition is satisfied for the system of vector fields $\mathbb{X}_1,\dots,\mathbb{X}_d$. In particular the horizontal Brownian motion admits a smooth density with respect to the Haar measure. This density is called the horizontal heat kernel. There are no explicit formulas  for the horizontal heat kernel beyond the two-step generating case (see further comments below). However, due to the rich structure of Carnot groups (namely the existence of scaling structure) one can get quite precise estimates for the horizontal heat kernel and its derivatives, see \cite{MR1218884}. We refer to \cite{MR2154760} for a  discussion on Carnot groups and horizontal Brownian motions on them.

\begin{example}[Commutative case] The groups $\left( \mathbb{R}^d ,+ \right)$ are the only commutative Carnot groups.  Any standard Brownian motion in $\mathbb{R}^d$ is then a horizontal Brownian motion in the Carnot group $\left( \mathbb{R}^d ,+ \right)$.
\end{example}

\begin{example}[Heisenberg groups] Recall the Heisenberg group $\mathbf{H}^{2n+1} =\mathbb{R}^{2n} \times \mathbb{R}$
endowed with the group law
\[
(x,\alpha) \star (y, \beta)=\left( x+y, \alpha + \beta
+ \omega (x,y) \right),
\]
where $\omega$ is the standard symplectic form on
$\mathbb{R}^{2n}$, that is
\[
\omega(x,y)= x^t \left(
\begin{array}{ll}
0 & -\mathrm{I}_{n} \\
\mathrm{I}_{n} & ~~~0
\end{array}
\right) y.
\]
On the Lie algebra
$\mathfrak{h}^{2n+1}$, the Lie bracket is given by
\[
[ (x,\alpha) , (y, \beta) ]=\left( 0,2 \omega (x,y) \right),
\]
and it is easily seen that
\[
\mathfrak{h}^{2n+1}=\mathcal{V}_1 \oplus \mathcal{V}_2,
\]
where $\mathcal{V}_1 =\mathbb{R}^{2n} \times \{ 0 \}$ and
$\mathcal{V}_2= \{ 0 \} \times \mathbb{R}$. Therefore
$\mathbf{H}^{2n+1}$ is a Carnot group of step 2. If $(B(t))_{t \ge 0}$ is a standard Brownian motion in $\mathbb{R}^{2n}$, then the process
\[
Y(t)=\left( B(t), \int_0^t  \omega (B(s),\circ dB(s))\right)
\]
is a horizontal Brownian motion in $\mathbf{H}^{2n+1}$.
\end{example}

\begin{example}[Engel group] The Engel group $\mathbb E$ is the $4$-dimensional Lie group of matrices
\[
\mathbb{E}= \left\{ 
\left(
\begin{array}{llll}
1 & x & \frac{x^2}{2} & z \\
0 & 1 & x & w \\
0 & 0 & 1 & y \\
0 & 0 & 0 & 1
\end{array}
\right), x,y,w, z \in \mathbb{R} \right\}.
\]
Its Lie algebra $\mathfrak{E}$ is generated by the matrices
\[
X= \left(
\begin{array}{llll}
0 & 1 & 0 & 0 \\
0 & 0 & 1 & 0 \\
0 & 0 & 0 & 0 \\
0 & 0 & 0 & 0
\end{array}
\right),
Y= \left(
\begin{array}{llll}
0 & 0 & 0 & 0 \\
0 & 0 & 0 & 0 \\
0 & 0 & 0 & 1 \\
0 & 0 & 0 & 0
\end{array}
\right)
\]
\[
W= \left(
\begin{array}{llll}
0 & 0 & 0 & 0 \\
0 & 0 & 0 & 1 \\
0 & 0 & 0 & 0 \\
0 & 0 & 0 & 0
\end{array}
\right),
Z= \left(
\begin{array}{llll}
0 & 0 & 0 & 1 \\
0 & 0 & 0 & 0 \\
0 & 0 & 0 & 0 \\
0 & 0 & 0 & 0
\end{array}
\right),
\]
for which we have the following structure relations,
\[
[X,Y]=W, \quad [X,W]=Z
\]
while all other brackets are zero. The Engel group is easily seen to be a Carnot group of step 3 with stratification
\[
\mathfrak{E}= \mathbf{span} \{ X,Y \} \oplus  \mathbf{span} \{ W \} \oplus  \mathbf{span} \{ Z \}
\]
If $(B(t))_{t \ge 0}$ is a two-dimensional Brownian motion, then the process
\[
Y(t)= \left(
\begin{array}{llll}
1 & B_1(t) & \frac{1}{2}(B_1(t))^2 & \frac{1}{2} \int_0^t (B_1(s))^2 dB_2(s) \\
0 & 1 & B_1(t) & \int_0^t B_1(s) dB_2(s) \\
0 & 0 & 1 & B_2(t) \\
0 & 0 & 0 & 1
\end{array}
\right)
\]
is a horizontal Brownian motion on $\mathbb{E}$. To the best of our knowledge, there is no analogue of the L\'evy area formula for the horizontal Brownian motion on the Engel group in the sense that there exists no closed formula for the Fourier transform of the functional $\int_0^t (B_1(s))^2 dB_2(s)$. However, we point out that the heat kernel of the horizontal Brownian motion might be analyzed through the representation theory of nilpotent groups, see \cite{MR3518678} for further details. 

\end{example}
\subsection{$H$-type groups}\label{sec-H-group}
$H$-type groups form a special class of Carnot groups. Actually, a $H$-type group is simply $\R^{2n+m} =\R^{2n}\times \R^{m} $ equipped with the following product:
$$
v*w=v+w+ [v,w]
$$  
where $[\cdot,\cdot]: \R^{2n+m}  \times \R^{2n+m} \to \{0\}\times   \R^{m} $ is a Lie bracket operation on $\R^{2n+m}$ whose  center is  $\{0\}\times   \R^{m}$ and such that the map $J_z:\R^{2n} \to \R^{2n}$ defined for $z\in \R^{m}$  by:
\begin{align}\label{representation}
\langle J_z(x),y\rangle= \langle [x,y],z\rangle \textrm{ for all } x,y \in \R^{2n},
\end{align}
is orthogonal when $|z|=1$. Here, we identify $x\in \R^{2n}$ with $(x,0)\in  \R^{2n}\times \R^{m}$ and $z\in  \R^{m} $ with $(0,z) \in \R^{2n}\times \R^{m}$ ($\langle \cdot, \cdot \rangle$ still denotes the  Euclidean inner product). It is easy to see that any H-type group is a Carnot group of step 2 with stratification:
\[
\mathbb{R}^{2n+m}=\mathbb{R}^{2n} \oplus \mathbb{R}^m.
\]

H-type groups can be constructed using the representation theory of Clifford algebras. More precisely, let $\mathbf{Cl}(\mathbb R^m)$ be the Clifford algebra of $\mathbb R^m$, that is $\mathbf{Cl}(\mathbb R^m)$ is the free algebra generated by $\mathbb R^m$ subject to the relations
\[
u\cdot v+v\cdot u= - 2\langle u,v\rangle, \quad u,v \in \mathbb R^m.
\]
A representation of $\mathbf{Cl}(\mathbb R^m)$ on $\mathbb{R}^{2n}$ is then an algebra morphism $J$ from $\mathbf{Cl}(\mathbb R^m)$ to the space of linear  maps $\mathbb{R}^{2n} \to \mathbb{R}^{2n} $ such that  for every $u \in \mathbb R^m$ with $|u|=1$, $J_u$ is an orthogonal map. If $J$ is such representation, then equation \eqref{representation}, defines a Lie bracket  $ \R^{2n+m} \to \{0\}\times   \R^{m} $ with the required properties.

The Heisenberg group $\mathbf{H}^{2n+1}$ comes from the representation of $\mathbf{Cl}(\mathbb R)\simeq \mathbb C$ on $\mathbb{R}^{2n}$ given by
\[
J_{b}= \left(
\begin{array}{ll}
0  & - b{I}_{n} \\
b {I}_{n} & 0
\end{array}
\right), \quad b \in \mathbb{R}.
\]

Denoting $\bH$ the field of quaternions and $\mathrm{Im}(\bH)$ the set of imaginary quaternions, the quaternionic Heisenberg group  comes from the representation of $\mathbf{Cl}(\mathbb R^3) \simeq \mathbf{Cl}(\mathrm{Im}(\bH))$ on $\mathbb{R}^{4n}\simeq \bH^n$ which is  given by:
\[
J_u (q_1,\dots,q_n)=(uq_1,\dots,uq_n), \quad u \in \mathrm{Im}(\bH),\  q_i \in \bH.
\]

In a similar way, denoting $\mathbb O$ the set of octonions and $\mathrm{Im}(\mathbb O)$ the set of imaginary octonions, the octonionic Heisenberg group  comes from the representation of $\mathbf{Cl}(\mathbb R^7) \simeq \mathbf{Cl}(\mathrm{Im}(\mathbb O))$ on $\mathbb{R}^{8}\simeq \mathbb O$ which is  given by:
\[
J_u(x) =ux, \quad u \in \mathrm{Im}(\mathbb O),\  x \in \mathbb O.
\]

On a H-type group, the left-invariant vector fields which coincide with $(\partial/\partial x_j)_{1 \leq j \leq n}$ at the identity write:
$$
\mathbb{X}_q = \frac{\partial}{\partial x_j} +
\sum_{j=1}^m \langle J_{u_j} x, e_q\rangle \frac{\partial}{\partial z_j},
\quad q=1,\dots,2n,
$$
where $(e_1,\dots e_{2n})$ is the canonical basis of $\R^{2n}$ and $(u_1,\dots u_{m})$ the canonical basis of $\R^{m}$.  

Accordingly, if $(B(t))_{t \ge 0}$ is a standard Brownian motion in $\R^{2n}$, then the process
\[
Y(t)= \left(B(t),  \int_0^t [B(s),dB(s)]\right), \quad t\geq 0,
\]
is a horizontal Brownian motion on the H-type group. In this respect, the $\mathbb{R}^m$ valued functional
\[
S(t) = \int_0^t [B(s),dB(s)]
\]
is a natural generalization of the complex or quaternionic stochastic areas. The following results then hold true:

\begin{lemma}
Let $r(t)=| B(t) |$, $t \ge 0$. Then, the couple
\[
(r(t) , S(t))_{t \ge 0}
\]
is a Markov process with generator
\[
\mathcal{L}=\frac{1}{2}  \frac{\partial^2}{\partial r^2}+ 
\frac{2n-1}{2r} \frac{\partial}{\partial r}+\frac{1}{2}r^2 \Delta_{\R^m}  .
\]
Consequently, the following equality in distribution holds:
 \[
\left( r(t) ,S(t) \right)_{t \ge 0} \overset{d}{=}  \left( r(t),\beta\left(\int_0^t  r^2(s)ds\right)\right)_{t \ge 0},
\]
where $(r(t))_{t \ge 0}$ is a $2n$-dimensional Bessel process  and $(\beta(t))_{t \ge 0}$ is a standard $m$-dimensional Brownian motion independent from $(r(t))_{t \ge 0}$.
\end{lemma}

\begin{theorem}
For $t>0$ and $x \in \mathbb R^{2n}$, and $\lambda \in \mathbb R^m$
\[
\mathbb{E}\left( e^{i\langle \lambda, S(t)\rangle} | B(t)=x\right)=\left(\frac{|\lambda| t}{\sinh |\lambda | t} \right)^{n} e^{-\frac{|x|^2}{2t}(|\lambda | t \coth |\lambda| t -1) },
\] 
\end{theorem}

\begin{remark}
For further informations about H-type groups and the study of associated horizontal heat kernels, including precise asymptotics and estimates we refer to \cite{MR2473130}, \cite{Eldredge1,Eldredge2} and \cite{HQLi2,HQLi3}. For formulas associated to more general 2-step Carnot groups we refer to  \cite{MR544727,Gaveau}  and \cite{garofalo2021heat} and the references therein.
\end{remark}
%
%
\section{Brownian winding functionals}
Brownian winding functionals are also stochastic integrals of one forms along Brownian paths. Compared to stochastic areas which reflect the sub-Riemannian structures of principal bundles over homogeneous spaces, they are tied to spherical components of Brownian motions in base spaces. When the unit sphere of the base space is a Lie group, the underlying winding form is (up to a time change) the corresponding Maurer-Cartan form. In order to illustrate this circle of ideas, we shall revisit the well-studied case of the complex plane and its celebrated Spitzer Theorem. Once we do, we shall extend the known results to both complex and quaternionic flat, projective and hyperbolic spaces.

\subsection{Complex windings}

In the punctured complex plane $\mathbb C \setminus \{ 0 \}$, consider the one-form
\[
\eta =\frac{x dy-ydx}{x^2+y^2}=\frac{1}{2i}\frac{\bar{z} dz-zd\bar{z}}{|z|^2}.
\]
Given a smooth path $\gamma: [0,+\infty) \to \mathbb{C} \setminus \{ 0 \}$ one has the polar representation
\[
\gamma(t)=  | \gamma (t) | \exp \left( i \left( \theta(0)+\int_{\gamma[0,t]} \eta \right)\right), \quad t \ge 0,
\]
where  $\theta(0)$ is such that $\gamma(0)=| \gamma(0)| \exp (i \theta (0))$. It is therefore natural to call $\eta$ the winding form around $z=0$ since the integral of a  path $\gamma$ along this form quantifies the angular motion of this path. 

The integral of the winding form along the paths of a two-dimensional Brownian motion  $B (t)=B_1(t)+iB_2(t)$  yields the Brownian winding functional:
\[
\zeta (t)=\int_{B[0,t]} \eta=\int_0^t \frac{B_1(s) dB_2(s)-B_2(s)dB_1(s)}{B_1(s)^2+B_2(s)^2}.
\]

The process $(\zeta (t))_{t \ge 0}$ is not a Markov process in its own natural filtration. However, if we consider the 3-dimensional process
 \[
 X(t)=(B_1(t),B_2(t),\zeta (t)),
 \]
 then $X(t)$ is solution of a stochastic differential equation
 \begin{align*}
 \begin{cases}
 dX_1(t)& =dB_1(t) \\
 dX_2(t)& =dB_2(t) \\
 dX_3(t)& =-\displaystyle\frac{X_2(t)}{ X_1(t)^2+X_2(t)^2} dB_1(t)+\frac{X_1(t)}{ X_1(t)^2+X_2(t)^2}  dB_2(t).
 \end{cases}
 \end{align*}
 As a consequence $X(t)$ is a Markov process with generator
 \[
 L=\frac{1}{2}(\mathbb{X}^2+\mathbb{Y}^2)
 \]
 where 
 \[
 \mathbb{X}=\frac{\partial}{\partial x}-\frac{y}{x^2+y^2} \frac{\partial}{\partial \zeta}
 \]
 \[
 \mathbb{Y}=\frac{\partial}{\partial y}+\frac{x}{x^2+y^2} \frac{\partial}{\partial \zeta}.
 \]

\begin{lemma}\label{skew-winding}
Let $r(t)=| B(t) |=\sqrt{ B_1(t)^2 +B_2(t)^2 }$, $t \ge 0$. Then, the couple
\[
(r(t) , \zeta(t))_{t \ge 0}
\]
is a Markov process with generator
\[
\mathcal{L}=\frac{1}{2}  \frac{\partial^2}{\partial r^2}+ \frac{1}{2r} \frac{\partial}{\partial r}+\frac{1}{2r^2}  \frac{\partial^2}{\partial \zeta^2}.
\]
Consequently, the following equality in distribution holds:
 \[
\left( r(t) ,\zeta(t) \right)_{t \ge 0} \overset{d}{=}  \left( r(t),\gamma\left(\int_0^t  \frac{1}{r(s)^2}ds\right)\right)_{t \ge 0},
\]
where $(r(t))_{t \ge 0}$  is a 2-dimensional Bessel process and  $(\gamma (t))_{t \ge 0}$ is a standard  Brownian motion independent from $(r(t))_{t \ge 0}$.

\end{lemma}

\begin{proof}
The proof is close to that of Lemma \ref{rad heisenberg}. From It\^o's formula, we have
\begin{align*}
dr(t) &=\frac{dt}{2 r(t)}+\frac{B_1(t) dB_1(t)+B_2(t) dB_2(t)}{ \sqrt{ B_1(t)^2 +B_2(t)^2 }} \\
d\zeta(t)&=\frac{1}{r(t)} \frac{B_1(t) dB_2(t)-B_2(t) dB_1(t)}{ \sqrt{ B_1(t)^2 +B_2(t)^2 }}.
\end{align*}
Since the two processes
\begin{align*}
\beta(t)=\int_0^t \frac{B_1(s) dB_1(s)+B_2(s) dB_2(s)}{ \sqrt{ B_1(s)^2 +B_2(s)^2 }} \\
\gamma (t)=\int_0^t \frac{B_1(s) dB_2(s)-B_2(s) dB_1(s)}{ \sqrt{ B_1(s)^2 +B_2(s)^2 }},
\end{align*}
are two independent Brownian motions, the conclusion easily follows.
\end{proof}

The analogue of L\'evy area formula is due to Yor \cite{YorWinding} and reads:

\begin{theorem}[Yor's winding formula]\label{Yorwinding}
For $t>0$ and $x \in \R^2$, and $\lambda \ge 0$
\[
\mathbb{E}\left( e^{i\lambda \zeta(t)} | r(t)=|x|\right)=\frac{I_{\lambda}(| x |\rho/t)}{I_{0}(| x| \rho/t)}.
\] 
where $I_{\lambda}$ stands for the modified Bessel function of the first kind with index $\lambda$ and $\rho=| B(0)|$.
\end{theorem}

\begin{proof}
From Lemma \ref{skew-winding}, 
\begin{align*}
\mathbb{E}\left( e^{i\lambda \zeta(t)} | r(t)=| x |\right)& =\mathbb{E}\left( e^{i \lambda \gamma_{\int_0^t  \frac{1}{r(s)^2}ds}} | r(t)=| x |\right) \\
 &=\mathbb{E}\left( e^{-\frac{\lambda^2}{2} \int_0^t  \frac{1}{r(s)^2}ds} | r(t)=| x |\right).
\end{align*}
Note that $(r(t))_{t \ge 0}$ is a Bessel diffusion with dimension 2. Therefore, using Example \ref{Hartman-Watson}, one has
\[
\mathbb{E}\left( e^{-\frac{\lambda^2}{2} \int_0^t  \frac{1}{r(s)^2}ds} | r(t)=| x |\right) =\frac{I_{\lambda}(| x | \rho/t)}{I_{0}(| x | \rho/t)},
\]
as desired.
\end{proof}

\begin{remark}
The inverse Fourier  transform of the function $\lambda \to \frac{I_{\lambda}(| x |\rho/t)}{I_{0}(| x| \rho/t)} $ can be expressed in terms of the unwrapped von Mises distribution, which is a distribution closely related to the Hartman-Watson distribution, see \cite{YorWinding}.
\end{remark}

We are now in position to prove Spitzer theorem \cite{Spitzer} concerning the asymptotic windings of the planar Brownian motion.

\begin{theorem}[Spitzer theorem]
When $t \to +\infty$, in distribution
\[
\frac{ 2 }{\ln t} \zeta (t) \to C_1
\]
where $C_1$ is a Cauchy distribution with parameter 1, i.e. $C_1$ is a random variable with density $1/[\pi (1+ x^2)]$.
\end{theorem}

\begin{proof}
According to Section \ref{Bessel section}, the density of $r(t)$ is given by
\begin{align*}
\frac{1}{t} r \exp \left(-\frac{\rho^2+r^2}{2t} \right) I_0 \left( \frac{ \rho r }{t} \right).
\end{align*}
Therefore, from the proof of Theorem \ref{Yorwinding}, one deduces
\begin{align*}
\mathbb{E}\left( e^{i\lambda \zeta(t)} \right)
& =\frac{1}{t} \int_0^{+\infty} r \exp \left(-\frac{\rho^2+r^2}{2t} \right) I_\lambda  \left( \frac{ \rho r }{t} \right)dr
\\& =\exp \left(-\frac{\rho^2}{2t} \right)\int_0^{+\infty} r \exp \left(-\frac{r^2}{2} \right) I_\lambda  
\left( \frac{ \rho r }{\sqrt{t}} \right)dr.
\end{align*}
Using the expansion
\[
I_{\lambda} (x)=\sum_{m=0}^{+\infty} \frac{1}{m! \Gamma(m+\lambda+1)} \left( \frac{x}{2} \right)^{2m+\lambda},
\]
we see that
\[
I_{\frac{2 \lambda}{\ln t}}  \left( \frac{ \rho r }{\sqrt{t}} \right) \sim_{t \to +\infty} \left( \frac{ \rho r }{2\sqrt{t}} \right)^{\frac{2\lambda}{\ln t}} \sim_{t \to +\infty} e^{-\lambda},
\]
whence
\[
\lim_{t \to +\infty} \mathbb{E}\left( e^{i2\lambda \frac{\zeta(t)}{\ln t}} \right)=e^{-\lambda}
\]
as claimed. 
\end{proof}

\begin{remark}
The Yor's winding formula admits many other consequences besides Spitzer's theorem. In particular, one can infer from it the law of winding numbers of the Brownian loop, which is a Brownian motion started at $x  \neq 0$ and conditioned to come back to $x$ at time 1. Indeed, one has
\[
\mathbb{E}\left( e^{i\lambda \zeta(1)} | |B(1)|=|x|\right)=\frac{I_{\lambda}(| x |^2)}{I_{0}(| x|^2)}.
\] 
After \cite{YorWinding}, it follows then that 
\[
\mathbb{E}\left( e^{i\lambda \zeta(1)} | B(1)=x\right)=e^{-|x|^2}\sum_{n \in \mathbb{Z}} I_{|n+\lambda|}(|x|^2).
\] 
So, the distribution of  $\zeta(1)$ conditionally to $B(1)=x$ can be computed using a Fourier inversion formula and we obtain that for $n \in \mathbb{Z}\setminus \{ 0 \}$
\begin{align*}
\mathbb{P}( \zeta(1) =2\pi n )=e^{-|x|^2} ( \Phi_{|x|^2} ((2n-1)\pi) -\Phi_{|x|^2} ((2n+1)\pi))
\end{align*}
where
\[
\Phi_{r} (u)=\frac{u}{\pi} \int_0^{+\infty} e^{- r \cosh (t) }\frac{dt}{t^2+u^2}.
\]
For $n=0$ one gets
\[
\mathbb{P}( \zeta(1) =0 )=1-2e^{-|x|^2} \Phi_{|x|^2}(\pi).
\]
\end{remark}

\subsection{Quaternionic windings}
Let $\mathbb{H}$ be the quaternionic field 
\[
\mathbb{H}=\{q=t+xI+yJ+zK, (t,x,y,z)\in\R^4\},
\]
where  $I,J,K \in \mathfrak{su}(2)$ are given by
\[
I=\left(
\begin{array}{ll}
i & 0 \\
0&-i 
\end{array}
\right), \quad 
J= \left(
\begin{array}{ll}
0 & 1 \\
-1 &0 
\end{array}
\right), \quad 
K= \left(
\begin{array}{ll}
0 & i \\
i &0 
\end{array}
\right).
\]
Then the quaternionic norm is given by $|q |^2 =t^2 +x^2+y^2+z^2$ and the set of unit quaternions is identified with the Lie group $\mathbf{SU}(2)$. Now, consider a smooth path $\gamma: [0,+\infty) \to \mathbb{H} \setminus \{ 0 \}$  and write its polar decomposition: 
\[
\gamma (t) = | \gamma (t) | \Theta (t), \quad t \geq 0,
\]
with $\Theta (t) \in \mathbf{SU}(2)$. Then,

\begin{definition}
The winding path $(\theta(t))_{t \geq 0} \in \mathfrak{su}(2)$ along $\gamma$ is defined by:
\[
\theta(t) = \int_0^t \Theta (s)^{-1}  d\Theta (s).
\]
The quaternionic winding form is the $\mathfrak{su}(2)$-valued one-form $\eta$ such that
\[
\theta(t)=\int_{\gamma[0,t]} \eta.
\]
Equivalently, 
\begin{equation*}
\theta(t)= \int_{\Theta [0,t]} \omega, 
\end{equation*}
where $\omega$ is the Maurer-Cartan form of the Lie algebra $\mathfrak{su}(2)$. 
\end{definition}
In order to study the stochastic winding in $\mathbb{H}$, we need first to compute $\eta$ in real coordinates $(t,x,y,z)$. To this end, we write 
\begin{equation*}
\Theta = \left(\begin{array}{ll}
\theta_1 & \theta_2 \\
-\overline{\theta_2} & \overline{\theta_1}
\end{array}\right),
\end{equation*}
where 
\begin{equation*}
\theta_1 = \frac{\gamma_0+i\gamma_1}{|\gamma|}, \quad \theta_2 = \frac{\gamma_2+i\gamma_3}{|\gamma|}, \quad \gamma = (\gamma_i)_{i=0,1,2,3}. 
\end{equation*}
Since $\Theta$ is unitary and has determinant one, then 
\begin{equation*}
\Theta^{-1} = \left(\begin{array}{ll}
\overline{\theta_1} &  -\theta_2 \\
\overline{\theta_2} & \theta_1
\end{array}\right),
\end{equation*}
so that 
\begin{equation*}
\Theta^{-1}\overset{\cdot}{\Theta} = \left(\begin{array}{ll}
\overline{\theta_1}\overset{\cdot}{\theta_1} + \theta_2 \overline{\overset{\cdot}{\theta_2}}  &   \overline{\theta_1}\overset{\cdot}{\theta_2} - \theta_2 \overline{\overset{\cdot}{\theta_1}} \\
 \overline{\theta_2} \overset{\cdot}{\theta_1} - \theta_1 \overline{\overset{\cdot}{\theta_2}}  &  \overline{\theta_2}\overset{\cdot}{\theta_2} + \theta_1 \overline{\overset{\cdot}{\theta_1}}  
\end{array}\right).
\end{equation*}
After straightforward computations, we end up with the following expression of $\eta = \eta_1 I + \eta_2  J + \eta_3 K$: 
\begin{eqnarray*}
\eta_1 & = & \frac{tdx- xdt + zdy - ydz}{|q|^2} \\ 
\eta_2 &= & \frac{tdy- ydt + xdz - zdx}{|q|^2} \\ 
\eta_3 & = & \frac{tdz- zdt + ydx - xdy}{|q|^2} .
\end{eqnarray*} 

Note that $\eta$ may be more concisely written in quaternionic coordinates as: 
\begin{equation*}
\eta = \frac{1}{2} \left(\frac{\overline{q}dq - \overline{dq}q}{|q|^2}\right) = \frac{1}{|q|^2}\mathrm{Im}(\overline{q}dq),
\end{equation*}
where 
\begin{equation*}
\overline{q} = t - xI - yJ - z K, \quad dq = dt + dx I + dy J + dz K
\end{equation*}
are respectively the quaternionic conjugate and the differential of $q$.

These computations naturally lead to the following definition : 
\begin{definition}
The winding number of a quaternionic Brownian motion $W=W_0+W_1 I+W_2J + W_3K $, not started from 0, is defined by the Stratonovich stochastic line integral:
\begin{equation*}
\zeta (t) := \int_{W(0,t]} \eta, \quad t \geq 0.
\end{equation*}
\end{definition}
The study of $\zeta$ is based on the following lemma which is a well-known consequence of the skew-product decomposition of Euclidean Brownian motions (see Example \ref{BM on warped product}).

\begin{lemma}
Let  $W=W_0+W_1 I+W_2J + W_3K $ be a quaternionic Brownian motion not started from 0. There exists a 
 Bessel process $(R(t))_{t \ge 0}$ of dimension four (or equivalently index one) and a $\mathbf{SU}(2)$-valued Brownian motion $(\Theta (t))_{t \geq 0}$ independent from the process  $(R(t))_{t \ge 0}$ such that 
 \[
 W(t)=R(t) \Theta (A(t)),
 \]
 where 
 \begin{equation*}
A(t):= \int_0^t \frac{ds}{R^2 (s)}.
\end{equation*}
\end{lemma}
As a consequence of the previous lemma, one readily has:
\[
\zeta (t)= \int_0^t \Theta (A(s))^{-1}  \circ d\Theta (A(s))=\int_0^{A(t)} \Theta (s)^{-1}  \circ d\Theta (s).
\]
But
\[
B (t) := \int_0^{t} \Theta (s)^{-1}  \circ d\Theta (s), \quad t \geq 0,
\]
is a three-dimensional ($\mathfrak{su}(2)$-valued) Euclidean Brownian motion. Consequently, the quaternionic winding process $\zeta$ has the same distribution as: 
\begin{equation*}
(B^1 (A(t)), B^2 (A(t)), B^3 (A(t)))_{t \geq 0}.
\end{equation*}

In particular, the characteristic function of the winding process at time $t$ is given by: 
\begin{equation*}
\mathbb{E}_{\rho}[e^{i\langle \lambda, \zeta (t)\rangle}] = \mathbb{E}_{\rho}[e^{-|\lambda|^2 A(t)/2}], \quad \rho := |W_0| > 0,
\end{equation*}
for any $\lambda \in \mathbb{R}^3$. But according to Example \ref{Hartman-Watson}, one has 
\begin{equation*}
\mathbb{E}_{\rho}[e^{-|\lambda|^2 A(t)/2} | R(t) = r] = \frac{I_{\sqrt{1+|\lambda|^2}}(r\rho/t)}{I_{1}(r\rho/t)}.
\end{equation*}
 Appealing further to the semigroup density of the Bessel process (Section \ref{Bessel section}), it follows that:
\begin{align}\label{characteristic winding}
\mathbb{E}_{\rho}[e^{i\langle \lambda, \zeta (t)\rangle}] &= \frac{e^{-\rho^2/(2t)}}{t\rho} \int_0^{\infty} I_{\sqrt{1+|\lambda|^2}}\left(\frac{r\rho}{t}\right) e^{-r^2/(2t)} r^2 dr. 
\end{align}

In particular, it is worth noting that the quaternionic field has  a regularizing effect on the (spectral) variable $\lambda$ compared to the complex field.

Using \eqref{characteristic winding}, we are now able to determine the limiting behavior of $\zeta (t)$ as $t \rightarrow \infty$.

\begin{teor}[Quaternionic Spitzer theorem]\label{T1}
The following convergence in distribution holds:
\begin{equation*}
\lim_{t \rightarrow +\infty}\frac{2}{\sqrt{\log t}} \zeta (t) = \mathcal{N}(0, {\rm I}_3)
\end{equation*}
where ${\rm I}_3$ is the $3 \times 3$ identity matrix. 
\end{teor}

\begin{proof}
We shall give two proofs of this limiting theorem. Of course, the first proof relies on the explicit representation \eqref{characteristic winding} while the second one on Girsanov's theorem and Yor's transform. The reason is that the second proof is easier to generalize to the curved geometric settings studied afterwards.

\underline{\textbf{Proof 1}}: Performing the variable change $r \mapsto \sqrt{t} r$ in the  integral \eqref{characteristic winding}, we get 
\begin{equation*}
\mathbb{E}_{\rho}[e^{i\langle \lambda, \zeta (t)\rangle}] = \frac{e^{-\rho^2/(2t)}}{\rho} \int_0^{\infty} \sqrt{t} I_{\sqrt{1+|\lambda|^2}}\left(\frac{r\rho}{\sqrt{t}}\right) e^{-r^2/2} r^2 dr. 
\end{equation*}
Expanding further the Bessel function: 
\begin{equation*}
I_{\sqrt{1+|\lambda|^2}}\left(\frac{r\rho}{\sqrt{t}}\right) = \sum_{j \geq 0}\frac{1}{\Gamma(j+1+\sqrt{1+|\lambda|^2})j!} \left(\frac{r\rho}{\sqrt{t}}\right)^{2j+\sqrt{1+|\lambda|^2}},
\end{equation*}
we infer that the large-time behavior of $\zeta (t)$ is governed by the lowest-order term. Finally, rescaling $\lambda$ by $\sqrt{\log t}/\sqrt{2}$, then 
\begin{equation*}
\lim_{t \rightarrow +\infty} e^{-(\log(t)/2)\left(\sqrt{1+2|\lambda|^2/\log(t)} - 1\right)} = e^{-|\lambda|^2/2}
\end{equation*}
whence the result follows. 

\underline{\textbf{Proof 2}}: We can also derive the limiting behavior of $\zeta$ by using Girsanov's theorem as well. To proceed, recall the stochastic differential equation satisfied by the Bessel process $R$: 
\[
dR(t)=\frac3{2R(t)}dt+d\xi(t), \quad R (0) = \rho > 0,
\]
where $(\xi(t))_{t \geq 0}$ is a one-dimensional standard Brownian motion. Then, letting $\mu=\sqrt{|\lambda|^2+1}-1$, we can consider the martingale
\[
D_t^{(\mu)}=\exp\bigg(\mu\int_0^t\frac{1}{R (s)}d\xi(s)-\frac{\mu^2}{2}\int_0^t\frac{1}{R (s)^2}ds \bigg).
\]
By It\^o's formula, we have
\[
D_t^{(\mu)}=\left(\frac{R(t)}{\rho}\right)^\mu\, \exp\left(-\left(\frac{1}{2}\mu^2+\mu\right)\int_0^t\frac{1}{R (s)^2}ds \right)
\]
Hence, Girsanov's theorem shows that $(R(t))_{t \geq 0}$ is  a Bessel process of dimension $4\mu+3$ under the probability measure $\mathbb{P}^{(\mu)}$ with Radon-Nikodym density $D_t^{(\mu)}$ and 
\[
\mathbb{E}_{\rho}[e^{i\langle \lambda, \zeta (t)\rangle}] = \mathbb E_{\rho}\left(e^{-\frac{|\lambda|^2}{2}A(t)} \right)=(\rho)^\mu\,\mathbb E^{(\mu)}_{\rho}\left(\frac{1}{(R(t))^{\mu}}\right). 
\]
Setting 
\begin{equation*}
\lambda(t):=\frac{\sqrt{2}\lambda}{\sqrt{\log t}}, \qquad \mu(t):=\sqrt{|\lambda(t)|^2+1}-1=\sqrt{\frac{2|\lambda|^2}{\log t}+1}-1,
\end{equation*}
it follows that 
\[
\lim_{t\to\infty}\mathbb E_{\rho}\left(e^{-\frac{|\lambda|^2}{\log t}A(t)} \right)=\lim_{t\to\infty}(\rho)^{\mu(t)} \mathbb E^{(\mu(t))}_{\rho}\left(\frac{1}{(R(t))^{\mu(t)}}\right) = e^{-|\lambda|^2/2},
\]
where the last equality follows from the scaling property of $(R(t))_{t \geq 0}$. 
\end{proof}

\subsection{Windings in the complex projective line}

Brownian windings in Riemann surfaces can be meaningfully defined and studied as well, see \cite{MR1834236}.  The latter are the connected complex manifolds with complex dimension one (equivalently the real dimension is two). By the 
uniformisation theorem, the only simply-connected and compact Riemann surface of genus zero is the Riemann sphere $S^2$ or equivalently (from the complex algebraic point of view) the complex projective line $\mathbb{C}P^1$.  
%
%

In order to study the winding process there, we appeal to the well known Hopf fibration. More precisely, let us consider the 3-dimensional sphere $\bS^{3}$: 
\[
\bS^{3}=\lbrace z=(z_1,z_{2})\in \mathbb{C}^{2}, | z | =1\rbrace.
\]
Note that this manifold is isometric to the Lie group 
\[
\mathbf{SU}(2) = \left\{ \left(\begin{array}{ll}
z_1 & z_2 \\
-\overline{z_2} & \overline{z_1}
\end{array}\right) , (z_1,z_{2})\in \bS^{3}  \right\}
\]
equipped with its bi-invariant metric coming from the Killing form.
Now, $\mathbf{U}(1) \simeq \mathbb{S}^1$ acts isometrically on $\bS^{3}$ by left multiplication:
\[
e^{i\theta} \cdot  ( z_1, z_2)=(e^{i\theta} z_1, e^{i\theta} z_2)
\]
and the quotient space $\bS^{3}/ \mathbf{U}(1)$ is, by definition,  the complex projective line $\mathbb{C} P^1$. The  Riemannian metric on  $\mathbb{C} P^1$ is such that the projection map  $\bS^{3}\to \mathbb{C} P^1$ is a Riemannian submersion with totally geodesic fibers isometric to $\mathbb{S}^1$ (see Section \ref{Section Riemannian submersion} for the definition of Riemannian submersion). In this way, one obtains the Hopf fibration
\[
\mathbf{U}(1)\to \bS^{3}\to \mathbb{C} P^1
\]
 which will be studied (for odd dimensional spheres) in great details in Chapter 4. Locally, $\mathbb{C} P^1$ is parametrised using the complex inaffine coordinates:
\[
w=z_2^{-1} z_1, \quad z=(z_1,z_2) \in \bS^{3},
\]
which allows to identify $\mathbb{C} P^1$ with the one-point compactification $\mathbb{C}\cup \{ \infty \}$. 

Now, if $\gamma: [0,+\infty) \to \mathbb{C}P^1 \setminus \{ 0 , \infty \}$ is a $C^1$-path, one can consider its polar decomposition:
\[
\gamma (t) = | \gamma (t) | e^{i\theta (t)}
\]
with $\theta (t) \in \mathbb{R}$. The  winding form on $\mathbb{C} P^1$ is therefore the  one-form $\eta$ such that:
\[
\theta(t)=\int_{\gamma[0,t]} \eta=\frac{1}{2i}\int_0^t \frac{ \overline{\gamma}(s)d\gamma(s)-d\overline{\gamma}(s)\gamma(s)}{|\gamma(s)|^2}, \quad t \geq 0.
\]

Consequently, it is natural to make the following definition: 
\begin{definition}
The winding process of the complex projective line $\mathbb{C} P^1$ is defined by 
\begin{equation*}
\zeta (t) := \int_{w[0,t]} \eta = \frac{1}{2i}\int_0^t \frac{ \overline{w}(t)dw(t)-d\overline{w}(t)w(t)}{|w(t)|^2}, \quad t \geq 0,
\end{equation*}
where $(w(t))_{t \geq 0}$ is the Brownian motion on $\mathbb{C}P^1$. 
\end{definition}

The symmetric space $\mathbb{C}P^1$ is isometric to the two-dimensional Euclidean sphere with radius $\frac{1}{2}$. Moreover, the Riemannian distance from 0 to $w\in \mathbb{C} P^1$ is given by the formula:
\[
r=\arctan |w|.
\]
Therefore, the generator of the Brownian motion on  $\mathbb{C}P^1$ in spherical coordinates $(r,\phi)$,  $z=\tan(r)e^{i\phi}$,  reads:
\[
\frac{1}{2} \left( \frac{\partial^2}{\partial r^2} + 2 \cot 2r \frac{\partial}{\partial r}  +\frac{4}{\sin^2 2r}  \frac{\partial^2}{\partial \phi^2}\right), \quad r\in[0,\pi/2],\  \phi \in \mathbb{R},
\]
where $r$ parametrizes the Riemannian distance from 0 in $\mathbb{C}P^1$. This shows that  the winding process of the Brownian motion in $\mathbb{C}P^1$ is given by
\[
\zeta(t)=\beta\left(\int_0^t \frac{4ds}{\sin^2 2r(s)}\right), \quad t \geq 0,
\]
where $r(t)$ is the Jacobi diffusion started at $r_0\in (0,\pi /2)$ with generator 
\[
\frac{1}{2} \left( \frac{\partial^2}{\partial r^2} + 2 \cot 2r \frac{\partial}{\partial r}  \right)
\]
and $\beta$ is a Brownian motion independent from $r$.

\begin{theorem}\label{windingCP}
When $t \to \infty$, in distribution we have
\[
\frac{\zeta(t)}{t} \to \mathcal{C}_2,
\]
where $\mathcal{C}_2$ is a Cauchy distribution with parameter 2.
\end{theorem}

\begin{proof}
Let $\lambda >0$. We have 
\[
\mathbb{E}\left( e^{i \zeta \phi(t)} \right)=\mathbb{E}\left( e^{-\frac{\lambda^2}{2} \int_0^t \frac{4ds}{\sin^2 2r(s)}} \right)=e^{-2\lambda^2t} \mathbb{E}\left( e^{-2\lambda^2 \int_0^t \cot^2 2r(s) ds} \right).
\]
The process $r$ is the unique strong solution of a stochastic differential equation
\[
r(t)=r_0+\int_0^t \cot 2r(s) ds +\gamma(t),
\]
where $\gamma$ is a Brownian motion. For $\lambda \ge 0$, let us consider the local martingale:
\begin{align*}
D_t^\lambda =\exp \left( 2\lambda \int_0^t \cot 2r(s) d\gamma(s) -2\lambda^2 \int_0^t \cot^2 2r(s) ds \right).
\end{align*}
From It\^o's formula, we compute
\[
D_t^\lambda =e^{2\lambda t}(\sin 2r(t))^{\lambda} \exp\left(-2\lambda^2 \int_0^t \cot^2 2r(s) ds \right).
\]
Moreover, since $D_t\le e^{2\lambda t}$, we know that $D$ is a martingale and as such, we may introduce the following probability measure $\mathbb{P}^\lambda$:
\[
\mathbb{P}_{/ \mathcal{F}_t} ^\lambda=D_t \mathbb{P}_{/ \mathcal{F}_t}=(\sin 2r_0)^{-\lambda} e^{2\lambda t} (\sin 2r(t))^\lambda e^{- 2\lambda^2 \int_0^t \cot^2 2r(s)ds} \mathbb{P}_{/ \mathcal{F}_t},
\]
where $\mathcal{F}$ is the natural filtration of $r$. Denoting $\mathbb{E}^{\lambda}$ the corresponding expectation, we get:
\[
\mathbb{E}\left( e^{i \lambda \phi(t)} \right)=(\sin 2r_0)^\lambda e^{-2(\lambda^2+\lambda) t}\mathbb E^{\lambda}\left((\sin2r(t)) ^{-\lambda}  \right). 
\]
By Girsanov's theorem we know that the process
\[
\beta (t)=\gamma(t)-2\lambda\int_0^t \cot 2r(s)ds 
\]
is a Brownian motion under $\mathbb{P}^\lambda$. As a matter of fact, 

\[
d r(t)= d\beta(t)+(2\lambda+1)\cot 2r(t)dt,
\] 
 is a Jacobi diffusion under $\mathbb{P}^{\lambda}$ with generator
\[
\mathcal{L}^{\lambda,\lambda}=\frac{1}{2} \frac{\partial^2}{\partial r^2}+\left(\left(\lambda+\frac{1}{2}\right)\cot r-\left(\lambda+\frac{1}{2}\right) \tan r\right)\frac{\partial}{\partial r}.
\]
Writing
\[
\mathbb{E}\left( e^{i \lambda \frac{\phi(t)}{t}} \right)=(\sin 2r_0)^{\frac{\lambda}{t}} e^{-2\left(\frac{\lambda^2}{t^2}+\frac{\lambda}{t}\right) t}\mathbb E^{\frac{\lambda}{t}}\left((\sin2r(t)) ^{-\frac{\lambda}{t}}  \right) 
\]
and using the expression of the Jacobi semi-group density $q_t^{\lambda,\lambda}(r_0,r)$ given in the Appendix 2 (Chapter \ref{sec-appendix-2}), we arrive at the limit
\[
\lim_{t \to \infty} \mathbb{E}\left( e^{i \lambda \frac{\zeta(t)}{t}} \right)=e^{-2\lambda}.
\]
\end{proof}

\begin{remark}
We refer to \cite{MR1834236} for further results about the Brownian windings on compact Riemann surfaces.
\end{remark}

\subsection{Windings in the complex hyperbolic space}

The 3-dimensional complex anti-de Sitter space $\mathbf{AdS}^3(\mathbb C)$ which is the quadric defined by:
\[
\mathbf{AdS}^3(\mathbb C)=\lbrace z=(z_1,z_{2})\in \mathbb{C}^{2},| z_1 |^2-|z_2|^2 =1\rbrace.
\]
It is equipped with the Lorentz metric inherited from the metric with real signature $(2,2)$ on $\mathbb{C}^{2} \sim \mathbb{R}^4$, 
see Section \ref{Section complex anti-de Sitter} for more details.
It is also isometric to the real special linear group 
\[
\mathbf{SL}(2,\R)= \left\{ \left(\begin{array}{ll}
a & b \\
c & d
\end{array}\right), a,b,c,d \in \R, \, ad-bc=1 \right\}
\]
through the mapping 
\begin{equation*}
z_1= \frac{1}{2} (a+ib)+\frac{1}{2} (d+ic), \quad 
z_2=\frac{1}{2} (d+ic) -\frac{1}{2} (a+ib),
\end{equation*}
 $\mathbf{SL}(2,\R)$ being equipped with its bi-invariant Lorentz metric induced from the Killing form, see \cite{Thesis} and \cite{Wang1} for more details.

The circle group $\mathbf{U}(1) \simeq \mathbb{S}^1$  acts isometrically on $\mathbf{AdS}^3(\mathbb C)$ by left multiplication:
\[
e^{i\theta} \cdot  ( z_1, z_2)=(e^{i\theta} z_1, e^{i\theta} z_2)
\]
and the quotient space $\mathbf{AdS}^3(\mathbb C)/ \mathbf{U}(1)$ is, by definition,  the complex hyperbolic line $\mathbb{C} H^1$. In the 
$\mathbf{SL}(2,\R)$ picture, this quotient is nothing else but the Poincar\'e half-plane $\mathbf{SL}(2,\R)/\mathbf{SO}(2)$ which is isometric to the Poincar\'e disc.

Geometrically, the  K\"ahler metric on  $\mathbb{C} H^1$ is such that the projection map  $\mathbf{AdS}^3(\mathbb C)\to \mathbb{C} H^1$ is a pseudo-Riemannian submersion with totally geodesic fibers isometric to $\mathbb{S}^1$ (see Chapter \ref{sec-BM-RM} for the definition of Riemannian submersion). Note that the corresponding  fibration
\[
\mathbf{U}(1)\to \mathbf{AdS}^3(\mathbb C) \to \mathbb{C} H^1
\]
is a special  anti-de Sitter fibration which will be studied in great details in Chapter \ref{sec-hBM}.  One can parametrize $\mathbb{C} H^1$ using the complex inhomogeneous coordinate:
\[
w=z_2^{-1} z_1, \quad z=(z_1,z_2) \in \mathbf{AdS}^3(\mathbb C),
\]
identifying $\mathbb{C} H^1$ with the Poincar\'e disc. 
Accordingly, the Riemannian distance from 0 to $w$ is given by the formula:
\[
r=\arctanh |w|.
\]

Mimicking the positively-curved setting, one considers a $C^1$-path $\gamma: [0,+\infty) \to \mathbb{C}H^1 \setminus \{ 0 , \infty \}$   and makes use of its polar decomposition:
\[
\gamma (t) = | \gamma (t) | e^{i\theta (t)}
\]
with $\theta (t) \in \mathbb{R}$. The  winding form on $\mathbb{C} H^1$ is therefore the one-form $\eta$ such that:
\[
\theta(t)=\int_{\gamma[0,t]} \eta=\frac{1}{2i}\int_0^t \frac{ \overline{\gamma}(s)d\gamma(s)-d\overline{\gamma}(s)\gamma(s)}{|\gamma(s)|^2}, \quad t \geq 0
\]

and leads to the following definition: 
\begin{definition}
The winding process of the complex hyperbolic space $\mathbb{C} H^1$ is defined by 
\begin{equation*}
\zeta (t) := \int_{w[0,t]} \eta = \frac{1}{2i}\int_0^t \frac{ \overline{w}(t)dw(t)-d\overline{w}(t)w(t)}{|w(t)|^2}, \quad t \geq 0,
\end{equation*}
where $(w(t))_{t \geq 0}$ is the Brownian motion on $\mathbb{C}H^1$. 
\end{definition}

Moreover, the generator of the Brownian motion on  $\mathbb{C}H^1$  in spherical coordinates $(r,\phi)$ reads 
\[
\frac{1}{2} \left( \frac{\partial^2}{\partial r^2} + 2 \coth 2r \frac{\partial}{\partial r}  +\frac{4}{\sinh^2 2r}  \frac{\partial^2}{\partial \phi^2}\right), 
\]
where we recall that $r$ parametrises the Riemannian distance from $0$. This shows that  the winding  process of the Brownian motion on $\mathbb{C}H^1$ is given by
\[
\zeta(t)=\beta\left(\int_0^t \frac{4ds}{\sinh^2 2r(s)}\right),
\]
where $r(t)$ is the hyperbolic Jacobi diffusion started at $r_0\in (0,\infty)$ with generator 
\[
\frac{1}{2} \left( \frac{\partial^2}{\partial r^2} + 2 \coth 2r \frac{\partial}{\partial r}  \right)
\]
and $\beta$ is a Brownian motion independent from $r$.

\begin{theorem}\label{windingCH1}
When $t \to \infty$, in distribution we have
\[
{\zeta(t)}\to \mathcal{C}_{\ln \coth r_0},
\]
where $\mathcal{C}_{\ln \coth r_0}, r_0 > 0,$ is a Cauchy distribution with parameter $\ln \coth r_0$.
\end{theorem}
\begin{proof}
The process $r$ solves the stochastic differential equation
\[
r(t)=r_0+\int_0^t \coth 2r(s) ds +\gamma(t),
\]
where $\gamma$ is a standard Brownian motion.  Let $\lambda>0$ and consider the local martingale
\[
D_t^\lambda=\exp \left( \lambda \int_0^t \frac{2d\gamma(s)}{\sinh 2r(s)} -\frac{\lambda^2}{2} \int_0^t \frac{4ds}{\sinh^2 2r(s)}\right).
\]
Using It\^o's formula, we rewrite it as:
\[
D_t^\lambda =\left(\frac{\tanh r(t)}{\tanh r_0}\right)^\lambda \exp \left( -\frac{\lambda^2}{2} \int_0^t \frac{4ds}{\sinh^2 2r(s)}\right),
\]
which is clearly uniformly bounded and in turn, is a martingale. 
Let $\mathcal{F}$ be the natural filtration of $r$ and consider the new  probability measure $\mathbb{P}^\lambda$:
\[
\mathbb{P}_{/ \mathcal{F}_t} ^\lambda=D^\lambda_t \mathbb{P}_{/ \mathcal{F}_t}.
\]
It follows that
\begin{equation}\label{eq-E-tanh}
\mathbb{E} \left( e^{i \lambda \zeta(t)} \right)=\mathbb{E}^\lambda \left( \left(\frac{\tanh r_0}{\tanh r(t)}\right)^\lambda\right). 
\end{equation}
But, the process
\[
\beta(t)=\gamma(t)-\int_0^t \frac{2\lambda ds}{\sinh 2r(s)}, \quad t \geq 0
\]
is a $\mathbb{P}^\lambda$-Brownian motion while a standard comparison argument shows that $r(t)\to +\infty$ almost surely as $t\to+\infty$. As a result, 
\begin{equation*}
 \lim_{t \to \infty} \tanh r(t)=1,   
\end{equation*}
almost surely so that \eqref{eq-E-tanh} finally yields:
\[
\lim_{t \to \infty} \mathbb{E} \left( e^{i \lambda \zeta(t)} \right)=(\tanh r_0)^\lambda.
\]
\end{proof}

\subsection{Windings  in the quaternionic projective space}

As previously, $\mathbb{H}$ is the quaternionic field and $I,J,K \in \mathbf{SU}(2)$ are the Pauli matrices. Define the quaternionic sphere $\bS^{7}$ by: 
\[
\bS^{7}=\lbrace q=(q_1,q_{2})\in \mathbb{H}^{2}, | q | =1\rbrace.
\]
Then, $\mathbf{SU}(2)$ isometrically acts on $\bS^{7}$ by left multiplication:
\[
q \cdot (q_1,q_2)=(qq_1, qq_2)
\]
and the quotient space $\bS^{7}/ \mathbf{SU}(2)$ is the quaternionic projective line $\bH P^1$. The quaternionic K\"ahler metric on  $\bH P^1$ is such that the projection map  $\bS^{7}\to \bH P^1$ is a Riemannian submersion with totally geodesic fibers isometric to $\mathbf{SU}(2)$. Note that the corresponding  fibration
\[
\mathbf{SU}(2)\to \bS^{7}\to \bH P^1
\]
is called the quaternionic Hopf fibration.  One can locally parametrize $\bH P^1$ using the quaternionic  inhomogeneous coordinate:
\[
w=q_2^{-1} q_1, \quad q=(q_1,q_2) \in \bS^{7}
\]
with the convention that $0^{-1} q =\infty$. This allows to identify $\bH P^1$ with the one-point compactification $\mathbb{H}\cup \{ \infty \}$. This identification will be in force in what follows. In affine coordinates, the Riemannian distance from 0 is given by the formula:
\[
r=\arctan |w|.
\]

If $\gamma: [0,+\infty) \to \mathbb{H}P^1 \setminus \{ 0 , \infty \}$ is a $C^1$-path, one can consider its polar decomposition:
\[
\gamma (t) = | \gamma (t) | \Theta (t)
\]
with $\Theta (t) \in \mathbf{SU}(2)$ and define the winding path $(\theta (t))_{t\ge0}$ in $\mathfrak{su}(2)$ as
\[
\theta(t) = \int_0^t \Theta (s)^{-1}  d\Theta (s).
\]
The quaternionic winding form on $\bH P^1$ is therefore the $\mathfrak{su}(2)$-valued one-form $\eta$ such that:
\[
\theta(t)=\int_{\gamma[0,t]} \eta=\frac{1}{2}\int_0^t \frac{ \overline{\gamma}(s)d\gamma(s)-d\overline{\gamma}(s)\gamma(s)}{|\gamma(s)|^2}, \quad t \geq 0.
\]

Consequently, it is natural to make the following definition: 
\begin{definition}
The winding process of the quaternionic projective line is defined by $\bH P^1$: 
\begin{equation*}
\zeta (t) := \int_{w[0,t]} \eta = \frac{1}{2}\int_0^t \frac{ \overline{w}(t)dw(t)-d\overline{w}(t)w(t)}{|w(t)|^2}, \quad t \geq 0,
\end{equation*}
where $(w(t))_{t \geq 0}$ is the Brownian motion in $\mathbb{H}P^1$. 
\end{definition}

%
As in the flat setting, the study of $\zeta$ makes use of the following skew-product decomposition.

\begin{lemma}
Let  $w$ be a Brownian motion on $\mathbb{H}P^1$ not started from 0 or $\infty$. There exist a 
  Jacobi diffusion $(r(t))_{t \ge 0}$ with  generator
 \[
 \frac{1}{2}\left(\frac{\partial^2}{\partial r^2}+6\cot 2r\frac{\partial}{\partial r}\right)
 \]
  and a Brownian motion $\Theta (t)$ on $\mathbf{SU}(2)$ independent from the process  $(r(t))_{t \ge 0}$ such that 
 \[
 w(t)=\tan r(t) \, \Theta({A(t)}),
 \]
 where 
 \begin{equation*}
A(t):= \int_0^t \frac{4 ds}{\sin^2(2r(s))}.
\end{equation*}
\end{lemma}
\begin{proof}
This follows from  the fact that since $\HP^1$ is isometric to the 4-dimensional real sphere with radius $1/2$, the operator $(1/2)\Delta_{\HP^1}$ may be decomposed in polar coordinates as:
\begin{equation}\label{Polar-S}
\frac{1}{2}\left(\frac{\partial^2}{\partial r^2}+6\cot 2r\frac{\partial}{\partial r}+\frac{4}{\sin^2 2r}\Delta_{\mathbf{SU}(2)} \right).
\end{equation}
\end{proof}

As a consequence, we obtain the equality in distribution:
\begin{equation}\label{W-eq-HP}
\zeta (t) \overset{\mathcal{D}}=\beta\left(\int_0^t \frac{4ds}{\sin^2(2r(s))}\right), 
\end{equation}
where $\beta$ is a $3$-dimensional standard Brownian motion which is independent from the process $r$. The analogue of Theorem \ref{T1} for the quaternionic projective line is then: 
\begin{theorem}\label{windingHP}
When $t \to \infty$, we have
\[
\frac{\zeta(t)}{t} \to \mathcal{N}(0, 2\mathrm{I}_3),
\]
in distribution.
\end{theorem}

\begin{proof}
Let $\lambda=(\lambda_1,\lambda_2,\lambda_3) \in \R^3$ and use \eqref{W-eq-HP} to write: 
\[
\mathbb{E}\left( e^{i \lambda\cdot \zeta(t)} \right)=\mathbb{E}\left( e^{-\frac{|\lambda|^2}{2} \int_0^t \frac{4ds}{\sin^2 2r(s)}} \right)=e^{-2|\lambda|^2t} \mathbb{E}\left( e^{-2|\lambda|^2 \int_0^t \cot^2 2r(s) ds} \right).
\]
From \eqref{Polar-S}, the process $r$ is the (unique) solution of the stochastic differential equation: 
\[
r(t)=r (0)+3\int_0^t \cot 2r(s) ds +\kappa(t), \quad r (0) \in (0, \pi/2),
\]
where $\kappa$ is a standard Brownian motion. In order to apply Girsanov's Theorem, we introduce the following local martingale: 
\begin{align*}
D_t^{(\mu)}=\exp \left( 2\mu \int_0^t \cot 2r(s) d\gamma(s) -2\mu^2 \int_0^t \cot^2 2r(s) ds \right), \quad \mu \geq 0.
\end{align*}
From It\^o's formula, we readily obtain: 
\[
D_t^{(\mu)} =e^{2\mu  t}\left(\frac{\sin 2r(t)}{\sin 2r (0)}\right)^{\mu} \exp\left(-2(\mu^2+2\mu) \int_0^t \cot^2 2r(s) ds \right),
\]
which shows in particular that $D_t^{(\mu)}$ is a martingale. Now, consider the new probability measure $\mathbb{P}^{(\mu)}$:
\[
\mathbb{P}_{/ \mathcal{F}_t} ^{(\mu)}=D_t^{(\mu)} \mathbb{P}_{/ \mathcal{F}_t}=(\sin 2r (0))^{-\mu} e^{2\mu t} (\sin 2r(t))^{\mu} e^{- 2(\mu^2+2\mu) \int_0^t \cot^2 2r(s)ds} \mathbb{P}_{/ \mathcal{F}_t},
\]
where $(\mathcal{F}_t)_{t \geq 0}$ is the natural filtration of $r$. If we choose $\mu=\sqrt{|\lambda|^2+1}-1$, then we get:
\[
\mathbb{E}\left( e^{i \lambda \cdot\zeta(t)} \right)=(\sin 2r (0))^\mu e^{-2(|\lambda|^2+\mu) t}\mathbb E^{(\mu)}\left(\frac{1}{(\sin2r(t)) ^{\mu}}\right).
\]
Moreover, Girsanov's theorem implies that the process
\[
\xi(t) : =\kappa(t)-2\mu\int_0^t \cot 2r(s)ds 
\]
is a Brownian motion under $\mathbb P^\mu$. Consequently, 
\[
d r(t)= d\xi(t)+(2\mu+3)\cot 2r(t)dt,
\] 
so that under the probability $\mathbb{P}^\mu$, $r$ is a Jacobi diffusion with generator:
\[
\mathcal{L}^{(\mu+1,\mu+1)}=\frac{1}{2} \frac{\partial^2}{\partial r^2}+\left(\left(\mu+\frac{3}{2}\right)\cot r-\left(\mu+\frac{3}{2}\right) \tan r\right)\frac{\partial}{\partial r}.
\]
Writing,
\[
\mathbb{E}\left( e^{i \lambda\cdot \frac{\zeta(t)}{\sqrt{t}}} \right)=(\sin 2r (0))^{\sqrt{\frac{|\lambda|^2}{t}+1}-1} e^{-2\left(\frac{|\lambda|^2}{t}+\sqrt{\frac{|\lambda|^2}{t}+1}-1\right) t}
\mathbb E^{\left(\sqrt{\frac{|\lambda|^2}{t}+1}-1\right)}\left((\sin2r(t)) ^{1-\sqrt{\frac{|\lambda|^2}{t}+1}}  \right) 
\]
 we end up with the limit
 \[
\lim_{t \to \infty} \mathbb{E}\left( e^{i \lambda\cdot \frac{\zeta(t)}{\sqrt{t}}} \right)=\lim_{t \to \infty}e^{-2\left(\sqrt{|\lambda|^2t+t^2}-t\right)}= e^{-|\lambda|^2},
\]
as required.
\end{proof}
%

\subsection{Windings in the quaternionic hyperbolic space}

The quaternionic anti-de Sitter space $\mathbf{AdS}^{7}(\mathbb{H})$ is defined as the quaternionic pseudo-hyperboloid:
\[
\mathbf{AdS}^{7}(\mathbb{H})=\lbrace q=(q_1,q_{2})\in \mathbb{H}^{2}, {|q|}^2_H =-1\rbrace,
\]
where 
\[
{|q|}_H^2 :=|q_1|^2-|q_{2}|^2.
\]

The group $\mathbf{SU}(2)$, viewed as the set of unit quaternions, acts isometrically on $\mathbf{AdS}^{7}(\mathbb{H})$ by left multiplication and the quotient space 
\begin{equation*}
\mathbf{AdS}^{7}(\mathbb{H})/ \mathbf{SU}(2)
\end{equation*}
is the quaternionic hyperbolic space $\bH H^1$ endowed with its canonical quaternionic K\"ahler metric. One can parametrize $\bH H^1$ using the quaternionic  homogeneous coordinate
\[
w=q_2^{-1} q_1, \quad q=(q_1,q_2) \in \mathbf{AdS}^{7}(\mathbb{H}).
\]
This allows the identification $\bH H^1$ with the unit open ball in $\mathbb{H}$ and will be in force in the sequel. In affine coordinates, the Riemannian distance $r$ from 0 is given by the formula
\[
\tanh r= |w|.
\]

If $\gamma: [0,+\infty) \to \mathbb{H}H^1 \setminus \{ 0  \}$ is a $C^1$ path, as before, one can consider its polar decomposition
\[
\gamma (t) = | \gamma (t) | \Theta (t)
\]
with $\Theta (t) \in \mathbf{SU}(2)$ and define the winding path $\theta (t) \in \mathfrak{su}(2)$ as
\[
\theta(t) = \int_0^t \Theta (s)^{-1}  d\Theta (s).
\]
The quaternionic winding form on $\bH H^1$ is then  the $\mathfrak{su}(2)$-valued one-form $\eta$ such that
\[
\theta(t)=\int_{\gamma[0,t]} \eta=\frac{1}{2}\int_0^t \frac{ \overline{\gamma}(s)d\gamma(s)-d\overline{\gamma}(s)\gamma(s)}{|\gamma(s)|^2}, \quad t \geq 0.
\]

%

The winding process of a Brownian motion on $\HH^1$ is then given by
\[
\zeta(t)=\mathrm{Im}\int_0^t w^{-1}(s)dw(s)=\frac12\int_0^t \frac{ \overline{w}(s)dw(s)-d\overline{w}(s)w(s)}{|w(s)|^2}.
\]

As before, to study $\zeta$, we shall make use of a skew-product decomposition.

\begin{lemma}
Let  $w$ be a Brownian motion on $\mathbb{H}H^1$ not started from 0. There exist a 
  hyperbolic Jacobi process $(r(t))_{t \ge 0}$ with generator
 \[
 \frac{1}{2}\left(\frac{\partial^2}{\partial r^2}+6\coth 2r\frac{\partial}{\partial r}\right)
 \]
  and a Brownian motion $\Theta (t)$ on $\mathbf{SU}(2)$ independent from the process  $(r(t))_{t \ge 0}$ such that 
 \[
 w(t)=\tanh r(t) \, \Theta_{A(t)},
 \]
 where 
 \begin{equation*}
A(t):= \int_0^t \frac{4 ds}{\sinh^2(2r(s))}.
\end{equation*}
\end{lemma}
\begin{proof}
This follows from  the fact that since $\mathbb{H}H^1$ is isometric to the 4-dimensional real hyperbolic space (with curvature $-1/4$),  the operator $(1/2)\Delta_{\mathbb{H}H^1}$ may be decomposed in polar coordinates as:
\begin{equation}\label{Polar}
\frac{1}{2}\left(\frac{\partial^2}{\partial r^2}+6\coth 2r\frac{\partial}{\partial r}+\frac{4}{\sinh^2 2r}\Delta_{\mathbf{SU}(2)} \right).
\end{equation}
\end{proof}

%
As a consequence, we obtain the equality in distribution:
\begin{equation}\label{W-eq-HH}
\zeta (t) \overset{\mathcal{D}}=\beta\left(\int_0^t \frac{4ds}{\sinh^2(2r(s))}\right), 
\end{equation}
where $\beta$ is a $3$-dimensional standard Brownian motion which is independent from the process $r$. 

Unlike the  $\mathbb{H}P^1$ case, the Brownian motion on $\mathbb{H}H^1$ is transient and as shown below, the corresponding winding process will have a limit in distribution when $t \to +\infty$. Moreover, the computations of the limiting distribution are more involved compared to the flat and the spherical settings. 
\begin{theorem}\label{windingCP}
For any $\lambda \in \mathbb{R}^3$,  
\[
\lim_{t \rightarrow +\infty} \mathbb{E}[e^{i\lambda \cdot \zeta(t)}] = \tanh (r (0))^{\sqrt{|\lambda|^2+1}-1} \left( 1 +\frac{1}{2\cosh(r (0))^2}\left(\sqrt{|\lambda|^2+1}-1\right) \right).
\] 
\end{theorem}

\begin{proof}
Let $\lambda \in \R^3$. We have 
\[
\mathbb{E}\left( e^{i \lambda\cdot \zeta(t)} \right)=\mathbb{E}\left( e^{-2|\lambda|^2 \int_0^t \frac{ds}{\sinh^2 2r(s)}} \right).
\]
The process $r$ solves the stochastic differential equation:
\[
r(t)=r (0)+3\int_0^t \coth(2r(s)) ds +\psi(t),
\]
where $\psi$ is a standard one-dimensional Brownian motion. In order to compute the characteristic function of $\zeta$, we shall look for an exponential local martingale of the form 
\begin{equation*}
D_t^{(\nu, \kappa)} := \exp\bigg(\int_0^t[\nu \coth(r(s)) + \kappa \tanh(r(s))] d\gamma(s) - \frac{1}{2}\int_0^t [\nu \coth(r(s)) + \kappa \tanh(r(s))]^2ds \bigg). 
\end{equation*}
To this end, we use It\^o's formula to derive: 
\begin{equation*}
\left(\frac{\sinh(r(t))}{\sinh(r (0))}\right)^{\nu} =  \exp\bigg(\nu \int_0^t\coth(r(s)) d\gamma(s) + \nu \int_0^t\coth^2(r(s))ds + 2\nu t\bigg), \quad r (0) > 0,
\end{equation*}
 \begin{equation*}
\left(\frac{\cosh(r(t))}{\cosh(r (0))}\right)^{\kappa} =  \exp\bigg(\kappa \int_0^t\tanh(r(s)) d\gamma(s) + \kappa \int_0^t\tanh^2(r(s))ds + 2 \kappa t\bigg),
\end{equation*}
so that 
\begin{multline*}
D_t^{(\nu, \kappa)} := e^{-2(\nu +\kappa+\nu\kappa/2)t} \left(\frac{\sinh(r(t))}{\sinh(r (0))}\right)^{\nu}\left(\frac{\cosh(r(t))}{\cosh(r (0))}\right)^{\kappa} 
\\\exp\bigg(-\frac{\nu^2+2\nu}{2}\int_0^t \coth^2(r(s))ds  -\frac{\kappa^2 +2\kappa}{2} \int_0^t \tanh(r(s))^2ds \bigg). 
\end{multline*}
Writing 
\begin{align*}
\frac{1}{\sinh^2 2r(s)} = \frac{[\coth(r(s)) - \tanh(r(s))]^2}{4} = \frac{\coth^2(r(s)) + \tanh^2(r(s)) - 2}{4}, 
\end{align*}
and choosing  
\begin{equation*}
\nu = \sqrt{1+|\lambda|^2} - 1 = -\kappa -2, 
\end{equation*}
then 
\begin{equation*}
D_t^{(\nu, \kappa)} := e^{4t} \left(\frac{\tanh(r(t))}{\tanh(r (0))}\right)^{ \sqrt{1+|\lambda|^2} - 1}\left(\frac{\cosh(r (0))}{\cosh(r(t))}\right)^{2}e^{-2|\lambda|^2 \int_0^t \frac{ds}{\sinh^2 2r(s)}}.
\end{equation*}
Since $\nu \geq 0$ then $D_t^{(\nu, \kappa)}$ is a bounded local martingale.
So $D_t^{(\nu, \kappa)}$ is a martingale and we can define the probability measure 
\begin{equation*}
\mathbb{P}_{/ \mathcal{F}_t}^{(\nu, \kappa)} := D_t^{(\nu, \kappa)} \mathbb{P}_{/ \mathcal{F}_t} 
\end{equation*}
under which the process $(r(t))_{t \geq 0}$ solves the SDE 
\begin{equation*}
r(t)=r (0)+\int_0^t \left[\left(\frac{3}{2}+\nu\right)\coth r(s) - \left(\frac{1}{2} + \nu\right) \tanh(r(s))\right]ds +\tilde{\gamma}(t)
\end{equation*}
for some $\mathbb{P}^{(\nu, \kappa)}$-Brownian motion $\tilde{\gamma}$. This is a hyperbolic Jacobi process of parameters ($1+\nu, -1-\nu$) and 
\begin{align*}
\mathbb{E}\left(e^{-2|\lambda|^2 \int_0^t \frac{ds}{\sinh^2 2r(s)}}\right) & = e^{-4t} \mathbb{E}^{(\nu,-\nu-2)}\left\{\left(\frac{\tanh(r(t))}{\tanh(r (0))}\right)^{1-\sqrt{1+|\lambda|^2}}\left(\frac{\cosh(r(t))}{\cosh(r (0))}\right)^{2}\right\}
\\& = e^{-4t} \frac{\left(\tanh(r (0))\right)^{\sqrt{1+|\lambda|^2}-1}}{\cosh^2(r (0))}\mathbb{E}^{(\nu,-\nu-2)}\left\{\left(1-\frac{1}{\cosh^2(r(t))}\right)^{-\nu/2} \cosh^2(r(t))\right).
\end{align*}
Using the generalized binomial theorem, we get further: 
\begin{equation*}
\mathbb{E}\left(e^{-2|\lambda|^2 \int_0^t \frac{ds}{\sinh^2 2r(s)}}\right)  = e^{-4t}  \frac{\left(\tanh(r (0))\right)^{\sqrt{1+|\lambda|^2}-1}}{\cosh^2(r (0))}\sum_{k \geq 0}\frac{(\nu/2)_k}{k!} \mathbb{E}^{(\nu,-\nu-2)}\left(\frac{1}{\cosh^{2k-2}(r(t))}\right).
\end{equation*}
Consequently, the long-time behavior of $\zeta(t)$ is given by the lowest-order term $k=0$ in the series above: 
\begin{equation}\label{Limit}
\lim_{t \rightarrow +\infty} \mathbb{E}\left(e^{-2|\lambda|^2 \int_0^{+\infty} \frac{ds}{\sinh^2 2r(s)}}\right) = \frac{\left(\tanh(r (0))\right)^{\sqrt{1+|\lambda|^2}-1}}{\cosh^2(r (0))}\lim_{t \rightarrow +\infty}e^{-4t}  \mathbb{E}^{(\nu,-\nu-2)}[\cosh^2(r(t))]. 
\end{equation}
To this end, we need the following lemma: 

\begin{lemma}\label{limit cosh}
For any $\alpha,\beta$, consider the family of hyperbolic Jacobi processes $\gamma_{\alpha,\beta}$ that are solutions of
\[
dr_{\alpha,\beta} (t)=(\alpha \coth r_{\alpha,\beta} (t) +\beta \tanh r_{\alpha,\beta} (t))dt +d\gamma(t), \quad r_{\alpha,\beta} (0)=r (0)>0.
\]
Then
\begin{align*}
\mathbb{E} ((\cosh r_{\alpha,\beta}(t)^2))=\frac{1+2\beta}{2(1+\alpha+\beta)} + e^{2(1+\alpha+\beta)t} \left( \cosh(r (0))^2 -\frac{1+2\beta}{2(1+\alpha+\beta)}\right).
\end{align*}
\end{lemma}

\begin{proof}
Let
\[
f(r)=(\cosh (r))^2.
\]
We have
\[
Lf =2(1+\alpha+\beta) f -(1+2\beta)
\]
where
\[
L=\frac{1}{2} \frac{d^2}{dr^2} +(\alpha \coth r +\beta \tanh r)\frac{d}{dr}.
\]
Thus, denoting
\[
\phi (t) := \mathbb{E} ((\cosh r_{\alpha,\beta}(t)^2))
\]
and using It\^o's formula, we obtain the following differential equation
\[
\phi'(t)=2(1+\alpha+\beta) \phi(t) -(1+2\beta),
\]
which proves the Lemma. 
\end{proof}
Specializing the lemma to $\alpha = 3/2+\nu, \beta = -1/2-\nu$, then 
\begin{equation*}
\mathbb{E}^{(\nu,-\nu-2)}[\cosh^2(r(t))] = \lim_{t \rightarrow +\infty}e^{-4t} \mathbb{E}^{(\nu,-\nu-2)}[\cosh^2(r(t))] = \cosh^2(r (0)) + \frac{\nu}{2}. 
\end{equation*}
Plugging this expression into \eqref{Limit}, we are done. 
\end{proof}
We close the section with an explicit expression of the density of $\zeta_{\infty}$. To this end, we write 
\begin{multline*}
\tanh (r (0))^{\sqrt{|\lambda|^2+1}-1}\left( 1 +\frac{1}{2\cosh(r (0))^2}(\sqrt{|\lambda|^2+1}-1) \right) = \tanh (r (0))^{\sqrt{|\lambda|^2+1}-1} + \\ \frac{\tanh(r (0))}{2}\partial_{r (0)} \tanh (r (0))^{\sqrt{|\lambda|^2+1}-1},
\end{multline*}
and recall  the characteristic function of the three-dimensional relativistic Cauchy random variable: 
\begin{equation*}
e^{-y(\sqrt{|\lambda|^2+1}-1)} = \frac{ye^y}{2\pi^2}\int_{\mathbb{R}^3} e^{i\lambda \cdot x} \frac{K_2(\sqrt{|x|^2+y^2})}{|x|^2+y^2} dx,
\end{equation*}
where $K_2$ is the modified Bessel function of the second kind, (see \cite{MR2506430} for more details). After straightforward computations, we end up with:
\begin{corollary}
Let $r (0) > 0$ be the hyperbolic distance from the origin of a Brownian motion in $\mathbb{H}H^1$. Then, the distribution of $\zeta_{\infty}$ is absolutely continuous with respect to Lebesgue measure in $\mathbb{R}^3$ and its density is given by: 
\begin{multline*}
\frac{-\ln(\tanh(r (0)))}{2\pi^2\tanh(r (0))} \frac{K_2(\sqrt{|x|^2+\ln^2(\tanh(r (0)))})}{|x|^2+\ln^2(\tanh(r (0)))} + \\ 
\frac{\tanh(r (0))}{2}\partial_{u}\left\{\frac{-\ln(\tanh(u))}{2\pi^2\tanh(u)} \frac{K_2(\sqrt{|x|^2+\ln^2(\tanh(u))})}{|x|^2+\ln^2(\tanh(u))}\right\}(r (0)).
\end{multline*}
\end{corollary}

\section{Matrix generalizations}

To conclude the chapter with an opening to further research directions, we briefly discuss a matrix-variate generalization of the L\'evy area formula for the stochastic area process and of the Yor formula for the winding functional. We intended to include them here since they are closely related to the so-called complex Wishart or Laguerre process studied in \cite{DemLag} and since the winding process implicitly appeared there. In particular, they are matrix-variate extensions of the stochastic area and of the winding functional arising from the Heisenberg group.

Let $w$ be a $n \times m, n \geq m,$ complex Brownian matrix: 
\begin{equation*}
w_{jk} = w_{jk}^1 + iw_{jk}^2, \quad 1 \leq j \leq n, \, 1 \leq k \leq m,
\end{equation*}
where $(w_{jk}^1)_{j,k}, (w_{jk}^2)_{j,k}$ are two independent families of independent real Brownian motions. Then the Laguerre process of size $m$ and dimension $2n$ is 
\begin{equation*}
X = w^{\star} w. 
\end{equation*}
It is the complex analogue of the Wishart process introduced by M. F. Bru in \cite{Bru} and the name `Laguerre' was picked in \cite{DemLag} from \cite{Kon-Oco} where the authors show that its eigenvalues process is a Doob-transform of independent squared Bessel processes. 

Now, the matrix analogue of the (Heisenberg) stochastic area process may be defined as
\[
A(t) := A_{n,m}(t):= \frac{1}{2i} \mathrm{tr} \left[ \int_0^t  w^* dw-  dw^* \, w \right].
\]
Using the formulas
\begin{equation*}
(dw)(dw^*) = 2mI_n dt, \quad (dw^*)(dw) = 2nI_m dt,
\end{equation*}
where the brackets are computed component-wise and $I_n, I_m$ are the identity matrices of sizes $n,m$ respectively,
 we further compute the (one-dimensional) bracket 
\begin{equation*}
\left(\mathrm{tr}[w^* dw-  dw^* \, w]\right)\left(\mathrm{tr}[w^* dw-  dw^* \, w]\right) = -4\mathrm{tr}(w^*w)dt.
\end{equation*}
Consequently, there exists a real Brownian motion $\beta$ independent from $w$ such that the matrix stochastic area is represented as: 
\begin{equation*}
A(t) =\beta\left(\int_0^t\mathrm{tr}(w(s)^*w(s)) ds\right), \quad t \geq 0,
\end{equation*}
where $(\mathrm{tr}(w(s)^*w(s)))_{s \geq 0}$ is a squared Bessel process of dimension $2nm$. It is then a well-known result of Williams   that for any $1$-dimensional squared Bessel process $R^2(t)$ started from $0$:
\[
\mathbb{E}\left(\exp\left(-\frac{\lambda^2}2\int_0^t R(s)^2ds\right)  \right)=\left(\frac{1}{\cosh (\lambda t)}\right)^{1/2}.
\]
We can then compute the distribution of $A(t)$ started from $0$: for any $\lambda>0$ and $t>0$, one has
\[
\mathbb{E}\left(e^{i\lambda A(t)} \right)=\mathbb{E}\left(e^{-\frac{\lambda^2}2\int_0^t \mathrm{tr}(w(s)^*w(s)) ds } \right)=\left(\frac{1}{\cosh (\lambda t)}\right)^{nm}.
\]
Therefore we have 
\[
\mathbb{E}\left(e^{i \frac{\lambda}{t}A(t)} \right)=\left(\frac{1}{\cosh (\lambda )}\right)^{nm}.
\]
We come now to the matrix analogue of the complex winding process. It may be defined as: 
\[
 \frac{1}{2i} \mathrm{tr} \left[ \int_0^t X(s)^{-1} [w^* dw-  dw^* \, w]X(s)^{-1} \right]
\]
where $X(s)^{-1}$ is the inverse of $X(s)$ (see Lemma 1 in \cite{DemLag}). Similar computations as above shows that the winding functional may be represented as: 
\begin{equation*}
\gamma_{\int_0^t\mathrm{tr}(X(s)^{-1})ds}, \quad t \geq 0,
\end{equation*}
where $\gamma$ is an independent real Brownian motion. When $n=m$, it was proved in \cite{DemLag} that (see Corollary 6 there): 
\begin{equation*}
\frac{4}{(m \log t)^2}\int_0^t\mathrm{tr}(X(s)^{-1})ds \quad \overset{\mathcal{D}}{\longrightarrow} \quad  T_1(\beta), \quad t \rightarrow \infty,
\end{equation*}  
where $T_1(\beta)$ is the first hitting time of $1$ by a standard real Brownian motion, whence a matrix-variate extension of Spitzer's theorem readily follows. 

When $n > m$, one may mimic the proofs of the results written in \cite{DemLag}, Section 5.1. 
The only major difference is that the stochastic differential equation satisfied by $\log(\det(X(t)))$ comes with an additional drift equal to 
\begin{equation*}
2(n-m)\int_0^t\mathrm{tr}(X(s)^{-1})ds,
\end{equation*}
which slightly alters the absolute-continuity relation displayed in \cite{DemLag}, Proposition 6: the latter becomes (with the same notations there): 
\begin{equation*}
Q_x^{n+\nu}{}_{|\mathscr{F}_t} = \left(\frac{\det(X(t))}{\det(x)}\right)^{\nu/2} \exp\left(-\frac{1}{2}\left[\nu^2 + 2\nu(n-m)\right] \int_0^t\mathrm{tr}(X(s)^{-1})ds\right) Q_x^n{}_{|\mathscr{F}_t}.
\end{equation*}
Accordingly, Corollary 5 in \cite{DemLag} remains valid after one performs in the right hand side the substitutions 
\begin{equation*}
m \rightarrow n, \quad \nu \rightarrow \sqrt{\nu^2 + (n-m)^2} - (n-m).
\end{equation*}
 However, the proof of Corollary 6 in \cite{DemLag} is no longer valid since one seeks $\nu(n,m,t) \geq 0$ such that  
\begin{equation*}
\sqrt{[\nu(n,m,t)]^2 + (n-m)^2} = \frac{\nu}{m\log t}
 \end{equation*}
 for $\nu \geq 0$. But this equation has no solutions for sufficiently large time $t$ since the left hand side is bounded below by $n-m \geq 1$.

\chapter{Brownian motions on Lie groups and Riemannian manifolds}\label{sec-BM-RM}

In this Chapter, we review  the definitions and some basic properties of Brownian motions on Lie groups and on Riemannian manifolds. We also present some useful tools from stochastic differential geometry. 

\section{Brownian motions on Lie groups}\label{BM Lie group}

We start with Brownian motions on (unimodular) Lie groups. For sake of simplicity, we first present the construction  of these processes in the setting of linear groups. Let $\mathbf{G}$ be a closed subgroup of the group $\mathbf{GL}(n,\mathbb R)$ of $n \times n$ real invertible matrices. The Lie algebra of $\mathbf{G} $ is the linear space defined by
\[
\mathfrak{g}=\left\{ X \in \mathbb{R}^{n\times n}, e^{tX} \in \mathbf{G} \right\}.
\]
The Lie bracket on $\mathfrak{g}$ is given by the anti-commutator
\[
[X,Y]=XY-YX.
\]
Assume that $\mathbf{G}$ is unimodular: for every $X \in \mathfrak{g}$,  the trace of the map $\mathfrak g \to \mathfrak g$, $Y \to [X,Y]$ is zero. Equivalently, the Haar measure is bi-invariant.

Let $e$ be the identity. If $(V_1,...,V_d)$ is a basis of $\mathfrak{g}$ and if $(B^1(t),...,B^d(t))_{t \geq 0}$ is a $d$-dimensional Brownian motion in $\mathbb{R}^d$, then the process
\[
B^{\mathfrak{g}}(t)=\sum_{i=1}^d B^i(t) V_i
\]
is called the  Brownian motion on the Lie algebra $\mathfrak{g}$ with respect to the basis $(V_1,...,V_d)$.

\begin{theorem} \label{bmliegroup} The Stratonovich differential equation
\begin{align}\label{eq:ed-def-group}
d X(t) =X(t)  \circ dB^{\mathfrak{g}}(t),~~~~ X(0)=e
\end{align}
has a unique solution taking its values  in $\mathbf{G}$.  This solution $(X(t))_{t \ge 0}$ satisfies for every $ s \ge 0$, $$(X(s)^{-1}X(t+s) )_{t \ge 0}\ \overset{\mathcal{D}}{=}\ (X(t))_{t \ge
0},$$ 

and is called a  left Brownian motion  on $\mathbf{G}$ (with respect to the basis $V_1,...,V_d$).
\end{theorem}

\begin{proof}
We give a sketch of the proof and refer to \cite[Chapter 5, Section 35]{Rogers2} for more details. Existence and uniqueness for the stochastic differential equation \eqref{eq:ed-def-group} follow from the standard theorem for solutions of stochastic differential equations, since the coefficients are Lipschitz continuous. Let  $T>0$ and consider the sequence of $\mathbf{G}$-valued  processes $(X^n(t))_{0 \le t \le T}$ defined inductively by
\begin{align*}
X^n(t) =  \exp \left( \frac{2^n}{T} (t -t_k) (B^{\mathfrak{g}}(t_{k+1}) -B^{\mathfrak{g}}(t_{k})) \right) X^n(t_k), \quad t_k \le t \le t_{k+1},
\end{align*}
where $t_k=kT/2^n$, $k =0,\dots, 2^n$. From Theorem 2 in \cite{MR942031} (or Section 4.8 in \cite{MR0247684}), the sequence of semimartingales $(X^n(t))_{0 \le t \le T}$ converges in probability to $(X(t))_{0 \le t \le T}$ uniformly on $[0,T]$. Since $X^n$ is $\mathbf{G}$-valued, we deduce that $X$ is $\mathbf{G}$-valued because $\mathbf{G}$ is closed. The fact that $(X(s)^{-1}X(t+s) )_{t \ge 0}\overset{\mathcal{D}}{=}(X(t))_{t \ge 0}$  follows from It\^o's formula, since it can be seen that $(X(s)^{-1}X(t+s) )_{t \ge 0}$ and  $(X(t))_{t \ge 0}$ solve the same stochastic differential equation in law with identical initial condition.
\end{proof}

In the same way, we call the solution of the stochastic differential equation
\[
d X(t) =\circ dB^{\mathfrak{g}}(t) \,  X(t) ,~~~~ X(0)=e,
\]
a right Brownian motion on $\mathbf{G}$.  It is easily seen that if $(X(t))_{t \ge 0}$ is a left Brownian motion on $\mathbf{G}$, then $(X(t)^{-1})_{t \ge 0}$ is a right Brownian motion on $\mathbf{G}$.

\begin{example}
The first basic example is $(\mathbb{R}^d,+)$. In that case, the
Lie algebra is generated by the vector fields
$\frac{\partial}{\partial x_1}, ...,\frac{\partial}{\partial x_d}$
and the  Brownian motion on $(\mathbb{R}^d,+)$ is nothing else
but the usual Euclidean Brownian motion.
\end{example}

\begin{example}
The second basic example is the circle. Let
\[
\mathbb{S}^1=\left\{ z \in \mathbb{C}, | z | =1 \right\}.
\]
Using polar coordinates the Lie algebra of $\mathbb{S}^1$ is $\mathbb{R}$ and is generated
by $\frac{\partial}{\partial \theta}$. The  Brownian
motion on $\mathbb{S}^1$ is given by
\[
X(t)=e^{i B(t)}, \quad t \ge 0,
\]
where $(B(t))_{t \ge 0}$ is a Brownian motion on
$\mathbb{R}$.
\end{example}

\begin{example}
Let us consider the Lie group $\mathbf{SO} (3)$, i.e. the group of
$3 \times 3$, real, orthogonal matrices of determinant $1$. Its
Lie algebra $\mathfrak{so} (3)$ consists of $3 \times 3$, real,
skew-adjoint matrices of trace $0$. A basis of $\mathfrak{so} (3)$
is formed by
\[
V_{1}=\left(
\begin{array}{ccc}
~0~ & ~1~ & ~0~ \\
-1~ & ~0~ & ~0~ \\
~0~ & ~0~ & ~0~
\end{array}
\right) ,\text{ }V_{2}=\left(
\begin{array}{ccc}
~0~ & ~0~ & ~0~ \\
~0~ & ~0~ & ~1~ \\
~0~ & -1~ & ~0~
\end{array}
\right) ,\text{ }V_{3}=\left(
\begin{array}{ccc}
~0~ & ~0~ & ~1~ \\
~0~ & ~0~ & ~0~ \\
-1~ & ~0~ & ~0~
\end{array}
\right).
\]
A left Brownian motion on $\mathbf{SO} (3)$ is
therefore given by the solution of the linear equation
\[
dX(t)=X(t) \left(
\begin{array}{ccc}
~0~ & ~\circ dB^1_t~ & ~\circ dB^3_t~ \\
~-\circ dB^1_t~ & ~0~ & ~\circ dB^2_t~ \\
~-\circ dB^3_t~ & -\circ dB^2_t~ & ~0~
\end{array}
\right), \quad X(0)=I_3.
\]
We refer to  Chapter 5 in \cite{Rogers2} for a detailed study of that process. As an historical remark, we point out that this process was first introduced and studied by J. Perrin \cite{MR1509284} in 1928, before stochastic calculus was even developed.   
\end{example}

Let now $\mathbf G$ be an arbitrary unimodular Lie group with Lie algebra $\mathfrak g$ and identity $e$. An element $V \in \mathfrak{g}$ can be seen as a left invariant vector field on $\mathbf G$ through the following action on smooth functions:
\[
V(f)(g):=\lim_{t \to 0} \frac{ f\left( g e^{tV} \right)-f(g)}{t}.
\]
We can then construct Brownian motions on $\mathbf G$ as before.

\begin{theorem} \label{bmliegroup2} 
Let $(V_1,...,V_d)$ be a basis of left invariant vector fields on $\mathbf{G}$. The Stratonovich differential equation
\begin{align*}
d X(t) =\sum_{i=1}^d V_i (X(t)) \circ dB^i(t),~~~~ X(0)=e
\end{align*}
has a unique solution taking its values  in $\mathbf{G}$. This solution $(X(t))_{t \ge 0}$ satisfies for every $ s \ge 0$, $$(X(s)^{-1}X(t+s) )_{t \ge 0}\overset{\mathcal{D}}{=}(X(t))_{t \ge
0}.$$ The process $(X(t))_{t \ge 0}$ is called a  left Brownian motion with on $\mathbf{G}$.
\end{theorem}

Note that the generator of $X(t)$ is given by the left-invariant second order differential operator $\frac{1}{2} \sum_{i=1}^d V_i^2$.

Of course, one similarly defines right Brownian motions on $\mathbf{G}$ using basis of right invariant vector fields.

\section{Brownian motions on Riemannian manifolds}

Let $(\mathbb{M},g)$ be a smooth and connected Riemannian manifold. In a local orthonormal frame  of vector fields $V_1,\cdots,V_n$, one can compute the length of the gradient of a smooth function $f$ as follows:
\[
| \nabla f |^2 :=\sum_{i=1}^n (V_i f)^2.
\]
Let us denote by $\mu$ the Riemannian volume measure  on $(\mathbb{M},g)$. We can then consider the pre-Dirichlet form
\[
\mathcal{E}(f,g)=\int \langle \nabla f , \nabla g \rangle d\mu, \quad f,g \in C_0^\infty(\mathbb{M}).
\]
There exists a unique second order operator $\Delta$ such that for every  $f,g \in C_0^\infty(\mathbb{M})$,
\[
\mathcal{E}(f,g)=-\int f \Delta g d\mu =-\int g \Delta f d\mu.
\]
The operator $\Delta$ is called the Laplace-Beltrami operator. Locally, we have
\[
\Delta=\sum_{i=1}^n V_i^2 -D_{V_i} V_i,
\]
where $D$ denotes the Levi-Civita connection on $\mathbb{M}$.

\begin{definition}
A Brownian motion $(X(t))_{ t \ge 0}$ on $\mathbb{M}$ is a diffusion process with generator $\frac{1}{2} \Delta$, that is $(X(t))_{ t \ge 0}$ is a continuous stochastic process on $\M$ such that for every $f \in C^\infty(\mathbb{M})$,
\[
f(X(t))-\frac{1}{2} \int_0^t \Delta f(X(s)) ds, \quad 0 \le t < \mathbf{e}
\]
is a local martingale, where $\mathbf{e}$ is the lifetime of $(X(t))_{ t \ge 0}$ on $\mathbb{M}$.
\end{definition}

In this monograph, very often we will not explicitly mention the possibly finite lifetime of $(X(t))_{ t \ge 0}$ and implicitly consider all the processes up to their lifetimes. 
One can construct Brownian motions by using the theory of Dirichlet forms (see \cite{Fukushima}), using the minimal closed extension of $\mathcal{E}$. If the metric $g$ is complete (which we always assume for Riemannian metrics in this monograph), then one can prove that $\Delta$ is essentially self-adjoint on $C_0^\infty(\mathbb{M})$ (Theorem 11.5 in \cite{MR2569498}) and there is therefore a unique closed extension of $\mathcal{E}$ which turns out to be a strictly local Dirichlet form. A Brownian motion on $\M$ is then the Hunt process associated with this closed extension.
Note that even in the complete case the process $(X(t))_{t \ge 0}$ may have a finite lifetime.  However it is a well-known result by A. Grigor'yan  and K.-T. Sturm \cite[Theorem 4]{MR1301456} that a sufficient condition for $(X(t))_{t \ge 0}$ to have a finite lifetime is that for some $x_0 \in \M$ and $r_0>0$
 \[
 \int_{r_0}^{+\infty} \frac{ r \, dr}{\ln \mu (B(x_0,r))} =+\infty,
 \]
 where $B(x_0,r)$ denotes the metric ball with radius $r$ for the Riemannian distance. This is for instance satisfied if for some constants $C_1,C_2>0$ one has  $\mu (B(x_0,r))\le C_1e^{C_2 r^2}$.

Brownian motion is associated with the heat semigroup on the manifold, more precisely:
\begin{itemize}
\item \hspace{.1in} By using the spectral theorem for $\Delta$ in the Hilbert space $L^2(\M,\mu)$, we may construct a strongly continuous contraction semigroup $(P_t)_{t \ge 0}$ in $L^2(\M,\mu)$ whose infinitesimal generator is $\frac{1}{2} \Delta$;
\item \hspace{.1in} By using the ellipticity of $\Delta$, we may prove that $(P_t)_{t \ge 0}$ admits a heat kernel: There is a smooth function $p(t,x,y)$, $t \in (0,+\infty), x,y \in \mathbb{M}$, called the heat kernel, such that for every $f \in L^2(\M,\mu)$ and $x \in \mathbb{M}$ ,
\[
P_t f (x)=\int_{\mathbb{M}} p(t,x,y) f(y) d\mu (y).
\]
 Moreover, the heat kernel enjoys the two following properties:
\begin{itemize}
\item \hspace{.1in} (Symmetry) $p(t,x,y)=p(t,y,x)$;
\item \hspace{.1in} (Chapman-Kolmogorov relation) $p(t+s,x,y)=\int_{\mathbb{M}} p(t,x,z)p(s,z,y)d\mu(z)$. 
\end{itemize}
\item \hspace{.1in} The semigroup  $(P_t)_{t \ge 0}$ is a sub-Markov semigroup: If $ 0\le f \le 1$ is a function in $L^2(\M,\mu)$, then $0 \le P_t f \le 1$.
\item \hspace{.1in} By using the Riesz-Thorin interpolation theorem, $(P_t)_{t \ge 0}$ defines a contraction semigroup on $L^p(\M,\mu)$, $1 \le p \le \infty$.
\item \hspace{.1in} The semigroup  $(P_t)_{t \ge 0}$ is the transition semigroup of the Brownian motion $(X(t))_{t\ge 0}$ meaning that for every $f \in L^\infty (\M,\mu)$,
\[
\mathbb{E}\left( f(X(t)) \mid X(0)=x\right)=P_tf(x).
\]
\end{itemize} 
We refer to \cite{MR2569498} and \cite{MR4432545} more for details on the construction of the heat kernel and heat semigroup on Riemannian manifolds, see also \cite{Hsu}.

The Brownian motion on a manifold can be studied using stochastic calculus. Indeed, in a local orthonormal frame $V_1,\dots,V_n$, we have
\begin{align}\label{SDE BM local}
dX(t)=\sum_{i=1}^n V_i(X(t)) \circ dB^i(t) -\frac{1}{2} \sum_{i=1}^nD_{V_i} V_i (X(t)) dt,
\end{align}
where $B^1,\dots,B^n$ is a Brownian motion in $\mathbb{R}^n$.

If needed, one can also do (local) computations on the Brownian motion $(X(t))_{t \ge 0}$ using local coordinates as follows. Let $(x^1,\dots,x^n)$ be a local coordinate system on $\mathbb M$. We denote $g_{ij}=g \left(\frac{\partial}{\partial x^i}, \frac{\partial}{\partial x^j}\right)$ and  by $(h_{ij})_{1 \le i,j \le n}$ the positive definite symmetric square root of the matrix $(g_{ij})_{1 \le i,j \le n}$. We denote by $(g^{ij})_{1 \le i,j \le n}$ the inverse matrix of $(g_{ij})_{1 \le i,j \le n}$ and by  $(h^{ij})_{1 \le i,j \le n}$ the inverse matrix of $(h_{ij})_{1 \le i,j \le n}$. The vector fields
\[
V_i=\sum_{j=1}^n h^{ij} \frac{\partial}{\partial x^i}, \quad 1 \le i \le n
\]
define then a local orthonormal frame. Using It\^o's formula and \eqref{SDE BM local} one obtains that for $1 \le i \le n$,
\[
dX^i(t)=\sum_{j=1}^n (V_j x^i)(X(t)) \circ dB^j(t) -\frac{1}{2} \sum_{j=1}^n(D_{V_j} V_j x^i ) (X(t)) dt
\]
where $(X^1(t) ,\dots, X^n(t))$ are the components of $(X(t))_{t \ge 0}$ in the local coordinate chart. The above expression can be further simplified. Indeed,
\[
V_j x^i=h^{ij}
\]
and
\begin{align*}
\sum_{j=1}^n D_{V_j} V_j &=\sum_{j=1}^n D_{\sum_{k=1}^n h^{jk} \frac{\partial}{\partial x^k}} \left(\sum_{l=1}^n h^{jl} \frac{\partial}{\partial x^l} \right) \\
 &=\sum_{j,k,l=1}^n h^{jk}\frac{\partial h^{jl}}{\partial x^k}\frac{\partial}{\partial x^l} +\sum_{j,k,l=1}^n h^{jk}h^{jl} D_{\frac{\partial}{\partial x^k}} \frac{\partial}{\partial x^l} \\
 &=\sum_{j,k,l=1}^n h^{jk}\frac{\partial h^{jl}}{\partial x^k}\frac{\partial}{\partial x^l} +\sum_{j,k,l,m=1}^n h^{jk}h^{jl} \Gamma_{kl}^m \frac{\partial}{\partial x^m}
\end{align*}
where $\Gamma_{kl}^m$ are the Christoffel symbols of the Levi-Civita connection.  This yields
\[
dX^i(t)=\sum_{j=1}^nh^{ij}(X(t)) \circ dB^j(t) -\frac{1}{2} \sum_{j,k=1}^n h^{jk}\frac{\partial h^{ji}}{\partial x^k}(X(t)) dt-\frac{1}{2} \sum_{j,k,l=1}^n h^{jk}h^{jl} \Gamma_{kl}^i (X(t)) dt.
\]
Equivalently, in It\^o's form
\[
dX^i(t)=\sum_{j=1}^nh^{ij}(X(t)) dB^j(t)-\frac{1}{2} \sum_{j,k,l=1}^n h^{jk}h^{jl} \Gamma_{kl}^i (X(t)) dt.
\]
Therefore, for the quadratic covariations, we have
\[
d \langle X^i,X^j\rangle(t)=g^{ij}(X(t)). 
\]
We note that the stochastic differential equation \eqref{SDE BM local} in general only locally describes $X(t)$, since it might not be always possible to find a globally defined orthonormal frame of vector fields. However, using a partition of unity, a Brownian motion can be characterized as the stochastic process that satisfies in any chart the stochastic differential equation \eqref{SDE BM local} and this is often enough in applications.

 We further note that it is possible to construct the Brownian motion from a globally defined stochastic differential equation on the orthonormal frame bundle of the manifold; this is the so-called Eels-Elworthy-Malliavin construction of the Brownian motion, see Section \ref{section construction HBM} for a related construction. Since this construction will not be needed in the present monograph, we do not discuss further about this but refer to the monograph \cite{Hsu} (Proposition 3.2.2).

\section{Semimartingales and stochastic line integrals}

Stochastic calculus tools are available to study stochastic processes taking values in manifolds, see \cite{MR675100,MR1735806, Hsu}.   In this short section, we briefly discuss the construction of the Stratonovich line integral of semimartingales.  Throughout the section, let $\M$ be a manifold.

\begin{definition}
Let $(X(t))_{t \ge 0}$ be a $\M$ valued stochastic process. The process $(X(t))_{t \ge 0}$ is called a semimartingale if for every smooth function $f:\M \to \R$, the process $(f(X(t)))_{t\ge 0}$ is a semimartingale.
\end{definition}

For instance, on a Riemannian manifold, the Brownian motion is a semimartingale. If $X(t)$ is a semimartingale on $\M$ and $(x^1,\dots , x^n)$ a local coordinate system, from It\^o's formula, one can locally describe $X(t)$ using the dynamics
\[
dX(t)=\sum_{i=1}^n \frac{\partial}{\partial x^i} (X(t)) \circ dX^i(t)
\]
where $X^1(t),\dots,X^n(t)$ are the coordinates of $X(t)$ in the local coordinate system.

If   $\alpha$ is a smooth one-form on $\M$ one can then define the Stratonovich stochastic line integral $\int_{X[0,t]} \alpha$ which is a real valued semimartingale satisfying in any local coordinate system $(x^1,\dots,x^n)$
\[
\int_{X[0,t]} \alpha=\sum_{i=1}^n \int_0^t \alpha_i (X(s)) \circ dX^i(s)
\]
for $t\ge0$ up to the first exit time of the chart by $X$, where
\[
\alpha=\sum_{i=1}^n \alpha_i dx^i.
\]
It is possible to prove that $\int_{X[0,t]} \alpha$ is actually independent from the choice of the local coordinates and is defined as long as $X$ is, see \cite[Section 2.4]{Hsu} for more details.

\section{Examples of Riemannian Brownian motions}

In this section we discuss several examples of Riemannian Brownian motions.

\begin{example}[Brownian motion on the sphere]\label{ex-BM-sphere}
Consider the unit sphere
\[
\mathbb{S}^{n-1}=\left\{ x \in \mathbb{R}^n, | x |=1 \right\}, \quad n \ge 3.
\]

The Riemannian structure on $\mathbb{S}^{n-1}$ is the one induced from the inner product of $\mathbb{R}^n$. That is, if $x \in \mathbb{S}^{n-1}$, $\mathbb{X}, \mathbb{Y} \in {T}_x \mathbb{S}^{n-1}$, then
\[
g_x(\mathbb{X},\mathbb{Y})=\langle \mathbb{X},\mathbb{Y} \rangle,
\]
where $\mathbb{X}$ and $\mathbb{Y}$ are seen as vectors in $\mathbb{R}^n$ and  $\langle \mathbb{X},\mathbb{Y} \rangle$ is the usual inner product. The Riemannian distance $d$ between $x,y \in \mathbb{S}^{n-1}$ is seen to be given by
\[
d(x,y)={\arccos} \langle x,y \rangle.
\]
Geodesics are thus parts of great circles. If we denote by $\mu$, the Riemannian measure of $\mathbb{S}^{n-1}$, it turns out that $\mu$ is nothing else but the Euclidean surface measure of  $\mathbb{S}^{n-1}$ and we have the well-known integration formula in polar coordinates: If $f$ is an integrable function on $\mathbb{R}^n$, then
\[
\int_{\mathbb{R}^n} f (x) dx =\int_0^{+\infty} \int_{\mathbb{S}^{n-1} } f (r u) r^{n-1} d\mu (u) dr.
\]
We may  observe that using the last formula with the function $f(x)=e^{-| x |^2}$, leads to the following formula for the volume of the sphere:
\[
\mu(\mathbb{S}^{n-1})=2\frac{\pi^{n/2}}{\Gamma \left( \frac{n}{2} \right)}.
\]
The Laplace-Beltrami operator  $\mathbb{S}^{n-1}$ is easily described from the Laplace operator of $\mathbb{R}^n$. Indeed, if $f : \mathbb{S}^{n-1} \to \mathbb{R}$ is a smooth function, then $f$ may be extended as a function defined on $\mathbb{R}^n \setminus \{ 0 \}$ by setting
\[
\tilde{f}(x)=f \left( \frac{x}{| x |} \right).
\]
We may then check that if $\Delta_{\mathbb{S}^{n-1}}$ denotes the Laplace-Beltrami operator on $\mathbb{S}^{n-1}$, we have
\[
\Delta_{\mathbb{S}^{n-1}} f (x)=\Delta \tilde f (x).
\]
From the previous observation, in the polar coordinates $x = r u$, $r > 0$, $u \in \mathbb{S}^{n-1}$, we have
\[
\Delta=\frac{ \partial^2}{\partial r^2} +\frac{n-1}{r} \frac{\partial}{\partial r} +\frac{1}{r^2} \Delta_{\mathbb{S}^{n-1}}.
\]
One deduces that if $(B(t))_{t \ge 0}$ is a Brownian motion in $\R^n$ not started at 0, then 
\[
B(t) = r (t) \, \beta \left( \int_0^t \frac{ds}{r(s)^2} \right)
\]
where $(r(t))_{t \ge 0}= (|B(t)|)_{t \ge 0}$ is a diffusion with generator
\[
\frac{1}{2} \left( \frac{ \partial^2}{\partial r^2} +\frac{n-1}{r} \frac{\partial}{\partial r} \right) 
\]
and $(\beta(t))_{t \ge 0}$ is a Brownian motion on $\mathbb{S}^{n-1}$ which is independent from the diffusion $(r(t))_{t \ge 0}$. This decomposition is called the skew product decomposition of the Euclidean Brownian motion (see also Example \ref{BM on warped product} below).

 Using It\^o's formula for the process $\frac{B(t)}{|B(t)|}$, one easily sees that $(\beta(t))_{t \ge 0}$ therefore solves in distribution the stochastic differential equation:
\begin{equation}\label{eq-strook-sde}
d\beta^j (t) =\sum_{i=1}^n \left(\delta_{ij}- \beta^j(t)\beta^i(t) \right)\circ dB^i(t), \quad 1 \le j \le n.
\end{equation}
This is called the Stroock stochastic differential equation for the Brownian motion on $\mathbb{S}^{n-1}$.
\end{example}

\begin{example}[Brownian motion on the real hyperbolic space]
For $n \ge 2$, let
\[
H^n=\left\{ z=(x,y), x\in \R^{n-1}, y>0 \right\}
\]
be the upper half space with origin $O=(0,1) \in H^n$. We equip $H^n$ with the Riemannian metric
\[
ds^2=\frac{dx^2_1+\cdots+dx^2_{n-1}+dy^2}{y^2}.
\]
The resulting Riemannian manifold is called the real hyperbolic space of dimension $n$; see \cite{MR768584}, page 265.  The Riemannian distance between $z=(x,y)$ and $z'=(x',y')$ is given by
\[
\cosh d(z,z')=\frac{ |x-x'|^2+y^2+y'^2}{2yy'}.
\]
The Riemannian volume measure is given by
\[
d\mu=\frac{dx \, dy}{y^n}.
\]
The Laplace-Beltrami operator is 
\[
\Delta_{H^n}=y^2 \left( \frac{\partial^2}{\partial y^2}+\sum_{i=1}^{n-1} \frac{\partial^2}{\partial x_i^2} \right)-\frac{n-2}{2} y \frac{\partial}{\partial y}.
\]
Given the form of the Laplace-Beltrami operator, one deduces that the Brownian motion $Z(t)=(X(t), Y(t))$ on $H^n$ issued from the origin $O$ solves a stochastic differential equation
\begin{align*}
\begin{cases}
dX^i(t) =Y(t) dB^i(t) , \quad 1 \le i \le n-1 \\
dY(t) =  -\frac{n-2}{4} \,  Y(t) dt +Y(t) d\beta (t)
\end{cases}
\end{align*}
where $(B(t),\beta (t))$ is a Brownian motion in $\R^n$. Remarkably, this stochastic differential equation can be explicitly integrated:
\begin{align*}
\begin{cases}
X^i(t) =\int_0^t \exp \left( -\frac{n}{4} s +\beta(s) \right) dB^i(s) , \quad 1 \le i \le n-1 \\
Y(t) = \exp \left( -\frac{n}{4} t +\beta(t) \right) .
\end{cases}
\end{align*}
As is apparent from the last formula, the Brownian motion on the real hyperbolic space is therefore closely related to the study of exponential functionals of the Brownian motion; we refer to \cite{MR2203676} for further details about that connection.
\end{example}

\begin{example}[Brownian motion on a sub-manifold of the Euclidean space]
Let $(\M,g)$ be a Riemannian manifold which is isometrically embedded in $\mathbb{R}^n$. Let $e_1,\cdots,e_n$ be the standard orthonormal basis of $\mathbb{R}^n$ and for each $x \in \M$, let $P_i (x)$ be the orthogonal projection of $e_i$ to ${T}_x \M$. Thus, $P_i$ is a vector field on $\M$. If $(B(t))_{t \ge 0}$ is a Brownian motion in $\mathbb{R}^n$, then the solution of the stochastic differential equation
\[
dX(t)=\sum_{i=1}^n P_i (X(t)) \circ dB^i(t)
\]
is a Brownian motion on $(\M,g)$, see Theorem 3.1.4 in \cite{Hsu}. We note that from Nash theorem, any Riemannian manifold can be isometrically embedded in $\mathbb{R}^n$ for some $n$ which is large enough with respect to the dimension of $\M$. Therefore, any Brownian motion can be constructed in this way by means of a stochastic differential equation in $\mathbb{R}^n$ for $n$ large enough. For instance the unit sphere $\mathbb S^{n-1}$ is isometrically embedded in $\R^n$ and one can see that in that case
\[
\left \langle P_i (x), e_j \right \rangle=\delta_{ij}-x^ix^j
\]
This  yields back the Stroock stochastic differential equation \eqref{eq-strook-sde} for the Brownian motion on $\mathbb{S}^{n-1}$.
\end{example}

\begin{example}[Brownian motion on warped products ]\label{BM on warped product}
Let $(\M_1 , g_1)$  and  $(\M_2,g_2)$ be   Riemannian manifolds and $f$ be a smooth and positive function on $\M_1$. The manifold $\M_1 \times \M_2$ equipped with the metric $g_1 \oplus f g_2$ is called a warped product of $\M_1$ and $\M_2$. The Riemannian volume on the warped product is given by
\[
f^{\frac{n_2}{2}}d\mu_{\M_1} \otimes d\mu_{\M_2}
\]
where $n_2=\mathrm{dim} \, \M_2$. Therefore, the Laplace Beltrami operator on the warped product is given by
\[
\Delta_{\M_1 \times \M_2}=\Delta_{\M_1 } +\frac{n_2}{2} \frac{\langle \nabla_{\M_1} f ,\  \cdot\  \rangle }{f} +\frac{1}{f} \Delta_{\M_2}.
\]
It follows that if $(X^1(t))_{t \ge 0}$ is a diffusion on $\M_1$ with generator
\[
\frac{1}{2} \left( \Delta_{\M_1 } +\frac{n_2}{2} \frac{\langle \nabla_{\M_1} f ,\ \cdot\ \rangle }{f} \right)
\]
and if $(X^2(t))_{t \ge 0}$  is an independent Brownian motion on $\M_2$ then the process
\[
X(t)= \left(X^1(t), X^2\left(\int_0^t \frac{ds}{f(X^1(s))} \right)\right)
\]
is a Brownian motion on $\M_1 \times \M_2$. 
For instance if $\M_1 =\mathbb{R}_{>0}$ and $\M_2=\bS^{n-1}$ equipped with the metrics $g_1=dr^2$ and $g_2$ the canonical metric on $\bS^{n-1}$, then the warped product
\[
g_1\oplus r^2 g_2
\]
is a metric on  $\mathbb{R}_{>0} \times \bS^{n-1}$ such that the map $\mathbb{R}_{>0} \times \bS^{n-1} \to \mathbb{R}^n$, $(r,u) \to ru$ is an isometry, where $\mathbb{R}^n$ is equipped with its canonical Euclidean metric.  This yields the celebrated skew-product representation of the Brownian motion on $\mathbb{R}^n$: If $X^1(t)$ is a diffusion with generator
\[
\frac{1}{2} \left( \frac{d^2}{dr^2} +\frac{n-1}{r} \frac{d}{dr} \right)
\]
and $X^2(t)$ is an independent Brownian motion on $\bS^{n-1}$, then the process
\[
X(t)=X^1(t)X^2\left(\int_0^t \frac{ds}{X^1(s)^2} \right)
\]
is a Brownian motion in $\mathbb R^n$.
\end{example}

\begin{example}[Brownian motion on Riemannian unimodular Lie groups]
Let $\mathbf G$ be a unimodular Lie group with Lie algebra $\mathfrak g$. The choice of an inner product $\left\langle \cdot,\cdot \right\rangle_{\mathfrak g}$ on $\mathfrak g$ induces a left-invariant Riemannian metric on $\mathbf G$. Given an orthonormal basis $V_1,\dots,V_n$ of $\mathfrak g$, one can then construct a left Brownian motion on $\mathbf G$ by solving the globally defined stochastic differential equation
\[
d X(t) =\sum_{i=1}^n V_i(X(t))  \circ dB^{i}(t).
\]
We note that the distribution of the process $X(t)$ is independent from the choice of the orthonormal basis $V_1,\dots,V_n$; it only depends on the choice of the inner product $\left\langle \cdot,\cdot \right\rangle_{\mathfrak g}$ itself. Since $X(t)$ has generator $\frac{1}{2} \sum_{i=1}^n V_i^2$ which is the Laplace-Beltrami operator of the left invariant Riemannian metric, one deduces that $X(t)$ is the Brownian motion for the left invariant Riemannian metric on $\mathbf G$. The same discussion also applies to right Brownian motion and right invariant Riemannian metrics. Observe that if $\mathbf G$ admits a bi-invariant Riemannian metric (for instance if $\mathbf G$ is compact), then the left invariant Laplace-Beltrami operator coincides with the right invariant Laplace-Beltrami operator. As a consequence, in that case, the left and  right Brownian motions on $\mathbf G$ have the same distribution.  
\end{example}

\begin{example}[Brownian motion on general Riemannian  Lie groups]
Let $\mathbf G$ be a non necessarily unimodular Lie group with Lie algebra $\mathfrak g$. As before, the choice of an inner product $\left\langle \cdot,\cdot \right\rangle_{\mathfrak g}$ on $\mathfrak g$ induces a left invariant Riemannian metric on $\mathbf G$. Let $V_1,\dots,V_n$ be an orthonormal basis of $\mathfrak g$ for that inner product. We can find constants $\omega^{k}_{ij}$'s  such that
\[
[V_i,V_j]=\sum_{k=1}^n \omega_{ij}^k V_k.
\]
Those numbers are called the structure constants associated with the orthonormal basis  $V_1,\dots,V_n$. Let us denote by $\mu$ a left-invariant Haar measure on $\mathbf G$. The adjoint with respect to $\mu$ of the left invariant vector field $V_i$ is then given by
\[
V_i^*=-V_i+ \sum_{k=1}^n \omega_{ik}^k.
\]
We note that the group $\mathbf G$ is unimodular if and only if for every $i$, $\sum_{k=1}^n \omega_{ik}^k=0$. Therefore, the Laplace-Beltrami operator for the left invariant Riemannian metric is given by
\[
\Delta=-\sum_{i=1}^n V_i^* V_i=\sum_{i=1}^nV_i^2 -\sum_{i,k=1}^n \omega_{ik}^k V_i.
\]
The left Brownian motion on $\mathbf G$ is then given by the solution of the stochastic differential equation
\[
d X(t) =\sum_{i=1}^n V_i(X(t))  \circ dB^{i}(t)-\frac{1}{2}\left(\sum_{i,k=1}^n \omega_{ik}^k V_i(X(t))\right)dt. 
\]
If the group $\mathbf G$ is not unimodular, there is therefore a \textit{correcting drift} in the stochastic differential equation defining the left Brownian motion. 
\end{example}

%
%

\section{The unitary Brownian motion}

In this section we study in details the example of the Brownian motion on the unitary group. This study illustrates the interactions between the study of Brownian motions on Lie groups and  the theory of random matrices. The literature on the unitary Brownian motion is quite extensive and we refer to  \cite{dyson1962brownian,hall2018eigenvalue,collins2018spectral,meckes2018convergence} for further viewpoints.

The unitary group $\mathbf{U}(n)$ is a compact simple subgroup of the general linear group and its Lie  algebra $\mathfrak{u}(n)$ is the vector space of skew-Hermitian matrices:
\[
\mathbf{U}(n)=\left\{ M \in \mathbb C^{n\times n}, M^*M={I}_n \right\}, 
\]
\[
\mathfrak{u}(n)=\left\{ A \in \mathbb C^{n\times n}, A^*+A=0 \right\}.
\]
One can equip $\mathfrak{u}(n)$ with the inner product (so-called Killing form):
\[
B(A_1,A_2)=-(1/2)\textrm{tr}(A_1A_2), \quad A_1, A_2\in \mathfrak{u}(n),
\]
which induces on $\mathbf{U}(n)$ a bi-invariant Riemannian metric.

Note that with respect the Killing form an orthonormal basis of $\mathfrak{u}(n)$ can be given by
\[
\{E_{\ell j}-E_{j \ell}, i(E_{\ell j}+E_{j\ell}), T_\ell, \quad 1\le \ell <j \le n\}
\]
where $E_{\ell j} =(\delta_{(\ell, j)}(k,m))_{1 \le k , m \le n}$, $T_\ell=\sqrt{2}iE_{\ell\ell}$.

Therefore the Brownian motion $(A(t))_{t \ge 0}$ on $\mathfrak{u}(n)$ is  of the form
\[
A(t) =\sum_{1\le \ell<j\le n}(E_{\ell j}-E_{j\ell})B_{\ell j}(t)+i(E_{\ell j}+E_{j\ell})\tilde{B}_{\ell j}(t)+\sum_{j=1}^n T_{j}{B}_j(t),\quad t\ge0, 
\]
where $B_{\ell j}$, $\tilde{B}_{\ell j}$, ${B}_j$ are  independent standard real Brownian motions. 
We denote
\[
A(t) =\sum_{1\le \ell,j\le n}A_{\ell j}(t)E_{\ell j}.
\]
It can be easily seen that for all $1\le \ell \not= j\le n$
\[
dA_{\ell j}(t)\, d\overline{A}_{\ell j}(t)=2dt,\quad dA_{\ell j}(t)\, d{A}_{\ell j}(t) =0
\] 
and  for all $1\le \ell\le n$ that
\[
dA_{\ell \ell}(t)\, d\overline{A}_{\ell \ell}(t)=2dt, \quad dA_{\ell \ell}(t)\, d{A}_{\ell \ell}(t)=-2dt.
\]

The Brownian motion on $\mathbf{U}(n)$ thus satisfies the stochastic differential equation
\begin{equation}\label{eq-UBM-SDE}
dU(t)=U(t)\circ dA(t).
\end{equation}

A matrix in $\U(n)$ has $n$ complex eigenvalues of modulus one. We consider the eigenvalues of the Brownian motion process $(U(t))_{t\ge0}$, and denote them by $(e^{i\lambda_1 (t)}, \dots, e^{i\lambda_n(t)})$. 

Let $D_n=\{  x_1<\cdots <x_n < x_1+2\pi \}$, and consider the stopping time
\begin{equation}\label{eq-tau-N1}
\tau_{D_n}=\inf\{t>0, \lambda(t)\not\in D_n\}. 
\end{equation}

\begin{theorem}\label{eigen process u}
Assume $\lambda (0)\in D_n$. Then up to time $\tau_{D_n}$, the  process $\lambda(t)=(\lambda_1(t),\dots,\lambda_n(t))$, $t\ge0$ satisfies the following stochastic differential equation
\begin{equation}\label{eq-Un-lambda-sde}
d\lambda_j(t)=\sqrt{2}\,dB_j(t)+ \sum_{p\not=j}\cot\frac{\lambda_j-\lambda_p}{2} \, dt,\quad  j=1,\dots, n
\end{equation}
where $(B_j(t))_{t \ge 0}$, $j=1,\dots, n$ is a Brownian motion in $\mathbb{R}^n$.
\end{theorem}

\begin{proof}
Using the diagonalization theory of unitary matrices, one can write
$$U=V\Lambda V^*,$$ 
where $V\in \mathbf{U}(n)$ and $\Lambda=\mathrm{diag}\{e^{i \lambda_1},\dots, e^{i\lambda_n}\}$. Let us denote $dM=dV^*\circ V$ and $dN=V^*\circ dU \circ V$. We have 
\begin{equation}\label{eq-lambda}
d\Lambda=dM\circ \Lambda-\Lambda\circ dM+dN
\end{equation}
hence
\begin{equation}\label{eq-dU1}
ie^{i\lambda_j}\circ d\lambda_j=dN_{jj},\quad dM_{jk}=\frac{1}{e^{i\lambda_j}-e^{i\lambda_k}}\circ dN_{jk}\quad \mbox{for $j\not=k$}. 
\end{equation}
Rewriting \eqref{eq-UBM-SDE} in the form of a It\^o's stochastic differential equation we have that
\begin{equation}\label{eq-UBM-Ito}
dU(t)=U(t)dA(t)-nU(t)dt.
\end{equation}
It follows that for all $1\le j,j',k,k'\le n$,
\begin{align*}
{(dU)_{jk}(d\overline{U})_{j'k'}}={2dt}\,\delta_{jj'}\delta_{kk'}.
\end{align*}
We can then compute that
\begin{align*}
\frac{(dN)_{jk}(d\overline{N})_{j'k'}}{2dt}&=\sum_{p,p',\ell,\ell'}V^*_{jp}\overline{V}^*_{j'p'}V_{\ell k}\overline{V}_{\ell' k'}(dU)_{p\ell} (d\overline{U})_{p'\ell'} \\
&=(V^*V)_{jj'} (V^*V)_{k'k}=\delta_{jj'}\delta_{kk'}.
\end{align*}
If we denote by $dE$ the local martingale part of $dN$ and by $dF$ its finite variation part, then from \eqref{eq-UBM-Ito}  and the definitions of $dM$ and $dN$ we know that
\begin{align*}
{dF}&= -(V^*UV) dt+\frac12(dV^* dU V+V^*dUdV)\\
&=-(V^*UV) dt+\frac12 (dMdN-dNdM).
\end{align*}
Note that here, and throughout the rest of this monograph, we use $dXdY$ to denote the quadratic covariation of matrix-valued processes. Namely the $jk$-component of $(dXdY)$ is given by the sum of the quadratic covariations  $(dXdY)_{jk}=\sum_{p}\langle dX_{jp},dY_{pk}\rangle=\sum_{p} dX_{jp}dY_{pk}$. 
We can then compute that
\begin{align}\label{eq-dMdN-a}
(dMdN)_{jk}&=\sum_{p\not=j}\frac{1}{e^{i\lambda_j}-e^{i\lambda_p}}dN_{jp}dN_{pk}+dM_{jj}dN_{jk}.
\end{align}
Moreover, from \eqref{eq-dU1} we can also obtain that for $k\not=p$,
\[
dN_{pk}=\frac{e^{i\lambda_k}-e^{i\lambda_p}}{e^{-i\lambda_k}-e^{-i\lambda_p}}d\overline{N}_{kp}.
\]
Plugging this expression into \eqref{eq-dMdN-a} we obtain that 
\begin{align*}
(dMdN)_{jk}&=\sum_{p\not=j}\frac{1}{e^{i\lambda_j}-e^{i\lambda_p}}\frac{e^{i\lambda_k}-e^{i\lambda_p}}{e^{-i\lambda_k}-e^{-i\lambda_p}}(2dt) \delta_{jk}+dM_{jj}dN_{jk}\\
&=\sum_{p\not=k}\frac{(2dt) \delta_{jk}}{e^{-i\lambda_k}-e^{-i\lambda_p}}+dM_{jj}dN_{jk}.
\end{align*}
Similarly we have
\begin{align*}
(dNdM)_{jk}&=\sum_{p\not=k}\frac{1}{e^{i\lambda_p}-e^{i\lambda_k}}dN_{jp}dN_{pk}+dN_{jk}dM_{kk}\\
&=\sum_{p\not=k}\frac{1}{e^{i\lambda_p}-e^{i\lambda_k}}\frac{e^{i\lambda_k}-e^{i\lambda_p}}{e^{-i\lambda_k}-e^{-i\lambda_p}}(2dt) \delta_{jk}+dN_{jk}dM_{kk}\\
&=-\sum_{p\not=k}\frac{(2dt) \delta_{jk}}{e^{-i\lambda_k}-e^{-i\lambda_p}}+dN_{jk}dM_{kk}.
\end{align*}
Hence
\[
(dMdN-dNdM)_{jk}=4dt\sum_{p\not=k}\frac{ \delta_{jk}}{e^{-i\lambda_k}-e^{-i\lambda_p}}+(dM_{jj}-dM_{kk})dN_{jk}.
\]
Therefore
\begin{align*}
dF_{jk}&=-ne^{i\lambda_j}\delta_{jk} \, dt+2dt\sum_{p\not=k}\frac{ \delta_{jk}}{e^{-i\lambda_k}-e^{-i\lambda_p}}+\frac12(dM_{jj}-dM_{kk})dN_{jk}.
\end{align*}
In particular we have for all $1\le j\le n$ that
\begin{align*}
dF_{jj}&=dt\, \left(-ne^{i\lambda_j}+\sum_{p\not=j}\frac{ 2}{e^{-i\lambda_j}-e^{-i\lambda_p}}\right).
\end{align*}
Moreover we also have
\begin{align*}
dE_{jj}d\overline{E}_{j'j'}&=dN_{jj}d\overline{N}_{j'j'}=\sum_{k,p,m,\ell=1}^nV^*_{jk}(dU)_{kp}V_{pj}\overline{V}^*_{j'm}(d\overline{U})_{m\ell}\overline{V}_{\ell j'}\\
&=2dt\, \delta_{jj'},
\end{align*}
which implies that 
\[
d\langle\lambda_j,\lambda_{j'}\rangle=2dt\, \delta_{jj'}.
\]
 Therefore
\begin{align*}
id\lambda_j&=e^{-i\lambda_j}\circ (dE_{jj}+dF_{jj})\\
&=\sqrt{2}i\,dB_j(t)+dt\, \left(-n+\sum_{p\not=j}\frac{2 e^{-i\lambda_j}}{e^{-i\lambda_j}-e^{-i\lambda_p}}\right)+\frac12\langle de^{-i\lambda_j},ie^{i\lambda_j} d\lambda_j\rangle 
\end{align*}
Noing that
\[
\sum_{p\not=j}\frac{2 e^{-i\lambda_j}}{e^{-i\lambda_j}-e^{-i\lambda_p}}=i\sum_{p\not=j}\left(\cot\frac{\lambda_j-\lambda_p}{2}-i\right),
\]
we end up with
\[
d\lambda_j=\sqrt{2}\,dB_j(t)+dt\, \sum_{p\not=j}\cot\frac{\lambda_j-\lambda_p}{2},
\]
where $(B_j(t))_{t \ge 0}$, $j=1,\dots, n$ is a Brownian motion in $\mathbb{R}^n$.

\end{proof}

The above theorem implies that the eigenvalue process $(\lambda(t))_{t\ge0}$ is a diffusion process with generator 

\begin{equation}\label{eq-generator}
L=\sum_{j=1}^n\frac{\partial^2}{\partial \lambda_j^2}+  \sum_{p<j}\cot\frac{\lambda_j-\lambda_p}{2} \left(\frac{\partial}{\partial \lambda_j}-\frac{\partial}{\partial \lambda_p}\right).
\end{equation}
The next lemma shows that $L$ is a Doob's transform of the Dirichlet Laplacian on $D_n$. More precisely, the following lemma follows from direct computations, see for instance \cite{Faraut}, page 289, for the details.

\begin{lemma}\label{lemma-h}
Let $L$ be as given in \eqref{eq-generator}. Consider the function
\begin{equation}\label{eq-h}
h(\lambda)=\prod_{j<k}\sin\frac{\lambda_j-\lambda_k}{2}.
\end{equation}
Then, for every smooth function $f$ on $D_n$ :
\begin{equation}\label{eq-inter}
Lf= \frac1h\Delta_{D_n}(fh)+\gamma f
\end{equation}
where $\gamma=\sum_{j=1}^n\left(n-j-\frac{n-1}{2}\right)^2=\frac{n^3-n}{12}$ and $\Delta_{D_n}=\sum_{j=1}^n\frac{\partial^2}{\partial \lambda_j^2}$.
\end{lemma}

Note in particular that $h$ satisfies
\[
\Delta_{D_n} h=-\gamma h
\]
and that $h$ is zero on the boundary of $D_n$. The function $h$ is the principal Dirichlet eigenfunction of $\Delta_{D_n}$.


\begin{theorem}\label{thm-non-collide}
Let $\lambda$ be as in Theorem \ref{eigen process u} and $\tau_{D_n}$ be the stopping time  given in \eqref{eq-tau-N1}. Assume that $\lambda(0) \in {D_n}$. Then
\[
\mathbb{P}( \tau_{D_n} <+\infty)=0.
\]
In other words, for all $t\ge0$, $\lambda(t)\in {D_n}$ a.s. Moreover $(\lambda(t))_{t \ge 0}$ is the unique strong solution of the stochastic differential equation  \ref{eq-Un-lambda-sde}.
\end{theorem}

\begin{proof}
Let $h(\lambda)$ be as given in \eqref{eq-h}. Consider the process 
\[
\Omega(t):=V(\lambda_1(t),\dots, \lambda_n(t)),
\]
where $V(\lambda_1,\dots, \lambda_n)=\frac12\log h=\frac12\sum_{j<k}\log \sin\frac{\lambda_j-\lambda_k}{2}$. We can compute that
\begin{align}\label{eq-Omega1}
d\Omega(t)=\sum_{i=1}^n\left(\partial_iV d\lambda_i+\frac12 \partial^2_i V d\langle \lambda_i\rangle\right)
={L}Vdt+dM(t)
\end{align}
where $M(t)$ is a local martingale satisfying $dM(t)=\sqrt{2}\sum_{i=1}^n\partial_i V dB^i_t$ and $L$ is as in \eqref{eq-generator}.
Since $\partial_i h=h\sum_{j\not=i}\cot\frac{\lambda_i-\lambda_j}{2}$ we have that  $\sum_{i=1}^n\partial_i h=0$ and
\[
\sum_{i=1}^n\left(\sum_{\ell\not=i}\cot\frac{\lambda_i-\lambda_\ell}{2}\right) \partial_i V=\frac12\sum_{i=1}^n\left(\sum_{j\not=i}\cot\frac{\lambda_i-\lambda_j}{2}\right)^2=\frac12\sum_{i=1}^n \left(\frac{\partial_i h}{h}\right)^2. 
\]
Moreover, since $\sum_{i=1}^n\partial^2_i h=-\gamma h$ we have that
\[
\sum_{i=1}^n\partial_i^2 V=\sum_{i=1}^n\left(\frac{\partial_i^2h}{2h}-\frac{(\partial_ih)^2}{2h^2}\right)=-\frac{\gamma}{2}-\frac12\sum_{i=1}^n \left(\frac{\partial_i h}{h}\right)^2,
\]

where $\gamma=\frac{n^3-n}{12}$ is as given in Lemma \ref{lemma-h}. Therefore
\[
{L}V=-\frac{\gamma}2.
\]
Plugging it back into \eqref{eq-Omega1} we have for any $t\ge0$,
\[
M(t)=\Omega(t)-\Omega(0)+\frac{\gamma}{2}t.
\]
On $\{\tau_{D_n}<+\infty\}$, by letting $t\to \tau_{D_n}$ we have the right hand side of the above inequality goes to $-\infty$. This implies that $M({\tau_{D_n}})=-\infty$. However, since $M(t)$ is a time changed Brownian motion, we infer that $\{\tau_{D_n}<+\infty\}$ is a null set. We conclude
\[
\mathbb{P}( \tau_{D_n} <+\infty)=0.
\]
The existence and uniqueness of the strong solution to the stochastic differential equations \eqref{eq-Un-lambda-sde} are addressed in \cite[Corollary 2.4]{cepa2001brownian}. 
\end{proof}

Next we compute the density function of $\lambda(t)$, $t\ge0$. 

\begin{theorem}\label{density rho}
Let $\lambda$ be the eigenvalues process as given before. Assume that $\lambda(0)\in D_n$.
The density at time $t>0$ of $\lambda(t)$  with respect to the Lebesgue measure  on $\R^n$ is given by
\begin{equation}\label{eq-kernel}
\frac{e^{\gamma t} }{n(2\pi)^n} \frac{ h(x)}{ h(\lambda(0))} \sum_{K\in\mathcal{Z}_n}  \exp\left(-\sum_{j=1}^n\frac{k_j^2}{2n^2}t \right)\det \left(\exp\left( i \frac{k_j(\lambda_j(0)-x_k)}{n}\right)\right)_{1\le k,j \le n} 1_{D_n} (x),
\end{equation}
where $\mathcal Z_n=\{K=(k_1,\dots, k_n), k_1\equiv \cdots \equiv k_n \mod n\}$.
\end{theorem}

\begin{proof}
Let $X(t)$, $t\ge0$ be the $n$-dimensional Brownian motion killed at the boundary of $D_n$. Its generator is $\frac12\Delta_{D_n}$ with the Dirichlet boundary conditions.   Let $p_t$ be the transition density of $X(t)$. It is known from \cite{hobson1996non} that for $x,y \in D_n$
\begin{equation}\label{eq-Diri-kernel}
p_t(x,y)=\sum_{\sigma\in\mathcal{S}_n}\sum_{K\in\mathcal{Z}_n} \frac{ \mathrm{sgn}(\sigma)}{n(2\pi)^n}\exp\left(-\sum_{j=1}^n\frac{k_j^2}{2n^2}t \right)\exp\left( i\sum_{j=1}
^n \frac{k_j(x_j-y_{\sigma(j)})}{n}\right)
\end{equation}
where $\mathcal Z_n=\{K=(k_1,\dots, k_n), k_1\equiv \cdots \equiv k_n \mod n\}$. 
From the intertwining of generators \eqref{eq-inter} we obtain that
\[
 e^{\frac{t{L}}2} f (x) =  e^{\frac{\gamma}2 t} \frac{1}{h(x)} e^{\frac12t\Delta_{D_n} } ( h f) (x)
\]
for any smooth function  $f$ on $D_n$. If we denote by $q_t(x, y)$ the heat kernel associated to the semigroup $ e^{\frac{t{L}}2}$, then the above equality reads as
\begin{align*}
  \int_{D_n} q_t(x, y)  f(y) \, dy 
 =  e^{\frac{\gamma t}2}  \frac{1}{h(x)}  \int_{D_n} h(y) p_{t} (x,y)  f(y) \, dy. 
\end{align*}
Hence we have
\[
q_t(x, y)= e^{\frac{\gamma}2 t}  \frac{h(y)}{h(x)}  p_{t} (x,y). 
\]
By plugging \eqref{eq-Diri-kernel} into the above equality we obtain that
\begin{align*}
q_t(x,y)&= \frac{e^{\frac{\gamma}2 t} }{h(x)} \frac{ 1}{n(2\pi)^n} h(y)\sum_{K\in\mathcal{Z}}  \exp\left(-\sum_{j=1}^n\frac{k_j^2}{2n^2}t \right)\sum_{\sigma\in\mathfrak{S}_n} \mathrm{sgn}(\sigma)  \exp\left( i\sum_{j=1}
^n \frac{k_j(x_j-y_{\sigma(j)})}{n}\right)\\
&=\frac{e^{\frac{\gamma}2 t} }{n(2\pi)^n} \frac{ h(y)}{ h(x)} \sum_{K\in\mathcal{Z}}  \exp\left(-\sum_{j=1}^n\frac{k_j^2}{2n^2}t \right)\det \left(\exp\left( i \frac{k_j(x_j-y_k)}{n}\right)\right)_{1\le k,j \le n}
\end{align*}
The conclusion follows.
\end{proof}

Next we deduce the limit law of $\lambda$.

\begin{theorem}\label{limit eigenvalue}
Let $\lambda$ be the eigenvalues process as given previously. Assume that
\[
\lambda_1(0)<\cdots< \lambda_n(0).
\]
Then, when $t \to +\infty$, $\lambda(t)$ converges in distribution to the probability measure on $\R^n$  given by
\[
 d\nu=a h(x)^2   \, \, \mathbf{1}_{D_n} (x) dx,
\]
where $a>0$ is the normalization constant. Moreover, we have the following quantitative estimate: There exists a constant $C>0$ such that for any bounded Borel function $f$ on  $D_n$ and $t \ge 0$,
\begin{equation}\label{quant_1}
\left | \mathbb{E} ( f( \lambda(t)) ) - \int_{D_n} f d\nu \right| \le C e^{-\frac{t}{2n}} \| f \|_\infty.
\end{equation}
\end{theorem}

\begin{proof}
From \eqref{eq-kernel} we know that  when $t \to +\infty$, the term of leading order in this sum corresponds to  $(k_1,\dots,k_n)=(-\frac{n(n-1)}{2}, -\frac{n(n-3)}{2},\dots, \frac{n(n-1)}{2})$ either $n$ is even or odd. Hence we have that 
\[
\sum_{j=1}^n\frac{k_j^2}{2n^2}=\frac{n^3-n}{24}=\frac{\gamma}2.
\]
Moreover, by using a Vandermonde determinant we have that
\begin{align*}
 \mathrm{det} \left(e^{-i\frac{k_j}n x_k} \right)_{1 \le j,k \le n} &= e^{i\frac{n-1}{2}\sum_{j=1}^nx_j} \prod_{j< \ell} \left( e^{-ix_j}-e^{-ix_\ell}\right)\\
 &=(2i)^{\frac{n(n-1)}{2}} \prod_{j<\ell} \sin\frac{x_\ell-x_j}{2}\\
 &=  (-2i)^{\frac{n(n-1)}{2}}h(x).
 \end{align*}
and 
\begin{align*}
 \mathrm{det} \left( e^{i\frac{k_j}n\lambda_{k} (0)} \right)_{1 \le j,k \le n}&= e^{-i\frac{n-1}{2}\sum_{j=1}^n\lambda_j(0)} \prod_{j< \ell} \left( e^{i\lambda_j(0)}-e^{i\lambda_\ell(0)}\right)\\
 &=  (2i)^{\frac{n(n-1)}{2}}h(\lambda(0)).
 \end{align*}

Therefore the leading term writes, up to a constant, as 
$
   h(x)^2.
$
 The next order in $t$  corresponds to   $(k_1,\cdots, k_n)=(-\frac{n(n-1)}{2}+1, -\frac{-n(n-3)}{2}+1, \dots, \frac{n(n-1)}2+1)$  which yields $e^{-\frac1{2n}t}$  in \eqref{quant_1}.
\end{proof}

To conclude let us note that the heat kernel with respect to the Haar measure of the unitary Brownian can be computed using similar methods. Indeed, the heat kernel issued from 0 is a central function and it is therefore enough to compute its restriction to the subgroup of unitary diagonal matrices (which is a torus). The restriction of the heat kernel to this subgroup can then be computed as in Theorem \eqref{density rho}. We refer to Proposition 12.6.3 in \cite{Faraut} for this explicit formula. 

More generally, we point out that formulas for heat kernels on some Lie groups and homogeneous spaces might be obtained using representation theory; see for instance \cite{Faraut,MR0404526,MR2372464,MR3052688,MR2018351,MR1168067}.

\chapter{Horizontal Brownian motions}\label{sec-hBM}

The goal of the Chapter is to present, in a general framework, the notion of horizontal Brownian motion associated with a Riemannian submersion or Riemannian foliation. We refer to Chapter 8 in  \cite{Besse} and the monograph \cite{MR2110043} for additional reading regarding the theory of Riemannian submersions and to the monograph \cite{MR2731662} for additional reading on the stochastic analysis  and geometry of horizontal Brownian motions.

\section{Riemannian submersions and horizontal Brownian motions}\label{Section Riemannian submersion}

We first define the notion of horizontal Brownian motion on a Riemannian manifold. For this, we need to distinguish a \textit{particular} set of directions within the tangent spaces of a Riemannian manifold. This can be done by using the notion of Riemannian submersion. 

\

Let $(\M , g)$  and  $(\B,j)$ be  smooth and connected complete Riemannian manifolds. Most of the results below hold when $g$ is only assumed to be semi-Riemannian with signature $(\dim \B, \dim \M -\dim \B)$, but $j$ will  always  be Riemannian.

\begin{definition}
A smooth surjective map $\pi: (\M , g)\to (\B,j)$ is called a Riemannian submersion if its derivative maps $T_x\pi : T_x \M \to T_{\pi(x)} \B$ are orthogonal projections, i.e. for every $ x \in \M$, the map $ T_{x} \pi (T_{x} \pi)^*: T_{p(x)}  \B \to T_{p(x)} \B$ is the identity.
\end{definition}

\begin{Example}(\textbf{Warped products}) Let $(\M_1 , g_1)$  and  $(\M_2,g_2)$ be   Riemannian manifolds and $f$ be a smooth and positive function on $\M_1$. Then the first projection $(\M_1 \times \M_2,g_1 \oplus f g_2) \to (\M_1, g_1)$ is a Riemannian submersion.
\end{Example}

\begin{Example}\label{example iso 4}(\textbf{Quotient by an isometric action})
Let $(\M , g)$ be a Riemannian manifold and $\mathbf G$ be a closed subgroup of the isometry group of $(\M , g)$. Assume that the projection map $\pi$ from $\M$ to the quotient space $\M /\G$ is a smooth submersion. Then there exists a unique Riemannian metric $j$ on $\M /\G$ such that $\pi$ is a Riemannian submersion. We refer to 9.12 in \cite{Besse} for further details.
\end{Example}

If $\pi$ is a Riemannian submersion and $b \in \B$, the set  $\pi^{-1}(\{ b \})$ is called a fiber. 

\

For $ x \in \M$, $\mathcal{V}_x =\mathbf{Ker} (T_x\pi)$ is called the vertical space at $x$. The orthogonal complement of $\mathcal{V}_x$ shall be denoted by $\mathcal{H}_x$ and will be referred to as the horizontal space at $x$. We have an orthogonal decomposition
\[
T_x \M=\mathcal{H}_x \oplus \mathcal{V}_x
\]
and a corresponding splitting of the metric
\[
g=g_{\mathcal{H}} \oplus g_{\mathcal{V}}.
\]
The vertical distribution $\mathcal V$ is of course integrable since it is the tangent distribution to the fibers, but the horizontal distribution is in general not integrable. Actually, in almost all the situations we will consider,  the horizontal distribution is everywhere bracket-generating in the sense that for every $x \in \M$, $\mathbf{Lie} (\mathcal{H}) (x)=T_x \M$. 

\

 If $f \in C^\infty(\M)$ we define its vertical gradient $\nabla_{\mathcal{V}}$ as the projection of its gradient onto the vertical distribution and its horizontal gradient $\nabla_{\mathcal{H}}$ as the projection of the gradient onto the horizontal distribution. We define then the vertical Laplacian  $\Delta_{\mathcal{V}}$ as the generator of the pre-Dirichlet form
\[
\mathcal{E}_{\mathcal{V}}(f_1,f_2)=-\int_\M \langle \nabla_{\mathcal{V}} f_1 , \nabla_{\mathcal{V}} f_2 \rangle d\mu, \quad f_1,f_2 \in C_0^\infty(\mathbb{M}),
\]
where $\mu$ is the Riemannian volume measure on $\M$. Similarly, we define  the horizontal Laplacian  $\Delta_{\mathcal{H}}$ as the generator of the pre-Dirichlet form
\[
\mathcal{E}_{\mathcal{H}}(f_1,f_2)=-\int_\M \langle \nabla_{\mathcal{H}} f_1 , \nabla_{\mathcal{H}} f_2 \rangle d\mu \quad 
f_1, f_2 \in C_0^\infty(\mathbb{M}).
\]

We note that from the definitions, if $\Delta$ denotes the Laplace-Beltrami operator on $(\M,g)$ then
\begin{align}\label{decomposition laplace}
\Delta=\Delta_{\mathcal H}+ \Delta_{\mathcal V}
\end{align}

\begin{definition}
A horizontal Brownian motion $(X(t))_{ t \ge 0}$ on $\mathbb{M}$ is a diffusion process with generator $\frac{1}{2} \Delta_{\mathcal{H}}$, that is for every $f \in C^\infty(\mathbb{M})$,
\[
f(X(t))-\frac{1}{2} \int_0^t \Delta_{\mathcal{H}} f(X(s)) ds, \quad 0 \le t < \mathbf{e}
\]
is a  local martingale, where $\mathbf{e}$ is the lifetime of $(X(t))_{ t \ge 0}$ on $\mathbb{M}$.
\end{definition}

If $\X_1,\dots,\X_n$ is a local orthonormal frame of horizontal vector fields and $Z_1,\dots,Z_m$ a local orthonormal frame of the vertical distribution, then we have
\[
\Delta_{\mathcal{H}}=-\sum_{i=1}^n \X_i^* \X_i
\]
and
\[
\Delta_{\mathcal{V}}=-\sum_{i=1}^m Z_i^* Z_i,
\]
where the adjoints are (formally) understood in $L^2(\M,\mu)$. Classically, we have
\[
\X_i^*=-\X_i+\sum_{k=1}^n g\left( D_{\X_k} \X_k, \X_i\right) 
+\sum_{k=1}^m g\left(D_{Z_k} Z_k, \X_i\right),
\]
where $D$ is the Levi-Civita connection. As a consequence, we obtain
\[
\Delta_{\mathcal{H}}=\sum_{i=1}^n \X_i^2 -\sum_{i=1}^n (D_{\X_i}\X_i)_{\mathcal{H}} -\sum_{i=1}^m (D_{Z_i}Z_i)_{\mathcal{H}},
\]
where $(\cdot)_{\mathcal{H}}$ denotes the horizontal part of the vector. In a similar way we  have
\[
\Delta_{\mathcal{V}}=\sum_{i=1}^m Z_i^2 -\sum_{i=1}^n (D_{\X_i}\X_i)_{\mathcal{V}} -\sum_{i=1}^m (D_{Z_i}Z_i)_{\mathcal{V}}.
\]

\

We  note that from H\"ormander's theorem, the operator $\Delta_{\mathcal{H}}$ is locally subelliptic if the horizontal distribution is bracket generating. Of course, the vertical Laplacian is never subelliptic because the vertical distribution is always integrable. We have the following theorem (see \cite{MR4432545} and Section 5.1 in \cite{MR3587668}).

\begin{proposition}
Assume that the horizontal distribution $\mathcal{H}$ is everywhere bracket-generating. The horizontal Laplacian $\Delta_{\mathcal{H}}$ is essentially self-adjoint in $L^2(\M,\mu)$ on the space $C_0^\infty(\M)$. In particular, the horizontal Brownian motion on $\M$ can be constructed as the Hunt process associated with the unique closed extension of the pre-Dirichlet form $\mathcal{E}_{\mathcal{H}}$.
\end{proposition}

In all the situations described in the present monograph  the  horizontal distribution $\mathcal{H}$ will be everywhere bracket-generating and therefore the existence of the horizontal Brownian motion is ensured by the previous theorem. As already pointed out, the bracket generating condition implies from H\"ormander's theorem that the horizontal Laplacian is a locally subelliptic operator (see Section 2.2 in \cite{MR4432545} for a discussion on local subellipticity). It follows then from essential self-adjointness and local subellipticity that $\Delta_{\mathcal H}$ generates a semigroup which uniquely solves the heat equation and  admits a smooth kernel. More precisely, we have the following results which are explained in details in  Section 2.5 of \cite{MR4432545}).

If $\Delta_\mathcal{H}=-\int_0^{+\infty} \lambda dE_\lambda$ denotes the spectral
decomposition of $\Delta_\mathcal{H}$ in $L^2 (\bM,\mu)$, then by definition, the
heat semigroup $(P_t)_{t \ge 0}$ is given by $P_t= \int_0^{+\infty}
e^{-\lambda t/2} dE_\lambda$. It is a one-parameter family of bounded operators on
$L^2 (\bM,\mu)$. Since the quadratic form $\mathcal{E}_\mathcal{H}$ is a Dirichlet form, we deduce 
that $(P_t)_{t \ge 0}$ is a sub-Markov semigroup: it transforms non-negative functions into non-negative functions and satisfies
\begin{equation*}
P_t 1 \le 1.
\end{equation*}
The sub-Markov property and  Riesz-Thorin interpolation classically allows to construct the semigroup $(P_t)_{t \ge 0}$ in $L^p(\M,\mu)$ and for $f \in L^p(\M,\mu)$ one has
\begin{equation*}
||P_tf||_{L^p(\bM , \mu)} \le ||f||_{L^p(\bM, \mu)},\ \  1\le p\le \infty.
\end{equation*}

\begin{theorem}[Horizontal heat kernel]
There is a smooth function $p(t,x,y)$, $t \in (0,+\infty), x,y \in \mathbb{M}$, such that for every $f \in L^p(\M,\mu)$, $1 \le p \le \infty$ and $x \in \mathbb{M}$ ,
\[
P_t f (x)=\int_{\mathbb{M}} p(t,x,y) f(y) d\mu (y).
\]
The function $p(t,x,y)$ is called the horizontal heat kernel associated to $(P_t)_{t \ge 0}$. It satisfies furthermore:
\begin{enumerate}
\item (Symmetry) $p(t,x,y)=p(t,y,x)$;
\item (Chapman-Kolmogorov relation) $p(t+s,x,y)=\int_{\mathbb{M}} p(t,x,z)p(s,z,y)d\mu(z)$. 
\end{enumerate}
Moreover, for $f \in L^p(\M,\mu)$, $1 < p < \infty$, the function
\[
u (t,x)= P_t f (x), \quad t \ge 0, x\in \mathbb{M}.
\]
is the unique solution of the Cauchy problem
\[
\frac{\partial u}{\partial t}=\frac{1}{2} \Delta_{\mathcal H} u,\quad u (0,x)=f(x).
\]
\end{theorem}

The semigroup  $(P_t)_{t \ge 0}$ is also  the transition semigroup of the horizontal Brownian motion $(X(t))_{t\ge 0}$ meaning that for every $f \in L^\infty (\M,\mu)$,
\[
\mathbb{E}\left( f(X(t)) \mid X(0)=x\right)=P_tf(x).
\]

Apart from continuity, this Dirichlet form construction provides little information about pathwise properties of the process. In particular, it is not apparent that the horizontal Brownian motion is a semimartingale. For this reason, we present in the next section a stochastic differential equation construction of the horizontal Brownian motion in a more general context that shows it is a semimartingale. 

In our submersion situation, one can also often explicitly construct the horizontal Brownian motion on $\mathbb{M}$ from the Brownian motion on $\mathbb{B}$. It uses the notion of horizontal lift that we present below.

\

A smooth vector field $\X$ is said to be projectable if there exists a smooth vector field $\overline{\X}$ on $\B$ such that for every $x \in \M$,  $T_x \pi ( \X(x))= \overline {\X} (\pi (x))$. In that case, we say that $\X$ and $\overline{\X}$ are $\pi$-related.

\begin{definition}
A vector field $\X$ on $\M$ is called basic if it is projectable and horizontal.
\end{definition}

If  $\overline{\X}$ is a smooth vector field on $\B$, then there exists a unique basic vector field $\X$ on $\M$ which is $\pi$-related to $\overline{\X}$. This vector is called the horizontal lift of $\overline{\X}$.

\

A $C^1$-curve $\gamma: [0,+\infty) \to \M$ is said to be horizontal if for every $t \ge 0$,
\[
\gamma'(t) \in \mathcal{H}_{\gamma(t)}.
\]
\begin{definition}
Let $\bar{\gamma}: [0,+\infty) \to \mathbb{B}$ be a $C^1$ curve. Let $x \in \mathbb{M}$, such that $\pi(x)=\bar{\gamma}(0)$. Then, there exists a unique $C^1$ horizontal curve $\gamma: [0,+\infty) \to \M$ such that $\gamma (0)=x$ and $\pi (\gamma(t))=\bar{\gamma}(t)$. The curve $\gamma$ is called the horizontal lift of $\bar{\gamma}$ at $x$.
\end{definition}

One can locally describe horizontal lifts using orthonormal frames of basic vector fields. Let $\X_1,\dots,\X_n$ be a local orthonormal frame of basic vector fields around $x$ and  let us denote by $\overline{\X}_1,\dots,\overline{\X}_n$ the vector fields on $\B$ which are $\pi$-related to $\X_1,\dots,\X_n$. One can then write
\[
\bar{\gamma}'(t)=\sum_{i=1}^n \bar{\gamma}_i'(t) \overline{\X}_i ( \bar{\gamma} (t) )
\]
where $\bar{\gamma}_i'(t) = j\left( \bar{\gamma}'(t),  \overline{\X}_i ( \bar{\gamma} (t) ) \right)$. Solving then the differential equation 
\[
\gamma'(t)=\sum_{i=1}^n \bar{\gamma}_i'(t)\X_i (\gamma (t) ), \quad \gamma(0)=x
\]
on $\M$ yields the horizontal lift of $\bar{\gamma}$ at $x$.

The notion of horizontal lift may   be extended to semimartingales on $\mathbb{B}$ by using stochastic calculus. A semimartingale $M(t)$ on $\M$ is called horizontal if for every one-form $\theta$ on $\M$ whose kernel contains the horizontal distribution $\mathcal H$, one has
\[
\int_{M[0,t]} \theta =0, \quad t \ge 0.
\]

\begin{definition}\label{def lift}
Let $(\overline{M}(t))_{t \ge 0}$ be a semimartingale such that $\overline{M}(0)=\overline{x}\in \B$. Let $x \in \mathbb{M}$, such that $\pi(x)=\overline{x}$. Then, there exists a unique horizontal semimartingale $(M(t))_{t \ge 0}$ on $\M$ such that $M (0)=x$ and $\pi (M(t) )=\overline{M}(t)$, $t \ge 0$. The semimartingale $M$ is called the horizontal lift of $\overline{M}$ at $x$.
\end{definition}

As before, one can locally describe horizontal lifts of semimartingales  using local orthonormal frames of basic vector fields. One can write
\[
d\overline{M}(t) =\sum_{i=1}^n \overline{\X}_i ( \overline{M} (t) )\circ dM^i(t)
\]
for some real valued semimartingales $M^1,\dots,M^n$. Solving then the stochastic differential equation 
\[
dM(t)=\sum_{i=1}^n \X_i (M (t) ) \circ dM^i(t), \quad M(0)=x
\]
on $\M$ yields the horizontal lift of $\overline{M}$ at $x$.

\begin{theorem}
Assume that the fibers of the submersion $\pi$  have all zero mean curvature. Let $(B(t))_{t \ge 0}$ be a Brownian motion on $\mathbb{B}$ started at $b \in \mathbb{B}$. Let $x \in \mathbb{M}$ such that $\pi (x)=b$. The horizontal lift $(X(t))_{t \ge 0}$ of $(B(t))_{t \ge 0}$ at $x$ is a horizontal Brownian motion.
\end{theorem}

\begin{proof}
Indeed, if $\X_1,\dots,\X_n$ is a local orthonormal frame of basic vector fields and $Z_1,\dots,Z_m$ is a local orthonormal frame of the vertical distribution,  let us denote by $\overline{\X}_1,\dots,\overline{\X}_n$ the vector fields on $\B$ which are $\pi$-related to $\X_1,\dots,\X_n$ . We have
\[
\Delta_{\B}=\sum_{i=1}^n \overline{\X}_i^2 -\sum_{i=1}^n D_{\overline{\X}_i}\overline{\X}_i.
\]
Therefore,  $(B(t))_{t \ge 0}$ locally solves a stochastic differential equation
\[
dB(t)=\sum_{i=1}^n \overline{\X}_i(B(t)) \circ dB^i(t) -\frac{1}{2} \sum_{i=1}^nD_{\overline{\X}_i} \overline{\X}_i (B(t)) dt,
\]
where $B^1,...,B^n$ is a Brownian motion in $\mathbb{R}^n$. Since it is easy to check that  $ D_{\overline{\X}_i}\overline{\X}_i$ is $\pi$-related to $(D_{\X_i}\X_i)_{\mathcal{H}}$, we deduce that $(X(t))_{t \ge 0}$ locally solves the stochastic differential equation
\[
dX(t)=\sum_{i=1}^n \X_i(X(t)) \circ dB^i(t) -\frac{1}{2} \sum_{i=1}^n (D_{\X_i}\X_i)_{\mathcal{H}} (X(t)) dt.
\]
We now recall that
\[
\Delta_{\mathcal{H}}=\sum_{i=1}^n \X_i^2 -\sum_{i=1}^n (D_{\X_i}\X_i)_{\mathcal{H}} -\sum_{i=1}^m (D_{Z_i}Z_i)_{\mathcal{H}}.
\]
If the fibers of the submersion $\pi$  have all zero mean curvature, the vector
\[
T=\sum_{i=1}^m D_{Z_i}Z_i
\]
is always orthogonal to $\mathcal{H}$. Thus 
\[
\Delta_{\mathcal{H}}=\sum_{i=1}^n \X_i^2 -\sum_{i=1}^n (D_{\X_i}\X_i)_{\mathcal{H}} ,
\]
and $(X(t))_{t \ge 0}$ is a horizontal Brownian motion.
\end{proof}

\begin{remark}\label{harmonic submersion}
A similar proof shows that if the fibers of the submersion $\pi$ all have  zero mean curvature and if $(X(t))_{t \ge 0}$ is a Brownian motion on $\M$, then $B(t) =\pi (X(t))$ is a Brownian motion on $\mathbb{B}$.
\end{remark}

In this monograph, we shall almost always consider submersions with totally geodesic fibers.

\begin{definition}
A Riemannian submersion $\pi: (\M , g)\to (\B,j)$ is said to be totally geodesic if for every $b \in \B$, the set $\pi^{-1}(\{ b \})$ is a totally geodesic submanifold of $\M$.
\end{definition}

Observe that for totally geodesic submersions, the mean curvature of the fibers are zero, and thus the horizontal Brownian motion may be constructed as a lift of the Brownian motion on the base space and  the Brownian motion on the base space may be constructed as the projection of the Brownian motion on the top space. 
Here are some important examples of totally geodesic Riemannian submersions.

\begin{Example}(\textbf{Submersions from spheres})
R. Escobales proved in \cite{Escobales} (see also Section 2.3 in \cite{MR2110043})  that, up to equivalence, the only Riemannian submersions with connected totally geodesic fibers from a unit sphere are given by:
\begin{itemize}
\item \hspace{.1in}The complex Hopf fibrations:
\[
\mathbb{S}^1 \rightarrow \mathbb{S}^{2n+1} \rightarrow {\mathbb{C}P}^n.
\]
\item \hspace{.1in} The quaternionic Hopf fibrations:
\[
\mathbb{S}^3 \rightarrow \mathbb{S}^{4n+3} \rightarrow {\mathbb{H}P}^n.
\]
\item \hspace{.1in}The octonionic Hopf fibration:
\[
\mathbb{S}^7 \rightarrow \mathbb{S}^{15} \rightarrow {\mathbb{O}P}^1.
\]
\end{itemize}
The complex Hopf fibrations will be studied in Chapter \ref{sec-BM-complex}. The quaternionic Hopf fibrations will be studied in Chapter \ref{Chapter quaternionic hopf}. For the octonionic Hopf fibration we refer to Chapter \ref{Chapter octonionic hopf}.
  
G. Baditoiu and S. Ianus proved in \cite{Baditoiu} (see also Section 7.4 in \cite{MR2110043}) that, up to equivalence, the only Riemannian submersions with connected totally geodesic fibers from a unit pseudo-sphere are given by:
\begin{itemize}
\item \hspace{.1in} The complex anti-de Sitter fibrations:
\[
\mathbb{S}^1\to \mathbf{AdS}^{2n+1}(\mathbb{C})\to\mathbb{C}H^n.
\]
\item \hspace{.1in} The quaternionic anti-de Sitter fibrations:
\[
\mathbb{S}^3 \rightarrow \mathbf{AdS}^{4n+3}(\mathbb{H}) \rightarrow {\mathbb{H}}H^n.
\]
\item \hspace{.1in} The octonionic anti-de Sitter fibration:
\[
\mathbb{S}^7 \rightarrow \mathbf{AdS}^{15}(\mathbb{O}) \rightarrow {\mathbb{O}}H^1.
\]
\end{itemize}
Note that in those examples the metric on $\mathbf{AdS}$ is semi-Riemannian: the metric is positive definite on the horizontal space and negative definite on the vertical space. The complex anti-de Sitter fibrations will be studied in Chapter \ref{sec-BM-complex}. The quaternionic anti-de Sitter fibrations will be studied in Chapter \ref{chap-quater-anti}. For the octonionic anti-de Sitter  fibration we refer to Chapter \ref{Chapter octonionic hopf}.
\end{Example}

\begin{Example}(\textbf{Quotient by an isometric action})
As in Example  \ref{example iso 4} let $(\M , g)$ be a Riemannian manifold and $\mathbf G$ be a closed one-dimensional subgroup of the isometry group of $(\M , g)$ which is generated by a complete Killing vector field $Z$. Assume that the projection map $\pi$ from $\M$ to $\M /\mathbf{G}$ is a smooth submersion.  Then the fibers are totally geodesic if and only if the integral curves of $Z$ are geodesics, which is the case if and only if $Z$ has a constant length.
\end{Example}

\begin{Example}\label{principal bundle def} (\textbf{Principal bundle})
Let $\M$ be a principal bundle over $\B$ with fiber Lie group $\G$. Then, given a Riemannian metric $j$ on $\B$, a $\G$-invariant  metric $k$ on $ \G$ and a $\G$ compatible connection form $\theta$, there exists a unique Riemannian metric $g$ on $\M$ such that the bundle projection map $\pi: \M \to \B$ is a Riemannian submersion with totally geodesic fibers isometric to $(\G,k)$ and such that the horizontal distribution of $\theta$ is the orthogonal complement of the vertical distribution. We refer to \cite{Vilms}, page 78, for a proof. This class of examples will be studied in details in Section \ref{horizontal BM bundle}.
\end{Example}

\begin{Example}(\textbf{B\'erard-Bergery fibration})\label{BB fibration}
Let $\G$ be a Lie group and $\mathbf H, \mathbf K$ be two compact subgroups of $\G$ with $\mathbf K \subset \mathbf H$. Then we have the natural fibration
\[
\mathbf H/\mathbf K \to \G/\mathbf K \to \G/\mathbf H
\]
with projection map $\pi: \G/\mathbf K \to \G/\mathbf H, \, \alpha \mathbf K \mapsto \alpha \mathbf H$. Let $\mathfrak g$ be the Lie algebra of $\G$ and $\mathfrak h, \mathfrak k$  be the Lie algebras of $\mathbf H, \mathbf K$. We have
\[
\mathfrak k \subset \mathfrak h \subset \mathfrak g.
\]
We choose an $\mathbf{Ad}_\G (\mathbf H)$-invariant complement $\mathfrak m$ to $\mathfrak h$ in $\mathfrak g$ and an  $\mathbf{Ad}_\G (\mathbf K)$-invariant complement $\mathfrak p$ to $\mathfrak k$ in $\mathfrak h$. We have $\mathfrak g= \mathfrak m \oplus \mathfrak h$, $\mathfrak h= \mathfrak p \oplus \mathfrak k$, and $\mathfrak p \oplus \mathfrak m$ is an $\mathbf{Ad}_\G (\mathbf K)$-invariant complement to $\mathfrak k$ in $\mathfrak g$. An $\mathbf{Ad}_\G (\mathbf H)$-invariant inner product on $\mathfrak m$ defines a $\G$-invariant Riemannian metric $j$ on $\G/\mathbf H$ and an $\mathbf{Ad}_\G (\mathbf K)$-invariant inner product on $\mathfrak p$ defines a $\mathbf H$-invariant Riemannian metric $k$ on $\mathbf H/\mathbf K$. The orthogonal sum of those inner products defines a $\G$-invariant  Riemannian metric $g$ on $\G / \mathbf K$.

The projection map $\pi: (\G/\mathbf K, g) \to (\G/\mathbf H, j)$ is then a Riemannian submersion with totally geodesic fibers isometric to $(\mathbf H/\mathbf K ,k)$, see Theorem 9.80 in \cite{Besse}. Equipped with those metrics, the Riemannian fibration
\[
\mathbf H/\mathbf K \to \G/\mathbf K \to \G/\mathbf H
\]
is then called the B\'erard-Bergery fibration, see \cite{MR370432}.  Notice that we have then a (commutative) diagram of fibrations:

\[
  \begin{tikzcd}
    & \mathbf{K} \arrow[swap]{dl} \arrow{d} & \\
    \mathbf H \arrow{r} \arrow{d}  & \G  \arrow{d} \arrow[swap]{dr} &  \\
    \mathbf H / \mathbf{K}  \arrow{r} & \G / \mathbf{K} \arrow{r} & \G / \mathbf{H}  
  \end{tikzcd}
\]

The following, all  of them appearing in the monograph, are examples of such fibrations and yield our most important examples of horizontal Brownian motions generating interesting stochastic area functionals:

\begin{itemize}
\item \hspace{.1in} $\G=\mathbf{U}(n)$, $\mathbf K=\mathbf{U}(n-1)$, $\mathbf H=\mathbf{U}(n-1)\mathbf{U}(1)$. In that case $\G/\mathbf H$ is the complex projective space $\mathbb{C}P^{n-1}$, and the B\'erard-Bergery fibration reduces to the classical  complex Hopf fibration $\mathbf{U}(1) \to \mathbb{S}^{2n-1} \to \mathbb{C}P^{n-1}$. 
\item \hspace{.1in} $\G=\mathbf{U}(n-1,1)$, $\mathbf K=\mathbf{U}(n-1)$, $\mathbf H=\mathbf{U}(n-1)\mathbf{U}(1)$. In that case $\G/\mathbf H$ is the complex hyperbolic space $\mathbb{C}H^{n-1}$, and the B\'erard-Bergery fibration is the anti-de Sitter fibration  $\mathbf{U}(1) \to \mathbf{AdS}^{2n-1} \to \mathbb{C}H^{n-1}$. 
\item \hspace{.1in} $\G=\mathbf{Sp}(n)$, $\mathbf K=\mathbf{Sp}(n-1)$, $\mathbf H=\mathbf{Sp}(n-1)\mathbf{Sp}(1)$. In that case $\G/\mathbf H$ is the quaternionic projective space $\mathbb{H}P^{n-1}$, and the B\'erard-Bergery fibration is the quaternionic  Hopf fibration $\mathbf{Sp}(1) \to \mathbb{S}^{4n-1} \to \mathbb{H}P^{n-1}$. 
\item \hspace{.1in} $\G=\mathbf{Sp}(n-1,1)$, $\mathbf K=\mathbf{Sp}(n-1)$, $\mathbf H=\mathbf{Sp}(n-1)\mathbf{Sp}(1)$. In that case $\G/\mathbf H$ is the quaternionic hyperbolic space $\mathbb{H}H^{n-1}$, and the B\'erard-Bergery fibration is the quaternionic  anti-de Sitter fibration $\mathbf{Sp}(1) \to \mathbf{AdS}_\mathbb{H}^{4n-1} \to \mathbb{H}H^{n-1}$. 
\item \hspace{.1in} $\G=\mathbf{Spin}(9)$, $\mathbf K=\mathbf{Spin}(7)$, $\mathbf H=\mathbf{Spin}(8)$. In that case $\G/ \mathbf H$ is the octonionic projective line $\mathbb{O}P^{1} \simeq \mathbb{S}^8$, and the B\'erard-Bergery fibration is the octonionic  Hopf fibration $\mathbb{S}^7 \to \mathbb{S}^{15} \to \mathbb{S}^8$. 
\item \hspace{.1in} $\G=\mathbf{U}(n)$, $\mathbf K=\mathbf{U}(n-k)$, $\mathbf H=\mathbf{U}(n-k)\mathbf{U}(k)$ with $k \ge 1$. In that case $\G/\mathbf H$ is the complex Grassmannian $G_{n,k}$ and the B\'erard-Bergery fibration, the Stiefel fibration.
\item \hspace{.1in} $\G=\mathbf{U}(n-k,k)$, $\mathbf K=\mathbf{U}(n-k)$, $\mathbf H=\mathbf{U}(n-k)\mathbf{U}(k)$ with $k \ge 1$. In that case $\G/\mathbf H$ is the complex hyperbolic Grassmannian  and the B\'erard-Bergery fibration, the hyperbolic Stiefel fibration.
\end{itemize}

\end{Example}
\

Let $\pi$ be a Riemannian submersion. Notice that if $\X$ is a basic vector field and $Z$ is a vertical vector field, then $T_x\pi ( [\X,Z](x))=0$ and thus $[\X,Z]$ is a vertical vector field, which means that the flow generated by basic fields preserve fibers, i.e. transforms one fiber into another one. The following result is due to Hermann \cite{Hermann}.

\begin{proposition}\label{isometry}
The submersion $\pi$ has totally geodesic fibers if and only if the flow generated by any basic vector field induces an isometry between the fibers.
\end{proposition}

\begin{proof}
We denote by $D$ the Levi-Civita connection on $\M$.  Let $\X$ be a basic vector field. If $Z_1,Z_2$ are vertical fields, the Lie derivative of the metric $g$ with respect to $\X$ can be computed as
\[
(\mathcal{L}_\X g)(Z_1,Z_2)= g\left( D_{Z_1} \X ,Z_2 \right) +
g\left( D_{Z_2} \X ,Z_1 \right).
\]
Because $\X$ is orthogonal to $Z_2$, we now have $g\left( D_{Z_1} \X ,Z_2 \right)=-g\left(\X ,D_{Z_1} Z_2 \right)$. Similarly 
$g\left(D_{Z_2} \X ,Z_1 \right)=-g\left(\X ,D_{Z_2} Z_1 \right)$, whence it follows that
\begin{align*}
(\mathcal{L}_\X g)(Z_1,Z_2)& =-g\left(\X ,D_{Z_1} Z_2 +D_{Z_2} Z_1 \right) \\
 &=-2 g\left(\X ,D_{Z_1} Z_2 \right).
\end{align*}
Thus the flow generated by any basic vector field induces an isometry between the fibers if and only if $D_{Z_1} Z_2$ is always vertical which is equivalent to the fact that the fibers are totally geodesic submanifolds.
\end{proof}

 The second  result  that characterizes totally geodesic submersions is due to B\'erard-Bergery and Bourguignon \cite{BeBo}.
 
 \begin{theorem}\label{commutation2}
The Riemannian submersion $\pi$ has totally geodesic fibers if and only if any basic vector field $\X$ commutes with the vertical Laplacian $\Delta_{\mathcal{V}}$. In particular, if $\pi$ has totally geodesic fibers, then the horizontal Laplacian and vertical Laplacian commute, i.e. for every $f \in C^\infty(\M)$,
\[
\Delta_{\mathcal{H}} \Delta_{\mathcal{V}} f=\Delta_{\mathcal{V}} \Delta_{\mathcal{H}} f.
\]
\end{theorem}
 
 \begin{proof}
 Assume that the submersion is totally geodesic. Let $\X$ be a basic vector field and $\xi_t$ be the flow it generates. Since $\xi$ induces an isometry between the fibers, we have
 \[
 \xi_t^* ( \Delta_{\mathcal{V}})= \Delta_{\mathcal{V}}.
 \]
 Differentiating at $t=0$ yields $[\X,\Delta_{\mathcal{V}}]=0$. 
 
 Conversely, assume that for every basic field $\X$, $[\X,\Delta_{\mathcal{V}}]=0$.  Let $\X_1,\dots,\X_n$ be a local orthonormal frame of basic vector fields and $Z_1,\cdots,Z_m$ be a local orthonormal frame of the vertical distribution. The second order part of the operator $[\X,\Delta_{\mathcal{V}}]$ must be zero. Given the expression of $ \Delta_{\mathcal{V}}$, this implies
 \[
 \sum_{i=1}^m [\X,Z_i] Z_i=0.
 \]
 This implies that the flow generated by $\X$ induces isometries between the fibers.
 
 \
 
 Finally, as we have already seen, if the submersion is totally geodesic then in a local basic orthonormal frame
 \[
 \Delta_{\mathcal{H}}=\sum_{i=1}^n \X_i^2 -\sum_{i=1}^n (D_{\X_i}\X_i)_{\mathcal{H}}.
 \]
 Since the vectors  $(D_{\X_i}\X_i)_{\mathcal{H}}$ are basic,  from the previous result $\Delta_{\mathcal{H}}$ commutes with $\Delta_{\mathcal{V}}$.
  \end{proof}

 \section{Pathwise construction of the horizontal Brownian motion on Riemannian foliations}\label{section construction HBM}

In this section, we show how to construct horizontal Brownian motions by solving a globally defined stochastic differential equation. Though the construction will not be used elsewhere in the monograph, it is conceptually important and is pivotal to  develop a stochastic calculus of variations associated with horizontal Brownian motions, see \cite{B, MR3969194}, \cite{MR1735806} and \cite{Elworthy}. In particular, the construction shows that horizontal Brownian motions are semimartingales.

The construction we present  is reminiscent of the Eels-Elworthy-Malliavin construction of Brownian motion on Riemannian manifolds (see \cite{Hsu}). It is actually possible to consider horizontal Brownian motions without having a globally defined submersion. The construction can be done in the context of Riemannian foliations.

\begin{definition}
Let $\M$ be a smooth and connected $n+m$-dimensional manifold. A $m$-dimensional foliation $\mathcal{F}$ on $\M$ is defined by a maximal collection of pairs $\{ (U_\alpha, \pi_\alpha), \alpha \in I \}$ of open subsets $U_\alpha$ of $\M$ and submersions $\pi_\alpha: U_\alpha \to U_\alpha^0$ onto open subsets of $\mathbb{R}^n$ satisfying:
\begin{itemize}
\item \hspace{.1in} $\cup_{\alpha \in I} U_\alpha =\M$;
\item \hspace{.1in} If $U_\alpha \cap U_\beta \neq \emptyset$, there exists a local diffeomorphism $\Psi_{\alpha \beta}$ of $\mathbb{R}^n$ such that $\pi_\alpha=\Psi_{\alpha \beta} \pi_\beta$ on $U_\alpha \cap U_\beta $.
\end{itemize}
\end{definition}

The maps $\pi_\alpha$ are called disintegrating maps of $\mathcal{F}$. The connected components of the sets $\pi_\alpha^{-1}(c)$, $c \in \mathbb{R}^n$, are called the plaques of the foliation. A foliation arises from an integrable sub-bundle of $T\M$, to be denoted by $\mathcal{V}$ and referred to as the vertical distribution. These are the vectors tangent to the leaves, the maximal integral sub-manifolds of $\mathcal{V}$.  

\

Foliations have been extensively studied and numerous monographs are devoted to them. We refer in particular to the monograph by Tondeur \cite{Tondeur}.

\

In the sequel, we shall only  be interested in Riemannian foliations with bundle like metric, which are the Riemannian foliations locally described by submersions. 

\begin{definition}
Let $\M$ be a smooth and connected $n+m$-dimensional Riemannian manifold. A $m$-dimensional foliation $\mathcal{F}$ on $\M$ is said to be Riemannian with a bundle like metric if the disintegrating maps $\pi_\alpha$ are Riemannian submersions onto $U_\alpha^0$ with its given Riemannian structure. If moreover the leaves are totally geodesic sub-manifolds of $\M$, then we say that the Riemannian foliation is totally geodesic with a bundle like metric.
\end{definition}

Observe that if we have a Riemannian submersion $\pi : (\M,g) \to (\mathbb{B},j)$, then $\M$ is equipped with a Riemannian foliation with bundle like metric whose leaves are the fibers of the submersion. Of course, there are many Riemannian foliations with bundle like metric that do not come from a Riemannian submersion.

\begin{Example}(\textbf{Contact manifolds})\label{contact example}
Let $(\M,\theta)$ be a $2n+1$-dimensional smooth contact manifold. On $\M$ there is a unique smooth vector field $S$, the so-called Reeb vector field, that satisfies
\[
\theta(S)=1,\quad \mathcal{L}_S(\theta)=0,
\]
where $\mathcal{L}_S$ denotes the Lie derivative with respect to  $S$. On $\M$ there is a foliation, the Reeb foliation, whose leaves are the orbits of the vector field $S$.  As it is well-known (see for instance \cite{Tanno}), it is always possible to find a Riemannian metric $g$ and a $(1,1)$-tensor field $J$ on $\M$ so that for every  vector fields $X,  Y$
\[
g(X,S)=\theta(X),\quad J^2(X)=-X+\theta (X) S, \quad g(X,JY)=(d\theta)(X,Y).
\]
The triple $(\M, \theta,g)$ is called a contact Riemannian manifold. We see then that the Reeb foliation is totally geodesic with bundle like metric if and only if the Reeb vector field $S$ is a Killing field, that is,
\[
\mathcal{L}_S g=0.
\]
In that case $(\M, \theta,g)$ is called a K-contact Riemannian manifold.
\end{Example}

\begin{Example}(\textbf{Sub-Riemannian manifolds with transverse symmetries})
The concept of sub-Riemannian manifold with transverse symmetries was introduced in \cite{BG}. Let $\M$ be a smooth, connected  manifold with dimension $n+m$. We assume that $\bM$ is equipped with a bracket generating distribution $\mathcal{H}$ of dimension $n$ and a fiberwise inner product $g_\mathcal{H}$ on that distribution. It is said that $\M$ is a sub-Riemannian manifold with transverse symmetries if there exists a $m$- dimensional Lie algebra $\mathcal{V}$ of sub-Riemannian Killing vector fields such that for every $x \in \bM$, 
 \[
 T_x \bM= \mathcal{H}_x \oplus \mathcal{V}_x,
 \]
 where 
 \[
  \mathcal{V}_x=\{ Z(x), Z \in \mathcal{V} \}.
 \]
 The choice of an inner product $g_{\mathcal{V}}$ on the Lie algebra $\mathcal{V}$ naturally endows $\bM$ with a  Riemannian metric that makes the decomposition $T_x\M=\mathcal{H}_x \oplus \mathcal{V}_x$ orthogonal:
\[
g=g_\mathcal{H} \oplus  g_{\mathcal{V}}.
\]
The sub-bundle of $\M$ determined by vector fields in $\mathcal{V}$ gives a foliation on $\M$ which is easily seen to be totally geodesic with bundle like metric.
\end{Example}

Since Riemannian foliations with a bundle like metric can locally be described by a Riemannian submersion, we can define a horizontal Laplacian $\Dh$ and a vertical Laplacian $\Dv$. This allows to define the horizontal  Brownian motion as we did in the previous section.

\

Let $\M$ be a smooth and connected  manifold with dimension $n+m$. In the sequel, we assume that $\bM$ is equipped with a Riemannian foliation $\mathcal{F}$ with bundle-like metric $g$ and totally geodesic  $m$-dimensional leaves.

\

The sub-bundle $\mathcal{V}$ formed by vectors tangent to the leaves is referred  to as the set of \emph{vertical directions}. The sub-bundle $\mathcal{H}$ which is normal to $\mathcal{V}$ is referred to as the set of \emph{horizontal directions}.   The metric $g$ can be split as
\[
g=g_\mathcal{H} \oplus g_{\mathcal{V}}.
\]
\



On the Riemannian manifold $(\M,g)$ there is the Levi-Civita connection that we denote as before by $D$, but this connection is not adapted to the study of foliations because the horizontal and  the vertical bundle may not be parallel. More adapted to the geometry of the foliation is the Bott's connection that we now define. It is an easy exercise to check that, since the foliation is totally geodesic, there exists a unique affine connection $\nabla$ such that:

\begin{itemize}
\item \hspace{.1in} $\nabla$ is metric, that is, $\nabla g =0$;
\item \hspace{.1in} For $X,Y \in \Gamma^\infty(\Ho)$, $\nabla_X Y \in \Gamma^\infty(\Ho)$;
\item \hspace{.1in} For $U,V \in \Gamma^\infty(\V)$, $\nabla_U V \in \Gamma^\infty(\V)$;
\item \hspace{.1in} For $X,Y \in \Gamma^\infty(\Ho)$, $T(X,Y) \in  \Gamma^\infty(\V)$ and  for $U,V \in \Gamma^\infty(\V)$, $T(U,V) \in  \Gamma^\infty(\Ho)$, where $T$ denotes the torsion tensor of $\nabla$;
\item \hspace{.1in} For $X \in \Gamma^\infty(\Ho), U \in \Gamma^\infty(\V)$, $T(X,U)=0$.
\end{itemize}

In terms of the Levi-Civita connection,  the Bott connection writes

\begin{align} \label{Bott connection}
\nabla_X Y =
\begin{cases}
 ( D_X Y)_{\mathcal{H}} , \quad X,Y \in \Gamma^\infty(\mathcal{H}) \\
 [X,Y]_{\mathcal{H}}, \quad X \in \Gamma^\infty(\mathcal{V}), Y \in \Gamma^\infty(\mathcal{H}) \\
 [X,Y]_{\mathcal{V}}, \quad X \in \Gamma^\infty(\mathcal{H}), Y \in \Gamma^\infty(\mathcal{V}) \\
 ( D_X Y)_{\mathcal{V}}, \quad X,Y \in \Gamma^\infty(\mathcal{V})
\end{cases}
\end{align}
where the subscript $\Ho$  (resp. $\mathcal{V}$) denotes the projection on $\mathcal{H}$ (resp. $\mathcal{V}$). Observe that for horizontal vector fields $X,Y\in \Gamma^\infty(\mathcal{H})$ the torsion $T(X,Y)$  is given by
\[
T(X,Y)=-[X,Y]_\V.
\]
Also observe that for $X,Y \in \Gamma^\infty(\mathcal{V})$ we actually have  $( D_X Y)_{\mathcal{V}}= D_X Y$ because the leaves are assumed to be totally geodesic. \\

Recall that a frame at $x \in \M$ can be described as a linear isomorphism $u: \mathbb{R}^{n+m} \to T_x\M$ such that for the standard basis $\left\{  e_{i} \right\}_{i=1}^{n+m}$ of $\mathbb{R}^{n+m}$ the collection $\left\{ u\left( e_{i} \right)\right\}_{i=1}^{n+m}$ is a basis (frame) for $T_{x}\M$. The collection of all such frames $\mathcal{F}\left( \M \right):=\bigcup_{x \in \M} \mathcal{F}\left( \M \right)_{x}$ is called the \emph{frame bundle} with the group $\mathbf{GL}\left( \mathbb{R}, n+m \right)$ acting on the bundle. If $\M$ is in addition Riemannian, we can restrict ourselves to the consideration of Euclidean isometries  $u: \left( \mathbb{R}^{n+m}, \langle \cdot,\cdot \rangle \right) \to \left( T_x\M, g \right)$ with the group $\mathbf{O}\left( \mathbb{R}, n+m \right)$ acting on the bundle. The \emph{orthonormal frame bundle} will be denoted by $\mathcal{O}(\M)$. 
An isometry $u: (\mathbb{R}^{n+m}, \langle \cdot,\cdot \rangle ) \longrightarrow (T_x\M, g)$ will be called \emph{horizontal} if $u ( \mathbb{R}^n \times \{ 0 \}) \subset \mathcal{H}_x$ and $u (  \{ 0 \}  \times \mathbb{R}^m ) \subset \mathcal{V}_x$. The \emph{horizontal frame bundle} $\mathcal{O}_\mathcal{H} (\M)$ is then defined as the set of $(x,u) \in \mathcal{O} (\M)$ such that $u$ is horizontal.

The Bott connection $\nabla$ induces a decomposition of each tangent space $T_{u}\mathcal{O}(\M)$ into the direct sum of a horizontal subspace and a vertical subspace as described in \cite[Section 2.1]{Hsu}. Using such decomposition, one can then lift smooth paths on $\M$ into smooth horizontal paths on $\mathcal{O}(\M)$. Such a lift is usually called the horizontal lift to $\mathcal{O}(\M)$. However, to avoid the confusion with the notion of horizontality given by the foliation on $\M$,  it shall  simply be referred to as the $\nabla$-lift to $\mathcal{O}(\M)$.

 Let $e_1, \dots, e_n, f_1, \dots, f_m$ be the standard basis of $\mathbb{R}^{n+m}$.
We denote by $A_i$ the vector field on $\mathcal{O} (\M)$ such that   $A_i (x,u)$ is the lift of $u(e_i)$, $i=1, ..., n$, $\left( x, u \right) \in \mathcal{O} (\M)$. 
%

\

For local computations, it is convenient to work in normal frames. 

\begin{lemma}[See \cite{BKW}] \label{frame}
Let $x \in \M$. Around $x$, there exist a local orthonormal horizontal frame $\{\X_1,\dots,\X_n \}$ and a local orthonormal vertical frame $\{Z_1,\dots,Z_m \}$ such that the following structure relations hold
\[
[\X_i,\X_j]=\sum_{k=1}^n \omega_{ij}^k \X_k +\sum_{k=1}^m \gamma_{ij}^k Z_k
\]
\[
[\X_i,Z_k]=\sum_{j=1}^m \beta_{ik}^j Z_j,
\]
where $\omega_{ij}^k,  \gamma_{ij}^k,  \beta_{ik}^j $ are smooth functions such that:
\[
 \beta_{ik}^j=- \beta_{ij}^k.
\]
Moreover, at $x$, we have
\[
 \omega_{ij}^k=0,  \beta_{ij}^k=0.
\]
\end{lemma}

In the above frame, the Christoffel symbols of the Bott connection are given by
\begin{align*}
\begin{cases}
\nabla_{\X_i} \X_j =\frac{1}{2} \sum_{k=1}^n \left( \omega_{ij}^k +\omega_{ki}^j+\omega_{kj}^i\right)\X_k \\
\nabla_{Z_j} \X_i =0 \\
\nabla_{\X_i} Z_j=\sum_{k=1}^m \beta_{ij}^{k} Z_k
\end{cases}.
\end{align*}
%
%

For $x_0 \in \M$ we let $\{ \X_1,\dots, \X_n, Z_1, \dots, Z_m \}$  be a normal frame around $x_0$. If  $u \in \mathcal{O}_\mathcal{H} (\M)$ is a horizontal isometry, we can find an orthogonal matrix $\left\{ e_i^j \right\}_{i, j=1}^{n}$ such that $u(e_i)=\sum_{j=1}^n e_i^j \X_j$, and  $u(f_i)=\sum_{j=1}^m f_i^j  Z_j$ for $f_i^j$, $i=1, ..., n$, $j=1, ..., m$. Let $\overline{\X}_j$ be the vector field on $\mathcal{O}_\mathcal{H} (\M)$ defined by
\[
\overline{\X}_j f (x,u)=\lim_{t \to 0} \frac{ f( e^{t\X_j}(x), u)-f(x,u)}{t},
\]
where $e^{t\X_j}(x)$ is the exponential map on $\M$.
\begin{lemma}\label{l.4.2}
Let  $x_0 \in \M$ and  $(x,u) \in \mathcal{O}_\mathcal{H} (\M)$, then
\begin{align*}
& A_i(x,u)=\sum_{j=1}^n e_i^j \overline{\X}_j
\\
& -\sum_{j,k,l,r=1}^n e_i^j e_r^l \, g\left( \nabla_{\X_j} \X_l, \X_k\right) \frac{\partial}{\partial e_r^k}-
\sum_{j=1}^n\sum_{k,l,r=1}^m e_i^j f_r^l \, g\left( \nabla_{\X_j} Z_l, Z_k\right) \frac{\partial}{\partial f_r^k}.
\end{align*}
 In particular, at $x_0$ we have
\[
A_i(x_0,u)=\sum_{j=1}^n e_i^j \overline{\X}_j.
\]
\end{lemma}

\begin{proof}
Let $u: \mathbb{R}^{n+m} \to T_x \M$ be a horizontal isometry and $x(t)$, $t>0$ be a smooth curve in $\M$ such that $x(0)=x$ and $x'(0)=u(e_i)$. We denote by $x^*(t)=(x(t),u(t))$ the $\nabla$-lift to $\mathcal{O} (\M)$ of $x(t)$ and by $x'_1(t),\dots,x'_n(t)$ the components of $x'(t)$ in the horizontal frame $\X_1,\dots,\X_n$. Since $\nabla$ is adapted to the foliation $\mathcal{F}$, the curve $x^*(t)$ takes its values in the horizontal frame bundle $\mathcal{O}_\mathcal{H} (\M)$. By definition of $A_i$, one has
\[
A_i=\sum_{j=1}^n x_j'(0) \overline{\X}_j +\sum_{k,l=1}^n  u_{kl}'(0)\frac{\partial}{\partial e_k^l}+\sum_{k,l=1}^m  v_{kl}'(0)\frac{\partial}{\partial f_k^l},
\]
where $u_{kl}(t)=\langle u(t)(e_k), \X_l \rangle$ and $v_{kl}(t)=\langle u(t)(f_k), Z_l \rangle$. Since $ u(t)(e_k)$ and $u(t)(f_k)$ are parallel along $x(t)$, one has
\[
\nabla_{x'(t)} u(t)(e_k)=0, \quad \nabla_{x'(t)} u(t)(f_k)=0.
\]
At $t=0$, this yields the expected result.
\end{proof}

%
%
%
%
%
%

\begin{proposition}
Let $\pi: \mathcal{O} (\M) \to \M$ be the bundle projection map. For a smooth $f:\M \to \mathbb{R}$, and $(x,u) \in \mathcal{O}_\mathcal{H} (\M) $,
\[
\left( \sum_{i=1}^n A_i^2 \right)(f\circ \pi)(x,u)=(\Delta_\mathcal{H} f )\circ \pi (x,u).
\]
\end{proposition}

\begin{proof}
It is enough to prove this identity at $x_0$. Using the fact that at $x_0$ we have $\langle \nabla_{\X_j} \X_l, \X_k\rangle =\langle \nabla_{\X_j} Z_l, Z_k\rangle=0$, we see that
\[
\sum_{i=1}^n A_i^2 =\sum_{j=1}^n \overline{\X}_j^2.
\]
The conclusion follows.
\end{proof}
As a straightforward corollary, we can construct the horizontal Brownian motion as follows.

\begin{corollary}\label{projec}
Let $\left (U(s) \right )_{s \geq 0}$ be a solution to the Stratonovich stochastic differential equation

\begin{equation}\label{e.HorizBM}
dU(s) =\sum_{i=1}^n A_i (U(s)) \circ dB^i(s), \quad U(0) \in \mathcal{O}_\mathcal{H} (\M),
\end{equation}
then $X(s)=\pi(U(s))$, $s\ge0$ is a horizontal Brownian motion on $\M$, that is, a diffusion process with  generator $\frac{1}{2} \Delta_{\mathcal H}$. 
\end{corollary}

\begin{remark}
The stochastic differential equation \eqref{e.HorizBM} is defined up to a possibly finite explosion time.
\end{remark}

\section{Horizontal Brownian motions and sub-Riemannian geometry}

While the geometry intrinsically associated with Brownian motions is Riemannian geometry, the geometry associated with horizontal Brownian motions is sub-Riemannian geometry.
A sub-Riemannian manifold is a smooth  manifold $\bM$
equipped with a fiber inner product $g(\cdot,\cdot)$ on the tangent
bundle $T\bM$ and a non-holonomic, or bracket generating, subbundle
$\mathcal H \subset T\bM$. This means that if we denote by
$L(\mathcal H)$ the Lie algebra of the vector fields generated by
the global $C^\infty$ sections of $\mathcal H$, then $\text{span}
\{\X(x)\mid \X\in L(\mathcal H)\} = T_x(\bM)$ for every $x\in \bM$. A
piecewise smooth curve $\gamma:[a,b]\to \bM$ is called admissible,
or horizontal, if it is tangent to $\mathcal H$, i.e. if $\gamma'(t)
\in \mathcal H_{\gamma(t)}$, whenever $\gamma'(t)$ is defined. The
horizontal length of $\gamma$ is then defined as
\[ \ee_\mathcal H(\gamma) = \int_a^b \sqrt{g(\gamma'(t),\gamma'(t))} dt. \]
Denoting by $\mathcal H(x,y)$ the collection of all horizontal
curves joining $x, y\in \bM$, one defines a distance $d(x,y)$
between $x$ and $y$ by minimizing on the length of all $\gamma\in
\mathcal H(x,y)$, i.e.
\[
d(x,y) = \underset{\gamma \in \mathcal H(x,y)}{\inf}\ \ee_\mathcal
H(\gamma). \] Such distance was introduced by Carath\'eodory in his
seminal paper \cite{Car} on formalization of the classical
thermodynamics. In this framework, loosely speaking, horizontal curves correspond to adiabatic processes. To be precise, in \cite{Car} the question of whether $d(x,y)$ be a
true distance was left open. This question was answered by the
fundamental connectivity theorem of Chow-Rashevsky, see
\cite{Montgomery}, which states that if $\bM$ is connected and $\mathcal H$
is bracket generating, then $\mathcal H(x,y)\not= \emptyset$ for
every $x,y\in \bM$. As a consequence, $d(x,y)$ is finite and
therefore it is a genuine distance. Such metric is nowadays known as
the sub-Riemannian, or Carnot-Carath\'eodory distance on $\bM$. Besides the above cited references, in order to understand the importance of the analysis and geometry of sub-Riemannian spaces, one may
consult  the monographs  \cite{Montgomery}, and \cite{Agrachev}.
 We note that when $\mathcal
H = T\bM$, then the distance $d(x,y)$ is simply the Riemannian
distance associated with the inner product $g (\cdot,\cdot)$, and
thus sub-Riemannian manifolds encompass Riemannian ones.

Notably
some aspects of the geometry of  sub-Riemannian manifolds are
considerably less regular than their Riemannian ancestors.  Some of
the major differences between the two geometries are the following, see \cite{Agrachev}:
\begin{enumerate}
\item The Hausdorff dimension of the metric space $(\bM,d)$ is greater than its topological
dimension;
\item The exponential map defined by the geodesics of the metric space   $(\bM,d)$ is in general not a local diffeomorphism in a
neighborhood of the point at which it is based;
\item The space of horizontal paths joining
two fixed points may have singularities (the so-called abnormal geodesics, see \cite{Montgomery}).
\end{enumerate}
Despite those difficulties, sub-Riemannian geometry has now grown into a mature field that can be used to study horizontal Brownian motions. As an illustration, we present below a comparison principle for the horizontal Brownian motion in the context of sub-Riemannian Sasakian foliations. In that context  the relevant sub-Riemannian curvature invariants are well understood, see \cite{MR3606555}.\\

Let $(\M, \theta,g)$ be a complete K-contact $n+1$-dimensional Riemannian manifold with Reeb vector field $S$, see Example \ref{contact example}. In that context, the Bott connection $\nabla$ introduced in \eqref{Bott connection} is the unique connection that satisfies:
\begin{enumerate}
\item $\nabla\theta=0$;
\item $\nabla S=0$;
\item $\nabla g=0$;
\item ${T}(X,Y)=d\theta(X,Y)S$ for any $X,Y\in \Gamma^\infty(\mathcal{H})$;
\item ${T}(S,X)=0$ for any vector field $X\in \Gamma^\infty(\mathcal{H})$.
\end{enumerate}
We have denoted by $T$ the  torsion tensor of $\nabla$. The Reeb foliation is called Sasakian if $\nabla J=0$, see \cite{MR2382957}. In a Sasakian foliation, for every non-vanishing horizontal vector field $X$,  $T\M$ is always generated by $[X,\Ho]$ and $\Ho$. Therefore the sub-Riemannian structure on a Sasakian foliation is fat (see~\textup{\cite{RS}} for a detailed discussion of such structures). In particular all sub-Riemannian geodesics are normal and from Corollary 6.1 in \textup{\cite{RT}}, for every $x_0 \in \M$, the distance function $x \to r (x)=d(x_0,x)$ is locally semi-concave in $\M\setminus \{ x_0 \}$. In particular, it is twice differentiable almost everywhere.  The cut-locus $\mathbf{Cut} (x_0)$ of  $x_0$ for the distance $d$ is defined as the complement of the set of $y$'s in $\M$ such that there exists a unique length minimizing normal geodesic joining $x_0$ and $y$, and $x_0$ and $y$ are not conjugate along such geodesic. From Corollary 32 in \textup{\cite{BR2}}, $x \neq x_0$ is in the cut-locus $\mathbf{Cut} (x_0)$ if and only if $r$ fails to be semi-convex at $x$. Therefore, $\mathbf{Cut} (x_0)$ has $\mu$ measure 0. Finally, at any point $x$ for which the function $x \to r (x)  $ is differentiable, there exists a unique length minimizing sub-Riemannian geodesic and this geodesic is normal.

We now introduce  relevant tensors that control the behavior of the horizontal Brownian motions. We first define a symmetric two-tensor $\mathbf{K}_{\ch,J}$ such that for a horizontal vector field $X \in \Gamma^\infty(\mathcal{H})$,
\[
\mathbf{K}_{\mathcal{H},J} (X,X) =  \frac{1}{|X|^2_{\ch}} \langle R( X,JX)JX , X\rangle_\mathcal{H},
\]
at points where $X$ does not vanish, where $R$ is the Riemannian curvature tensor of the Bott connection. The quantity $\mathbf{K}_{\mathcal{H},J}$ is sometimes called the pseudo-Hermitian sectional curvature of the Sasakian manifold (see \cite{Barletta} for a geometric interpretation). It can be seen as the CR analog of the holomorphic sectional curvature of a K\"ahler manifold. 
We will  also denote
\[
\mathbf{Ric}_{\mathcal{H},J^\perp} (X,X) =\mathbf{Ric}_\mathcal{H} (X,X)-\mathbf{K}_{\mathcal{H},J} (X,X)
\]
where $\mathbf{Ric}_\mathcal{H}$ is the horizontal Ricci curvature of the Bott connection (i.e. the horizontal trace of the Ricci curvature of the Bott connection).
We consider the following functions

$$F_{\mathrm{Rie}}(r,k) = \begin{cases}  \sqrt{k} \cot \sqrt{k} r & \text{if $k > 0$,} \\
\frac{1}{r} & \text{if $k = 0$,}\\ \sqrt{|k|} \coth \sqrt{|k|} r & \text{if $k < 0$.} \end{cases}$$

$$F_{\mathrm{Sas}}(r,k) = \begin{cases}  \frac{\sqrt{k}(\sin \sqrt{k}r  -\sqrt{k} r \cos \sqrt{k} r)}{2 - \cos \sqrt{k} r - \sqrt{k} r \sin \sqrt{k} r} & \text{if $k > 0$,} \\
\frac{4}{r} & \text{if $k = 0$,}\\ \frac{\sqrt{|k|}( \sqrt{|k|} r \cosh \sqrt{|k|} r - \sinh \sqrt{|k|}r)}{2 - \cosh \sqrt{|k|} r + \sqrt{|k|} r \sinh \sqrt{|k|} r} & \text{if $k < 0$.} \end{cases}$$

We have then the following horizontal Laplacian comparison theorem (see \cite{MR3978951}):

\begin{theorem}[Horizontal Laplacian comparison theorem] \label{th:SasakianComp3}
Let $(\M,\mathcal{F},g)$ be a Sasakian foliation with sub-Riemannian distance $d$, where $\M$ has dimension $n+1$. Fix $x_0\in\M$ and define $r(x) = d(x_0 ,x)$.
Assume that for some $k_1, k_2 \in \mathbb{R}$
$$\mathbf{K}_{\mathcal{H},J}(v,v) \ge  k_1, \quad \mathbf{Ric}_{\mathcal{H},J^\perp}(v,v) \ge (n-2)k_2, \quad v \in \ch, {|v|}_g = 1.$$
Then outside of the $d$ cut-locus of $x_0$ and globally on $\M$ in the sense of distributions,
\[
\Delta_\Ho r \le F_{\mathrm{Sas}}(r,k_1) +    (n-2)  F_{\mathrm{Rie}}(r,k_2),
\]
where $\Delta_\Ho$ is the horizontal Laplacian of the foliation.
\end{theorem}

\begin{Example}
\
\begin{itemize}
\item \hspace{.1in} $\M$ is the $2n+1$-dimensional Heisenberg group $\mathbf{H}^{2n+1}$ equipped with its canonical contact form $\theta$. In that case $k_1=k_2=0$ and we therefore have
\[
\Delta_\Ho r \le \frac{2n+2}{r}
\]
\item \hspace{.1in} $\M$ is the $2n+1$-dimensional sphere $\mathbb{S}^{2n+1}$ equipped with its canonical contact form and associated sub-Riemannian structure. In that case $k_1=4, k_2=1$.
\item \hspace{.1in} $\M$ is the $2n+1$-dimensional complex anti-de Sitter space $\mathbf{AdS}^{2n+1}(\mathbb{C})$ equipped with its canonical contact form. In that case $k_1=-4, k_2=-1$.
\end{itemize}

\end{Example}

The probabilistic counterpart of this comparison theorem for the horizontal Laplacian is the following. As above, we fix a point $x_0 \in \M$.  For $x_1 \in \M$, we consider
the solution of the stochastic differential equation
\[
  \tilde{\xi}(t)=d (x_0,x_1)+\frac{1}{2} \int_0^t
  \left(F_{\mathrm{Sas}}(\tilde{\xi}(s),k_1) + (n-2)
    F_{\mathrm{Rie}}(\tilde{\xi}(s),k_2) \right) ds+ \beta(t)
\]
where $\beta$ is a standard Brownian motion. The following result was then proved in \cite{MR4136477}.

\begin{theorem}\label{comparison 1d}
  Let $(\M,\mathcal{F},g)$ be a Sasakian foliation with sub-Riemannian
  distance $d$.  Assume that for some $k_1, k_2 \in \mathbb{R}$,
$$\mathbf{K}_{\mathcal{H},J}(v,v) \ge  k_1, \quad \mathbf{Ric}_{\mathcal{H},J^\perp}(v,v) \ge (n-2)k_2, \quad v \in \ch,\ {|v|}_g = 1.$$
Let $( ( \xi (t) )_{t \ge 0} , ( \mathbb P_x )_{x \in \M} )$ be the horizontal
Brownian motion on $\M$. Then, for
$x_1 \in \M$, $R>0$, and $s \le R$
\[
  \mathbb P_{x_1}\big( d (x_0,\xi (t)) <s , \, t \le \tau_R \big) \ge
  \mathbb P_{x_1}\big( \tilde{\xi} (t)<s, \, t \le \tilde{\tau}_R \big),
\]
where $\tau_R$ is the hitting time of the $d $-geodesic sphere in $\M$
with center $x_0$ and radius $R$ and $\tilde{\tau}_R$ the hitting time
of the level $R$ by $\tilde \xi$.
\end{theorem}

Such comparison theorem for the horizontal Brownian motion might  be used to obtain growth estimates for the horizontal Brownian motion, and we obtain the following results, see \cite{MR4136477} :

\begin{proposition}
  Let $(\M,\mathcal{F},g)$ be a Sasakian foliation. Assume that for
  some $k_1, k_2 \le 0$,
$$\mathbf{K}_{\mathcal{H},J}(v,v) \ge  k_1, \quad \mathbf{Ric}_{\mathcal{H},J^\perp}(v,v) \ge (n-2)k_2,
\quad v \in \ch,\ {|v|}_g = 1.$$
Let $( ( \xi (t) )_{t \ge 0} , ( \mathbb P_x )_{x \in \M} )$ be the horizontal
Brownian motion on $\M$.  Then for every
$x_0,x_1 \in \M$,
\[
  \mathbb P_{x_1} \left( \limsup_{t \to +\infty} \frac{d(x_0 , \xi(t))}{t}
    \le (n-2) \sqrt{|k_2|} + \sqrt{|k_1|} \right)=1
\]
\end{proposition}

When $k_1,k_2=0$, the above estimate can be refined and we have a
law of iterated logarithm.

\begin{proposition}
  Let $(\M,\mathcal{F},g)$ be a Sasakian foliation.  Assume
  that
  $$\mathbf{K}_{\mathcal{H},J}(v,v) \ge 0, \quad
  \mathbf{Ric}_{\mathcal{H},J^\perp}(v,v) \ge 0, \quad v \in \ch,\ {|v|}_g = 1.$$ Let $( ( \xi (t) )_{t \ge 0} , ( \mathbb P_x )_{x \in \M} )$ be the horizontal
Brownian motion on $\M$.
  Then for every $x_0,x_1 \in \M$,
  \[
    \mathbb P_{x_1} \left( \limsup_{t \to +\infty} \frac{d(x_0 ,
        \xi (t))}{\sqrt{2t \ln \ln t} } \le 1 \right)=1.
  \]
\end{proposition}

 Comparison principles can also be used to prove horizontal heat kernel estimates and Dirichlet eigenvalue estimates for the metric balls of the sub-Riemannian distance; We refer to \cite{MR4136477} for more details.

For further applications of sub-Riemannian geometry to the study of horizontal Brownian motions we refer the interested reader to \cite{B},  \cite{MR2154760},  \cite{MR2166728},  \cite{carfa}, \cite{MR3490551}, and \cite{GT1,GT2}.

\section{Examples of horizontal Brownian motions}

In this section we collect several examples of horizontal Brownian motions.

\begin{Example}
We first come  back to an example studied earlier that encompasses the Heisenberg group. Let
\[
\alpha=\sum_{i=1}^n \alpha_i(x)dx^i
\]
be a smooth one-form on $\mathbb{R}^n$ and let $(B(t))_{t \ge 0}$ be a $n$-dimensional Brownian motion. The process
\[
X(t)=\left(B(t), \int_{B[0,t]} \alpha \right)
\]
is a diffusion process in $\mathbb{R}^n \times \mathbb{R}$ with generator
\[
L=\frac{1}{2} \sum_{i=1} \X_i^2,
\]
where 
\[
\X_i=\frac{\partial}{\partial x_i} +\alpha_i(x) \frac{\partial}{\partial z} .
\]
We can then interpret $(X(t))_{t \ge 0}$ as a horizontal Brownian motion. Indeed, consider the Riemannian metric $g$ on $\mathbb{R}^n \times \mathbb{R}$  that makes $\X_1,\dots, \X_n , \frac{\partial}{\partial z}$ orthonormal. The map
\[
\pi : (\mathbb{R}^n \times \mathbb{R}, g) \to (\mathbb{R}^n, \mathbf{eucl.})
\]
such that $\pi (x,z)=x$ is then a Riemannian submersion and $\sum_{i=1} \X_i^2$ is the horizontal Laplacian of this submersion. Therefore $(X(t))_{t \ge 0}$ is a horizontal Brownian motion for this submersion.
\end{Example}

 In Examples \ref{SUexample}, \ref{SOexample}, \ref{SLexample} and \ref{SEexample} below, we consider the horizontal Brownian motions in some 3-dimensional Lie groups. The horizontal distribution in those cases is associated with canonical sub-Riemannian structures in those groups, see \cite{MR2902707}. We note that for those examples the heat kernel of the horizontal Brownian motions can be computed using the group Fourier transform, see \cite{ABGR}.

 \begin{Example}(\textbf{Horizontal Brownian motion on $\mathbf{SU}(2)$})\label{SUexample}
In what follows, we consider the Lie group $\mathbf{SU} (2)$, i.e. the group of
$2 \times 2$, complex, unitary matrices of determinant $1$. Its Lie algebra $\mathfrak{su} (2)$ consists of $2 \times 2$, complex,
skew-adjoint matrices of trace $0$. A basis of $\mathfrak{su} (2)$
is formed by the Pauli matrices:
 \[
\text{ }X=\left(
\begin{array}{cc}
~0~ & ~1~ \\
-1~& ~0~
\end{array}
\right) ,\text{ }Y=\left(
\begin{array}{cc}
~0~ & ~i~ \\
~i~ & ~0~
\end{array}
\right),
Z=\left(
\begin{array}{cc}
~i~ & ~0~ \\
~0~ & -i~
\end{array}
\right) ,
\]
for which the following relationships hold
\begin{align}\label{Liestructure}
[Z,X]=2Y, \quad [X,Y]=2Z, \quad [Y,Z]=2X.
\end{align}
We denote $\mathbb X,\mathbb Y,\mathbb Z$ the left invariant vector fields on 
$\mathbf{SU} (2)$ corresponding to the Pauli matrices. 
The Laplace-Beltrami operator for the bi-invariant Riemannian structure of 
$\mathbf{SU} (2) \simeq \mathbb{S}^3$ is
\begin{align*}
\Delta  =\mathbb X^2+\mathbb Y^2+ \mathbb Z^2.
\end{align*}
It is in the center of the universal enveloping algebra generated by the vector fields $\mathbb X,\mathbb Y,\mathbb Z$. 

The operator
\[
\Delta_{\mathcal H} =\mathbb X^2+\mathbb Y^2.
\]
is the horizontal Laplacian of a Riemannian submersion that we know describe. The one-parameter subgroup $\mathbf K=\{e^{zZ}, z \in \mathbb R \}$ acts isometrically on the right on $\mathbf{SU} (2)$ and the quotient map $\mathbf{SU} (2) \to \mathbf{SU} (2)/\mathbf K$ is a submersion. There exists then a unique Riemannian structure on $\mathbf{SU} (2)/\mathbf K$ that makes $\pi$ a totally geodesic Riemannian submersion (for this Riemannian metric $\mathbf{SU} (2)/\mathbf K$ turns out to be isometric with the two-dimensional sphere with radius $1/2$). To study the horizontal Brownian of that submersion, we  use the cylindrical coordinates:
\[
(r,\theta, z) \rightarrow \exp \left(r \cos \theta X +r \sin \theta Y \right) \exp (z Z)= \left(
\begin{array}{cc}
 \cos(r) e^{iz}& \sin(r)e^{i(\theta -z)}\\
 -\sin(r) e^{-i(\theta -z)}& \cos(r) e^{-iz}
\end{array}\right),
\]
with 
\[
0 \le r < \frac{\pi}{2}, \text{ }\theta \in [0,2\pi], \text{ }z\in [-\pi,\pi].
\]

Simple but tedious computations show that in these coordinates, the left-regular representation sends the matrices
 $X$, $Y$ and $Z$ to the
left-invariant vector fields: 
\[
\mathbb X=\cos (-\theta +2z) \frac{\partial}{\partial r}+\sin (-\theta +2z) \left( \tan r \frac{\partial}{\partial z}
+\left(\tan r+\frac{1}{\tan r}\right)  \frac{\partial}{\partial \theta}\right),
\]
\[
\mathbb Y=-\sin(2z-\theta) \frac{\partial}{\partial r}+\cos (2z-\theta) \left( \tan r \frac{\partial}{\partial z}
+\left(\tan r+\frac{1}{\tan r}\right)  \frac{\partial}{\partial \theta}\right),
\]
\[
\mathbb Z=\frac{\partial}{\partial z},
\]

We therefore obtain
\begin{align*}
\Delta_{\mathcal H}& =\mathbb X^2+\mathbb Y^2 \\
 & =\frac{\partial^2}{\partial r^2}+2 \cot 2r\frac{\partial}{\partial r}+\left( 2+\frac{1}{\tan^2 r}
+\tan^2 r  \right) \frac{\partial^2}{\partial \theta^2} +\tan^2 r \frac{\partial^2}{\partial z^2}+2(1+\tan^2r)
\frac{\partial^2}{\partial z \partial \theta}
\end{align*}

That computation shows that if $(X(t))_{t \ge 0}$ is the horizontal Brownian motion that is the matrix-valued solution of the stochastic differential equation (written in Stratonovich form)
\[
dX(t) =X(t) (\mathbb X \circ dB^1(t) + \mathbb Y \circ dB^2(t)), \quad X(0)=\mathbf I_2,
\]
where $(B^1_t,B^2_t)_{t \ge 0}$ is a two-dimensional Brownian motion, then, in distribution,
\[
X(t)=\exp \left( r(t) (\mathbb X\cos \theta(t)  +\mathbb Y\sin \theta(t) )\right) \exp (Z(t) \mathbb Z), \quad t\ge 0,
\]
where $(r(t),\theta(t), Z(t))_{t \ge 0}$ solve the following stochastic differential equations (written in It\^o's form):
\begin{align*}
\begin{cases}
dr(t) =2 \cot 2 r(t) dt +dB^1(t) \\
d\theta(t)=\frac{2}{\sin 2 r(t)} dB^2(t) \\
dZ(t) =\tan r(t) dB^2(t).
\end{cases}
\end{align*}
\end{Example}

 \begin{Example}(\textbf{Horizontal Brownian motion on $\mathbf{SO}(3)$})\label{SOexample}
The Lie group $\mathbf{SO}(3)$ is the group of $3\times 3$, real, orthogonal matrices of determinant $1$. Topologically, $\mathbf{SO}(3)$ is a double cover of the Lie group $\mathbf{SU}(2)$.  A basis of the Lie algebra $\mathfrak{so}(3)$ is $\{X,Y,Z\}$
where
$$
X=\begin{pmatrix}
0 & 1 & 0\\
-1 & 0 & 0\\
0 & 0 & 0
\end{pmatrix} \ \ 
Y=\begin{pmatrix}
0 & 0 & 1\\
0 & 0 & 0\\
-1 & 0 & 0
\end{pmatrix} \ \ 
Z=\begin{pmatrix}
0 & 0 & 0\\
0  & 0 & 1\\
0 & -1 & 0
\end{pmatrix}
$$
They satisfy the following commutation rules
\begin{equation}\label{brackets} 
[X,Y]=-Z,\quad [X,Z]= Y,\quad [Y,Z]=-X.
\end{equation}
In this case we have a submersion with totally geodesic fibers 
$$\pi : \mathbf{SO}(3) \to \mathbb{S}^2$$ 
which gives the fibration
\[
\mathbf{SO}(2) \to \mathbf{SO}(3) \to \mathbb{S}^2.
\]

The associated horizontal Laplacian of this fibration is the left invariant diffusion operator on  $\mathbf{SO}(3)$ is given by
\[
\Dh=\mathbb X^2+\mathbb Y^2
\]
where (by abuse of notation) $\mathbb X$ and $\mathbb Y$ denote the left invariant vector fields corresponding to the basis elements $X$ and $Y$.
To study $\Dh$ it is  convenient to use the cylindrical coordinates:
\[
(r,\theta,z)\rightarrow \exp(r\cos\theta X+r\sin\theta Y)\exp(zZ)=A,
\]
where $r \in [0,\pi)$, $z \in [-\pi,\pi]$, $\theta \in [0,2\pi]$.
First note that
\begin{align*}
&A=\\
&\begin{pmatrix}
\cos r & \sin r\cos(\theta+z) & \sin r \sin(\theta+z)\\
-\sin r\cos\theta & \sin\theta\sin(\theta+z)+\cos r\cos\theta\cos(\theta+z) & -\sin\theta\cos(\theta+z)+\cos r\cos\theta\sin(\theta+z)\\
-\sin r\sin\theta & \cos r\sin\theta\cos(\theta+z)-\cos\theta\sin(\theta+z) & \cos r\sin\theta\sin(\theta+z)+\cos\theta\cos(\theta+z)
\end{pmatrix}
\end{align*}
Therefore, the left-invariant vector fields are given by:
\begin{equation*}
\begin{cases}
\mathbb X= \cos(\theta+z)\frac{\partial}{\partial r}-\frac{\sin(\theta+z)}{\sin r}\frac{\partial}{\partial \theta}+\sin(\theta+z)\tan\left(\frac{r}{2}\right)\frac{\partial}{\partial z}\\
\mathbb Y= \sin(\theta+z)\frac{\partial}{\partial r}+\frac{\cos(\theta+z)}{\sin r}\frac{\partial}{\partial \theta}-\cos(\theta+z)\tan\left(\frac{r}{2}\right)\frac{\partial}{\partial z}\\
\mathbb Z=\frac{\partial}{\partial z}
\end{cases}
\end{equation*}

and we  obtain
\begin{eqnarray*}
\Dh&=&\mathbb X^2+\mathbb Y^2\\
&=&\frac{\partial^2}{\partial r^2}+\frac{1}{\sin^2r}\frac{\partial^2}{\partial\theta^2}+\tan^2\left(\frac{r}{2}\right)\frac{\partial^2}{\partial z^2}-2\tan\left(\frac{r}{2}\right)\frac{1}{\sin r}\frac{\partial}{\partial z}\frac{\partial}{\partial\theta}+\cot r\frac{\partial}{\partial r}
\end{eqnarray*}

This computation shows that if $(X(t))_{t \ge 0}$ is the horizontal Brownian motion, which is the matrix-valued solution of the stochastic differential equation (written in Stratonovich form)
\[
dX(t) =X(t) (\mathbb X \circ dB^1_t + \mathbb Y \circ dB^2(t)), \quad X(0)=\mathbf I_2,
\]
where $(B^1(t),B^2(t))_{t \ge 0}$ is a two-dimensional Brownian motion, then, in distribution,
\[
X(t)=\exp \left( r(t) (X\cos \theta(t)  + Y\sin \theta(t) )\right) \exp (Z(t) Z), \quad t\ge 0,
\]
where $(r(t),\theta(t), Z(t))_{t \ge 0}$ solve the following stochastic differential equations (written in It\^o's form):
\begin{align*}
\begin{cases}
dr(t) = \cot  r(t) dt +dB^1(t) \\
d\theta(t)=\frac{1}{\sin  r(t)} dB^2(t) \\
dZ(t) =-\tan \frac{r(t)}{2} dB^2(t).
\end{cases}
\end{align*}

\end{Example}

 \begin{Example}(\textbf{Horizontal Brownian motion on $\mathbf{SL}(2)$})\label{SLexample}
The Lie group $\mathbf{SL} (2)$ is the group of
$2 \times 2$, real matrices of determinant $1$. Its
Lie algebra $\mathfrak{sl} (2)$ consists of $2 \times 2$ matrices of trace $0$.
A basis of $\mathfrak{sl} (2)$ is formed by the matrices:
 \[
X=\left(
\begin{array}{cc}
~1~ & ~0~ \\
~0~ & -1~
\end{array}
\right)
,\text{ }Y=\left(
\begin{array}{cc}
~0~ & ~1~ \\
~1~ & ~0~
\end{array}
\right),
\text{ }Z=\left(
\begin{array}{cc}
~0~ & ~1~ \\
-1~& ~0~
\end{array}
\right)
,
\]
for which the following relations hold
\begin{align}\label{Liestructure}
[X,Y]=2Z, \quad [X,Z]=2Y, \quad [Y,Z]=-2X.
\end{align}

We denote by $\mathbb X,\mathbb Y,\mathbb Z$ the corresponding left invariant vector fields on 
$\mathbf{SL} (2)$ and consider the left invariant Riemannian metric on $\mathbf{SL} (2)$ that makes $\mathbb X,\mathbb Y,\mathbb Z$ an orthonormal frame. 
The Laplace-Beltrami operator for this Riemannian structure is
\begin{align*}
\Delta  =\mathbb X^2+\mathbb Y^2+\mathbb Z^2.
\end{align*}

The operator
\[
\Delta_{\mathcal H} =\mathbb X^2+\mathbb Y^2.
\]
is the horizontal Laplacian of a Riemannian submersion that we know describe. The one-parameter subgroup $\mathbf K=\{e^{tZ}, t \in \mathbb R \}$ acts isometrically on the right on $\mathbf{SL} (2)$ and the quotient map $\mathbf{SL} (2) \to \mathbf{SL} (2)/\mathbf K$ is a submersion. There exists then a unique Riemannian structure on $\mathbf{SL} (2)/\mathbf K$ that makes $\pi$ a totally geodesic Riemannian submersion (for this Riemannian metric $\mathbf{SL} (2)/\mathbf K$ is isometric with the two-dimensional real hyperbolic space). To study the horizontal Brownian of that submersion, we introduce the cylindrical coordinates:
\[
(r,\theta, z) \rightarrow \exp \left(r \cos \theta X +r \sin \theta Y \right) \exp (z Z)
\]
\[= \left(
\begin{array}{cc}
 \cosh(r) \cos (z) + \sinh(r)\cos (\theta +z) & \cosh(r) \sin (z) + \sinh(r)\sin (\theta +z)\\
-\cosh(r) \sin (z) + \sinh(r)\sin (\theta +z)& \cosh(r) \cos (z)  - \sinh(r)\cos (\theta +z)
\end{array}\right),
\]
with 
\[
 r >0, \text{ }\theta \in [0,2\pi], \text{ }z\in [-\pi,\pi].
\]

Simple but tedious computations show that in these coordinates, the left-regular representation sends the matrices $\mathbb X$, $\mathbb Y$ and $\mathbb Z$ to the
left-invariant vector fields: 
\[
\mathbb X=\cos (\theta + 2 z) \frac{\partial}{\partial r}-\sin (\theta +2z) \left( \tanh r \frac{\partial}{\partial z}
+\left(\frac{1}{\tanh r}-\tanh r\right)  \frac{\partial}{\partial \theta}\right),
\]
\[
\mathbb Y=\sin (\theta +2 z) \frac{\partial}{\partial r} +\cos (\theta + 2z) \left( \tanh r \frac{\partial}{\partial z}
+\left(\frac{1}{\tanh r}- \tanh r)\right)  \frac{\partial}{\partial \theta} \right),
\]
\[
\mathbb Z=\frac{\partial}{\partial z}.
\]

We therefore obtain
\begin{align*}
\Delta_\mathcal{H} & =\mathbb X^2+\mathbb Y^2 \\
 & =\frac{\partial^2}{\partial r^2}+  2 \coth 2r \frac{\partial}{\partial r}
  +\tanh^2 r \frac{\partial^2}{\partial z^2} 
 + \frac{4}{\sinh^2 2r} \frac{\partial^2}{\partial \theta^2}
+2(1-\tanh^2 r) \frac{\partial^2}{\partial \theta \partial z}.
\end{align*}

As a consequence, if $(X(t))_{t \ge 0}$ is the horizontal Brownian motion, which is the matrix-valued solution of the stochastic differential equation (written in Stratonovich form)
\[
dX(t) =X(t) (\mathbb X \circ dB^1_t + \mathbb Y \circ dB^2_t), \quad X(0)=\mathbf I_2,
\]
where $(B^1_t,B^2_t)_{t \ge 0}$ is a two-dimensional Brownian motion, then, in law,
\[
X(t)=\exp \left( r(t) (X\cos \theta(t)  + Y\sin \theta(t) )\right) \exp (Z(t) Z), \quad t\ge 0,
\]
where $(r(t),\theta(t), Z(t))_{t \ge 0}$ solve the following stochastic differential equations (written in It\^o's form):
\begin{align*}
\begin{cases}
dr(t) =2 \mathrm{ coth} 2 r(t) dt +dB^1_t \\
d\theta(t)=\frac{2}{\sinh 2 r(t)} dB^2_t \\
dZ(t) =\tanh r(t) dB^2_t.
\end{cases}
\end{align*}
It is apparent that the analysis of the horizontal Brownian motion on $\mathbf{SL} (2)$  is the  hyperbolic counterpart of the analysis of the horizontal Brownian motion on $\mathbf{SU} (2)$. There exist other interesting horizontal Brownian motions on $\mathbf{SL} (2)$ which are associated with non equivalent sub-Riemannian structures on $\mathbf{SL} (2)$, see for instance Section 3 in \cite{MR3685249} for a description of those sub-Riemannian structures. 
\end{Example}

 \begin{Example}(\textbf{Horizontal Brownian motion on $\mathbf{SE}(2)$})\label{SEexample}
 The Lie group $\mathbf{SE} (2)$ is the group of Euclidean motions in the plane. It is isomorphic to the Lie group of $3 \times 3$ matrices
 \begin{equation*}
\left\{  \left(
\begin{array}{lll}
\cos \theta & -\sin \theta & x \\
\sin \theta & \cos \theta & y \\
0 & 0 & 1  
\end{array}
\right), x,y \in \R, \theta \in \R /2\pi \mathbb Z \right\}.
 \end{equation*}
 The Lie algebra $\mathfrak{se} (2)$ is generated by the matrices
 \[
 X=\left(
\begin{array}{lll}
0 & -1 & 0 \\
1 & 0 & 0 \\
0 & 0 & 0  
\end{array}
\right), Y=\left(
\begin{array}{lll}
0 & 0 & 1 \\
0 & 0 & 0 \\
0 & 0 & 0  
\end{array}
\right), Z=\left(
\begin{array}{lll}
0 & 0 & 0 \\
0 & 0 & 1 \\
0 & 0 & 0  
\end{array}
\right)
 \]
Clearly the following relations hold
 \[
 [X,Y]=Z, \,  [X,Z]=-Y, \, [Y,Z]=0.
 \]
We denote by $\mathbb X,\mathbb Y,\mathbb Z$ the associated left invariant vector fields on 
$\mathbf{SE} (2)$. Simple computations show that
 \[
 \mathbb X=\frac{\partial}{\partial \theta}, \, \mathbb Y =\cos \theta \frac{\partial}{\partial x} + \sin \theta \frac{\partial}{\partial y} , \, \mathbb Z =-\sin \theta \frac{\partial}{\partial x} + \cos \theta \frac{\partial}{\partial y}.
 \]
 
Consider the left invariant Riemannian metric on $\mathbf{SE} (2)$ that makes $\mathbb X,\mathbb Y,\mathbb Z$ an orthonormal frame. 
The Laplace-Beltrami operator for this Riemannian structure is
\begin{align*}
\Delta  =\mathbb X^2+\mathbb Y^2+\mathbb Z^2=\frac{\partial^2}{\partial \theta^2}+\frac{\partial^2}{\partial x^2}+\frac{\partial^2}{\partial y^2} .
\end{align*}
Consider now the left-invariant contact form 
\[
\eta=-\sin \theta \, dx +\cos \theta \, dy.
\]
The vector field $\mathbb Z$ is the  Reeb vector field of the above form  and the horizontal Laplacian of the Reeb foliation (see Example \ref{contact example}) is given by:
\[
\Delta_\mathcal{H}= \mathbb X^2+\mathbb Y^2=\frac{\partial^2}{\partial \theta^2}+\left(\cos \theta \frac{\partial}{\partial x} + \sin \theta \frac{\partial}{\partial y} \right)^2.
\]
It follows that the horizontal Brownian (issued from the identity) of the foliation is given in distribution as
\[
\left(
\begin{array}{ccc}
\cos B(t) & -\sin B(t) & \int_0^t \cos (B(s)) d\beta(s) \\
\sin B(t) & \cos B(t) & \int_0^t \sin (B(s)) d\beta(s)  \\
0 & 0 & 1  
\end{array}
\right)
\]
where $(B(t),\beta(t))_{t \ge 0}$ is a two-dimensional Brownian motion.
 \end{Example}

 \begin{Example}
 
 Let $\beta \in \R$, $\beta \neq 0$. Consider the Lie group of $3 \times 3$ matrices
 \begin{equation*}
G=\left\{  \left(
\begin{array}{ccc}
 e^{\beta \theta} & 0 & x \\
0 & 1 & y \\
0 & 0 & 1  
\end{array}
\right), x,y, \theta \in \R \right\}.
 \end{equation*}
 The Lie algebra $\mathfrak{g}$ is generated by the matrices
 \[
 X=\left(
\begin{array}{lll}
\beta & 0 & 0 \\
0 & 0 & 0 \\
0 & 0 & 0  
\end{array}
\right), Y=\left(
\begin{array}{lll}
0 & 0 & 1 \\
0 & 0 & 1 \\
0 & 0 & 0  
\end{array}
\right), Z=\left(
\begin{array}{lll}
0 & 0 & 0 \\
0 & 0 & -\beta \\
0 & 0 & 0  
\end{array}
\right)
 \]
 which are associated with the left invariant vector fields
 \[
 \mathbb X=\frac{\partial}{\partial \theta}, \, \mathbb Y =e^{\beta \theta} \frac{\partial}{\partial x} +  \frac{\partial}{\partial y} , \, \mathbb Z =-\beta \frac{\partial}{\partial y}.
 \]
 Clearly the following relations hold
 \[
 [X,Y]=\beta Y+ Z, \,  [X,Z]=0, \, [Y,Z]=0.
 \]
 Consider now the left-invariant contact form 
\[
\eta=\frac{1}{\beta} \left(e^{-\beta \theta} \, dx - \, dy \right).
\]
The vector field $\mathbb Z$ is the  Reeb vector field of the above form  and the horizontal Laplacian of the Reeb foliation  is given by:
\[
\Delta_\mathcal{H}= \mathbb X^2+\mathbb Y^2-\beta \mathbb X=\frac{\partial^2}{\partial \theta^2}-\beta \frac{\partial}{\partial \theta}+\left( e^{\beta \theta} \frac{\partial}{\partial x} +  \frac{\partial}{\partial y} \right)^2.
\]
Note the appearance of the term $-\beta \frac{\partial}{\partial \theta}$ comes from the fact that $G$ is not unimodular because the Haar volume form is given by
\[
e^{-\beta \theta}\, d\theta \wedge dx \wedge dy.
\]
It follows that the horizontal Brownian motion (issued from the identity) of the foliation is given in distribution as
\[
\left(
\begin{array}{ccc}
 \exp( \beta B(t)-\beta^2 t ) & 0 & \int_0^t \exp (\beta B(s)-\beta^2 s) d\beta(s) \\
0 & 1 & \beta(t)  \\
0 & 0 & 1  
\end{array}
\right)
\]
where $(B(t),\beta(t))_{t \ge 0}$ is a two-dimensional Brownian motion.
 \end{Example}

\begin{Example}(\textbf{Signature and log-signature of the Brownian motion})

We discuss here an example of horizontal Brownian motion that plays a central role in the theory of rough paths (see \cite{MR2154760,MR2604669}).
The truncated tensor algebra $T_N(\mathbb{R}^{d})$ over $\mathbb{R}^{d}$ is given by
\[
T_N(\mathbb{R}^{d})=\bigoplus^{N}_{k=0}(\mathbb{R}^{d})^{\otimes k}
\]
with the convention that $(\mathbb{R}^{d})^{0}=\mathbb{R}$. All computations in this truncated algebra are done at degree at most $N$, i.e. $e_{i_1} \otimes  \cdots \otimes e_{i_k} =0$, if $k > N$. If $(B(t))_{t \ge 0}$ is a $d$-dimensional Brownian motion the $T_N(\mathbb{R}^{d})$-valued process
\[
 X(t)=S_N(B)(t):=\sum_{k=0}^{N} \int_{\Delta^k [0,t]}  \circ dB^{\otimes k}, \quad t \ge 0,
\]
is called the signature of $B$ of order $N$. The iterated integrals over the simplex
\[
\Delta^k [0,t]=\left\{ 0 \le t_1 \le \cdots \le t_k \le t \right\}
\]
appearing in the definition of the signature are understood in the sense of Stratonovich.
Note that the signature is the solution to a stochastic differential equation that writes
\begin{align}\label{SDE: signature}
	dX(t)=X(t)\otimes dB(t)=\sum_{i=1}^dW_i(X(t))\circ dB(t)^i,\quad X(0)=\mathbf{1}=(1,0,0,...)\in T_N(\mathbb R^d),
\end{align}
where the $W_i$'s are polynomial vector fields on $T_N(\mathbb R^d)$. It is a well-known theorem by Chen \cite{Chen54}  that the signature of a path is a Lie element. To be more precise, consider the Lie group
\[
\G_N(\mathbb R^d) =\exp ( \mathfrak{g}_N(\mathbb{R}^d)),
\]
where $\mathfrak{g}_N(\mathbb{R}^d)$ is the Lie sub-algebra of $T_N(\mathbb{R}^{d})$ generated by the canonical basis $e_i, i=1,\dots,d,$   of $\mathbb{R}^d$, and the Lie bracket  is given by $[a,b]=a\otimes b-b\otimes a$.  Then Chen's theorem asserts that for every $t \ge 0$, we almost surely have $X(t) \in \G_N(\mathbb R^d) $. More precisely, from the Chen-Strichartz formula (see \cite{MR2154760}), one has the following explicit following representation of $X(t)$:
\[
X(t)= \exp \left( \sum_{I, l(I) \le N } \Lambda_I (B)(t) e_{I} \right),
\]
where:
\begin{itemize}
\item \hspace{.1in} For $z \in T_N(\mathbb{R}^{d})$, $\exp (z) =\sum_{k=0}^N \frac{1}{k!} z^{\otimes k}$;
\item \hspace{.1in} If $I\in \{1,...,d\}^k$ is a word, 
\[
e_I= [e_{i_1},[e_{i_2},\cdots,[e_{i_{k-1}},
e_{i_{k}}]\cdots],
\]
and
$
l(I)=k;
$
\item \hspace{.1in} Let $\mathfrak{S}_k$ be the set of the permutations of
$\{1,...,k\}$, then
\[
\Lambda_I (B)(t)= \sum_{\sigma \in \mathfrak{S}_k} \frac{\left(
-1\right) ^{e(\sigma )}}{k^{2}\left(
\begin{array}{l}
k-1 \\
e(\sigma )
\end{array}
\right) } \int_{0 \leq t_1 \leq ... \leq t_k \leq t} 
\circ dB^{\sigma^{-1}(i_1)}(t_1) \cdots  
\circ dB^{\sigma^{-1}(i_k)}(t_k),\quad t \ge 0.
\]
In the above, for $\sigma \in \mathfrak{S}_k$,  $e(\sigma)$ is the cardinality of the set
$
\{ j \in \{1,...,k-1 \} , \sigma (j) > \sigma(j+1) \}.
$
\end{itemize}

It is well known that $\G_N(\mathbb R^d)$ is a free nilpotent  and simply connected Lie group whose Lie algebra inherits, from the grading of $T_N(\mathbb{R}^{d})$, a stratification
\[
\mathfrak{g}_N(\mathbb{R}^d)=\mathcal{V}_{1}\oplus \cdots \oplus \mathcal{V}_{N},
\]
with
\begin{equation*}
\dim \mathcal{V}_{j}= \frac{1}{j} \sum_{i \mid j} \mu (i)
d^{\frac{j}{i}}, \text{ } j \leq N,
\end{equation*}
 where $\mu$ is the M\"obius function and the sum is over the divisors $i$ of $j$. From the Hall-Witt theorem we can then construct a basis of $\mathfrak{g}_N(\mathbb{R}^d)$ which is adapted to the stratification
\[
\mathfrak{g}_N(\mathbb{R}^d)=\mathcal{V}_{1}\oplus \cdots \oplus \mathcal{V}_{N},
\]
and such that every element of this basis is an iterated bracket of the $e_i$'s.  Let $\mathcal{B}$ denote such a basis and for $x \in \mathfrak{g}_N(\mathbb{R}^d)$, let $[x]_\mathcal{B}\in \mathbb{R}^n$ be the coordinate vector of $x$ in the basis $\mathcal{B}$ where $n=\dim \mathfrak{g}_N(\mathbb{R}^d)$. The log-signature of order $N$ of $(B(t))_{t \ge 0}$ is the $\mathbb{R}^n$-valued process
\[
U(t)=[\exp^{-1}( X(t))]_\mathcal{B} = \sum_{I, l(I) \le N } \Lambda_I (B)(t) [e_{I}]_\mathcal{B}.
\] 

From \cite{MR2154760}, $(U(t))_{t\ge0}$ solves in $\mathbb{R}^n$ a rough differential equation

\begin{align*}
U(t)=\sum_{i=1}^d\int_0^tV_i(U_s)\circ dB^i(s),\end{align*}
where the vector fields $V_1,\dots,V_d$ are polynomials and generate a Lie algebra isomorphic to $ \mathfrak{g}_N(\mathbb{R}^d)$. Actually $V_1,\dots,V_d$ are left invariant vector fields for a polynomial group law $\star$ on $\mathbb{R}^n$ such that $(\mathbb{R}^n,\star)$ is isomorphic to $\G_N(\mathbb{R}^d)$. 

The choice of the basis $\mathcal{B}$ equips the Lie group $(\mathbb{R}^n,\star)$ with a left-invariant Riemannian metric and the projection map $\pi : \mathbb{R}^n \to \mathbb{R}^d$ on the first $d$ components is a Riemannian submersion if $\mathbb R^d$ is equipped with its standard Euclidean structure. The log-signature  $(U(t))_{t\ge0}$ is then the horizontal Brownian of that submersion. 

Note that for $d=2$ and $N=2$, we have 
\[
U(t)= \left( B^1(t),B^2(t), \frac{1}{2} \int_0^t B^1(s) dB^2(s)-B^2(s) dB^1(s) \right)
\]
which is the horizontal Brownian on the $3$-dimensional Heisenberg group. 

\end{Example}

\begin{Example}(\textbf{Horizontal Brownian motion on vector bundles})
A  class of examples  that naturally arise in stochastic calculus are horizontal Brownian motions on  vector bundles. We present here the case of the tangent bundle, but the construction may be extended to any vector bundle. Let $(\mathbb{M},g)$ be a  smooth and connected Riemannian manifold with dimension $n$. Let $(B(t))_{t \ge 0}$ be a Brownian motion on $\M$ started at $x$. Let $\para_{0,t} : T_x \M \to T_{B(t)} \mathbb{M}$ be the stochastic parallel transport along the paths of $(B(t))_{t \ge 0}$. Let now $v \in T_x \M$ and consider the tangent bundle $T\M$ valued process:
\[
X(t)=\left( B(t) , \para_{0,t} v \right).
\]
The process $(X(t))_{t \ge 0}$  can be interpreted as a horizontal Brownian for some Riemannian submersion. The submersion is simply the bundle projection map $\pi : T\mathbb{M} \to \mathbb{M}$. One then needs to construct a Riemannian metric on $T\mathbb{M}$ that makes $\pi$ a Riemannian submersion. We call a $C^1$ curve $\gamma(t)=(x(t),v(t))$, $t\ge0$ to be horizontal if $v(t)$ is parallel-transported along $x$. This uniquely determines the rank $n$ horizontal bundle in $TT\mathbb{M}$. Now, if $\X$ is a vector field on $\mathbb{M}$, define its horizontal lift $\X^\Hh$ as the unique horizontal vector field on $TT\mathbb{M}$ that projects onto $\X$. Define its vertical lift as the unique vertical vector field $\X^{\mathcal V}$ on $TT\mathbb{M}$ such that for every smooth $f:\mathbb{M} \to \mathbb{R}$ one has
\[
\X^\V (f^*)=\X f,
\]
where $f^*(x,v)=df_x (v)$. The Sasaki metric $g$ on $T\mathbb{M}$ is then the unique metric such that if $\X_1,\dots, \X_n$ is a local orthonormal frame on $\mathbb{M}$,  then  $\X^\Hh_1,\dots, \X^\Hh_n,\X_1^\V,\dots,\X_n^\V$ is a local orthonormal frame on $T\mathbb{M}$. It is then easy to check that $\pi$ is  a Riemannian submersion with totally geodesic fibers and that $(X(t))_{t \ge 0}$ is a horizontal Brownian motion for this submersion.
\end{Example}

\section{Horizontal Brownian motion on principal bundles}\label{horizontal BM bundle}

In this section, we study  in details the horizontal Brownian motion on principal bundles, as it generalizes most of the situations encountered later in this monograph.

\subsection{Skew-product decomposition}\label{skew general bundle}

Let $(\M,j)$ be a Riemannian manifold. We assume that $\M$  is the total space of  a principal bundle structure $\pi:\M \to \B$ where $(\B,h)$ is a Riemannian manifold. Here, this means that $\pi$ is the quotient map of the free action  (on the right by convention) of a Lie group $\G$ on $\M$.  Let  $\mathfrak g$ denote the Lie algebra of $\G$.  Let now $\theta$ be a $\G$-compatible connection form on $\M$, i.e. $\theta$ is a $\mathfrak g$-valued one-form on $\M$ such that:
\begin{enumerate}
\item For any $g \in \G$, we have
\[
g^*\theta=\mathbf{Ad}_{g^{-1}} \circ \theta;
\]
where for $X \in \mathfrak g$, 
\[
\mathbf{Ad}_{g^{-1}} X =\lim_{t \to 0} \frac{1}{t}(g^{-1} e^{tX} g- \mathbf{e}),
\]
with $\mathbf{e}$ the unit element in $\G$.
\item For any $V \in \mathfrak g$,
\[
\theta (\hat{V})=V,
\]
where $\hat V$ is the vector field on $\M$ defined by
\begin{align}\label{fundamental vertical field}
\hat V f (m)= \lim_{t \to 0} \frac{f( m e^{tV})-f(m)}{t}, \quad m \in \M.
\end{align}
\end{enumerate}

\begin{remark}
If $X$ is a vector field on $\M$ we will often denote $\hat{\theta}(X):=\widehat{\theta (X)}$.
\end{remark}

We assume that $\G$ is equipped with an invariant Riemannian metric $k$ and that the bundle projection map $\pi$ is a Riemannian submersion $\pi: (\M,j) \to (\B,h)$ with totally geodesic fibers isometric to $(\G, k)$ and whose horizontal distribution is the kernel of the connection form $\theta$.

Our goal in the setting described above is to study the horizontal Brownian motion on $\M$ associated with the previous submersion. We will consider a local trivialization of $\M$: $(x,g)$, $x\in U \subset \mathbb B$, $g \in \G$, where $U$ is the domain of a chart  with coordinates $(x^1,\dots,x^n)$. Such trivialization is not unique nor canonical in general, but in the examples treated in the text we will work with canonical ones.
  In the trivialization, the submersion $\pi$ simply reads $(x,g) \to x$. We introduce the $\mathfrak g$-valued one  form $\alpha$ defined on $U$ by
\begin{align}\label{solder form}
\alpha_x \left(\frac{\partial }{\partial x^i} \right)=\theta_{x,\mathbf{e}} \left(\frac{\partial }{\partial y^i} \right)
\end{align} 
where $y^i=x^i \circ \pi $.
\begin{theorem}\label{skew-product principal bundle}
Let $(B(t))_{t \ge 0}$ be a Brownian motion on the base manifold $\B$ (started from a point inside $U$ and considered up to the first exit time of $U$). Then, in the above trivialization, the process 
\[
X(t)=(B(t), \Theta (t))
\]
is a horizontal Brownian motion on $\M$ where $\Theta(t)$ solves the Stratonovich stochastic differential equation
\[
d\Theta(t)=-\circ d \mathfrak{a} (t) \, \Theta (t), 
\]
and $\mathfrak{a} (t)$ is the Stratonovich line integral $\mathfrak{a} (t)=\int_{B[0,t]} \alpha$.

\end{theorem}

\begin{proof}
We consider first a smooth curve $\gamma:\mathbb R \to \B$ and a smooth curve $\tilde{\gamma}: \R \to \M$ such that $\tilde{\gamma}(t)=(\beta (t), \Theta (t))$, $t\ge0$ in $\M$ projects down to $\gamma$ . In the coordinates $x^1,\dots,x^n$ we have
\[
\gamma'(t)=\sum_{i=1}^n \gamma_i'(t) \frac{\partial}{\partial x^i}
\]

In  the trivialization the horizontal lift  to $\M$ of $\frac{\partial}{\partial x^i}$ is given by $\frac{\partial}{\partial y^i}-\widehat{\theta} \left( \frac{\partial}{\partial y^i} \right)$. Therefore
\[
\tilde{\gamma}'(t)=\sum_{i=1}^n \gamma_i'(t)\left(  \frac{\partial}{\partial y^i}-\widehat{\theta} \left( \frac{\partial}{\partial y^i} \right) \right).
\]
This yields $\beta (t)=\gamma (t)$. Then, we observe that at the point $\tilde{\gamma}(t)= (\beta  (t), \Theta (t))$ we have
\begin{align*}
\sum_{i=1}^n \gamma_i'(t)\widehat{\theta}_{(\beta (t) , \Theta (t))} \left( \frac{\partial}{\partial y^i} \right) &=\sum_{i=1}^n \gamma_i'(t)\widehat{\mathbf{Ad}_{\Theta(t)^{-1}} \circ\theta}_{(\beta  (t), \mathbf{e}) }\left( \frac{\partial}{\partial y^i} \right) \\
 &=\widehat{\mathbf{Ad}_{\Theta(t)^{-1}}\circ \alpha}_{\gamma(t) } (\gamma'(t))
\end{align*}
Therefore
\begin{align*}
\Theta'(t)& =-\alpha_{\gamma(t) }(\gamma'(t)) \Theta(t),
\end{align*}
which yields
\[
d\Theta (t)=- d \left(\int_{\gamma[0,t]} \alpha \right)\Theta (t).
\]
Thus, any horizontal lift of the smooth curve $\gamma$ is of the form $(\gamma(t), \Theta (t))$ where $\Theta$ solves the above differential equation in the Lie group $\G$. The extension to the case where $\gamma$ is a Brownian motion on $\B$ follows from Malliavin's transfer principle.
\end{proof}

Loosely speaking, the theorem states that the horizontal Brownian motion can locally be described as a skew product of a Brownian motion  $B(t)$ on the base space and of the stochastic development in the Lie group $\G$ of the Brownian functional $ \int_{B[0,t]} \alpha $. For related results, we refer to \cite{MR2083702} and \cite{MR929666}.


\subsection{Contact manifolds and stochastic area processes on K\"ahler manifolds}\label{Sasakian manifold}

Let $(\M,\theta)$ be a $2n+1$-dimensional smooth contact manifold, i.e $\theta$ is a one-form on $\M$ which is not degenerate in the sense that $\theta \wedge (d\theta)^n$ is a volume form on $\M$. On $\M$ there is a unique smooth vector field $T$, the so-called Reeb vector field, that satisfies
\[
\theta(T)=1,\quad \mathcal{L}_T(\theta)=0,
\]
where $\mathcal{L}_T$ denotes the Lie derivative with respect to  $T$.
 As it is known (see for instance \cite{Tanno}), it is always possible to find a Riemannian metric $g$ and a $(1,1)$-tensor field $J$ on $\M$ so that for every  vector fields $X, Y$
\begin{align}\label{def J contact}
g(X,T)=\theta(X),\quad J^2(X)=-X+\theta (X) T, \quad g(X,JY)=(d\theta)(X,Y).
\end{align}
The triple $(\M, \theta,g)$ is called a contact Riemannian manifold. If the Reeb vector field $T$ is a Killing field, that is,
\[
\mathcal{L}_T g=0.
\]
the triple $(\M, \theta,g)$ is called a K-contact Riemannian manifold. On a K-contact Riemannian manifold, one can  locally equip $\M$ with a principal bundle $\pi: \M \to \B$, where $\B$ is the quotient space of $\M$ by the isometric action of the one-parameter group generated by $T$. The contact form $\theta$ is a compatible connection and $\pi$ is a totally geodesic submersion when $\B$ is equipped with the push-forward of the metric $g$. We note that the structure group $\G$ is either a circle $\mathbb S^1$ (circle bundle case) or a line $\mathbb R$ (line bundle case). Under some topological assumptions ($\B$ being a Hodge manifold), this local description of $\M$ as a principal bundle can be made global; this is the Boothby-Wang fibration, see \cite{MR112160}.

Observe that the horizontal distribution $\mathcal{H}$ is then the kernel of $\theta$ and that $\mathcal{H}$ is bracket generating because $\theta$ is a contact form and thus non-degenerate. Sasakian manifolds are the $K$-contact manifolds for which $J$ is integrable (i.e. has a vanishing Nijenhuis tensor). In that case the base manifold $\B$ is a K\"ahler manifold whose K\"ahler form is locally given by $d\alpha$ and  the stochastic line integral $ \int_{B[0,t]} \alpha$ is genuinely interpreted as a generalized stochastic area process. For further details about Sasakian manifolds and Sasakian geometry, we refer to the monograph \cite{MR2382957}.

The above setup describes the following cases studied in the monograph:
\begin{itemize}
\item \hspace{.1in} $\M$ is the $2n+1$-dimensional Heisenberg group $\mathbf{H}^{2n+1}$ equipped with its canonical contact form $\theta$. In that case $\B$ is $\mathbb{C}^n$, $\G=\mathbb R$ and the stochastic line integral $ \int_{B[0,t]} \alpha$ is the L\'evy stochastic area \eqref{area heisenberg multi} .
\item \hspace{.1in} $\M$ is the $2n+1$-dimensional sphere $\mathbb{S}^{2n+1}$ equipped with its canonical contact form. In that case $\B$ is the complex projective space $\mathbb{C}P^n$, $\G=\mathbb S^1$ and the stochastic line integral $ \int_{B[0,t]} \alpha$ is the  stochastic area of Definition \ref{stochastic area sphere} .
\item \hspace{.1in} $\M$ is the $2n+1$-dimensional complex anti-de Sitter space $\mathbf{AdS}^{2n+1}(\mathbb{C})$ equipped with its canonical contact form. In that case $\B$ is the complex hyperbolic space $\mathbb{C}H^n$, $\G=\mathbb S^1$ and the stochastic line integral $ \int_{B[0,t]} \alpha$ is the stochastic area of Definition \ref{stochastic area hyperbolic}.
\end{itemize}

In addition to those examples which are later studied in detail in the text, it might appear interesting to develop a general theory of stochastic areas in the context of homogeneous K\"ahler manifolds. It is known (see \cite{MR77878} and \cite{MR66011}) that all compact, simply connected, homogeneous K\"ahler manifolds are generalized flag manifolds and such spaces have been extensively studied in the literature.

A generalized flag manifold is a homogeneous space of the form $\G/\mathbf{K}$ where $\G$ is a compact Lie group and $\mathbf{K}$ is  the centralizer of a torus in $\G$. The list of all compact generalized flag manifolds associated with the classical Lie groups is as follows:

\begin{itemize}
\item \hspace{.1in} \textbf{Type $A_{n}$}: For $n=m+\sum n_i$, $1<n_1<\cdots<n_k$
\[
\mathbf{SU}(n)/\mathbf{S} \left( \mathbf{U}(n_1) \times \cdots \times \mathbf{U}(n_k) \times \mathbf{U}(1)^m \right)
\]
\item \hspace{.1in} \textbf{Type $B_{n}$}: 
\[
\mathbf{SO}(2n+1)/ \mathbf{U}(n_1) \times \cdots \times \mathbf{U}(n_k) \times \mathbf{U}(1)^m \times \mathbf{SO}(2\ell+1) 
\]
\item \hspace{.1in} \textbf{Type $C_{n}$}: 
\[
\mathbf{Sp}(n)/ \mathbf{U}(n_1) \times \cdots \times \mathbf{U}(n_k) \times \mathbf{U}(1)^m \times \mathbf{Sp}(\ell) 
\]
\item \hspace{.1in} \textbf{Type $D_{n}$}: 
\[
\mathbf{SO}(2n)/ \mathbf{U}(n_1) \times \cdots \times \mathbf{U}(n_k) \times \mathbf{U}(1)^m \times \mathbf{SO}(2\ell) 
\]
\end{itemize}
For the exceptional Lie groups $\mathbf{G}_2$, $\mathbf{F}_4$, $\mathbf{E}_6$, $\mathbf{E}_7$, $\mathbf{E}_8$ there are 3,11,16,31,40 non-isomorphic generalized flag manifolds respectively.

Generalized flag manifolds are naturally equipped with a K\"ahler Einstein metric, see Chapter 8 in \cite{MR2011126},  and the associated $\mathbb{S}^1$ Boothby-Wang fibrations have been studied in \cite{MR3943490}. In particular, the area form $\alpha$ on a generalized flag manifold $\G/\mathbf{K}$ can be computed using the theory of representation of Lie algebras.
For this form, it appears reasonable to infer that a L\'evy area type formula should be available for 
\[
\mathbb{E} \left( e^{i \int_{w[0,t]} \alpha} \mid w(t)=x \right)
\]
where $w$ is a Brownian motion on $\G/\mathbf{K}$.

\subsection{Positive and negative 3-Sasakian manifolds and quaternionic stochastic areas}\label{3Sasakian manifold}
The 3-Sasakian structures are the quaternionic analogue of the Sasakian structures described in the previous section. Consider a smooth $(4n+3)$-dimensional Riemannian manifold $(\M,g)$, admitting three distinct K-contact structures i.e.\ non-degenerate one-forms $\theta_\alpha$, for $\alpha=1,2,3$ such that $(\M,\theta_\alpha,g)$ is a contact Riemannian manifold and each Reeb vector field $Z_\alpha$ is Killing for the Riemannian metric $g$. Furthermore, we assume that
\begin{equation}\label{eq:conditions3K}
(a)\quad g(Z_\alpha,Z_\beta) = \delta_{\alpha\beta}, \qquad (b)\quad [Z_\alpha,Z_\beta] = 2\epsilon_{\alpha\beta\gamma} Z_\gamma, 
\end{equation}
where $\epsilon_{\alpha\beta\gamma}$ denotes the Levi-Civita symbol. We call $(\M,g)$ a 3K-contact (resp.\ negative 3K-contact) structure if for distinct $\alpha,\beta,\gamma\in \{1,2,3\}$ it holds
\begin{equation}
J_\alpha J_\beta =\epsilon_{\alpha\beta\gamma} J_\gamma, \qquad (\text{resp. } J_\alpha J_\beta =-\epsilon_{\alpha\beta\gamma} J_\gamma),
\end{equation}
and $J_\alpha$ is defined from $\theta_\alpha$ as in \eqref{def J contact}. The bundle generated by the Reeb vector fields $\mathcal V = \mathrm{span}\{Z_1,Z_2,Z_3\}$ is integrable and  letting
\begin{equation}
\mathcal{H} = \bigcap_{\alpha=1}^3 \ker\theta_\alpha,
\end{equation}
we have $T_p \M = \mathcal{H}_p \oplus \mathcal V_p$, with $\mathcal{H} \perp \mathcal V$. On a 3K-contact Riemannian manifold, one can  locally equip $\M$ with a principal bundle $\pi: \M \to \B$, where $\B$ is the quotient space of $\M$ by the isometric action of the 3-parameter group generated by $Z_1,Z_2,Z_3$. The quaternionic contact form 
\[
\theta= \theta_1 I +\theta_2 J +\theta_3 K
\]
where  $I,J,K \in \mathfrak{su}(2)$ are given by
\[
I=\left(
\begin{array}{ll}
i & 0 \\
0&-i 
\end{array}
\right), \quad 
J= \left(
\begin{array}{ll}
0 & 1 \\
-1 &0 
\end{array}
\right), \quad 
K= \left(
\begin{array}{ll}
0 & i \\
i &0 
\end{array}
\right)
\]
is a compatible connection and $\pi$ is a totally geodesic submersion when $\B$ is equipped with the push-forward of the metric $g$. We note that the structure group $\G$ is $\mathbf{SU}(2)$. 3-Sasakian manifolds (resp. 3-negative Sasakian manifolds) are the $3K$-contact manifolds (resp. negative 3K-contact manifolds) for which $J$-bundle generated by $J_1,J_2,J_3$ satisfies some integrability condition (see \cite{MR1798609}). In that case the base manifold $\B$ is a quaternionic K\"ahler manifold whose quaternionic K\"ahler form is locally given by $d\alpha$ and  the stochastic line integral $ \int_{B[0,t]} \alpha$ is  interpreted as a quaternionic generalized stochastic area process. 
The above setup describes the following cases studied in the monograph:
\begin{itemize}
\item \hspace{.1in} $\M$ is the $4n+3$-dimensional sphere $\mathbb{S}^{4n+3}$ equipped with its  3 canonical contact forms coming from the isometric action of $\mathbf{SU}(2)$ on $\mathbb{S}^{4n+3}$. In that case $\B$ is the quaternionic projective space $\mathbb{H}P^n$, $\G=\mathbf{SU}(2)$ and the stochastic line integral $ \int_{B[0,t]} \alpha$ is the  quaternionic stochastic area of Definition \ref{quaternionic stochastic area sphere} .
\item \hspace{.1in} $\M$ is the $4n+3$-dimensional quaternionic anti-de Sitter space $\mathbf{AdS}^{4n+3}(\mathbb{H})$ equipped with its  3 canonical contact forms coming from the isometric action of $\mathbf{SU}(2)$ on $\mathbf{AdS}^{4n+3}(\mathbb{H})$. In that case $\B$ is the quaternionic hyperbolic space $\mathbb{H}H^n$, $\G=\mathbf{SU}(2)$ and the stochastic line integral $ \int_{B[0,t]} \alpha$ is the quaternionic stochastic area of Definition \ref{quaternionic stochastic area hyperbolic}.
\end{itemize}

\subsection{Horizontal Laplacian}\label{Horizontal Laplacian bundle}

We now come back to the general setup of section \ref{skew general bundle}. For explicit computations, it is useful to have a formula for the horizontal Laplacian in the local trivialization $(x,g), x \in U \subset \B,  g \in \G$ described above. If $(x^1,\dots,x^n)$ is a coordinate system on $U$, with a slight abuse of notation it will be convenient to still write $x^i$ for the lifted coordinate.

The horizontal lift  to $\M$ of $\frac{\partial}{\partial x^i}$ is given in the trivialization by $\frac{\partial}{\partial x^i}-\cf \left( \frac{\partial}{\partial x^i} \right)$. The Laplacian on $\B$ writes
\[
\Delta_{\B}=\sum_{i,j=1}^n h^{ij}\frac{\partial^2}{\partial x^i \partial x^j}+\frac{1}{\sqrt{\mathrm{det}\, h}} \sum_{i,j=1}^n \frac{\partial}{\partial x^i} \left( h^{ij}\sqrt{\mathrm{det}\, h} \right) \frac{\partial}{\partial x^j}
\]
where $h^{ij}=h(dx^i,dx^j)$ is the inverse of the matrix $h_{ij}= h\left( \frac{\partial}{\partial x^i},\frac{\partial}{\partial x^i}\right)$ and $\mathrm{det}\, h$ denotes the determinant of the matrix $h^{ij}$. 
We deduce that the horizontal Laplacian is given by
\begin{align*}
\Delta_{\mathcal H}  &= \sum_{i,j=1}^n h^{ij} \left(\frac{\partial}{\partial x^i}-\cf \left( \frac{\partial}{\partial x^i} \right)\right)\left(\frac{\partial}{\partial x^j}-\cf \left( \frac{\partial}{\partial x^j} \right)\right) \\
 & +\frac{1}{\sqrt{\mathrm{det}\, h}} \sum_{i,j=1}^n \frac{\partial}{\partial x^i} \left( h^{ij}\sqrt{\mathrm{det}\, h} \right) \left(\frac{\partial}{\partial x^j}-\cf \left( \frac{\partial}{\partial x^j} \right)\right) \\
  &=\sum_{i,j=1}^n h^{ij}\frac{\partial^2}{\partial x^i \partial x^j}- \sum_{i,j=1}^n  h^{ij}\left(\frac{\partial}{\partial x^i} \cf \left( \frac{\partial}{\partial x^j} \right)+\cf \left( \frac{\partial}{\partial x^i} \right)\frac{\partial}{\partial x^j}\right) \\
  &+\sum_{i,j=1}^n h^{ij} \cf \left( \frac{\partial}{\partial x^i} \right)\cf \left( \frac{\partial}{\partial x^j} \right) +\frac{1}{\sqrt{\mathrm{det}\, h}} \sum_{i,j=1}^n \frac{\partial}{\partial x^i} \left( h^{ij}\sqrt{\mathrm{det}\, h} \right) \left(\frac{\partial}{\partial x^j}-\cf \left( \frac{\partial}{\partial x^j} \right)\right)
\end{align*}
Similarly, we can derive an expression for the norm of the horizontal gradient: For $f \in C^\infty (\M)$
\begin{align*}
{| \nabla_\mathcal{H} f |}^2&:=g(\nabla_\mathcal{H} f ,\nabla_\mathcal{H} f ) \\
 &=\sum_{i,j=1}^n h^{ij}\frac{\partial f}{\partial x^i}  \frac{\partial f}{\partial x^j}-2 \sum_{i,j}^n h^{ij} \frac{\partial f}{\partial x^i} \cf \left( \frac{\partial}{\partial x^j} \right)f+\sum_{i,j=1}^n h^{ij} \cf \left( \frac{\partial}{\partial x^i} \right) f \, \cf \left( \frac{\partial}{\partial x^j} \right)f.
\end{align*}

Those expressions can be further simplified if more assumptions are made.

\begin{Example}(Radial part of the horizontal Laplacian on Abelian bundles)

We assume here that the Lie algebra $\mathfrak g$ is one-dimensional, i.e. the Lie group $\G$ is either $\mathbb R$ or $\mathbb S^1$.  For instance this is the case if $(\M,\theta,g)$ is a Sasakian manifold as in Section  \ref{Sasakian manifold}. 
In that case, the fibers are one-dimensional and we can consider the fiber coordinate $z$ such that $\theta \left( \frac{\partial}{\partial z} \right)  =1$. The  connection form $\theta$ can  be written as
\[
\theta =dz+\sum_{i=1}^n \theta_i dx^i
\]
for some  smooth functions $\theta_i$ which do not depend on $z$. The one-form $\alpha$ on $\B$ is given by
\[
\alpha=\sum_{i=1}^n \theta_i dx^i.
\]
We note that $\cf \left( \frac{\partial}{\partial x^i} \right)=\theta_i \frac{\partial}{\partial z}$.
We will next assume that the base space $(\B,h)$ is a rank-one symmetric space, which implies that the isometric group of $\B$ acts transitively on $\B$. We denote by $r$ the radial variable (distance from a fixed point within the chart) and are interested in computing the norm of the horizontal gradient and the  horizontal Laplacian $\Delta_{\mathcal H}$ on functions $f$ depending on $(r,z)$ only. For such functions, since $\B$ is a rank-one symmetric space we have first
\[
\sum_{i,j=1}^n h^{ij}\frac{\partial f}{\partial x^i}  \frac{\partial f}{\partial x^j}=\left( \frac{\partial f}{\partial r} \right)^2.
\]
Then,
\begin{align*}
\sum_{i,j=1}^n h^{ij} \frac{\partial f}{\partial x^i} \cf \left( \frac{\partial}{\partial x^j} \right)f &=\sum_{i,j=1}^n h^{ij} \theta_j  \frac{\partial f}{\partial x^i} \frac{\partial f}{\partial z} \\
 &=\sum_{i,j=1}^n h (dx^i,dx^j) \theta_j  \frac{\partial f}{\partial x^i} \frac{\partial f}{\partial z} \\
 &=h \left( \sum_{i=1}^n  \frac{\partial f}{\partial x^i} \, dx^i , \sum_{i=1}^n \theta_i dx^i \right)\frac{\partial f}{\partial z} \\
  &=h ( df ,\alpha) \frac{\partial f}{\partial z}= h ( dr ,\alpha) \frac{\partial f}{\partial z} \frac{\partial f}{\partial r}.
\end{align*}
We now note that from Gauss lemma, $ h ( dr ,\alpha) =\alpha \left(\frac{\partial}{\partial r} \right)$ and therefore
\[
\sum_{i,j=1}^n h^{ij} \frac{\partial f}{\partial x^i} \cf \left( \frac{\partial}{\partial x^j} \right)f =\alpha \left(\frac{\partial}{\partial r} \right) \frac{\partial f}{\partial z} \frac{\partial f}{\partial r}.
\]
Finally, one has
\begin{align*}
\sum_{i,j=1}^n h^{ij} \cf \left( \frac{\partial}{\partial x^i} \right) f \, \cf \left( \frac{\partial}{\partial x^j} \right)f& =\left(\sum_{i,j=1}^n h^{ij} \theta_i \theta_j \right) \left(\frac{\partial f}{\partial z}\right)^2 \\
=&h (\alpha,\alpha)\left(\frac{\partial f}{\partial z}\right)^2 ={| \alpha |}^2 \left(\frac{\partial f}{\partial z}\right)^2.
\end{align*}
One concludes
\[
{| \nabla_\mathcal{H} f |}^2=\left( \frac{\partial f}{\partial r} \right)^2-2 \alpha \left(\frac{\partial}{\partial r} \right) \frac{\partial f}{\partial z} \frac{\partial f}{\partial r}+{| \alpha |}^2 \left(\frac{\partial f}{\partial z}\right)^2.
\]
To compute the horizontal Laplacian, the easiest way is to observe that the volume measure of $\M$ is given in the coordinates $(x^1,\dots,x^n,\theta)$ by $\sqrt{\mathrm{det}\, h} \, dx d\theta$, from which we deduce using the above formula for the horizontal gradient that
\[
\Delta_{\mathcal H} f =\frac{\partial^2 f}{\partial r^2}+\frac{1}{v}\frac{\partial v}{\partial r} \frac{\partial f}{\partial r}-2 \alpha \left(\frac{\partial}{\partial r} \right) \frac{\partial^2 f}{\partial z\partial r}+{| \alpha |}^2 \frac{\partial^2 f}{\partial z^2}
\]
where $v=\sqrt{\mathrm{det}\, h}$. 

The above results are easy to be generalized to the case where $\mathfrak{g}=\mathbb R^m$ and $\B$ is still a rank-one symmetric space. In that case, we have $m$ fiber coordinates $z^1,\dots,z^m$. The  connection form $\theta$ can then be written as $\theta=(\theta^1,\dots,\theta^m)$ where
\[
\theta^k =dz^k+\sum_{i=1}^n \theta^k_i dx^i
\]
for some  smooth functions $\theta^k_i$ which do not depend on $z$. The one-form $\alpha$ on $\B$ is given by $\alpha=(\alpha^1,\dots,\alpha^m)$ where
\[
\alpha^k=\sum_{i=1}^n \theta^k_i dx^i.
\]
One has then
\[
\cf \left( \frac{\partial}{\partial x^i} \right)=\sum_{k=1}^m \theta^k_i \frac{\partial}{\partial z^k}.
\]
One can then compute as before
\[
{| \nabla_\mathcal{H} f |}^2=\left( \frac{\partial f}{\partial r} \right)^2-2\sum_{k=1}^m  \alpha^k \left(\frac{\partial}{\partial r} \right) \frac{\partial f}{\partial z^k} \frac{\partial f}{\partial r}+\sum_{k,l=1}^m h(\alpha^k,\alpha^l) \frac{\partial f}{\partial z^k}\frac{\partial f}{\partial z^l}
\]
and
\[
\Delta_{\mathcal H} f =\frac{\partial^2 f}{\partial r^2}+\frac{1}{v}\frac{\partial v}{\partial r} \frac{\partial f}{\partial r}-2\sum_{k=1}^m  \alpha^k \left(\frac{\partial}{\partial r} \right) \frac{\partial^2 f}{\partial z^k \partial r}+\sum_{k,l=1}^m h(\alpha^k,\alpha^l) \frac{\partial^2 f}{\partial z^k \partial z^l}
\]
where $v=\sqrt{\mathrm{det}\, h}$. 
\end{Example}

\begin{Example}(Radial part of the horizontal Laplacian on $\mathbf{SU}(2)$-bundles)\label{radial part SU2 bundle}

We assume here that the Lie algebra $\mathfrak g$ is $\mathfrak{su}(2)$, i.e. the Lie group $\G$ is $\mathbf{SU}(2)$.  This is for instance the case if $(\M,\theta,g)$ is a 3-Sasakian manifold as in Section  \ref{3Sasakian manifold}. The connection form can be written as
\[
\theta= \theta^1 I +\theta^2 J +\theta^3 K
\]
where  $I,J,K \in \mathfrak{su}(2)$ are given by
\[
I=\left(
\begin{array}{ll}
i & 0 \\
0&-i 
\end{array}
\right), \quad 
J= \left(
\begin{array}{ll}
0 & 1 \\
-1 &0 
\end{array}
\right), \quad 
K= \left(
\begin{array}{ll}
0 & i \\
i &0 
\end{array}
\right).
\]
Let us denote $Z_1=\hat{I}$, $Z_2=\hat{J}$ and $Z_3=\hat{K}$ (see notation \ref{fundamental vertical field}). We note  that 
\[
\cf \left( \frac{\partial}{\partial x^i} \right)=\theta_i^1 Z_1+\theta_i ^2 Z_2+\theta_i^3 Z_3
\]
for some functions $\theta_i^k$, $k=1,2,3$, which depend only on $x$ by $\mathbf{SU}(2)$ compatibility of $\theta$. Let us observe that
\[
\alpha=\left(\sum_{i=1}^n \theta_i^1dx^i\right) I +\left(\sum_{i=1}^n \theta_i^2dx^i\right) J +\left(\sum_{i=1}^n \theta_i^3dx^i\right) K
\]
We will denote $\alpha^l=\sum_{i=1}^n \theta_i^ldx^i$, $l=1,2,3$.
As before, we  assume that the base space $(\B,h)$ is a rank-one symmetric space, and denote by $r$ the radial variable. We are interested in computing the norm of the horizontal gradient and the  horizontal Laplacian $\Delta_{\mathcal H}$ on functions $f$ depending on $(r,z)$ only. For such functions, since $\B$ is a rank-one symmetric space we still have
\[
\sum_{i,j=1}^n h^{ij}\frac{\partial f}{\partial x^i}  \frac{\partial f}{\partial x^j}=\left( \frac{\partial f}{\partial r} \right)^2.
\]
Then,
\begin{align*}
\sum_{i,j=1}^n h^{ij} \frac{\partial f}{\partial x^i} \cf \left( \frac{\partial}{\partial x^j} \right)f &=\sum_{i,j=1}^n\sum_{l=1}^3 h^{ij} \theta_j^l  \frac{\partial f}{\partial x^i} Z_l f \\
 &=\sum_{i,j=1}^n\sum_{l=1}^3 h (dx^i,dx^j) \theta_j^l  \frac{\partial f}{\partial x^i} Z_l f \\
 &=\sum_{l=1}^3 h \left( \sum_{i=1}^n  \frac{\partial f}{\partial x^i} \, dx^i , \sum_{i=1}^n \theta_i^l dx^i \right)Z_l f \\
  &=\sum_{l=1}^3 h ( df ,\alpha^l) Z_l f= \sum_{l=1}^3 h ( dr ,\alpha^l) Z_l f \frac{\partial f}{\partial r}.
\end{align*}
From Gauss lemma, $ h ( dr ,\alpha^l) =\alpha^l \left(\frac{\partial}{\partial r} \right)$ and therefore
\[
\sum_{i,j}^n h^{ij} \frac{\partial f}{\partial x^i} \cf \left( \frac{\partial}{\partial x^j} \right)f =\hat{\alpha} \left(\frac{\partial}{\partial r} \right) f \frac{\partial f}{\partial r}
\]
where
\[
\hat{\alpha}\left(\frac{\partial}{\partial r} \right)=\alpha^1 \left(\frac{\partial}{\partial r} \right)Z_1 +\alpha^2 \left(\frac{\partial}{\partial r} \right)Z_2 +\alpha^3 \left(\frac{\partial}{\partial r} \right)Z_3.
\]
Finally, one has
\begin{align*}
\sum_{i,j=1}^n h^{ij} \cf \left( \frac{\partial}{\partial x^i} \right) f \, \cf \left( \frac{\partial}{\partial x^j} \right)f& =\left(\sum_{l,m=1}^3\sum_{i,j=1}^n h^{ij} \theta^l_i \theta^m_j \right) Z_l f Z_m f \\
=&\sum_{l,m=1}^3 h (\alpha^l,\alpha^m)Z_l f Z_m f.
\end{align*}
One concludes
\[
{| \nabla_\mathcal{H} f |}^2=\left( \frac{\partial f}{\partial r} \right)^2-2 \hat{\alpha} \left(\frac{\partial}{\partial r} \right)f \frac{\partial f}{\partial r}+\sum_{l,m=1}^3 h (\alpha^l,\alpha^m)Z_l f Z_m f.
\]
To compute the horizontal Laplacian, we observe that the volume measure of $\M$ is given in the trivialization $(x^1,\dots,x^n, g)$ by $\sqrt{\mathrm{det}\, h} \, dx d\nu$ where $\nu$ is the Haar measure on $\mathbf{SU}(2)$, from which we deduce  that
\[
\Delta_{\mathcal H} f =\frac{\partial^2 f}{\partial r^2}+\frac{1}{v}\frac{\partial v}{\partial r} \frac{\partial f}{\partial r}-2 \hat{\alpha} \left(\frac{\partial}{\partial r} \right) \frac{\partial }{\partial r} f+\sum_{l,m=1}^3 h (\alpha^l,\alpha^m)Z_l  Z_m f.
\]
\end{Example}

 \chapter{Horizontal Brownian motions of the Hopf and anti-de Sitter fibrations}\label{sec-BM-complex}

In this chapter, and throughout the monograph if $z_j=x_j+iy_j$ is a local system of  holomorphic coordinates on a complex manifold $\M$, we will denote
\[
\frac{\partial }{\partial z_j }= \frac{1}{2} \left(  \frac{\partial }{\partial x_j } -i \frac{\partial }{\partial y_j }\right) , \quad \frac{\partial }{\partial \bar{z}_j }= \frac{1}{2} \left(  \frac{\partial }{\partial x_j } +i \frac{\partial }{\partial y_j }\right)
\]
and
\[
dz_j= dx_j+idy_j, \quad d\overline{z}_j= dx_j-idy_j.
\]
If
\[
h=\sum_{i,j=1}^n h\left(\frac{\partial }{\partial z_i }, \frac{\partial }{\partial \bar{z}_j } \right) dz_i \otimes d\bar{z}_j 
\]
is an Hermitian metric, then the associated Riemannian metric on $\M$ is given by
\[
\frac{1}{2}\sum_{i,j=1}^n h\left(\frac{\partial }{\partial z_i }, \frac{\partial }{\partial \bar{z}_j } \right) (dz_i \otimes d\bar{z}_j+d\bar{z}_j \otimes dz_i )
\]
and the associated Hermitian form is given by
\[
\omega=\frac{i}{2}\sum_{i,j=1}^n h \left(\frac{\partial}{\partial z_i}, \frac{\partial}{\partial \overline{z}_j} \right) dz_i \wedge d\bar{z}_j.
\]
The metric $h$ is called a K\"ahler metric if $\omega$ is closed, i.e. $d\omega=0$ where $d$ is the exterior derivative. The Dolbeault operators on forms are denoted $\partial$ and $\bar{\partial}$: For instance if $f$ is a smooth function,
\[
\partial f=\sum_j \frac{\partial f }{\partial z_j }dz_j
\]
and
\[
\bar{\partial} f=\sum_j \frac{\partial f }{\partial \bar{z}_j } d\overline{z}_j .
\]
The Dolbeault operators satisfy the relations
\[
\partial^2=\bar{\partial}^2=\partial \bar{\partial}+\bar{\partial} \partial=0
\]
and
\[
d=\partial+\bar{\partial}.
\]

\section{Complex Hopf fibration}

\subsection{Horizontal Brownian motion}
 The complex projective space $\mathbb{C}P^n$ can be defined as the set of complex lines in $\mathbb{C}^{n+1}$. To parametrize points in $\mathbb{C}P^n$, it is convenient to  use the local affine coordinates given by $w_j=z_j/z_{n+1}$, $1 \le j \le n$, $z \in \mathbb{C}^{n+1}$, $z_{n+1}\neq 0$. In these coordinates, the Riemannian structure of $\mathbb{C}P^n$ is easily worked out from the standard Riemannian structure of the Euclidean sphere. More precisely, let us  consider the unit sphere
\[
\bS^{2n+1}=\lbrace z=(z_1,\dots,z_{n+1})\in \mathbb{C}^{n+1}, | z |^2:=\sum_{i=1}^{n+1}|z_i|^2 =1\rbrace,
\]
then the map $\bS^{2n+1} \setminus\{z_{n+1}=0 \}  \to \mathbb{C}^n$, $ (z_1,\dots,z_{n+1}) \to (z_1/z_{n+1},\cdots,z_n/z_{n+1})$ is a submersion. Therefore, $\mathbb{C}^n$ is diffeomorphic to a dense open subset of $\mathbb{C}P^n$.  The unique Riemannian metric on $\mathbb{C}^n$ that makes the above submersion a  Riemannian submersion  induces the so-called the Fubini Study metric on $\mathbb{C}P^n$. Explicitly, one has the following:

\begin{theorem}\label{Laplace complex projective}
In affine coordinates, the Laplace-Beltrami operator on $\mathbb{C}P^n$ is given by:

\begin{equation}\label{eq-Laplacian-CPn1}
\Delta_{\mathbb{C}P^n}=4(1+|w|^2)\sum_{k=1}^n \frac{\partial^2}{\partial w_k \partial\overline{w}_k}+ 4(1+|w|^2)\mathcal{R} \overline{\mathcal{R}}
\end{equation}
where $|w|^2=\sum_{k=1}^n | w_k|^2$ and
\[
\mathcal{R}=\sum_{j=1}^n w_j \frac{\partial}{\partial w_j}.
\] 
Consequently, the Fubini Study Hermitian metric $h$ on $\mathbb{C}P^n$ satisfies
\[
h \left(\frac{\partial}{\partial w_k}, \frac{\partial}{\partial \overline{w}_j} \right)= \frac{(1+|w|^2) \delta_{kj} -\overline{w}_k w_j}{(1+|w|^2)^2},\quad 1\le k,j\le n.
\]
\end{theorem}
\begin{proof}
Since
\[
w_j=z_j/z_{n+1}
\]
one obtains that on $\bS^{2n+1}$,  for $1\leq k\leq n$
\begin{eqnarray*}
\frac{\partial}{\partial z_k} &=&\frac{1}{z_{n+1}} \frac{\partial}{\partial w_k}\\
\frac{\partial}{\partial \overline{z}_k} &=&\frac{1}{\overline{z}_{n+1}} \frac{\partial}{\partial \overline{w}_k} \\
\end{eqnarray*}
and
\begin{eqnarray*}
\frac{\partial}{\partial z_{n+1}} &=&-\frac{1}{z_{n+1}} \left(\sum_{j=1}^n w_j\frac{\partial}{\partial w_j}\right) \\
\frac{\partial}{\partial \overline{z}_{n+1}}  &=& -\frac{1}{\overline{z}_{n+1}} \left(\sum_{j=1}^n\overline{w}_j\frac{\partial}{\partial \overline{w}_j}\right).
\end{eqnarray*}
This yields
\begin{align*}
\sum_{k=1}^{n+1}\left(\frac{\partial^2}{\partial z_k\partial\overline{z}_k}+\frac{\partial^2}{\partial \overline{z}_k\partial z_k}\right) &=\frac{2}{|z_{n+1}|^2}\left(\frac{\partial^2}{\partial w_k\partial\overline{w}_k}+\mathcal{R}\overline{\mathcal{R}}\right) \\
  & =2(1+|w|^2)\left(\frac{\partial^2}{\partial w_k\partial\overline{w}_k}+\mathcal{R}\overline{\mathcal{R}}\right),
\end{align*}
where $\mathcal{R}=\sum_{k=1}^nw_k\frac{\partial}{\partial w_k}$. Therefore
\[
\Delta_{\mathbb{C}P^n}=4(1+|w|^2)\sum_{k=1}^n \frac{\partial^2}{\partial w_k \partial\overline{w}_k}+ 4(1+|w|^2)\mathcal{R} \overline{\mathcal{R}}
\]
and the Fubini Study metric is obtained by inverting the principal symbol of $\Delta_{\mathbb{C}P^n}$. Indeed, from the formula the  principal symbol of $\Delta_{\mathbb{C}P^n}$ is given by the Hermitian matrix $(1+|w|^2)(I_n+\mathrm{w} \mathrm{w}^*)$ where $\mathrm{w}$ is the column  vector in $\mathbb C^n$ with components $(w_1,\cdots,w_n)$. This matrix is the Hermitan cometric matrix, i.e. the inverse of the Hermitian metric matrix. Since the inverse of $I_n+\mathrm{w} \mathrm{w}^*$ is given by $I_n-\frac{\mathrm{w} \mathrm{w}^*}{1+|w|^2}$  the result follows.
\end{proof}

From the expression of the Fubini study metric one deduces that the Hermitian form on $\mathbb{C}P^n$ is given by
\begin{align*}
\omega &=\frac{i}{2}\sum_{k,j=1}^n h \left(\frac{\partial}{\partial w_k}, \frac{\partial}{\partial \overline{w}_j} \right) dw_k \wedge d\bar{w}_j \\
 &=\frac{i}{2}\sum_{k,j=1}^n\frac{(1+|w|^2) \delta_{kj} -\overline{w}_k w_j}{(1+|w|^2)^2} dw_k \wedge d\bar{w}_j
\end{align*}
Consider then the function 
\[
K(w)=\ln (1+|w|^2).
\]
A computation shows that
\begin{align*}
\frac{i}{2}\partial \bar{\partial}K&=\frac{i}{2}\sum_{k,j=1}^n \frac{(1+|w|^2) \delta_{kj} -\overline{w}_k w_j}{(1+|w|^2)^2} dw_k \wedge d\bar{w}_j=\omega.
\end{align*}
Therefore $d\omega=(\partial +\bar{\partial})\omega=0$ and the Fubini Study metric is K\"ahler. The function $K$ is called a K\"ahler potential.

One can also compute the Laplacian and the Fubini Study metric in the real affine coordinates $w_j=u_j+iv_j$. A straightforward computation yields
\[
\mathcal{R}=\mathcal{R}_1+i\mathcal{R}_2
\]
with
\[
\mathcal{R}_1=\sum_{j=1}^n u_j \frac{\partial}{\partial u_j}+v_j \frac{\partial}{\partial v_j}
\]
\[
\mathcal{R}_2=\sum_{j=1}^n v_j \frac{\partial}{\partial u_j}-u_j \frac{\partial}{\partial v_j}
\]
and
\[
\Delta_{\mathbb{C}P^n}=(1+|w|^2)\left(\sum_{j=1}^n \left( \frac{\partial^2}{\partial u_j^2} +\frac{\partial^2}{\partial v_j^2} \right) + \mathcal{R}_1^2+\mathcal{R}_2^2  \right).
\]
In those real affine coordinates the matrix of the Riemannian cometric is given by
\[
(1+|w|^2) \left(\mathrm{I}_{2n}+\mathrm{R}_1 \mathrm{R}^*_1+\mathrm{R}_2 \mathrm{R}^*_2 \right)
\]
where $\mathrm{R}_i$ is the real valued column vector with the $2n$ components
 \[
 \mathrm{R}_i =\left(  \mathcal{R}_i u_j, \mathcal{R}_i v_j \right)_{1 \le j \le n}.
 \]
 The two vectors $\mathrm{R}_1$ and $\mathrm{R}_2$ are orthogonal and both of them have norm $|w|^2$. It follows that the matrix of the Riemannian metric (which is the inverse of the cometric matrix) is given by
 \[
 \frac{1}{1+|w|^2} \left( \mathrm{I}_{2n}-\frac{1}{1+|w|^2} \left(\mathrm{R}_1 \mathrm{R}^*_1+\mathrm{R}_2 \mathrm{R}^*_2\right)\right).
 \]
 We also note that the determinant of the cometric matrix is $(1+|w|^2)^{2n+2}$ which implies that the volume Riemannian measure is $\frac{dw}{(1+|w|^2)^{n+1}}$.

The affine coordinates do not provide a global chart of $\mathbb{C}P^n$, since an hypersurface at $\infty$ is not described by this chart.  However, the above discussion is the local description of a globally defined Riemannian submersion $\bS^{2n+1} \to \mathbb{C}P^n$, that can be constructed as follows. There is an isometric group action of $\mathbb{S}^1=\mathbf{U}(1)$ on $\bS^{2n+1}$ which is  defined by $$e^{i\theta}\cdot(z_1,\dots, z_n) = (e^{i\theta} z_1,\dots, e^{i\theta} z_n). $$

The quotient space $\bS^{2n+1} / \mathbf{U}(1)$ can be identified with $\mathbb{C}P^n$ and the projection map $$\pi :  \bS^{2n+1} \to \mathbb{C}P^n$$ is a Riemannian submersion with totally geodesic fibers isometric to $\mathbf{U}(1)$. The fibration
\[
\mathbf{U}(1) \to \bS^{2n+1} \to \mathbb{C}P^n
\]
 is called the Hopf fibration. We note that the unitary group $\mathbf{U}(n+1) $ acts transitively and isometrically on $\bS^{2n+1}$. This action commutes with the $\mathbf{U}(1)$ action. Therefore $\mathbf{U}(n+1) $  also acts transitively and isometrically on $\mathbb{C}P^n$. The stabilizer by this action of a point in $\mathbb{C}P^n$ is isomorphic to the group $\mathbf{U}(1)\mathbf{U}(n)$. It follows that $\mathbb{C}P^n$ is a Riemannian homogeneous space which can be identified with the coset space  $\frac{\mathbf{U}(n+1)}{ \mathbf{U}(1)\mathbf{U}(n)}$.

\

The submersion $\pi$ allows to construct the Brownian motion on $\mathbb{C}P^n$ from the Brownian motion on $\bS^{2n+1}$. Indeed,  let $(z(t))_{t \ge 0}$ be a Brownian motion on  $\bS^{2n+1}$ started at the north pole \footnote{We will call north pole the point with complex coordinates $z_1=0,\dots, z_{n+1}=1$. }. Since $\mathbb{P}( \exists t \ge 0, z_{n+1}(t)=0 )=0$, one can use the local description of the submersion $\pi$ to infer that
\begin{align}\label{BMsphere}
w(t)= \left( \frac{z_1(t)}{z_{n+1}(t)} , \dots, \frac{z_n(t)}{z_{n+1}(t)}\right), \quad t \ge 0,
\end{align}
is a Brownian motion on $\mathbb{C}P^n$, see Remark \ref{harmonic submersion}. From Theorem \ref{Laplace complex projective}, the generator of $(w(t))_{t \ge 0}$ is given by
\[
2(1+|w|^2)\sum_{k=1}^n \frac{\partial^2}{\partial w_k \partial\overline{w}_k}+ 2(1+|w|^2)\mathcal{R} \overline{\mathcal{R}}.
\]

\begin{proposition}\label{prop-r-sphere-gene}
The process $r(t)=\arctan |w(t)|$ is a Jacobi diffusion with generator
\[
\frac{1}{2} \left( \frac{\partial^2}{\partial r^2}+((2n-1)\cot r-\tan r)\frac{\partial}{\partial r}\right).
\]
In particular, the density of $r(t)$, $t>0$, is given by the formula
\[
\mathbb{P} (r(t) \in dr)=2 \left(\sum_{m=0}^{+\infty} (2m+n){m+n-1\choose n-1}e^{-2m(m+n)t} P_m^{n-1,0}(\cos 2r) \right)(\sin r)^{2n-1} \cos r \,dr
\]
where
\[
P_m^{n-1,0}(x)=\frac{(-1)^m}{2^m m!(1-x)^{n-1}}\frac{d^m}{dx^m}((1-x)^{n-1+m}(1+x)^{m})
\]
is a Jacobi polynomial.

\end{proposition}

\begin{proof}
Let $r=\arctan |w|$. A direct computation shows that the operator 
\[
\Delta_{\mathbb{C}P^n}=4(1+|w|^2)\sum_{k=1}^n \frac{\partial^2}{\partial w_k \partial\overline{w}_k}+ 4(1+|w|^2)\mathcal{R} \overline{\mathcal{R}}
\]
acts on functions depending only on $r$ as
\[
\frac{\partial^2}{\partial r^2}+((2n-1)\cot r-\tan r)\frac{\partial}{\partial r}.
\]
It\^o's formula shows then that  $r(t)=\arctan |w(t)|$ is a Jacobi diffusion with generator
\[
\frac{1}{2} \left( \frac{\partial^2}{\partial r^2}+((2n-1)\cot r-\tan r)\frac{\partial}{\partial r}\right).
\]
The computation for the density of $r(t)$ follows from the known formula for the heat kernel of Jacobi diffusions, see Appendix 2 in Chapter \ref{sec-appendix-2}.
\end{proof}

\begin{remark}
If $(\beta(t))_{t \ge 0}=(\beta_1(t),\dots, \beta_{n+1}(t))_{t \ge 0}$ is a Brownian motion in $\mathbb{C}^{n+1}$, then from a classical skew-product decomposition (see Example \ref{BM on warped product}),
\begin{align*}
\frac{\beta(t)}{| \beta(t) |}=  z\left( \int_0^t \frac{ds}{ | \beta(s)|^2}\right),
\end{align*}
where $z(t)$ is a Brownian motion on the sphere $\bS^{2n+1}$. Therefore we deduce  from \eqref{BMsphere} that
\begin{align}\label{plki}
\frac{ \sqrt{ \sum_{i=1}^n | \beta_i (t)|^2}}{|\beta_{n+1} (t)|  }=\tan r \left(  \int_0^t \frac{ds}{ | \beta(s)|^2} \right).
\end{align}
The process $\sqrt{ \sum_{i=1}^n | \beta_i (t)|^2}$ is a Bessel process with dimension $2n$ and $| \beta_{n+1} (t)|$ is a Bessel process with dimension 2. The equality \eqref{plki} is  a special case of a general skew-product representation of Jacobi processes due to Warren and Yor \cite{WarrenYor}.
\end{remark}

\

Consider now the real one-form $\alpha$ on $\mathbb{C}P^n$ which is given  in local affine coordinates by
\[
\alpha=\frac{i}{2(1+|w|^2)}\sum_{j=1}^n(w_jd\overline{w}_j-\overline{w}_jdw_j)
\]
where $|w|^2=\sum_{j=1}^n |w_j|^2$. We can compute that
\[
\alpha=\frac{1}{2i}(\partial-\bar{\partial})K
\]
where $K(w)=\ln (1+|w|^2)$ is the K\"ahler potential of the Fubini Study metric. Therefore we have
\[
d\alpha=(\partial +\bar{\partial})\alpha=i \partial \bar{\partial} K=2\omega
\]
where $\omega$ is the K\"ahler form. The following definition is thus natural:

\begin{definition}\label{stochastic area sphere}
Let $(w(t))_{t \ge 0}$ be a Brownian motion on $\mathbb{C}P^n$ started at 0. The generalized stochastic area process of $(w(t))_{t \ge 0}$ is defined by
\[
\theta(t)=\int_{w[0,t]} \alpha=\frac{i}{2}\sum_{j=1}^n \int_0^t \frac{w_j(s)  d\overline{w}_j(s)-\overline{w}_j(s) dw_j(s)}{1+|w(s)|^2},
\]
where the above stochastic integrals are understood in the Stratonovich, or equivalently in the It\^o sense.
\end{definition}
 
 \begin{theorem}\label{thm-hor-BM-X}
 Let $(w(t))_{t \ge 0}$ be a Brownian motion on $\mathbb{C}P^n$ started at 0 and $(\theta(t))_{t\ge 0}$ be its stochastic area process. The $\mathbb{S}^{2n+1}$-valued diffusion process
 \[
 X(t)=\frac{e^{-i\theta(t)} }{\sqrt{1+|w(t)|^2}} \left( w(t),1\right), \quad t \ge 0
 \]
 is a horizontal Brownian motion for the submersion $\pi$.
 \end{theorem}

 \begin{proof}
This can deduced from Theorem \ref{skew-product principal bundle} but we provide all details. The key point is that the submersion $\pi$ is associated with the $\mathbf{U}(1)$-bundle structure on $\bS^{2n+1}$ that comes from the Hopf fibration
 \[
\mathbf{U}(1) \to \bS^{2n+1} \to \mathbb{C}P^n.
\]
 More precisely, the horizontal distribution of this submersion is the kernel of the standard contact (and connection) form on $\bS^{2n+1}$ 
\[
\eta=-\frac{i}{2}\sum_{j=1}^{n+1}(\overline{z}_jdz_j-z_jd\overline{z}_j),
\]
and the fibers of the submersion are the orbits of the Reeb vector field. We are therefore in the setting of Section \ref{skew general bundle} and the methods explained there can be applied here.

\

As above, let $(w_1,\dots, w_n)$ be  the local affine coordinates for $\mathbb{C}P^n$ given by $w_j=z_j/z_{n+1}$, and $\theta$ be the local fiber coordinate, i.e., $(w_1, \dots, w_n)$ parametrizes the complex lines passing through the north pole, while $\theta$ determines a point on the line that is of unit distance from the north pole. These coordinates are given by the map
\begin{align}\label{invar}
(w,\theta)\longrightarrow \frac{e^{i\theta} }{\sqrt{1+|w|^2}} \left( w,1\right),
\end{align}
where  $\theta \in \R/2\pi\mathbb{Z}$, and $w \in \mathbb{C}P^n$. In these coordinates, we compute that
 \begin{align}\label{eq-contact-eta-sphere}
\eta =  d\theta+\frac{i}{2(1+|w|^2)}\sum_{k=1}^n(w_kd\overline{w}_k-\overline{w}_kdw_k).
\end{align}
As a consequence, the horizontal lift to $\bS^{2n+1}$  of the vector field $\frac{\partial}{\partial w_i}$ is given by $\frac{\partial}{\partial w_i}-\alpha\left(\frac{\partial}{\partial w_i} \right)\frac{\partial}{\partial \theta}$. If we consider now a smooth curve $\gamma$ starting at 0 in $\mathbb{C}P^n$, we can write in the homogeneous system of coordinates
\[
\gamma(t)=(\gamma_1(t),\dots,\gamma_n(t)).
\]
Since
\[
\gamma'(t)=\sum_{i=1}^n \gamma'_j(t)\frac{\partial}{\partial w_j},
\]
we deduce that the horizontal lift, $Z$, of $\gamma$ at the north pole is given in the set of coordinates \eqref{invar}, by
\begin{align*}
Z'(t)&=\sum_{j=1}^n \gamma'_j(t)\left(\frac{\partial}{\partial w_j}-\alpha\left(\frac{\partial}{\partial w_j} \right)\frac{\partial}{\partial \theta} \right) \\
 &=\sum_{j=1}^n \gamma'_j(t)\frac{\partial}{\partial w_j}- \sum_{j=1}^n \alpha\left(\frac{\partial}{\partial w_j} \right) \gamma'_j(t)\frac{\partial}{\partial \theta}.
\end{align*}
As a consequence,
\[
Z(t)=\frac{e^{-i\theta(t)} }{\sqrt{1+|\gamma(t)|^2}} \left( \gamma(t),1\right),
\]
with
\[
\theta(t) =\sum_{i=1}^n \int_0^t \alpha\left(\frac{\partial}{\partial w_i} \right)  \gamma_i'(t)dt =\int_{\gamma[0,t]} \alpha.
\]

Similarly, the lift $X(t)$ of the Brownian motion $(w(t))_{t \ge 0}$ is $\frac{e^{-i\theta(t)} }{\sqrt{1+|w(t)|^2}} \left( w(t),1\right)$ with
\[
\theta(t) =\sum_{i=1}^n \int_0^t \alpha\left(\frac{\partial}{\partial w_i} \right)  dw_i =\int_{w[0,t]} \alpha.
\]
\end{proof}

\begin{theorem}
In the system of coordinates \eqref{invar}, the horizontal Laplacian of the Hopf fibration is given by
\begin{equation}\label{eq-Del-H-sphere}
\Delta_{\mathcal H}=4(1+|w|^2)\sum_{k=1}^n \frac{\partial^2}{\partial w_k \partial\overline{w}_k}+ 4(1+|w|^2)\mathcal{R} \overline{\mathcal{R}}+|w|^2\ \frac{\partial^2}{\partial \theta^2}-2i(|w|^2+1)(\mathcal{R} -\overline{\mathcal{R}})\frac{\partial}{\partial\theta}.
\end{equation}
and the generator of the horizontal Brownian motion is  $\frac{1}{2} \Delta_{\mathcal H}$.
\end{theorem}

\begin{proof}
We know that the Laplace-Beltrami operator on $\mathbb{C}P^n$ is given by:

\begin{equation*}
\Delta_{\mathbb{C}P^n}=4(1+|w|^2)\sum_{k=1}^n \frac{\partial^2}{\partial w_k \partial\overline{w}_k}+ 4(1+|w|^2)\mathcal{R} \overline{\mathcal{R}}
\end{equation*}
where
\[
\mathcal{R}=\sum_{j=1}^n w_j \frac{\partial}{\partial w_j}.
\] 
As we have seen,  the horizontal lift to $\bS^{2n+1}$  of the vector field $\frac{\partial}{\partial w_j}$ is given by 
\[
\frac{\partial}{\partial w_j}-\alpha\left(\frac{\partial}{\partial w_j} \right)\frac{\partial}{\partial \theta}=\frac{\partial}{\partial w_j}+\frac{i}{2}\frac{\overline{w}_j}{1+\rho^2}\frac{\partial}{\partial\theta},
\]
where $\rho=|w|$. Substituting $\frac{\partial}{\partial w_j}$ by its lift in the expression \eqref{eq-Laplacian-CPn1} of $\Delta_{\mathbb{C}P^n}$ yields the formula for the horizontal Laplacian and it follows from the general theory of Chapter 3 that the generator  of the horizontal Brownian motion is  $\frac{1}{2} \Delta_{\mathcal H}$.
\end{proof}

 \begin{theorem}\label{skew complex hopf}
 Let $r(t)=\arctan |w(t)|$. The process $\left( r(t), \theta(t)\right)_{t \ge 0}$ is a diffusion with generator
 \[
\frac{1}{2} \left( \frac{\partial^2}{\partial r^2}+((2n-1)\cot r-\tan r)\frac{\partial}{\partial r}+\tan^2r\frac{\partial^2}{\partial \theta^2}\right).
 \]
 As a consequence the following equality in distribution holds
 \[
\left( r(t) ,\theta(t) \right)_{t \ge 0}=\left( r(t),B\left(\int_0^t \tan^2 r(s)ds\right) \right)_{t \ge 0},
\]
where $(B(t))_{t \ge 0}$ is a standard Brownian motion independent from $(r(t))_{t\ge0}$.
\end{theorem}

\begin{proof}
A direct computation shows that  $\Delta_{\mathcal H}$ acts on functions depending only on $(\rho:=|w|, \theta)$ as
\[
\left(1+\rho^2\right)^2\frac{\partial^2}{\partial \rho^2}+\left(\frac{(2n-1)(1+\rho^2)}{\rho}+(1+\rho^2)\rho\right)\frac{\partial}{\partial \rho}+\rho^2\frac{\partial^2 }{\partial \theta^2}.
\]
The change of variable $\rho =\tan r$ completes the proof.
\end{proof}

\subsection{Horizontal heat kernel}

In this section, we study the distribution of the horizontal Brownian motion $(X(t))_{t \ge 0}$ on $\mathbb S^{2n+1}$. From Theorem \ref{thm-hor-BM-X} we have that
\begin{align*}
 \mathbb{E}( f(X(t)))& =\mathbb{E} \left[ f \left( \frac{e^{-i\theta(t)} }{\sqrt{1+|w(t)|^2}} \left( w(t),1\right) \right) \right] \\
  & = \int_{\mathbb{C}P^n} \int_{-\pi}^{\pi} f \left( \frac{e^{-i\theta } }{\sqrt{1+|w|^2}} \left( w,1\right) \right) p_{t/2} (\arctan |w|, \theta ) d \mu (w,\theta)
 \end{align*}
 where $\mu$ is the Riemannian volume measure on $\mathbb S^{2n+1}$ and $p_t(r,\theta)$, $r \in [0,\pi/2), \theta \in [-\pi,\pi]$ is the heat kernel of the operator
 \[
L=\frac{\partial^2}{\partial r^2}+((2n-1)\cot r-\tan r)\frac{\partial}{\partial r}+\tan^2r\frac{\partial^2}{\partial \theta^2}.
 \]
 with $2\pi$-periodic boundary condition\footnote{This means that  in this section the operator $\frac{\partial^2}{\partial \theta^2}$ is understood  as the Laplacian on the circle $\mathbb S^1$ } in the variable $\theta$.
The invariant and symmetric measure of $L$ is explicitly given by
\[
d\mu_r=\frac{2\pi^n}{\Gamma(n)}(\sin r)^{2n-1}\cos r drd\theta.
\]
We note that the normalization constant is chosen in such a way that
\[
\int_{-\pi}^{\pi}\int_0^{\frac{\pi}{2}}d\mu_r=\mu(\bS^{2n+1})=\frac{2\pi^{n+1}}{\Gamma (n+1)},
\]
so that the following change of variable holds
\[
 \int_{\mathbb{C}P^n} \int_{-\pi}^{\pi} g ( \arctan |w|, \theta ) d \mu (w,\theta)=\int_0^{\pi/2}  \int_{-\pi}^{\pi} g(r,\theta) d\mu_r
\]
for any function $g$.

\begin{prop}\label{heat hopf fibration}
For $t>0$, $r\in[0,\frac{\pi}{2})$, $ \theta\in[-\pi,\pi]$, the horizontal  kernel $p_t$ has the following spectral decomposition:
\[
p_t(r, \theta)=\frac{\Gamma(n)}{2\pi^{n+1}}\sum_{k=-\infty}^{+\infty}\sum_{m=0}^{+\infty} (2m+|k|+n){m+|k|+n-1\choose n-1}e^{-\lambda_{m,k}t+ik \theta}(\cos r)^{|k|}P_m^{n-1,|k|}(\cos 2r),
\]
where $\lambda_{m,k}=4m(m+|k|+n)+2|k|n$ and
\[
P_m^{n-1,|k|}(x)=\frac{(-1)^m}{2^m m!(1-x)^{n-1}(1+x)^{|k|}}\frac{d^m}{dx^m}((1-x)^{n-1+m}(1+x)^{|k|+m})
\]
is a Jacobi polynomial.
\end{prop}

\begin{proof}
The idea is to expand $p_t(r, \theta)$ as a Fourier series in $\theta$. Let
\[
p_t(r, \theta)=\sum_{k=-\infty}^{+\infty} e^{ik\theta}\phi_k(t,r)
\]
be such a Fourier expansion. Since $p_t$ satisfies $\frac{\partial p_t}{\partial t}=Lp_t$, we have
\[
\frac{\partial\phi_k}{\partial t}=\frac{\partial^2\phi_k}{\partial r^2}+((2n-1)\cot r-\tan r)\frac{\partial\phi_k}{\partial r}-k^2\tan^2 r\phi_k.
\]
By writing $\phi_k(t,r)$ in the form
\[
\phi_k(t,r)=e^{-2n|k|t}(\cos r)^{|k|}g_k(t, \cos 2r), 
\]
we get 
\begin{equation}\label{eq5}
\frac{\partial g_k}{\partial t}=4\Psi_k(g_k),
\end{equation}
where
\[
\Psi_k=(1-x^2)\frac{\partial^2}{\partial x^2}+[(|k|+1-n)-(|k|+1+n)x]\frac{\partial}{\partial x}.
\]
In fact (\ref{eq5}) is well-known as the Jacobi differential equation, and the eigenvectors are given by
\[
P_m^{n-1,|k|}(x)=\frac{(-1)^m}{2^m m!(1-x)^{n-1}(1+x)^{|k|}}\frac{d^m}{dx^m}((1-x)^{n-1+m}(1+x)^{|k|+m}),
\]
which satisfies that
\[
\Psi_k(P_m^{n-1,|k|})(x)=-m(m+n+|k|)P_m^{n-1,|k|}(x).
\]
Therefore, we obtain the spectral decomposition
\[
p_t(r, \theta)=\sum_{k=-\infty}^{+\infty}\sum_{m=0}^{+\infty} \alpha_{m,k}e^{-[4m(m+|k|+n)+2|k|n]t}e^{ik \theta}(\cos r)^{|k|}P_m^{n-1,|k|}(\cos 2r),
\]
where  the constants $\alpha_{m,k}$'s are to be determined by the initial condition at time $t=0$. 

To compute them, we use the fact that $(P_m^{n-1,|k|}(x)(1+x)^{|k|/2})_{m\geq0}$ is an orthogonal basis of $L^2([-1,1],(1-x)^{n-1}dx)$, i.e.,
\[
\int_{-1}^1 P_m^{n-1,|k|}(x)P_l^{n-1,|k|}(x)(1-x)^{n-1}(1+x)^{|k|}dx=\frac{2^{n+|k|}}{2m+|k|+n}\frac{\Gamma(m+n)\Gamma(m+|k|+1)}{\Gamma(m+1)\Gamma(m+n+|k|)}\delta_{ml}.
\]
For a smooth function $f(r, \theta)$, we can write
\[
f(r, \theta)=\sum_{k=-\infty}^{+\infty}\sum_{m=0}^{+\infty} b_{m,k}e^{ik \theta}(1+\cos 2r)^{|k|/2} P_m^{n-1,|k|}(\cos 2r)
\]
where the $b_{m,k}$'s are constants, and thus
\[
f(0,0)=\sum_{k=-\infty}^{+\infty}\sum_{m=0}^{+\infty} b_{m,k}2^{|k|/2} P_m^{n-1,|k|}(1).
\]
Now, since
\begin{align*}
& \int_{-\pi}^\pi\int_0^\frac{\pi}{2} p_t(r, \theta)\overline{f(r, \theta)}d\mu_r \\
&= \frac{2\pi^n}{\Gamma(n)}\int_{-\pi}^{\pi}\int_0^{\frac{\pi}{2}}p_t(r, \theta)\overline{f(r, \theta)}(\sin r)^{2n-1}\cos rdrd\theta \\
&= \frac{2\pi^n}{\Gamma(n)} \sum_{k=-\infty}^{+\infty}\sum_{m=0}^{+\infty}\int_{-\pi}^{\pi}\int_0^{\frac{\pi}{2}}\alpha_{m,k}\overline{b_{m,k}}2^{|k|/2}e^{-\lambda_{m,k}t}\bigg|P_m^{n-1,|k|}(\cos 2r)\bigg|^2(\cos r)^{2|k|+1}(\sin r)^{2n-1}drd\theta \\
&=\frac{2\pi^n}{\Gamma(n)}\sum_{k=-\infty}^{+\infty}\sum_{m=0}^{+\infty}\alpha_{m,k}\overline{b_{m,k}}e^{-\lambda_{m,k}t}2^{-n-|k|/2-1}\\
& \quad\quad\cdot\int_{-\pi}^{\pi}\left(\int_0^{\frac{\pi}{2}}2^{n+|k|+1}(\cos r)^{2|k|+1}\bigg|P_m^{n-1,|k|}(\cos 2r)\bigg|^2(\sin r)^{2n-1}dr\right)d\theta\\
&=\frac{2\pi^n}{\Gamma(n)} \sum_{k=-\infty}^{+\infty}\sum_{m=0}^{+\infty}\alpha_{m,k}\overline{b_{m,k}}e^{-\lambda_{m,k}t}2^{-n-|k|/2-1}(2\pi)\|P_m^{n-1,|k|}\|^2
\end{align*}
where $\lambda_{m,k}=4m(m+|k|+n)+2|k|n$, we obtain that
\[
\lim_{t\rightarrow 0}\int_{-\pi}^\pi\int_0^\frac{\pi}{2} p_t(r,\theta)\overline{f(r,\theta)}d\mu_r= \overline{f(0,0)}
\]
as soon as $\alpha_{m,k}=\frac{\Gamma(n)}{2\pi^{n+1}}(2m+|k|+n){m+|k|+n-1\choose n-1}$.
\end{proof}

\begin{remark}
We note that $p_t(r,\theta)=p_t (r,-\theta)$. This could also be  seen directly from the symmetries of the operator $L$. 
\end{remark}

The spectral decomposition of the heat kernel is explicit and useful but is not really geometrically meaningful. We shall now study another representation of the kernel which is more geometrically meaningful and will turn out to be much more  convenient  when dealing with the small-time asymptotics estimations.
The key idea is to consider the generator of the Riemannian Brownian motion on the sphere $\bS^{2n+1}$. This generator is the Riemannian  Laplace-Beltrami operator and its radial part $\Delta$ is given as 
\[
 \frac{\partial^2}{\partial \delta^2}+2n\cot\delta\frac{\partial}{\partial\delta},
\]
where $\delta$ denotes the Riemannian distance from the north pole. Leveraging on the relation $\cos\delta=\cos r\cos\theta$ we can write $\Delta$ in the coordinates $(r,\theta)$ as
\[
\Delta=\frac{\partial^2}{\partial r^2}+((2n-1)\cot r-\tan r)\frac{\partial}{\partial r}+\frac{1}{\cos^2r}\frac{\partial^2}{\partial\theta^2}.
\]
Observe that since $\Delta$ and $\frac{\partial}{\partial \theta}$ commute  we formally  have
\begin{align}\label{commutation}
e^{tL}=e^{-t\frac{\partial^2}{\partial\theta^2}}e^{t\Delta}.
\end{align}

This gives a way to express the sub-Riemannian heat kernel in terms of the Riemannian one.
Let us recall, see Section \ref{Jacobi diffusions},  that the Riemannian heat kernel  on $\bS^{2n+1}$ writes
\begin{equation}\label{eq6}
q_t(\cos\delta)=\frac{\Gamma(n)}{2\pi^{n+1}}\sum_{m=0}^{+\infty}(m+n)e^{-m(m+2n)t}C_m^n(\cos \delta),
\end{equation}
where 
\[
C_m^n(x)=\frac{(-1)^m}{2^m}\frac{\Gamma(m+2n)\Gamma(n+1/2)}{\Gamma(2n)\Gamma(m+1)\Gamma(n+m+1/2)}\frac{1}{(1-x^2)^{n-1/2}}\frac{d^m}{dx^m}(1-x^2)^{n+m-1/2}
\]
is a Gegenbauer polynomial.  Another expression of $q_t (\cos \delta)$ that is useful for the computation of small-time asymptotics is 
\begin{equation}\label{heat_kernel_odd}
q_t (\cos \delta)= e^{n^2t} \left( -\frac{1}{2\pi \sin \delta} \frac{\partial}{\partial \delta} \right)^n V(t,\delta)
\end{equation}
where $V(t,\delta)=\frac{1}{\sqrt{4\pi t}} \sum_{k \in \mathbb{Z}} e^{-\frac{(\delta-2k\pi)^2}{4t} }$.

Using the commutation (\ref{commutation}) and the formula $\cos \delta =\cos r \cos \theta$, we then infer the following proposition.
\begin{prop}\label{prop1}
For $t>0$, $r\in[0,\pi/2)$, $ \theta\in[-\pi,\pi]$, 
\begin{equation}\label{eq-sphere-8}
p_t(r, \theta)=\frac{1}{\sqrt{4\pi t}}\int_{-\infty}^{+\infty}e^{-\frac{(y+i \theta)^2}{4t} }q_t(\cos r\cosh y)dy.
\end{equation}
\end{prop}
\begin{proof}
Let 
\[
h_t(r, \theta)=\frac{1}{\sqrt{4\pi t}}\int_{-\infty}^{+\infty}e^{-\frac{(y+i \theta)^2}{4t} }q_t(\cos r\cosh y)dy,
\]
and ${L}_0=\frac{\partial^2}{\partial r^2}+((2n-1)\cot r-\tan r)\frac{\partial}{\partial r}$, then we have
\[
{L}={L}_0+\tan^2r\frac{\partial^2}{\partial\theta^2}.
\]
Using the fact that
\[
\frac{\partial}{\partial t}\left(\frac{e^{-\frac{y^2}{4t} }}{\sqrt{4\pi t}}\right)=\frac{\partial^2}{\partial y^2}\left(\frac{e^{-\frac{y^2}{4t}} }{\sqrt{4\pi t}}\right)
\]
and
\[
\frac{\partial}{\partial t}\left(q_t(\cos r\cos \theta) \right)=\Delta(q_t(\cos r\cos\theta)),
\]
we get
\begin{align*}
&{L}h_t(r, \theta) = \left({L}_0+\tan^2 r\frac{\partial^2}{\partial \theta^2} \right)h_t(r, \theta)\\
&=\int_{-\infty}^{+\infty}\left[ {L}_0\left(\left(\frac{e^{-\frac{(y+i \theta)^2}{4t}}}{\sqrt{4\pi t}}\right)q_t(\cos r\cosh y)\right)+\tan^2r\frac{\partial^2}{\partial \theta^2}\left(\frac{e^{-\frac{(y+i \theta)^2}{4t}}}{\sqrt{4\pi t}}\right)q_t(\cos r\cosh y)  \right]dy \\
&= \int_{-\infty}^{+\infty}\left[\left(\Delta-\frac{1}{\cos^2r}\frac{\partial^2}{\partial \theta^2}\right)\left(\left(\frac{e^{-\frac{(y+i \theta)^2}{4t}}}{\sqrt{4\pi t}}\right)q_t(\cos r\cosh y)\right) \right. \\
&\qquad\qquad \left.+\tan^2r\frac{\partial^2}{\partial \theta^2}\left(\frac{e^{-\frac{(y+i \theta)^2}{4t}}}{\sqrt{4\pi t}}\right)q_t(\cos r\cosh y)\right]dy\\
&=\int_{-\infty}^{+\infty} \left[\left(\frac{e^{-\frac{(y+i \theta)^2}{4t}}}{\sqrt{4\pi t}}\right)\frac{\partial q_t}{\partial t}+\frac{1}{\cos^2 r}\frac{\partial}{\partial t}\left(\frac{e^{-\frac{(y+i \theta)^2}{4t}}}{\sqrt{4\pi t}}\right)q_t-\tan^2r\frac{\partial}{\partial t}\left(\frac{e^{-\frac{(y+i \theta)^2}{4t}}}{\sqrt{4\pi t}}\right)q_t\right]dy \\
&= \frac{\partial }{\partial t}h_t(r, \theta).
\end{align*}
Moreover, it suffices to check the initial condition for functions of the form $f(r, \theta)=e^{i\lambda \theta}g(r)$ where $\lambda\in\mathbb{R}$ and $g$ is smooth. We observe that
\[
\int_0^{\frac{\pi}{2}}\int_0^{2\pi}h_t(r, \theta)f(r, \theta)d\mu_r=e^{t\lambda^2}(e^{t\Delta} g)(0).
\]
Thus $h_t(r, \theta)$ is the desired horizontal heat kernel.
\end{proof}

We now derive some consequences of the integral representation for the horizontal heat kernel.

\begin{prop}
For $\lambda\in\mathbb{C}$, $\mathbf{Re}\lambda>0$, $r\in[0,\pi/2)$, $ \theta\in[-\pi,\pi]$,
\begin{equation}\label{eq-sphere-7}
\int_0^{+\infty}p_t(r, \theta)e^{-n^2t-\frac{\lambda}{t}}dt=\int_{-\infty}^{\infty}\frac{\Gamma(n)dy}{2^{n+2}\pi^{n+1}\left(\cosh \sqrt{y^2+4\lambda}-\cos r\cos( \theta+iy)\right)^{n}}.
\end{equation}
\end{prop}
\begin{proof}
From \eqref{eq-sphere-8} we have that
\begin{eqnarray*}
& & \int_0^{+\infty}p_t(r, \theta)e^{-n^2t-\frac{\lambda}{t}}dt=\int_0^{+\infty}e^{-\frac{\lambda}{t}}p_t(r, \theta)e^{-n^2t}dt\\
&=& \frac{1}{\sqrt{4\pi}}\int_{-\infty}^{+\infty}\int_0^{+\infty}e^{-n^2t-\frac{y^2+4\lambda}{4t}}q_t(\cos r\cos( \theta+iy))\frac{dt}{\sqrt{t}}dy,
\end{eqnarray*}
we want to compute 
\[
\int_0^{+\infty}e^{-n^2t-\frac{y^2+4\lambda}{4t}}q_t(\cos r\cos( \theta+iy))\frac{dt}{\sqrt{t}}.
\]
Notice that 
\begin{eqnarray*}
\int_0^{+\infty}e^{-n^2t-\frac{y^2+4\lambda}{4t}}e^{t\Delta}\frac{dt}{\sqrt{t}}
&=& \int_0^{+\infty}e^{-\frac{y^2+4\lambda}{4t}}e^{-t(n^2-\Delta)}\frac{dt}{\sqrt{t}} \\
&=& \frac{\sqrt{\pi}}{\sqrt{n^2-\Delta}}e^{-\sqrt{y^2+4\lambda}\sqrt{n^2-\Delta}}.
\end{eqnarray*}
By a result in Taylor \cite[pp. 95]{T2}, it is known that he kernel of the operator $A^{-1}e^{-tA}$ for $A=\sqrt{n^2-\Delta}$ is given by
\[
\frac{1}{2}\pi^{-(n+1)}\Gamma(n)(2\cosh t-2\cos \delta)^{-n},
\]
where  $\cos\delta=\cos r\cos( \theta+iy)$. Plug in $t=\sqrt{y^2+4\lambda}$, we obtain
\[
\frac{1}{\sqrt{n^2-\Delta}}e^{-\sqrt{y^2+4\lambda}\sqrt{n^2-\Delta}}=\frac{\Gamma(n)}{2^{n+1}\pi^{n+1}\left(\cosh \sqrt{y^2+4\lambda}-\cos r\cos( \theta+iy)\right)^{n}},
\]
hence complete the proof.
\end{proof}

We can deduce the Green function of $-L+n^2$ immediately from the above proposition.

\begin{prop}
The Green function of the conformal sub-Laplacian $-L+n^2$ is given by 
\begin{equation*}
G(r, \theta)=\frac{\Gamma\left(\frac{n}{2}\right)^2}{8\pi^{n+1}(1-2\cos r\cos \theta+\cos^2r)^{n/2}} .
\end{equation*}
\end{prop}
\begin{proof}
Let us assume $r\not =0$, $ \theta\not=0$, and let $\lambda\rightarrow 0$ in \eqref{eq-sphere-7}, we have
\begin{eqnarray*}
G(r, \theta) &=& \frac{\Gamma(n)}{2^{n+2}\pi^{n+1}}\int_{-\infty}^{+\infty}\frac{dy}{(\cosh y-\cos r\cos( \theta+iy))^n} \\
            &=& \frac{\Gamma(n)}{2^{n+2}\pi^{n+1}}\int_{-\infty}^{+\infty}\frac{dy}{(\cosh y(1-\cos r\cos \theta)-i\cos r\sin \theta\sinh y)^n} \\
            &=& \frac{\Gamma(n)}{2^{n+2}\pi^{n+1}}\frac{1}{(1-2\cos r\cos \theta+\cos^2r)^{n/2}}\int_{-\infty}^{+\infty}\frac{dy}{(\cosh y)^n}.    
\end{eqnarray*} 
Notice that 
\begin{eqnarray*}
\int_{-\infty}^{+\infty}\frac{dy}{(\cosh y)^n}&=&\frac{\pi(n-2)!!}{(n-1)!!} \qquad \mbox{when $n>0$ is odd} \\
                                              &=&\frac{2(n-2)!!}{(n-1)!!}\qquad \mbox{when $n>0$ is even}
\end{eqnarray*}
where $n!!$ denotes the double factorial such that
\begin{eqnarray*}
n!!&=&n\cdot(n-2)\cdots5\cdot 3\cdot 1 \qquad\mbox{when $n>0$ is odd} \\
    &=& n\cdot(n-2)\cdots6\cdot 4\cdot 2 \qquad\mbox{when $n>0$ is even} .
\end{eqnarray*}
Moreover, since
\begin{eqnarray*}
\Gamma(n)\frac{(n-2)!!}{(n-1)!!}&=&\frac{2^{n-1}}{\pi}\Gamma\left(\frac{n}{2}\right)^2 \qquad \mbox{when $n>0$ is odd} \\
                                              &=&2^{n-2}\Gamma\left(\frac{n}{2}\right)^2\qquad \mbox{when $n>0$ is even}
\end{eqnarray*}
we obtain
\[
\Gamma(n)\int_{-\infty}^{+\infty}\frac{dy}{(\cosh y)^n}=2^{n-1}\Gamma\left(\frac{n}{2}\right)^2 \qquad\mbox{for all $n\in\mathbb{Z}_{> 0}$}.
\]
This implies our conclusion.
\end{proof}

To conclude the section and illustrate again the usefulness of the integral representation of the horizontal heat kernel, we discuss further results in the case $n=1$. We will particularly be interested in small-time asymptotics of the heat kernels.  So up to the end of this section we assume $n=1$ and refer to \cite{BW1} for the study of the small-time asymptotics when $n \ge 2$.

Let us first describe some properties of $q_t$ that will be useful in the sequel.
From the Poisson summation formula, when $n=1$ we obtain that for $\delta \in \mathbb{R}$:
\begin{align*}
q_t (\cos \delta)&=\frac{\sqrt{\pi}e^t}{4 t^{\frac{3}{2}}}  \frac{1}{\sin\delta} \sum_{k \in \mathbb{Z}} (\delta +2k\pi)e^{-\frac{(\delta+2k\pi)^2}{4t}}\\
 & =\frac{\sqrt{\pi}e^t}{4 t^{\frac{3}{2}}} \frac{\delta}{\sin \delta} e^{-\frac{\delta^2}{4t}} \left(1 +2\sum_{k=1}^{+\infty} e^{-\frac{k^2\pi^2}{t}}\left( \cosh \frac{k\pi \delta}{t} +2k\pi \frac{\sinh \frac{k\pi \delta}{t}}{\delta} \right) \right).
\end{align*}
From those expressions, we  obtain precise estimates:

\begin{itemize}
\item \hspace{.1in} Let $\varepsilon>0$, for $x \in (-1+\varepsilon,1]$ and $t >0$:
\begin{align}\label{estimate1}
q_t (x)=\frac{\sqrt{\pi}e^t}{4 t^{\frac{3}{2}}}  \frac{{\arccos} x}{\sqrt{1-x^2}} e^{-\frac{({\arccos}x)^2}{4t}} \left(1 +R_1(t,x) \right),
\end{align}
where for some positive constants $C_1$ and $C_2$ depending only in $\varepsilon$, $\mid R_1(t,x) \mid \le C_1 e^{-\frac{C_2}{t}}$.
\item \hspace{.1in} For $x \in [1,+\infty)$ and $t >0$:
\begin{align}\label{estimate2}
q_t (x)=\frac{\sqrt{\pi}e^t}{4 t^{\frac{3}{2}}} \frac{\arcosh{x} }{\sqrt{x^2-1}} e^{\frac{(\mathrm{arcosh }x)^2}{4t}} \left(1 +R_2(t,x) \right),
\end{align}
where for some positive constants $C_3$ and $C_4$, $\mid R_2(t,x) \mid \le C_3 e^{-\frac{C_4}{t}}$.
\end{itemize}

We now derive precise asymptotics of the horizontal heat kernel when $t \to 0$. We start with the points of the form $(0,\theta)$, $\theta\in[0,\pi]$. Note that for the range $\theta\in[-\pi,0]$ one can use the symmetry $p_t(r,-\theta)=p_t(r,\theta)$.

\begin{proposition}\label{asymptotics1}
For $t>0$ and $\theta \in [0,\pi]$,
\[
p_t (0,\theta)=\frac{\pi^2 e^t}{4t^2} e^{-\frac{2\pi \theta -\theta^2}{4t}}\sum_{k \in 
\mathbb{Z}}e^{-\frac{k(k+1)\pi^2}{t}}
\frac{ (2k+1)+2k e^{-\frac{\pi}{2t} (\theta+2k\pi)}} { \left( 1+e^{-\frac{\pi}{2t} 
(\theta+2k\pi)} \right)^2},
\]
therefore, when $t \rightarrow 0$,
\[
p_t(0,\theta)=\frac{\pi^2 e^t}{4t^2} e^{-\frac{2\pi \theta -\theta^2}{4t}} \left( 
1+O(e^{-\frac{C}{t}}) \right)
\]
\end{proposition}

\begin{proof}
Let $\theta \in (0,\pi]$. From \eqref{eq-sphere-8} we have
\[
p_t (0,\theta)=\frac{1}{\sqrt{4\pi t}} \int_{-\infty}^{+\infty} e^{-\frac{y^2}{4t}} 
q_t ( \cosh (y-i\theta))) dy,
\]
But
\[
q_t (\cosh (y-i\theta))=\frac{\sqrt{\pi}e^t}{4 t^{\frac{3}{2}}}  \frac{1}{\sinh (y-i\theta)}
 \sum_{k \in \mathbb{Z}} (y-i\theta -2ik\pi)e^{\frac{(y-i\theta-2ik\pi)^2}{4t}}\\
\]
and for $k\in \mathbb{Z}$, from the residue theorem,
\[
\int_{-\infty}^{+\infty} \frac{y-i\theta-2ik\pi}{\sinh (y-i\theta)} e^{-\frac{iy}{2t}(\theta+2k\pi)}dy
=2\pi^2 e^{\frac{(\theta+2k\pi)^2-(2k+1)\pi(\theta+2k\pi) }{2t} }
\frac{ (2k+1)+2k e^{-\frac{\pi}{2t} (\theta+2k\pi)}} { \left( 1+e^{-\frac{\pi}{2t} 
(\theta+2k\pi)} \right)^2}.
\]
The result easily follows.
\end{proof}

\begin{proposition}
For $r \in \left(0,\frac{\pi}{2}\right)$, when $ t \to 0$,
\begin{equation}\label{eq-c-S-pt}
p_t (r,0) \sim   \frac{r}{\sin r} \sqrt{\frac{1}{1-r \mathrm{cotan} r} }\frac{\sqrt{\pi}e^{-\frac{r^2}{4t}}}{4t^{\frac{3}{2}}}.
\end{equation}
\end{proposition}
\begin{remark}
Here, and throughout the rest of the monograph, we use the notation $\sim$ to capture the leading term of the expansion. For instance, \eqref{eq-c-S-pt} is interpreted as 
\[
p_t (r,0) =  \frac{r}{\sin r} \sqrt{\frac{1}{1-r \mathrm{cotan} r} }\frac{\sqrt{\pi}e^{-\frac{r^2}{4t}}}{4t^{\frac{3}{2}}}(1+o(1)).
\]
\end{remark}
\begin{proof}
We fix $r \in \left(0,\frac{\pi}{2}\right)$. From the Proposition \ref{prop1} and due to the estimates \eqref{estimate1} and \eqref{estimate1} on $q_t$ we get:
\begin{align*}
p_t (r,0) \sim_{t \to 0} \frac{1}{8t^2}(J_1 (t)+J_2(t)),
\end{align*}
where
\[
J_1 (t)=\int_{ \cosh y \le \frac{1}{\cos r}} e^{- \frac{y^2 + (\mathrm{arcos} (\cos r \cosh y))^2}{4t}}
\frac{\mathrm{arcos} (\cos r \cosh y)}{\sqrt{1 -\cos^2 r \cosh^2 y}} dy
\]
and
\[
J_2 (t)=\int_{ \cosh y \ge \frac{1}{\cos r}} e^{- \frac{y^2 - (\mathrm{arcosh} (\cos r \cosh y))^2}{4t} }\frac{\mathrm{arcosh} (\cos r \cosh y)}{\sqrt{\cos^2 r \cosh^2 y-1}} dy.
\]
We now analyze the two above integrals in small times thanks to the Laplace method. 

On the interval $\left[ -\mathrm{arcosh} \frac{1}{\cos r} ,\mathrm{arcosh} \frac{1}{\cos r} \right]$, the function
\[
f(y)=y^2 + (\mathrm{arcos} (\cos r \cosh y))^2
\]
has a unique minimum which is attained at $y=0$ and, at this point:
\[
f''(0)=2(1-r\mathrm{cotan} r).
\]
Therefore, thanks to the Laplace method
\[
J_1 (t)\sim_{t \to 0} e^{-\frac{r^2}{4t}} \frac{r}{\sin r} \sqrt{\frac{4\pi t}{1-r \mathrm{cotan} r} }.
\]
We now analyze the second integral. On $\left( -\infty , -\mathrm{arcosh} \frac{1}{\cos r} \right) \cup \left( \mathrm{arcosh} \frac{1}{\cos r}, +\infty \right)$, the function
\[
g(y)=y^2 - (\mathrm{arcosh} (\cos r \cosh y))^2,
\]
has no minimum. Therefore, from the Laplace method $J_2 (t)$ is negligeable with respect to $J_1(t)$ when $t \to 0$.
\end{proof}

The previous proposition can be extended by the same method when $\theta \neq 0$. 
If we fix $r\in \left( 0, \frac{\pi}{2} \right), \theta \in [-\pi, \pi]$, then the function
\[
f(y)=(y-i\theta)^2 + (\mathrm{arcos} (\cos r \cosh y))^2,
\]
defined on the strip $\mid \mathrm{Re} (y) \mid  < \mathrm{arcosh} \frac{1}{\cos r}$ has a critical point at $iz (r,\theta)$ where $
z(r,\theta)$ is the unique solution in $[-\pi,\pi]$ to the equation:
\[
z(r,\theta)-\theta=\cos r \sin z (r,\theta)\frac{ \mathrm{arcos} (\cos \theta (r,\theta) \cos r ) }{\sqrt{1-\cos^2 r \cos^2 z (r,\theta)}}.
\]
Indeed, with $u(r,\theta)= \cos r \cos \theta$
$$\frac{\partial}{\partial \theta} \left( \theta - \cos r \sin \theta \frac{ \mathrm{arcos} (\cos \theta  \cos r ) }{\sqrt{1-\cos^2 r \cos^2 \theta}}\right)=  \frac{\sin^2 r}{1-u(r,\theta)^2} \left( 1-  \frac{u(r,\theta) {\arccos} u(r,\theta)}{\sqrt{1-u^2(r,\theta)}} \right)
$$
which is positive. So this last function is bijective from $[-\pi,\pi]$ on itself.

We observe that at the point $z (r,\theta)$,  $f''(iz (r,\theta))$ is a positive real number:

$$ f''(iz (r,\theta))= 2  \frac{\sin^2 r}{1-u(r,\theta)^2} \left( 1-  \frac{u(r,\theta) {\arccos} u(r,\theta)}{\sqrt{1-u^2(r,\theta)}} \right)
$$
where $u(r,\theta)= \cos r \cos z (r,\theta)$.
By the same method than in the previous proposition, we obtain therefore:
\begin{proposition}
Let $r\in \left( 0, \frac{\pi}{2} \right), \theta \in [-\pi, \pi]$. When $t \to 0$,

$$p_t(r,\theta) \sim \frac{1}{\sin r} \frac{\mathrm{arccos } u(r,\theta)}{\sqrt{ 1-  \frac{u(r,\theta) {\arccos} u(r,\theta)}{\sqrt{1-u^2(r,\theta)} }}} \frac{\sqrt{\pi}e^{-\frac{(z (r,\theta)-\theta)^2 \tan^2 r}{4t \sin^2 z (r,\theta)}}}{4 t^{\frac{3}{2}}}.
$$
\end{proposition}

\subsection{Characteristic function of the stochastic  area and limit theorem}

In this section, we study the characteristic function of the stochastic area functional $\theta(t)$ as given in Definition \ref{stochastic area sphere}.
Let $\lambda \ge 0$, $r \in [0,\pi/2)$ and 
\[
I(\lambda,r):=\mathbb{E}\left(e^{i \lambda \theta(t)}\mid r(t)=r\right).
\]
From Theorem \ref{skew complex hopf}, we know that
\begin{align*}
I(\lambda,r)& =\mathbb{E}\left(e^{i \lambda B\left({\int_0^t \tan^2 r(s)ds}\right)}\mid r(t)=r\right) \\
 &=\mathbb{E}\left(e^{- \frac{\lambda^2}{2} \int_0^t \tan^2 r(s)ds}\mid r(t)=r\right) 
\end{align*}
and $r$ is a diffusion with Jacobi generator
\[
\mathcal{L}^{n-1,0}=\frac{1}{2} \left( \frac{\partial^2}{\partial r^2}+((2n-1)\cot r-\tan r)\frac{\partial}{\partial r}\right)
\]
started at $0$.  The Jacobi generator 
\[
\mathcal{L}^{\alpha,\beta}=\frac{1}{2} \frac{\partial^2}{\partial r^2}+\left(\left(\alpha+\frac{1}{2}\right)\cot r-\left(\beta+\frac{1}{2}\right) \tan r\right)\frac{\partial}{\partial r}, \quad \alpha,\beta >-1
\]
is studied in details in the Appendix 2 to which we refer for further details. We denote by $q_t^{\alpha,\beta}(r_0,r)$ the transition density with respect to the Lebesgue measure of the diffusion with generator $\mathcal{L}^{\alpha,\beta}$.
\begin{theorem}\label{FThj}
For $\lambda \ge 0$, $r \in [0,\pi/2)$, and $t >0$ we have
\begin{equation}\label{eq-ft-cond}
\mathbb{E}\left(e^{i \lambda \theta(t)}\mid r(t)=r\right)=\mathbb{E}\left(e^{- \frac{\lambda^2}{2} \int_0^t \tan^2 r(s)ds}\mid r(t)=r\right) =\frac{e^{-n \lambda t}}{(\cos r)^\lambda} \frac{q_t^{n-1,\lambda}(0,r)}{q_t^{n-1,0}(0,r)}.
\end{equation}
\end{theorem}

\begin{proof}
From Proposition \ref{prop-r-sphere-gene} we have
\[
dr(t)= \frac{1}{2} \left( (2n-1)\cot r(t)-\tan r(t) \right)dt+d\gamma(t),
\]
where $\gamma$ is a standard Brownian motion.
Consider the local martingale
\begin{align*}
D_t& =\exp \left( -\lambda \int_0^t \tan r(s) d\gamma(s) -\frac{\lambda^2}{2}  \int_0^t \tan^2 r(s) ds \right)  \\
 &=\exp \left( -\lambda \int_0^t \tan r(s) dr(s)+\frac{\lambda}{2}(2n-1)t -\frac{\lambda +\lambda^2}{2}  \int_0^t \tan^2 r(s) ds \right). 
\end{align*}
From It\^o's formula, we have
\begin{align*}
\ln \cos r(t) & =-\int_0^t \tan r(s) dr(s)-\frac{1}{2} \int_0^t \frac{ds}{\cos^2 r(s)} \\
 &=-\int_0^t \tan r(s) dr(s)-\frac{1}{2} \int_0^t \tan^2 r(s) ds-\frac{1}{2} t.
\end{align*}
As a consequence, we deduce that
\[
D_t =e^{n\lambda t} (\cos r(t))^\lambda e^{- \frac{\lambda^2}{2} \int_0^t \tan^2 r(s)ds}.
\]
This expression of $D$ implies that almost surely $D_t \le e^{n\lambda t}$ and thus $D$ is a martingale. Let us denote by $\mathcal{F}$ the natural filtration of $r$ and consider the probability measure $\mathbb{P}^\lambda$ defined by
\[
\mathbb{P}_{/ \mathcal{F}_t} ^\lambda=e^{n\lambda t} (\cos r(t))^\lambda e^{- \frac{\lambda^2}{2} \int_0^t \tan^2 r(s)ds} \mathbb{P}_{/ \mathcal{F}_t}.
\]
We have then for every bounded and Borel function $f$ on $[0,\pi /2]$,
\begin{align*}
\mathbb{E}\left(f(r(t))e^{- \frac{\lambda^2}{2} \int_0^t \tan^2 r(s)ds}\right)=e^{-n\lambda t} \mathbb{E}^\lambda \left( \frac{f(r(t))}{(\cos r(t))^\lambda} \right).
\end{align*}
From Girsanov theorem, the process
\[
\beta(t):=\gamma(t)+\lambda \int_0^t \tan r(s) ds
\]
is a Brownian motion under the probability measure $\mathbb{P}^\lambda$. Since
\[
dr(t)= \frac{1}{2} \left( (2n-1)\cot r(t)-(2\lambda +1)\tan r(t) \right)dt+d\beta(t),
\]
the proof is complete.
\end{proof}

By inverting the above Fourier transform, we obtain a formula for the joint density of the couple $(r(t),\theta(t))$.

\begin{corollary}
For $\theta \in \mathbb R$, $r \in [0,\pi/2)$, and $t >0$ we have,
\[
\mathbb{P} ( r(t) \in dr, \theta(t) \in d\theta)=\frac{1}{2\pi} \int_{\mathbb R} \frac{e^{-n |\lambda | t-i\lambda \theta}}{(\cos r)^{|\lambda|}}q_t^{n-1,|\lambda|}(0,r) d\lambda.
\]
\end{corollary} 

We also deduce an expression for the Fourier transform of $\theta(t)$.
\begin{corollary}\label{cor-theta}
For $\lambda \in \mathbb{R}$ and $t \ge 0$,
\[
\mathbb{E}\left(e^{i \lambda \theta(t)}\right)=e^{-n | \lambda | t}\int_0^{\pi /2}  \frac{q_t^{n-1,| \lambda |}(0,r)}{(\cos r)^{| \lambda|}} dr 
\]
\end{corollary}

We are now in position to obtain the limit theorem for the stochastic area.

\begin{theorem}\label{limit CP}
When $t \to +\infty$, the following convergence in distribution takes place
\[
\frac{\theta(t)}{t} \to \mathcal{C}_n,
\]
where $\mathcal{C}_n$ is a Cauchy distribution with parameter $n$.
\end{theorem}

\begin{proof}
We just need to show that for $\lambda >0$, $\lim_{t\to+\infty}\mathbb{E}\left(e^{i \lambda \theta(t)/t}\right)=e^{-n  \lambda  }$. From Corollary \ref{cor-theta} we have for every $t>0$,
\[
\mathbb{E}\left(e^{i \lambda \frac{\theta(t)}{t}}\right)=e^{-n  \lambda  }\int_0^{\pi /2}  \frac{q_t^{n-1,\frac{ \lambda }{t}}(0,r)}{(\cos r)^{ \frac{\lambda}{t}}} dr 
\]
Using the formula for $q_t^{n-1, \lambda }(0,r)$ which is given in the Appendix 2 (see Equation \eqref{eq-qt-ab-0}), we obtain that
\[
\lim_{t \to \infty}  \int_0^{\pi /2}  \frac{q_t^{n-1,\frac{ \lambda }{t}}(0,r)}{(\cos r)^{ \frac{\lambda}{t}}} dr=  \int_0^{\pi /2}  q_\infty^{n-1,0}(0,r) dr =1.
\]
\end{proof}

\subsection{Skew-product decompositions on the Berger sphere and homogenisation}

To conclude the  study of the horizontal Brownian motions of the Hopf fibration, we exhibit  the connection between the  stochastic area process on $\mathbb{C}P^n$ and an interesting class of diffusions on $\bS^{2n+1}$. These diffusions naturally arise when one looks at the canonical variation of the standard metric in the direction of the fibers of the Hopf fibration.  We compute the density of these diffusions and prove a homogenisation result when the fibers collapse.  

\

As we mentioned before, the quotient space $\bS^{2n+1} / \mathbf{U}(1)$ is the projective complex space $\mathbb{C}P^n$ and the projection map $\pi :  \bS^{2n+1} \to \mathbb{C}P^n$ is a Riemannian submersion with totally geodesic fibers isometric to $\mathbf{U}(1)$. The sub-bundle $\mathcal{V}$ of ${T}\bS^{2n+1}$ formed by vectors tangent to the fibers of the submersion is referred  to as the set of \emph{vertical directions}. The sub-bundle $\mathcal{H}$ of ${T}\bS^{2n+1}$  which is normal to $\mathcal{V}$ is referred to as the set of \emph{horizontal directions}.   The standard metric  $g$ of of $\bS^{2n+1}$ can be split as
\begin{equation*}
g=g_\mathcal{H} \oplus g_{\mathcal{V}},
\end{equation*}
where the sum is orthogonal. We introduce the one-parameter family of Riemannian metrics:
\begin{equation}\label{eq-metric-B}
g_{\lambda}=g_\mathcal{H} \oplus  \frac{1}{\lambda^2 }g_{\mathcal{V}}, \quad \lambda >0,
\end{equation}
The Riemannian manifold $(\bS^{2n+1}, g_{\lambda})$ is called the Berger\footnote{For Marcel Berger (1927-2016).} sphere with parameter $\lambda >0$. The case $\lambda=1$ corresponds to the standard metric on $\bS^{2n+1}$. When $\lambda \to 0$, in the Gromov-Hausdorff sense,  $(\bS^{2n+1}, g_{\lambda})$ converges to $\bS^{2n+1}$ endowed with the Carnot-Carath\'eodory metric. When $\lambda \to \infty$, $(\bS^{2n+1}, g_{\lambda})$ converges to $\mathbb{C}P^n$ endowed with its standard Fubini-Study metric.

The following skew-product decomposition  for the Brownian motion on the Berger sphere holds.

\begin{theorem}
Let $\lambda >0$. Let $(w(t))_{t \ge 0}$ be a Brownian motion on $\mathbb{C}P^n$ started at 0 and $(\theta(t))_{t\ge 0}$ be its stochastic area process. Let $(B(t))_{t \ge 0}$ be a real Brownian motion independent from $w$. The $\mathbb{S}^{2n+1}$-valued diffusion process
 \[
 \beta^\lambda(t)=\frac{e^{i\lambda B(t)-i\theta(t)} }{\sqrt{1+|w(t)|^2}} \left( w(t),1\right), \quad t \ge 0
 \]
 is a Brownian motion on the Berger sphere $(\bS^{2n+1}, g_{\lambda})$.
\end{theorem}

\begin{proof}
Let $\lambda >0$ and let $ (\beta^\lambda(t))_{t \ge 0}$ be a Brownian motion on the Berger sphere $(\bS^{2n+1}, g_{\lambda})$. We work in the coordinates \eqref{invar}. From \eqref{eq-metric-B}  the generator of $\beta^\lambda$ is
\[
2(1+|w|^2)\sum_{k=1}^n \frac{\partial^2}{\partial w_k \partial\overline{w}_k}+ 2(1+|w|^2)\mathcal{R} \overline{\mathcal{R}}+\frac{1}{2}(\lambda^2+|w|^2)\ \frac{\partial^2}{\partial \theta^2}-i(|w|^2+1)(\mathcal{R} -\overline{\mathcal{R}})\frac{\partial}{\partial\theta}.
\]
This generator can be written as 
\[
\frac12\Delta_{\mathcal H}+\frac{1}{2}\lambda^2 \frac{\partial^2}{\partial \theta^2},
\]
where $\Delta_{\mathcal H}$ is the horizontal Laplacian of the Hopf fibration as given in \eqref{eq-Del-H-sphere}. Since $\Delta_{\mathcal H}$ and $\frac{\partial^2}{\partial \theta^2}$ commute, the result easily follows.
\end{proof}
This representation of the Brownian motion on the Berger sphere yields an interesting result.

\begin{corollary}
Let $\lambda >0$ and let $ (\beta^\lambda(t))_{t \ge 0}$ be a Brownian motion on the Berger sphere $(\bS^{2n+1}, g_{\lambda})$. Let $\eta$ be the standard contact form on $\bS^{2n+1}$. The process
\[
\gamma(t)=\frac{1}{\lambda} \int_{\beta^\lambda [0,t]} \eta,
\]
is a real Brownian motion.
\end{corollary}

\begin{proof}
In the coordinates \eqref{invar}, recall from \eqref{eq-contact-eta-sphere} we have
 \begin{align*}
\eta =  d\theta+\frac{i}{2(1+|w|^2)}\sum_{k=1}^n(w_kd\overline{w}_k-\overline{w}_kdw_k).
\end{align*}
Since
\[
\beta^\lambda(t)=\frac{e^{i\lambda B(t)-i\theta(t)} }{\sqrt{1+|w(t)|^2}} \left( w(t),1\right),
\]
we deduce that
\begin{align*}
 \int_{\beta^\lambda [0,t]} \eta &= \int_{\beta^\lambda [0,t]} d\theta+ \int_{\beta^\lambda [0,t]} \frac{i}{2(1+|w|^2)}\sum_{k=1}^n(w_kd\overline{w}_k-\overline{w}_kdw_k)  \\
  &=\int_0^t (\lambda dB(s) -d\theta (s))  +\frac{i}{2}\sum_{j=1}^n \int_0^t \frac{w_j(s)  d\overline{w}_j(s)-\overline{w}_j(s) dw_j(s)}{1+|w(s)|^2}\\
 &=\lambda B(t)-\theta(t)+\frac{i}{2}\sum_{j=1}^n \int_0^t \frac{w_j(s)  d\overline{w}_j(s)-\overline{w}_j(s) dw_j(s)}{1+|w(s)|^2} \\
 &=\lambda B(t).
\end{align*}
\end{proof}

We now turn to the homogenisation result.

\begin{theorem}
Let $ (\beta^\lambda(t))_{t \ge 0, \lambda >0}$ be a family of Brownian motions on the Berger spheres $(\bS^{2n+1}, g_{\lambda})$ started at the north pole. Let $f:\bS^{2n+1}\to \R$ be a bounded and Borel function. For every $t>0$, one has
\[
\lim_{\lambda \to \infty} \mathbb{E}\left( f(\beta^\lambda(t))\right)=\frac{1}{2\pi} \int_0^{2\pi} \mathbb{E}\left[f \left( \frac{e^{i\theta} }{\sqrt{1+|w(t)|^2}} \left( w(t),1\right)\right)\right] d\theta,
\]
where $w$ is a Brownian motion on $\mathbb{C}P^n$ started at 0.
\end{theorem}

\begin{proof}
We already know that,  in the coordinates \eqref{invar},  the generator of $\beta^\lambda$ is
\[
2(1+|w|^2)\sum_{k=1}^n \frac{\partial^2}{\partial w_k \partial\overline{w}_k}+ 2(1+|w|^2)\mathcal{R} \overline{\mathcal{R}}+\frac{1}{2}(\lambda^2+|w|^2)\ \frac{\partial^2}{\partial \theta^2}-i(|w|^2+1)(\mathcal{R} -\overline{\mathcal{R}})\frac{\partial}{\partial\theta}.
\]
By symmetry, the heat kernel of $\beta^\lambda$ only depends of the variables $r=\arctan |w|$ and $\theta$. The radial part of the generator of $\beta^\lambda$ is then
\[
L_\lambda=\frac{1}{2}\left(\frac{\partial^2}{\partial r^2}+((2n-1)\cot r-\tan r)\frac{\partial}{\partial r}+(\tan^2r+\lambda^2)\frac{\partial^2}{\partial \theta^2}\right).
\]
The heat kernel $q_t^\lambda(r,\theta)$ of the diffusion with generator $L_\lambda$ can be computed as in Proposition \ref{heat hopf fibration} and one obtains
\begin{align*}
q_t^\lambda(r,\theta)=&\frac{\Gamma(n)}{2\pi^{n+1}}\sum_{k=-\infty}^{+\infty}\sum_{m=0}^{+\infty} (2m+|k|+n){m+|k|+n-1\choose n-1}\\
&\qquad\qquad\cdot e^{-\frac{1}{2}(\lambda_{m,k}+k^2\lambda^2)t+ik \theta}(\cos r)^{|k|}P_m^{n-1,|k|}(\cos 2r),
\end{align*}
where $\lambda_{m,k}=4m(m+|k|+n)+2|k|n$ and $P_m^{n-1,|k|}$ are  Jacobi polynomials. Since 
\begin{align*}
 &\bigg|q_t^\lambda(r,\theta)-\frac{\Gamma(n)}{2\pi^{n+1}}\sum_{m=0}^{+\infty} (2m+n){m+n-1\choose n-1}e^{-\frac{1}{2}\lambda_{m,0}t}P_m^{n-1,|0|}(\cos 2r)\bigg|\\
&\le\frac{\Gamma(n)}{2\pi^{n+1}}\sum_{k\not=0}\sum_{m=0}^{+\infty} e^{-\frac{1}{2}k^2\lambda^2 t}\bigg|(2m+|k|+n){m+|k|+n-1\choose n-1}e^{-\frac{1}{2}\lambda_{m,k}t}(\cos r)^{|k|}P_m^{n-1,|k|}(\cos 2r)\bigg|,
\end{align*}
one easily deduces that in $L^2(\mu)$, for $t>0$,
\[
\lim_{\lambda \to +\infty} q_t^\lambda(r,\theta)=q_t^\infty(r),
\]
where
\[
q_t^\infty(r):=\frac{\Gamma(n)}{2\pi^{n+1}}\sum_{m=0}^{+\infty} (2m+n){m+n-1\choose n-1}e^{-2m(m+n)t}P_m^{n-1,0}(\cos 2r),
\]
and
\[
d\mu=\frac{2\pi^n}{\Gamma(n)}(\sin r)^{2n-1}\cos r drd\theta
\]
is the invariant and symmetric measure of $L_\lambda$. We now conclude by observing that $q_t^\infty(r)$ is the heat kernel at 0 of the radial part of the Brownian motion on $\mathbb{C}P^n$.
\end{proof}

For further aspects of the theory of homogenisation on homogeneous spaces we refer to \cite{MR3787732}.

 \section{Complex anti-de Sitter fibration}
 
 \subsection{Horizontal Brownian motion}\label{Section complex anti-de Sitter}
 
 As a differentiable manifold, the  complex hyperbolic space $\mathbb{C}H^n$ can be defined as the open unit ball in $\mathbb{C}^n$. Its Riemannian structure can be constructed as follows. Let 
\[
\mathbf{AdS}^{2n+1}(\mathbb{C})=\{ z \in \mathbb{C}^{n+1}, | z_1|^2+\cdots+|z_n|^2 -|z_{n+1}|^2=-1 \}
\]
be the $2n+1$ dimensional anti-de Sitter space. We equip $\mathbf{AdS}^{2n+1}(\mathbb{C})$ with its standard Lorentz metric with signature $(2n,1)$ inherited from $\mathbb C^{n+1}$. The Riemannian structure on $\mathbb{C}H^n$ is then such that the map
\begin{align*}
\begin{array}{llll}
\pi :& \mathbf{AdS}^{2n+1}(\mathbb{C}) & \to & \mathbb{C}H^n \\
  & (z_1,\dots,z_{n+1}) & \to & \left( \frac{z_1}{z_{n+1}}, \dots, \frac{z_n}{z_{n+1}}\right)
\end{array}
\end{align*}
is an indefinite Riemannian submersion whose one-dimensional fibers are definite negative. This submersion is associated with a fibration. Indeed, the group $\mathbf{U}(1)$ acts isometrically on $\mathbf{AdS}^{2n+1}(\mathbb{C})$, and the quotient space of $\mathbf{AdS}^{2n+1}(\mathbb{C})$ by  this action is isometric to $\mathbb{C}H^n$.  The fibration
\[
\mathbf{U}(1)\to \mathbf{AdS}^{2n+1}(\mathbb{C})\to\mathbb{C}H^n
\]
 is called the complex anti-de Sitter fibration and the Riemannian metric on $\mathbb{C}H^n$ the Bergman metric. We note that $\mathbf{AdS}^{2n+1}(\mathbb{C})$ is not simply connected.
 
 \
 
 To parametrize $\mathbb{C}H^n$, we will use the global affine coordinates given by $w_j=z_j/z_{n+1}$ where $(z_1,\dots, z_{n+1})\in \M$ with  $\M=\{z\in \mathbb{C}^{n,1}, \sum_{k=1}^n|z_{k}|^2-|z_{n+1}|^2<0 \}$. Those coordinates provide a diffeomorphism between $\mathbb{C}H^n$ and the unit open ball of $\mathbb C^n$

\begin{theorem}\label{Laplace CHn}
In affine coordinates, the Laplace-Beltrami operator for the Bergman metric of $\mathbb{C}H^n$ is given by
\[
\Delta_{\mathbb{C}H^n}=4(1-|w|^2)\sum_{k=1}^n \frac{\partial^2}{\partial w_k \partial\overline{w}_k}- 4(1-|w|^2)\mathcal{R} \overline{\mathcal{R}}
\]
where $|w|^2=\sum_{i=1}^n |w_i|^2 <1$ and
\[
\mathcal{R}=\sum_{j=1}^n w_j \frac{\partial}{\partial w_j}.
\] 
Consequently, the Bergman metric $h$ on $\mathbb{C}H^n$ satisfies
\[
h \left(\frac{\partial}{\partial w_i}, \frac{\partial}{\partial \overline{w}_j} \right)= \frac{(1-|w|^2) \delta_{ij} +\overline{w}_i w_j}{(1-|w|^2)^2}.
\]
\end{theorem}
\begin{proof}
Since
\[
w_j=z_j/z_{n+1}
\]
one obtains that on $\mathbf{AdS}^{2n+1}(\C)$,  for $1\leq k\leq n$
\begin{eqnarray*}
\frac{\partial}{\partial z_k} &=&\frac{1}{z_{n+1}} \frac{\partial}{\partial w_k}\\
\frac{\partial}{\partial \overline{z}_k} &=&\frac{1}{\overline{z}_{n+1}} \frac{\partial}{\partial \overline{w}_k} \\
\end{eqnarray*}
and
\begin{eqnarray*}
\frac{\partial}{\partial z_{n+1}} &=&-\frac{1}{z_{n+1}}\left(\sum_{j=1}^n w_j\frac{\partial}{\partial w_j}\right) \\
\frac{\partial}{\partial \overline{z}_{n+1}}  &=& -\frac{1}{\overline{z}_{n+1}} \left(\sum_{j=1}^n\overline{w}_j\frac{\partial}{\partial \overline{w}_j}\right).
\end{eqnarray*}
This yields
\[
\sum_{k=1}^{n}\left(\frac{\partial^2}{\partial z_k\partial\overline{z}_k}+\frac{\partial^2}{\partial \overline{z}_k\partial z_k}\right)-\frac{\partial^2}{\partial z_{n+1}\partial\overline{z}_{n+1}}-\frac{\partial^2}{\partial \overline{z}_{n+1}\partial z_{n+1}}=2(1-|w|^2)\left(\frac{\partial^2}{\partial w_k\partial\overline{w}_k}-\mathcal{R}\overline{\mathcal{R}}\right),
\]
where $\mathcal{R}=\sum_{k=1}^nw_k\frac{\partial}{\partial w_k}$. Therefore
\[
\Delta_{\mathbb{C}H^n}=4(1-|w|^2)\sum_{k=1}^n \frac{\partial^2}{\partial w_k \partial\overline{w}_k}- 4(1-|w|^2)\mathcal{R} \overline{\mathcal{R}}
\]
and the Bergman metric is obtained by inverting the principal symbol of $\Delta_{\mathbb{C}H^n}$ as it was done in the $\mathbb{C}P^n$ case.
\end{proof}

From the expression of the Bergman metric one deduces that the Hermitian form on $\mathbb{C}H^n$ is given by
\begin{align*}
\omega &=\frac{i}{2}\sum_{k,j=1}^n h \left(\frac{\partial}{\partial w_k}, \frac{\partial}{\partial \overline{w}_j} \right) dw_k \wedge d\bar{w}_j \\
 &=\frac{i}{2}\sum_{k,j=1}^n \frac{(1-|w|^2) \delta_{kj} +\overline{w}_k w_j}{(1-|w|^2)^2} dw_k \wedge d\bar{w}_j
\end{align*}
Consider then the function 
\[
K(w)=-\ln (1-|w|^2).
\]
A computation shows that
\begin{align*}
\frac{i}{2}\partial \bar{\partial}K&=\frac{i}{2}\sum_{k,j=1}^n \frac{(1-|w|^2) \delta_{kj} +\overline{w}_k w_j}{(1-|w|^2)^2} dw_k \wedge d\bar{w}_j=\omega.
\end{align*}
Therefore $d\omega=(\partial +\bar{\partial})\omega=0$ and the Bergman metric is K\"ahler. The function $K$ is therefore a K\"ahler potential for $\mathbb{C}H^n$.

One can also compute the Laplacian and the Riemannian metric in the real affine coordinates $w_j=u_j+iv_j$. As for the complex Hopf fibration
\[
\mathcal{R}=\mathcal{R}_1+i\mathcal{R}_2
\]
with
\[
\mathcal{R}_1=\sum_{j=1}^n u_j \frac{\partial}{\partial u_j}+v_j \frac{\partial}{\partial v_j}
\]
\[
\mathcal{R}_2=\sum_{j=1}^n v_j \frac{\partial}{\partial u_j}-u_j \frac{\partial}{\partial v_j}
\]
and
\[
\Delta_{\mathbb{C}H^n}=(1-|w|^2)\left(\sum_{j=1}^n \left( \frac{\partial^2}{\partial u_j^2} +\frac{\partial^2}{\partial v_j^2} \right) - \mathcal{R}_1^2-\mathcal{R}_2^2  \right). 
\]
In those real affine coordinates the matrix of the cometric is given by
\[
(1-|w|^2) \left(\mathrm{I}_{2n}-\mathrm{R}_1 \mathrm{R}^*_1-\mathrm{R}_2 \mathrm{R}^*_2 \right)
\]
where $\mathrm{R}_i$ is the real valued column vector with the $2n$ components
 \[
 \mathrm{R}_i =\left(  \mathcal{R}_i u_j, \mathcal{R}_i v_j \right)_{1 \le j \le n}.
 \]
 The two vectors $\mathrm{R}_1$ and $\mathrm{R}_2$ are orthogonal and both of them have norm $|w|^2$. It follows that the matrix of the metric (which is the inverse of the cometric matrix) is given by
 \[
 \frac{1}{1-|w|^2} \left( \mathrm{I}_{2n}+\frac{1}{1-|w|^2} \left(\mathrm{R}_1 \mathrm{R}^*_1+\mathrm{R}_2 \mathrm{R}^*_2\right)\right).
 \]
 We also note that the determinant of the cometric marix is $(1-|w|^2)^{2n+2}$ which implies that the volume Riemannian measure is $\frac{dw}{(1-|w|^2)^{n+1}}$.

\begin{proposition}
Let $(w(t))_{t \ge 0}$ be a Brownian motion on $\CH$ started at 0, i.e. the diffusion with generator
\[
\frac{1}{2} \Delta_{\mathbb{C}H^n}=2(1-|w|^2)\sum_{k=1}^n \frac{\partial^2}{\partial w_k \partial\overline{w}_k}- 2(1-|w|^2)\mathcal{R} \overline{\mathcal{R}}.
\]
The process $r(t)=\arctanh |w(t)|$ is a Jacobi diffusion with generator
\[
\frac{1}{2} \left( \frac{\partial^2}{\partial r^2}+((2n-1)\coth r+\tanh r)\frac{\partial}{\partial r}\right).
\]
In particular, the density of $r(t)$ with respect to the Lebesgue measure, $t>0$, is given by the formula
\begin{align*}
&\mathbb{P}(r(t) \in dr) \\
=& \frac{C_ne^{-n^2t/2}}{ t^{1/2} } \left( \int_r^{\infty}\frac{\sinh\theta }{(\cosh^2\theta - \cosh^2r)^{1/2}} \left(-\frac{1}{\sinh\theta}\frac{d}{d\theta}\right)^ne^{-\theta^2/(2t)} d\theta \right) (\sinh r)^{2n-1}\cosh{r} \, dr.
\end{align*}
where $C_n$ is a normalizing constant.
\end{proposition}

\begin{proof}
Let $r=\arctanh |w|$. A direct computation shows that the operator $\Delta_{\mathbb{C}H^n}$ acts on functions depending only on $r$ as
\[
\frac{\partial^2}{\partial r^2}+((2n-1)\coth r+\tanh r)\frac{\partial}{\partial r}.
\]
Equivalently, It\^o's formula shows that  $r(t)=\arctanh |w(t)|$ is a hyperbolic  Jacobi diffusion with generator
\[
\frac{1}{2} \left( \frac{\partial^2}{\partial r^2}+((2n-1)\coth r+\tanh r)\frac{\partial}{\partial r}\right).
\]
The expression of the semi-group density of $r(t)$ is known and yields the claimed formula; see Section \ref{section hyp Jacobi} (Formula \ref{heat kernel CHn}) but also \cite{Dem-Heat} and references therein.
\end{proof}

 Let $\alpha$ be the one-form on $\mathbb{C}H^n$ which is  given in affine coordinates by
\[
\alpha=\frac{i}{2(1-|w|^2)}\sum_{j=1}^n(w_jd\overline{w}_j-\overline{w}_jdw_j),
\]
where $|w|^2=\sum_{j=1}^n|w_j|^2<1$.  We can compute that
\[
\alpha=\frac{1}{2i}(\partial-\bar{\partial})K
\]
where $K(w)=-\ln (1-|w|^2)$ is the K\"ahler potential of the Bergman metric. Therefore we have
\[
d\alpha=(\partial +\bar{\partial})\alpha=i \partial \bar{\partial} K=2\omega
\]
where $\omega$ is the K\"ahler form.


 We can then naturally define the stochastic area process on $\CH$ as follows:
\begin{definition}\label{stochastic area hyperbolic}
Let $(w(t))_{t \ge 0}$ be a Brownian motion on $\CH$ started at 0. The generalized stochastic area process of $(w(t))_{t \ge 0}$ is defined by
\[
\theta(t)=\int_{w[0,t]} \alpha=\frac{i}{2}\sum_{j=1}^n \int_0^t \frac{w_j(s)  d\overline{w}_j(s)-\overline{w}_j(s) dw_j(s)}{1-|w(s)|^2},
\]
where the above stochastic integrals are understood in the Stratonovich sense or equivalently It\^o sense.
\end{definition}

As in in the Heisenberg group case or the Hopf fibration case, the stochastic area process is intimately related to a diffusion on the total space of the fibration.

\begin{theorem}\label{FGHT}
 Let $(w(t))_{t \ge 0}$ be a Brownian motion on $\mathbb{C}H^n$ started at 0 and $(\theta(t))_{t\ge 0}$ be its stochastic area process. The $\mathbf{AdS}^{2n+1}(\mathbb{C})$-valued diffusion process
 \[
 Y(t)=\frac{e^{i\theta(t)} }{\sqrt{1-|w(t)|^2}} \left( w(t),1\right), \quad t \ge 0
 \]
 is the horizontal lift at $(0,1)$ of $(w(t))_{t \ge 0}$ by the submersion $\pi$.
\end{theorem}

\begin{proof}
This can be proved using Theorem \ref{skew-product principal bundle} but as in the complex Hopf fibration case Theorem \ref{thm-hor-BM-X}  we give all details. As a difference with respect to Theorem \ref{thm-hor-BM-X} notice the sign change in $\theta(t)$.  Again, the key-point is to observe the compatibility of the submersion $\pi$ with the $\mathbf{U}(1)$-bundle structure of $\mathbf{AdS}^{2n+1}(\mathbb{C})$. Namely, the horizontal distribution of this submersion is the kernel of the standard connection form on $\mathbf{AdS}^{2n+1}(\mathbb{C})$ 
\[
\eta=-\frac{i}{2}\left(\sum_{j=1}^{n}(\overline{z}_jdz_j-z_j d\overline{z}_j)-(\overline{z}_{n+1}dz_{n+1}-z_{n+1} d\overline{z}_{n+1})\right).
\]

\

Let $(w_1,\dots, w_n)$ be  the  coordinates for $\mathbf{AdS}^{2n+1}(\mathbb{C})$ given by $w_j=z_j/z_{n+1}$, and $\theta$ be the local fiber coordinate. These coordinates are given by the map
\begin{align}\label{invar2}
(w,\theta)\longrightarrow \frac{e^{i\theta} }{\sqrt{1-|w|^2}} \left( w,1\right),
\end{align}
where  $\theta \in \R/2\pi\mathbb{Z}$, and $w \in \mathbb{C}H^n$. In these coordinates, we compute that
 \begin{align*}
\eta=-d\theta+\frac{i}{2(1-|w|^2)}\sum_{j=1}^n(w_jd\overline{w}_j-\overline{w}_jdw_j).
\end{align*}
As a consequence, the horizontal lift to $\mathbf{AdS}^{2n+1}(\mathbb{C})$  of the vector field $\frac{\partial}{\partial w_i}$ is given by $\frac{\partial}{\partial w_i}-\alpha\left(\frac{\partial}{\partial w_i} \right)\frac{\partial}{\partial \theta}$. Therefore, the lift of $(w(t))_{t \ge 0}$ is $\frac{e^{i\theta(t)} }{\sqrt{1-|w(t)|^2}} \left( w(t),1\right)$ with
\[
\theta(t) =\sum_{i=1}^n \int_0^t \alpha\left(\frac{\partial}{\partial w_i} \right)  dw_i =\int_{w[0,t]} \alpha.
\]
\end{proof}

\begin{theorem}
In the system of coordinates \eqref{invar2}, the horizontal Laplacian of the anti-de Sitter fibration is given by
\[
\Delta_{\mathcal H}=4(1-|w|^2)\sum_{k=1}^n \frac{\partial^2}{\partial w_k \partial\overline{w}_k}- 4(1-|w|^2)\mathcal{R} \overline{\mathcal{R}}+|w|^2\ \frac{\partial^2}{\partial \theta^2}+2i(1-|w|^2)(\mathcal{R} -\overline{\mathcal{R}})\frac{\partial}{\partial\theta}.
\]
and the generator of the horizontal Brownian motion is  $\frac{1}{2} \Delta_{\mathcal H}$.
\end{theorem}

\begin{proof}
Similarly as in the case of $\mathbb{C}P^n$, we first observe that the Laplace-Beltrami operator for the Bergman metric of $\mathbb{C}H^n$ is given by
\[
\Delta_{\mathbb{C}H^n}=4(1-|w|^2)\sum_{k=1}^n \frac{\partial^2}{\partial w_k \partial\overline{w}_k}- 4(1-|w|^2)\mathcal{R} \overline{\mathcal{R}}
\]
where
\[
\mathcal{R}=\sum_{j=1}^n w_j \frac{\partial}{\partial w_j}.
\] 
We now observe that the horizontal lift of the vector field $\frac{\partial}{\partial w_j}$ to $\mathbf{AdS}^{2n+1}(\mathbb{C})$   is given by 
\[
\frac{\partial}{\partial w_j}-\alpha\left(\frac{\partial}{\partial w_j} \right)\frac{\partial}{\partial \theta}=\frac{\partial}{\partial w_j}+\frac{i}{2}\frac{\overline{w}_j}{1-\rho^2}\frac{\partial}{\partial\theta},
\]
where $\rho=|w|$. Substituting $\frac{\partial}{\partial w_j}$ by its lift in the expression of $\Delta_{\mathbb{C}H^n}$ yields
\[
\Delta_{\mathcal H}=4(1-|w|^2)\sum_{k=1}^n \frac{\partial^2}{\partial w_k \partial\overline{w}_k}- 4(1-|w|^2)\mathcal{R} \overline{\mathcal{R}}+|w|^2\ \frac{\partial^2}{\partial \theta^2}+2i(1-|w|^2)(\mathcal{R} -\overline{\mathcal{R}})\frac{\partial}{\partial\theta}.
\]
\end{proof}

\begin{theorem}\label{diff-H}
Let $r(t)=\arctanh |w(t)|$. The process $\left( r(t), \theta(t)\right)_{t \ge 0}$ is a diffusion with generator
 \[
\frac{1}{2} \left( \frac{\partial^2}{\partial r^2}+((2n-1)\coth r+\tanh r)\frac{\partial}{\partial r}+\tanh^2r\frac{\partial^2}{\partial \theta^2}\right).
 \]
 As a consequence the following equality in distribution holds
 \begin{equation}\label{eq-r-theta-H}
\left( r(t) ,\theta(t) \right)_{t \ge 0}=\left( r(t),B\left(\int_0^t \tanh^2 r(s)ds\right)\right)_{t \ge 0},
\end{equation}
where $(B(t))_{t \ge 0}$ is a standard Brownian motion independent from $r$.
\end{theorem}

\begin{proof}
With $\rho=|w|$, we compute then that $\Delta_{\mathcal H}$ acts on functions depending only on $(\rho, \theta)$ as
\[
\left(1-\rho^2\right)^2\frac{\partial^2}{\partial \rho^2}+\left(\frac{(2n-1)(1-\rho^2)}{\rho}-(1-\rho^2)\rho\right)\frac{\partial}{\partial \rho}+\rho^2\frac{\partial^2 }{\partial \theta^2}.
\]
The change of variable $\rho =\tanh r$ finishes the proof.
\end{proof}

\subsection{Joint density of the radial process and stochastic area}

In this section we first compute the joint density of $(r(t),\theta(t))_{t \ge 0}$. From Theorem \ref{diff-H}, we have for $t>0, r\ge 0, \, \theta \in \R$, 
\[
\mathbb{P} \left( r(t) \in dr, \theta(t) \in d\theta\right)=\frac{2\pi^n}{\Gamma(n)}p_{t/2}(r,\theta)(\sinh r)^{2n-1}\cosh r drd\theta,
\]
where $p_{t}(r,\theta)$ is the heat kernel, with respect to the measure $\frac{2\pi^n}{\Gamma(n)} (\sinh r)^{2n-1}\cosh r drd\theta $, of the operator
\[
L=\frac{\partial^2}{\partial r^2}+((2n-1)\coth r+\tanh r)\frac{\partial}{\partial r}+\tanh^2r\frac{\partial^2}{\partial \theta^2}.
\]
\begin{prop}\label{prop1hyp}
For $t>0$, $r\in[0,+\infty)$, $ \theta\in(-\infty,+\infty)$, the  heat kernel $p_t(r,\theta)$ is given by
\begin{equation}\label{eq8}
p_t(r, \theta)=\frac{1}{\sqrt{4\pi t}}\int_{-\infty}^{+\infty}e^{\frac{(y-i \theta)^2}{4t} }q_{t,2n+1}(\cosh r\cosh y)dy,
\end{equation}
where $q_{t,2n+1}$ is the heat kernel on the real hyperbolic space of dimension $2n+1$ (see Section \ref{section hyp Jacobi}).
In particular
\begin{equation}\label{eq-pt-Gr}
p_t(r, \theta)=\frac{\Gamma(n+1)e^{-n^2t+\frac{\pi^2}{4t}}}{(2\pi)^{n+2} t}\int_{-\infty}^{+\infty}\int_0^{+\infty}\frac{e^{\frac{(y-i\theta)^2-u^2}{4t}}\sinh u\sin\left(\frac{\pi u}{2t}\right)}{\left(\cosh u+\cosh r\cosh y\right)^{n+1}}du
dy.
\end{equation}
\end{prop}
\begin{proof}
The following computations are based on geometric ideas that we will describe in a remark after the proof. We first decompose

\[
L=\frac{\partial^2}{\partial r^2}+((2n-1)\coth r+\tanh r)\frac{\partial}{\partial r}+\tanh^2r\frac{\partial^2}{\partial \theta^2}
\]

 as

\begin{equation*}
L = \Delta_{\mathbb C H^n}+\tanh^2 r \frac{\partial^2}{\partial \theta^2} ,
\end{equation*}

where 
\[
\Delta_{\mathbb C H^n}=\frac{\partial^2}{\partial r^2}+((2n-1)\coth r+\tanh r)\frac{\partial}{\partial r}
\]
 denotes the radial part of the Laplacian on the complex hyperbolic space $\mathbb C H^n$.   Note that we can also write

\[
L={\square}_{\mathbf{AdS}^{2n+1}(\mathbb{C})}+\frac{\partial^2}{\partial \theta^2},
\]
where
\[
{\square}_{\mathbf{AdS}^{2n+1}(\mathbb{C})}=\Delta_{\mathbb C H^n}-\frac{1}{\cosh^2 r}\frac{\partial^2}{\partial \theta^2}.
\]
Note that ${\square}_{\mathbf{AdS}^{2n+1}(\mathbb{C})}$ and $\frac{\partial^2}{\partial \theta^2}$ commute. Therefore
\begin{align*}
e^{tL}&=e^{t ({\square}_{\mathbf{AdS}^{2n+1}(\mathbb{C})}+\frac{\partial^2}{\partial \theta^2})}\\
 &=e^{t \frac{\partial^2}{\partial \theta^2} } e^{ t {\square}_{\mathbf{AdS}^{2n+1}(\mathbb{C})} }
\end{align*}
We deduce that the heat kernel of $L$ can be written as
\begin{align}\label{formula 1.1.1}
p_t(r,\theta)=\frac{1}{\sqrt{4\pi t}}\int_{-\infty}^{+\infty} e^{-\frac{(u-\theta)^2}{4t} }p_t^{{\square}_{\mathbf{AdS}^{2n+1}(\mathbb{C})}} (r,u)   du,
\end{align}
where $p_t^{{\square}_{\mathbf{AdS}^{2n+1}(\mathbb{C})}} (r,u) $ the heat kernel at $(0,0)$ of ${\square}_{\mathbf{AdS}^{2n+1}(\mathbb{C})}$ with respect to the measure 
 \[
 \frac{2\pi^n}{\Gamma(n)} (\sinh r)^{2n-1}\cosh r drd\theta.
 \]
 The idea to compute \eqref{formula 1.1.1} is to perform an analytic extension in the fiber variable. More precisely,   let us consider the analytic change of variables $\tau : (r,\theta) \to (r,i\theta)$  that will be applied on functions of the type $f(r,u)=g(r) e^{-i\lambda u}$, with $g$ smooth and compactly supported on $[0,+\infty)$ and $\lambda >0$. One sees that
\begin{align}\label{polinh23}
{\square}_{\mathbf{AdS}^{2n+1}(\mathbb{C})} (f\circ \tau)=(\Delta_{H^{2n+1}} f ) \circ \tau
\end{align}
where
\[
\Delta_{H^{2n+1}}=\Delta_{\mathbb C H^n}+\frac{1}{\cosh^2 r}\frac{\partial^2}{\partial \theta^2}
\]
With the change of variable $\cosh \delta= \cosh r \cosh \theta$, and after straightforward computations it holds that
\[
\Delta_{H^{2n+1}}=\frac{\partial^2}{\partial \delta^2}+2n\coth\delta\frac{\partial}{\partial \delta}.
\]

Thus $\Delta_{H^{2n+1}}$ is the radial part of the Laplacian on the real hyperbolic space of dimension $2n+1$. 
%
Since for every function $f$,
 \[
( e^{ t {\square}_{\mathbf{AdS}^{2n+1}(\mathbb{C})} } f ) (0,0)=  (e^{t\Delta_{H^{2n+1}} }(f\circ  \tau^{-1})) (0,0),
 \]
 one deduces  that for an exponential function $g$ depending only on $u$,
  \[
 \int_{-\infty}^{+\infty} g(u) p_t^{{\square}_{\mathbf{AdS}^{2n+1}(\mathbb{C})}} (r,u)   du=\int_{-\infty}^{+\infty} q_{t, 2n+1}(\cosh r \cosh u)g(-iu)  du.
 \]

 Therefore, coming back to \eqref{formula 1.1.1}, one infers that
 \[
 \frac{1}{\sqrt{4\pi t}}\int_{-\infty}^{+\infty} e^{-\frac{(u-\theta)^2}{4t} }p_t^{{\square}_{\mathbf{AdS}^{2n+1}(\mathbb{C})}} (r,u)   du=\frac{1}{\sqrt{4\pi t}}\int_{-\infty}^{+\infty}e^{\frac{(u-i \theta)^2}{4t} }q_{t,2n+1}(\cosh r\cosh u)du.
 \]
 Plugging Gruet's formula \eqref{GruetFor} (in Appendix 2) to the right hand side  of the last equality , we complete the proof. 
\end{proof}

\begin{remark}
The (hyperbolic) operator
\[
{\square}_{\mathbf{AdS}^{2n+1}(\mathbb{C})}=\Delta_{\mathbb C H^n}-\frac{1}{\cosh^2 r}\frac{\partial^2}{\partial \theta^2}.
\]
that appeared in the previous proof is the radial part of the d'Alembertian i.e. the Laplace-Beltrami operator on $\mathbf{AdS}^{2n+1}(\mathbb{C})$ equipped  with its standard Lorentz metric with signature $(2n,1)$. The decomposition 

\[
L={\square}_{\mathbf{AdS}^{2n+1}(\mathbb{C})}+\frac{\partial^2}{\partial \theta^2},
\]
is then the Lorentz analogue of the decomposition \eqref{decomposition laplace}.
\end{remark}

In the remainder of the section we give further results in the case $n=1$, showing how the integral representation of $p_t (r,\theta)$ can be used. Up to the end of the section we assume that $n=1$. In that case, note that $\mathbf{AdS}^{3}(\mathbb{C})$ is actually isometric to the Lie group $\mathbf{SL}(2,\R)$.

 Using the well-known formula for the heat kernel on the 3-dimensional real hyperbolic space, we deduce from Proposition \ref{prop1hyp} that  for $t>0, r>0 , \theta \in \R$ 
\begin{equation*}
p_t (r,\theta)=\frac{e^{-t}}{(4\pi t)^2} \int_{-\infty}^{+\infty} e^{- \frac{\arcosh^2(\cosh r \cosh y) -(y-i\theta)^2 } {4t}}
            \frac {\arcosh(\cosh r \cosh y)} {\sqrt{\cosh^2 r \cosh^2 y - 1}}dy
\end{equation*}

This formula can be used to obtain precise asymptotics  when $t \to 0$.

We start with the points of the form $(0,\theta), \theta \in \R$. For these points we have

$$p_t(0,\theta)=\frac{e^{-t}}{(4\pi t)^2} e^{-\frac{\theta^2}{4t}} \int_{-\infty}^{+\infty} e^{\frac{-iy\theta}{2t}} \frac{y}{\sinh y} dy. 
$$
 
A computation of the integral is possible using residue calculus and gives the following: 

\begin{proposition} For $ \theta \in \R$ and $t>0$,
$$p_t(0,\theta)=\frac{e^{-t}}{8 t^2} \frac{e^{-\frac{2 \pi |\theta| +  \theta^2}{4t}}} {\left(1+ e^{-\frac{ \pi |\theta|}{2t}} \right)^2} 
$$
therefore, for all $\theta\in \R$, when $t \rightarrow 0$, 
$$
p_t(0,\theta) \sim \frac{1}{8 t^2} e^{-\frac{2\pi |\theta| +\theta^2}{4t}}.
$$
 \end{proposition}  
%
%

\begin{proposition}
For $r >0$, when $ t \to 0$,
\[
p_t (r,0) \sim   \frac{1}{(4\pi t)^{\frac{3}{2}}} \frac{r}{\sinh r} \sqrt{\frac{1}{r \coth r -1} } e^{-\frac{r^2}{4t}}.
\]
\end{proposition}

\begin{proof}
We have for $r>0$
$$p_t(r,0)= \frac{e^{-t}}{(4\pi t)^2} \int_{-\infty}^{+\infty} e^{- \frac{\arcosh^2(\cosh r \cosh y) -y^2 } {4t}}
            \frac {\arcosh(\cosh r \cosh y)} {\sqrt{\cosh^2 r \cosh^2 y - 1}}dy
$$

We now analyze the  above integral in small times thanks to the Laplace method.

On $\R$, 
 the function
\[
f(y)=\arcosh (\cosh r \cosh y))^2 -y^2
\]
has a unique minimum which is attained at $y=0$ and is equal to $r^2$, at this point:
$$
f''(0)=2(r \coth r -1). 
$$
The result follows by the Laplace method.
\end{proof} 
 
The previous proposition can be extended by the same method when $\theta \neq 0$. 
Let $r>0, \theta \in \R$ and consider the function
\[
f(y)=  (\arcosh (\cosh r \cosh y))^2-(y-iz)^2.
\]
This function is well defined and holomorphic on the strip $|\mathrm{Im} (y) |  < \mathrm{arcos} \left(\frac{ - 1}{\cosh r}\right)$ 
and it has for all $r>0, \theta \in \R$ a critical point at $iz (r,\theta)$ where $
z (r,\theta)$ is the unique solution in $(-\mathrm{arcos} \left(\frac{ - 1}{\cosh r}\right),\mathrm{arcos} \left(\frac{ - 1}{\cosh r}\right))$ to the equation:
\[
z (r,\theta)-\theta=\cosh r \sin z (r,\theta)\frac{ \arcosh  (\cosh r \cos z (r,\theta) ) }{\sqrt{\cosh^2 r \cos^2 z (r,\theta) -1}}.
\]

Indeed the function $\theta \rightarrow  \cosh r \sin z (r,\theta)\frac{ \arcosh  (\cosh r \cos z (r,\theta) ) }{\sqrt{\cosh^2 r \cos^2 z (r,\theta) -1}}$ is continuous, strictly increasing  from $-\infty$ to $\infty$ and with a derivative greater than $1$.

At the critical point, $f''(iz (r,\theta))$ is a positive and real number
$$f''(iz (r,\theta))= 2 \frac{\sinh^2 r}{u(r,\theta)^2-1} \left[ \frac{u(r,\theta) \arcosh u(r,\theta)} {\sqrt {u(r,\theta)^2-1}}-1\right]$$

with $u(r,\theta)= \cosh r \cos z (r,\theta)$ since $u>-1$.

By the same method as in the previous proposition, we therefore obtain:
\begin{proposition}
Let $r>0, \theta \in \R$. When $t \to 0$,

$$p_t(r,\theta) \sim \frac{1}{\sinh r} \frac{\mathrm{arcosh }  \, u(r,\theta)}{\sqrt{   \frac{u(r,\theta) \mathrm{arcosh }\, u(r,\theta)}{\sqrt{u^2(r,z)-1} }-1}} \frac{e^{-\frac{(z (r,\theta)-\theta)^2 \tanh^2 r}{4t \sin^2 z (r,\theta)}}}{(4\pi t)^{\frac{3}{2}}}
$$
with $u(r,\theta)= \cosh r \cos z (r,\theta)$.
\end{proposition}

We refer to Bonnefont \cite{Bonnefont2} for further results in the case $n=1$ and to Wang \cite{Wang2} for the study of the small time asymptotics in the case $n \ge 2$.

\subsection{Horizontal heat kernel}

Because of the $2\pi$ periodic condition on the operator $\frac{\partial}{\partial \theta}$ which was not used in the previous computations,  let us note that $p_t(r,\theta)$ is not the heat kernel of the horizontal Brownian motion 
 \[
 Y(t)=\frac{e^{i\theta(t)} }{\sqrt{1-|w(t)|^2}} \left( w(t),1\right), \quad t \ge 0.
 \]
on $\mathbf{AdS}^{2n+1}(\mathbb{C})$, it is instead  the heat kernel of the horizontal Brownian motion on the universal cover of $\mathbf{AdS}^{2n+1}(\mathbb{C})$. One can  obtain the formula for the horizontal heat kernel on $\mathbf{AdS}^{2n+1}(\mathbb{C})$ by wrapping the formula for $p_t(r,\theta)$.

Indeed, we have that
\begin{align*}
 \mathbb{E}( f(Y(t)))& =\mathbb{E} \left[ f \left( \frac{e^{i\theta(t)} }{\sqrt{1-|w(t)|^2}} \left( w(t),1\right) \right) \right] \\
  & = \int_{\mathbb{C}H^n} \int_{-\infty}^{+\infty} f \left( \frac{e^{i\theta } }{\sqrt{1-|w|^2}} \left( w,1\right) \right) p_{t/2} (\arctan |w|, \theta ) d \theta d \mu (w)  \\
  & = \sum_{k \in \mathbb{Z}} \int_{\mathbb{C}H^n} \int_{2k\pi}^{(2k+2)\pi} f \left( \frac{e^{i\theta } }{\sqrt{1-|w|^2}} \left( w,1\right) \right) p_{t/2} (\arctan |w|, \theta )d \theta d \mu (w) \\
  & =  \int_{\mathbb{C}H^n} \int_{0}^{2\pi} f \left( \frac{e^{i\theta } }{\sqrt{1-|w|^2}} \left( w,1\right) \right)\left( \sum_{k \in \mathbb{Z}} p_{t/2} (\arctan |w|, \theta +2k\pi ) \right)d \theta d \mu (w)
 \end{align*}
 where $\mu$ is the Riemannian volume measure on $\mathbb{C}H^n$. We therefore deduce
 
 \begin{prop}\label{prop135}
For $t>0$, $r\in[0,+\infty)$, $ \theta\in [0,2\pi]$, the horizontal  heat kernel is given by
\begin{equation}\label{eq83}
p_t(r, \theta)=\frac{1}{\sqrt{4\pi t}} \sum_{k \in \mathbb Z} \int_{-\infty}^{+\infty} e^{\frac{(y-i \theta-2ik \pi)^2}{4t} }q_{t,2n+1}(\cosh r\cosh y)dy.
\end{equation}
\end{prop}

\subsection{Characteristic function of the stochastic area and limit theorem}

In this section we study the characteristic function of the stochastic area $\theta(t)$.
Let
\[
\mathcal{L}^{\alpha,\beta}=\frac{1}{2} \frac{\partial^2}{\partial r^2}+\left(\left(\alpha+\frac{1}{2}\right)\coth r+\left(\beta+\frac{1}{2}\right) \tanh r\right)\frac{\partial}{\partial r}, \quad \alpha,\beta >-1
\]
be the hyperbolic Jacobi generator. We will denote by $q_t^{\alpha,\beta}(r_0,r)$ the transition density with respect to the Lebesgue measure of the diffusion with generator $\mathcal{L}^{\alpha,\beta}$.

Let $\lambda \ge 0$, $r \in [0,+\infty)$ and 
\[
I(\lambda,r)=\mathbb{E}\left(e^{i \lambda \theta(t)}\mid r(t)=r\right).
\]
From Theorem \ref{diff-H},  we have
\begin{align*}
I(\lambda,r)& =\mathbb{E}\left(e^{i \lambda B\left({\int_0^t \tanh^2 r(s)ds}\right)}\mid r(t)=r\right) \\
 &=\mathbb{E}\left(e^{- \frac{\lambda^2}{2} \int_0^t \tanh^2 r(s)ds}\mid r(t)=r\right) 
\end{align*}
and $r$ is a diffusion with the  hyperbolic Jacobi generator
\[
\mathcal{L}^{n-1,0}=\frac{1}{2} \left( \frac{\partial^2}{\partial r^2}+((2n-1)\coth r+\tanh r)\frac{\partial}{\partial r}\right)
\]
started at $0$. 

\begin{theorem}\label{FTCHn}
For $\lambda \ge 0$, $r \in [0,+\infty)$, and $t >0$
\[
\mathbb{E}\left(e^{i \lambda \theta(t)}\mid r(t)=r\right)=\mathbb{E}\left(e^{- \frac{\lambda^2}{2} \int_0^t \tanh^2 r(s)ds}\mid r(t)=r\right) =\frac{e^{n \lambda t}}{(\cosh r)^\lambda} \frac{q_t^{n-1,\lambda}(0,r)}{q_t^{n-1,0}(0,r)}
\]
\end{theorem}

\begin{proof}
The proof follows the same lines as in Theorem \ref{FThj}. We have
\begin{equation}\label{eq-sde-r}
dr(t)= \frac{1}{2} \left( (2n-1)\coth r(t)+\tanh r(t) \right)dt+d\gamma(t),
\end{equation}
where $\gamma$ is a standard Brownian motion. For later use, we observe here that it implies that we almost surely have
\begin{align}\label{transient}
r(t) \ge \left( n -\frac{1}{2} \right) t +\gamma(t),
\end{align}
and thus $r(t)\to +\infty$ almost surely when $t \to \infty$.
Consider now the local martingale
\begin{align*}
D_t& =\exp \left( \lambda \int_0^t \tanh r(s) d\gamma(s) -\frac{\lambda^2}{2}  \int_0^t \tanh^2 r(s) ds \right)  \\
 &=\exp \left( \lambda \int_0^t \tanh r(s) dr(s)-\frac{\lambda}{2}(2n-1)t -\frac{\lambda +\lambda^2}{2}  \int_0^t \tanh^2 r(s) ds \right). 
\end{align*}
From It\^o's formula, we have
\begin{align*}
\ln \cosh r(t) & =\int_0^t \tanh r(s) dr(s)+\frac{1}{2} \int_0^t \frac{ds}{\cosh^2 r(s)} \\
 &=\int_0^t \tanh r(s) dr(s)-\frac{1}{2} \int_0^t \tanh^2 r(s) ds+\frac{1}{2} t.
\end{align*}
As a consequence, we deduce that
\[
D_t =e^{-n\lambda t} (\cosh r(t))^\lambda e^{- \frac{\lambda^2}{2} \int_0^t \tanh^2 r(s)ds}.
\]
We claim that $D_t$ is martingale. To see this we just need to show that $D_t$ is uniformly integrable, which will be implied by the fact that for all $1<p<+\infty$,
\[
\mathbb E\left(\sup_{0\le s\le t} (D_s)^p\right)<+\infty.
\]
By Doob's maximal inequality, we just need to show $\mathbb E(D_t^p)<+\infty$ for all $1<p<+\infty$. This will follow from
$D_t\le (\cosh r(t))^\lambda $.  Indeed, from the comparison principle for stochastic differential equations, one has

\[
r(t) \le h(t),
\]
where $h$ is the radial part of the Brownian motion in the $(2n+1)$-real hyperbolic space:
\[
h(t)=n \int_0^t \coth h(s) ds +\gamma(t). 
\]
The distribution of the random variable $\cosh h(t)$ is well understood  and has moments of any order, as such so does $D_t$.

  Let us denote by $\mathcal{F}$ the natural filtration of $r$ and consider the probability measure $\mathbb{P}^\lambda$ defined by
\[
\mathbb{P}_{/ \mathcal{F}_t} ^\lambda=D_t \mathbb{P}_{/ \mathcal{F}_t}=e^{-n\lambda t} (\cosh r(t))^\lambda e^{- \frac{\lambda^2}{2} \int_0^t \tanh^2 r(s)ds} \mathbb{P}_{/ \mathcal{F}_t}.
\]
We have for every bounded and Borel function $f$ on $[0,+\infty]$,
\begin{align}\label{eq-adsc-chara-1}
\mathbb{E}\left(f(r(t))e^{- \frac{\lambda^2}{2} \int_0^t \tanh^2 r(s)ds}\right)=e^{n\lambda t} \mathbb{E}^\lambda \left( \frac{f(r(t))}{(\cosh r(t))^\lambda} \right).
\end{align}
From Girsanov theorem, the process
\[
\beta(t)=\gamma(t)-\lambda \int_0^t \tanh r(s) ds
\]
is a Brownian motion under the probability $\mathbb{P}^\lambda$. The proof is complete by plugging 
\begin{equation}\label{eq-rt-H}
dr(t)= \frac{1}{2} \left( (2n-1)\coth r(t)+(2\lambda +1)\tanh r(t) \right)dt+d\beta(t),
\end{equation}
in \eqref{eq-adsc-chara-1}.
\end{proof}

As an immediate consequence of Theorem \ref{FTCHn}, we deduce an expression for  the characteristic function of the stochastic area process.
\begin{corollary}\label{FTh}
For $\lambda \in \mathbb{R}$ and $t \ge 0$,
\[
\mathbb{E}\left(e^{i \lambda \theta(t)}\right)=e^{n | \lambda | t}\int_0^{+\infty}  \frac{q_t^{n-1,| \lambda |}(0,r)}{(\cosh r)^{| \lambda|}} dr. 
\]
\end{corollary}

The density of $\theta(t)$ was determined in \cite{Dem-Windings} and its derivation relies on an intertwining relation between the generator of the diffusion $(r(t))_{t \geq 0}$ under $\mathbb{P}^{\lambda}$ and the radial part of the Maass Laplacian of the hyperbolic complex ball. More precisely, if 
\begin{equation*}
\mathcal{L}^{(n-1, |\lambda|)} = \frac{1}{2} \left(\partial_r^2 + ((2n-1)\coth r + (2|\lambda| + 1)\tanh(r)) \partial_r\right), \quad r \geq 0,
\end{equation*}
then the following intertwining relation holds:
 
\begin{equation}\label{Inter}
 \frac{1}{\cosh^{|\lambda|}(r)} L^{(n-1, |\lambda|)}(\cdot) (r) = \left[\mathcal{L}^{(n-1, |\lambda|)}+ \frac{\rho^2}{2}\right] \left(\frac{1}{\cosh^{|\lambda|}} \cdot\right)(r), 
\end{equation}
where
\begin{equation*}
2L^{(n-1, |\lambda|)} = \partial_r^2 + ((2n-1)\coth r + \tanh(r)) \partial_r + \frac{\lambda^2}{\cosh^2(r)} + n^2
\end{equation*}
is the radial part of the (shifted) generalized Maass Laplacian:
\begin{equation*}
D^{(n,|\lambda|/2)} + n^2 := 4(1-|w|^2)\left\{ \sum_{i,j=1}^n(\delta_{ij}- w_iw_j) \partial_{w_i\overline{w}_j}^2 +\frac{|\lambda|}{2}\sum_{i=1}^n(w_i\partial_{w_i} - \overline{w_i}\overline{\partial_{w_i}}) + \frac{\lambda^2}{4}\right\} + n^2. 
 \end{equation*}
The heat kernel of $L^{(n-1, |\lambda|)}$ was derived in \cite{Aya-Int} and leads after some computations to the following result.

\begin{theorem}
Let $s_{t,2n+1}$ be the heat kernel (with respect to the volume measure) of the $2n+1$-dimension real hyperbolic space $H^{2n+1}$:
\begin{align*}
s_{t,2n+1}(\cosh(x)) & = 
\frac{e^{-n^2t/2}}{(2\pi)^n\sqrt{2\pi t}}\left(-\frac{1}{\sinh(x)}\frac{d}{dx}\right)^ne^{-x^2/(2t)}, 
\quad x > 0. 
\end{align*}
Then the density $f_t$ of $\theta(t)$ reads
\begin{multline*}
f_{t}(v) = \frac{2e^{-v^2/(2t)}}{\sqrt{2\pi t}} \int_0^{\infty} dr \cosh (r) \sinh^{2n-1}(r) \\  \int_{0}^{\infty} du \cos\left(\frac{uv}{t}\right)e^{u^2/(2t)} s_{t,2n+1}(\cosh(u)\cosh(r)), \,\quad v \in \mathbb{R}.
\end{multline*}
\end{theorem}


\
In this expression, we notice the presence of the Gaussian factor. The 
following theorem shows that this factor indeed encodes the long time behaviour of $\theta(t)$.
\begin{theorem}\label{LimitCHn}
When $t \to +\infty$, the following convergence in distribution takes place
\[
\frac{\theta(t)}{\sqrt{t}} \to \mathcal{N}(0,1)
\]
where $\mathcal{N}(0,1)$ is a normal distribution with mean 0 and variance 1.
\end{theorem}
\begin{proof}
From \eqref{transient}, one has $r(t)\to +\infty$ almost surely as $t\to+\infty$. It then follows that
\[
\coth r(t)\to 1,\quad \tanh r(t)\to 1\quad \mbox{a.s.}
\]
and hence
\[
\lim_{t\to+\infty}\frac{1}{t}\int_0^t \tanh^2 r(s)ds=1\quad \mbox{a.s.}
\]
Therefore, as a consequence of \eqref{eq-r-theta-H}, we have
\[
\lim_{t\to+\infty}\frac{\theta(t)}{\sqrt{t}}=\lim_{t\to+\infty}B_{\frac{1}{t}\int_0^t \tanh^2 r(s)ds}=B_1 \quad \mbox{a.s.}
\]
\end{proof}

\begin{remark}
We would like to point out the following difference between the proofs of Theorem \ref{FThj} and Theorem \ref{FTCHn}, where the main ingredient is Girsanov Theorem. In the former, the choice of the minus sign in the local martingale 
\[
-\lambda\int_0^t \tan(r(s)) ds, \quad \lambda > 0,
\]
is made to ensure the boundedness of the term $[\cos(r(t))]^{\lambda}$ since the cosine is singular near $\pi/2$. In the latter, the hyperbolic cosine has no singularity and both choices 
\[
\pm \lambda\int_0^t \tanh(r(s)) ds, \quad \lambda > 0,
\]
lead to the same result. Indeed, if we substitute $\lambda \mapsto -\lambda$, then the resulting exponential local martingale is bounded by $e^{n\lambda t}$ and we obtain
\[
\mathbb{E}\left(e^{i \lambda \theta(t)}\mid r(t)=r\right)=
e^{-n \lambda t}(\cosh r)^\lambda \frac{q_t^{n-1,-\lambda}(0,r)}{q_t^{n-1,0}(0,r)}.
\]
Besides, one easily checks the intertwining relation: 
\[
\mathcal{L}^{\alpha,\beta}\left(r \mapsto \frac{f(r)}{[\cosh(r)]^{2\beta}}\right) = \frac{1}{(\cosh r)^{2\beta}}\mathcal{L}^{\alpha,-\beta}(f)(r) 
-2(\alpha+1)\beta f(r),
\]
which shows that 
\[
q_t^{n-1,-\lambda}(0,r) = \frac{e^{2n\lambda t}}{[\cosh(r)]^{2\lambda}}q_t^{n-1,\lambda}(0,r).
\]

\end{remark}
\chapter{Horizontal Brownian motions of the quaternionic Hopf fibrations}\label{Chapter quaternionic hopf}

\section{Stochastic area process on the quaternionic projective spaces $\mathbb{H}P^n$}\label{sec--s-geometry}

\subsection{Quaternionic Hopf fibration}\label{Quaternionic Hopf}

As before the quaternionic field is defined by
\[
\mathbb{H}=\{q=t+xI+yJ+zK, (t,x,y,z)\in\R^4\},
\]
where  $I,J,K \in \mathfrak{su}(2)$ are given by
\[
I=\left(
\begin{array}{ll}
i & 0 \\
0&-i 
\end{array}
\right), \quad 
J= \left(
\begin{array}{ll}
0 & 1 \\
-1 &0 
\end{array}
\right), \quad 
K= \left(
\begin{array}{ll}
0 & i \\
i &0 
\end{array}
\right).
\]
We recall that the quaternionic conjugate of $q=t+xI+yJ+zK$ is defined by
\[
\overline{q}=t-xI-yJ-zK
\]
and that we have the quaternionic relations $I^2=J^2=K^2=IJK=-1$. 
%
On $\mathbb H^n$, we will consider the $\mathbb H$-valued vector fields
\begin{equation}\label{eq-partial-def}
\frac{\partial}{\partial q_i}=\frac{1}{2} \left( \frac{\partial}{\partial t_i}- \frac{\partial}{\partial x_i} I -\frac{\partial}{\partial y_i} J-\frac{\partial}{\partial z_i}K\right)
\end{equation}
and 
\[
\frac{\partial}{\partial \overline{q}_i}=\frac{1}{2} \left( \frac{\partial}{\partial t_i}+ \frac{\partial}{\partial x_i} I +\frac{\partial}{\partial y_i} J+\frac{\partial}{\partial z_i}K\right)
\]
and also the $\mathbb H$-valued one-forms
\[
dq_i=dt_i +dx_i I +dy_i J+dz_i K
\]
and 
\[
d\overline{q}_i=dt_i -dx_i I -dy_i J-dz_i K.
\]
The quaternionic projective space $\mathbb{H}P^n$ can be defined as the set of quaternionic lines in $\mathbb{H}^{n+1}$. To parametrize points in $\mathbb{H}P^n$, we will use the local affine coordinates given by $w_j= q_jq_{n+1}^{-1}$, $1 \le j \le n$, $q \in \mathbb{H}^{n+1}$, $q_{n+1}\neq 0$. Though this only provides a local set of coordinates (we are missing an hypersurface at $\infty$) we will do all our computations in those coordinates. In affine coordinates, the Riemannian structure of $\mathbb{H}P^n$ is easily worked out from the standard Riemannian structure of the Euclidean sphere $\mathbb{S}^{4n+3}$. Indeed, if we consider the unit sphere
\[
\bS^{4n+3}=\lbrace q=(q_1,\dots,q_{n+1})\in \mathbb{H}^{n+1}, | q | =1\rbrace,
\]
then, the map 
\begin{align*}
\begin{cases}
\bS^{4n+3} -\{q_{n+1}=0 \}  \to \mathbb{H}^n \\
(q_1,\dots,q_{n+1}) \to ( q_1q_{n+1}^{-1},\dots,  q_n q_{n+1}^{-1})
\end{cases}
\end{align*}
is a submersion. The unique Riemannian metric $h$ on $\mathbb{H}P^n$ that makes this map a Riemannian submersion is the standard Riemannian metric on $\mathbb{H}P^n$. This metric can be characterized as follows.

\begin{proposition}\label{metric HPn}
If $g_1,g_2$ are smooth functions on  $\mathbb{H}P^n$, then in affine coordinates
\begin{equation}\label{eq-co-metric}
   h(dg_1,dg_2)=2(1+|w|^2)\left( \sum_{i=1}^n \left( \frac{\partial g_1}{\partial \overline{w}_i}\frac{\partial g_2}{\partial w_i}+\frac{\partial g_2}{\partial \overline{w}_i}\frac{\partial g_1}{\partial w_i} \right) +\overline{\mathcal{R}} g_1 \mathcal{R} g_2 +\overline{\mathcal{R}} g_2 \mathcal{R} g_1 \right),  
\end{equation}
 where
 \[
 \mathcal{R} g = \sum_{i=1}^n \frac{\partial g}{\partial w_i} w_i.
 \]
 \end{proposition}

\begin{proof}
We use the change of variables formula $w_j= q_jq_{n+1}^{-1}$. Let  $f$ be a (real-valued) smooth function on $\mathbb H^{n+1}$ such that there exists a function $g$ on $\mathbb{H}P^n$ satisfying
\[
f(q_1,\dots,q_{n+1})=g \left( q_1q_{n+1}^{-1},\dots, q_nq_{n+1}^{-1}\right), \quad q \in \mathbb H^{n+1} .
\]
By definition of the metric $h$ one has for $q \in \bS^{4n+3}$

\[
h(dg,dg)\left( w \right)=4 \sum_{i=1}^{n+1}\frac{\partial f}{\partial \overline{q}_i} (q)\frac{\partial f}{\partial q_i} (q).
\]

We first compute $\frac{\partial f}{\partial q_i} $ for $1 \le i \le n$. Using the fact that $w_i= q_iq_{n+1}^{-1}$ one gets that for $1 \le i \le n$
\[
\frac{\partial f}{\partial q_i} (q)= q_{n+1}^{-1} \frac{\partial g}{\partial w_i} (w).
\]

To compute $\frac{\partial f}{\partial q_{n+1}} $ we need to use the symmetries of the function $f$. Using the fact that
\begin{align*}
\begin{cases}
\frac{d}{ds}f(e^{s}q_1,\dots, e^{s}q_{n+1} )\mid_{s=0} =0 \\
\frac{d}{ds}f( q_1e^{sI},\dots, q_{n+1}e^{sI})\mid_{s=0} =0 \\
\frac{d}{ds}f( q_1e^{sJ},\dots, q_{n+1}e^{sJ})\mid_{s=0} =0 \\
\frac{d}{ds}f( q_1e^{sK},\dots, q_{n+1}e^{sK})\mid_{s=0} =0
\end{cases}
\end{align*}
we get
\begin{align*}
\begin{cases}
\sum_{i=1}^{n+1} t_i \frac{\partial f}{\partial t_i} + x_i \frac{\partial f}{\partial x_i}+ y_i \frac{\partial f}{\partial y_i}+ z_i \frac{\partial f}{\partial z_i} =0 \\
\sum_{i=1}^{n+1} -x_i \frac{\partial f}{\partial t_i} + t_i \frac{\partial f}{\partial x_i}+ z_i \frac{\partial f}{\partial y_i}- y_i \frac{\partial f}{\partial z_i} =0 \\
\sum_{i=1}^{n+1} -y_i \frac{\partial f}{\partial t_i} - z_i \frac{\partial f}{\partial x_i}+ t_i \frac{\partial f}{\partial y_i}+ x_i \frac{\partial f}{\partial z_i} =0 \\
\sum_{i=1}^{n+1} -z_i \frac{\partial f}{\partial t_i} + y_i \frac{\partial f}{\partial x_i}- x_i \frac{\partial f}{\partial y_i}+ t_i \frac{\partial f}{\partial z_i} =0
\end{cases}.
\end{align*}
This can be more concisely written in quaternionic form as 
\[
\sum_{i=1}^{n+1} \frac{\partial f}{\partial q_i} q_i =0.
\]
In particular, one obtains 
\[
 \frac{\partial f}{\partial q_{n+1}}=-\sum_{i=1}^{n} \frac{\partial f}{\partial q_i} q_i q_{n+1}^{-1}.
 \]
 This yields
 \[
  \frac{\partial f}{\partial q_{n+1}}=-q_{n+1}^{-1} \sum_{i=1}^{n} \frac{\partial g}{\partial w_i} w_i.
 \]
 One conclude
  \begin{align*}
  \sum_{i=1}^{n+1} \frac{\partial f}{\partial \overline{q}_i}\frac{\partial f}{\partial q_i}=\frac{1}{|q_{n+1}|^2}  \left( \sum_{i=1}^n \frac{\partial g}{\partial \overline{w}_i}\frac{\partial g}{\partial w_i} + \overline{\mathcal{R}} g\mathcal{R} g \right),
 \end{align*}
where
 \[
 \mathcal{R} g = \sum_{i=1}^n \frac{\partial g}{\partial w_i} w_i.
 \]
 Since on $\bS^{4n+3}$ one has $\frac{1}{|q_{n+1}|^2} =1+|w|^2$, one concludes
 \begin{equation}\label{eq-h-quad}
 h(dg,dg)=4(1+|w|^2)\left( \sum_{i=1}^n \frac{\partial g}{\partial \overline{w}_i}\frac{\partial g}{\partial w_i} +\overline{\mathcal{R}} g\mathcal{R} g  \right).
 \end{equation}
 Using the polarization formula
 \[
 h(dg_1,dg_2)=\frac{1}{4} \left( h \left( dg_1+dg_2, dg_1+dg_2\right)-h \left( dg_1-dg_2, dg_1-dg_2\right)  \right)
 \]
 concludes the proof.
\end{proof}

It will be convenient to use the above formula for the metric in its real form. For this, we use the real affine coordinates on $\mathbb{H}P^n$ defined by
\[
w_i=\tau_i+\alpha_i I +\beta_i J +\gamma_i K
\]
\begin{corollary}
If $g_1,g_2$ are smooth functions on  $\mathbb{H}P^n$, then in real affine coordinates
\begin{align}\label{eq-co-metric2-S}
   h(dg_1,dg_2) =(1+|w|^2) \left( \sum_{i=1}^n \left( \frac{\partial g_1}{\partial \tau_i}\frac{\partial g_2}{\partial \tau_i}+\frac{\partial g_1}{\partial \alpha_i}\frac{\partial g_2}{\partial \alpha_i} +\frac{\partial g_1}{\partial \beta_i}\frac{\partial g_2}{\partial \beta_i}+\frac{\partial g_1}{\partial \gamma_i}\frac{\partial g_2}{\partial \gamma_i} \right)  +\sum_{i=1}^4 \mathcal{R}_i  g_1 \mathcal{R}_i g_2 \right),  
\end{align}
 where
\[
 \mathcal{R}_1 g = \sum_{i=1}^n \tau_i \frac{\partial g}{\partial \tau_i} +\alpha_i \frac{\partial g}{\partial \alpha_i}+\beta_i \frac{\partial g}{\partial \beta_i}+\gamma_i \frac{\partial g}{\partial \gamma_i}
 \]
 \[
 \mathcal{R}_2 g = \sum_{i=1}^n \alpha_i \frac{\partial g}{\partial \tau_i} -\tau_i \frac{\partial g}{\partial \alpha_i}-\gamma_i \frac{\partial g}{\partial \beta_i}+\beta_i \frac{\partial g}{\partial \gamma_i}
 \]
 \[
 \mathcal{R}_3 g = \sum_{i=1}^n \beta_i \frac{\partial g}{\partial \tau_i} +\gamma_i \frac{\partial g}{\partial \alpha_i}-\tau_i \frac{\partial g}{\partial \beta_i}-\alpha_i \frac{\partial g}{\partial \gamma_i}
 \]
 \[
 \mathcal{R}_4 g = \sum_{i=1}^n \gamma_i \frac{\partial g}{\partial \tau_i} -\beta_i \frac{\partial g}{\partial \alpha_i}+\alpha_i \frac{\partial g}{\partial \beta_i}-\tau_i \frac{\partial g}{\partial \gamma_i}
 \]
\end{corollary}

\begin{proof}
This easily follows from \eqref{eq-co-metric} by computing
\[
4\frac{\partial g}{\partial \overline{w}_i}\frac{\partial g}{\partial w_i} =\sum_{i=1}^n  \left(\frac{\partial g}{\partial \tau_i} \right)^2+ \left(\frac{\partial g}{\partial \alpha_i} \right)^2+ \left(\frac{\partial g}{\partial \beta_i} \right)^2+ \left(\frac{\partial g}{\partial \gamma_i} \right)^2
\]
and 
\[
\mathcal{R} g= \frac{1}{2} \left( \mathcal{R}_1 g  + (\mathcal{R}_2 g ) I+ (\mathcal{R}_3 g ) J + (\mathcal{R}_4 g ) K \right).
\]
\end{proof}

\begin{corollary}
 In real affine coordinates, the matrix of the Riemannian cometric   is given by the $4n \times 4n$ matrix
 \[
 (1+|w|^2) \left(\mathrm{I}_{4n}+\mathrm{R}_1 \mathrm{R}^*_1+\mathrm{R}_2 \mathrm{R}^*_2+\mathrm{R}_3 \mathrm{R}^*_3+\mathrm{R}_4 \mathrm{R}^*_4 \right),
 \]
 where $\mathrm{R}_i$ is the real valued column vector with the $4n$ components
 \[
 \mathrm{R}_i =\left( \mathcal{R}_i \tau_j , \mathcal{R}_i \alpha_j, \mathcal{R}_i \beta_j,\mathcal{R}_i \gamma_j\right)_{1 \le j \le n}.
 \]
\end{corollary}

\begin{proof}
This readily follows from \eqref{eq-co-metric2-S}.
\end{proof}

\begin{corollary}\label{volume HPn}
 In affine coordinates the Riemannian volume measure of $\HP^n$ is given by
\[
d\mu(w)=(1+|w|^2)^{-2(n+1)} dw.
\]
\end{corollary}

\begin{proof}
We need to compute the determinant of the Riemannian cometric matrix which, according to the previous corollary, is given by
\[
(1+|w|^2)^{4n} \mathrm{det} \left(\mathrm{I}_{4n}+\mathrm{R}_1 \mathrm{R}^*_1+\mathrm{R}_2 \mathrm{R}^*_2+\mathrm{R}_3 \mathrm{R}^*_3+\mathrm{R}_4 \mathrm{R}^*_4 \right).
\]
The 4 vectors $\mathrm{R}_i$ have length $|w|^2$ and are pairwise orthogonal, therefore the symmetric matrix 
\[
\mathrm{R}_1 \mathrm{R}^*_1+\mathrm{R}_2 \mathrm{R}^*_2+\mathrm{R}_3 \mathrm{R}^*_3+\mathrm{R}_4 \mathrm{R}^*_4
\]
has rank 4 and only the eigenvalue $|w|^2$ which has  multiplicity 4. This yields that the determinant of the cometric matrix is $(1+|w|^2)^{4n+4}$ and thus the expected result, since we recall that the density of the Riemannian volume form is given by the inverse of the square root of the determinant of the cometric. 
\end{proof}

We now have all the ingredients to compute the Laplacian on $\HP^n$.

\begin{theorem}\label{Laplace quaternionic projective}
In real affine coordinates, we have
\[
\Delta_{\HP^n}=(1+|w|^2) \left( \sum_{i=1}^n \left( \frac{\partial^2}{\partial \tau^2_i}+\frac{\partial^2}{\partial \alpha^2_i}+\frac{\partial^2}{\partial \beta^2_i}+\frac{\partial^2 }{\partial \gamma^2_i} \right)  +\sum_{i=1}^4 \mathcal{R}^2_i  -2\mathcal{R}_1  \right)
\]
\end{theorem}

\begin{proof}
This formula can be proved using an integration by parts. Indeed, $\Delta_{\HP^n}$ is the operator satisfying the integration by parts formula
\[
\int_{\HP^n} h(df,dg) (1+|w|^2)^{-2(n+1)} dw=-\int_{\HP^n} g \Delta_{\HP^n} f \, (1+|w|^2)^{-2(n+1)} dw.
\]
Using formula \eqref{eq-co-metric2-S} yields the formula for $\Delta_{\HP^n} $ after some elementary but tedious computations.
\end{proof}

The description of $\mathbb{H}P^n$ using the affine coordinates is incomplete since it does not chart a manifold at $\infty$. More intrinsically, we have a globally defined Riemannian submersion $\bS^{4n+3} \to \mathbb{H}P^n$, that can be constructed as follows. There is an isometric group action of the group $\mathbf{SU}(2)$ on $\bS^{4n+3}$ which is  defined by $$g \cdot(q_1,\dots, q_{n+1}) := (  q_1 g,\dots,  q_{n+1} g), $$
where $\mathbf{SU}(2)$ is identified with the set of unit quaternions.
%
%
The quotient space $\bS^{4n+3} / \mathbf{SU}(2)$ can be identified with $\mathbb{H}P^n$ and the projection map $$\pi :  \bS^{4n+3} \to \mathbb{H}P^n$$ is a Riemannian submersion with totally geodesic fibers isometric to $\mathbf{SU}(2)$. The fibration
\[
\mathbf{SU}(2) \to \bS^{4n+3} \to \mathbb{H}P^n
\]
 is called the quaternionic Hopf fibration. This fibration is an example of a B\'erard-Bergery fibration (see Example \ref{BB fibration}). 
 
 The submersion $\pi$ allows to construct the Brownian motion on $\mathbb{H}P^n$ from the Brownian motion on $\bS^{2n+1}$. Indeed,  let $(q(t))_{t \ge 0}$ be a Brownian motion on  $\bS^{4n+3}$ started at the north pole \footnote{We will call north pole the point with quaternionic coordinates $q_1=0,\dots, q_{n+1}=1$. }. Since $\mathbb{P}( \exists t \ge 0, q_{n+1}(t)=0 )=0$, one can use the local description of the submersion $\pi$ to infer that
\begin{align}\label{BMsphere_2}
w(t)= \left( q_1(t)q_{n+1}(t)^{-1}  , \dots,  q_n(t)q_{n+1}(t)^{-1} \right), \quad t \ge 0,
\end{align}
is a Brownian motion on $\mathbb{H}P^n$, see Remark \ref{harmonic submersion}. From Theorem \ref{Laplace quaternionic projective}, the generator of $(w(t))_{t \ge 0}$ is given by
\[
\frac{1}{2}(1+|w|^2) \left( \sum_{i=1}^n \left( \frac{\partial^2}{\partial \tau^2_i}+\frac{\partial^2}{\partial \alpha^2_i}+\frac{\partial^2}{\partial \beta^2_i}+\frac{\partial^2 }{\partial \gamma^2_i} \right)  +\sum_{i=1}^4 \mathcal{R}^2_i  -2\mathcal{R}_1  \right).
\]

\begin{proposition}
The process $r(t)=\arctan |w(t)|$ is a Jacobi diffusion with generator
\[
\frac{1}{2} \left( \frac{\partial^2}{\partial r^2}+((4n-1)\cot r-3\tan r)\frac{\partial}{\partial r}\right).
\]
In particular, the density of $r(t)$, $t>0$, is given by the formula
\[
\mathbb{P} (r(t) \in dr)=2 \left( \sum_{k=0}^{\infty}(2k+2n+1){k+2n\choose 2n-1} e^{-2k(k+2n+1)t} P_k^{2n-1,1}(\cos 2r) \right) (\sin r)^{4n-1}(\cos r)^3 dr
\]
where 
\[
P_k^{2n-1,1}(x)=\frac{(-1)^k}{2^kk!(1-x)^{2n-1}(1+x)}\frac{d^k}{dx^k}\left((1-x)^{2n-1+k}(1+x)^{1+k} \right)
\] 
is a Jacobi polynomial.
\end{proposition}

\begin{proof}
Let $r=\arctan |w|$. A direct  computation shows that the operator 
\[
(1+|w|^2) \left( \sum_{i=1}^n \left( \frac{\partial^2}{\partial \tau^2_i}+\frac{\partial^2}{\partial \alpha^2_i}+\frac{\partial^2}{\partial \beta^2_i}+\frac{\partial^2 }{\partial \gamma^2_i} \right)  +\sum_{i=1}^4 \mathcal{R}^2_i  -2\mathcal{R}_1  \right)
\]
acts on functions depending only on $r$ as
\[
 \frac{\partial^2}{\partial r^2}+((4n-1)\cot r-3\tan r)\frac{\partial}{\partial r}
 \]
It\^o's formula shows  that  $r(t)=\arctan |w(t)|$ is a Jacobi diffusion with generator
\[
\frac{1}{2} \left(  \frac{\partial^2}{\partial r^2}+((4n-1)\cot r-3\tan r)\frac{\partial}{\partial r} \right).
\]
The computation for the density of $r(t)$ follows from the known formula for the heat kernel of Jacobi diffusions, see Appendix 2.
\end{proof}

\begin{remark}
If $(\beta(t))_{t \ge 0}=(\beta_1(t),\dots, \beta_{n+1}(t))_{t \ge 0}$ is a Brownian motion in $\mathbb{H}^{n+1}$, then from a classical skew-product decomposition (see Example \ref{BM on warped product}),
\begin{align*}
\frac{\beta(t)}{| \beta(t) |}=  q\left( \int_0^t \frac{ds}{ | \beta(s)|^2}\right),
\end{align*}
where $q(t)$ is a Brownian motion on the sphere $\bS^{4n+3}$. We deduce therefore from \eqref{BMsphere_2} that
\begin{align}\label{plki2}
\frac{ \sqrt{ \sum_{i=1}^n | \beta_i (t)|^2}}{|\beta_{n+1} (t)|  }=\tan r \left(  \int_0^t \frac{ds}{ | \beta(s)|^2} \right).
\end{align}
The process $\sqrt{ \sum_{i=1}^n | \beta_i (t)|^2}$ is a Bessel process with dimension $4n$ and $| \beta_{n+1} (t)|$ is a Bessel process with dimension 4. The equality \eqref{plki2} is  again a special case of a general skew-product representation of Jacobi processes due to Warren and Yor \cite{WarrenYor}.
\end{remark}

 To use the general framework described in Section \ref{horizontal BM bundle}, we now compute a connection form of the fibration.  As before, we make use of the following notations. A point $q \in \mathbb H^{n+1}$ is  represented as
 \[
 q=(q_1,\dots,q_{n+1})
 \]
 where $q_i=t_i + x_i I+y_i J+z_i K  $. The quaternion conjugate of $q_i$ is denoted
 \[
 \overline{q}_i=t_i - x_i I- y_i J- z_i K.
 \]
 The $\mathbb H$-valued one form $dq_i$ is defined as
 \[
 dq_i = dt_i +  dx_i I+dy_i J+dz_i K 
 \]
and the conjugate form $d\overline{q}_i$ is defined as
 \[
 d\overline{q}_i = dt_i -  dx_i I- dy_i J-dz_i K .
 \]

\begin{theorem}\label{connection quater fib}
Consider on $\bS^{4n+3} \subset \mathbb{H}^{n+1}$ the $\mathfrak{su}(2)$ valued one-form
\[
\theta=\frac{1}{2}\sum_{i=1}^{n+1}  \overline{q}_i  dq_i  - d\overline{q}_i \, q_i   .
\]
We have:
\begin{itemize}
\item \hspace{.1in} $g^* \theta =\mathbf{Ad}_{g^{-1}} \circ \theta$;
\item \hspace{.1in} For any $V \in \mathfrak{su}(2)$,
\[
\theta (\hat{V})=V,
\]
where $\hat V$ is the vector field on $\bS^{4n+3}$ defined by
\begin{align*}
\hat V f (q)= \lim_{t \to 0} \frac{f(  q e^{tV} )-f(q)}{t}, \quad q \in \bS^{4n+3};
\end{align*}
\item \hspace{.1in} The null space $\mathbf{Ker} \, \theta $ is the horizontal space of the submersion $\pi :  \bS^{4n+3} \to \mathbb{H}P^n$.
\end{itemize}
\end{theorem}

\begin{proof}
First, we note that if $g \in \mathbf{SU}(2)$ and $q=(q_1,\dots,q_{n+1}) \in \bS^{4n+3}$
\begin{align*}
\theta_{g \cdot q}& =\frac{1}{2}\sum_{i=1}^{n+1}  \overline{q_i g} d(q_ig)    - d(\overline{q_ig}) \, q_ig  \\
 &=\frac{1}{2} g^{-1} \left( \sum_{i=1}^{n+1}  \overline{q}_i  dq_i  - d\overline{q}_i \, q_i \right) g \\
  & =\mathbf{Ad}_{g^{-1}} \circ \theta
\end{align*}
Therefore,  $g^* \theta =\mathbf{Ad}_g \circ \theta$. Then, one can compute that
 \[
 \frac{d}{ds}f(qe^{Is})\mid_{s=0}=\sum_{i=1}^{n+1}\left( -x_i \frac{\partial f}{\partial t_i} + t_i \frac{\partial f}{\partial x_i}+ z_i \frac{\partial f}{\partial y_i}- y_i \frac{\partial f}{\partial z_i}\right),
 \]
  \[
 \frac{d}{ds}f(qe^{Js})\mid_{s=0}=\sum_{i=1}^{n+1}\left( -y_i \frac{\partial f}{\partial t_i} - z_i \frac{\partial f}{\partial x_i}+ t_i \frac{\partial f}{\partial y_i}+ x_i \frac{\partial f}{\partial z_i}\right),
 \]
 and
 \[
 \frac{d}{ds}f(qe^{Ks})\mid_{s=0}=\sum_{i=1}^{n+1}\left( -z_i \frac{\partial f}{\partial t_i} + y_i \frac{\partial f}{\partial x_i}- x_i \frac{\partial f}{\partial y_i}+ t_i \frac{\partial f}{\partial z_i}\right).
 \]
 Therefore we have
  \begin{align}\label{I1 quater}
      \hat{I}=\sum_{i=1}^{n+1}\left(-x_i\frac{\partial }{\partial t_i}+t_i\frac{\partial }{\partial x_i}+z_i\frac{\partial }{\partial y_i}-y_i\frac{\partial }{\partial z_i}\right),
  \end{align}
   \begin{align}\label{I2 quater}
 \hat{J}=\sum_{i=1}^{n+1}\left(-y_i\frac{\partial }{\partial t_i}-z_i\frac{\partial }{\partial x_i}+t_i\frac{\partial }{\partial y_i}+x_i\frac{\partial }{\partial z_i}\right).
 \end{align}
 and
  \begin{align}\label{I3 quater}
 \hat{K}=\sum_{i=1}^{n+1}\left(-z_i\frac{\partial f}{\partial t_i}+y_i\frac{\partial f}{\partial x_i}-x_i\frac{\partial f}{\partial y_i}+t_i\frac{\partial f}{\partial z_i}\right).
 \end{align}
 Thanks to the quaternionic relations $I^2=J^2=K^2=IJK=-1$, one obtains 
 \begin{align*}
 \theta (\hat{I}) &=\frac{1}{2}\sum_{i=1}^{n+1} \overline{q}_i dq_i (\hat I)  \,   -d\overline{q}_i (\hat I) \, q_i \,  \\
  &=\frac{1}{2}\sum_{i=1}^{n+1} \overline{q}_i (-x_i+t_i I+z_i J-y_i K)  \,   -  (-x_i-t_i I-z_i J+y_i K)q_i  \\
  &=\frac{1}{2}\sum_{i=1}^{n+1}   \overline{q}_i q_i I \,   + I q_i \overline{q}_i   \\
  &=I.
 \end{align*}
 Similarly $\theta (\hat{J}) =\hat{J}$ and $\theta (\hat{K}) =\hat{K}$. 
 
 Finally, let $q \in \bS^{4n+3}$. We denote by $\mathcal V_q$ the vertical space at $q$ and $\mathcal H_q$ the horizontal space at $q$ of the submersion $\pi$. By definition of the action of $\mathbf{SU}(2)$ on $\bS^{4n+3}$ we have
 \[
 \mathcal V_q= \left\{ (q_1u,\dots,q_{n+1}u), u \in \mathfrak{su}(2) \right\}.
 \]
 $\mathcal H_q$ is then the orthogonal complement of $\mathcal V_q$ inside of the tangent space to  $\bS^{4n+3}$  for the usual inner product. The orthogonal complement of $\mathcal{V}_q$ consists of the set of $v$'s such that for every $u \in \mathfrak{su}(2)$,
 \[
 \sum_{i=1}^{n+1}\overline{v}_i q_iu  - u \overline{q}_i v_i  =0.
 \]
 This condition is equivalent to
 \[
  \left( \sum_{i=1}^{n+1}\overline{v}_i q_i  \right)u =u\left( \sum_{i=1}^{n+1}  \overline{q}_i v_i \right)
 \]
 and therefore to the fact that $\sum_{i=1}^{n+1}  \overline{v}_i q_i$ is real. This is also equivalent to $v \in \mathbf{Ker} \, \theta$.
 \end{proof}

From the previous result, 
\[
\theta=\frac{1}{2}\sum_{i=1}^{n+1}  \overline{q}_i  dq_i  - d\overline{q}_i \, q_i   .
\]
is therefore the connection form of the quaternionic Hopf fibration.

To study the horizontal Brownian motion on $\bS^{4n+3}$ we will use the results of  Section \ref{skew general bundle}. For doing so, we consider the following local trivialization of the bundle $\mathbf{SU}(2) \to \bS^{4n+3} \to \mathbb{H}P^n$
\begin{align}\label{trivialization quaternionic}
\left( w_1,\dots, w_n, g \right) \to \frac{1}{\sqrt{1+|w|^2}}\left(  w_1 g,\dots,  w_n g, g \right) \in \bS^{4n+3} \subset \mathbb H^{n+1}
\end{align}
where $g \in \mathbf{SU}(2)$ and $w_i \in \mathbb{H}$. 

Consider the $\mathfrak{su}(2)$ valued one-form $\boldsymbol\alpha$ on $\mathbb{H}P^n$ which is given  in local affine coordinates by
\[
\boldsymbol{\alpha}=\frac{1}{2}\sum_{j=1}^n  \frac{\overline{w}_jdw_j-  d\overline{w}_j \, w_j}{1+|w|^2}
\]
where $|w|^2=\sum_{j=1}^n |w_j|^2$. Using the fact that $q_i=\frac{w_i g}{\sqrt{1+|w|^2}}$, a straightforward computation shows that $\boldsymbol\alpha$ is the  form (in the sense of \eqref{solder form}) associated with $\theta$. We note that we can write
\[
\boldsymbol\alpha=\frac{1}{1+|w|^2}\left(\boldsymbol\alpha_1 I +\boldsymbol\alpha_2 J +\boldsymbol\alpha_3 K \right)
\]
where the one-valued one-forms $\boldsymbol\alpha_1,\boldsymbol\alpha_2,\boldsymbol\alpha_3$ are given by
\[
\boldsymbol\alpha_1=\sum_{j=1}^n \left( -\alpha_j d\tau_j+\tau_j d\alpha_j+\gamma_jd\beta_j-\beta_jd\gamma_j \right) 
\]
\[
\boldsymbol\alpha_2=\sum_{j=1}^n \left( -\beta_j d\tau_j-\gamma_j d\alpha_j+\tau_jd\beta_j+\alpha_jd\gamma_j \right)
\]
and
\[
\boldsymbol\alpha_3=\sum_{j=1}^n \left( -\gamma_j d\tau_j+\beta_j d\alpha_j-\alpha_jd\beta_j+\tau_jd\gamma_j \right) 
\]
where as before  we use the coordinates $w_i=\tau_i+\alpha_i I+\beta_iJ+\gamma_i K$. 


Finally, to conclude the section, we provide the formula for the radial part of the horizontal Laplacian of the quaternionic Hopf fibration.

\begin{theorem}\label{radial laplacian quater hopf}
In the trivialization \eqref{trivialization quaternionic}, the horizontal Laplacian acting on functions depending only of $r=\arctan |w|$ and $g \in \mathbf{SU}(2)$ is given by
\[
\frac{\partial^2}{\partial r^2}+((4n-1)\cot r-3\tan r)\frac{\partial}{\partial r}+\tan^2 r \Delta_{\mathbf{SU}(2)}
\]
where $\Delta_{\mathbf{SU}(2)}$ is the Laplacian on $\mathbf{SU}(2)$.
\end{theorem}

\begin{proof}
We use the methods and results of Section \eqref{horizontal BM bundle} and give two proofs of that result.

\textbf{First proof:} We will use the formulas obtained in Example \ref{radial part SU2 bundle}. For this, we need first to compute the Riemannian distance $r$ in $\mathbb{H}P^n$ from 0 to the point with affine coordinates $(w_1,\dots,w_n)$. We can do that computation using the trivialization \eqref{trivialization quaternionic} and compute the Riemannian distance $\varrho$ in $\mathbb{S}^{4n+3}$ from the north pole to the point with coordinates 
\[
\frac{1}{\sqrt{1+|w|^2}}\left(  w_1,\dots,  w_n,  1 \right);
\]
Indeed,
\[
w\to \frac{1}{\sqrt{1+|w|^2}}\left(  w_1,\dots,  w_n,  1 \right)
\]
is a Riemannian immersion of $\mathbb{H}P^n$ into $\bS^{4n+3}$. Since $\mathbb{S}^{4n+3}$ is isometrically embedded into $\mathbb{H}^{n+1}$ one has
\[
\cos \varrho= \frac{1}{\sqrt{1+|w|^2}}.
\]
This gives $r=\varrho=\arctan |w|$. Note that by using the formula
\[
\boldsymbol\alpha =\frac{1}{2}\sum_{j=1}^n  \frac{\overline{w}_j dw_j-  d\overline{w}_j\, w_j}{1+|w|^2}
\]
we have  $\boldsymbol\alpha \left( \frac{\partial}{\partial r} \right)=0$.
Next,  it is known (see Appendix 3) that the quaternionic projective space $\mathbb{H}P^n$ is a rank-one symmetric space for which the radial part of the Laplace-Beltrami operator is given by
\[
\frac{\partial^2}{\partial r^2}+((4n-1)\cot r-3\tan r)\frac{\partial}{\partial r}.
\]
Using Proposition \ref{metric HPn}, one obtains
\[
\frac{1}{(1+|w|^2)^2}h(\A_i,\A_j)= |w|^2 \delta_{ij}= \tan^2 r\, \delta_{ij}.
\]
From Example \ref{radial part SU2 bundle}, this allows to conclude that the horizontal Laplacian on $\bS^{4n+3}$ acting on functions depending only on $r$ and $g \in \mathbf{SU}(2)$ is given by 
\[
\frac{\partial^2}{\partial r^2}+((4n-1)\cot r-3\tan r)\frac{\partial}{\partial r}+\tan^2 r \Delta_{\mathbf{SU}(2)}.
\]

\

\textbf{Second proof:} The Laplacian on $\HP^n$ is given by
\[
\Delta_{\HP^n}=(1+|w|^2) \left( \sum_{i=1}^n \left( \frac{\partial^2}{\partial \tau^2_i}+\frac{\partial^2}{\partial \alpha^2_i}+\frac{\partial^2}{\partial \beta^2_i}+\frac{\partial^2 }{\partial \gamma^2_i} \right)  +\sum_{i=1}^4 \mathcal{R}^2_i  -2\mathcal{R}_1  \right)
\]
The horizontal Laplacian $\Delta_{\mathcal H}$ is the horizontal lift of $\Delta_{\HP^n}$, we therefore need to compute how the horizontal lifts of $\frac{\partial}{\partial \tau_i}$, $\frac{\partial}{\partial \alpha_i}$, $\frac{\partial}{\partial \beta_i}$, $\frac{\partial}{\partial \gamma_i}$ and $\mathcal{R}_i$ act on functions depending on $r$ and $g$ only. The horizontal lift of $\frac{\partial}{\partial \tau_i}$ is given by

\begin{align*}
 & \frac{\partial}{\partial \tau_i}-\frac{1}{1+|w|^2}\left(\alpha_1 \left(\frac{\partial}{\partial \tau_i}\right) \hat{I} +\alpha_2 \left(\frac{\partial}{\partial \tau_i}\right) \hat{J} +\alpha_3 \left(\frac{\partial}{\partial \tau_i}\right) \hat{K} \right) \\
 =& \frac{\partial}{\partial \tau_i}-\frac{1}{1+|w|^2}\left(-\alpha_i\hat{I} -\beta_i \hat{J} -\gamma_i \hat{K} \right).
\end{align*}

where the vector fields  $\hat{I},\hat{J},\hat{K}$ are given by \eqref{I1 quater},\eqref{I2 quater},\eqref{I3 quater}. Similarly, the horizontal lift of $\frac{\partial}{\partial x_i}$ is given by
\begin{align*}
  \frac{\partial}{\partial \alpha_i}-\frac{1}{1+|w|^2}\left(\tau_i \hat{I} -\gamma_i \hat{J} +\beta_i \hat{K} \right).
\end{align*}
 The horizontal lift of $\frac{\partial}{\partial y_i}$ is given by
\begin{align*}
  \frac{\partial}{\partial \beta_i}-\frac{1}{1+|w|^2}\left(\gamma_i \hat{I} +\tau_i \hat{J} -\alpha_i \hat{K} \right).
\end{align*}
The horizontal lift of $\frac{\partial}{\partial z_i}$ is given by
\begin{align*}
  \frac{\partial}{\partial \gamma_i}-\frac{1}{1+|w|^2}\left(-\beta_i \hat{I} +\alpha_i \hat{J}+ \tau_i \hat{K} \right).
\end{align*}
The horizontal lift of $\mathcal{R}_1$ is given by
\begin{align*}
  \sum_{i=1}^n \tau_i \frac{\partial }{\partial \tau_i} +\alpha_i \frac{\partial }{\partial \alpha_i}+\beta_i \frac{\partial }{\partial \beta_i}+\gamma_i \frac{\partial }{\partial \gamma_i}.
  \end{align*}
  The horizontal lift of $\mathcal{R}_2$ is given by
\begin{align*}
  \sum_{i=1}^n \alpha_i \frac{\partial }{\partial \tau_i} -\tau_i \frac{\partial }{\partial \alpha_i}-\gamma_i \frac{\partial }{\partial \beta_i}+\beta_i \frac{\partial }{\partial \gamma_i} +\frac{|w|^2}{1+|w|^2}\hat{I}.
  \end{align*}
  The horizontal lift of $\mathcal{R}_3$ is given by
  \[
  \sum_{i=1}^n \beta_i \frac{\partial }{\partial \tau_i} +\gamma_i \frac{\partial }{\partial \alpha_i}-\tau_i \frac{\partial }{\partial \beta_i}-\alpha_i \frac{\partial }{\partial \gamma_i}+\frac{|w|^2}{1+|w|^2}\hat{J}.
 \]

 The horizontal lift of $\mathcal{R}_4$ is given by
  \[
  \sum_{i=1}^n \gamma_i \frac{\partial }{\partial \tau_i} -\beta_i \frac{\partial }{\partial \alpha_i}+\alpha_i \frac{\partial }{\partial \beta_i}-\tau_i \frac{\partial }{\partial \gamma_i}+\frac{|w|^2}{1+|w|^2}\hat{K}.
 \]

 We deduce from this that the horizontal Laplacian, which is the horizontal lift of $\Delta_{\HP^n}$ is given, after simplifications, by
  \begin{align*}
  & (1+|w|^2) \left( \sum_{i=1}^n \left( \frac{\partial^2}{\partial \tau^2_i}+\frac{\partial^2}{\partial \alpha^2_i}+\frac{\partial^2}{\partial \beta^2_i}+\frac{\partial^2 }{\partial \gamma^2_i} \right)  +\sum_{i=1}^4 \hat{\mathcal{R}}^2_i  -2\hat{\mathcal{R}}_1  \right) \\
  +& 2 (\hat{\mathcal{R}}_2 \hat{I} +\hat{\mathcal{R}}_3 \hat{J}+\hat{\mathcal{R}}_4 \hat{K}) +|w|^2 \left( \hat{I}^2+\hat{J}^2+\hat{K}^2\right),
 \end{align*}
 where the $\hat{\mathcal{R}}_i$'s are the vector fields on $\bS^{4n+3}$ given by
 \[
 \hat{\mathcal{R}}_1  = \sum_{i=1}^n \tau_i \frac{\partial }{\partial \tau_i} +\alpha_i \frac{\partial }{\partial \alpha_i}+\beta_i \frac{\partial }{\partial \beta_i}+\gamma_i \frac{\partial }{\partial \gamma_i}
 \]
 \[
 \hat{\mathcal{R}}_2  = \sum_{i=1}^n \alpha_i \frac{\partial }{\partial \tau_i} -\tau_i \frac{\partial }{\partial \alpha_i}-\gamma_i \frac{\partial }{\partial \beta_i}+\beta_i \frac{\partial }{\partial \gamma_i}
 \]
 \[
\hat{ \mathcal{R}}_3  = \sum_{i=1}^n \beta_i \frac{\partial }{\partial \tau_i} +\gamma_i \frac{\partial }{\partial \alpha_i}-\tau_i \frac{\partial }{\partial \beta_i}-\alpha_i \frac{\partial }{\partial \gamma_i}
 \]
 \[
 \hat{\mathcal{R}}_4  = \sum_{i=1}^n \gamma_i \frac{\partial }{\partial \tau_i} -\beta_i \frac{\partial }{\partial \alpha_i}+\alpha_i \frac{\partial }{\partial \beta_i}-\tau_i \frac{\partial }{\partial \gamma_i}.
 \]
 A direct computation shows then that this operator acts on function depending only on $r$ and $g$ as
 \[
 \frac{\partial^2}{\partial r^2}+((4n-1)\cot r-3\tan r)\frac{\partial}{\partial r}+\tan^2 r \Delta_{\mathbf{SU}(2)}.
 \]
\end{proof}

\subsection{Stochastic area process on $\mathbb{H}P^n$}

In view of Section \ref{Quaternionic Hopf} the following definition is natural.

\begin{definition}\label{quaternionic stochastic area sphere}
Let $(w(t))_{t \ge 0}$ be a Brownian motion on $\mathbb{H}P^n$ started at $0$\footnote{We call $0$ the point in $\mathbb{H}P^n$ with affine coordinates $w_1=0,\dots, w_{n}=0$}. The quaternionic stochastic area process of $(w(t))_{t \ge 0}$ is a process in $\mathfrak{su}(2)$ defined by
\[
\A(t)=\int_{w[0,t]} \A=\frac{1}{2}\sum_{j=1}^n \int_0^t \frac{\overline{w}_j(s) dw_j(s)-  d\overline{w}_j(s) \, w_j(s)}{1+|w(s)|^2},
\]
where the above stochastic integrals are understood in the Stratonovich, or equivalently in the It\^o sense due to the skew-symmetric structure of the form.
\end{definition}
The following theorem shows that the quaternionic stochastic area process of the Brownian motion on $\mathbb{H}P^n$ can be interpreted as the fiber motion of the horizontal Brownian motion on $\mathbb{S}^{4n+3}$.

 \begin{theorem}\label{horizon-S}
 Let $(w(t))_{t \ge 0}$ be a Brownian motion on $\mathbb{H}P^n$ started at 0, and $(\Theta(t))_{t\ge 0}$ be the solution in $\mathbf{SU}(2)$ of the stochastic differential equation
 \begin{equation}\label{eq-Maurer-Cartan-S}
 d\Theta(t) =-  \circ d\A(t) \, \Theta(t) , \quad \Theta(0)=1.
 \end{equation}
 The $\bS^{4n+3}$-valued diffusion process
 \begin{equation}\label{eq-BM-S}
 X(t)=\frac{1 }{\sqrt{1+|w(t)|^2}} \left( w(t)\Theta(t),\Theta(t) \right), \quad t \ge 0
 \end{equation}
 is the horizontal lift at the north pole of $(w(t))_{t \ge 0}$ by the submersion $\bS^{4n+3}\to \bH P^n$.
 \end{theorem}

 \begin{proof}
This immediately follows from Theorem \ref{skew-product principal bundle} since that $\boldsymbol\alpha$ is the  form (in the sense of \eqref{solder form}) associated with the connection form $\theta$.. 
\end{proof}

Next we show that the fiber motion $\Theta(t)$ on $\mathbb{S}^{4n+3}$ is in fact a time-changed Brownian motion process on $\mathbf{SU}(2)$. 
\begin{theorem}\label{diff1-S}
Let $r(t)=\arctan |w(t)|$. The process $\left( r(t), \Theta(t)\right)_{t \ge 0}$ is a diffusion with generator
 \[
\frac{1}{2} \left(\frac{\partial^2}{\partial r^2}+((4n-1)\cot r-3\tan r)\frac{\partial}{\partial r}+\tan^2 r \Delta_{\mathbf{SU}(2)} \right).
 \]
 As a consequence the following equality in distribution holds
 \begin{equation}\label{eq-mp-S}
\left( r(t) ,\Theta(t) \right)_{t \ge 0}=\left( r(t),\beta \left( \int_0^t \tan^2 r(s)ds\right)\right)_{t \ge 0},
\end{equation}
where $(\beta(t))_{t \ge 0}$ is a Brownian motion process on $\mathbf{SU}(2)$ independent from $r$.
\end{theorem}

\begin{proof}
This follows from Theorem \ref{radial laplacian quater hopf}.
\end{proof}

\begin{corollary}\label{diff2-S}
Let $r(t)=\arctan |w(t)|$. The process $\left( r(t), \A(t)\right)_{t \ge 0}$ is a diffusion with generator
 \[
\frac{1}{2} \left(\frac{\partial^2}{\partial r^2}+((4n-1)\cot r-3\tan r)\frac{\partial}{\partial r}+\tan^2 r \Delta_{\mathfrak{su}(2)} \right).
 \]
 As a consequence the following equality in distribution holds
 \begin{equation}\label{eq-mp-S}
\left( r(t) ,\A(t) \right)_{t \ge 0}=\left( r(t),\gamma\left(\int_0^t \tan^2 r(s)ds\right)\right)_{t \ge 0},
\end{equation}
where $(\gamma(t))_{t \ge 0}$ is a Brownian motion  on $\mathfrak{su}(2)$ independent from $r$.
\end{corollary}

\begin{proof}
From the previous theorem,
\begin{align*}
\A (t)&=- \int_0^t   \circ d\Theta (s) \, \Theta (s)^{-1} \\
 & =-\int_0^{t}  \circ d\beta \left(\int_0^s \tan^2 r(u)du\right) \,  \beta \left(\int_0^s \tan^2 r(u)du\right)^{-1} \\
 &=-\int_0^{\int_0^t \tan^2 r(u)du}  \circ d\beta \left(s\right) \,  \beta \left(s\right)^{-1}.
\end{align*}
Since
\[
\gamma (t) :=- \int_0^{t}   \circ d\beta (s) \, \beta(s)^{-1},  \quad t \geq 0
\]
is a Brownian motion  on $\mathfrak{su}(2)$. The proof is complete.
\end{proof}

\subsection{Horizontal heat kernel}

In this section, we study the distribution of the horizontal Brownian motion $(X(t))_{t \ge 0}$ on $\mathbb S^{4n+3}$. From Theorem \eqref{horizon-S}, we have that
\begin{align*}
 \mathbb{E}( f(X(t)))& =\mathbb{E} \left[ f \left( \frac{1 }{\sqrt{1+|w(t)|^2}} \left( w(t)\Theta(t),\Theta(t) \right) \right) \right]
 \end{align*}
 where $(\Theta(t))_{t\ge 0}$ is the solution in $\mathbf{SU}(2)$ of the stochastic differential equation
 \begin{equation*}
 d\Theta(t) =-  \circ d\A(t) \, \Theta(t) , \quad \Theta(0)=1,
 \end{equation*}
 and, as before, $\A(t)=\frac{1}{2}\sum_{j=1}^n \int_0^t \frac{\overline{w}_j(s) dw_j(s)-  d\overline{w}_j(s) \, w_j(s)}{1+|w(s)|^2}$ is the generalized stochastic area process. Thanks to Theorem \ref{diff1-S} one can then write
 \begin{align*}
  & \mathbb{E} \left[ f \left( \frac{1 }{\sqrt{1+|w(t)|^2}} \left( w(t)\Theta(t),\Theta(t) \right) \right) \right] \\
  =&  \int_{\mathbb{H}P^n} \int_{\mathbf{SU}(2)} f \left( \frac{1 }{\sqrt{1+|w|^2}} \left( wg,g\right) \right) k_{t/2} (\arctan |w|, g ) d \mu (w,g)
 \end{align*}
 where $\mu$ is the Riemannian volume measure on $\mathbb S^{4n+3}$ and $k_t(r,g)$, $r \in [0,\pi/2), g \in \mathbf{SU}(2)$ is the heat kernel of the operator
 \[
\frac{\partial^2}{\partial r^2}+((4n-1)\cot r-3\tan r)\frac{\partial}{\partial r}+\tan^2 r \Delta_{\mathbf{SU}(2)}.
 \]
 We now note that the Lie group $\mathbf{SU}(2)$ with its bi-invariant Riemannian metric  is isometric to the 3-dimensional sphere $\mathbb{S}^3$ equipped with its standard metric. Therefore, $k_t(r,g)=p_t(r, d(e,g))$ where $d(e,g)$ is the distance in $\mathbf{SU}(2)$  from the identity $e$ to $g$ and $p_t(r, \eta)$, $r\in[0,\frac{\pi}{2}]$, $ \eta\in[0,\pi]$ is the heat kernel of the operator
 \[
 L=\frac{\partial^2}{\partial r^2}+((4n-1)\cot r-3\tan r)\frac{\partial}{\partial r}+\tan^2 r \left(\frac{\partial^2}{\partial \eta^2}+2\cot\eta\frac{\partial}{\partial \eta}  \right).
 \]
 
The invariant and symmetric measure of $L$ is explicitly given by
\[
d\mu_r=\frac{8\pi^{2n+1}}{\Gamma(2n)}(\sin r)^{4n-1}(\cos r)^3(\sin\eta)^2drd\eta.
\]
We note that the normalization constant is chosen here in such a way that
\[
\int_{0}^{\pi}\int_0^{\frac{\pi}{2}}d\mu_r=\mu(\bS^{4n+3})=\frac{2\pi^{2n+2}}{\Gamma (2n+2)},
\]
so that the following change of formula holds
\[
 \int_{\mathbb{H}P^n} \int_{\mathbf{SU}(2)} f ( \arctan |w|, d (e,g) ) d \mu (w,g)=\int_0^{\pi}  \int_0^{\frac{\pi}{2}} f(r,\eta) d\mu_r (r,\eta)
\]
for any function $f$.

\subsubsection{Spectral decomposition of the heat kernel}\label{spectral}

In this section, we derive the spectral decomposition of the subelliptic heat kernel of the heat semigroup $P_t=e^{tL}$, where as above
\[
 L=\frac{\partial^2}{\partial r^2}+((4n-1)\cot r-3\tan r)\frac{\partial}{\partial r}+\tan^2 r \left(\frac{\partial^2}{\partial \eta^2}+2\cot\eta\frac{\partial}{\partial \eta}  \right).
 \]
\begin{proposition}
For $t>0$, $r\in[0,\frac{\pi}{2}]$, $ \eta\in[0,\pi]$, the horizontal heat kernel is given by
\begin{equation}\label{pt}
p_t(r,\eta)=\sum_{m=0}^{+\infty}\sum_{k=0}^{\infty}\alpha_{k,m}e^{-4[k(k+2n+m+1)+nm]t}\frac{\sin (m+1)\eta}{\sin \eta}(\cos r)^mP_k^{2n-1,m+1}(\cos 2r)
\end{equation}
where $\alpha_{k,m}=\frac{\Gamma(2n)}{2\pi^{2n+2}}(2k+m+2n+1)(m+1){k+m+2n\choose 2n-1}$ and 
\[
P_k^{2n-1,m+1}(x)=\frac{(-1)^k}{2^kk!(1-x)^{2n-1}(1+x)^{m+1}}\frac{d^k}{dx^k}\left((1-x)^{2n-1+k}(1+x)^{m+1+k} \right)
\] 
is a Jacobi polynomial.
\end{proposition} 
\begin{proof}
The idea is to   expand the subelliptic kernel in spherical harmonics as follows, 
\[
p_t(r,\eta)=\sum_{m=0}^{+\infty}\frac{\sin (m+1)\eta}{\sin \eta}\phi_m(t,r)
\]
where $\frac{\sin (m+1)\eta}{\sin \eta}$ is the eigenfunction of $\frac{\partial^2}{\partial \eta^2}+2\cot \eta\frac{\partial}{\partial \eta}$ which is associated to the eigenvalue $-m(m+2)$. To determine $\phi_m$, we use  $\frac{\partial p_t}{\partial t}=Lp_t$ and find that
\[
\frac{\partial\phi_m}{\partial t}=\frac{\partial^2\phi_m}{\partial r^2}+\left((4n-1)\cot r-3\tan r \right)\frac{\partial\phi_m}{\partial r}-m(m+2)\tan^2r \, \phi_m.
\]
Let $\phi_m(t,r)=e^{-4nmt}(\cos r)^m\varphi_m(t,r)$, then $\varphi_m(t,r)$ satisfies the equation
\[
\frac{\partial\varphi _m}{\partial t}=\frac{\partial^2\varphi_m}{\partial r^2}+[(4n-1)\cot r-(2m+3)\tan r]\frac{\partial\varphi_m}{\partial r}.
\]
We now use a change of variable and denote by $\varphi_m(t,r)=g_m(t,\cos 2r)$, then we have that $g_m(t,x)$ satisfies the equation
\[
\frac{\partial g_m}{\partial t}=4(1-x^2)\frac{\partial^2 g_m}{\partial x^2}+4[(m+2-2n)-(2n+m+2)x]\frac{\partial g_m}{\partial x}.
\]
We denote $\Psi_m=(1-x^2)\frac{\partial^2 }{\partial x^2}+[(m+2-2n)-(2n+m+2)x]\frac{\partial}{\partial x}$, and find that
\[
\frac{\partial g_m}{\partial t}=4\Psi_m(g_m).
\]
The equation $$\Psi_m(g_m)+k(k+2n+m+1)g_m=0$$ is a Jacobi differential equation for all $k\geq 0$.  We denote the eigenvector of $\Psi_m$ corresponding to the eigenvalue $-k(k+2n+m+1)$ by $P_k^{2n-1,m+1}(x)$, then it is known that  $P_k^{2n-1,m+1}(x)$ is the Jacobi polynomial
\[
P_k^{2n-1,m+1}(x)=\frac{(-1)^k}{2^kk!(1-x)^{2n-1}(1+x)^{m+1}}\frac{d^k}{dx^k}\left((1-x)^{2n-1+k}(1+x)^{m+1+k} \right).
\] 
At the end we can  write the spectral decomposition as
\[
p_t(r,\eta)=\sum_{m=0}^{+\infty}\sum_{k=0}^{\infty}\alpha_{k,m}e^{-4[k(k+2n+m+1)+nm]t}\frac{\sin (m+1)\eta}{\sin \eta}(\cos r)^mP_k^{2n-1,m+1}(\cos 2r)
\]
where the constants $\alpha_{k,m}$'s are  determined by considering the initial condition.

Note that $\left(P_k^{2n-1,m+1}(x)(1+x)^{(m+1)/2}\right)_{k\geq0}$ is an orthogonal basis of the Hilbert space $L^2([-1,1],(1-x)^{2n-1}dx)$, more precisely
\begin{align*}
 & \int_{-1}^1 P_k^{2n-1,m+1}(x)P_l^{2n-1,m+1}(x)(1-x)^{2n-1}(1+x)^{m+1}dx \\
 =&\frac{2^{2n+m+1}}{2k+m+2n+1}\frac{\Gamma(k+2n)\Gamma(k+m+2)}{\Gamma(k+1)\Gamma(k+2n+m+1)}\delta_{kl}.
\end{align*}
For a smooth function $f(r, \theta)$, we can write
\[
f(r, \eta)=\sum_{m=0}^{+\infty}\sum_{k=0}^{+\infty} b_{k,m}\frac{\sin (m+1)\eta}{\sin \eta}P_k^{2n-1,m+1}(\cos 2r)\cdot(\cos r)^{m}
\]
where the $ b_{k,m}$'s are constants. We then obtain 
\[
f(0,0)=\sum_{m=0}^{+\infty}\sum_{k=0}^{+\infty} b_{k,m}(m+1)P_k^{2n-1,m+1}(1).
\]
Observe that $P_k^{2n-1,m+1}(1)={2n-1+k\choose k}$.
Moreover, since
\begin{eqnarray*}
& &\int_{0}^\pi\int_0^\frac{\pi}{2} p_t(r, \eta){f(-r, -\eta)}d\mu_r \\
&=&\frac{4\pi^{2n+2}}{\Gamma(2n)}\sum_{m=0}^{+\infty}\sum_{k=0}^{+\infty}\alpha_{k,m}b_{k,m}e^{-\lambda_{k,m}t}
\left(\int_0^{\frac{\pi}{2}}(\cos r)^{2m+3}P_k^{2n-1,m+1}(\cos 2r)^2(\sin r)^{4n-1}dr\right)\\
&=&\frac{2\pi^{2n+2}}{\Gamma(2n)} \sum_{m=0}^{+\infty}\sum_{k=0}^{+\infty}\frac{\alpha_{k,m}b_{k,m}e^{-\lambda_{m,k}t}}{2k+m+2n+1}\frac{\Gamma(k+2n)\Gamma(k+m+2)}{\Gamma(k+1)\Gamma(k+2n+m+1)}
\end{eqnarray*}
where $\lambda_{k,m}=4k(k+2n+m+1)+nm$, we obtain that
\[
\lim_{t\rightarrow 0}\int_{0}^\pi\int_0^\frac{\pi}{2} p_t(r, \eta)f(-r, -\eta)d\mu_r= f(0,0)
\]
as soon as $\alpha_{k,m}=\frac{\Gamma(2n)}{2\pi^{2n+2}}(2k+m+2n+1)(m+1){k+m+2n\choose 2n-1}$.
\end{proof}

Comparing this expansion with the one obtained in Proposition \ref{heat hopf fibration}, we obtain a very nice formula relating $p_t$ to  the heat kernel of the complex Hopf fibration
 \[
  \mathbb{S}^1 \to \bS^{4n+1} \to  \mathbb{C}P^{2n}.
 \]
More precisely, we proved in Proposition \ref{heat hopf fibration} that  the subelliptic kernel $p_t^{CR}(r, \theta)$ of the complex Hopf  fibration writes:
\begin{align*}
p_t^{CR}(r, \theta)=\frac{\Gamma(2n)}{2\pi^{2n+1}}\sum_{m=-\infty}^{+\infty}&\sum_{k=0}^{+\infty} (2k+|m|+2n){k+|m|+2n-1\choose 2n-1}\\
&\quad\quad \cdot e^{-\lambda_{k,m}t+im \theta}(\cos r)^{|m|}P_k^{2n-1,|m|}(\cos 2r),
\end{align*}
where $\lambda_{k,m}=4k(k+|m|+2n)+4|m|n$, and $r$, $\theta$ are the Riemannian distance on $\mathbb{C}P^{2n}$, $\bS^1$ respectively. By comparing the kernels on $\bS^{4n+3}$ and $\bS^{4n+1}$, we easily obtain that
\begin{proposition}
Let $p_t^{CR}(r, \theta)$ and $p_t(r, \theta)$ denote the horizontal heat kernels on respectively the complex sphere $\mathbb S^{4n+1}$ and the quaternionic sphere $\mathbb S^{4n+3}$, then for $r\in[0,\frac{\pi}{2})$, $ \theta\in[0,\pi]$,
\begin{equation}\label{pt-qt}
-\frac{e^{4nt}}{2\pi\sin \theta\cos r}\frac{\partial}{\partial \theta} p_t^{CR}(r, \theta)=p_t(r,\theta).
\end{equation}
\end{proposition}

\subsubsection{Integral representation of the horizontal heat kernel}

Our goal in this section is to provide an alternative representation of $p_t(r,\eta)$.

From \eqref{decomposition laplace} we have that
\[
L=\tilde{\Delta}_{\bS^{4n+3}}-\tilde{\Delta}_{\mathbf{SU}(2)},
\]
where $\tilde{\Delta}_{\bS^{4n+3}}$ denotes here the Laplace-Beltrami operator on $\bS^{4n+3}$ acting on functions depending only on $(r,\eta)$ and $\tilde{\Delta}_{\mathbf{SU}(2)}$ is the radial part of the Laplacian on $\mathbf{SU}(2) \simeq \mathbb{S}^3$. Explicitly, we have
\[
\tilde{\Delta}_{\bS^{4n+3}}=\frac{\partial^2}{\partial r^2}+((4n-1)\cot r-3\tan r)\frac{\partial}{\partial r}+\frac{1}{\tan^2 r} \left(\frac{\partial^2}{\partial \eta^2}+2\cot\eta\frac{\partial}{\partial \eta}  \right)
\]
and 
\[
\tilde{\Delta}_{\mathbf{SU}(2)}=\frac{\partial^2}{\partial \eta^2}+2\cot\eta\frac{\partial}{\partial \eta}. 
\]
Since $L$ commutes with $\tilde{\Delta}_{\mathbf{SU}(2)}$  we heuristically  have that
\[
e^{tL}=e^{-t\tilde{\Delta}_{\mathbf{SU}(2)}}e^{t\tilde{\Delta}_{\bS^{4n+3}}}.
\]

We denote by $q_t$ the heat kernel of the heat semigroup $e^{t\tilde{\Delta}_{\bS^{4n+3}}}$, then  the subelliptic heat kernel $p_t(r,\eta)$ can be obtained by applying the heat semigroup $e^{-t\tilde{\Delta}_{\mathbf{SU}(2)}}$ on $q_t$, i.e.
\[
p_t(r,\eta)=(e^{-t\tilde{\Delta}_{\mathbf{SU}(2)}}q_t)(r,\eta).
\]
We denote by $h_t(\eta)$ the heat kernel of $\tilde{\Delta}_{\mathbf{SU}(2)}$, that is, $h_t$ is the fundamental solution of 
\[
\frac{\partial }{\partial t}h_t(\eta)=\left(\frac{\partial^2}{\partial\eta^2}+2\cot\eta\frac{\partial}{\partial\eta}\right)h_t(\eta).
\]
Consider $g_t(\eta):=h_t(i\eta)$, then $g_t(\eta)$ satisfies 
\[
\frac{\partial }{\partial t}g_t(\eta)=-\left(\frac{\partial^2}{\partial\eta^2}+2\coth\eta\frac{\partial}{\partial\eta}\right)g_t(\eta).
\]
If we denote $\tilde{\Delta}_{\mathbf{SL}(2)}=\frac{\partial^2}{\partial \eta ^2}+2 \coth \eta \frac{\partial}{\partial \eta} $, then we have that
\begin{equation}\label{pt-qts}
p_t(r,\eta)=(e^{t\tilde{\Delta}_{\mathbf{SL}(2)}}q_t)(r,-i\eta).
\end{equation}

As a conclusion, an integral representation of $p_t$,  can be obtained from an explicit expression of the heat semigroup $e^{t\tilde{\Delta}_{\mathbf{SL}(2)}}$. 
\begin{lemma}
Let $\tilde{\Delta}_{\mathbf{SL}(2)}=\frac{\partial^2}{\partial \eta ^2}+2 \coth \eta \frac{\partial}{\partial \eta} $. For every $f:\mathbb{R}_{\ge 0} \to \mathbb{R}$ which is smooth and rapidly decreasing we have:
\begin{equation}\label{semigroup}
(e^{t\tilde{\Delta}_{\mathbf{SL}(2)}} f)(\eta)=\frac{e^{-t}}{\sqrt{\pi t}} \int_0^{+\infty} \frac{ \sinh s \sinh \left(  \frac{\eta s}{2t}\right) }{\sinh \eta} e^{-\frac{s^2+\eta^2}{4t}} f( s ) ds, \quad t \ge 0, \eta \ge 0.
\end{equation}
\end{lemma}

\begin{proof}
We observe that:
\[
\tilde{\Delta}_{\mathbf{SL}(2)} f=\frac{1}{h} (\tilde{\Delta}_{\R^3}-1)(hf),
\]
where
\[
\tilde{\Delta}_{\R^3} =\frac{\partial^2}{\partial \eta ^2}+\frac{2}{  \eta} \frac{\partial}{\partial \eta} , \quad h(\eta)=\frac{\sinh \eta} {\eta}.
\]
As a consequence, we have
\[
(e^{t\tilde{\Delta}_{\mathbf{SL}(2)}} f)(\eta)=\frac{e^{-t}}{h(\eta)}e^{t \tilde{\Delta}_{\R^3}} (hf) (\eta).
\]
We are thus let with the computation of $e^{t \tilde{\Delta}_{\R^3}}$. The operator $\tilde{\Delta}_{\R^3}$ is the radial part of the Laplacian $\Delta_{\mathbb{R}^3}$, thus for $x \in \mathbb{R}^3$,
\[
e^{t \Delta_{\mathbb{R}^3}}( f \circ s)(x)= (e^{t\tilde{\Delta}_{\R^3}}f )(s(x)),
\]
where $s(x)=| x |$. Since,
\[
e^{t \Delta_{\mathbb{R}^3}}( f \circ s)(x)=\frac{1}{(4\pi t )^{3/2}} \int_{\mathbb{R}^3} e^{-\frac{{|y-x|}^2}{4t} } f({|y|}) dy,
\]
a routine computation in spherical coordinates shows that
\[
e^{t\tilde{\Delta}_{\R^3}}f (\eta)=\frac{1}{\sqrt{\pi t}} \int_0^{+\infty} \frac{s}{\eta} \sinh \left( \frac{\eta s}{2t}\right) e^{-\frac{s^2+\eta^2}{4t}} f ( s ) ds.
\]
The conclusion easily follows.
\end{proof}

Now we can immediately deduce the integral representation of the subelliptic heat kernel on $\bS^{4n+3}$: 
\begin{proposition}\label{prop-pt}
For $t>0$, $r\in[0,\pi/2)$, $ \eta\in[0,\pi]$, 
\begin{equation}\label{pt-int} 
p_t(r, \eta)=\frac{e^{-t}}{\sqrt{\pi t}} \int_0^{+\infty} \frac{ \sinh y \sin \left(  \frac{\eta y}{2t}\right) }{\sin \eta} e^{-\frac{y^2-\eta^2}{4t}} q_t( \cos r\cosh y ) dy.
\end{equation}
\end{proposition}

\begin{proof}
Because of \eqref{pt-qts}, by plugging in $q_t$ to \eqref{semigroup}, we immediately obtain the desired integral representation.
\end{proof}

For the record we note here that the heat kernel $q_t$  on $\bS^{4n+3}$ with respect to its Riemannian structure has been well studied, see Section \ref{Jacobi diffusions}. We list two useful representations of it: 
\begin{itemize}
\item[(1)]\hspace{.1in}  The spectral decomposition of $q_t$ is given by
\begin{equation}
q_t{(\cos\delta)}=\frac{\Gamma(2n+1)}{2\pi^{2n+2}}\sum_{m=0}^{+\infty}(m+2n+1)e^{-m(m+4n+2)t}C_m^{2n+1}(\cos \delta),
\end{equation}
where $\delta$ is the Riemannian distance from the north pole and
\[
C_m^{2n+1}(x)=\frac{(-1)^m}{2^m}\frac{\Gamma(m+4n+2)\Gamma(2n+3/2)}{\Gamma(4n+2)\Gamma(m+1)\Gamma(2n+m+3/2)}\frac{1}{(1-x^2)^{2n+1/2}}\frac{d^m}{dx^m}(1-x^2)^{2n+m+1/2}
\]
is a Gegenbauer polynomial.  
\item[(2)]\hspace{.1in} Another expression of $q_t (\cos \delta)$ which is for instance useful for the computation of small-time asymptotics is 
\begin{equation}\label{heat_kernel_odd}
q_t (\cos \delta)= e^{(2n+1)^2t} \left( -\frac{1}{2\pi \sin \delta} \frac{\partial}{\partial \delta} \right)^{2n+1} V(t,\delta),
\end{equation}
where $V(t,\delta)=\frac{1}{\sqrt{4\pi t}} \sum_{k \in \mathbb{Z}} e^{-\frac{(\delta-2k\pi)^2}{4t} }$. 
\end{itemize}

\subsection{Characteristic function of the stochastic area and limit theorem}
We now study the characteristic function of $\A(t)$.
Let $\lambda \in \R^3$, $r \in [0,\infty)$ and 
\[
I(\lambda,r)=\mathbb{E}\left(e^{i \lambda\cdot \A(t)}\bigm| r(t)=r\right).
\]
From Corollary \ref{diff2-S}, we know that
\begin{align*}
I(\lambda,r)& =\mathbb{E}\left(e^{i \lambda\cdot \gamma\left({\int_0^t \tan^2 r(s)ds}\right)}\bigm| r(t)=r\right) \\
 &=\mathbb{E}\left(e^{- \frac{|\lambda|^2}{2} \int_0^t \tan^2 r(s)ds}\bigm| r(t)=r\right) 
\end{align*}
and $r$ is a diffusion with the generator
\[
\mathcal{L}^{2n-1,1}=\frac{1}{2}\frac{\partial^2}{\partial r^2}+\frac12((4n-1)\cot r-3\tan r)\frac{\partial}{\partial r}
\]
started at $0$.  More generally, the circular Jacobi generator is defined by:
\[
\mathcal{L}^{\alpha,\beta}=\frac{1}{2} \frac{\partial^2}{\partial r^2}+\left(\left(\alpha+\frac{1}{2}\right)\cot r-\left(\beta+\frac{1}{2}\right) \tan r\right)\frac{\partial}{\partial r}, \quad \alpha,\beta >-1,
\]
and we refer the reader to the Appendix 2 for further details. We denote by $q_t^{\alpha,\beta}(r_0,r)$ its corresponding transition density with respect to the Lebesgue measure.
\begin{theorem}\label{FThj-S}
For $\lambda \in \R^3$, $r \in [0,\pi/2)$, and $t >0$ we have
\begin{equation}\label{eq-ft-cond-S}
\mathbb{E}\left(e^{i \lambda\cdot \A(t)}\bigm| r(t)=r\right)
 =\frac{e^{-2n \mu t}}{(\cos r)^{\mu}} \frac{q_t^{2n-1,\mu+1}(0,r)}{q_t^{2n-1,1}(0,r)},
\end{equation}
where $\mu=\sqrt{|\lambda|^2+1}-1$.
\end{theorem}

\begin{proof}
Note 
\[
dr(t)= \frac{1}{2} \left( (4n-1)\cot r(t)-3\tan r(t) \right)dt+d\gamma(t),
\]
where $\gamma$ is a standard Brownian motion.
Consider the local martingale defined for any $\mu>0$ by
\begin{align*}
D_t& =\exp \left( -\mu \int_0^t \tan r(s) d\gamma(s) -\frac{\mu^2}{2}  \int_0^t \tan^2 r(s) ds \right)  \\
 &=\exp \left( -\mu \int_0^t \tan r(s) dr(s)+\frac{\mu}{2}(4n-1)t -\frac{3\mu +\mu^2}{2}  \int_0^t \tan^2 r(s) ds \right).
\end{align*}
From It\^o's formula, we have
\begin{align*}
\ln \cos r(t) & =-\int_0^t \tan r(s) dr(s)-\frac{1}{2} \int_0^t \frac{ds}{\cos^2 r(s)} \\
 &=-\int_0^t \tan r(s) dr(s)-\frac{1}{2} \int_0^t \tan^2 r(s) ds-\frac{1}{2} t.
\end{align*}
As a consequence, we deduce that
\[
D_t =e^{2n\mu t} (\cos r(t))^{\mu} e^{- (\frac{\mu^2}{2}+\mu) \int_0^t \tan^2 r(s)ds}.
\]
This expression of $D$ implies that almost surely $D_t \le e^{2n\mu t}$ and thus $D$ is a true martingale. Let us denote by $\mathcal{F}$ the natural filtration of $r$ and consider the probability measure $\mathbb{P}^{\mu}$ defined by
\[
\mathbb{P}_{/ \mathcal{F}_t} ^{\mu}=e^{2n\mu t} (\cos r(t))^\mu e^{- (\frac{\mu^2}{2}+{\mu}) \int_0^t \tan^2 r(s)ds} \mathbb{P}_{/ \mathcal{F}_t}.
\]
We have then for every bounded and Borel function $f$ on $[0,\pi /2]$,
\begin{align*}
\mathbb{E}\left(f(r(t))e^{- (\frac{{\mu}^2}{2}+{\mu}) \int_0^t \tan^2 r(s)ds}\right)=e^{-2n{\mu} t} \mathbb{E}^{\mu} \left( \frac{f(r(t))}{(\cos r(t))^{\mu}} \right).
\end{align*}
By Girsanov Theorem, the process defined by
\[
\beta(t)=\gamma(t)+{\mu} \int_0^t \tan r(s) ds
\]
is a Brownian motion under the probability $\mathbb{P}^{\mu}$. Since
\[
dr(t)= \frac{1}{2} \left( (4n-1)\cot r(t)-(2{\mu} +3)\tan r(t) \right)dt+d\beta(t),
\]
the proof is complete by letting $\mu=\sqrt{|\lambda|^2+1}-1$.
\end{proof}

By inverting the Fourier transform \eqref{eq-ft-cond-S}, we get a formula for the joint density of $(r(t) ,\A(t))$.

\begin{corollary}
For $ r \in [0,\pi/2]$, $\A \in \mathbb R^3$, and $t>0$
\[
\mathbb{P} ( r(t) \in dr, \A(t) \in d \A)=\frac{1}{(2\pi)^3} \int_{\mathbb R^3} \frac{e^{-2n(\sqrt{|\lambda|^2+1}-1) t}}{(\cos r)^{\sqrt{|\lambda|^2+1}-1}} q_t^{2n-1,\sqrt{|\lambda|^2+1}}(0,r) e^{-i\lambda \cdot \A} d\lambda
\]
\end{corollary}

\begin{corollary}\label{cor-theta-S}
For $\lambda \in \mathbb{R}^3$ and $t \ge 0$,
\[
\mathbb{E}\left(e^{i {\lambda}\cdot \A(t)}\right)=e^{-2n \mu t}\int_0^{\pi /2}  \frac{q_t^{2n-1,\mu+1}(0,r)}{(\cos r)^{\mu}} dr,
\]
where $\mu=\sqrt{|\lambda|^2+1}-1$.
\end{corollary}

We are now in position to prove a central limit type theorem for $\A(t)$. 

\begin{theorem}\label{limit CP}
When $t \to +\infty$, the following convergence in distribution takes place
\[
\frac{\A(t)}{\sqrt t} \to \mathcal{N}(0, 2n \mathrm{I}_3).
\]
\end{theorem}

\begin{proof}
From Corollary \ref{cor-theta-S} we have for every $t>0$,
\[
\mathbb{E}\left(e^{i \lambda\cdot \frac{\A(t)}{\sqrt t}}\right)=e^{-2n(  \sqrt{t|\lambda|^2+t^2}-t) }\int_0^{\pi /2}  \frac{q_t^{2n-1, \sqrt{\frac{|\lambda|^2}{t}+1} }(0,r)}{(\cos r)^{ \sqrt{\frac{|\lambda|^2}{t}+1} -1}} dr. 
\]
Using the formula for $q_t^{n-1, \mu }(0,r)$ which is given in Appendix 2, we obtain by dominated convergence that
\[
\lim_{t \to \infty}  \int_0^{\pi /2}  \frac{q_t^{2n-1, \sqrt{\frac{|\lambda|^2}{t}+1} }(0,r)}{(\cos r)^{ \sqrt{\frac{|\lambda|^2}{t}+1} -1}} dr=  \int_0^{\pi /2}  q_\infty^{2n-1,1}(0,r) dr =1.
\]
On the other hand,
\[
\lim_{t \to \infty}  \sqrt{t|\lambda|^2+t^2}-t =\frac{1}{2} | \lambda|^2,
\]
thus one concludes
\[
\lim_{t \to \infty} \mathbb{E}\left(e^{i \lambda\cdot \frac{\A(t)}{\sqrt t}}\right)=e^{- n | \lambda |^2}.
\]
\end{proof}

\subsection{Formula for the density}
The derivation of the density of $\A(t)$ in this setting is rather direct compared to its anti de-Sitter  analog. Actually, we shall use the explicit expression of the circular Jacobi heat kernel to give a more elaborate expression of the characteristic function, namely: 
\begin{corollary}[of Corollary \ref{cor-theta-S}]
The characteristic function of  the generalized stochastic area process admits the absolutely-convergent expansion: 
\begin{multline*}
\mathbb{E}\left(e^{i {\lambda}\cdot \A(t)}\right) = e^{-2n\mu t}\sum_{j \geq 0} (-1)^j \frac{(2n)_j}{j!} (2j+2n+\mu+1) e^{-2j(j+2n+\mu+1)t} \\ \frac{\mu(\mu+2)}{4}\frac{\Gamma(j+\mu/2)\Gamma(j+2n+\mu+1)}{\Gamma(j+\mu/2+2n+2)\Gamma(j+\mu+2)},
\end{multline*}
where, as before, $\mu=\sqrt{|\lambda|^2+1}-1$.
\end{corollary}
\begin{proof}
Recall the notation $q_t^{(2n-1, \mu+1)}(0,r)$ for the heat kernel of the spherical Jacobi operator $\mathcal{L}^{2n-1,\mu+1}$. Then, it is known that (see the Appendix 2 \eqref{eq-qt-ab-0}): 
\begin{multline}\label{CirJac} 
q_t^{(n-1, \mu+1)}(0,r) = \frac{2}{\Gamma(2n)}[\cos(r)]^{2\mu+3}[\sin(r)]^{4n-1} \\ \sum_{j \geq 0}(2j+2n+\mu+1) e^{-2j(j+2n+\mu+1)t} \frac{\Gamma(j+2n+\mu+1)}{\Gamma(j+\mu+2)} P_j^{(2n-1, \mu+1)}(\cos(2r)). 
\end{multline}
Expanding the Jacobi polynomials 
\begin{align*}
P_j^{(2n-1, \mu+1)}(\cos(2r)) &= \frac{(2n)_j}{j!} {}_2F_1\left(-j, j+2n+\mu+1, 2n; \frac{1-\cos(2r)}{2}\right)
\\&= \frac{(2n)_j}{j!} \sum_{m=0}^j\frac{(-j)_m(j+2n+\mu+1)_m}{(2n)_mm!} \sin^{2m}(r), 
\end{align*}
where ${}_2F_1$ is the Gauss hypergeometric function, we are led (by the virtue of Corollary \ref{cor-theta-S}) to the following Beta integral: 
\begin{equation*}
\int_0^{\pi/2} \cos^{\mu+3}(r)[\sin(r)]^{4n-1+2m} dr = \frac{\Gamma((\mu/2)+2)\Gamma(2n+m)}{2\Gamma(2n+2+m+\mu/2)}.
\end{equation*}
Consequently, 
\begin{align*}
\int_0^{\pi/2} \cos^{\mu+3}(r)[\sin(r)]^{4n-1} P_j^{(2n-1, \mu+1)}(\cos(2r)) dr  &= \Gamma\left(\frac{\mu}{2}+2\right) \frac{\Gamma(2n+j)}{2j!}.
\end{align*}
We also obtain that
\begin{align*}
\sum_{m=0}^j\frac{(-j)_m(j+2n+\mu+1)_m}{\Gamma(2n+2+m+\mu/2) m!} &= \frac{\Gamma((\mu/2)+2)}{\Gamma(2n+2+\mu/2)} \frac{\Gamma(2n+j)}{2j!}, 
\end{align*}
\begin{align*}
{}_2F_1\left(-j, j+2n+\mu+1, 2n + 2 + \frac{\mu}{2}; 1\right) &= \frac{\Gamma((\mu/2)+2)}{\Gamma(2n+2+\mu/2)} \frac{\Gamma(2n+j)}{2j!}, 
\end{align*}
\begin{align*}
\frac{j!}{(2n+2+\mu/2)_j}P_j^{(2n+1+\mu/2, \mu/2-1)}(-1)
 &= \frac{\Gamma((\mu/2)+2)}{\Gamma(2n+2+\mu/2)} \frac{\Gamma(2n+j)}{2j!} \frac{(-1)^j j!}{(2n+2+\mu/2)_j}, 
 \end{align*}
 and
 \begin{align*}
 P_j^{(\mu/2-1, 2n+1+\mu/2)}(1)
&=  \frac{\Gamma((\mu/2)+2)}{\Gamma(2n+2+\mu/2)} \frac{\Gamma(2n+j)}{2j!} \frac{(-1)^j (\mu/2)_j}{(2n+2+\mu/2)_j}\\
&= \frac{\mu(\mu+2)}{4} \frac{\Gamma(2n+j)}{2j!} \frac{(-1)^j \Gamma(j+\mu/2)}{\Gamma(j+2n+2+\mu/2)},
\end{align*}
where we used the symmetry relation $P_j^{(a,b)}(-u) = (-1)^jP_j^{(b,a)}(u)$ and the special value 
\begin{equation*}
P_j^{(a,b)}(1) =\frac{(a+1)_j}{j!}.
\end{equation*}
Keeping in mind \eqref{CirJac}, the corollary is proved. 
\end{proof}
\begin{remark}
At the end of the proof, we simplified with $\Gamma(\mu/2)$ and this is allowed when $\mu \neq 0 \leftrightarrow \lambda \neq 0$. When $\mu = 0$, our computations remain valid and should be understood as a limit when $\mu \rightarrow 0$. 
In this case, the only non vanishing term corresponds to $j=0$ so that 
\begin{equation*}
\lim_{\mu \rightarrow 0} (2n+\mu+1)\frac{(\mu/2)((\mu/2)+1)\Gamma(\mu/2)\Gamma(2n+\mu+1)}{\Gamma(2n+2+\mu/2)}= 1.
\end{equation*}
\end{remark}

Now, in order to invert the characteristic function and recover the density of $\A(t)$, it suffices to express: 
\begin{equation*}
(2j+2n+\mu+1) e^{-2\mu(j+n)t} \frac{\mu(\mu+2)}{4}\frac{\Gamma(j+\mu/2)\Gamma(j+2n+\mu+1)}{\Gamma(j+\mu/2+2n+2)\Gamma(j+\mu+2)},
\end{equation*}
as a Fourier transform in $\lambda$, where we recall the relation $\mu+1 = \sqrt{|\lambda|^2+1}$. To proceed, we write for $j \geq 1$:
\begin{equation*}
\frac{\Gamma(j+\mu/2)\Gamma(j+2n+\mu+1)}{\Gamma(j+\mu/2+2n+2)\Gamma(j+\mu+2)} = \frac{(j+\mu+2n)\cdots(j+\mu+2)}{(j+(\mu/2)+2n)\cdots(j+\mu/2)},
\end{equation*}
so that
\begin{equation}\label{Coeffic1}
\frac{(2j+2n+\mu+1)\mu(\mu+2)(j+\mu+2n)\cdots(j+\mu+2)}{4(j+(\mu/2)+2n+1)\cdots(j+\mu/2)}  = 2^{2n} + \sum_{k=0}^{2n+1} \frac{a_{k,n}(j)}{\mu+2j+2k} 
\end{equation}
for some real coefficients $a_{k,n}(j)$. For $j=0$, the same decomposition holds: 
\begin{align}
&\frac{\mu(\mu+2)}{4}\Gamma(\mu/2)\frac{(2n+\mu+1)\Gamma(2n+\mu+1)}{\Gamma(\mu/2+2n+2)\Gamma(\mu+2)}  = \frac{\Gamma((\mu/2)+2)\Gamma(2n+\mu+2)}{\Gamma(\mu/2+2n+2)\Gamma(\mu+2)} \nonumber
\\
&\quad\quad = 2^{2n} + \sum_{k=1}^{2n+1} \frac{a_{k,n}(0)}{\mu+2k} 
 = 2^{2n} + \sum_{k=0}^{2n+1} \frac{a_{k,n}(0)}{\mu+2k}, \label{Coeffic2}
\end{align}
with $a_{0,n}(0) = 0$. Finally, using the integral
\begin{equation*}
\frac{1}{\mu+ 2j+2k} = \int_0^{\infty} e^{-u\mu} e^{-2(j+k)u} du, 
\end{equation*}
followed by the Fourier transform of the $3$-dimensional relativistic Cauchy distribution (see \cite{BMR09}, Lemma 2.1 with $d=3, m=1$):
\begin{multline*}
e^{-\mu[u+(2j+2n)t]}  = 2[u+(2j+2n)t] e^{[u+(2j+2n)t]} \left(\frac{1}{2\pi} \right)^2 \\ \int_{\mathbb{R}^3}e^{i\lambda \cdot x}\frac{K_2\left(\sqrt{|x|^2+[u+(2j+2n)t]^2 }\right)}{|x |^2+[u+(2j+2n)t]^2 } dx,
\end{multline*}
we arrive at the following expression for the density of $\A(t)$ (note that the modified Bessel function $K_{\nu}(v)$ is equivalent to $\sqrt{\pi/(2v)}e^{-v}$ at infinity and that that the coefficients $(a_k(j), k =0, \dots, 2n+2)$ are polynomials in $j$ whose degrees are uniformly bounded by $2n+2$, so that we can use Fubini Theorem to interchange the order of integration).
\begin{theorem}
The density of the quaternionic stochastic area process $\A(t)$ is given by the following absolutely-convergent series: 
\begin{multline*}
\frac{1}{2\pi^2}\sum_{j \geq 0} (-1)^j \frac{(2n)_j}{j!} e^{-2j(j+2n+1)t} \\ 
\left\{2^{2n}(2j+2n)e^{(2j+2n)t}\frac{K_2\left(\sqrt{|x|^2+(2j+2n)^2t^2 }\right)}{|x |^2+(2j+2n)^2t^2 }+ \right.
\\ \left. \sum_{k=0}^{2n+1} a_{k,n}(j)\int_0^{\infty} e^{-2(j+k)u}[u+(2j+2n)t] e^{[u+(2j+2n)t]}\frac{K_2\left(\sqrt{|x|^2+[u+(2j+2n)t]^2 }\right)}{|x |^2+[u+(2j+2n)t]^2}du\right\},
\end{multline*}
where $x \in \mathbb{R}^3$ and the coefficients $(a_{k,n}(j))$ are defined by \eqref{Coeffic1} and \eqref{Coeffic2}.
\end{theorem} 
\begin{remark}
It is clear that the limiting behavior of the density of $\A(t)/\sqrt{t}$ is given by the first term $j=0$ in the above series. Moreover, the equivalence  
\begin{equation*}
K_{\nu}(v) \sim \sqrt{\frac{\pi}{2v}}e^{-v}, \quad v \rightarrow +\infty,
\end{equation*}
shows that for any $u \geq 0$, 
\begin{align*}
\lim_{t \rightarrow \infty} t^{3/2} [u+2nt] e^{(u+2nt)}\frac{K_2\left(\sqrt{t|x|^2+(u+2nt)^2}\right)}{t|x |^2+(u+2nt)^2} 
 = \sqrt{\frac{\pi}{2}} \frac{1}{(2n)^{3/2}} e^{-|x|^2/(4n)}. 
\end{align*}
Consequently, the density of $\A(t)/\sqrt{t}$ converges as $t \rightarrow \infty$ to 
\begin{equation*}
\frac{1}{(4\pi n)^{3/2}} e^{-|x|^2/(4n)}  \left\{2^{2n} + \sum_{k=0}^{2n+1} a_{k,n}(j)\int_0^{\infty} e^{-2ku}du\right\}.
\end{equation*}
But, substituting $\mu =0$ in \eqref{Coeffic2}, we get 
\begin{equation*}
2^{2n} + \sum_{k=0}^{2n+1} \frac{a_{k,n}(0)}{2k} = 1,
\end{equation*}
whence we recover that (using for instance Scheff\'e's Lemma) $\A(t)/\sqrt{t}$ converges in distribution to a $3$-dimensional normal distribution of covariance matrix $2n \mathrm{I}_3$, as shown in Theorem \ref{limit CP}.
\end{remark}

\section{The horizontal heat kernel on the twistor space of  $\mathbb{H}P^{n}$}

We now turn to the second main object of our study.
In this section we study the subelliptic heat kernel of the Riemannian submersion  $\mathbb{C}P^{2n+1} \to\mathbb{H}P^n$ which is induced from the quaternionic Hopf fibration. 
\subsection{Projected Hopf fibration}

Besides the action of  $\mathbf{SU}(2)$ on $\mathbb{S}^{4n+3}$ that induces the quaternionic Hopf fibration which was studied in the previous sections, we can also consider the action of $\mathbb{S}^1$ on $\mathbb{S}^{4n+3}$ that induces the complex Hopf fibration:
\begin{equation*}
\mathbb{S}^1 \to \mathbb{S}^{4n+3}  \to  \mathbb{C}P^{2n+1}.
\end{equation*}
We can then see $\mathbb{S}^1$ as a subgroup of $\mathbf{SU}(2)$ and deduce a a fibration
\begin{equation*}
\mathbf{SU}(2)  / \mathbb{S}^1 =  \mathbb{C}P^1 \to \mathbb{C}P^{2n+1} \to \mathbb{H}P^n
\end{equation*}
that makes the following diagram commutative

\[
  \begin{tikzcd}
    & \mathbb{S}^1 \arrow[swap]{dl} \arrow{d} & \\
    \mathbf{SU}(2) \arrow{r} \arrow{d}  & \mathbb{S}^{4n+3}   \arrow{d} \arrow[swap]{dr} &  \\
   \mathbb{C}P^1  \arrow{r} & \mathbb{C}P^{2n+1} \arrow{r} & \mathbb{H}P^n  
  \end{tikzcd}
\]

We consider the sub-Laplacian ${\mathcal{L}}$ on  $\mathbb{C}P^{2n+1}$ which is the lift of the Laplace-Beltrami operator of  $\mathbb{H}P^n$. From the above diagram, ${\mathcal{L}}$ is also the projection of the sub-Laplacian $L$ of $\mathbb{S}^{4n+3} $ on $\mathbb{C}P^{2n+1}$. As we have seen before,  the radial part of $L$ is 
\[
\tilde{L}=\frac{\partial^2}{\partial r^2}+((4n-1)\cot r-3\tan r)\frac{\partial}{\partial r}+\tan^2r (\frac{\partial^2}{\partial \eta^2}+2\cot \eta\frac{\partial}{\partial \eta})
\]
where $r$ is the Riemannian distance on $\mathbb{H}P^n$ and $\eta$ the Riemannian distance on $ \mathbf{SU}(2) $. The operator
\[
\tilde{\Delta}_{SU(2)}=\frac{\partial^2}{\partial \eta^2}+2\cot \eta\frac{\partial}{\partial \eta}
\]
is the radial part of the Laplace-Beltrami operator on $ \mathbf{SU}(2) $. As it has been proved in Baudoin-Bonnefont (see \cite{BB}), by using the fibration,
\begin{align*}
\mathbb{S}^1\to \mathbf{SU}(2)  \to  \mathbb{C}P^{1}
\end{align*}
we can write
\[
\frac{\partial^2}{\partial \eta^2}+2\cot \eta\frac{\partial}{\partial \eta}=\frac{\partial^2}{\partial \phi^2}+2\cot 2\phi\frac{\partial}{\partial \phi}+(1+\tan^2\phi )\frac{\partial^2}{\partial \theta^2}
\]
where $\phi$ is the Riemannian distance on $\mathbb{C}P^{1}$ and $\frac{\partial}{\partial \theta}$ the generator of the action of $\mathbb{S}^1$ on $\mathbf{SU}(2)$. We deduce that $\tilde{L}$ can be written as
\[
\frac{\partial^2}{\partial r^2}+((4n-1)\cot r-3\tan r)\frac{\partial}{\partial r}+\tan^2r\left( \frac{\partial^2}{\partial \phi^2}+2\cot 2\phi\frac{\partial}{\partial \phi}+(1+\tan^2\phi )\frac{\partial^2}{\partial \theta^2}\right).
\]
Therefore the radial part of the sub-Laplacian $\mathcal{L}$ on  $\mathbb{C}P^{2n+1}$ is given by 
\begin{equation}\label{L-CP}
\tilde{\mathcal{L}}=\frac{\partial^2}{\partial r^2}+((4n-1)\cot r-3\tan r)\frac{\partial}{\partial r}+\tan^2r\left( \frac{\partial^2}{\partial \phi^2}+2\cot 2\phi\frac{\partial}{\partial \phi}\right).
\end{equation}

We can also see , and this will be later used, that the Riemannian measure reads up to a normalization constant $(\sin r)^{4n-1} (\cos r)^3 \sin 2\phi dr d\phi$.
\subsection{Spectral decomposition of the subelliptic heat kernel on $\mathbb{C}P^{2n+1}$}
We now study the subelliptic heat kernel $h_t(r,\phi)$ associated to $\mathcal{L}$. Similarly as in Proposition \ref{spectral}, we get the spectral decomposition of $h_t$.
\begin{proposition}
For $t>0$, $r\in[0,\frac{\pi}{2})$, $ \phi\in[0,\pi]$, the subelliptic kernel is given by
\[
h_t(r,\phi)=\sum_{m=0}^{+\infty}\sum_{k=0}^{\infty}\sigma_{k,m}e^{-[4k(k+2n+2m+1)+2nm]t}(\cos r)^{2m}P_m^{0,0}(\cos2\phi)P_k^{2n-1,2m+1}(\cos 2r),
\]
where $\sigma_{k,m}=\frac{\Gamma(2n)}{4\pi^{2n+2}}(2k+2m+2n+1)(2m+1){k+2m+2n\choose 2n-1}$.
\end{proposition}
\begin{proof}
From \eqref{L-CP} we notice that $\frac{\partial^2}{\partial \phi^2}+2\cot 2\phi\frac{\partial}{\partial \phi}$ is the Laplacian on $\mathbb{C}P^1$.  It is known that the eigenfunction associated to the eigenvalue $-4m(m+1)$ is given by $P_m^{0,0}(\cos2\phi)$ where
\[
P_m^{0,0}(x)=\frac{(-1)^m}{2^m m!}\frac{d^m}{dx^m}(1-x^2)^{m},
\]
thus we can write the spectral decomposition of $h_t(r,\phi)$ as follows:
\[
h_t(r,\phi)=\sum_{m=0}^{+\infty}P_m^{0,0}(\cos2\phi)\phi_m(t,r).
\]
Plug it into the heat equation $\frac{\partial h_t}{\partial t}=\mathcal{L}h_t$, we obtain that
\[
\frac{\partial\phi_m}{\partial t}=\frac{\partial^2\phi_m}{\partial r^2}+\left((4n-1)\cot r-3\tan r \right)\frac{\partial\phi_m}{\partial r}-4m(m+1)\tan^2r\phi_m.
\]
Furthermore we consider the decomposition $\phi_m(t,r)=e^{-8nmt}(\cos r)^{2m}\varphi_m(t,r)$ where $\varphi_m(t,r)$ satisfies 
\[
\frac{\partial\varphi _m}{\partial t}=4\frac{\partial^2\varphi_m}{\partial r^2}+[(4n-1)\cot r-(4m+3)\tan r]\frac{\partial\varphi_m}{\partial r},
\]
and let $g_m(t,\cos 2r)=\varphi_m(t,r)$, then $g_m(t,x)$ is such that
\[
\frac{\partial g_m}{\partial t}=4(1-x^2)\frac{\partial^2 g_m}{\partial x^2}+4[(2m+2-2n)-(2n+2m+2)x]\frac{\partial g_m}{\partial x}.
\]
We denote 
\[
\Psi_m=(1-x^2)\frac{\partial^2 }{\partial x^2}+4[(2m+2-2n)-(2n+2m+2)x]\frac{\partial}{\partial x},
\]
then
\[
\frac{\partial g_m}{\partial t}=4\Psi_m(g_m).
\]
Moreover, for any $k\geq 0$, $$\Psi_m(g_m)+k(k+2n+2m+1)g_m=0$$  is the Jacobi differential equation whose eigenvector corresponding to the eigenvalue $-k(k+2n+2m+1)$  is given by 
\[
P_k^{2n-1,2m+1}(x)=\frac{(-1)^k}{2^kk!(1-x)^{2n-1}(1+x)^{2m+1}}\frac{d^k}{dx^k}\left((1-x)^{2n-1+k}(1+x)^{2m+1+k} \right).
\] 
Thus the spectral decomposition of $h_t(r,\phi)$ has the following form:
\[
h_t(r,\phi)=\sum_{m=0}^{+\infty}\sum_{k=0}^{\infty}\sigma_{k,m}e^{-[4k(k+2n+2m+1)+2nm]t}(\cos r)^{2m}P_m^{0,0}(\cos2\phi)P_k^{2n-1,2m+1}(\cos 2r),
\]
and by plugging in the initial condition, one can determine that  $\sigma_{k,m}=\frac{\Gamma(2n)}{4\pi^{2n+2}}(2k+2m+2n+1)(2m+1){k+2m+2n\choose 2n-1}$.
\end{proof}

We can compare the spectral decompositions of $h_t$ and $p_t$ and obtain the relations between $h_t$ and $p_t$ as follows.
\begin{proposition}
For $t>0$, $r\in[0,\frac{\pi}{2})$, $ \phi\in[0,\pi]$,
\begin{equation}\label{htpt}
h_t(r,\cos2\phi)=\frac{1}{2\pi}\int_0^\pi p_t(r,\cos\phi\cos\theta)d\theta.
\end{equation}
\end{proposition}
\begin{proof} 
Let us use the expression of $p_t$ in \eqref{pt} and recall (see Proposition \ref{heat kernel entrelacement}) that
\[
P_m^{0,0}(\cos2\phi)=\frac{1}{\pi}\int_0^\pi U_{2m+1}(\cos\phi\cos\theta)d\theta
\]
where $U_{2m+1}$ is the Chebyshev polynomial
\[
U_{2m+1}(\cos\phi)=\frac{\sin (2m+1)\phi}{\sin\phi}.
\]
Since $\cos\eta=\cos\phi\cos\theta$, we have that
\begin{align*}
h_t(r,\phi)=\sum_{m=0}^{+\infty}\sum_{k=0}^{\infty}&\sigma_{k,m}e^{-[4k(k+2n+2m+1)+2nm]t}(\cos r)^{2m}\left(\frac{1}{\pi}\int_0^\pi U_{2m+1}(\cos\phi\cos\theta)d\theta\right)\\
&\quad\quad \cdot P_k^{2n-1,2m+1}(\cos 2r),
\end{align*}
which implies the desired relation between the subelliptic heat kernels. The even terms on the right hand side of \eqref{htpt} vanish due to the fact that
\[
p_t(r+\pi,\eta)=p_t(r,\eta).
\]
\end{proof}

\chapter{Horizontal Brownian motions of the quaternionic anti-de Sitter fibrations}\label{chap-quater-anti}

\section{Stochastic area process on the quaternionic hyperbolic space $\bH H^n$}\label{sec-geometry}

\subsection{Quaternionic anti-de Sitter fibration}\label{sec-qua-ads-fib}

As before the quaternionic field is defined by
\[
\mathbb{H}=\{q=t+xI+yJ+zK, (t,x,y,z)\in\R^4\},
\]
where  $I,J,K \in \mathfrak{su}(2)$ are given by
\[
I=\left(
\begin{array}{ll}
i & 0 \\
0&-i 
\end{array}
\right), \quad 
J= \left(
\begin{array}{ll}
0 & 1 \\
-1 &0 
\end{array}
\right), \quad 
K= \left(
\begin{array}{ll}
0 & i \\
i &0 
\end{array}
\right).
\]
We will use similar notations as in Chapter \ref{Chapter quaternionic hopf}.

As  differentiable manifolds, the quaternionic  hyperbolic space $\mathbb{H}H^n$ can be defined as the unit ball in $\mathbb{H}^{n+1}$ and the quaternionic anti-de Sitter space $\mathbf{AdS}^{4n+3}(\mathbb{H})$ is defined as the quaternionic pseudo-hyperboloid:
\[
\mathbf{AdS}^{4n+3}(\mathbb{H})=\lbrace q=(q_1,\dots,q_{n+1})\in \mathbb{H}^{n+1}, {|q|}^2_H =-1\rbrace,
\]
where 
\[
{|q|}_H^2 := \sum_{k=1}^{n}|q_k|^2-|q_{n+1}|^2.
\] 
We note that ${| \cdot |}_H$ induces on $\mathbf{AdS}^{4n+3}(\mathbb{H})$ an indefinite Riemannian metric with signature $(4n,3)$. To parametrize points in $\mathbb{H}H^n$, we will use the global affine coordinates given by $w_j= q_jq_{n+1}^{-1}$, $1 \le j \le n$, $q \in \mathbf{AdS}^{4n+3}(\mathbb{H})$.  In affine coordinates, the Riemannian structure of $\mathbb{H}H^n$ can be worked out from the indefinite Riemannian metric on $\mathbf{AdS}^{4n+3}(\mathbb{H})$ as follows. Indeed, if we consider the map 
\begin{align*}
\begin{cases}
\mathbf{AdS}^{4n+3}(\mathbb{H}) \to \mathbb{H}H^n \\
(q_1,\dots,q_{n+1}) \to ( q_1q_{n+1}^{-1},\dots,  q_n q_{n+1}^{-1})
\end{cases}
\end{align*}
the unique Riemannian metric $h$ on $\mathbb{H}H^n$ that makes this map a Riemannian submersion is the standard Riemannian metric on $\mathbb{H}H^n$. This metric can be characterized as follows.

\begin{proposition}\label{metric HHn}
If $g_1,g_2$ are smooth functions on  $\mathbb{H}H^n$, then in affine coordinates
\begin{equation}\label{eq-co-metric ads}
   h(dg_1,dg_2)=2(1-|w|^2)\left( \sum_{i=1}^n \left( \frac{\partial g_1}{\partial \overline{w}_i}\frac{\partial g_2}{\partial w_i}+\frac{\partial g_2}{\partial \overline{w}_i}\frac{\partial g_1}{\partial w_i} \right) -\overline{\mathcal{R}} g_1 \mathcal{R} g_2 -\overline{\mathcal{R}} g_2 \mathcal{R} g_1 \right),  
\end{equation}
 where
 \[
 \mathcal{R} g = \sum_{i=1}^n \frac{\partial g}{\partial w_i} w_i.
 \]
 \end{proposition}

\begin{proof}
We use the change of variables formula $w_j= q_jq_{n+1}^{-1}$. Let  $f$ be a (real-valued) smooth function on $\mathbb H^{n+1}$ such that there exists a function $g$ on $\mathbb{H}H^n$ satisfying
\[
f(q_1,\dots,q_{n+1})=g \left( q_1q_{n+1}^{-1},\dots, q_nq_{n+1}^{-1}\right), \quad q \in \mathbb H^{n+1} .
\]
By definition of the metric $h$ one has for $q \in \mathbf{AdS}^{4n+3}(\mathbb{H})$

\[
h(dg,dg)\left( w \right)=4 \sum_{i=1}^{n}\frac{\partial f}{\partial \overline{q}_i} (q)\frac{\partial f}{\partial q_i} (q)- 4 \frac{\partial f}{\partial \overline{q}_{n+1}} (q)\frac{\partial f}{\partial q_{n+1}} (q).
\]

We first compute $\frac{\partial f}{\partial q_i} $ for $1 \le i \le n$. Using the fact that $w_i= q_iq_{n+1}^{-1}$ one gets that for $1 \le i \le n$
\[
\frac{\partial f}{\partial q_i} (q)= q_{n+1}^{-1} \frac{\partial g}{\partial w_i} (w).
\]

To compute $\frac{\partial f}{\partial q_{n+1}} $ we need to use the symmetries of the function $f$. Using the fact that
\begin{align*}
\begin{cases}
\frac{d}{ds}f(e^{s}q_1,\dots, e^{s}q_{n+1} )\mid_{s=0} =0 \\
\frac{d}{ds}f( q_1e^{sI},\dots, q_{n+1}e^{sI})\mid_{s=0} =0 \\
\frac{d}{ds}f( q_1e^{sJ},\dots, q_{n+1}e^{sJ})\mid_{s=0} =0 \\
\frac{d}{ds}f( q_1e^{sK},\dots, q_{n+1}e^{sK})\mid_{s=0} =0
\end{cases}
\end{align*}
we get
\begin{align*}
\begin{cases}
\sum_{i=1}^{n+1} t_i \frac{\partial f}{\partial t_i} + x_i \frac{\partial f}{\partial x_i}+ y_i \frac{\partial f}{\partial y_i}+ z_i \frac{\partial f}{\partial z_i} =0 \\
\sum_{i=1}^{n+1} -x_i \frac{\partial f}{\partial t_i} + t_i \frac{\partial f}{\partial x_i}+ z_i \frac{\partial f}{\partial y_i}- y_i \frac{\partial f}{\partial z_i} =0 \\
\sum_{i=1}^{n+1} -y_i \frac{\partial f}{\partial t_i} - z_i \frac{\partial f}{\partial x_i}+ t_i \frac{\partial f}{\partial y_i}+ x_i \frac{\partial f}{\partial z_i} =0 \\
\sum_{i=1}^{n+1} -z_i \frac{\partial f}{\partial t_i} + y_i \frac{\partial f}{\partial x_i}- x_i \frac{\partial f}{\partial y_i}+ t_i \frac{\partial f}{\partial z_i} =0=0
\end{cases}
\end{align*}
This can be more concisely written in quaternionic form as 
\[
\sum_{i=1}^{n+1} \frac{\partial f}{\partial q_i} q_i =0
\]
In particular, one obtains therefore
\[
 \frac{\partial f}{\partial q_{n+1}}=-\sum_{i=1}^{n} \frac{\partial f}{\partial q_i} q_i q_{n+1}^{-1}.
 \]
 This yields
 \[
  \frac{\partial f}{\partial q_{n+1}}=-q_{n+1}^{-1} \sum_{i=1}^{n} \frac{\partial g}{\partial w_i} w_i.
 \]
 One conclude
  \begin{align*}
   \sum_{i=1}^{n}\frac{\partial f}{\partial \overline{q}_i} (q)\frac{\partial f}{\partial q_i} (q)-  \frac{\partial f}{\partial \overline{q}_{n+1}} (q)\frac{\partial f}{\partial q_{n+1}} (q)=\frac{1}{|q_{n+1}|^2}  \left( \sum_{i=1}^n \frac{\partial g}{\partial \overline{w}_i}\frac{\partial g}{\partial w_i} + \overline{\mathcal{R}} g\mathcal{R} g \right),
 \end{align*}
where
 \[
 \mathcal{R} g = \sum_{i=1}^n \frac{\partial g}{\partial w_i} w_i.
 \]
 Since on $\mathbf{AdS}^{4n+3}(\mathbb{H})$ one has $\frac{1}{|q_{n+1}|^2} =1-|w|^2$, one concludes
 \begin{equation}\label{eq-h-quad 2}
 h(dg,dg)=4(1-|w|^2)\left( \sum_{i=1}^n \frac{\partial g}{\partial \overline{w}_i}\frac{\partial g}{\partial w_i} -\overline{\mathcal{R}} g\mathcal{R} g  \right)
 \end{equation}
 and the conclusion follows by polarization.
\end{proof}

It will be more convenient to use the above formula for the metric in its real form. For this, we use the real affine coordinates on $\mathbb{H}H^n$ defined by
\[
w_i=\tau_i+\alpha_i I +\beta_i J +\gamma_i K.
\]

From the expression of Proposition \ref{metric HHn}, it is straightforward to adapt the arguments given in Section \ref{Quaternionic Hopf} to get the following results that we therefore state without proof.

\begin{corollary}
If $g_1,g_2$ are smooth functions on  $\mathbb{H}H^n$, then in real affine coordinates
\begin{align}\label{eq-co-metric2}
   h(dg_1,dg_2) =(1-|w|^2) \left( \sum_{i=1}^n \left( \frac{\partial g_1}{\partial \tau_i}\frac{\partial g_2}{\partial \tau_i}+\frac{\partial g_1}{\partial \alpha_i}\frac{\partial g_2}{\partial \alpha_i} +\frac{\partial g_1}{\partial \beta_i}\frac{\partial g_2}{\partial \beta_i}+\frac{\partial g_1}{\partial \gamma_i}\frac{\partial g_2}{\partial \gamma_i} \right)  -\sum_{i=1}^4 \mathcal{R}_i  g_1 \mathcal{R}_i g_2 \right),  
\end{align}
 where
\[
 \mathcal{R}_1 g = \sum_{i=1}^n \tau_i \frac{\partial g}{\partial \tau_i} +\alpha_i \frac{\partial g}{\partial \alpha_i}+\beta_i \frac{\partial g}{\partial \beta_i}+\gamma_i \frac{\partial g}{\partial \gamma_i}
 \]
 \[
 \mathcal{R}_2 g = \sum_{i=1}^n \alpha_i \frac{\partial g}{\partial t_i} -\tau_i \frac{\partial g}{\partial \alpha_i}-\gamma_i \frac{\partial g}{\partial \beta_i}+\beta_i \frac{\partial g}{\partial \gamma_i}
 \]
 \[
 \mathcal{R}_3 g = \sum_{i=1}^n \beta_i \frac{\partial g}{\partial t_i} +\gamma_i \frac{\partial g}{\partial \alpha_i}-\tau_i \frac{\partial g}{\partial \beta_i}-\alpha_i \frac{\partial g}{\partial \gamma_i}
 \]
 \[
 \mathcal{R}_4 g = \sum_{i=1}^n \gamma_i \frac{\partial g}{\partial t_i} -\beta_i \frac{\partial g}{\partial \alpha_i}+\alpha_i \frac{\partial g}{\partial \beta_i}-\tau_i \frac{\partial g}{\partial \gamma_i}.
 \]
\end{corollary}

\begin{corollary}
 In real affine coordinates, the matrix of the Riemannian cometric   is given by the $4n \times 4n$ matrix
 \[
 (1-|w|^2) \left(\mathrm{I}_{4n}-\mathrm{R}_1 \mathrm{R}^*_1-\mathrm{R}_2 \mathrm{R}^*_2-\mathrm{R}_3 \mathrm{R}^*_3-\mathrm{R}_4 \mathrm{R}^*_4 \right),
 \]
 where $\mathrm{R}_i$ is the real valued column vector with the $4n$ components
 \[
 \mathrm{R}_i =\left( \mathcal{R}_i \tau_j , \mathcal{R}_i \alpha_j, \mathcal{R}_i \beta_j,\mathcal{R}_i \gamma_j\right)_{1 \le j \le n}.
 \]
\end{corollary}

\begin{proof}
This readily follows from \eqref{eq-co-metric2}.
\end{proof}

\begin{corollary}
 In affine coordinates the Riemannian volume measure of $\mathbb H H^n$ is given by
\[
d\mu(w)=(1-|w|^2)^{-2(n+1)} dw.
\]
\end{corollary}

\begin{proof}
   The proof is similar to that of Corollary \ref{volume HPn} and let to the reader. 
\end{proof}

\begin{theorem}\label{Laplace quaternionic hyperbolic}
In real affine coordinates, the Laplace-Beltrami operator of $\mathbb H H^n$ is given by 
\[
\Delta_{\mathbb H H^n}=(1-|w|^2) \left( \sum_{i=1}^n \left( \frac{\partial^2}{\partial \tau^2_i}+\frac{\partial^2}{\partial \alpha^2_i}+\frac{\partial^2}{\partial \beta^2_i}+\frac{\partial^2 }{\partial \gamma^2_i} \right)  -\sum_{i=1}^4 \mathcal{R}^2_i  +2\mathcal{R}_1  \right)
\]
\end{theorem}

We note that there is an isometric group action on the right of the group $\mathbf{SU}(2)$ on $\mathbf{AdS}^{4n+3}(\mathbb{H})$ which is  defined by $$g \cdot(q_1,\dots, q_{n+1}) = (  q_1 g,\dots,  q_{n+1} g), $$
where $\mathbf{SU}(2)$ is identified with the set of unit quaternions.

The quotient space $\mathbf{AdS}^{4n+3}(\mathbb{H}) / \mathbf{SU}(2)$ can be identified with $\mathbb{H}H^n$ and the projection map $$\ \mathbf{AdS}^{4n+3}(\mathbb{H}) \to \mathbb{H}H^n$$ is a (semi) Riemannian submersion with totally geodesic fibers isometric to $\mathbf{SU}(2)$. The fibration
\[
\mathbf{SU}(2) \to \mathbf{AdS}^{4n+3}(\mathbb{H}) \to \mathbb{H}H^n
\]
 is called the quaternionic anti-de Sitter fibration. This fibration is another example of a B\'erard-Bergery fibration (see Example \ref{BB fibration}). 
 
 From Theorem \ref{Laplace quaternionic hyperbolic}, the generator of the Brownian motion $(w(t))_{t \ge 0}$ on $\mathbb{H}H^n$ is given by
\[
\frac{1}{2}(1-|w|^2) \left( \sum_{i=1}^n \left( \frac{\partial^2}{\partial \tau^2_i}+\frac{\partial^2}{\partial \alpha^2_i}+\frac{\partial^2}{\partial \beta^2_i}+\frac{\partial^2 }{\partial \gamma^2_i} \right)  -\sum_{i=1}^4 \mathcal{R}^2_i  +2\mathcal{R}_1  \right).
\]

\begin{proposition}
The process $r(t)=\arctanh |w(t)|$ is a hyperbolic Jacobi diffusion with generator
\[
\frac{1}{2} \left( \frac{\partial^2}{\partial r^2}+((4n-1)\coth r+3\tanh r)\frac{\partial}{\partial r}\right).
\]
In particular, the density of $r(t)$, $t>0,$ with respect to the Lebesgue measure admits the following representation (see \cite{Dem-Heat}):
\[
\mathbb{P} (r(t) \in dr)= C_n \frac{e^{-(2n+1)^2t/2}}{\cosh(r)} \left( \int_0^{\infty} y^{2n-1}u(t,y) K_1(y\cosh(r)) dy \right)\,\sinh^{4n-1}(r) \cosh^3(r) \, dr,
\]
where $C_n$ is a normalizing constant, $K_1$ is the modified Bessel function of the second kind, and $u(t,y)$ is the density of the Hartman-Watson distribution at time $t$. 
\end{proposition}

\begin{proof}
Let $r=\arctanh |w|$. A direct  computation shows that the operator 
\[
(1-|w|^2) \left( \sum_{i=1}^n \left( \frac{\partial^2}{\partial \tau^2_i}+\frac{\partial^2}{\partial \alpha^2_i}+\frac{\partial^2}{\partial \beta^2_i}+\frac{\partial^2 }{\partial \gamma^2_i} \right)  -\sum_{i=1}^4 \mathcal{R}^2_i  +2\mathcal{R}_1  \right)
\]
acts on functions depending only on $r$ as
\[
 \frac{\partial^2}{\partial r^2}+((4n-1)\coth r+3\tanh r)\frac{\partial}{\partial r}
 \]
It\^o's formula therefore shows  that  $r(t)=\arctan |w(t)|$ is a hyperbolic Jacobi diffusion with generator
\[
\frac{1}{2} \left(  \frac{\partial^2}{\partial r^2}+((4n-1)\coth r+3\tanh r)\frac{\partial}{\partial r} \right).
\]
\end{proof}

 \begin{remark}
See Section \ref{section hyp Jacobi} (Formula \ref{heat kernel HHn}) for an alternative formula of the heat kernel of the diffusion with generator
\[
\frac{1}{2} \left( \frac{\partial^2}{\partial r^2}+((4n-1)\coth r+3\tanh r)\frac{\partial}{\partial r}\right).
\]
 \end{remark}
 
Once again,  to use the general framework described in Section \ref{horizontal BM bundle}, we now compute a connection form of the quaternionic anti-de Sitter space.  The proof of the theorem follows very closely that of Theorem \ref{connection quater fib}, we therefore let it to the reader.

\begin{theorem}\label{connection quaternio ads}
Consider on $\mathbf{AdS}^{4n+3}(\mathbb H) \subset \mathbb{H}^{n+1}$ the $\mathfrak{su}(2)$ valued one-form

\[
\theta=\frac{1}{2}\left(\sum_{i=1}^{n}  \overline{q}_i  dq_i  - d\overline{q}_iq_i - (\overline{q}_{n+1}  dq_{n+1}  - d\overline{q}_{n+1}  \, q_{n+1}  ) \right) .
\]
We have:
\begin{itemize}
\item \hspace{.1in}  $g^* \theta =\mathbf{Ad}_{g^{-1}} \circ \theta$;
\item \hspace{.1in} For any $V \in \mathfrak{su}(2)$,
\[
\theta (\hat{V})=-V,
\]
where $\hat V$ is the vector field on $\mathbf{AdS}^{4n+3}(\mathbb H)$ defined by
\begin{align*}
\hat V f (q)= \lim_{t \to 0} \frac{f( e^{tV}\cdot q )-f(q)}{t}, \quad q \in \mathbf{AdS}^{4n+3}(\mathbb H);
\end{align*}
\item \hspace{.1in} The null space $\mathbf{Ker} \, \theta $ is the horizontal space of the submersion $\mathbf{AdS}^{4n+3}(\mathbb H) \to \mathbb{H}H^n$.
\end{itemize}
\end{theorem}

From the previous result, 

\[
-\theta=-\frac{1}{2}\left(\sum_{i=1}^{n}  \overline{q}_i  dq_i  - d\overline{q}_i q_i - (\overline{q}_{n+1}  dq_{n+1}  - d\overline{q}_{n+1}  \, q_{n+1}  ) \right) .
\]

is therefore the connection form of the quaternionic anti-de Sitter fibration.

To study the horizontal Brownian motion on $\bS^{4n+3}$ we will use the results of  Section \ref{skew general bundle}. For doing so, we consider the following trivialization of the bundle $\mathbf{SU}(2) \to \mathbf{AdS}^{4n+3}(\mathbb H) \to \mathbb{H}H^n$
\begin{align}\label{trivialization quaternionic hyperbolic}
\left( w_1,\dots, w_n, g \right) \to \frac{1}{\sqrt{1-|w|^2}}\left(  w_1 g,\dots,  w_n g, g \right) \in \mathbf{AdS}^{4n+3}(\mathbb H) \subset \mathbb H^{n+1}
\end{align}
where $g \in \mathbf{SU}(2)$ and $w \in \mathbb{H}H^n$. 

Consider the $\mathfrak{su}(2)$ valued one-form $\boldsymbol\alpha$ on $\mathbb{H}H^n$ which is given  in  affine coordinates by

\[
\boldsymbol\alpha=\frac{1}{2}\sum_{j=1}^n  \frac{\overline{w}_jdw_j-  d\overline{w}_j \, w_j}{1-|w|^2}
\]
where $|w|^2=\sum_{j=1}^n |w_j|^2$. Using the fact that $q_i=\frac{w_i g}{\sqrt{1-|w|^2}}$, a straightforward computation shows that $\boldsymbol\alpha$ is the form (in the sense of \eqref{solder form}) associated with $\theta$. We  can write
\[
\boldsymbol\alpha=\frac{1}{1-|w|^2}\left(\boldsymbol\alpha_1 I +\boldsymbol\alpha_2 J +\boldsymbol\alpha_3 K \right)
\]
where the real-valued one-forms $\boldsymbol\alpha_1,\boldsymbol\alpha_2,\boldsymbol\alpha_3$ are given by
\[
\boldsymbol\alpha_1=\sum_{j=1}^n \left( -\alpha_j d\tau_j+\tau_j d\alpha_j+\gamma_jd\beta_j-\beta_jd\gamma_j \right) 
\]
\[
\boldsymbol\alpha_2=\sum_{j=1}^n \left( -\beta_j d\tau_j-\gamma_j d\alpha_j+\tau_jd\beta_j+\alpha_jd\gamma_j \right)
\]
and
\[
\boldsymbol\alpha_3=\sum_{j=1}^n \left( -\gamma_j d\tau_j+\beta_j d\alpha_j-\alpha_jd\beta_j+\tau_jd\gamma_j \right) 
\]
where we denote as before $w_i=\tau_i+\alpha_i I+\beta_iJ+\gamma_i K$. 


Finally, to conclude the section, we provide the formula for the radial part of the horizontal Laplacian of the quaternionic Hopf fibration. Since the proof is very close to that of Theorem \ref{radial laplacian quater hopf} we let the proof to the reader.

\begin{theorem}\label{radial laplacian quater hyperbolic}
In the trivialization \eqref{trivialization quaternionic hyperbolic}, the horizontal Laplacian acting on functions depending only of $r=\arctanh |w|$ and $g \in \mathbf{SU}(2)$ is given by
\[
\frac{\partial^2}{\partial r^2}+((4n-1)\coth r+3\tanh r)\frac{\partial}{\partial r}+\tanh^2 r \Delta_{\mathbf{SU}(2)}
\]
where $\Delta_{\mathbf{SU}(2)}$ is the Laplacian on $\mathbf{SU}(2)$.
\end{theorem}

\subsection{Stochastic area process on $\mathbb{H}H^n$}

In view of Section \ref{sec-qua-ads-fib} the following definition is natural.

\begin{definition}\label{quaternionic stochastic area hyperbolic}
Let $(w(t))_{t \ge 0}$ be a Brownian motion on $\mathbb{H}H^n$ started at $0$.  The quaternionic stochastic area process of $(w(t))_{t \ge 0}$ is a process in $\mathfrak{su}(2)$ defined by
\[
\A(t)=\int_{w[0,t]} \boldsymbol\alpha=\frac{1}{2}\sum_{j=1}^n \int_0^t \frac{\overline{w}_j(s) dw_j(s)-  d\overline{w}_j(s) \, w_j(s)}{1-|w(s)|^2},
\]
where the above stochastic integrals are understood in the Stratonovich, or equivalently in the It\^o sense.
\end{definition}

The following theorem shows that the quaternionic stochastic area process of the Brownian motion on $\mathbb{H}P^n$ can be interpreted as the fiber motion of the horizontal Brownian motion on $\mathbb{S}^{4n+3}$.

 \begin{theorem}\label{horizon-Shyp}
 Let $(w(t))_{t \ge 0}$ be a Brownian motion on $\mathbb{H}H^n$ started at 0, and $(\Theta(t))_{t\ge 0}$ be the solution in $\mathbf{SU}(2)$ of the stochastic differential equation
 \begin{equation}\label{eq-Maurer-Cartan-Shyp}
 d\Theta(t) =  \circ d\A(t) \, \Theta(t) , \quad \Theta(0)=1.
 \end{equation}
 The $\mathbf{AdS}^{4n+3}(\mathbb H)$-valued diffusion process
 \begin{equation}\label{eq-BM-S}
 X(t)=\frac{1 }{\sqrt{1-|w(t)|^2}} \left( w(t)\Theta(t),\Theta(t) \right), \quad t \ge 0
 \end{equation}
 is the horizontal lift at $(0,\dots,0,1) \in \mathbf{AdS}^{4n+3}(\mathbb H)$ of $(w(t))_{t \ge 0}$ by the submersion $\mathbf{AdS}^{4n+3}(\mathbb H)\to \bH H^n$.
 \end{theorem}

 \begin{proof}
This follows from Theorem \ref{skew-product principal bundle} and Theorem \ref{connection quaternio ads} because $-\theta$ is the connection form.
\end{proof}

As for the quaternionic Hopf fibration, the fiber motion $\Theta(t)$ on $\mathbf{AdS}^{4n+3}(\mathbb H)$ is in fact a time-changed Brownian motion process on $\mathbf{SU}(2)$ and the following two results are proved similarly to Theorem \ref{diff1-S} and Corollary \ref{diff2-S}.
\begin{theorem}\label{diff1-H}
Let $r(t)=\arctanh |w(t)|$. The process $\left( r(t), \Theta(t)\right)_{t \ge 0}$ is a diffusion with generator
 \[
\frac{1}{2} \left(\frac{\partial^2}{\partial r^2}+((4n-1)\coth r+3\tanh r)\frac{\partial}{\partial r}+\tanh^2 r \Delta_{\mathbf{SU}(2)} \right).
 \]
 As a consequence the following equality in distribution holds
 \begin{equation}\label{eq-mp-H}
\left( r(t) ,\Theta(t) \right)_{t \ge 0}=\left( r(t),\beta\left(\int_0^t \tanh^2 r(s)ds\right)\right)_{t \ge 0},
\end{equation}
where $(\beta(t))_{t \ge 0}$ is a Brownian motion process on $\mathbf{SU}(2)$ independent from $r$.
\end{theorem}

\begin{corollary}\label{cor-a(t)-h}
Let $r(t)$ and $\A(t)$ be given as previously. Then
\[
\left( r(t), \A(t) \right)_{t \ge 0} \overset{d}{=} \left( r(t),\gamma\left({\int_0^t \tanh^2 r(s)ds}\right)\right)_{t \ge 0},
\]
where $\gamma(t)$, $t\ge0,$ is a standard Brownian motion process in $\mathfrak{su}(2)$ which is independent from the process $r(t)$.
\end{corollary}

\subsection{The horizontal heat kernel of the quaternionic anti-de Sitter space}\label{sec-kernel}

Recall from the previous section, the radial part of the horizontal Laplacian of the quaternionic anti-de Sitter fibration
\[
\mathbf{SU}(2)\to \mathbf{AdS}^{4n+3}(\mathbb{H})\to \bH H^n
\]
is given by
\begin{equation*}
 \frac{\partial^2}{\partial r^2}+((4n-1)\coth r+3\tanh r)\frac{\partial}{\partial r^2}+\tanh^2 r\left(\frac{\partial^2}{\partial \eta^2}+2\cot\eta\frac{\partial}{\partial \eta}  \right),
\end{equation*}
where $r \ge 0$ is the radial coordinate on the base space  $\bH H^n$ and $\eta \in [0,\pi)$ is the radial coordinate on $\mathbf{SU}(2)$. This operator shall be denoted by $L$. We denote by $p_t(r,\eta)$ the  heat kernel started at $(0,0)$ of $L$  with respect to the  measure 
\begin{equation*}
\nu (dr, d\eta) = \frac{8\pi^{2n+1}}{\Gamma(2n)}(\sinh r)^{4n-1}(\cosh r)^3(\sin\eta)^2drd\eta.
\end{equation*} 
 
 We note that since $\mathbf{SU}(2)$ and $\bH H^n$ are both rank one symmetric spaces,   $p_t(r,\eta)$ completely characterizes the horizontal heat kernel on $\mathbf{AdS}^{4n+3}(\mathbb{H})$ issued from the point with quaternionic coordinates $(0,\cdots,0,1)$. Moreover, the Lie group $\mathbf{Sp}(1,n) $ acts transitively on $\mathbf{AdS}^{4n+3}(\mathbb{H})$, therefore the choice of the origin point $(0,\dots,0,1)$ is irrelevant and a similar analysis yields the heat kernel issued from any point. For simplicity, we restrict ourselves to the computation of $p_t(r,\eta)$.

In this section, we shall derive an integral representation of the  heat kernel $p_t(r, \eta)$ by two different methods.

\begin{itemize}
\item \hspace{.1in} The first method is geometric and uses an analytic continuation in the fiber variables;
\item \hspace{.1in}  The second method is related to the heat kernel associated with the generalized Maass Laplacian. 
\end{itemize}

\subsubsection{First method}


\begin{theorem}\label{IntRep} 
Let $r \ge 0$ and $ \eta \in [0,\pi)$, then:
 \begin{align}\label{eq-IntRep}
 & p_t(r,\eta)\nonumber\\
 = &\frac2\pi \int_{0}^{\infty} \frac{\sinh u}{\sin \eta}  \left\{\sum_{m \geq 0} e^{-m(m+2)t} \sin[(m+1)\eta] \sinh[(m+1)u] \right\}q_{t, 4n+3}(\cosh r \cosh u) du \nonumber\\
 =&\frac{ e^{t}}{  \sqrt{\pi t}  }  \sum_{k\in\mathbb Z}\int_0^{+\infty} \frac{ \sinh y \sin \left(  \frac{ (\eta+2k\pi) y}{2t}\right) }{\sin \eta} e^{\frac{y^2- (\eta+2k\pi)^2}{4t}} q_{t,4n+3}( \cosh r\cosh y ) dy,
\end{align}
where
\begin{equation}\label{eq-qt-ads-rie}
q_{t, 4n+3}(\cosh s) = \frac{e^{-(2n+1)^2t}}{(2\pi)^{2n+1}\sqrt{4\pi t}}\left(-\frac{1}{\sinh(s)}\frac{d}{ds}\right)^{2n+1} e^{-s^2/(4t)}
\end{equation} 
is the heat kernel on the real $4n+3$ dimensional real hyperbolic space.
\end{theorem}
%

\begin{proof}
The following computations are based on geometric ideas that we will describe after the proof. For conciseness, we omit some of the technical details since a second proof of the result will be given in the next section. We first decompose
\begin{equation*}
L = \frac{\partial^2}{\partial r^2}+((4n-1)\coth r+3\tanh r)\frac{\partial}{\partial r^2}+\tanh^2 r\left(\frac{\partial^2}{\partial \eta^2}+2\cot\eta\frac{\partial}{\partial \eta}  \right)
\end{equation*}
 as follows:

\begin{equation*}
L = \Delta_{\bH H^n}+\tanh^2 r {\Delta}_{\mathbf{SU}(2)},
\end{equation*}

where 
\[
\Delta_{\bH H^n}=\frac{\partial^2}{\partial r^2}+((4n-1)\coth r+3\tanh r)\frac{\partial}{\partial r^2}
\]
 is the radial part of the Laplacian on the quaternionic hyperbolic space $\bH H^n$ and 
 \[
 \Delta_{\mathbf{SU}(2)}=\frac{\partial^2}{\partial \eta^2}+2\cot\eta\frac{\partial}{\partial \eta} 
 \]
is the radial part of the Laplacian on $\mathbf{SU}(2)$. Note that we can also write
\[
L={\square}_{\mathbf{AdS}^{4n+3}(\mathbb{H})}+{\Delta}_{\mathbf{SU}(2)},
\]
where
\[
{\square}_{\mathbf{AdS}^{4n+3}(\mathbb{H})}=\Delta_{\bH H^n}-\frac{1}{\cosh^2 r}{\Delta}_{\mathbf{SU}(2)}
\]
is the radial part of the d'Alembertian. Note $\Delta_{\bH H^n}$ and ${\Delta}_{\mathbf{SU}(2)}$ commute. Therefore
\begin{align*}
e^{tL}&=e^{t ({\square}_{\mathbf{AdS}^{4n+3}(\mathbb{H})}+{\Delta}_{\mathbf{SU}(2)})}\\
 &=e^{t {\Delta}_{\mathbf{SU}(2)}} e^{ t {\square}_{\mathbf{AdS}^{4n+3}(\mathbb{H})} }.
\end{align*}
We deduce that the heat kernel of $L$ can be written as
\begin{align}\label{formula 1}
p_t(r,\eta)=\int_0^\pi S_t(\eta ,u) p_t^{{\square}_{\mathbf{AdS}^{4n+3}(\mathbb{H})}} (r,u)  (\sin u)^2 du,
\end{align}
where $S_t$ is the heat kernel of 
\[
 \Delta_{\mathbf{SU}(2)}=\frac{\partial^2}{\partial \eta^2}+2\cot\eta\frac{\partial}{\partial \eta} 
 \]
with respect to the measure $ \sin^2\eta d\eta$, $\eta\in[0,\pi)$,
 and $p_t^{{\square}_{\mathbf{AdS}^{4n+3}(\mathbb{H})}} (r,u) $ the heat kernel at $(0,0)$ of ${\square}_{\mathbf{AdS}^{4n+3}(\mathbb{H})}$ with respect to the measure 
 \[
 d\mu_{4n+3}(r,u)=\frac{8\pi^{2n+1}}{\Gamma(2n)}(\sinh r)^{4n-1}(\cosh r)^3(\sin u)^2drdu,\quad  r\in[0,\infty),\ u\in [0,\pi).
 \]
 The idea to compute \eqref{formula 1} is to perform an analytic extension in the fiber variable  $u$. More precisely,   let us consider the analytic change of variables $\tau : (r,\eta) \to (r,i\eta)$  that will be applied on functions of the type $f(r,u)=g(r) e^{-i\lambda u}$, with $g$ smooth and compactly supported on $[0,+\infty)$ and $\lambda >0$. One sees that
\begin{align}\label{polinh}
{\square}_{\mathbf{AdS}^{4n+3}(\mathbb{H})} (f\circ \tau)=(\Delta_{H^{4n+3}} f ) \circ \tau
\end{align}
where
\[
\Delta_{H^{4n+3}}=\Delta_{\bH H^n}+\frac{1}{\cosh^2 r}{\Delta}_{P}
\]
and
\[
\Delta_{P}=\frac{\partial^2}{\partial \eta^2}+2\coth \eta \frac{\partial}{\partial \eta}.
\]
Introducing a new variable $\delta$ such that $\cosh \delta= \cosh r \cosh \eta$, one sees after straightforward computations that
\[
\Delta_{H^{4n+3}}=\frac{\partial^2}{\partial \delta^2}+(4n+2)\coth\delta\frac{\partial}{\partial \delta}.
\]

Thus $\Delta_{H^{4n+3}}$ is the radial part of the Laplacian on the real hyperbolic space of dimension $4n+3$. One deduces
\begin{align*}
e^{tL} (f \circ \tau)&=e^{t {\Delta}_{\mathbf{SU}(2)}} e^{ t {\square}_{\mathbf{AdS}^{4n+3}(\mathbb{H})} } (f \circ \tau) \\
 &= e^{t {\Delta}_{\mathbf{SU}(2)}} ( (e^{t\Delta_{H^{4n+3}} }f) \circ \tau) \\
 &=(e^{-t \Delta_{P} }e^{t\Delta_{H^{4n+3}}} f ) \circ \tau.
\end{align*}

  Now, since for every $f$ in the domain of ${\square}_{\mathbf{AdS}^{4n+3}(\mathbb{H})}$,
 \[
( e^{ t {\square}_{\mathbf{AdS}^{4n+3}(\mathbb{H})} } f ) (0,0)=  (e^{t\Delta_{H^{4n+3}} }(f\circ  \tau^{-1})) (0,0),
 \]
 one deduces  that for a function $g$ depending only on $u$, namely $g(u)=e^{-i\lambda u}$
  \[
 \int_0^\pi g(u) p_t^{{\square}_{\mathbf{AdS}^{4n+3}(\mathbb{H})}} (r,u)  (\sin u)^2 du=\int_0^{+\infty} q_{t, 4n+3}(\cosh r \cosh u)g(-iu) (\sinh u)^2 du.
 \]

 Therefore, coming back to \eqref{formula 1}, one infers that
 \[
 \int_0^\pi S_t(\eta ,u) p_t^{{\square}_{\mathbf{AdS}^{4n+3}(\mathbb{H})}} (r,u)  (\sin u)^2 du=\int_0^{+\infty} q_{t, 4n+3}(\cosh r \cosh u) S_t(\eta,-iu) (\sinh u)^2 du.
 \]
 One concludes with the well known formulas (see Section 8.6 in  \cite{Faraut})
 \begin{align}\label{eq-s-t}
 S_t(\eta,u) & =\frac{2}{\pi \sin \eta \sin u} \sum_{m \geq 0} e^{-m(m+2)t} \sin[(m+1)\eta] \sin[(m+1)u] \\
  &=\frac{e^{t}}{  \sqrt{\pi t}  \sin \eta \sin u } \sum_{k\in\Z} \sinh \left(  \frac{ (\eta+2k\pi) u}{2t}\right) e^{\frac{-u^2-(\eta+2k\pi)^2}{4t}}. \nonumber
 \end{align}
\end{proof}

\subsubsection{Second method}

The strategy of the second method appeals  to some results proved in \cite{Int-OM}.


\begin{proof}
Firstly, we decompose the subelliptic heat kernel in the basis of Chebyshev polynomials of the second kind 
\begin{equation*}
U_m(\cos(\eta)) = \frac{\sin((m+1)\eta)}{\sin(\eta)}, \quad m \geq 0, 
\end{equation*}
which are eigenfunctions of $\Delta_{\mathbf{SU}(2)}$: 
\begin{equation*}
\left(\frac{\partial^2}{\partial \eta^2}+2\cot\eta\frac{\partial}{\partial \eta}\right) U_m(\cos(\cdot))(\eta) = -m(m+2) U_m(\cos(\eta)), \quad m \geq 0.
\end{equation*}
Accordingly, 
\begin{equation*}
p_t(r,\eta) = \sum_{m \geq 0} f_m(t,r) U_m(\cos(\eta))
\end{equation*}
where for each $m$, $f_m(t,\cdot)$ solves the following heat equation:
\begin{align*}
\partial_t (f_m)(t,r) & =  \left\{\frac{\partial^2}{\partial r^2}+((4n-1)\coth r+3\tanh r)\frac{\partial}{\partial r^2} -m(m+2) \tanh^2 r\right\}(f_m)(t,r) 
\\& = \left\{\frac{\partial^2}{\partial r^2}+((4n-1)\coth r+3\tanh r)\frac{\partial}{\partial r^2} +\frac{m(m+2)}{\cosh^2 r} -m(m+2) \right\}(f_m)(t,r). 
\end{align*}
Up to a constant $(-m(m+2) - (2n+1)^2)$, the operator 
\begin{equation*}
\mathcal{L}_m := \frac{\partial^2}{\partial r^2}+((4n-1)\coth r+3\tanh r)\frac{\partial}{\partial r^2} +\frac{m(m+2)}{\cosh^2 r} + (2n+1)^2
\end{equation*}
 is an instance of the one considered in \cite{Int-OM} and denoted there $L_{2n}^{\alpha \beta}$ with 
\begin{equation*}
\alpha = 1+\frac{m}{2}, \qquad \beta = - \frac{m}{2}. 
\end{equation*}
From Theorem 2 in \cite{Int-OM}, we deduce that the solution to the wave Cauchy problem associated with the subelliptic Laplacian is given by 
\begin{equation*}
\cos(s\sqrt{-\mathcal{L}_m})(f)(w) = -\frac{\sinh(s)}{(2\pi)^{2n}} \left(\frac{1}{\sinh s}\frac{d}{ds}\right)^{2n} \int_{\mathbb{H}H^n} K_m(s,w,y)f(y)\frac{dy}{(1-|y|^2)^{2n+2}}  
\end{equation*}
where $f \in C_0^{\infty}(\mathbb{H}H^n)$,
\begin{multline*}
K_m(s,w,y) = \frac{(1-\overline{\langle w,y \rangle})^{1+m/2}}{(1-\langle w, y\rangle)^{-m/2}} \frac{\sqrt{\cosh^2(s) - \cosh^2(d(w,y))_+}}{\cosh^2(d(w,y))} 
\\ {}_2F_1\left(m+2, -m, \frac{3}{2}; \frac{\cosh(d(w,y)) - \cosh(s)}{2\cosh(d(w,y))}\right),
\end{multline*}
 ${}_2F_1$ is the Gauss hypergeometric function, and $dy$ stands for the Lebesgue measure in $\mathbb{C}^n$. Using the spectral formula  
\begin{equation*}
e^{tL} = \frac{1}{\sqrt{4\pi t}}\int_{\mathbb{R}} e^{-s^2/(4t)} \cos(s\sqrt{-L}) ds,
\end{equation*}
holding true for any non-positive self-adjoint operator, we deduce that the solution to the heat Cauchy problem associated with $\mathcal{L}_m$ is given by: 
\begin{multline*}
e^{t\mathcal{L}_m}(f)(w)  = - \frac{e^{-m(m+2)t-(2n+1)^2t}}{\sqrt{4\pi t}(2\pi)^{2n}}  \int_{\mathbb{R}}ds\sinh(s) e^{-s^2/(4t)} \left(\frac{1}{\sinh s}\frac{d}{ds}\right)^{2n+1} \\ \int_{\mathbb{H}H^n} K_m(s,w,y)f(y)\frac{dy}{(1-{|y|}^2)^{2n+2}}.
\end{multline*}
Performing $2n+1$ integration by parts in the outer integral we further get: 
\begin{multline*} 
e^{t\mathcal{L}_m}(f)(w) = -\frac{e^{-m(m+2)t-(2n+1)^2t}}{\sqrt{4\pi t}(2\pi)^{2n}}  \int_{\mathbb{R}}ds\sinh(s) \left(-\frac{1}{\sinh s}\frac{d}{ds}\right)^{2n+1} e^{-s^2/(4t)} \\ \int_{\mathbb{H}H^n} K_m(s,w,y)f(y)\frac{dy}{(1-{|y|}^2)^{2n+2}} 
= -\frac{e^{-m(m+2)t-(2n+1)^2t}}{\sqrt{4\pi t}(2\pi)^{2n}}  \int_{\mathbb{H}H^n}f(y)\frac{dy}{(1-{|y|}^2)^{2n+2}}
\end{multline*}
and
\begin{align*} 
&\int_{\mathbb{R}}ds \sinh(s) K_m(s,w,y) \left(-\frac{1}{\sinh s}\frac{d}{ds}\right)^{2n+1} e^{-s^2/(4t)}= 2\frac{e^{-m(m+2)t-(2n+1)^2t}}{\sqrt{4\pi t}(2\pi)^{2n}}  \\
&\quad\quad \cdot\int_{\mathbb{H}H^n}f(y)\frac{dy}{(1-{|y|}^2)^{2n+2}}\int_{d(w,y)}^{\infty}d(\cosh(s))K_m(s,w,y) \left(-\frac{1}{\sinh s}\frac{d}{ds}\right)^{2n+1} e^{-s^2/(4t)}.
\end{align*}
 Recalling the heat kernel on the hyperbolic space $H^{4n+3}$:
\begin{equation*}
q_{t, 4n+3}(\cosh s) := \frac{e^{-(2n+1)^2t}}{(2\pi)^{2n+1}\sqrt{4\pi t}}\left(-\frac{1}{\sinh(s)}\frac{d}{ds}\right)^{2n+1} e^{-s^2/(4t)},
\end{equation*} 
we get
\begin{multline*}
e^{t\mathcal{L}_m}(f)(0) = 4\pi e^{-m(m+2)t}  \int_{\mathbb{H}H^n}f(y)\frac{dy}{(1-{|y|}^2)^{2n+2}}\\ \int_{d(0,y)}^{\infty}d(\cosh(s))K_m(s,0,y) q_{t, 4n+3}(\cosh(s)).
\end{multline*}
As a result, the subelliptic heat kernel of $\mathcal{L}_m$ reads
\begin{align*}
&\frac{dy}{(1-{|y|}^2)^{2n+2}} \int_{d(0,y)}^{\infty}d(\cosh(s))K_m(s,0,y) q_{t, 4n+3}(\cosh(s)) \\
&= dr \cosh^3(r)\sinh^{4n-1}(r)  \int_{r}^{\infty}d(\cosh(s))K_m(s,0,y) q_{t, 4n+3}(\cosh(s)).
\end{align*}
Performing the variable change $\cosh(s) = \cosh(r)\cosh(u)$ for $u \geq 0$, we transform the last expression into 
\begin{multline*}
dr \cosh^3(r)\sinh^{4n-1}(r) \\ \cdot\int_{0}^{\infty} \sinh^2(u) {}_2F_1\left(m+2, -m, \frac{3}{2}; \frac{1- \cosh(u)}{2}\right) q_{t, 4n+3}(\cosh(r)\cosh(u)) du.
\end{multline*}
Since 
\begin{equation*}
(m+1) {}_2F_1\left(m+2, -m, \frac{3}{2}; \frac{1- \cosh(u)}{2}\right) = \frac{\sinh[(m+1)u]}{\sinh(u)},
\end{equation*}
we finally recover the first integral representation displayed in Theorem \ref{IntRep}. 
\end{proof}


\subsubsection{Relation to the horizontal heat kernel of the complex anti-de Sitter fibration}

In $\mathbb{C}^{2n+1}, n \geq 1$, we consider the signed quadratic form: 
\begin{equation*}
|(z_1, \dots, z_{2n+1})|^2_H = \sum_{i=1}^{2n} |z_i|^2 - |z_{2n+1}|^2. 
\end{equation*}
Then, the  complex anti-de Sitter space $\mathbf{AdS}^{4n+1}(\mathbb{C})$  is the quadric defined by 
\begin{equation*}
\{z = (z_1, \dots, z_{n+1}) \in \mathbb{C}^{n+1}, |z|_H = -1\},
\end{equation*}
and the circle group $U(1)$ acts on it by $z \mapsto ze^{i\theta}$. This action gives rise to the AdS fibration over the complex hyperbolic space $\mathbb{C}H_n$ with $\mathbf{U}(1)$-fibers. 

\begin{proposition}
Let $p_t^{\C}(r,\eta)$ be the horizontal heat kernel on the $4n+1$-dimensional  complex anti-de Sitter space. Then
\[
-\frac{e^{-4nt}}{2\pi\cosh r\sin\eta}\frac{\partial}{\partial\eta}p_t^{\C}(r, \eta)=p_t(r,\eta).
\]
\end{proposition}
\begin{proof}
From Proposition \ref{prop135}  we know that for $\eta\in[-\pi,\pi]$
\[
p_t^{\C}(r, \eta)=\frac{1}{\sqrt{4\pi t}}\sum_{k\in \Z}\int_{-\infty}^{+\infty}e^{\frac{y^2}{4t} }q_{t,4n+1}(\cosh r\cosh (y+i \eta+2k\pi i))dy
\]
where $q_{t,4n+1}$ is the Riemannian heat kernel of the $4n+1$-dimensional hyperbolic space. It is well known that 
\[
q_{t,4n+3}(x)=-\frac{e^{-(4n+1)t}}{2\pi}\frac{d}{dx}q_{t,4n+1}(x)
\]
hence we can easily obtain that
\begin{align*}
& \frac{\partial}{\partial\eta}p_t^{\C}(r, \eta) \\
=&\frac{1}{\sqrt{4\pi t}} \sum_{k\in \Z}\int_{-\infty}^{+\infty}e^{\frac{y^2}{4t} }\frac{\partial}{\partial\eta} q_{t,4n+1}(\cosh r\cosh (y+i \eta+2k\pi i))dy\\
=&-\frac{2\pi e^{(4n+1)t}\cosh r}{\sqrt{4\pi t}} \sum_{k\in \Z}\int_{-\infty}^{+\infty}e^{\frac{y^2-(\eta+2k\pi)^2}{4t} }\sinh \frac{(\eta+2k\pi) y}{2t}\sinh y\, q_{t,4n+3}(\cosh r\cosh y)dy\\
=&-{2\pi e^{(4n+1)t}\cosh r}\,p_t(r,\eta)
\end{align*}
\end{proof}

\subsection{Characteristic function of the stochastic area and limit theorem}

In this section we study the characteristic function of the stochastic area $\A(t)$ (see Definition \ref{quaternionic stochastic area hyperbolic}).
Let
\[
\mathcal{L}^{\alpha,\beta}=\frac{1}{2} \frac{\partial^2}{\partial r^2}+\left(\left(\alpha+\frac{1}{2}\right)\coth r+\left(\beta+\frac{1}{2}\right) \tanh r\right)\frac{\partial}{\partial r}, \quad \alpha,\beta >-1
\]
be the hyperbolic Jacobi generator. We will denote by $q_t^{\alpha,\beta}(r_0,r)$ the heat kernel with respect to the Lebesgue measure of the diffusion it generates.

Let $\lambda=(\lambda_I, \lambda_J, \lambda_K) \in [0,\infty)^3$, $r \in [0,+\infty),$ then Corollary \ref{cor-a(t)-h}  entails:
\begin{align*}
\mathbb{E}\left(e^{i \lambda\cdot \A(t)}\bigm| r(t)=r\right) & =\mathbb{E}\left(e^{i \lambda\cdot \gamma\left({\int_0^t \tanh^2 r(s)ds}\right)}\bigm| r(t)=r\right) \\
 &=\mathbb{E}\left(e^{- \frac{|\lambda|^2}{2} \int_0^t \tanh^2 r(s)ds}\bigm| r(t)=r\right) 
\end{align*}
where $|\lambda|^2=\lambda_I^2+\lambda_J^2+\lambda_K^2$ and $(r(t))_{t\ge0}$ is a diffusion generated by 
\[
\mathcal{L}^{2n-1,1}=\frac{1}{2}  \frac{\partial^2}{\partial r^2}+\frac12\left((4n-1)\coth r+3\tanh r\right)\frac{\partial}{\partial r},
\]
and started at $0$. 

\begin{theorem}\label{FTHHn}
For $\lambda \in \R^3$, $r \in [0,+\infty)$, and $t >0$
\[
\mathbb{E}\left(e^{i \lambda\cdot \A(t)}\bigm| r(t)=r\right)=\frac{e^{2nt \mu}}{(\cosh r)^{\mu}} \frac{q_t^{2n-1,\mu+1}(0,r)}{q_t^{2n-1,1}(0,r)}.
\]
where $\mu=\sqrt{|\lambda|^2+1}-1$.
\end{theorem}

\begin{proof}
Note
\begin{equation}\label{eq-sde-r-H}
dr(t)= \frac{1}{2} \left( (4n-1)\coth r(t)+3\tanh r(t) \right)dt+d\gamma(t),
\end{equation}
where $\gamma$ is a standard Brownian motion. It implies that almost surely we have
\begin{align}\label{transient-H}
r(t) \ge \left( 2n -\frac{1}{2} \right) t +\gamma(t),
\end{align}
and thus $r(t)\to +\infty$ almost surely when $t \to \infty$.
Consider now the local martingale for any $\mu>0$ given by
\begin{align*}
D_t& :=\exp \left({\mu} \int_0^t \tanh r(s) d\gamma(s) -\frac{\mu^2}{2}  \int_0^t \tanh^2 r(s) ds \right)  \\
 &=\exp \left({\mu} \int_0^t \tanh r(s) dr(s)-\frac{\mu}{2}(4n-1)t -\frac{3{\mu} +{\mu}^2}{2}  \int_0^t \tanh^2 r(s) ds \right). 
\end{align*}
From It\^o's formula, we have
\begin{align*}
\ln \cosh r(t) & =\int_0^t \tanh r(s) dr(s)+\frac{1}{2} \int_0^t \frac{ds}{\cosh^2 r(s)} \\
 &=\int_0^t \tanh r(s) dr(s)-\frac{1}{2} \int_0^t \tanh^2 r(s) ds+\frac{1}{2} t.
\end{align*}
As a consequence, we deduce that
\[
D_t =e^{-2n{\mu} t} (\cosh r(t))^{\mu} e^{- \frac{{\mu}^2+2{\mu}}{2} \int_0^t \tanh^2 r(s)ds}.
\]
From this formula, it is easy to prove that $D_t$, $t\ge0$ is a true martingale.

  Let  $\mathcal{F}$ denote the natural filtration of $(r(t))_{t\ge0}$ and consider the probability measure $\mathbb{P}^{\mu}$ defined by
\[
\mathbb{P}_{/ \mathcal{F}_t} ^{\mu}=D_t \mathbb{P}_{/ \mathcal{F}_t}=e^{-2n{\mu} t} (\cosh r(t))^{\mu} e^{- \frac{{\mu}^2+2{\mu}}{2} \int_0^t \tanh^2 r(s)ds} \mathbb{P}_{/ \mathcal{F}_t}.
\]
We have then for every bounded and Borel function $f$ on $[0,+\infty]$,
\begin{align*}
\mathbb{E}\left(f(r(t))e^{- \frac{{\mu}^2+2{\mu}}{2} \int_0^t \tanh^2 r(s)ds}\right)=e^{2n{\mu} t} \mathbb{E}^{\mu} \left( \frac{f(r(t))}{(\cosh r(t))^{\mu}} \right).
\end{align*}
From Girsanov theorem, the process
\[
\gamma^{\mu}(t)=\gamma(t)-{\mu} \int_0^t \tanh r(s) ds
\]
is a Brownian motion under the probability $\mathbb{P}^{\mu}$. We note that
\begin{equation}\label{eq-rt-H}
dr(t)= \frac{1}{2} \left( (4n-1)\coth r(t)+(2{\mu} +3)\tanh r(t) \right)dt+d\gamma^{\mu}(t).
\end{equation}
Hence  we  have
 \begin{align*}
\mathbb{E}\left(e^{- \frac{\mu^2+2\mu}{2} \int_0^t \tanh^2 r(s)ds}\biggm| r(t)=r\right)=\frac{e^{2nt \mu}}{(\cosh r)^{\mu}} \frac{q_t^{2n-1,\mu+1}(0,r)}{q_t^{2n-1,1}(0,r)}.
\end{align*}
The proof is complete by letting $\mu=\sqrt{|\lambda|^2+1}-1$.

 \end{proof}

As an immediate corollary of Theorem \ref{FTHHn}, we deduce an expression for  the characteristic function of the stochastic area process.
\begin{corollary}\label{FTh}
For $\lambda \in \R^3$ and $t \ge 0$,
\[
\mathbb{E}\left(e^{i \lambda\cdot \A(t)}\right)=e^{2n \mu t}\int_0^{+\infty} \frac{q_t^{2n-1,\mu+1}(0,r)}{(\cosh r)^{\mu}} dr,
\]
where $\mu=\sqrt{|\lambda|^2+1}-1$.
\end{corollary}

We are now in position to prove the following central limit type theorem.

\begin{theorem}\label{LimitCHn}
When $t \to +\infty$, the following convergence in distribution takes place
\[
\frac{\A(t)}{\sqrt{t}} \to \mathcal{N}(0,\mathrm{I}_3)
\]
where $\mathcal{N}(0,\mathrm{I}_3)$ is a $3$-dimensional normal distribution with mean 0 and variance matrix $\mathrm{I}_3$.
\end{theorem}
\begin{proof}
This is a consequence of $r(t)\to +\infty$ almost surely as $t\to+\infty$. We have
\[
\coth r(t)\to 1,\quad \tanh r(t)\to 1\quad \mbox{a.s.}
\]
hence
\[
\lim_{t\to+\infty}\frac{1}{t}\int_0^t \tanh^2 r(s)ds=1\quad \mbox{a.s.}
\]
Then from Corollary \ref{cor-a(t)-h}, we have
\[
\lim_{t\to+\infty}\frac{\A(t)}{\sqrt{t}}=\lim_{t\to+\infty}\gamma\left({\frac{1}{t}\int_0^t \tanh^2 r(s)ds}\right)=\gamma(1) \quad \mbox{a.s.}
\]
\end{proof}

\subsection{Formula for the density}

 In this section, we compute the density of the quaternionic stochastic area process $\A(t)$.  In order to invert the Fourier transform displayed in Corollary \ref{FTh}, we first need a suitable expression for the heat kernel of the hyperbolic Jacobi operator: 
\begin{equation*}
\mathcal{L}^{n, \mu} = \frac{1}{2} \left( \frac{\partial^2}{\partial r^2} + ((4n-1)\coth r + (2\mu + 3)\tanh(r)) \frac{\partial}{\partial r}\right), \quad r \geq 0,
\end{equation*}
subject to Neumann boundary condition at $r = 0$. Though the heat kernel of this operator may be expressed through Jacobi functions (\cite{Koor}), we shall derive below another one which not only leads to the sought density but has also the merit to involve the heat kernel of the $4n+1$-dimensional real hyperbolic space. The derivation is a bit technical and for ease of reading, we shall proceed into three steps. More precisely, we shall firstly map the above hyperbolic Jacobi operator into another one by letting it act on functions of the form $r \mapsto f(r)/\cosh^{\mu}(r)$, where $f$ is a smooth test function. Secondly, we shall exploit results in \cite{Int-OM} to derive the heat kernel of the newly-obtained operator: in this step, we follow the lines of Theorem 2 in \cite{Bau-Dem}. Finally, we use known Fourier transforms to obtain the density of the quaternionic stochastic area process. We start with the following straightforward lemma:
\begin{lemma}
Let $f$ be a smooth function on $\mathbb{R}_+$. Then  
\begin{equation}\label{Inter}
 \mathcal{L}^{n, \mu} \left(\frac{f}{\cosh^{\mu}} \right)(r)=  \frac{1}{\cosh^{\mu}(r)} L^{n, \mu}(f) (r)
\end{equation}
where
\begin{equation*}
2L^{n, \mu} := \frac{\partial^2}{\partial r^2} + ((4n-1)\coth r + 3\tanh(r)) \frac{\partial}{\partial r} + \frac{\mu(\mu+2)}{\cosh^2(r)}  -\mu(4n+\mu+2).
\end{equation*} 
\end{lemma}
\begin{proof}
The result follows from straightforward computations. 
\end{proof}

The operator $2L^{n, \mu} + \mu(4n+\mu+2) + (2n+1)^2$ is an instance of the radial part of the operator $\Delta_{\alpha\beta}$ studied in \cite{Int-OM} with $\alpha = 1+(\mu/2), \beta = 1-\alpha = -\mu/2$ and double complex dimension $2n$ (see p.229 there). 
Using the same reasoning of the proof of Theorem 2 in \cite{Bau-Dem}, we prove the following: 
\begin{proposition}\label{HKR}
Let $f$ be a smooth compactly-supported function in $\mathbb{H}H^n$. Then, the heat semi-group $e^{tL^{n, \mu}}(f)(0)$ reads: 
\begin{multline*}
\frac{e^{-[(2n+1)^2+\mu(4n+\mu+2)]t/2}}{(2\pi)^{2n}\sqrt{2\pi t}} \int_{\bH H^n} f(w) \frac{dw}{(1-|w|^2)^{2n+2}} 
\\ \int_{d(0,w)}^{\infty} dx \sinh(x) K_{\mu}(x,w) \left(\frac{1}{\sinh(x)}\frac{d}{dx}\right)^{2n} e^{-x^2/(2t)}, 
\end{multline*}
where $d(0,w) = r$ is the geodesic distance in $\mathbb{H}H^n$: 
\begin{equation*}
\cosh^2(d(0,w)) = \frac{1}{1-|w|^2},
\end{equation*}
and
\begin{multline*}
K_{\mu}(x,w) := \frac{1}{\cosh(r)\sqrt{\cosh^2(x)- \cosh^2(r)}} \\  {}_2F_1\left(\mu+1, -(\mu+1), \frac{1}{2}; \frac{\cosh(r) - \cosh(x)}{2\cosh(r)} \right).
\end{multline*}
where $ {}_2F_1$ is the hypergeometric function. 
\end{proposition}
\begin{proof}
Consider the `switched' wave Cauchy problem associated $\Delta_{\alpha\beta}, \alpha = 1+(\mu/2), \beta = -\mu/2,$ displayed in eq. (1.1) in \cite{Int-OM}. From Theorem 2 in that paper, its solution is given by: 
\begin{equation*}
u(x,z) = \frac{1}{(2\pi)^{2n}}\left(\frac{1}{\sinh(x)}\partial_x\right)^{2n-1} \int_{d(z,w) < |s|} f(w)K_{\mu} (x,z,w) \frac{dw}{(1-|w|^2)^{2n+2}},
\end{equation*}
where $(x,z) \in \mathbb{R} \times \mathbb{H}H^n$ and 
\begin{multline*}
K_{\mu}(x,z,w) := \frac{(1-\overline{\langle z, w\rangle})^{1+(\mu/2)}(1-\langle z, w\rangle)^{-(\mu/2)}}{\cosh(d(z,w))\sqrt{\cosh^2(x)- \cosh^2(d(z,w))}} \\  {}_2F_1\left(\mu+1, -(\mu+1), \frac{1}{2}; \frac{\cosh(d(z,w)) - \cosh(x)}{2\cosh(d(z,w))} \right).
\end{multline*}
Following the proof of Theorem 2 in \cite{Bau-Dem}, we next deduce the heat kernel of $\Delta_{\alpha\beta}$ from $u(x,z)$. 
To this end, we differentiate $x \mapsto u(x,z)$ to get the solution to the `standard' wave Cauchy problem associated $\Delta_{\alpha\beta}$: 
\begin{equation*}
v(x,z) = \partial_x u(x,z) = \frac{\sinh(x)}{(2\pi)^{2n}}\left(\frac{1}{\sinh(x)}\partial_x\right)^{2n} \int_{d(z,w) < |x|} f(w)K(x,z,w) \frac{dw}{(1-|w|^2)^{2n+2}}. 
\end{equation*}
Then, we use the spectral formula :
\begin{equation*}
e^{tL}  = \frac{1}{\sqrt{4\pi t}}\int_{\mathbb{R}}e^{-x^2/(4t)} \cos(x\sqrt{-L}) dx,
\end{equation*}
relating the heat semigroup of a self-adjoint non positive operator $L$ to the solution of its wave Cauchy problem (we wrote the wave propagator as $\cos(x\sqrt{-L})$ which should be understood in the spectral sense). According to this formula and from Proposition 2 in \cite{Int-OM}, we deduce that $\Delta_{\alpha\beta}, \alpha = 1+(\mu/2), \beta = -\mu/2,$ is a non positive self-adjoint operator and that (we perform $2n$ integrations by parts then use Fubini Theorem): 
\begin{multline*}
e^{t\Delta_{\alpha\beta}}(f)(z) = \frac{1}{\sqrt{\pi t}}\int_{0}^{\infty} e^{-x^2/(4t)}v(x,z) dx = \frac{1}{(2\pi)^{2n}\sqrt{\pi t}} 
\\ \int_{\bH H^n} f(w) \frac{dw}{(1-|w|^2)^{2n+2}} 
\int_{d(z,w)}^{\infty} dx \sinh(x) K_{\mu}(x,z,w) \left(-\frac{1}{\sinh(x)}\frac{d}{dx}\right)^{2n} e^{-x^2/(4t)}. 
\end{multline*}
Specializing this formula to $z=0$, we see from the definition of $K_{\mu}(x,z,w)$ that the heat kernel of $e^{t\Delta_{\alpha\beta}}(f)(0)$ is radial. Keeping in mind the aforementioned relation between the radial part of $\Delta_{\alpha\beta}$ with the special parameters 
$\alpha = 1+(\mu/2), \beta = -\mu/2,$ and $L^{n, \mu} + \mu(4n+\mu+2) + (2n+1)^2$, the statement of the proposition follows (we simply wrote $K(x,w)$ for $K(x,0,w)$). 
\end{proof}
With the help of Proposition \eqref{HKR}, we are ready to derive the density of $\A(t)$. 
\begin{theorem}
Let $s_{t, 4n+1}(\cosh(x))$ be the heat kernel of the $4n+1$-dimensional real hyperbolic space (see \eqref{eq-kernel-hyper-jacobi})
\begin{align*}
s_{t,4n+1}(\cosh(x)) = \frac{e^{-(2n)^2t/2}}{(2\pi)^{2n}\sqrt{2\pi t}}\left(\frac{1}{\sinh(x)}\frac{d}{dx}\right)^{2n}e^{-x^2/(2t)}, \quad x\ge0
\end{align*}
and 
\begin{equation*}
 {\it I}_{m-1/2}(u) = \sum_{j \geq 0} \frac{\sqrt{\pi}}{j!\Gamma(j+m+1/2)} \left(\frac{u}{2}\right)^{2j+m-1/2}
\end{equation*}
be the modified Bessel function. Define also the time-dependent symmetric polynomials $Q_{2m}, m \geq 0,$ in $(v_1,v_2,v_3)$ of degree $2m$ by: 
\begin{equation*}
Q_{2m}(v_1,v_2, v_3, t) := e^{|v|^2/(2t)}\left(\Delta_v^m  e^{-|v|^2/(2t)}\right), \quad v \in \mathbb{R}^3,
\end{equation*}
where $\Delta_v$ is the Euclidean Laplacian in $\mathbb{R}^3$ acting on $v$. Then the density of the quaternionic stochastic area process $\A(t)$ is given by: 
\begin{multline*}
\frac{e^{-(4n+1)t/2}}{(2\pi t)^{3/2}}e^{-|v|^2/(2t)}  \int_0^{\infty}dr \sinh(r)^{4n-1}\cosh^2(r)  \int_{0}^{\infty} du\,  s_{t, 4n+1}(\cosh(u)\cosh(r)) \\
\sum_{m \geq 0} \frac{(-1)^m}{m!}\left(\frac{u}{2}\right)^{m+1/2} {\it I}_{m-1/2}\left(\frac{u}{2}\right)Q_{2m}(v_1,v_2,v_3, t), \quad v \in \mathbb{R}^3.
\end{multline*}
\end{theorem}

\begin{remark}
The polynomial $Q_{2m}$ may be expressed as a linear combination of products of (even) Hermite polynomials: 
\begin{equation*}
H_j(x):= (-1)^je^{x^2/2}\frac{d^j}{dx^j}e^{-x^2/2}. 
\end{equation*}
Indeed, it suffices to expand: 
\begin{equation*}
\Delta_v^m = \sum_{j_1+j_2+j_3 = m} \frac{m!}{j_1!j_2!j_3!} \partial_{v_1}^{2j_1}\partial_{v_2}^{2j_2}\partial_{v_3}^{2j_3} 
\end{equation*}
to get the representation: 
\begin{equation}\label{Hermite}
Q_{2m}(v_1,v_2,v_3, t) = \frac{1}{t^m} \sum_{j_1+j_2+j_3 = m} \frac{m!}{j_1!j_2!j_3!}H_{2j_1}\left(\frac{v_1}{\sqrt{t}}\right)H_{2j_2}\left(\frac{v_2}{\sqrt{t}}\right)H_{2j_3}\left(\frac{v_3}{\sqrt{t}}\right).
\end{equation}
\end{remark}

\begin{proof}
Since the radial part of the measure (this is the volume measure of $\bH H^n$)
\begin{equation*}
\frac{dw}{(1-|w|^2)^{2n+2}}
\end{equation*}
is $\sinh(r)^{4n-1}\cosh^3(r) dr, |w| = \tanh(r)$, then the intertwining relation \eqref{Inter} together with Proposition \ref{HKR} yield:
\begin{multline*}
e^{2n\mu t} \frac{q_t^{2n-1,\mu+1}(0,r)}{\cosh^{\mu}(r)} = \frac{e^{-[(2n+1)^2+\mu(\mu+2)]t/2}}{(2\pi)^{2n}\sqrt{2\pi t}} \sinh(r)^{4n-1}\cosh^3(r)
\\ \int_{d(0,w) = r}^{\infty} dx \sinh(x) K_{\mu}(x,w) \left(\frac{1}{\sinh(x)}\frac{d}{dx}\right)^{2n} e^{-x^2/(2t)}. 
\end{multline*}
Performing the variable change $\cosh(x) = \cosh(u)\cosh(r)$ and using the expression of the heat kernel $s_{t,4n+1}$, we equivalently write:  
\begin{multline*}
e^{2n\mu t} \frac{q_t^{2n-1,\mu+1}(0,r)}{\cosh^{\mu}(r)}  = e^{-2nt - (\mu+1)^2t/2} \sinh(r)^{4n-1}\cosh^2(r) \\ 
\int_{0}^{\infty} du\  {}_2F_1\left(-(\mu+1), \mu+1, \frac{1}{2}; \frac{1- \cosh(u)}{2} \right) s_{t, 4n+1}(\cosh(u)\cosh(r)).
\end{multline*}
But, the identity 
\begin{equation*}
{}_2F_1\left(-(\mu+1),  (\mu+1), \frac{1}{2}; \frac{1 - \cosh(u)}{2}\right) = \cosh((\mu+1)u),
\end{equation*}
entails further: 
\begin{multline*}
e^{2n\mu t} \frac{q_t^{2n-1,\mu+1}(0,r)}{\cosh^{\mu}(r)} = e^{-2nt - (\mu+1)^2t/2} \sinh(r)^{4n-1}\cosh^2(r) \\ \int_{0}^{\infty} du  \cosh((\mu+1)u) s_{t, 4n+1}(\cosh(u)\cosh(r)).
\end{multline*}
Consequently, recalling $(\mu+1)^2 = |\lambda|^2+1$, the characteristic function of $\A(t)$ admits the following expression: 
\begin{multline}\label{NewFor}
\mathbb{E}\left(e^{i \lambda\cdot \A(t)}\right) = e^{-(4n+1)t/2} \int_0^{\infty}dr \sinh(r)^{4n-1}\cosh^2(r) \\ \int_{0}^{\infty} du\  e^{-|\lambda|^2t/2} \cosh(\sqrt{|\lambda|^2+1}u) s_{t, 4n+1}(\cosh(u)\cosh(r)).
\end{multline} 
In order to derive the density of $\A(t)$, it suffices to write $e^{-|\lambda|^2t/2} \cosh(\sqrt{|\lambda|^2+1}u)$ as a Fourier transform in the variable $\lambda$ and to apply Fubini Theorem. To this end, we expand: 
\begin{align*}
\cosh(\sqrt{|\lambda|^2+1}u)  & = \sum_{j \geq 0} \frac{u^{2j}}{(2j)!} \sum_{m=0}^j\binom{j}{m}|\lambda|^{2m} 
\\& =  \sqrt{\pi}\sum_{j \geq 0} \frac{u^{2j}}{2^{2j}\Gamma(j+1/2)} \sum_{m=0}^j\frac{1}{m!(j-m)!}|\lambda|^{2m} 
\\& =  \sqrt{\pi}\sum_{m \geq 0} \frac{|\lambda|^{2m}}{m!}\sum_{j \geq m}\frac{1}{(j-m)!}\frac{u^{2j}}{2^{2j}\Gamma(j+1/2)} 
\\&= \sum_{m \geq 0} \frac{|\lambda|^{2m}}{m!}\left(\frac{u}{2}\right)^{m+1/2} {\it I}_{m-1/2}\left(u\right),
\end{align*}
and write
\begin{align*}
|\lambda|^{2m}e^{-|\lambda|^2t/2} & = \frac{(-1)^m}{(2\pi t)^{3/2}} \int_{\mathbb{R}^3} \left(\Delta_v^m e^{i \lambda \cdot v}\right) e^{-|v|^2/(2t)} dv 
\\& = \frac{(-1)^m}{(2\pi t)^{3/2}} \int_{\mathbb{R}^3} e^{i\lambda \cdot v} \left(\Delta_v^m  e^{-|v|^2/(2t)}\right) dv 
\\& = \frac{(-1)^m}{(2\pi t)^{3/2}} \int_{\mathbb{R}^3} e^{i \lambda \cdot v}e^{-|v|^2/(2t)} Q_{2m}(v_1,v_2,v_3, t) dv.
\end{align*}
Using the bound (\cite{Erd}, p.208),
\begin{equation*}
|H_{2j}(x)| \leq e^{x^2/4}2^{2j}j!, 
\end{equation*}
we can see from \eqref{Hermite} that 
\begin{equation*}
|Q_{2m}(v_1, v_2, v_3, t)| \leq  \frac{m!2^{2m}}{t^m} e^{|v|^2/(4t)}\sum_{j_1+j_2+j_3 = m}1 = \frac{m!2^{2m}}{t^m} e^{|v|^2/(4t)}\frac{(m+2)(m+1)}{2}.
\end{equation*}
Combined with the following bound for the modified Bessel function (see e.g. \cite{Erd}, p.14):
\begin{equation*}
{\it I}_{m-1/2}\left(u\right) \leq \left(\frac{u}{2}\right)^{m-1/2} \frac{e^u}{\Gamma(m+1/2)}, 
\end{equation*}
we get:
\begin{multline*}
e^{-|\lambda|^2t/2} \cosh(\sqrt{|\lambda|^2+1}u)  = \frac{1}{(2\pi t)^{3/2}} \int_{\mathbb{R}} e^{i\langle \lambda, v \rangle} e^{-|v|^2/(2t)} \\ 
\sum_{m \geq 0} \frac{(-1)^m}{m!}\left(\frac{u}{2}\right)^{m+1/2} {\it I}_{m-1/2}\left(\frac{u}{2}\right)Q_{2m}(v_1,v_2,v_3),
\end{multline*}
where the series on the right hand side is absolutely convergent and is bounded by: 
\begin{equation*}
e^{u+|v|^2/(4t)} \sum_{m \geq 0}\frac{u^{2m}(m+1)(m+2)}{2t^m\Gamma(m+1/2)}. 
\end{equation*}
Plugging this Fourier transform on the right hand side of \eqref{NewFor}, we only need to check that Fubini Theorem applies. But, the estimate 
\begin{equation*}
s_{4n+1}(\cosh(\delta)) \leq C \frac{\delta}{\sinh(\delta)} e^{-\delta^2/(2t)}, \quad C, \delta > 0,
\end{equation*}
together with
\begin{equation*}
\cosh^{-1}[\cosh(u)\cosh(r)] \geq \cosh^{-1}\left[\frac{1}{2}(\cosh(u+r)\right] \geq (r+u), \quad r,u \rightarrow +\infty, 
\end{equation*}
shows that $s_{4n+1}(\cosh(u)\cosh(r)) \leq Ce^{-(r+u)^2/(2t)}$. Hence, Fubini Theorem applies which finishes the proof.
\end{proof}

\section{The horizontal heat kernel on the twistor space of $\mathbb{H}H^n$}\label{sec-twistor}

\subsection{Radial part of the sub-Laplacian on $\mathbb{C}H_1^{2n+1}$}

Besides the action of  $\mathbf{SU}(2)$ on $\mathbf{AdS}^{4n+3}(\mathbb{H})$ that induces the quaternionic anti-de Sitter fibration which was studied in the previous sections, we can also consider the action of $\mathbb{S}^1$ on $\mathbf{AdS}^{4n+3}(\mathbb{H})$ that induces the  fibration:
\begin{align*}
\mathbb{S}^1 \to  \mathbf{AdS}^{4n+3}(\mathbb{H})  \to  \mathbb{C}H_1^{2n+1},
\end{align*}
where we simply define $ \mathbb{C}H_1^{2n+1}$ as the complex pseudo-hyperbolic space $ \mathbf{AdS}^{4n+3}(\mathbb{H})  / \mathbb{S}^1$.  Note that the metric on $ \mathbb{C}H_1^{2n+1}$ has signature $(4n,2)$. We can then see $\mathbb{S}^1$ as a subgroup of $\mathbf{SU}(2)$ and have the classical Hopf  fibration
\begin{align*}
 \mathbb{S}^1  \to \mathbf{SU}(2) \to \mathbb{C}P^1.
\end{align*}

We can therefore construct the following commutative fibration diagram

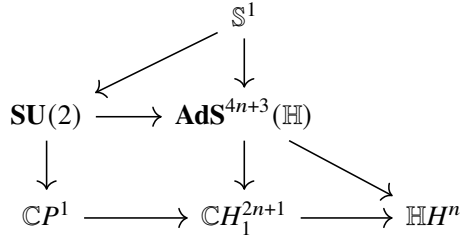
\[
  \begin{tikzcd}
    & \mathbb{S}^1 \arrow[swap]{dl} \arrow{d} & \\
    \mathbf{SU}(2) \arrow{r} \arrow{d}  &  \mathbf{AdS}^{4n+3}(\mathbb{H})   \arrow{d} \arrow[swap]{dr} &  \\
   \mathbb{C}P^1  \arrow{r} & \mathbb{C}H_1^{2n+1} \arrow{r} & \mathbb{H}H^n  
  \end{tikzcd}
\]

The  fibration
\[
\mathbb{C}P^1 \to \mathbb{C}H_1^{2n+1} \to \mathbb{H}H^n
\]
shows that $\mathbb{C}H_1^{2n+1}$ is therefore the twistor space of the quaternionic K\"ahler manifold $\mathbb{H}H^n$.

We consider the sub-Laplacian ${\mathcal{L}}$ on  $\mathbb{C}H_1^{2n+1} $. It is the lift of the Laplace-Beltrami operator of  $\mathbb{H}H^n$. From the above diagram, ${\mathcal{L}}$ is also the projection of the sub-Laplacian $L$ of $\mathbf{AdS}^{4n+3}(\mathbb{H}) $ on $\mathbb{C}H_1^{2n+1} $. As we have seen before,  the radial part of $L$ is 
\[
L=\frac{\partial^2}{\partial r^2}+((4n-1)\coth r+3\tanh r)\frac{\partial}{\partial r}+\tanh^2r \left(\frac{\partial^2}{\partial \eta^2}+2\cot \eta\frac{\partial}{\partial \eta}\right)
\]
where $r$ is the radial coordinate on $\mathbb{H}H^n$ and $\eta$ the radial coordinate on $ \mathbf{SU}(2) $. The operator
\[
\Delta_{\mathbf{SU}(2)}=\frac{\partial^2}{\partial \eta^2}+2\cot \eta\frac{\partial}{\partial \eta}
\]
is the radial part of the Laplace-Beltrami operator on $ \mathbf{SU}(2) $. Besides, it has been proved in Baudoin-Bonnefont (see \cite{BB}), by using the Hopf fibration,
\begin{align*}
\mathbb{S}^1 \to  \mathbf{SU}(2)  \to  \mathbb{C}P^{1}
\end{align*}
we can write
\[
\frac{\partial^2}{\partial \eta^2}+2\cot \eta\frac{\partial}{\partial \eta}=\frac{\partial^2}{\partial \phi^2}+2\cot 2\phi\frac{\partial}{\partial \phi}+(1+\tan^2\phi )\frac{\partial^2}{\partial \theta^2}
\]
where $\phi$ is the radial coordinate on $\mathbb{C}P^{1}$ and $\frac{\partial}{\partial \theta}$ the generator of the action of $\mathbb{S}^1$ on $\mathbf{SU}(2)$.
Therefore the radial part of the sub-Laplacian $\mathcal{L}$ on  $\mathbb{C}H_1^{2n+1}$ is given by 
\begin{equation}\label{L-CP}
\mathcal{L}=\frac{\partial^2}{\partial r^2}+((4n-1)\coth r+3\tanh r)\frac{\partial}{\partial r}+\tanh^2r\left( \frac{\partial^2}{\partial \phi^2}+2\cot 2\phi\frac{\partial}{\partial \phi}\right),
\end{equation}
and the invariant measure, up to a normalization constant,  is $(\sinh r)^{4n-1} (\cosh r)^3 \sin 2\phi dr d\phi$.

\subsection{Integral representation of the subelliptic heat kernel}

From \eqref{L-CP} we notice that $\frac{\partial^2}{\partial \phi^2}+2\cot 2\phi\frac{\partial}{\partial \phi}$ is the radial part of the Laplacian on $\mathbb{C}P^1$.  It is known that the eigenfunction associated to the eigenvalue $-4m(m+1)$ is given by $P_m^{0,0}(\cos2\phi)$ where $P_m^{0,0}$ is the Legendre polynomial
\[
P_m^{0,0}(x)=\frac{(-1)^m}{2^m m!}\frac{d^m}{dx^m}(1-x^2)^{m}.
\]
Moreover, the heat kernel of $\frac{\partial^2}{\partial \phi^2}+2\cot 2\phi\frac{\partial}{\partial \phi}$ is given by
\[
U(t) (\phi_1,\phi_2)=\sum_{m=0}^{+\infty} (2m+1) e^{-4m(m+1)t} P_m^{0,0}(\cos2\phi_1)P_m^{0,0}(\cos2\phi_2).
\]

By using the same methods as before, we obtain:

\begin{theorem}
Let $h_t(r,\phi)$ be the subelliptic heat kernel of $\mathcal L$.
Then, one has for $r \ge 0$ and $ \phi \in [0,\pi)$
 \begin{align*}
 & h_t(r,\phi)\\
 = & \int_{0}^{\infty} (\sinh u)^2 \left\{\sum_{m=0}^{+\infty} (2m+1) e^{-4m(m+1)t} P_m^{0,0}(\cos2\phi )P_m^{0,0}(\cosh 2u)\right\}q_{t, 4n+3}(\cosh r \cosh u) du.
\end{align*}
where $q_{t,4n+3}$
is the heat kernel on the real $4n+3$ dimensional real hyperbolic space as given \eqref{eq-qt-ads-rie}.
\end{theorem}

\chapter{Horizontal Brownian motions of the octonionic Hopf and anti-de Sitter fibrations}\label{Chapter octonionic hopf}

In this chapter, we study the horizontal Brownian motions of the octonionic fibrations. There exist only two octonionic  fibrations: The octonionic Hopf fibration
\[
\mathbb{S}^7 \to \mathbb{S}^{15} \to{\mathbb{O}P}^1,
\]
and the octonionic anti-de Sitter fibration
\begin{equation*}
\mathbb{S}^7 \to \mathbf{AdS}^{15}(\mathbb{O}) \rightarrow \mathbb{O}H^1.
\end{equation*}
The methods of this chapter are quite close to the methods developed in the previous chapters, but the non-associativity of the octonionic multiplication induces some new difficulties.

\section{Horizontal Brownian motion of the octonionic Hopf  fibration}

\subsection{Octonionic Hopf fibration}\label{Octonionic Hopf}

In this section, we  describe the octonionic Hopf fibration.  As in Section \ref{Octonions}, we consider the division algebra of octonions which is described by 
\begin{equation*}
\mathbb{O}=\left\{x=\sum_{j=0}^{7}x_j e_j, x_j\in \mathbb{R} \right\},
\end{equation*}
where the multiplication rules are given by 
\begin{equation*}
e_i e_j=e_j \text{  if  } i=0,
\end{equation*}
\begin{equation*}
e_i e_j=e_i \text{  if  } j=0,
\end{equation*}
\begin{equation*}
e_i e_j=-\delta_{ij}e_0+\epsilon_{ijk}e_k \text{  otherwise},
\end{equation*}
where $\delta_{ij}$ is the Kronecker delta and $\epsilon_{ijk}$ is the completely antisymmetric tensor with value $1$ when $ijk = 123, 145, 176, 246, 257, 347, 365$.
We recall that the octonionic norm is defined for $x \in \mathbb{O}$ by
\begin{equation*}
|x|^2=\sum_{j=0}^{7}x^2_j,
\end{equation*}
and that for every $x,y \in \mathbb O$, $| x y |=|x | \, |y|$.
%

The octonionic projective space $\mathbb{O}P^1$ can be defined as the set of octonionic lines in $\mathbb{O}^{2}$. To parametrize points in $\mathbb{O}P^{1}$, we will use the local affine coordinates given by $w= xy^{-1}$, $(x,y) \in \mathbb{O}^{2}$, $y\neq 0$. Though this only provides a local set of coordinates (we are missing an hypersurface at $\infty$) we will do all our computations in those coordinates and therefore identify $\mathbb{O}P^1$ with $\mathbb O \cup \{ \infty \}$. In affine coordinates, the Riemannian structure of $\mathbb{O}P^1$ is easily worked out from the standard Riemannian structure of the 15-dimensional Euclidean sphere $\mathbb{S}^{15}$. Indeed, if we consider the unit sphere
\[
\bS^{15}=\lbrace (x,y)\in \mathbb{O}^{2}, |x|^2+|y|^2 =1\rbrace,
\]
then the map 
\begin{align*}
\begin{cases}
\bS^{15} -\{y \neq 0 \}  \to \mathbb{O} \\
(x,y) \to xy^{-1}
\end{cases}
\end{align*}
is a submersion with fibers isometric to $\mathbb S^7$. The unique Riemannian metric $h$ on $\mathbb{O}P^1$ that makes this map a Riemannian submersion is the standard Riemannian metric on $\mathbb{O}P^1$.  Equipped with this metric, $\mathbb{O}P^1$ is then isometric to the  sphere $\mathbb S^8$ with radius $1/2$. 

In addition to the above described octonionic Hopf fibration 
\[
\mathbb{S}^7 \rightarrow \mathbb{S}^{15} \rightarrow {\mathbb{O}P}^1,
\]
 on $\mathbb{S}^{15}$ one can also consider the complex Hopf fibration
\[
\mathbb{S}^1 \rightarrow \mathbb{S}^{15} \rightarrow {\mathbb{C}}P^7
\]
and the quaternionic one
\[
\mathbb{S}^3 \rightarrow \mathbb{S}^{15} \rightarrow {\mathbb{H}}P^3.
\]
Similarly to Example \ref{BB fibration}, this would potentially yield two interesting commutative diagrams of homogeneous fibrations:

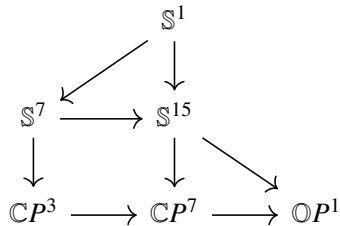
\[
  \begin{tikzcd}
    & \mathbb{S}^1 \arrow[swap]{dl} \arrow{d} & \\
   \mathbb{S}^7 \arrow{r} \arrow{d}  & \mathbb{S}^{15}   \arrow{d} \arrow[swap]{dr} &  \\
   \mathbb{C}P^3  \arrow{r} & \mathbb{C}P^{7} \arrow{r} & \mathbb{O}P^1  
  \end{tikzcd}
\]

and

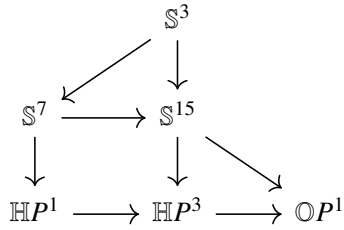
\[
  \begin{tikzcd}
    & \mathbb{S}^3 \arrow[swap]{dl} \arrow{d} & \\
   \mathbb{S}^7 \arrow{r} \arrow{d}  & \mathbb{S}^{15}   \arrow{d} \arrow[swap]{dr} &  \\
   \mathbb{H}P^1  \arrow{r} & \mathbb{H}P^{3} \arrow{r} & \mathbb{O}P^1  
  \end{tikzcd}
\]

However, for topological reasons, those diagrams actually do not exist. Indeed,  in the first diagram the submersion
\[
  \mathbb{C}P^{7} \to \mathbb{O}P^1
\]

does not exist, see \cite{Ucci} and  \cite{AR85} page 258.  In the second diagram the submersion
\[
 \mathbb{H}P^{3}  \to \mathbb{O}P^1
\]
does not exist, see \cite{Ucci} and \cite{Escobales2}.

\subsection{Radial part of the horizontal Laplacian}

The horizontal Laplacian $\Delta_\mathcal{H}$ of the octonionic Hopf fibration is the horizontal lift of the Laplace-Beltrami operator of $\mathbb{O}P^1$. It can be written as
\[
\Delta_\mathcal{H}=\Delta_{\mathbb{S}^{15}}-\Delta_\mathcal{V},
\]
where $\Delta_{\mathbb{S}^{15}}$ is  the Laplace-Beltrami operator on  $\mathbb{S}^{15}$ and $\Delta_\mathcal{V}$ is the vertical Laplacian of the fibration. To study $\Delta_\mathcal{H}$, we will use the following local trivialization of $\mathbb{S}^{15}$:

\begin{equation}\label{trivial octo}
(w,u) \mapsto \left(\frac{  w \, u}{\sqrt{1+{|w|}^2}},\frac{ u}{\sqrt{1+{|w|}^2}}\right)\in \mathbb{S}^{15},
\end{equation}

where $w  \in \mathbb{O}P^1\backslash \{\infty \}$ and $u \in \mathbb S^7 \setminus \{q\}$.  Here $q$ denotes the antipodal point to $p=(1,0,\dots,0)$ in $\mathbb{S}^7$. This parametrizes the open set  $\Omega=\{ (x,y) \in \mathbb{S}^{15}, y  \neq 0,   \frac{y}{{|y|}}  \neq q\}$.

Since  the octonionic multiplication is alternating we have
\[
\pi \left(\frac{  w \, u}{\sqrt{1+{|w|}^2}},\frac{ u}{\sqrt{1+{|w|}^2}}\right)=w
\]
where $\pi$ is the Hopf projection:

\begin{align*}
\bS^{15} \setminus \{y \neq 0 \}  &\to \mathbb{O}P^1 \\
(x,y) &\mapsto xy^{-1}.
\end{align*}
Therefore $u$ is thought of as the fiber coordinate on $\mathbb S^7$.  A function $f$ on $\bS^{15}$ will be called radial (for the octonionic fibration) if  there exists a function $g$ such that
\[
f \left(\frac{  w \, u}{\sqrt{1+{|w|}^2}},\frac{ u}{\sqrt{1+{|w|}^2}}\right)= g(r , \eta)
\]
where $r=\arctan {|w|}\in [0,\pi/2), \eta= \delta_{\mathbb S^7}( p,g )\in [0,\pi)$. Here, we denote by $\delta_{\mathbb S^7} ( p,g )$ the Riemannian distance in $\bS^7$ from $p$ to $g$. 
The variable $r$ can be interpreted as the Riemannian distance on $\mathbb{O}P^1$ from the point $w=0$. The variable $\eta$ can be interpreted  as the Riemannian distance on $\mathbb{S}^7$ from $p$.  We note that, geometrically, the boundary of $\Omega$ corresponds to the boundary values $r=\pi/2$, $\eta=\pi$.

Due to the large group of  symmetries\footnote{It is $\mathbf{Spin}(9) \subset \mathbf{SO}(16)$, see \cite{MR3084336}.} of the octonionic Hopf fibration , if $f$ is a radial function, so is $\Delta_{\mathbb{S}^{15}} f$. The same property holds for the horizontal Laplacian $\Delta_{\mathcal H}$ and one has the following formula:
\begin{prop}\label{prop octo radial}
In the trivialization \eqref{trivial octo}, $\Delta_{\mathcal H}$ acts on functions depending only on $r$ and $u $ as the operator
\begin{equation}\label{eq:5-octo}
\frac{{\partial}^2}{\partial {r}^2}+(7\cot r-7\tan r)\frac{{\partial}}{\partial {r}}+{\tan^2 r}\, \Delta_{\mathbb S^7}.
\end{equation}	

Therefore, the  horizontal Laplacian acts on radial functions as the operator
\begin{equation}\label{eq:6-octo}
\frac{{\partial}^2}{\partial {r}^2}+(7\cot r-7\tan r)\frac{{\partial}}{\partial {r}}+{\tan^2 r}\, \left(\frac{{\partial}^2}{\partial {\eta}^2}+6\cot \eta\frac{{\partial}}{\partial {\eta}} \right).
\end{equation}	
\end{prop}
\begin{proof}
The idea is to compute first how  the Laplace-Beltrami operator on $\mathbb{S}^{15}$ acts on functions depending only on $(r,u)$  and then use the fact that $\Delta_{\mathcal H}=\Delta_{\mathbb{S}^{15}}-\Delta_\mathcal{V}$. 

As a Riemannian manifold, $\mathbb{O}P^1$ is a compact rank one symmetric space. General formulas for the radial parts of Laplacians on rank one symmetric spaces are well-known (see Appendix 3). In particular, the radial part of the  Laplace-Beltrami operator on $\mathbb{O}P^1$ is
\[
\frac{{\partial}^2}{\partial {r}^2}+(7\cot r-7\tan r)\frac{{\partial}}{\partial {r}}.
\]
We note that $\mathbb{O}P^1$ is isometric to the 8-dimensional sphere with radius $1/2$ and,  thanks to the identity $\cot r-\tan r=2 \cot (2r)$ the   Laplace-Beltrami operator on $\mathbb{O}P^1$ might also, if needed, be written as 
\[
\frac{{\partial}^2}{\partial {r}^2}+ 14 \cot (2r) \frac{{\partial}}{\partial {r}}.
\]
On the other hand, the variable $u \in \mathbb S^7$ is a fiber coordinate, and the fibers of the Hopf map are totally geodesic submanifolds isometric to $\mathbb S^7$. Therefore $\Delta_\mathcal{V}$ acts in the local trivialization \eqref{trivial octo} as $\Delta_{\mathbb S^7}$.
From the symmetries of the fibration, one deduces that in the trivialization \eqref{trivial octo}, $\Delta_{\mathbb S^{15}}$ acts on functions depending only on $r$ and $u $ as the operator
\[
\frac{{\partial}^2}{\partial {r}^2}+(7\cot r-7\tan r)\frac{{\partial}}{\partial {r}} +g(r) \Delta_{\mathbb S^7}
\]
for some function $g$ to be computed.
One can compute $g$ by observing that on $\mathbb{S}^{15}$  the Riemannian distance $\delta$ from the north pole (the point with octonionic coordinates $(0,1)=(0,p)\in \mathbb{O}^2$) to the point
\[
 \left(\frac{  w \, u}{\sqrt{1+{|w|}^2}},\frac{ u}{\sqrt{1+{|w|}^2}}\right)
 \]
is given by
\[
\cos \delta = \cos r \cos \eta.
\]
This is because the right and left hand sides of the above equality are both the 9th Euclidean coordinate of
\[
 \left(\frac{  w \, u}{\sqrt{1+{|w|}^2}},\frac{ u}{\sqrt{1+{|w|}^2}}\right)
 \]
and
\[
\frac{\cos \eta}{\sqrt{1+{|w|}^2}} =\frac{\cos \eta}{\sqrt{1+\tan^2 r}}=\cos r \cos \eta.
\]
From the formula for the radial part of Laplacian on ${\mathbb{S}^{15}}$ starting from the north pole,  we can compute
\begin{align*}
\Delta_{\mathbb{S}^{15}} (\cos \delta ) &=\left(\frac{\partial^2}{\partial \delta^2}+14\cot{\delta}\frac{\partial}{\partial \delta}\right) \cos \delta \\
&=-15 \cos \delta.
\end{align*}
Using the other representation of $\Delta_{\mathbb{S}^{15}}$, together with the fact that the radial part of the Laplace-Beltrami operator on $\mathbb{S}^{7}$ is
\[
\frac{{\partial}^2}{\partial {\eta}^2}+6\cot \eta\frac{{\partial}}{\partial {\eta}}
\]
one deduces
\begin{align*}
& \Delta_{\mathbb{S}^{15}} (\cos \delta ) \\
=&\left(\frac{{\partial}^2}{\partial {r}^2}+(7\cot r-7\tan r)\frac{{\partial}}{\partial {r}} +g(r) \left(\frac{{\partial}^2}{\partial {\eta}^2}+6\cot \eta\frac{{\partial}}{\partial {\eta}} \right) \right)\cos r \cos \eta .
\end{align*}
Therefore,
\[
\left(\frac{{\partial}^2}{\partial {r}^2}+(7\cot r-7\tan r)\frac{{\partial}}{\partial {r}} +g(r) \left(\frac{{\partial}^2}{\partial {\eta}^2}+6\cot \eta\frac{{\partial}}{\partial {\eta}} \right) \right)\cos r \cos \eta=-15 \cos r \cos \eta.
\]
After a straightforward computation, this yields $g(r)=\frac{1}{\cos^2 r}$ and therefore in the trivialization \eqref{trivial octo}, $\Delta_{\mathcal H}$ acts on functions depending only on $r$ and $u $ as the operator
\begin{equation*}
\frac{{\partial}^2}{\partial {r}^2}+(7\cot r-7\tan r)\frac{{\partial}}{\partial {r}}+{\tan^2 r}\, \Delta_{\mathbb S^7}.
\end{equation*}	 
\end{proof}

As a consequence of the previous result, we can check that the Riemannian measure of $\mathbb{S}^{15}$ in the coordinates $(r,\eta)$,  is given by 
\begin{equation}\label{eq:5}
d{\mu}=\frac{56{\pi}^7}{\Gamma(8)}\sin^7 r \cos^7r\sin^6 \eta dr d\eta,
\end{equation}
where the normalization constant is chosen in such a way that
\begin{equation*}
\int_{0}^{\pi}\int_{0}^{\frac{\pi}{2}}d{\mu}=\mathrm{Vol} (\mathbb{S}^{15})=\frac{2\pi^8}{\Gamma(8)}.
\end{equation*}

\subsection{Horizontal Brownian motion and stochastic area}

Let $(X(t))_{t \ge 0}$ be the horizontal Brownian motion of the octonionic fibration issued from the north pole $(0,1) \in \mathbb{S}^{15}$. From Proposition \ref{prop octo radial}, one can write 

\[
X(t)=\left(\frac{  w(t) \, u(t)}{\sqrt{1+{|w(t)|}^2}},\frac{ u(t)}{\sqrt{1+{|w(t)|}^2}}\right)
\]
where $w(t)$ is a Brownian motion on $\mathbb{O}P^1$,   $r(t)=\arctan {|w(t)|}$ is a Jacobi diffusion with the generator 
\[
\mathcal{L}^{3,3} = \frac{1}{2}\left(\frac{{\partial}^2}{\partial {r}^2}+(7\cot r-7\tan r)\frac{{\partial}}{\partial {r}}\right)
\]
and $u(t)$ is an independent Brownian motion on $\mathbb S^7$ timed changed with the functional $\int_0^t \tan^2 r(s)ds$.

It follows that in distribution, the stochastic area is therefore given by
\[
\mathfrak{a}(t)=\beta \left( \int_0^t \tan^2 r(s) ds  \right),
\]
where $\beta$ is a Brownian motion in $\mathbb R^7$ independent from $r(t)$.

Mimicking the proof of Theorem \ref{eq-ft-cond}, we readily obtain:
\[
\mathbb{E} \left( e^{i \lambda\cdot \mathfrak{a}(t)}\mid r(t)=r \right) = 
\frac{e^{-4(\sqrt{\lambda^2+9} - 3)t}}{[\cos(r)]^{(\sqrt{\lambda^2+9} - 3)}}\quad 
\frac{q_t^{3,\sqrt{\lambda^2+9}}(0,r)}{q_t^{3,3}(0,r)}
\]
where we recall that $q_t^{\alpha, \beta}$ is the heat kernel of the Jacobi operator $\mathcal{L}^{\alpha,\beta}$. Consequently, 
\[
\mathbb{E} \left(e^{i \lambda\cdot \mathfrak{a}(t)}\right) = 
\frac{e^{-4(\sqrt{\lambda^2+9} - 3)t}}{[\cos(r)]^{(\sqrt{\lambda^2+9} - 3)}}
\quad 
\int_0^{\pi/2} \frac{q_t^{3,\sqrt{\lambda^2+9}}(0,r)}{[\cos(r)]^{(\sqrt{\lambda^2+9} - 3)}}dr,
\]
whence we deduce the following limiting result 
\[
\lim_{t \rightarrow +\infty} \mathbb{E} \left(e^{i \lambda\cdot \mathfrak{a}(t)/\sqrt{t}}\right) = e^{-4|\lambda|}
\]
which implies that $\mathfrak{a}(t)/\sqrt{t}$ converges to a Cauchy distribution. 

\subsection{Formulas for the horizontal heat kernel}

In this section, we derive the spectral decomposition of the subelliptic heat kernel of the heat semigroup $P_t=e^{t\Delta_{\mathcal H}}$ issued from the north pole. From Proposition \ref{prop octo radial} it is enough to compute the heat kernel $p_t(r,\eta)$ of the operator
\begin{equation}\label{eq-rad-octo-lapla}
\tilde{L}:=\frac{{\partial}^2}{\partial {r}^2}+(7\cot r-7\tan r)\frac{{\partial}}{\partial {r}}+{\tan^2 r}\, \left(\frac{{\partial}^2}{\partial {\eta}^2}+6\cot \eta\frac{{\partial}}{\partial {\eta}} \right).
\end{equation}
We will need the Jacobi polynomial
\begin{equation}\label{eq-jac-poly-3}
P^{3,m+3}_k(x)=\frac{(-1)^k}{2^k k! (1-x)^{3}(1+x)^{m+3}}\frac{{d}^k}{d {x}^k}((1-x)^{k+3}(1+x)^{m+3+k}).
\end{equation}

\begin{proposition}
For $t>0$, $r\in [0,\frac{\pi}{2})$, $\eta \in [0,\pi)$, the horizontal heat kernel is given by 
\begin{equation*}
p_t(r,\eta)=\sum_{m=0}^{\infty}\sum_{k=0}^{\infty}\alpha_{k,m} h_m(\eta)e^{-4(2m+k(k+m+7))t}(\cos r)^m P^{3,m+3}_k(\cos 2r),
\end{equation*}
where $\alpha_{k,m}=\frac{96}{\pi^{8}}(m+3)(2k+m+7)\left(\begin{array}{c}
k+m+6 \\ k+m+3
\end{array}\right)\left(\begin{array}{c}
m+5 \\ m
\end{array}\right)$ , and
\[
h_m(\eta)=\frac{\Gamma(7/2)}{\sqrt{\pi}\Gamma(3)}\int_{0}^{\pi}{(\cos \eta + \sqrt{-1}\sin \eta \cos \varphi)}^m{\sin^5 \varphi}d \varphi
\]
 is the normalized  eigenfunction of $\frac{{\partial}^2}{\partial {\eta}^2}+6\cot \eta\frac{{\partial}}{\partial {\eta}}$ that is associated to the eigenvalue $-m(m+6)$.
 \end{proposition}
\begin{proof}
We expand $p_t(r,\eta)$ in spherical harmonics as follows,
\begin{equation*}
p_t(r,\eta)=\sum_{m=0}^{\infty}h_m(\eta)\phi_m(t,r),
\end{equation*}
where $h_m(\eta)$ is the eigenfunction of $\frac{{\partial}^2}{\partial {\eta}^2}+6\cot \eta\frac{{\partial}}{\partial {\eta}}$ which is associated to the eigenvalue $-m(m+6)$. More precisely, $h_m(\eta)$ is given by
\begin{equation*}
h_m(\eta)=\frac{\Gamma(7/2)}{\sqrt{\pi}\Gamma(3)}\int_{0}^{\pi}{(\cos \eta + \sqrt{-1}\sin \eta \cos \varphi)}^m{\sin^5 \varphi}\,d \varphi.
\end{equation*}

To determine $\phi_m$, we use $\frac{{\partial}}{\partial {t}}p_t=\widetilde{L} p_t$, and find 
\begin{equation*}
\frac{{\partial}\phi_m}{\partial {t}} = \frac{{\partial}^2\phi_m}{\partial {r}^2}+(7\cot r-7\tan r)\frac{{\partial}\phi_m}{\partial {r}}-m(m+6)\,{\tan^2 r}\phi_m. 
\end{equation*}
Let $\phi_m:=e^{-8mt}(\cos r)^m \psi_m$. This substitution gives 
\begin{equation*}
\frac{{\partial}\psi_m}{\partial {t}} = \frac{{\partial}^2\psi_m}{\partial {r}^2}+(7\cot r-(2m+7)\tan r)\frac{{\partial}\psi_m}{\partial {r}}. 
\end{equation*}
Letting $\psi_m(t,r):=g_m(t,\cos 2r)$. Then the previous equation becomes 
\begin{equation*}
\frac{{\partial}g_m}{\partial {t}} =4(1-x^2) \frac{{\partial}^2 g_m}{\partial {x}^2}+4((m-(m+8)x)\frac{{\partial}g_m}{\partial {x}}. 
\end{equation*}
We get $\frac{{\partial}g_m}{\partial {t}}=4\Psi_m(g_m)$, where 
\begin{equation*}
\Psi_m=(1-x^2)\frac{{\partial}^2}{\partial {x}^2}+((m-(m+8)x)\frac{{\partial}}{\partial {x}}.
\end{equation*}
Note that the equation 
\begin{equation*}
\Psi_m(g_m)+k(k+m+7)g_m=0
\end{equation*}
is a Jacobi differential equation for all $k \geq 0$. We denote the eigenvector of $\Psi_m$ corresponding to the eigenvalue $-k(k+m+7)$ by $P^{3,m+3}_k(x)$ which is given as in \eqref{eq-jac-poly-3}.
We can therefore write the spectral decomposition of $p_t$ as 
\begin{equation*}
p_t(r,\eta)=\sum_{m=0}^{\infty}\sum_{k=0}^{\infty}\alpha_{k,m} h_m(\eta)e^{-4(k(k+m+7)+2m)t}\cos^m r P^{3,m+3}_k(\cos 2r). 
\end{equation*}
where the constants $\alpha_{k,m}$'s have to be determined by taking into account the initial condition. 

Note that $((1+x)^{\frac{m+3}{2}} P^{3,m+3}_k(x))_{k\geq 0}$ is a complete orthogonal basis of the Hilbert space $L^2([-1,1],(1-x)^{3}dx)$, more precisely 
\begin{equation*}
\int_{-1}^{1}P^{3,m+3}_k(x)P^{3,m+3}_l(x)(1-x)^{3}(1+x)^{m+3}dx
=\frac{2^{m+7}}{2k+m+7}\frac{\Gamma(k+4)\Gamma(k+m+4)}{\Gamma(k+m+7)\Gamma(k+1)}\delta_{kl}. 
\end{equation*}

On the other hand, $(h_m(\eta))_{m\ge 0}$ are the eigenfunctions of the self-adjoint operator $\frac{{\partial}^2}{\partial {\eta}^2}+6\cot \eta\frac{{\partial}}{\partial {\eta}}$ and thus form a complete orthonormal basis of $L^2 ( [0, \pi], (\sin \eta)^6  d\eta)$.

Thus, using the fact that $\left( \frac{1+\cos 2r}{2}\right)^{1/2} =\cos r$, for a smooth function $f(r,\eta)$, we can write 
\begin{equation*}
f(r,\eta)=\sum_{m=0}^{\infty}\sum_{k=0}^{\infty}\beta_{k,m} h_m(\eta)(\cos r )^mP^{3,m+3}_k(\cos 2r)
\end{equation*}
where the $\beta_{k,m}$'s are constants. We obtain then 
\begin{equation*}
f(0,0)=\sum_{m=0}^{\infty}\sum_{k=0}^{\infty}\beta_{k,m}P^{3,m+3}_k(1)
\end{equation*}
and we observe that $P^{3,m+3}_k(1)=\left(\begin{array}{c}
3+k \\ k
\end{array}\right)$. From \eqref{eq:5}, the measure $d{\mu}$ is given in cylindrical coordinates by 
\begin{equation*}
d{\mu} = \frac{56{\pi}^7}{\Gamma(8)}(\sin r)^{7}(\cos r)^{7}(\sin \eta)^{6}dr d\eta.
\end{equation*}
Moreover, we have
\begin{align*}
&\int_{0}^{\pi}\int_{0}^{\frac{\pi}{2}}p_t(r,\eta)f(-r,-\eta)d{\mu}\\
=&\frac{56{\pi}^7}{\Gamma(8)}\sum_{m=0}^{\infty}\sum_{k=0}^{\infty}\alpha_{k,m}\beta_{k,m}e^{-4(k(k+m+7)+2m)t}\int_{0}^{\pi} {h_m(\eta)}^2 \sin^6{\eta} d\eta \\ 
& \quad\quad\cdot \int_{0}^{\frac{\pi}{2}}{(\cos r)}^{2m+7}{P^{3,m+3}_k (\cos(2r))}^{2}{(\sin r)}^7 dr\\
=&\frac{56{\pi}^7}{\Gamma(8)}\sum_{m=0}^{\infty}\sum_{k=0}^{\infty}\frac{\alpha_{k,m}\beta_{k,m}e^{-4(k(k+m+7)+2m)t}}{2k+m+7}\ \frac{\sqrt{\pi} \Gamma(7/2)}{\Gamma(4)}\frac{6!m!}{(m+5)!}\frac{\Gamma(k+4)\Gamma(k+m+4)}{\Gamma(k+m+7)\Gamma(k+1)},
\end{align*}

where we used 
\begin{equation*}
\int_{0}^{{\pi}} {h_m(\eta)}^2 \sin^6{\eta} d\eta=\frac{\sqrt{\pi} \Gamma(7/2)}{\Gamma(4)(2m+6)}\frac{6!m!}{(m+5)!}.
\end{equation*}

From
\begin{equation*}
\lim_{t \rightarrow 0}\int_{0}^{\pi}\int_{0}^{\frac{\pi}{2}} p_t f d\mu = f(0,0),
\end{equation*}
we obtain   $\alpha_{k,m}=\frac{96}{\pi^{8}}(m+3)(2k+m+7)\left(\begin{array}{c}
k+m+6 \\ k+m+3
\end{array}\right)\left(\begin{array}{c}
m+5 \\ m
\end{array}\right)$ and the proof is completed.
\end{proof}	

\begin{proposition}\label{prop1octo}
Let $p_t^{Q}$ denote the  horizontal heat kernel for the quaternionic Hopf fibration
\[
\mathbb{S}^3 \to \mathbb{S}^{11}  \to \mathbb{H}P^2. 
\]
Then for $r\in [0,\frac{\pi}{2}),\eta\in [0,\pi)$, 
\begin{equation}\label{eq:1}
p_t(r,\eta)=\frac{144e^{16t}}{{\pi}^2\cos^2 r}\left(\frac{1}{\sin^2 \eta}\frac{\partial^2}{\partial \eta^2}p_t^{Q}-\frac{\cos \eta}{\sin^3 \eta}\frac{\partial}{\partial \eta}p_t^{Q}\right).
\end{equation}
\end{proposition}
\begin{proof}
From the Rodrigues formula, one can verify that 
\begin{equation}\label{eq:51}
\frac{1}{8}\left(\frac{1}{\sin^2 \eta}\frac{\partial^2}{\partial \eta^2}-\frac{\cos \eta}{\sin^3 \eta}\frac{\partial}{\partial \eta}\right)\frac{\sin(m+1)\eta}{\sin \eta}=\left(\begin{array}{c}
m+3 \\ m-2
\end{array}\right) h_{m-2}(\eta).
\end{equation}

On the other hand, on the $11$-dimensional quaternionic sphere $\mathbb{S}^{11}$, the spectral decomposition of the quaternionic subelliptic heat kernel $p_t^{Q}(r,\eta)$ is given by:
\begin{equation*}
p^{Q}_t(r,\eta)=\sum_{m=0}^{\infty}\sum_{k=0}^{\infty}\beta_{k,m} e^{-4(k(k+m+5)+2m)t}\frac{\sin (m+1)\eta}{\sin \eta}\cos^m r P^{3,m+1}_k(\cos 2r),
\end{equation*}
where 
\begin{equation*}
\beta_{k,m}=\frac{\Gamma(4)}{2\pi^{6}}(2k+m+5)(m+1)\left(\begin{array}{c}
k+m+4 \\ k+m+1
\end{array}\right).
\end{equation*}

Note that the octonionic subelliptic heat kernel $p_t(r,\eta)$ in the previous proposition which was given by:
\begin{equation*}
p_t(r,\eta)=\sum_{m=0}^{\infty}\sum_{k=0}^{\infty}\alpha_{k,m} h_m(\eta)e^{-4(k(k+m+7)+2m)t}\cos^m r P^{3,m+3}_k(\cos 2r),
\end{equation*}
where $\alpha_{k,m}=\frac{96}{\pi^{8}}(m+3)(2k+m+7)\left(\begin{array}{c}
k+m+6 \\ k+m+3
\end{array}\right)\left(\begin{array}{c}
m+5 \\ m
\end{array}\right)$.

From those two expressions of the heat kernels with \eqref{eq:51}, we can easily deduce that 
\begin{equation*}
p_t(r,\eta)=\frac{144e^{16t}}{{\pi}^2\cos^2 r}\left(\frac{1}{\sin^2 \eta}\frac{\partial^2}{\partial \eta^2}p_t^{Q}-\frac{\cos \eta}{\sin^3 \eta}\frac{\partial}{\partial \eta}p_t^{Q}\right).
\end{equation*}
\end{proof}



 Let $q_t$ be the Riemannian radial heat kernel on $\mathbb{S}^{11}$. For later use, see Appendix 2 for the details, we record here that:

(1) The spectral decomposition of $q_t$ is given by 
\begin{equation*}
q_t(\cos \delta)=\frac{\Gamma(5)}{2{\pi}^{6}}\sum_{m=0}^{\infty}(m+5)e^{-m(m+10)t}C^{5}_m(\cos \delta),
\end{equation*}
where $\delta$ is the Riemannian distance from the north pole and 
\begin{equation*}
C^{5}_m(x)=\frac{(-1)^m}{2^m}\frac{\Gamma(m+10)\Gamma(\frac{15}{2})}{\Gamma(10)\Gamma(m+1)\Gamma(m+\frac{11}{2})}\frac{1}{(1-x^2)^{\frac{5}{2}}}\frac{d^m}{d x^m}(1-x^2)^{m+9/2}
\end{equation*}
is a Gegenbauer polynomial. 

(2) 
\begin{equation}\label{heat kernel s11}
q_t(\cos \delta)=e^{25 t}\left(-\frac{1}{2\pi \sin \delta }\frac{\partial}{\partial \delta}\right)^{5}V,
\end{equation}
where $V(t,\delta)=\frac{1}{\sqrt{4\pi t}}\sum_{k\in \mathbb{Z}}e^{-\frac{(\delta-2k\pi)^2}{4t}}$. Note that the last formula shows that $\delta \to q_t(\cos \delta)$ can analytically be extended to $\delta \in \mathbb C$.

\begin{proposition}\label{prop:4}
For $r\in [0,\frac{\pi}{2}),\eta\in [0,\pi)$, we have:
\begin{equation}\label{eq:2}
p_t(r,\eta)=\frac{36e^{15t}}{{\pi}^2\sqrt{\pi t}\cos^2 r}\int_{0}^{\infty}\frac{g_t(\eta,y)}{t^2} e^{{-}\frac{y^2-{\eta}^2}{4t}}q_t(\cos r \cosh y)\sinh y dy,
\end{equation}
where 
\begin{align*}
g_t(\eta,y)&=\csc ^3(\eta)\left( \cos \left(\frac{\eta y}{2 t}\right)2 y (\eta-3 t \cot \eta)\right.\\
&+\left.\sin \left(\frac{\eta y}{2 t}\right) \left(8 t^2 \cot ^2 \eta+4 t^2 \csc ^2 \eta-6 t \eta \cot \eta+2 t+\eta^2-y^2\right)\right).
\end{align*}
\end{proposition}
\begin{proof}
Note that because of the rapid decay of the integrand and thanks to formula \eqref{heat kernel s11}, we can differentiate under the integral sign. Now from the Proposition~\ref{prop1octo}, we know 	
\begin{equation*}
p_t(r,\eta)=\frac{144e^{16t}}{{\pi}^2\cos^2 r}\left(\frac{1}{\sin^2 \eta}\frac{\partial^2}{\partial \eta^2}p_t^{Q}-\frac{\cos \eta}{\sin^3 \eta}\frac{\partial}{\partial \eta}p_t^{Q}\right).
\end{equation*}
From Proposition \ref{prop-pt} we have
\begin{equation*}
p^{Q}_t(r,\eta)=\frac{e^{-t}}{\sqrt{\pi t}}\int_{0}^{\infty}\frac{\sinh y \sin{\frac{\eta y}{2t}}}{\sin \eta} e^{{-}\frac{y^2-{\eta}^2}{4t}}q_t(\cos r \cosh (y))dy,
\end{equation*}
Plugging in those two ingredients gives the desired result. 
\end{proof}

\section{Horizontal Brownian motion of the octonionic anti-de Sitter fibration}

\subsection{The octonionic anti-de Sitter fibration}

In this section, we first describe the octonionic anti-de Sitter fibration
\begin{equation*}
\mathbb{S}^7 \to \mathbf{AdS}^{15}(\mathbb{O}) \to \mathbb{O}H^1.
\end{equation*}

The octonionic anti-de Sitter space $\mathbf{AdS}^{15}(\mathbb{O})$ is the quadric defined as the pseudo-hyperbolic space by:
\begin{equation*}
\mathbf{AdS}^{15}(\mathbb{O})=\{(x,y)\in\mathbb{O}^2, |(x,y)|^2_{\mathbb{O}}=-1 \},
\end{equation*}
where 
\begin{equation*}
|(x,y)|^2_{\mathbb{O}}:=|x|^2-|y|^2.
\end{equation*}

In real coordinates $x=\sum_{j=0}^{7}x_j e_j$, $y=\sum_{j=0}^{7}y_j e_j$, one has
\begin{equation*}
|(x,y)|^2_{\mathbb{O}}=x_0^2+\cdots+x_7^2-y_0^2-\cdots-y_7^2,
\end{equation*}
and as such, $\mathbf{AdS}^{15}(\mathbb{O})$ inherits from the Lorentzian metric with real signature $(8,8)$ on $\mathbb O \times \mathbb O$ a semi Riemannian structure with real signature $(8,7)$.

As a manifold, the octonionic hyperbolic  space $\mathbb{O}H^1$ can be defined as the unit ball in $\mathbb O$ and its Riemannian metric is such that the map 
\begin{align*}
\mathbf{AdS}^{15}(\mathbb{O})  &\to \mathbb{O}H^1 \\
(x,y) &\mapsto y^{-1}x
\end{align*}
is a semi-Riemannian submersion with fibers isometric to $\mathbb S^7$.

\subsection{Radial part of the horizontal Laplacian}

The sub-Laplacian $L$ on $\mathbf{AdS}^{15}(\mathbb{O})$ that we are interested in  is the horizontal Laplacian of the Riemannian submersion $\pi:\mathbf{AdS}^{15}(\mathbb{O}) \to \mathbb{O}H^1$, i.e the horizontal lift of the Laplace-Beltrami operator of $\mathbb{O}H^1$. It can be written as
\begin{equation}\label{eq:sub_Laplacian}
L=\square_{\mathbf{AdS}^{15}(\mathbb{O})}+\Delta_\mathcal{V},
\end{equation}
where $\square_{\mathbf{AdS}^{15}(\mathbb{O})}$ is the d'Alembertian, i.e., the Laplace-Beltrami operator of the pseudo-Riemannian metric and $\Delta_\mathcal{V}$ is the vertical Laplacian. Since the fibers of $\pi$ are totally geodesic and isometric to $\mathbb{S}^7\subset \mathbf{AdS}^{15}(\mathbb{O})$, we note that $\square_{\mathbf{AdS}^{15}(\mathbb{O})}$ and $\Delta_\mathcal{V}$ are commuting operators, and we can identify
\begin{equation}\label{eq:vertical_Laplacian}
\Delta_\mathcal{V}={\Delta}_{\mathbb{S}^{7}}.
\end{equation}
To study $L$, we introduce a set of coordinates  which provide an explicit local trivialization of the octonionic anti-de Sitter fibration. Take local coordinates $w\in \mathbb{O}H^1$ and $(\theta_1, ..., \theta_7)\in \mathbb S^7$, where $w$ is the local inhomogeneous coordinate on $\mathbb{O}H^1$ given by $w=y^{-1} x$, where $x,y\in \mathbb{O}$. Consider the north pole $p \in \mathbb{S}^7$ and take $Y_1, ..., Y_7$ to be a local frame of $T_p \mathbb{S}^7$, and denote by $\exp_p$ the Riemannian exponential map at $p$ on $\mathbb{S}^7$. Then the cylindrical coordinates we work with are given by 
\begin{align*}
\psi : \mathbb{O}H^1 \times \mathbb{S}^7 &\rightarrow \mathbf{AdS}^{15}(\mathbb{O})\\
(w,\theta_1, \dots, \theta_7) &\mapsto \left(\frac{ \exp_p(\sum_{i=1}^{7}\theta_i Y_i) w}{\sqrt{1-{\rho}^2}},\frac{ \exp_p(\sum_{i=1}^{7}\theta_i Y_i)}{\sqrt{1-{\rho}^2}}\right)
\end{align*}
where $\rho={{|w|}}$.

A function $f$ on $\mathbf{AdS}^{15}(\mathbb{O})$ is called radial cylindrical if it only depends on  the two coordinates $(\rho,\eta)\in [0,1)\times [0,\pi)$ where $\eta=\sqrt{\sum_{i=1}^{7}\theta^2_i}$. More precisely $f$  is radial cylindrical if there exists a function $g$ so that
\begin{equation*}
f \left(\frac{ \exp_p(\sum_{i=1}^{7}\theta_i Y_i) w}{\sqrt{1-{\rho}^2}},\frac{ \exp_p(\sum_{i=1}^{7}\theta_i Y_i)}{\sqrt{1-{\rho}^2}}\right)=g \left( \rho, \eta \right).
\end{equation*}

We denote by $\mathcal{D}$ the space and smooth and compactly supported functions on $[0,1) \times [0,\pi)$. Then the radial part of $L$ is defined as the operator $\widetilde{L}$  such that for any $f\in \mathcal{D}$, we have
\begin{equation*}
L(f \circ \psi)=(\widetilde{L}f)\circ \psi.
\end{equation*}

We now compute $\widetilde{L}$ in cylindrical coordinates.

\begin{proposition}\label{radial octo desi}
	The radial part of the sub-Laplacian on $\mathbf{AdS}^{15}(\mathbb{O})$ is given in the coordinates $(r,\eta)$ by the operator
\begin{equation}\label{eq:5}
\widetilde{L}=\frac{{\partial}^2}{\partial {r}^2}+(7\coth r+7\tanh r)\frac{{\partial}}{\partial {r}}+{\tanh^2 r}\left(\frac{{\partial}^2}{\partial {\eta}^2}+6\cot \eta\frac{{\partial}}{\partial {\eta}}\right).
\end{equation}	
where $r=\tanh^{-1}\rho$ is the Riemannian distance on $\mathbb{O}H^1$ from the origin. 
\end{proposition}
\begin{proof}
 Note that the radial part  of the Laplace-Beltrami operator on the octonionic hyperbolic space $\mathbb{O}H^1$ is
\begin{equation}\label{eq:LB_on_OH^1}
\widetilde{\Delta}_{\mathbb{O}H^1}=\frac{{\partial}^2}{\partial {r}^2}+(7\coth r+7\tanh r)\frac{{\partial}}{\partial {r}},
\end{equation}
and the radial part of the Laplace-Beltrami operator on $\mathbb{S}^{7}$ is
\begin{equation}\label{eq:LB_on_S^7}
\widetilde{\Delta}_{\mathbb{S}^7}=\frac{{\partial}^2}{\partial {\eta}^2}+6\cot \eta\frac{{\partial}}{\partial {\eta}}.
\end{equation}
Using the symmetries of the fibration the radial part of the d'Alembertian is given by
\begin{equation*}
\widetilde{\square}_{\mathbf{AdS}^{15}(\mathbb{O})}=\frac{{\partial}^2}{\partial {r}^2}+(7\coth r+7\tanh r)\frac{{\partial}}{\partial {r}} +g(r) \left(\frac{{\partial}^2}{\partial {\eta}^2}+6\cot \eta\frac{{\partial}}{\partial {\eta}} \right),
\end{equation*}
for some function $g$ to be computed.
In order to compute $g$, we will consider the analytic change of variables $\tau : (r,\eta) \rightarrow (r,i\eta)$ that will be applied on functions of the type $f(r,\eta)=h(r)e^{-i\lambda \eta}$, with $h$ smooth and compactly supported on $[0,\infty)$ and $\lambda>0$. Then we have 
\begin{equation*}
{\widetilde{\square}_{\mathbf{AdS}^{15}(\mathbb{O})}}(f \circ \tau)=({\widetilde{\Delta}_{{{H}^{15}}}}f) \circ \tau,
\end{equation*}
where 
\begin{equation*}
\widetilde{\Delta}_{{{H}^{15}}}=\frac{{\partial}^2}{\partial {r}^2}+(7\coth r+7\tanh r)\frac{{\partial}}{\partial {r}} -g(r) \left(\frac{{\partial}^2}{\partial {\eta}^2}+6\coth \eta\frac{{\partial}}{\partial {\eta}} \right).
\end{equation*}
is the Riemannian Laplacian on the real hyperbolic space ${{H}^{15}}$.

Now observe that on ${{H}^{15}}$ the Riemannian distance $\delta$ from the point  with octonionic coordinates $(0,1)$ is given by
\begin{equation*}
\cosh \delta = \cosh r \cosh \eta,
\end{equation*}
and we have 
\begin{equation*}
\widetilde{\Delta}_{{{H}^{15}}} =\left(\frac{\partial^2}{\partial \delta^2}+14\coth{\delta}\frac{\partial}{\partial \delta}\right) \cosh \delta=15 \cosh \delta.
\end{equation*}
Hence one deduces
\begin{equation*}
\left(\frac{{\partial}^2}{\partial {r}^2}+(7\coth r+7\tanh r)\frac{{\partial}}{\partial {r}} +g(r) \left(\frac{{\partial}^2}{\partial {\eta}^2}+6\coth \eta\frac{{\partial}}{\partial {\eta}} \right) \right)\cosh r \cosh \eta =15 \cosh r \cosh \eta.
\end{equation*}

After a straightforward computation, this yields $g(r)=-\frac{1}{\cosh^2 r}$ and therefore
\begin{align}\label{eq:LB_on_ADS^15}
\widetilde{\square}_{\mathbf{AdS}^{15}(\mathbb{O})}=&\frac{{\partial}^2}{\partial {r}^2}+(7\coth r+7\tanh r)\frac{{\partial}}{\partial {r}} -\frac{1}{\cosh^2 r} \left(\frac{{\partial}^2}{\partial {\eta}^2}+6\cot \eta\frac{{\partial}}{\partial {\eta}} \right)\\
=&\widetilde{\Delta}_{\mathbb{O}H^1}-\frac{1}{\cosh^2 r}\widetilde{\Delta}_{\mathbb{S}^7}.
\end{align}

Finally, to conclude, one notes that the sub-Laplacian $L$ is given by the difference between the Laplace-Beltrami operator of $\mathbf{AdS}^{15}(\mathbb{O})$ and the vertical Laplacian.  Therefore by \eqref{eq:sub_Laplacian} and \eqref{eq:vertical_Laplacian},
\begin{align*}
\widetilde{L} &=\widetilde{\square}_{\mathbf{AdS}^{15}(\mathbb{O})}+ \widetilde{\Delta}_{\mathbb{S}^{7}}\\
&= \frac{{\partial}^2}{\partial {r}^2}+(7\coth r+7\tanh r)\frac{{\partial}}{\partial {r}}+{\tanh^2 r}\left(\frac{{\partial}^2}{\partial {\eta}^2}+6\cot \eta\frac{{\partial}}{\partial {\eta}}\right).
\end{align*}
\end{proof}
	
\begin{remark}
As a consequence of the previous result, we can check that the Riemannian measure of $\mathbf{AdS}^{15}(\mathbb{O})$ in the coordinates $(r,\eta)$, which is the symmetric and invariant measure for $\widetilde{L}$ is given by 
\begin{equation}\label{eq:measure}
d{\mu}=\frac{56{\pi}^7}{\Gamma(8)}\sinh^7 r \cosh^7r\sin^6 \eta dr d\eta.
\end{equation}
\begin{equation*}
\end{equation*}
\end{remark}

\subsection{Stochastic area in the octonionic hyperbolic  space $\mathbb O H^1$}

Let $(X(t))_{t \ge 0}$ be the horizontal Brownian motion of the octonionic anti-de Sitter fibration issued from $(0,1) \in \mathbf{AdS}^{15}(\mathbb O)$. From Proposition \ref{radial octo desi}, one can write 

\[
X(t)=\left(\frac{  u(t) \, w(t)}{\sqrt{1-{|w(t)|}^2}},\frac{ u(t)}{\sqrt{1-{|w(t)|}^2}}\right)
\]
where $w(t)$ is a Brownian motion on $\mathbb{O}H^1$,   $r(t)=\arctanh {|w(t)|}$ is a hyperbolic Jacobi diffusion with the generator 
\[
 \frac{1}{2}\left(\frac{{\partial}^2}{\partial {r}^2}+(7\coth r+7\tanh r)\frac{{\partial}}{\partial {r}}\right),
\]
and $u(t)$ is an independent Brownian motion on $\mathbb S^7$ timed changed with the functional $\int_0^t \tanh^2 r(s)$.

Therefore, the stochastic area has the same distribution as
\[
\mathfrak{a}(t)=\beta \left( \int_0^t \tanh^2 r(s) ds  \right),
\]
where $\beta$ is a Brownian motion in $\mathbb R^7$ independent of $r(t)$. 
By a comparison argument, we have that 
\[
\lim_{t \rightarrow +\infty} r(t) = +\infty 
\]
almost surely. Consequently, 
\[
\lim_{t \rightarrow +\infty} \frac{1}{t}\int_0^t \tanh(r(s))ds = 1
\]
almost surely, which in turn implies that $\mathfrak{a}(t)/\sqrt{t}$ converges in distribution to the normal distribution $\mathcal{N}(0,1)$.

\subsection{Integral representations of the horizontal heat kernel}
In this section, we give the integral representation of the subelliptic heat kernel of 
\[
\widetilde{L}=\frac{{\partial}^2}{\partial {r}^2}+(7\coth r+7\tanh r)\frac{{\partial}}{\partial {r}}+{\tanh^2 r}\left(\frac{{\partial}^2}{\partial {\eta}^2}+6\cot \eta\frac{{\partial}}{\partial {\eta}}\right).
\]
Let $p_t(r,\eta)$ denote this heat kernel.

\begin{proposition}
	For $r\geq 0,\eta\in [0,\pi)$, we have:
	\begin{equation}\label{eq:2}
	p_t(r,\eta)=\int_{0}^{\infty}\sum_{m\geq 0}e^{-m(m+6)t}h_m(\eta)h_m(-iu)\,q_{t,15}(\cosh r \cosh u) \sinh^6udu
	\end{equation}	
	
where 
\begin{equation}\label{eq:riem.heat.kernel.on.H^15}
q_{t,15}(\cosh s):=\frac{e^{-49t}}{(2\pi)^{7}\sqrt{4\pi t}}\left(-\frac{1}{\sinh s}\frac{d}{ds}\right)^{7}e^{-s^2/4t}
\end{equation}
is the Riemannian heat kernel on the $15$-dimensional real hyperbolic space ${{H}}^{15}$, and $h_m$ is the normalized eigenfunction
\begin{equation}\label{eq:10}
h_m(\eta)=\frac{\Gamma(7/2)}{\sqrt{\pi}\Gamma(3)}\int_{0}^{\pi}{(\cos \eta + \sqrt{-1}\sin \eta \cos \varphi)}^m{\sin^5 \varphi}d \varphi
\end{equation}
of $\widetilde{\Delta}_{\bS^{7}}=\frac{{\partial}^2}{\partial {\eta}^2}+6\cot \eta\frac{{\partial}}{\partial {\eta}}$ that is associated to the eigenvalue $-m(m+6)$.
\end{proposition}	
\begin{proof}
Since $\pi:\mathbf{AdS}^{15}(\mathbb{O}) \to \mathbb{O}H^1$
is the totally geodesic foliation and $\widetilde{{\square}}_{\mathbf{AdS}^{15}(\mathbb{O})}$ and $\widetilde{\Delta}_{\mathbb{S}^7}$ commute, 
\begin{equation*}
e^{t\widetilde{L}}=e^{t(\widetilde{\square}_{\mathbf{AdS}^{15}(\mathbb{O})}+\widetilde{\Delta}_{\mathbb{S}^7})}=e^{t\widetilde{\Delta}_{\mathbb{S}^7}}e^{t\widetilde{\square}_{\mathbf{AdS}^{15}(\mathbb{O})}}.
\end{equation*}
We deduce that the heat kernel of $\widetilde{L}$ can be written as
\begin{equation}\label{eq:9}
p_t(r,\eta)=\int_{0}^{\pi}s_t(\eta,u)p^{\widetilde{\square}_{\mathbf{AdS}^{15}(\mathbb{O})}}_t(r,u)\sin^6udu,
\end{equation}
where $s_t$ is the heat kernel of \eqref{eq:LB_on_S^7} with respect to the measure $\sin^6\eta d\eta,\eta\in [0,\pi)$, and $p^{\widetilde{\square}_{\mathbf{AdS}^{15}(\mathbb{O})}}_t(r,u)$ the heat kernel at $(0,0)$ of $\widetilde{\square}_{\mathbf{AdS}^{15}(\mathbb{O})}$ with respect to the measure in \eqref{eq:measure}, i.e.,
\begin{equation*}
d\mu(r,u)=\frac{56{\pi}^7}{\Gamma(8)}\sinh^7 r\cosh^7 r\sin^6 u dr du, r\in [0,\infty), u\in[0,\pi).
\end{equation*}
In order to write \eqref{eq:9} more precisely, let us consider the analytic change of variables $\tau : (r,\eta) \rightarrow (r,i\eta)$ that will be applied on functions of the type $f(r,\eta)=h(r)e^{-i\lambda \eta}$, with $h$ smooth and compactly supported on $[0,\infty)$ and $\lambda>0$. Then as we saw in the proof of Proposition~\ref{radial octo desi} one can see that 
\begin{equation*}
\widetilde{\square}_{\mathbf{AdS}^{15}(\mathbb{O})}(f\circ \tau)=(\widetilde{\Delta}_{{{H}^{15}}}f)\circ \tau,
\end{equation*}
where 
\begin{equation*}
\widetilde{\Delta}_{{{H}^{15}}}=\widetilde{\Delta}_{\mathbb{O}H^1}+\frac{1}{\cosh^2 r}\widetilde{\Delta}_{P}, \quad \widetilde{\Delta}_{P}=\frac{{\partial}^2}{\partial {\eta}^2}+6\coth \eta\frac{{\partial}}{\partial {\eta}}.
\end{equation*}

Then one deduces
\begin{equation*}
e^{t\widetilde{L}}(f\circ \tau)=e^{t\widetilde{\Delta}_{\mathbb{S}^7}}e^{t\widetilde{\square}_{\mathbf{AdS}^{15}(\mathbb{O})}}(f \circ \tau)=e^{t\widetilde{\Delta}_{\mathbb{S}^7}}((e^{t{\widetilde{\Delta}_{{{H}^{15}}}}}f) \circ \tau)=(e^{-t\widetilde{\Delta}_{P}}e^{t{\widetilde{\Delta}_{{{H}^{15}}}}}f) \circ \tau.
\end{equation*}
Now, since for every $f(r,\eta)=h(r)e^{-i\lambda \eta}$,
\begin{equation*}
(e^{t\widetilde{\square}_{\mathbf{AdS}^{15}(\mathbb{O})}}f)(0,0)=(e^{t{\widetilde{\Delta}_{{{H}^{15}}}}})(f\circ \tau^{-1})(0,0),
\end{equation*}
one deduces that for a function $h$ depending only on $u$, 
\begin{equation*}
\int_{0}^{\pi}h(u)p^{\widetilde{\square}_{\mathbf{AdS}^{15}(\mathbb{O})}}_t(r,u)\sin^6udu=\int_{0}^{\infty}h(-iu)q_{t,15}(\cosh r \cosh u) \sinh^6udu.
\end{equation*}
Therefore, coming back to \eqref{eq:9}, one infers that using the analytic extension of $s_t$ one must have 
\begin{equation}\label{eq:4}
\int_{0}^{\pi}s_t(\eta,u)p^{\square_{\mathbf{AdS}^{15}(\mathbb{O})}}_t(r,u)\sin^6udu=\int_{0}^{\infty}s_t(\eta,-iu)q_{t,15}(\cosh r \cosh u)\sinh^6udu,
\end{equation}
where $q_{t,15}$ is the Riemannian heat kernel on the real hyperbolic space ${{H}^{15}}$ given in \eqref{eq:riem.heat.kernel.on.H^15}. 
Then our result follows from the well-known formula (see Appendix 2), 
\begin{equation*}
s_t(\eta,u)=\sum_{m\geq 0}e^{-m(m+6)t}h_m(\eta)h_m(u),
\end{equation*}
where
\begin{equation}\label{eq:eigen-function}
	h_m(\eta)=\frac{\Gamma(7/2)}{\sqrt{\pi}\Gamma(3)}\int_{0}^{\pi}{(\cos \eta + \sqrt{-1}\sin \eta \cos \varphi)}^m{\sin^5 \varphi}d \varphi.
\end{equation}
\end{proof}

\chapter{Horizontal Brownian motions of the Stiefel fibration}

\section{Brownian motion on the complex Grassmannian space}\label{Sec2}

\subsection{Stiefel fibration} \label{fibration}

Let $n\in \mathbb{N}$, $n \ge 2$, and $k\in\{1,\dots, n\}$. The complex Stiefel manifold $V_{n,k}$ is the set of  unitary $k$-frames  in $\C^n$. In matrix notation we have
\[
V_{n,k}=\{M \in \C^{n\times k}| M^* M=I_k\}.
\] 

As such $V_{n,k}$ is therefore an algebraic compact embedded submanifold of $\C^{n\times k}$ and  inherits from  $\C^{n\times k}$ a Riemannian  structure. We note that $V_{n,1}$ is isometric to the unit sphere  $\mathbb{S}^{2n-1}$. There is a right isometric action of the  unitary group $\mathbf{U}(k)$ on $V_{n,k}$, which is simply given by the right matrix multiplication: $M  g$, $M \in V_{n,k}$,  $g \in \mathbf{U}(k)$. The quotient space by this action $G_{n,k}:=V_{n,k} / \mathbf{U}(k)$ is the complex Grassmannian manifold.  It is a compact manifold of complex dimension $k(n-k)$. We note that $G_{n,k}$ can be identified with the set of $k$-dimensional subspaces of $\mathbb{C}^n$. In particular $G_{n,1}$ is the complex projective space $\mathbb{C}P^{n-1}$. Since $G_{n,k}$ and $G_{n,n-k}$ can be identified with each other via orthogonal complement, without loss of generality we will therefore assume throughout the chapter that \fbox{$k \le n-k$}, except when explicitly mentioned otherwise (see Section \ref{duality section}). 

Let us quickly comment on the Riemannian structure of $G_{n,k}$ that we will be using and that is induced from the one of $V_{n,k}$.  From Example \ref{BB fibration} , there exists a unique Riemannian metric on $G_{n,k}$ such that the projection map $\pi: V_{n,k} \to G_{n,k}$ is a Riemannian submersion, moreover the fibers of this submersion are totally geodesic submanifolds of $V_{n,k}$ which are isometric to $\mathbf{U}(k)$. This therefore yields a fibration:
\[
\mathbf{U}(k) \to V_{n,k} \to G_{n,k}
\]
which is referred to as the Stiefel fibration.  We note that for $k=1$ it is nothing else but the classical  Hopf fibration considered from the probabilistic viewpoint in Chapter \ref{sec-BM-complex}:
\[
\mathbf{U}(1) \to \mathbb{S}^{2n-1} \to \mathbb{C}P^{n-1}.
\]
%
%

\subsection{affine coordinates on $G_{n,k}$}\label{homogeneous}

We consider  the open set $\widehat{V}_{n,k} \subset V_{n,k}$ given by
\[
\widehat{V}_{n,k}= \left\{ \begin{pmatrix} X  \\ Z \end{pmatrix}  \in  V_{n,k} , \, Z \in \mathbb{C}^{k \times k}, \det Z \neq 0,  \right\}
\]
and the smooth map $p : \widehat{V}_{n,k} \to \C^{(n-k)\times k}$ given by $p \begin{pmatrix} X  \\ Z \end{pmatrix}  = X Z^{-1}.$ It is clear that for every $g \in \mathbf{U}(k)$ and $M \in V_{n,k}$, $p( M  g)= p(M)$. Since $p$ is a submersion from $\widehat{V}_{n,k}$ onto its image $ p ( \widehat{V}_{n,k}) = \C^{(n-k)\times k} $ we deduce that there exists a unique  diffeomorphism $\Psi$ from  an open set of $G_{n,k}$ onto $\C^{(n-k)\times k}$ such that 
\begin{align}\label{inh-coord}
 \Psi \circ \pi= p.
 \end{align}
 The map $\Psi$ induces a (local) coordinate chart on $G_{n,k}$ that we call homogeneous by analogy with the case $k=1$ which corresponds to the complex projective space. The Riemannian metric of $G_{n,k}$ is then transported to $\C^{(n-k)\times k}$ using the map $\Psi$. In the sequel, we will denote $\C^{(n-k)\times k}$ endowed with this Riemannian metric by $\widehat{G}_{n,k}$ in order to emphasize the Riemannian structure that is used. Note that by construction $\widehat{G}_{n,k}$ is isometric to an open  subset of $G_{n,k}$ and that it {differs} from $G_{n,k}$ by the sub-manifold at $\infty$, $\det (Z)=0$. In the case $n=2, k=1$, $G_{n,k}=\mathbb{C}P^{1}$ and the above description corresponds to the classical one-point compactification description $\mathbb{C}P^{1}=\mathbb{C} \cup \{ \infty \}$. We note that the Stiefel fibration
 \[
\mathbf{U}(k) \to V_{n,k} \to G_{n,k}
\]
yields a fibration
\[
\mathbf{U}(k) \to\widehat{V}_{n,k} \to \widehat{G}_{n,k}
\]
that we still refer to as the Stiefel fibration. The projection map $p:\widehat{V}_{n,k} \to \widehat{G}_{n,k}$,  $ \begin{pmatrix} X  \\ Z \end{pmatrix}  \to  X Z^{-1}$ is then a Riemannian submersion with totally geodesic fibers isometric to $\mathbf{U}(k)$.

\subsection{Brownian motions on $\widehat{G}_{n,k}$ and $V_{n,k}$}

In this section, we show how the Brownian motions on $\widehat{G}_{n,k}$ and $V_{n,k}$ can be constructed from a Brownian motion on the unitary group $\mathbf{U}(n)$. In the sequel, we will use the following block notations: For any $U\in \mathbf{U}(n)$ and $A\in \mathfrak{u}(n)$ we will write
\[
U=\begin{pmatrix}  X & Y \\ Z & V \end{pmatrix}, \quad A=\begin{pmatrix}  \alpha & \beta \\ \gamma & \ep \end{pmatrix}
\]
where $X,\gamma \in \C^{(n-k)\times k}$, $Y,\ep\in \C^{(n-k)\times (n-k)}$, $Z,\alpha \in \C^{k\times k}$, $V,\beta\in \C^{k\times (n-k)}$.  We recall that the Lie algebra $\mathfrak{u}(n)$ consists of all skew-Hermitian matrices
\[
\mathfrak{u}(n)=\{X\in \C^{n\times n}|X=-X^*\},
\]
which we equip  with the inner product $\langle X,Y\rangle_{\mathfrak{u}(n)}=-\frac12\mathrm{tr}(XY)$. This induces the canonical Riemannian metric on $\U(n)$.   Consider now on $\mathfrak{u}(n)$ a Brownian motion $(A(t))_{t \ge 0}$ and  the matrix-valued process  $(U(t))_{t \ge 0}$ that satisfy the Stratonovich stochastic differential equation:
\begin{equation}\label{eq-du-da}
\begin{cases}
dU(t)=U(t)\circ dA(t) ,\\
U(0)=\begin{pmatrix}  X_0 & Y_0 \\ Z_0 & V_0 \end{pmatrix}, \, \det Z_0 \neq 0.
\end{cases}
\end{equation}
The process  $(U(t))_{t \ge 0}$ is a Brownian motion on $\mathbf{U}(n)$ (which is not started from the identity). 

We will use  the block decompositions
\[
 U(t)=\begin{pmatrix}  X(t) & Y(t) \\ Z(t) & V(t) \end{pmatrix}
 \]
and
\[
A(t)=\begin{pmatrix}  \alpha(t) & \beta(t) \\ \gamma(t) & \ep (t) \end{pmatrix},
\]
and note that the processes $(\alpha(t))_{t \geq 0}$, $(\beta(t))_{t \geq 0}=-(\gamma(t)^*)_{t \geq 0}$ and $(\ep (t))_{t \geq 0}$ are independent. 
From \eqref{eq-du-da}, we  obtain the following system of stochastic differential equations:
\begin{align}\label{eq-matrix-sde}
dX &=X \circ d\alpha +Y \circ d\gamma =X  d\alpha +Y d\gamma +\frac12 (dX  d\alpha +dY  d\gamma )\nonumber \\
dY &=X \circ d\beta +Y \circ d\ep =X  d\beta +Y  d\ep +\frac12 (dX  d\beta +dY  d\ep)\nonumber\\
dZ &= Z \circ d\alpha +V \circ d\gamma =Z  d\alpha +V d\gamma +\frac12 (dZ  d\alpha +dV  d\gamma )\\
dV &=Z \circ d\beta +V \circ d\ep =Z  d\beta +V  d\ep +\frac12 (dZ  d\beta +dV  d\ep ).\nonumber
\end{align}

Recall that here we use $dXdY$ to denote the quadratic covariation of matrix-valued processes. Namely the $ik$-component of $(dXdY)$ is given by the sum of the quadratic covariations  $(dXdY)_{ik}=\sum_{j}\langle dX_{ij},dY_{jk}\rangle$. 

It  turns out that both of the processes $VV^{*}$ and $ ZZ^*=I_k-VV^{*}$ belong to the  family of so-called (complex) matrix Jacobi processes. In particular, one has the following
 
 \begin{proposition}
  The process $Z^*Z$ satisfies  the stochastic differential equation
 \begin{equation}\label{eq-sde-zz*}
 d( ZZ^*) =  \sqrt{ZZ^*}d\bB\sqrt{I_k-ZZ^*}+\sqrt{I_k-ZZ^*}d\bB^* \sqrt{ZZ^*}+\bigg(2kI_k-2nZZ^* \bigg)dt
\end{equation}
where $\bB$ is a $\mathbb{C}^{k \times k}$ Brownian motion. 
\end{proposition}
\begin{proof}
Since
\begin{equation}\label{eq-sde-zz*}
d(ZZ^*)=dZ\circ Z^*+Z\circ dZ^*,
\end{equation}
from \eqref{eq-matrix-sde} we know that the martingale part of $d(ZZ^*)$ is given by
\[
Vd\gamma Z^*+Zd\gamma^* V^*
\]
We first show that there exists a $\mathbb{C}^{k \times k}$ Brownian motion $\bB$ such that
\[
Zd\gamma^* V^*=  \sqrt{ZZ^*}d\bB\sqrt{I_k-ZZ^*},
\]
Set
\[
dM:= \sqrt{ZZ^*}^{-1}Zd\gamma^* V^*\sqrt{I_k-ZZ^*}^{-1}.
\]
It is then equivalent to show that $M$ is a $\mathbb{C}^{k \times k}$ Brownian motion.
For any $1\le i,j,i',j'\le k$,
\begin{align*}
dM_{ij} = \sum_{m,\ell,q=1}^k\sum_{p=1}^{n-k}\left(\sqrt{ZZ^*}^{-1}\right)_{im} Z_{m\ell}\,(d\gamma^*)_{\ell p}(V^*)_{pq}\left(\sqrt{I_k-ZZ^*}^{-1}\right)_{qj}
\end{align*}
and 
\begin{align*}
d\overline{M}_{i'j'}= \sum_{m',\ell',q'=1}^k\sum_{p'=1}^{n-k}\left(\sqrt{\overline{ZZ}^*}^{-1}\right)_{i'm'} \overline{Z}_{m'\ell'}\,(d\overline{\gamma}^*)_{\ell' p'}(\overline{V}^*)_{p'q'}\left(\sqrt{I_k-\overline{ZZ}^*}^{-1}\right)_{q'j'}.
\end{align*}
Since $(d\gamma^*)_{\ell p}(d\overline{\gamma}^*)_{\ell' p'}=2dt\, \delta_{pp'}\delta_{\ell\ell'}$ and $VV^*=I_k-ZZ^*$, we then obtain that
\[
dM_{ij} d\overline{M}_{i'j'}=2dt\, \delta_{ii'} \delta_{j'j}.
\]
This implies that the martingale part of $d(ZZ^*)$ is given by 
\[
 \sqrt{ZZ^*}d\bB\sqrt{I_k-ZZ^*}+ \sqrt{I_k-ZZ^*}d\bB^* \sqrt{ZZ^*}.
\]
Moreover from \eqref{eq-sde-zz*} we know that the drift part for $d(ZZ^*)$ is given by
\[
dF:=\frac12\left(dVd\gamma Z^*+Zd\gamma^*dV^* +dZ d\alpha Z^*+ Zd\alpha^* d Z^*\right)+dZdZ^*.
\]
Plugging in \eqref{eq-matrix-sde} we find that
\[
dF= 2dt\, \left(kI_k-nZZ^*\right).
\]
This then concludes the proof. 
\end{proof}

\begin{theorem}\label{main:s1-sphere}
Let $U(t)=\begin{pmatrix}  X(t) & Y(t) \\ Z(t) & V(t) \end{pmatrix}$, $t\ge0$ be the solution of \eqref{eq-du-da}.  
\begin{enumerate}
\item The process $\begin{pmatrix}  X(t)  \\ Z(t)  \end{pmatrix}_{t\ge0}$ is a Brownian motion on $V_{n,k}$;

\item We have $\mathbb{P} \left( \inf\{t>0, \det Z(t)=0\} <+\infty \right)=0$  and the process $(w (t))_{t \ge 0}:=(X(t)Z(t)^{-1})_{t \ge 0}$ is a Brownian motion on $\widehat{G}_{n,k}$  with  generator  $\frac12\Delta_{\widehat{G}_{n,k}}$, where 
\begin{align*}
\Delta_{\widehat{G}_{n,k}}&={{4}}\sum_{1\le i, i'\le n-k, 1\le j, j'\le k}(I_{n-k}+ww^*)_{ii'} (I_k+w^*w)_{j'j}\frac{\partial^2}{\partial w_{ij}\partial \bar{w}_{i'j'}}.
\end{align*}

\end{enumerate}

\end{theorem}

\begin{proof}
The first part of the theorem is straightforward to prove. Indeed, the  map $\Pi: \mathbf{U}(n) \to V_{n,k}$, $\begin{pmatrix}  X & Y  \\ Z  & V  \end{pmatrix} \to \begin{pmatrix}  X   \\ Z    \end{pmatrix}  $ is a Riemannian submersion with totally geodesic fibers, therefore the process $(\Pi(U(t)))_{t \ge 0}$ is a Brownian motion on $V_{n,k}$. We now proceed to the proof of the second part of the theorem. We first note that, as noticed above,  $(Z(t)Z^*(t) )_{ t \ge 0}$ is a matrix Jacobi process and therefore, from known properties of those processes, $\mathbb{P} \left( \inf\{t>0, \det (Z(t)Z^*(t))=0\} <+\infty \right)=0$. Consequently, $w=XZ^{-1}$  is defined for all times. 
Using then It\^o's formula, long but routine computations yield
\begin{align}\label{w dynamics}
dw=d(XZ^{-1}
)=(Y d\gamma -w V d\gamma) Z^{-1}.
\end{align}
whence it follows that
\begin{align*}
dw_{ij}d\bar{w}_{i'j}&=2(I_{n-k}+ww^*)_{ii'}(I_k+w^*w)_{j'j}dt.
\end{align*}
Consequently, $(w (t))_{t \ge 0}$ is a diffusion with generator $\frac12\Delta_{\widehat{G}_{n,k}}$. Finally, this process is indeed the Brownian motion on $\widehat{G}_{n,k}$ since $p$ is a Riemannian submersion with totally geodesic fibers.
\end{proof}

\begin{remark}
The complex Grassmannian $G_{n,k}$ is a compact irreducible  symmetric space of rank $k$ and the complex Stiefel manifold is a Riemannian homogeneous space (but is not a symmetric space). As such, the Brownian motions on $G_{n,k}$ and $V_{n,k}$ and their distributions  and pathwise properties can be studied using representation theory and stochastic differential geometry. The literature on those topics is nowadays quite substantial. We for instance refer to the early  works  by E. Dynkin \cite{MR0132607, MR0202206} and P. $\&$ M.P. Malliavin \cite{MR0359023} or more recent presentations like \cite{MR1284654},  \cite{MR3243323}, \cite{MR3201989}, \cite{MR1138584} and  the monograph \cite{MR2731662} (see in particular Chapter 7: Riemannian submersions and Symmetric spaces). In some sense, Theorem \ref{main:s1-sphere} provides a more pedestrian approach: We work in a specific choice of coordinates within the algebra of complex matrices and describe the $G_{n,k}$  and $V_{n,k}$ Brownian motions in those coordinates taking advantage of the  additional structure given by the matrix multiplication. Note that a similar pedestrian approach is taken in \cite{MR825735} to study the Brownian motion in the homogeneous space $\mathbf{GL}(n,\R)/\mathbf{O}(n,\R)$.
    
\end{remark}

\begin{proposition}\label{invariant-w}
If we consider on $\mathbb{C}^{(n-k)\times k}$ the probability measure
\[
d\mu:=c_{n,k}\det(I_k+w^*w)^{-n} dw
\]
with normalizing constant $c_{n,k}$, then $\mu$ is the unique symmetric and invariant probability measure for the diffusion $(w(t))_{t \ge 0}$: if $f,g$ are smooth and compactly supported functions on $\mathbb{C}^{(n-k)\times k}$, then
\[
\int_{\mathbb{C}^{(n-k)\times k}} (\Delta_{\widehat{G}_{n,k}}f) g\, d\mu=\int_{\mathbb{C}^{(n-k)\times k}} f(\Delta_{\widehat{G}_{n,k}}g)\, d\mu
\]
\end{proposition}

\begin{proof}
We denote $\partial_{ij}=\frac{\partial }{\partial w_{ij}}$, $\bpartial_{i j}=\frac{\partial }{\partial\bw_{ij}}$,  $A_{ii'jj'}=(\delta_{ii'}+(w w^*)_{ii'})(\delta_{j'j}+(w^*w)_{j'j})$, and $\rho=c_{n,k}\det(I_k+w^*w)^{-n}$. Let us denote
\[
\mathcal{T}(f,g)=2\sum_{1\le i, i'\le n-k, 1\le j, j'\le k}(I_{n-k}+ww^*)_{ii'}(I_k+w^*w)_{j'j}\bigg(\frac{\partial f}{\partial w_{ij}}\frac{\partial g}{\partial\bw_{i'j'}}+\frac{\partial g}{\partial w_{ij}}\frac{\partial f}{\partial\bw_{i'j'}}\bigg)
\]
By integration by parts we have
\begin{align*}
-\frac{1}{2}\int (\Delta_{\widehat{G}_{n,k}}f) g\, d\mu 
= \frac{1}{2}\mathcal{T}(f,g)+R.
\end{align*}
where 
\begin{align*}
R&=\sum_{1\le i, i'\le n-k, 1\le j, j'\le k} \bigg(\int [(\partial_{ij}f)\,(\bpartial_{i'j'} A_{ii'jj'})+(\bpartial_{i'j'}f)\,(\partial_{ij} A_{ii'jj'})]g\rho dm\\
&+\int [(\partial_{ij}f)\,(\bpartial_{i'j'}\rho)+(\bpartial_{i'j'}f)\,(\partial_{ij}\rho)]gA_{ii'jj'} dm\bigg)
\end{align*}
Since 
\[
\sum_{1\le i'\le n-k, 1\le j'\le k} \bpartial_{i'j'} A_{ii'jj'}=n({\bw}(I_k+J))_{ij},
\ 
\sum_{1\le i\le n-k, 1\le j\le k} \partial_{ij} A_{ii'jj'}=n({\bw}(I_k+\bar{J}))_{i'j'},
\]
and
\begin{align*}
\bpartial_{i'j'} \rho
&=-n  \rho\bigg(\bw(I_k+J)^{-1}\bigg)_{i'j'}, \quad
\partial_{ij} \rho=-n  \rho\bigg(\bw(I_k+\bar{J})^{-1} \bigg)_{ij},
\end{align*}
we  have
\begin{align*}
\sum_{i',j'}(\bpartial_{i'j'} A_{ii'jj'})g\rho + (\bpartial_{i'j'}\rho)gA_{ii'jj'}=0,
\quad
\sum_{i,j}(\partial_{ij} A_{ii'jj'})g\rho + (\partial_{ij}\rho)gA_{ii'jj'}=0.
\end{align*}
This gives $R=0$. We the have
\begin{align*}
\int_{\mathbb{C}^{(n-k)\times k}} (\Delta_{\widehat{G}_{n,k}}f) g\, d\mu 
= -\mathcal{T}(f,g).
\end{align*}
The bilinear first order operator $\mathcal{T}$ is in fact the carr\'e du champ operator. By similar computations as above we also obtain that
\[
\int_{\mathbb{C}^{(n-k)\times k}} f(\Delta_{\widehat{G}_{n,k}}g)\, d\mu= -\mathcal{T}(f,g),
\]
which then completes the proof.
\end{proof}
%
%
\subsection{Eigenvalue process} 

In this section, for later use, we collect some properties of the eigenvalues of the process $(J(t))_{t \ge 0}: = (w^* (t) w (t))_{t \ge 0}$  where $(w (t))_{t \ge 0}=(X(t)Z(t)^{-1})_{t \ge 0}$ is a Brownian motion on $\widehat{G}_{n,k}$ as in Theorem \ref{main:s1-sphere}. We note that
\[
J= w^* w =(Z^{-1})^* X^* XZ^{-1}=(Z Z^*)^{-1}-I_k
\]
Therefore $(I_k+J)^{-1}=ZZ^{*}$, and the properties of $J$ and its eigenvalues can be deduced from the corresponding properties of the matrix Jacobi process $ZZ^{*}$ after basic algebraic manipulations. In particular, one obtains the following result.

\begin{lemma}\label{thm-J}

Let $(J(t))_{t\ge0}$ be given as above, then $\mathbb{P}( \inf\{ t>0, \det (J(t))=0\}<+\infty)=0$ and there exists a Brownian motion $(\bB(t))_{t \ge 0}$ in $\C^{k \times k}$ such that:
\begin{equation}\label{eq-dJ-1}
dJ=\sqrt{I_{k}+J} d\bB^*\sqrt{I_k+J} \sqrt{J}+\sqrt{J}\sqrt{I_k+J}d\bB \sqrt{I_{k}+J}+2(n-k+\tr(J))(I_k+J)dt
\end{equation} 
\end{lemma}

\begin{proof}
Since $J=(ZZ^*)^{-1}-I_k$, we have
\begin{align*}
dJ&=-(ZZ^*)^{-1} \circ d(ZZ^*)\circ (ZZ^*)^{-1}\\
&=-(ZZ^*)^{-1} d(ZZ^*) (ZZ^*)^{-1}-\frac12(ZZ^*)^{-1} d(ZZ^*) d(ZZ^*)^{-1}-\frac12 d(ZZ^*)^{-1} d(ZZ^*) (ZZ^*)^{-1}.
\end{align*}
Plugging in \eqref{eq-sde-zz*} we have that
\begin{align*}
&-(ZZ^*)^{-1} d(ZZ^*) (ZZ^*)^{-1}\\
=&\sqrt{I_k+J}\, d\bB\sqrt{J}\sqrt{I_k+J} + \sqrt{I_k+J}\sqrt{J}d\bB^*\sqrt{I_k+J}+2 \left((n-k)I_k-kJ\right)(I_k+J)dt.
\end{align*}
We can also compute that
\[
d(ZZ^*) d(ZZ^*)^{-1}=d(ZZ^*)^{-1}d(ZZ^*) =-2(\tr(J)I_k+k J)dt. 
\]
Combining the above equations we then obtain the conclusion.
\end{proof}

We explicitly note that the proof of $\mathbb{P}( \inf\{ t>0, \det (J(t))=0\}<+\infty)=0$ uses the fact that $k \le n-k$.

\begin{remark}
The symmetric and invariant probability measure of the diffusion process $J$ can be easily obtained. Indeed, let $\mathbf{w}$ be a random variable on $\C^{(n-k)\times k}$ whose law is the probability measure with density $c_{n,k}\det(I_k+w^*w)^{-n} dw$.  From Proposition 1 in \cite{Forrester} one has for every bounded Borel function $g$ and some normalization constant $c'_{n,k}$ that
\begin{align*}
\mathbb{E} ( g (\mathbf{w}^* \mathbf{w}))&=c_{n,k} \int_{\C^{(n-k)\times k}} g (w^* w) \det (I_k +w^* w)^{-n} dw \\
 & =c'_{n,k} \int_{\Hh_k} g (S) \det (I_k +S)^{-n} \det(S)^{n-2k} dS.
\end{align*}
Thus, from Remark \ref{invariant-w}, the probability measure on the cone $\Hh_k$ of positive definite Hermitian matrices with density  $ c'_{n,k} \det (I_k +S)^{-n} \det(S)^{n-2k} dS$ is the invariant and symmetric probability measure for the diffusion process $J$.
\end{remark}

\begin{lemma}\label{thm-lambda}
Let   $\lambda(t)=(\lambda_1(t),\dots,\lambda_k(t))$, $t\ge0$ be the eigenvalue process of the diffusion matrix $J$. Assume that $\lambda_1(0) > \cdots > \lambda_k(0)$, then the process $\lambda(t)$ is non colliding, i.e.
\[
\mathbb{P}\left( \forall \, t \ge 0, \lambda_1(t) > \cdots > \lambda_k(t) \right)=1.
\]
Moreover, we have
\begin{equation}\label{eq-lambda-sde}
d\lambda_i=2(1+\lambda_i)\sqrt{\lambda_i}dB^i+2(1+\lambda_i) \bigg(n-2k+1-(2k-3)\lambda_i+2\lambda_i (1+\lambda_i)\sum_{\ell\not=i} \frac{1}{\lambda_i-\lambda_\ell}\bigg)dt,
\end{equation}
where $(B(t))_{t \ge 0}$ is a Brownian motion in $\mathbb{R}^k$.
\end{lemma}
%

We prove the result in several steps. As above,  we denote by $\lambda(t)= (\lambda_i(t))_{1\le i\le k}$ the eigenvalues of $J(t)$, $t\ge0$. Let $\mathcal{N}=\{\lambda\in (0, \infty)^k$, $\lambda_i\not=\lambda_j,\, \forall i\not=j\}$, and consider the stopping time
\begin{equation}\label{eq-tau-N}
\tau_{\mathcal{N}}=\inf\{t>0, \lambda(t)\not\in \mathcal{N}\}. 
\end{equation}

\begin{theorem}\label{thm-lambda}
Assume $\lambda (0)\in \mathcal{N}$. Then up to time $\tau_\mathcal{N}$, the eigenvalues $\lambda(t)=(\lambda_1,\dots,\lambda_k)(t)$, $t\ge0$ satisfy the following stochastic differential equation
\begin{equation}\label{eq-lambda-sde}
d\lambda_i=2(1+\lambda_i)\sqrt{\lambda_i}dB^i+2(1+\lambda_i) \bigg(n-2k+1-(2k-3)\lambda_i+2\lambda_i (1+\lambda_i)\sum_{\ell\not=i} \frac{1}{\lambda_i-\lambda_\ell}\bigg)dt.
\end{equation}
where $(B_t)_{t \ge 0}$ is a Brownian motion in $\mathbb{R}^k$.
\end{theorem}

\begin{proof}
We label the eigenvalues by $\lambda_1\ge \cdots\ge \lambda_k$. Note $J$ is Hermitian, hence it can be diagonalized by $J=V\Lambda V^*$ where $V\in \mathbf{U}(k)$ and $\Lambda=\mathrm{diag}\{ \lambda_1,\dots, \lambda_k\}$. Let $dU=dV^*\circ V$ and $dN=V^*\circ dJ \, V$. Then 
\begin{equation}\label{eq-lambda}
d\Lambda=dU\circ \Lambda-\Lambda\circ dU+dN
\end{equation}
hence
\begin{equation}\label{eq-dU}
d\lambda_i=dN_{ii},\quad dU_{ij}=\frac{1}{\lambda_i-\lambda_j}\circ dN_{ij}\quad \mbox{for $i\not=j$}. 
\end{equation}
From \eqref{eq-dJ} we know that
\begin{align*}
\frac{(dJ)_{ij}(dJ)_{i'j'}}{2dt}=(J+J^2)_{i'j}(I_k+J)_{ij'}+(J+J^2)_{ij'}(I_k+J)_{i'j}.
\end{align*}
We can then compute that
\begin{align*}
\frac{(dN)_{ij}(dN)_{i'j'}}{2dt}&=\sum_{p,p',\ell,\ell'}V^*_{ip}V^*_{i'p'}V_{\ell j}V_{\ell' j'}\left((J+J^2)_{p'\ell}(I_k+J)_{p\ell'}+(J+J^2)_{p\ell'}(I_k+J)_{p'\ell} \right)\\
&=(V^*(J+J^2)V)_{i'j}(V^*(I_k+J)V)_{ij'}+(V^*(J+J^2)V)_{ij'}(V^*(I_k+J)V)_{i'j}\\
&=(\Lambda+\Lambda^2)_{i'j}(I_k+\Lambda)_{ij'}+(\Lambda+\Lambda^2)_{ij'}(I_k+\Lambda)_{i'j}.
\end{align*}
If we denote by $dM$ the local martingale part of $dN$ and $dF$ the finite variation part, then from \eqref{eq-dJ-1} we know that
\begin{align*}
\frac{dF}{2dt}&=V^*(n-k+\mathrm{tr}(J))(I_k+J)V+\frac12\frac{(dV^* dJ\,V+V^*dJdV)}{2dt}\\
&=(n-k+\mathrm{tr}(J))(I_k+\Lambda)+\frac12\frac{dUdN+dN^*dU^*}{2dt}.
\end{align*}
Since 
\begin{align*}
\frac{(dUdN)_{ij}}{2dt}&=\sum_{\ell\not=i}\frac{1}{\lambda_i-\lambda_\ell}\frac{dN_{i\ell}dN_{\ell j}}{2dt}\\
&=\delta_{ij}\sum_{\ell\not=i}\frac{(1+\lambda_i)(1+\lambda_\ell)(\lambda_i+\lambda_\ell)}{\lambda_i-\lambda_\ell}
\end{align*}
we obtain that $(dUdN)^*=dUdN$. Hence
\begin{align*}
dF_{ij}&=2dt \delta_{ij}\bigg((n-k+\sum_{\ell=1}^k\lambda_\ell)(1+\lambda_i)+\sum_{\ell\not=i} \frac{(1+\lambda_i)(1+\lambda_\ell)(\lambda_i+\lambda_\ell)}{\lambda_i-\lambda_\ell}\bigg)\\
&=2dt \delta_{ij}(1+\lambda_i)\bigg(n-2k+1-(2k-3)\lambda_i+2\lambda_i(1+\lambda_i)\sum_{\ell\not=i} \frac{1}{\lambda_i-\lambda_\ell}\bigg).
\end{align*}
At last, we have that
\begin{align*}
dM_{ii}dM_{jj}&=dN_{ii}dN_{jj}=2dt((\Lambda+\Lambda^2)_{ji}(I_k+\Lambda)_{ij}+(\Lambda+\Lambda^2)_{ij}(I_k+\Lambda)_{ji})\\
&=4dt\delta_{ij}\lambda_i(1+\lambda_i)^2.
\end{align*}
Hence
\[
dM_{ii}=2(1+\lambda_i)\sqrt{\lambda_i}dB^i
\]
where the $B^i$'s are independent standard real Brownian motions. We conclude
\begin{align*}
d\lambda_i=dM_{ii}+dF_{ii}&=2(1+\lambda_i)\sqrt{\lambda_i}dB^i+2(1+\lambda_i) \bigg(n-2k+1-(2k-3)\lambda_i\\
&\quad+2\lambda_i(1+\lambda_i)\sum_{\ell\not=i} \frac{1}{\lambda_i-\lambda_\ell}\bigg)dt. 
\end{align*}
\end{proof}

Under the assumptions of the previous theorem, let us introduce: 
\begin{equation*}
\rho_i=\frac{1-\lambda_i}{1+\lambda_i}, \quad i=1,\dots, k.    
\end{equation*}
which is nothing else but the eigenvalue process of $2ZZ^*-I_k$. Then, as an application of It\^o's formula, we have
\begin{align*}
d\rho_i =-2\sqrt{1-\rho_i^2}dB^i-2\bigg((n-2k+(n-2k+2)\rho_i)+2\sum_{\ell\not=i}\frac{1-\rho_i^2}{\rho_\ell-\rho_i}\bigg)dt,
\end{align*}
where $(B(t))_{t \ge 0}$ is the same Brownian motion as in \eqref{eq-lambda-sde}. Therefore, $ \rho$ is a diffusion process  with generator given by

\[
\mathcal{L}_{n,k}=2\sum_{i=1}^k(1-\rho_i^2)\partial_i^2-2\sum_{i=1}^k\bigg(n-2k+(n-2k+2)\rho_i+2\sum_{\ell\not=i} \frac{1-\rho_i^2}{\rho_\ell-\rho_i}\bigg)\partial_i.
\]

Let $\tau=\inf\{t>0\mid  \exists\, i<j, \, \rho_i(t)=\rho_j(t)\}$ be the first colliding time of $\rho(t)$. The lemma below makes sure that the processes $\rho(t), t\ge0$ hence the processes $\lambda(t),t\ge0$ never collide.

\begin{theorem}\label{thm-non-collide}
Let $\tau_\mathcal{N}$ and $\tau$ be the stopping times as given above. Assume that at time $t=0$, $\lambda(0) \in {\mathcal{N}}$. Then 
\[
\mathbb{P}( \tau <+\infty)=0
\]
and consequently 
\[
\mathbb{P}( \tau_{\mathcal{N}} <+\infty)=0.
\]
\end{theorem}

\begin{proof}
We just need to show that $\mathbb{P}( \tau <+\infty)=0$, namely for every $t \ge 0$ we almost surely have
\[
\rho_1(t)<\cdots< \rho_k(t).
\]
Let $h=\Pi_{i>j}(\rho_i-\rho_j)$.  Using similar idea as previously, let us consider the process 
\[
\Omega_t:=V(\rho_1(t),\dots, \rho_k(t)),
\]
where $V(\rho_1,\dots, \rho_k)=\frac12\log h=\frac12\sum_{i>j}\log (\rho_i-\rho_j)$. We can compute that
\begin{align}\label{eq-Omega}
d\Omega_t=\sum_{i=1}^k\left(\partial_iV d\rho_i+\frac12 \partial^2_i V d\langle \rho_i\rangle\right)
=\mathcal{L}_{n,k}Vdt+dM_t
\end{align}
where $M_t$ is a local martingale satisfying $dM_t=-2\sum_{i=1}^k\sqrt{1-\rho_i^2}(\partial_i V) dB^i_t$ and 
\[
\mathcal{L}_{n,k}=2\sum_{i=1}^k(1-\rho_i^2)\partial_i^2-2\sum_{i=1}^k\bigg(n-2k+(n-2k+2)\rho_i+2\sum_{\ell\not=i} \frac{1-\rho_i^2}{\rho_\ell-\rho_i}\bigg)\partial_i.
\]
For any $1\le i\le k$,
\[
\partial_i V=\frac{\partial_i h}{2h},\quad 
\partial_i^2 V=\frac{\partial_i^2h}{2h}-\frac{(\partial_ih)^2}{2h^2}. 
\]
Since $\sum_{i=1}^k\partial_i h=0$ and $\sum_{i=1}^k\rho_i\partial_i h=\frac{k(k-1)}{2}h$,
we that 
\[
\sum_{i=1}^k\partial_i V=0,\quad \sum_{i=1}^k\rho_i\partial_i V=\frac{ k(k-1)}{4}. 
\]
Hence
\[
\mathcal{L}_{n,k}V=\sum_{i=1}^k(1-\rho_i^2)\left(\frac{\partial_i^2 h}{h} +\frac{(\partial_i h)^2}{h^2}\right)-\frac{(n-2k+2)k(k-1)}2.
\]
By a direct computation  we have that 
\begin{align*}
\sum_{i=1}^k\rho_i^2\frac{\partial_i^2 h}{h}=\frac13k(k-1)(k-2),
\quad
\sum_{i=1}^k\frac{\partial_i^2 h}{h}&=0,
\end{align*}
therefore we obtain that 
\[
\mathcal{L}_{n,k}V= -\frac{3n-4k+2}{6}k(k-1)+\sum_{i=1}^k(1-\rho_i^2)\frac{(\partial_i h)^2}{h^2}\ge -\frac{3n-4k+2}{6}k(k-1).
\]
Plug back into \eqref{eq-Omega} we have for any $t\ge0$,
\[
M_t\le\Omega_t-\Omega_0+\frac{3n-4k+2}{6}k(k-1)t.
\]
On $\{\tau<+\infty\}$, by letting $t\to \tau$ we have the right hand side of the above inequality goes to $-\infty$. This implies that $M_\tau=-\infty$. However, since $M_t$ is a time changed Brownian motion, we then obtain that $\{\tau<+\infty\}$ is a null set. 
\end{proof}

\begin{lemma}
If we consider the Vandermonde function
\[
h(\rho)=\prod_{i>j}(\rho_i-\rho_j),
\]
then we have for every smooth function $f$ on $[-1,1]^k$ that
\[
\mathcal{L}_{n,k}f=2 \left( \frac{1}{h}  \mathcal{G}_{n-2k, 0} (h f) +\frac{1}{6} k(k-1) \left( 3n-4k+2 \right)  f \right).
\]
where 
\begin{equation*}
\mathcal{G}_{a,b}:=\sum_{i=1}^k (1-\rho^2_i)\partial_i^2-(a-b+(a+b+2)\rho_i)\partial_i     
\end{equation*} 
is the generator of $k$ independent Jacobi diffusion operators valued in $[-1,1]$. 
\end{lemma}

\begin{proof}
 Let 
 
\[ \Gamma(f,g):=\frac12(\mathcal{G}_{\alpha,\beta}(fg)-f\mathcal{G}_{\alpha,\beta}g-g\mathcal{G}_{\alpha,\beta}(f))
\]
be the carr\'e du champ operator associated to $\mathcal{G}_{\alpha,\beta}$. We have
\[
\Gamma( h, f)=\sum_{i=1}^k (1-\rho^2_i)(\partial_i h)( \partial_i f).
\]
From the definition of $h$ it is clear that
\[
\partial_i h = h  \sum_{\ell \neq i} \frac{1}{\rho_i -\rho_\ell},
\]
thus, we obtain
\[
\Gamma(\log h, f)=\sum_{i=1}^k\sum_{j\not=i}\frac{1-\rho_i^2}{\rho_i-\rho_\ell}\partial_i f.
\]
On the other hand, thanks to a direct computation 
\begin{align*}
\mathcal{G}_{\alpha,\beta} h  & =\sum_{i=1}^k (1-\rho^2_i)\partial_i^2 h - \sum_{i=1}^k (\alpha-\beta+(\alpha+\beta+2)\rho_i)\partial_i h \\
 & =- k(k-1) \left( \frac{k-2}{3} +\frac{\alpha+\beta+2}{2} \right) h.
\end{align*}
In particular, one has
\[
\mathcal{G}_{n-2k, 0} (h ) =-\frac{1}{6} k(k-1) \left( 3n-4k+2 \right) h.
\]
We conclude
\begin{align*}
\frac{1}{h}  \mathcal{G}_{n-2k, 0} (h f)&= \frac{1}{h} \left(   \mathcal{G}_{n-2k, 0} (h) f + \mathcal{G}_{n-2k, 0} (f) h + 2 \Gamma( f,h)  \right) \\
 &=- \frac{1}{6}  k(k-1) \left( 3n-4k+2 \right)f + \mathcal{G}_{n-2k, 0} (f) +2 \Gamma( \log h,f) \\
 &= - \frac{1}{6} k(k-1) \left( 3n-4k+2 \right)  f +\frac{1}{2}  \mathcal{L}_{n,k}f.
\end{align*}
\end{proof}

We deduce the following formula for the density of the eigenvalue process which belongs to the family of Karlin-McGregor formulas, see \cite{AOW}, \cite{MR114248}. 

\begin{theorem}\label{density rho}
Let $\lambda$ be the eigenvalues process of $J$ and $\rho_i=\frac{1-\lambda_i}{1+\lambda_i}$. Let us assume that
\[
\rho_1(0)<\cdots< \rho_k(0).
\]
The density at time $t>0$ of $\rho(t)$  with respect to the Lebesgue measure $dx$ on $[-1,1]^k$ is given by
\[
e^{ \frac{1}{3} k(k-1) \left( 3n-4k+2 \right) t}  \frac{h(x)}{h(\rho(0))} \mathrm{det} \left( p^{n-2k,0}_{t} (\rho_i(0),x_j)\right)_{1\le i, j \le k}  \, \mathbf{1}_{\Delta_k} (x),
\]
where  $$\Delta_k:=\{ -1 \le  x_1<\cdots <x_k \le 1 \}.$$
\end{theorem}

\begin{proof}
Let $f$ be a smooth function defined on the simplex $\Delta_k$.  We almost everywhere extend $f$ to $[-1,1]^k$ by symmetrization, i.e. for every permutation $\sigma \in \mathfrak{S}_k$,
\[
f(x_{\sigma(1)},\dots, x_{\sigma(k)})= f(x_1,\dots,x_k).
\]
It follows from the intertwining of generators 
\[
\mathcal{L}_{n,k}=2 \left( \frac{1}{h}  \mathcal{G}_{n-2k, 2k-2} (h \, \cdot) +\frac{1}{6} k(k-1) \left( 3n-4k+2 \right) \right)
\]
that for the corresponding semigroups
\begin{align*}
 & e^{t\mathcal{L}_{n,k}} f (\rho(0)) \\
  =&  e^{ \frac{1}{3} k(k-1) \left( 3n-4k+2 \right) t}  \frac{1}{h(\rho(0))} e^{2t\mathcal{G}_{n-2k, 0} } ( h f) (\rho(0)) \\
 =&  e^{ \frac{1}{3} k(k-1) \left( 3n-4k+2 \right) t}  \frac{1}{h(\rho(0))}  \int_{[-1,1]^k} h(x) p^{n-2k,0}_{t} (\rho_1(0),x_1) \cdots p^{n-2k,0}_{t}(\rho_k(0),x_k) f(x) \, dx \\
 =&  \frac{ e^{ \frac{1}{3} k(k-1) \left( 3n-4k+2 \right) t} }{h(\rho(0))} \sum_{\sigma \in \mathfrak{S}_k}    \int_{-1<x_{\sigma(1)} <\cdots <x_{\sigma(k)} <1 } h(x) p^{n-2k,0}_{t} (\rho_1(0),x_1) \cdots p^{n-2k,0}_{t}(\rho_k(0),x_k) f(x) \, dx\\
 =& \frac{ e^{ \frac{1}{3} k(k-1) \left( 3n-4k+2 \right) t} }{h(\rho(0))} \sum_{\sigma \in \mathfrak{S}_k} \mathrm{sgn}(\sigma)    \int_{\Delta_k } h(x) p^{n-2k,0}_{t} (\rho_{\sigma(1)}(0),x_1) \cdots p^{n-2k,0}_{t}(\rho_{\sigma(k)} (0),x_k) f(x) \, dx.
\end{align*}
The conclusion follows immediately.
\end{proof}

We can deduce the limit law of $\rho$.

\begin{theorem}\label{limit eigenvalue}
Let $\lambda$ be the eigenvalues process of $J$ and $\rho_i=\frac{1-\lambda_i}{1+\lambda_i}$. Assume that
\[
\rho_1(0)<\cdots< \rho_k(0).
\]
Then, when $t \to +\infty$, $\rho(t)$ converges in distribution to the probability measure on $[-1,1]^k$  given by
\[
 d\nu=c_{n,k} \prod_{1 \le i < j \le k}(x_i-x_j)^2  \prod_{i=1}^k (1-x_i)^{n-2k}  \, \, \mathbf{1}_{\Delta_k} (x) dx,
\]
where $c_{n,k}$ is the normalization constant. Moreover, we have the following quantitative estimate: There exists a constant $C>0$ such that for any bounded Borel function $f$ on the simplex $\Delta_k$ and $t \ge 0$,
\begin{equation}\label{quant_1}
\left | \mathbb{E} ( f( \rho(t)) ) - \int_{\Delta_k} f d\nu \right| \le C e^{-2nt} \| f \|_\infty.
\end{equation}
\end{theorem}

\begin{remark}
If needed, the normalization constant $c_{n,k}$ can be computed from the so-called Selberg integral formula:
\begin{align*}
&\int_{[0,1]^n} \prod_{i=1}^n x_i^{\alpha-1}(1-x_i)^{\beta-1} \prod_{1\le i <j \le n} |x_i-x_j|^{2\gamma} dx_1 \cdots dx_n \\
&\quad\quad=\prod_{j=1}^{n-1}\frac{\Gamma(\alpha+j \gamma)\Gamma(\beta+j\gamma)\Gamma(1+(j+1)\gamma)}{\Gamma(\alpha+\beta+(n+j-1)\gamma)\Gamma(1+\gamma)}.
\end{align*}

\end{remark}

\begin{proof}
Using the formula \eqref{eq-jacobi-kernel} we can write
\[
  p^{n-2k,0}_{t}(x,y)
 =(1-y)^{n-2k}\sum_{m=0}^{+\infty}  C_m e^{-2m(m+n-2k+1)t}P_m^{n-2k,0}(x)P_m^{n-2k,0}(y),
\]
for some constants $C_m$. We now compute 
\begin{align*}
 & \mathrm{det} \left( p^{n-2k,0}_{t} (\rho_i(0),x_j)\right)_{1\le i, j \le k} \\
 =& \sum_{\sigma \in \mathfrak{S}_k} \mathrm{sgn}(\sigma) \prod_{i=1}^k  p^{n-2k,0}_{t}(\rho_{\sigma(i)} (0),x_i) \\
 =& \sum_{\sigma \in \mathfrak{S}_k} \mathrm{sgn}(\sigma) \prod_{i=1}^k \bigg[ (1-x_i)^{n-2k} \sum_{m=0}^{+\infty}  C_m e^{-2m(m+n-2k+1)t}P_m^{n-2k,0}(\rho_{\sigma(i)} (0))P_m^{n-2k,0}(x
 _i)\bigg]\\
 =& V(x) \sum_{\sigma \in \mathfrak{S}_k} \mathrm{sgn}(\sigma)  \sum_{m_1,\cdots,m_k=0}^{+\infty}  \prod_{i=1}^k C_{m_i}  e^{-2 m_i(m_i+n-2k+1)t}P_{m_i}^{n-2k,0}(\rho_{\sigma(i)} (0))P_{m_i}^{n-2k,0}(x_i)
\end{align*}
where $V(x)=\prod_{i=1}^k (1-x_i)^{n-2k}$. We can now write
\begin{align*}
 &  \sum_{\sigma \in \mathfrak{S}_k} \mathrm{sgn}(\sigma)  \sum_{m_1,\dots,m_k=0}^{+\infty}  \prod_{i=1}^k C_{m_i}  e^{-2 m_i(m_i+n-2k+1)t}P_{m_i}^{n-2k,0}(\rho_{\sigma(i)} (0))P_{m_i}^{n-2k,0}(x_i) \\
 =&  \sum_{m_1,\dots,m_k=0}^{+\infty} \left(\prod_{i=1}^k C_{m_i}  e^{-2 m_i(m_i+n-2k+1)t}  P_{m_i}^{n-2k,0}(x_i) \right)\sum_{\sigma \in \mathfrak{S}_k} \mathrm{sgn}(\sigma) \prod_{i=1}^k  P_{m_i}^{n-2k,0}(\rho_{\sigma(i)} (0)) \\
 =&  \sum_{m_1,\dots,m_k=0}^{+\infty} \left(\prod_{i=1}^k C_{m_i}  e^{-2 m_i(m_i+n-2k+1)t} P_{m_i}^{n-2k,0}(x_i) \right)   \mathrm{det} \left( P_{m_i}^{n-2k,0}(\rho_{j} (0))\right)_{1 \le i,j \le k}.
\end{align*} 
By skew-symmetrization, we can rewrite the previous sum as
\[
 \sum_{m_1 < \cdots <m_k} \left(\prod_{i=1}^k C_{m_i}  e^{-2 m_i(m_i+n-2k+1)t}   \right) \mathrm{det} \left( P_{m_i}^{n-2k,0}(x_j)\right)_{1 \le i,j \le k}  \mathrm{det} \left( P_{m_i}^{n-2k,0}(\rho_{j} (0))\right)_{1 \le i,j \le k}.
\]
When $t \to +\infty$, the term of leading order in this sum corresponds to  $(m_1,\dots,m_k)=(0,1,\dots,k-1)$ and, up to a constant, is given by
\[
e^{- \frac{1}{3} k(k-1) \left( 3n-4k+2 \right) t}   \mathrm{det} \left( P_{i-1}^{n-2k,0}(x_j)\right)_{1 \le i,j \le k}  \mathrm{det} \left( P_{i-1}^{n-2k,0}(\rho_{j} (0))\right)_{1 \le i,j \le k}
\]
which up to a constant is 
\[
e^{- \frac{1}{3} k(k-1) \left( 3n-4k+2 \right) t}    h(x) h(\rho(0))
\]
where, for the computation of the Vandermonde determinant, we used the fact that the Jacobi polynomial $P_m$ is a polynomial of degree $m$. The next order in $t$  corresponds to   $(m_1,\dots,m_k)=(0,1,\dots,k-2, k)$ which yields $e^{-2nt}$  in \eqref{quant_1}.

\end{proof}

\begin{remark}
The limit law $\nu$ is therefore the distribution of a Coulomb gas at inverse temperature 2 with a logarithmic confinement potential
\[
V(x)=-(n-2k) \ln (1-x).
\]
\end{remark}

\begin{remark}
Since $\rho_i=\frac{1-\lambda_i}{1+\lambda_i}$, we easily deduce the distribution and the limit law for the eigenvalues process $(\lambda(t))_{t \ge 0}$.

\end{remark}

\begin{remark}\label{Berezin}
As a byproduct, the previous proof exhibited a spectral expansion for the heat kernel of $\mathcal{L}_{n,k}$ with respect to the Lebesgue measure of the form:
\begin{align*}
& e^{ \frac{1}{3} k(k-1) \left( 3n-4k+2 \right) t}  \frac{h(x)}{h(\rho(0))}\sum_{m_1 < \cdots <m_k}\Bigg[\left(\prod_{i=1}^k C_{m_i}  e^{-2 m_i(m_i+n-2k+1)t}   \right)\\
 &\quad\quad  \cdot \mathrm{det} \left( P_{m_i}^{n-2k,0}(x_j)\right)_{1 \le i,j \le k}  \mathrm{det} \left( P_{m_i}^{n-2k,0}(\rho_{j} (0))\right)_{1 \le i,j \le k}\Bigg].
\end{align*}

From spectral theory, we deduce that if $0 \le m_1 < \cdots <m_k$ are integers the function
 \[
 \Phi_{m_1,\dots, m_k}( \rho_1,\dots, \rho_k) =\frac{  \mathrm{det} \left( P_{m_i}^{n-2k,0}(\rho_j)\right)_{1 \le i,j \le k} }{  \mathrm{det} \left( P_{i-1}^{n-2k,0}(\rho_j)\right)_{1 \le i,j \le k} }
 \]
  is an eigenfunction of $\mathcal{L}_{n,k}$ associated to the eigenvalue $$- \frac{1}{3} k(k-1) \left( 3n-4k+2 \right)+2 \sum_{i=1}^k m_i(m_i+n-2k+1).$$

 Those formulas for the spectrum and  zonal eigenfunctions of $G_{n,k}$ were first obtained in Berezin-Karpelevi\v{c}  \cite{MR0095216}. In fact, our approach yields an algebraic representation of such eigenfunctions. Indeed $\Phi_{m_1,\dots, m_k}$ is a symmetric polynomial (a multivariate Jacobi polynomial) and if we consider the unique polynomial function $\Phi^*_{m_1,\dots, m_k}$ defined on the set of $k \times k$ positive definite Hermitian matrices such that for every unitary $M \in \mathbf{U}(k)$, every positive definite Hermitian $X \in \mathbb{C}^{k \times k}$ and every diagonal matrix $D=\mathrm{diag} (\rho_1,\dots,\rho_k)$:
 \begin{equation*}
 \begin{cases}
 \Phi^*_{m_1,\dots, m_k}( M^* X M)= \Phi^*_{m_1,\dots, m_k}( X) \\
 \Phi^*_{m_1,\dots,m_k} (D)= \Phi_{m_1,\dots,m_k} (\rho_1,\dots,\rho_k),
 \end{cases}
 \end{equation*}
then the function $ \Phi^*_{m_1,\dots, m_k} \left((I_k-W^*W)(I_k + W^*W)^{-1} \right)$ is an eigenfunction of $\Delta_{G_{n,k}}$.

\end{remark}
%
%
%

\section{Skew-product decomposition of the Brownian motion of the Stiefel fibration}

\subsection{Connection form and horizontal Brownian motion on $V_{n,k}$}

Let us consider the Stiefel fibration
\begin{equation}\label{eq-fibration2}
\mathbf{U}(k) \rightarrow\widehat{V}_{n,k} \rightarrow \widehat{G}_{n,k}
\end{equation}
that was described in Sections \ref{fibration} and \ref{homogeneous}. According to this fibration, one can see $\widehat{V}_{n,k}$ as a $\mathbf{U}(k)$-principal bundle over $\widehat{G}_{n,k}$. The next lemma gives a formula for the connection form of this bundle.

\begin{lemma}\label{connection form}
Consider on $\widehat{V}_{n,k}$
the $\mathfrak{u}(k)$-valued one form
\begin{align}\label{eq-contact-x}
\omega:=\frac{1}{2} \left( (X^* \,  Z^*)d\begin{pmatrix}  X \\ Z  \end{pmatrix}-d(X^* \, Z^*)\begin{pmatrix}  X \\ Z  \end{pmatrix}\right) =\frac{1}{2} \left(X^*dX-dX^*X +Z^*dZ-dZ^*Z\right).
\end{align}
Then, $\omega$ is the connection form of the bundle $\mathbf{U}(k) \rightarrow\widehat{V}_{n,k} \rightarrow \widehat{G}_{n,k}$.
\end{lemma}

\begin{proof}
We first observe that if $v=\begin{pmatrix} X  \\ Z \end{pmatrix} \in V_{n,k}$, then the tangent space to $V_{n,k}$ at $v$ is given by
\[
T_v V_{n,k}= \left\{ \begin{pmatrix} A  \\ B \end{pmatrix} \in \C^{n \times k}, A^* X+X^* A+B^*Z +Z^*B=0  \right\}.
\]
Then, if $\theta \in \mathfrak{u}_k$, one easily computes that the generator of the one-parameter group $\{ q \to q e^{t \theta}\}_{t \in \mathbb{R}} $ is given by the vector field on $V_{n,k}$ whose value at $v$ is $\begin{pmatrix} X \theta  \\ Z \theta \end{pmatrix}$. Applying $\omega$ to this vector field yields $\theta$.   To show that $\omega$ is the connection form it remains therefore to prove that the kernel of $\omega$ is the horizontal space of the Riemannian submersion $\begin{pmatrix} X  \\ Z \end{pmatrix} \to XZ^{-1}$. This horizontal space at $v$, say $\mathcal{H}_v$, is the orthogonal complement of the vertical space at $v$, which is the subspace $\mathcal{V}_v$ of $T_v V_{n,k}$ tangent to the fiber of the submersion. The previous argument shows that
\[
\mathcal{V}_v = \left\{ \begin{pmatrix} X\theta  \\ Z\theta \end{pmatrix}, \theta \in \mathfrak{u}(k) \right\}.
\]
Therefore we have
\[
\mathcal{H}_v = \left\{ \begin{pmatrix} A  \\ B  \end{pmatrix}\in T_v V_{n,k}, \, \forall \,   \theta \in \mathfrak{u}(k), \, \mathrm{tr} \left( A^* X \theta +B^* Z \theta  \right)=0 \right\}.
\]
We deduce from this that
\[
\mathcal{H}_v = \left\{ \begin{pmatrix} A  \\ B  \end{pmatrix}\in T_v V_{n,k}, \,  A^* X  +B^* Z =X^* A +Z^*B \right\},
\]
from which it is clear that $\omega_{\mid \mathcal{H}}=0$.
\end{proof}

Our next goal is to describe the horizontal lift to $\widehat{V}_{n,k}$ of a Brownian motion on $\widehat{G}_{n,k}$. We still denote by $p:\widehat{V}_{n,k} \to \widehat{G}_{n,k}$ the Riemannian submersion. We recall that a continuous semimartingale $(M_t)_{t \ge 0}$ on $\widehat{V}_{n,k}$ is called \textit{horizontal} if for every $t\ge 0$,   $\int_{M[0,t]} \omega =0$, where $\int_{M[0,t]} \omega$ denotes   the Stratonovich line integral of $\omega$ along the paths of $M$. 
 If $(N_t)_{t \ge 0}$ is a continuous semimartingale on $\widehat{G}_{n,k}$ with $N_0 \in \widehat{G}_{n,k}$, then if $\widetilde{N}_0 \in\widehat{V}_{n,k}$ is such that $\ p(\widetilde{N}_0)=N_0$,  there exists a unique horizontal continuous semimartingale $(\widetilde{N}_t)_{t \ge 0}$ on $\widehat{V}_{n,k}$ such that $p(\widetilde{N}_t)=N_t$ for every $t \ge 0$. The semimartingale $(\widetilde{N}_t)_{t \ge 0}$ is then called the horizontal lift at $\widetilde{N}_0$ of $(N_t)_{t \ge 0}$ to $\widehat{V}_{n,k}$ (See Definition \ref{def lift}). 
 

 We then consider on $\widehat{G}_{n,k}$ the $\mathfrak{u}(k)$-valued one-form $\eta$ given  by
 \begin{align}\label{defeta}
 \eta:=& \frac12 \left( (I_k+w^* w)^{-1/2} (dw^* \, w-w^*dw)(I_k+w^* w)^{-1/2}  \right. \\
  & \left. - (I_k+w^* w)^{-1/2}\, d(I_k+w^* w)^{1/2}+d(I_k+w^* w)^{1/2} \, (I_k+w^* w)^{-1/2} \right). \notag
 \end{align}

\begin{theorem}\label{lift}
Let $(w (t))_{t \ge 0}$ be a Brownian motion on $\widehat{G}_{n,k}$ started at $w_0 \in \widehat{G}_{n,k}$ as in Theorem \ref{main:s1-sphere} and $\mathfrak{a}(t)= \int_{w[0,t]} \eta$. Let $\begin{pmatrix} X_0  \\ Z_0 \end{pmatrix} \in\widehat{V}_{n,k}$ be such that $X_0Z_0^{-1}=w_0$. The process
\[
\widetilde{w}(t):=\begin{pmatrix} w (t)  \\  I_k  \end{pmatrix}(I_k+w (t)^*w (t))^{-1/2}\Theta(t)
\]
is the horizontal lift at $\begin{pmatrix} X_0  \\ Z_0 \end{pmatrix}$ of $(w (t))_{t \ge 0}$ to $\widehat{V}_{n,k}$, where $(\Theta(t))_{t \ge 0}$ is the $\mathbf{U}(k)$-valued  solution of the Stratonovich stochastic differential equation
\begin{align*}
\begin{cases}
d\Theta(t) = \circ d\mathfrak a(t) \, \Theta(t) \\
\Theta_0=(Z_0 Z^*_0)^{-1/2}Z_0.
\end{cases}
\end{align*}
\end{theorem}

\begin{proof}
We use the method explained in the proof of Theorem \ref{skew-product principal bundle}. As before we denote by $p$ the submersion $\begin{pmatrix} X  \\ Z \end{pmatrix} \to XZ^{-1}$. It is easy to check that for every $t \ge 0$, $p( \widetilde{w}(t))=w (t)$ and that $\widetilde{w}_0=\begin{pmatrix} X_0  \\ Z_0 \end{pmatrix}$. It is therefore enough to prove that $\widetilde{w}$ is a horizontal semimartingale, i.e. that $\int_{\widetilde{w}[0,t]} \omega=0$. Denote
\[
X(t)= w (t)(I_k+w (t)^*w (t))^{-1/2}\Theta(t), \, \, Z(t)=(I_k+w (t)^*w (t))^{-1/2}\Theta(t)
\]
A long, but routine, computation shows that
\begin{align*}
&\frac{1}{2} \left(X^*\circ dX-\circ dX^*X +Z^*\circ dZ-\circ dZ^*Z\right)\\
=&-\frac{1}{2} \left( \circ d\Theta^*\Theta-\Theta^*\circ d\Theta+\Theta^*\bigg(\circ d(I_k+J)^{-1/2}\,(I_k+J)^{1/2}-(I_k+J)^{1/2}\circ d(I_k+J)^{-1/2}\bigg)\Theta \right. \\
&+\left. \Theta^*(I_k+J)^{-1/2}(\circ dw^*w-w^*\circ dw)(I_k+J)^{-1/2}\Theta \right).
\end{align*}
where $J =w^* w$. Since $\circ d\Theta^*=\circ d\Theta^{-1}=-\Theta^{-1}\circ d\Theta\, \Theta^{-1}$ and $\circ d\Theta = \circ d\mathfrak a \, \Theta $ with

\begin{align*}
\circ d \mathfrak a=& \frac12(I_k+J)^{-1/2} (\circ dw^* \, w-w^*\circ dw)(I_k+J)^{-1/2} \\
 & -\frac{1}{2} \left( (I_k+J)^{-1/2}\, \circ d(I_k+J)^{1/2}-\circ d(I_k+J)^{1/2} \, (I_k+J)^{-1/2}\right)
 \end{align*}
 we conclude that 
 \[
 \frac{1}{2} \left(X^*\circ dX-\circ dX^*X +Z^*\circ dZ-\circ dZ^*Z\right)=0
 \]
 and thus $\int_{\widetilde{w}[0,t]} \omega=0$.
\end{proof}
\subsection{Skew-product decomposition of the Stiefel Brownian motion}

We now turn to the description of the Brownian motion on $\widehat{V}_{n,k}$ as a skew-product.

\begin{theorem}\label{skew-sphere}
Let $(w (t))_{t \ge 0}$ be a Brownian motion on $\widehat{G}_{n,k}$ started at $w_0=X_0Z_0^{-1} \in \widehat{G}_{n,k}$ as in Theorem \ref{main:s1-sphere} and let $(\Omega(t))_{t \ge 0}$ be a Brownian motion on the unitary group $\mathbf{U}(k)$ independent from $(w (t))_{t \ge 0}$.  Let $(\Theta(t))_{t \ge 0}$ be the $\mathbf{U}(k)$-valued  solution of the Stratonovich stochastic differential equation
\begin{align*}
\begin{cases}
d\Theta(t) = \circ d\mathfrak a(t) \, \Theta(t) \\
\Theta_0=(Z_0 Z^*_0)^{-1/2}Z_0,
\end{cases}
\end{align*}
where  $\mathfrak{a} (t)= \int_{w[0,t]} \eta$.
The process $$\begin{pmatrix} w (t)  \\  I_k  \end{pmatrix}(I_k+w (t)^*w (t))^{-1/2}\Theta(t) \,  \Omega(t)$$
is a Brownian motion on $\widehat{V}_{n,k}$ started at $\begin{pmatrix} X_0  \\ Z_0 \end{pmatrix}$.
\end{theorem}

\begin{proof}
We denote by $\Delta_{\mathcal{H}}$ the horizontal Laplacian and by $\Delta_{\mathcal{V}}$ the vertical Laplacian of the Stiefel fibration; see Section \ref{Section Riemannian submersion} for the definitions of horizontal and vertical Laplacians. Since the submersion $\widehat{V}_{n,k} \to \widehat{G}_{n,k}$ is totally geodesic, the operators $\Delta_{\mathcal{H}}$ and $\Delta_{\mathcal{V}}$ commute. We note that the Laplace-Beltrami operator of $\widehat{V}_{n,k}$ is given by $\Delta_{\widehat{V}_{n,k}} = \Delta_{\mathcal{H}}+\Delta_{\mathcal{V}}$ and that the horizontal lift of the Brownian motion on $\widehat G_{n,k}$ is a diffusion with generator $\frac{1}{2} \Delta_{\mathcal{H}}$. The fibers of the submersion $\widehat{V}_{n,k} \to \widehat G_{n,k}$ are isometric to $\mathbf{U}(k)$, thus if $f$ is a bounded Borel function on $\widehat{V}_{n,k}$, one has
\[
e^{\frac{1}{2} t \Delta_{\mathcal{V}}}f \begin{pmatrix} X  \\ Z \end{pmatrix}=  \mathbb{E} \left ( f \begin{pmatrix} X \Omega(t)  \\ Z \Omega(t) \end{pmatrix}   \right).
\]
Since $e^{\frac{1}{2} t \Delta_{\mathcal{V}}}e^{\frac{1}{2} t \Delta_{\mathcal{H}}}= e^{\frac{1}{2} t \Delta_{\widehat{V}_{n,k} }}$, we conclude from Theorem \ref{lift}.
 \end{proof}

\section{Limit theorems}

Throughout the section, let $(w (t))_{t \ge 0}=(X(t)Z(t)^{-1})_{t \ge 0}$ be a Brownian motion on $\widehat{G}_{n,k}$ where $\begin{pmatrix} X(t)  \\ Z(t) \end{pmatrix}_{t \ge 0}$ is a Brownian motion on $\widehat{V}_{n,k}$ . Our goal is to prove Theorem \ref{limit_oi} below. Without loss of generality we will assume throughout the section that the eigenvalues of $Z_0 Z_0^*$ are distinct; Even  if the eigenvalues of the complex Jacobi process $Z(t)Z(t)^*$ are not distinct for $t=0$, they will be distinct for any $t>0$, and thus from the Markov property, the limit   Theorem \ref{limit_oi} still holds.

\subsection{Main limit theorem}

We first  give a limit theorem for the process $\left(\int_0^t \tr \left( w^*(s) w(s)) ds \right) \right)_{ t \ge 0}$ that shall be used in the next subsections.  Our result is the following:

\begin{theorem}\label{limit_oi}
Let $(J(t))_{t \ge 0}=(w (t)^* w (t))_{t \ge 0}$. The following convergence holds in distribution when $t \to +\infty$
\[
\frac{1}{t^2} \int_0^t  \tr(J)ds \to \mathcal{X},
\]
where $\mathcal{X}$ is a random variable on $[0,+\infty)$ with density $\frac{k(n-k)}{\sqrt{2\pi}  x^{3/2} }e^{-\frac{k^2(n-k)^2}{2x}}$ (therefore $\mathcal{X}$ is the inverse of a gamma distributed random variable).
\end{theorem}

The proof is rather long and will be decomposed in several steps. We first recall that from Lemma \ref{thm-J}, there exists a Brownian motion $(\bB(t))_{t \ge 0}$ in $\C^{k \times k}$ such that:
\begin{equation}\label{eq-dJ}
dJ=\sqrt{I_{k}+J} d\bB^*\sqrt{I_k+J} \sqrt{J}+\sqrt{J}\sqrt{I_k+J}d\bB \sqrt{I_{k}+J}+2(n-k+\tr(J))(I_k+J)dt
\end{equation}

\begin{lemma}\label{thm-J2}
We have 
\begin{equation}\label{eq-det-I+J}
d(\det(I_k+J))=\det(I_k+J)\tr\left( \sqrt{J}(d\bB+d\bB^*) \right)+2\det(I_k+J)\bigg(k(n-k)+ \tr(J) \bigg)dt,
\end{equation}
and therefore 
\begin{equation}\label{eq-logdet-I+J}
d(\log\det(I_k+J))=\tr\left(\sqrt{J}(d\bB+d\bB^*) \right)+2k(n-k)dt.
\end{equation}
\end{lemma}
\begin{proof}
By It\^o's formula we have
\[
d(\det(I_k+J))=\sum_{i,j=1}^k\frac{\partial \det(I_k+J)}{\partial J_{ij}}dJ_{ij}+\frac12 \sum_{i,j, i', j'=1}^k\frac{\partial^2 \det(I_k+J)}{\partial J_{ij}\partial J_{i'j'}}dJ_{ij}dJ_{i'j'}.
\]
First, we know that
\[
\frac{\partial \det(J)}{\partial J_{ij}}=
\frac{\partial \sum_{\ell=1}^k J_{i\ell}\tilde{J}_{i\ell}}{\partial J_{ij}}=\tilde{J}_{ij}
\]
where $\tilde{J}=\det (J)(J^T)^{-1}$ is the cofactor of $J$. Hence the first order term  writes $\det(I_k+J)\tr ((I_k+J)^{-1}dJ)$. Next, we will use the following formula to compute the cross second order derivatives:
\[
\frac{\partial^2 \det(J)}{\partial x\partial y}=(\det (J))\bigg( \tr \left(J^{-1}\frac{\partial^2 J}{\partial x\partial y} \right)
+\tr \left(J^{-1}\frac{\partial J}{\partial x} \right)\tr \left(J^{-1}\frac{\partial J}{\partial y}  \right)
-\tr\left(J^{-1}\frac{\partial J}{\partial x} J^{-1}\frac{\partial J}{\partial y} \right)
\bigg).
\]
Since $\frac{\partial J}{\partial J_{ij}}=E_{ij}$, clearly $\frac{\partial^2 J}{\partial J_{ij}\partial J_{i'j'}}=0$. We also have $J^{-1}\frac{\partial J}{\partial J_{ij}}=\sum_{\ell}(J^{-1})_{\ell i}E_{\ell j} $ and $\tr \left( J^{-1}\frac{\partial J}{\partial J_{ij}}\right)=(J^{-1})_{ji}$. Hence 
\[
\frac{\partial^2 \det(J)}{\partial J_{ij}\partial J_{i'j'}}=(\det (J))\bigg( 
(J^{-1})_{ji}(J^{-1})_{j'i'}
-(J^{-1})_{j'i}(J^{-1})_{ji'}
\bigg),
\]
and
\[
\frac{\partial^2 \det(I_k+J)}{\partial J_{ij}\partial J_{i'j'}}=(\det (I_k+J))\bigg( 
((I_k+J)^{-1})_{ji}((I_k+J)^{-1})_{j'i'}
-((I_k+J)^{-1})_{j'i}((I_k+J)^{-1})_{ji'}
\bigg).
\]
Moreover, from \eqref{eq-dJ} we know that
\[
dJ_{ij}dJ_{i'j'}=2dt\bigg((J+J^2)_{i'j}(I_k+J)_{ij'}+(J+J^2)_{ij'}(I_k+J)_{i'j} \bigg)
\]
Hence we have
\begin{align*}
&d(\det(I_k+J))=\det(I_k+J)\tr ((I_k+J)^{-1}dJ)\\
&+ \sum_{i,j, i', j'=1}^k\det (I_k+J)\bigg( 
((I_k+J)^{-1})_{ji}((I_k+J)^{-1})_{j'i'}
-((I_k+J)^{-1})_{j'i}((I_k+J)^{-1})_{ji'}
\bigg)\\
&\quad\quad\cdot\bigg((J+J^2)_{i'j}(I_k+J)_{ij'}+(J+J^2)_{ij'}(I_k+J)_{i'j} \bigg)dt\\
&=\det(I_k+J)\tr ((I_k+J)^{-1}dJ)
- 2(k-1) \det(I_k+J)\tr(J)dt.
\end{align*}
From \eqref{eq-dJ} we know
\[
\tr ((I_k+J)^{-1}dJ)=\tr\left(\sqrt{J}(d\bB+d\bB^*) \right)+ 2k (n-k+\tr (J))dt.
\]
Hence
\begin{align*}
&d(\det(I_k+J))=\det(I_k+J)\,\tr\left(\sqrt{J}(d\bB+d\bB^*) \right)+ 2\det(I_k+J)(k(n-k)+ \tr(J))dt.
\end{align*}
As a direct consequence of $d\langle \det(I_k+J),\det(I_k+J)\rangle=4\,\det(I_k+J)^2\tr(J)dt$, we obtain \eqref{eq-logdet-I+J} using It\^o's formula.
\end{proof}

\begin{lemma}\label{martingale g}
For every $\alpha \ge 0$ the process
\[
M_t^{\alpha}=e^{2k\alpha(n-k)t}  \left(\frac{\det(I_k+J_0)}{\det(I_k+J(t))}\right)^\alpha \exp\left(-2\alpha^2 \int_0^t \tr(J)ds \right)
\]
is a martingale.
\end{lemma}

\begin{proof}
Consider the exponential local martingale
\[
M_t^\alpha:=\exp\bigg(-\alpha\int_0^t \tr (\sqrt{J} (d\bB+ d\bB^*))-2\alpha^2\int_0^t \tr(J)ds \bigg),
\] 
where $\bB$ is the Brownian motion as given in Theorem \ref{thm-J2}. From Lemma \ref{thm-J2}, we have 
\[
\left(\frac{\det(I_k+J(t))}{\det(I_k+J_0)}\right)^\alpha=\exp\bigg(\alpha\left(\int_0^t\tr\left(\sqrt{J}(d\bB+d\bB^*) \right)+ 2k(n-k) ds \right)\bigg),
\]
and thus
\[
M_t^\alpha=e^{2k\alpha(n-k)t}  \left(\frac{\det(I_k+J_0)}{\det(I_k+J(t))}\right)^\alpha \exp\left(-2\alpha^2 \int_0^t \tr(J)ds \right).
\]
From this expression, it is clear that there exists a constant $C>0$ such that we almost surely have $|M_t^\alpha| \le Ce^{2k\alpha(n-k)t} $ and thus the process $(M_t^\alpha)_{t \ge 0}$ is a martingale.
\end{proof}

In the next lemma, we provide a formula for the Laplace transform of the functional $\int_0^t \tr({J(s)})ds$ using the Girsanov transform corresponding to the martingale $(M_t^\alpha)_{t \geq 0}$. 
Doing so leads to the semi-group density $p_t^{n-2k, \alpha}$ of the one-dimensional Jacobi diffusion with parameters $(n-2k,2\alpha)$. 
For arbitrary parameters $a,b > -1$, the semi-group density $p^{a,b}_t(u,v)$ of the Jacobi operator 
\[
2 (1-v^2)\frac{\partial^2}{\partial v^2}+ 2\left(b-a - (a+b+2)v  \right)\frac{\partial}{\partial v}
\]
admits the following expansion: 

\begin{align}\label{eq-jacobi-kernel}
 & p^{a,b}_t(u,v)=(1+v)^{b}(1-v)^{a} \sum_{m=0}^{+\infty} c_{m,a, b} e^{-2m(m+a+b+1)t}P_m^{a,b}(u)P_m^{a,b}(v), \quad u,v \in [-1,1],
\end{align}
where 
$(P_m^{a,b}(x)$, $m\in \mathbb{Z}_{\ge0})$ are the  Jacobi polynomials defined by the Rodrigues formula:
\[
P_m^{a,b}(x)=\frac{(-1)^m}{2^mm!(1-x)^{a}(1+x)^b}\frac{d^m}{dx^m}((1-x)^{a+m}(1+x)^{b+m}),
\]
and \begin{equation*}
c_{m,a, b}= \frac{2m+a+b+1}{2^{a+b+1}}\frac{\Gamma(m+a+b+1)\Gamma(m+1)}{\Gamma(m+a+1)\Gamma(m+b+1)}
\end{equation*}
is the inverse of the squared $L^2$-norm of $P_m^{a,b}$.

\begin{lemma}\label{kernel J}
For every $\alpha \ge 0$ and $t>0$
\begin{align*}
 & \mathbb E\left( e^{-2\alpha^2 \int_0^t \tr(J)ds}\right)=C e^{ \left(\frac{1}{3} k(k-1) \left( 3n-4k+6\alpha+2 \right) -2k(n-k)\alpha\right) t} \\
 &\quad\quad \cdot  \int_{\Delta_k} \mathrm{det}\left[ \left( \frac{2^{\alpha} p^{n-2k,2\alpha}_{t} \left(\frac{1-\lambda_i(0)}{1+\lambda_i(0)},x_j\right)}{(1+x_j)^{\alpha}}\right)_{ i, j } \right] \prod_{i>j}(x_i-x_j) \, \, dx,
\end{align*}
where 
\[
C=\prod_\ell \frac{( 1+\lambda_\ell(0))^\alpha}{2^\alpha} \prod_{i>j}  \frac{(1+\lambda_i(0))(1+\lambda_j(0))}{2(\lambda_j(0)-\lambda_i(0))}
\]
 is the normalization constant, $\lambda_1(0)\ge\cdots\ge\lambda_k(0)$ are the ordered eigenvalues of $J_0$,  $p^{n-2k,2\alpha}_{t}$ is given by the formula \eqref{eq-jacobi-kernel}  and 
\[
\Delta_k=\{ x \in [-1,1]^k, -1 \le x_1 <\cdots <x_k \le 1 \}. 
\]

\end{lemma}

\begin{proof}
Let $\alpha \ge 0$ and consider the probability measure $\mathbb{P}^{\alpha}$ defined by
\[
\mathbb{P}^{\alpha}|_{\mathcal{F}_t}=M^{\alpha}_t\cdot \mathbb{P}|_{\mathcal{F}_t}.
\] 
We first note that
\begin{align}\label{e-alpha}
\mathbb E\left( e^{-2\alpha^2 \int_0^t \tr(J)ds}\right)=e^{-2k(n-k)\alpha t}\, \mathbb E^{\alpha} \left[ \left(\frac{\det(I_k+J(t))}{\det(I_k+J_0)}\right)^\alpha  \right].
\end{align}
From Girsanov theorem, the process
\[
\beta(t)=\bB(t)+2 \alpha\int_0^t\sqrt{J} ds
\]
is under $\mathbb{P}^{\alpha}$ a $k \times k$-matrix-valued Brownian motion and we have
\[
dJ=\sqrt{I_{k}+J} d\beta^*\sqrt{I_k+J} \sqrt{J}+\sqrt{J}\sqrt{I_k+J}d\beta \sqrt{I_{k}+J}+2\left(n-k-2\alpha J+\tr(J)\right)(I_k+J)dt.
\]
We now denote by $\lambda(t)= (\lambda_i(t))_{1\le i\le k}$ the eigenvalues of $J(t)$, $t\ge0$. From the previous equation satisfied by $J$ we deduce that there exists a Brownian motion $(B(t))_{t \ge 0}$ in $\mathbb{R}^k$ for the probability measure $\mathbb{P}^{\alpha}$  such that
\begin{multline}\label{eq-lambda-sde-Pl}
d\lambda_i=2(1+\lambda_i)\sqrt{\lambda_i}dB^i+2(1+\lambda_i) \Bigg(n-2k+1-(2k+2\alpha-3)\lambda_i  + \\ 
 2\lambda_i (1+\lambda_i)\sum_{\ell\not=i} \frac{1}{\lambda_i-\lambda_\ell}\Bigg)dt.
\end{multline}
Let us denote  
\begin{equation*}
\rho_i=\frac{1-\lambda_i}{1+\lambda_i},\quad i=1,\dots, k.    
\end{equation*}
Then, using It\^o's formula and the previous equation, we have
\begin{align*}
d\rho_i =-2\sqrt{1-\rho_i^2}dB^i-2\left( \left(n-2k-2\alpha+(n-2k+2\alpha+2)\rho_i\right)+2\sum_{\ell\not=i}\frac{1-\rho_i^2}{\rho_\ell-\rho_i}\right)dt.
\end{align*}
In other words, the process $(\rho_i)_{i=1}^k$ is under $\mathbb{P}^{\alpha}$ a Vandermonde transform of $k$ independent real Jacobi processes. Consequently, by an argument similar to that of the proof of Theorem \ref{density rho} its semi-group density with respect to the Lebesgue measure $dx$ is given by 
\[
e^{ \frac{1}{3} k(k-1) \left( 3n-4k+6\alpha+2 \right) t}  \frac{\prod_{i>j}(x_i-x_j)}{\prod_{i>j}(\rho_i(0)-\rho_j(0))} \mathrm{det} \left( p^{n-2k,2\alpha}_{t} (\rho_i(0),x_j)\right)_{1\le i, j \le k}  \, \mathbf{1}_{\Delta_k} (x).
\]
We conclude then with \eqref{e-alpha}.
\end{proof}


Now, the proof of Theorem \ref{limit_oi} goes as follows: 
from Lemma \ref{kernel J}, we have for every $\alpha \ge 0$ and $t>0$
\begin{align}\label{laplace-asymp}
 \mathbb E\left( e^{-2\alpha^2 \int_0^t \tr(J)ds}\right) 
 =& C e^{ \left(\frac{1}{3} k(k-1) \left( 3n-4k+6\alpha+2 \right) -2k(n-k)\alpha\right) t}  \\
 &\cdot\int_{\Delta_k} \mathrm{det} \left( \frac{p^{n-2k,2\alpha}_{t} \left(\frac{1-\lambda_i(0)}{1+\lambda_i(0)},x_j\right)}{(1+x_j)^{\alpha}}\right)_{ i, j }  \prod_{i>j}(x_i-x_j) \, \, dx. \notag
\end{align}

In order to analyze the large time behavior of this Laplace transform  we we will use the formula \eqref{eq-jacobi-kernel}.  We can write
\begin{align*}
 p^{n-2k,2\alpha}_t(x,y)&=(1+y)^{2\alpha}(1-y)^{n-2k}\\
&\quad\quad\cdot\sum_{m=0}^{+\infty} c_{m,n-2k+ 2\alpha} e^{-2m(m+n-2k+2\alpha+1)t}P_m^{n-2k,2\alpha}(x)P_m^{n-2k,2\alpha}(y),
\end{align*}
Similarly to the proof of Theorem \ref{limit eigenvalue}, denoting as before $\rho_i(0)=\frac{1-\lambda_i(0)}{1+\lambda_i(0)}$ we now compute 
\begin{align*}
 & \mathrm{det} \left( p^{n-2k,2\alpha}_{t} (\rho_i(0),x_j)\right)_{1\le i, j \le k} \\
 =& \sum_{\sigma \in \mathfrak{S}_k} \mathrm{sgn}(\sigma) \prod_{i=1}^k  p^{n-2k,2\alpha}_{t}(\rho_{\sigma(i)} (0),x_i) \\
 =& \sum_{\sigma \in \mathfrak{S}_k} \mathrm{sgn}(\sigma) \prod_{i=1}^k \bigg[ (1-x_i)^{n-2k} (1+x_i)^{2\alpha} \sum_{m=0}^{+\infty} \bigg( c_{m,n-2k+ 2\alpha} e^{-2m(m+n-2k+2\alpha+1)t}\\
 &\qquad\qquad\cdot P_m^{n-2k,2\alpha}(\rho_{\sigma(i)} (0))P_m^{n-2k,2\alpha}(x
 _i)\bigg)\bigg]\\
 =& V_\alpha(x) \sum_{\sigma \in \mathfrak{S}_k} \mathrm{sgn}(\sigma)  \sum_{m_1,\dots,m_k=0}^{+\infty}  \prod_{i=1}^k c_{m_i,n-2k+ 2\alpha} e^{-\lambda_{m_i}t}  P_{m_i}^{n-2k,2\alpha}(\rho_{\sigma(i)} (0))P_{m_i}^{n-2k,2\alpha}(x_i)
\end{align*}
where $V_\alpha (x)=\prod_{i=1}^k (1-x_i)^{n-2k}(1+x_i)^{2\alpha}$ and $\lambda_{m_i}=2m_i(m_i+n-2k+2\alpha+1)$. We can now write
\begin{align*}
 &  \sum_{\sigma \in \mathfrak{S}_k} \mathrm{sgn}(\sigma)  \sum_{m_1,\dots,m_k=0}^{+\infty}  \prod_{i=1}^k c_{m_i,n-2k+ 2\alpha} e^{-\lambda_{m_i}t}P_{m_i}^{n-2k,2\alpha}(\rho_{\sigma(i)} (0))P_{m_i}^{n-2k,2\alpha}(x_i) \\
 =&  \sum_{m_1,\dots,m_k=0}^{+\infty} \left(\prod_{i=1}^k c_{m_i,n-2k+ 2\alpha} e^{-\lambda_{m_i}t} P_{m_i}^{n-2k,2\alpha}(x_i) \right)\sum_{\sigma \in \mathfrak{S}_k} \mathrm{sgn}(\sigma) \prod_{i=1}^k  P_{m_i}^{n-2k,2\alpha}(\rho_{\sigma(i)} (0)) \\
 =&   \sum_{m_1,\dots,m_k=0}^{+\infty} \left(\prod_{i=1}^k c_{m_i,n-2k+ 2\alpha} e^{-\lambda_{m_i}t} P_{m_i}^{n-2k,2\alpha}(x_i) \right)  \mathrm{det} \left( P_{m_i}^{n-2k,2\alpha}(\rho_{j} (0))\right)_{1 \le i,j \le k}.
\end{align*} 
By skew-symmetrization, we can rewrite the previous sum as
\begin{align*}
 \sum_{m_1 < \cdots <m_k} \left(\prod_{i=1}^k c_{m_i,n-2k+ 2\alpha} e^{-\lambda_{m_i}t}  \right)\mathrm{det} \left( P_{m_i}^{n-2k,2\alpha}(x_j)\right)_{1 \le i,j \le k}  \mathrm{det} \left( P_{m_i}^{n-2k,2\alpha}(\rho_{j} (0))\right)_{1 \le i,j \le k}.
\end{align*}
Let us note that when $t \to +\infty$, the term of leading order in this sum corresponds to  $(m_1,\dots,m_k)=(0,1,\dots,k-1)$ (this is the Weyl vector of the Lie group $\mathbf{SO}(2k)$) and is given by
\begin{align}\label{skew-fu}
\left(\prod_{i=1}^k c_{i-1,n-2k+ 2\alpha}   \right) e^{- \frac{1}{3} k(k-1) \left( 3n-4k+6\alpha+2 \right) t}   \mathrm{det} \left( P_{i-1}^{n-2k,2\alpha}(x_j)\right)_{1 \le i,j \le k}  \mathrm{det} \left( P_{i-1}^{n-2k,2\alpha}(\rho_{j} (0))\right)_{1 \le i,j \le k}
\end{align}
%

Using \eqref{skew-fu}, one then deduces that for every $\lambda \ge 0$,
\begin{equation}\label{eq-limit-Laplace}
\lim_{t \to +\infty} \mathbb E\left( e^{-\frac{\lambda}{t^2} \int_0^t \tr(J)ds}\right) = \tilde{C} e^{  -k(n-k) \sqrt{2\lambda} } ,
\end{equation}
where $\tilde{C}$ is a constant depending only on $n,k,\rho_i(0)$. For $\lambda=0$, $\mathbb E\left( e^{-\frac{\lambda}{t^2} \int_0^t \tr(J)ds}\right)=1$ and therefore $\tilde{C}=1$.
One now concludes using an inverse Laplace transform that the following convergence takes place in distribution when $t \to +\infty$
\[
\frac{1}{t^2} \int_0^t  \tr(J)ds \to \mathcal{X},
\]
where $\mathcal{X}$ is a random variable on $[0,+\infty)$ with density \begin{equation*}
\frac{k(n-k)}{\sqrt{2\pi}  x^{3/2} }e^{-\frac{k^2(n-k)^2}{2x}}.    
\end{equation*} 
We incidentally note that $\mathcal{X}$ is therefore distributed as the hitting time of $k(n-k)$ by a one-dimensional Brownian motion, even though this does not seem to be readily explainable. 

\begin{remark}
The integral displayed in Lemma \ref{kernel J} may be computed explicitly. Indeed, the semi-group density of the process $(\rho_i)_{i=1}^k$ under $\mathbb{P}^{\alpha}$ may be expanded in the basis of symmetric Jacobi polynomials as displayed in \eqref{densite} below (up to an affine transformation). One is then led to the integral of the product of a symmetric Jacobi polynomials and the symmetric function \begin{equation*}
(x_i)_{i=1}^k \mapsto \prod_{i=1}^k(1+x_i)^{\alpha}.    
\end{equation*}
The explicit expression of this integral is given by Lemma 3.2 in \cite{OO}.  Consequently, The Laplace transform displayed in Lemma \ref{kernel J} is given by a series of symmetric Jacobi polynomials in $(\rho_i(0))_{i=1}^k$. As a matter of fact, the limit as 
$t \rightarrow +\infty$ of the functional 
\begin{equation*}
\frac{\alpha}{t^2} \int_0^t \tr(J)ds    
\end{equation*}
is governed by the constant term of the resulting series and one readily retrieves the result stated in Theorem \ref{limit_oi}. 
\end{remark}

\subsection{Asymptotics of a generalized stochastic area}

By the definition of the one-form $\eta$ in \eqref{defeta}, we note that
\begin{align}\label{eq-tr-eta}
\int_{w[0,t]} \mathrm{tr} (\eta) & =\frac12 \mathrm{tr} \left[ \int_0^t   (I_k+J)^{-1/2} (\circ dw^* \, w-w^*\circ dw)(I_k+J)^{-1/2}\right] \notag\\
 & =\frac12 \mathrm{tr} \left[ \int_0^t   (I_k+J)^{-1/2} ( dw^* \, w-w^* dw)(I_k+J)^{-1/2}\right]
\end{align}
where as before $J=w^* w$. From simple but lenghty computations one can verify that
\[
\tr(d\eta)=\partial \overline{\partial}\log \det(I_k+w^*w),
\]
which implies that $i\tr(d\eta)$ is the K\"ahler form on $\widehat{G}_{n,k}$ (A similar computation is done at the beginning of Section \ref{Limit hyp stiefel kahl}). Actually, the function $K(w)=\log\det(I_k+w^*w)$ is a K\"ahler potential and 
\[
\mathrm{tr} (\eta)=-\frac{1}{2} (\partial -\bar{\partial})K.
\]
Therefore $i\int_{w[0,t]} \mathrm{tr} (\eta)$ can be considered as a generalized stochastic area process on $\widehat{G}_{n,k}$. In the proposition below we deduce large time limit distributions of such functional.

\begin{proposition}\label{thm-main-3}

The following convergence holds in distribution when $t \to +\infty$
\[
\frac{1}{it} \int_{w[0,t]} \mathrm{tr} (\eta) \to \mathcal{C}_{k(n-k)},
\]
where $\mathcal{C}_{k(n-k)}$ is a Cauchy distribution of parameter $k(n-k)$.
\end{proposition}

\begin{proof}
Using \eqref{w dynamics}, similarly to the proof of \eqref{eq-dJ} one can check that
\begin{align*}
dw^*w-w^*dw=\sqrt{I_{k}+J} d\mathbf{B}^*\sqrt{I_k+J}\sqrt{J}-\sqrt{J}\sqrt{I_k+J}d\mathbf{B}\sqrt{I_{k}+J}
\end{align*}
where $(\bB(t))_{t\ge0}$ is a $k\times k$-matrix-valued Brownian motion.
Therefore,
\[
 (I_k+J)^{-1/2} ( dw^* \, w-w^* dw)(I_k+J)^{-1/2}= d\bB^*\sqrt{J}- \sqrt{J}d\bB
\]
Consider then the diagonalization of $J=V\Lambda V^*$, where $V\in U(k)$ and $\Lambda=\mathrm{diag} \{\lambda_1,\dots, \lambda_k\}$. We obtain
\begin{align*}
d\bB^*\sqrt{J}- \sqrt{J}d\bB=V( V^{-1}d\bB^* V \sqrt{\Lambda}- \sqrt{\Lambda}V^{-1} d\bB V) V^{-1}.
\end{align*}
Therefore from \eqref{eq-tr-eta}, we have in distribution that
\[
\int_{w[0,t]} \mathrm{tr} (\eta)=  i \mathcal{B}_{\int_0^t \mathrm{tr}(J) ds} 
\]
where $\mathcal{B}$ is a one-dimensional Brownian motion independent from the process $\mathrm{tr}(J)$. Therefore, for every $\lambda>0$,
\[
\mathbb{E} \left( e^{-\lambda\frac{1}{i} \int_{w[0,t]} \mathrm{tr} (\eta)}  \right)=\mathbb{E} \left( e^{-\lambda  \mathcal{B}_{\int_0^t \mathrm{tr}(J) ds}} \right)=\mathbb{E} \left( e^{-\frac{\lambda^2}{2} \int_0^t  \mathrm{tr}(J)ds} \right).
\]

We conclude then from \eqref{eq-limit-Laplace} after straightforward computations
\end{proof}

 \subsection{Asymptotic windings}

We are now interested in the windings of the complex valued process $\det (Z(t))$. We first note that from Theorem \ref{skew-sphere}, we have identity in law
\[
\det (Z(t)) =\det (I_k+w (t)^*w (t))^{-1/2} \det \Theta(t) \, \det  \Omega(t).
\]
We shall then use the following result.

\begin{lemma}\label{trace lemma}
Let $G$ be a connected  Lie group of $m \times m$ matrices with Lie algebra $\mathfrak{g}$. Let $(M_t)_{t \ge 0}$ be a $\mathfrak{g}$-valued continuous semimartingale such that $M_0=0$ and let $(C_t)_{t \ge 0}$ be the $G$-valued solution of the Stratonovich stochastic differential equation
\begin{align*}
dC_t = (\circ d  M_t ) \, C_t 
\end{align*}
Then, for $t \ge 0$, $\det C_t =(\det C_0 )\exp \left( \tr (M_t) \right )$.
\end{lemma}

\begin{proof}
Let $T>0$. Consider on the time interval $[0,T]$ the sequence of $G$-valued  semimartingales $(C_t^n)_{0 \le t \le T}$ inductively defined by
\begin{align*}
C^n_t = C_{t_k}^n \exp \left( \frac{2^n}{T} (t -t_k) (M_{t_{k+1}} -M_{t_k}) \right), \quad t_k \le t \le t_{k+1},
\end{align*}
where $t_k=\frac{kT}{2^n}$, $k =0,\dots, 2^n$. From Theorem 2 in \cite{MR942031}, the sequence of semimartingales $(C_t^n)_{0 \le t \le T}$ converges in probability to $(C_t)_{0 \le t \le T}$ uniformly on $[0,T]$. However,
\begin{align*}
\det (C^n_t) = \det(C_{t_k}^n)  \exp \left( \frac{2^n}{T} (t -t_k) \tr (M_{t_{k+1}} -M_{t_k}) \right), \quad t_k \le t \le t_{k+1}.
\end{align*}
We deduce therefore by induction that 
\[
\det (C^n_T)=(\det C_0 ) \exp \left( \tr (M_T)  \right).
\]
Letting then $n \to +\infty$ yields the conclusion. 
\end{proof}

Using the previous lemma, we deduce the following:

\begin{lemma}\label{Lemma theta}
For every $t \ge 0$, 
$
\det \Theta(t)=\frac{\det Z_0}{ | \det Z_0|} \exp \left( \int_{w[0,t]} \mathrm{tr} (\eta) \right).
$
\end{lemma}

\begin{proof}
We have
\begin{align*}
\begin{cases}
d\Theta(t) = \circ d \left( \int_{w[0,t]} \eta \right) \, \Theta(t) \\
\Theta_0=(Z_0 Z^*_0)^{-1/2}Z_0.
\end{cases}
\end{align*}
Thus from Lemma \ref{trace lemma} we have $\det \Theta(t)= \frac{\det Z_0}{ | \det Z_0|} \exp \left( \mathrm{tr} \left( \int_{w[0,t]} \eta \right) \right)$.
\end{proof}

We are now finally in position to prove the following result.

\begin{theorem}\label{winding section}
One has the polar decomposition
\[
\det (Z(t))=\varrho(t) e^{i\theta(t)}
\]
where $0< \varrho(t) \le 1$ is a continuous semimartingale  and $\theta(t)$ is a continuous martingale such that the following convergence holds in distribution when $t \to +\infty$
\[
\frac{\theta(t)}{t} \to \mathcal{C}_{k(n-k)},
\]
where $\mathcal{C}_{k(n-k)}$ is a Cauchy distribution of parameter $k(n-k)$.
\end{theorem}

\begin{proof}
From the decomposition $\det (Z(t)) =\det (I_k+w (t)^*w (t))^{-1/2} \det \Theta(t) \, \det  \Omega(t)$ one deduces from Lemmas  \ref{trace lemma} and \ref{Lemma theta} that
\[
\varrho(t) =\det(I_k+J(t))^{-1/2}, \, \, i \theta(t) = i\theta_0+ \mathrm{tr}(D_t)+ \int_{w[0,t]} \mathrm{tr} (\eta) 
\]
where $D_t$ is a Brownian motion on $\mathfrak{u}(k)$ independent from $w$ and $\theta_0$ is such that $ e^{i\theta_0}=\frac{\det Z_0}{| \det Z_0|}$ . The conclusion follows then from Proposition \ref{thm-main-3}.
\end{proof}

Let us remark that it is also possible to compute the asymptotic law of the radial part $\varrho(t)$.   Indeed,  from the previous proof, we know that $\varrho(t) =\det(I_k+J(t))^{-1/2}$ and the limit distribution of the ordered eigenvalues of $(I_k-J)(I_k+J)^{-1}$ is computed explicitly in Theorem \ref{limit eigenvalue} to be a distribution with density
\[
c_{n,k} \prod_{1 \le i < j \le k}(x_i-x_j)^2  \prod_{i=1}^k (1-x_i)^{n-2k}  \, \, \mathbf{1}_{\Delta_k} (x) dx.
\]
Using then the Selberg's integral formula, one obtains that in distribution one has
\[
\varrho(t) \to \varrho_\infty,
\]
where $\rho_\infty$ is a random variable such that for every $s \ge 0$
\[
\mathbb{E}( \varrho_\infty^s)=\tilde{c}_{n,k}\prod_{j=0}^{k-1}\frac{\Gamma\left( \frac{s}{2}+j+1 \right)}{\Gamma \left( \frac{s}{2}+n-k+j+1 \right)}
\]
where $\tilde{c}_{n,k}$ is a normalization constant. Thus, using uniqueness of the Mellin transform, one concludes that
\[
 \varrho_\infty^2=\prod_{j=1}^k \mathfrak{B}_{j,n-k}
\]
where $\mathfrak{B}_{j,n-k}$ are independent beta random variables with parameters $(j,n-k)$. This recovers a result by A. Rouault (Proposition 2.4 in \cite{MR2365642}).

\subsection{The case $k \ge n-k$}\label{duality section}

In this section, we show how to handle the case of $k \ge n-k$. Thus, unlike the rest of the chapter, we assume in this section that $k \ge n-k$. This is essentially a duality argument equivalent to the isomorphism $G_{n,k} \simeq G_{n,n-k}$. Let
\[
U=\begin{pmatrix}  X & Y \\ Z & V \end{pmatrix}  \in \mathbf{U}(n)
\]
with $Z \in \mathbb{C}^{k \times k}$, $\det (Z) \neq 0$. Using $X^* X+ Z^*Z=I_k$ and $XX^* +YY^*=I_{n-k}$ we deduce the following equality of spectrum
\[
\mathbf{sp} ( Z^* Z)=\mathbf{sp} ( YY^*) \cup \{ 1 \}
\]
and that the eigenvalue 1 of $Z^* Z$ has multiplicity at least $2k-n$. In particular, we have
\[
\det ( Z^* Z)= \det ( Y Y^*)
\]
and
\[
\mathrm{tr}[(Z Z^*)^{-1}-I_k]= \mathrm{tr}[(Y Y^*)^{-1}-I_{n-k}].
\]
Consider now a Brownian motion
\[
U(t)=\begin{pmatrix}  X(t) & Y(t) \\ Z(t) & V(t) \end{pmatrix}  \in \mathbf{U}(n)
\]
with $Z_0 \in \mathbb{C}^{k \times k}$, $\det (Z_0) \neq 0$. The process $(U(t)^*)_{t \ge 0}$ is also a Brownian motion on $\mathbf{U}(n)$. Therefore,  when $t \to +\infty$,
\[
| \det (Z(t))|^2=| \det (Y(t)^*)|^2\to \prod_{j=1}^{n-k} \mathfrak{B}_{j,k}
\]
For the study of the winding process of $\det (Z(t))$, we notice that in the case $k \le n-k$ the only part of the proof of Theorem \ref{winding section} that actually uses the fact that $k \le n-k$ is the proof of Theorem \ref{limit_oi} (in the case $k > n-k$ the stochastic differential equation for  $J=w^*w$ is only defined up to the hitting time $\inf \{ t \ge 0, \det J(t) =0 \} <+\infty$ and the Girsanov transform method fails). To handle the case $k \ge n-k$, we note that $\mathrm{tr} J=\mathrm{tr}[(Z Z^*)^{-1}-I_k]= \mathrm{tr}[(Y Y^*)^{-1}-I_{n-k}]$. Since the process $(U(t)^*)_{t \ge 0}$ is also a Brownian motion on $\mathbf{U}(n)$, one can use the argument of the proof of Theorem \ref{limit_oi} to deduce that
\[
\frac{1}{t^2} \int_0^t  \mathrm{tr}[(Y Y^*)^{-1}-I_{n-k}] ds \to X,
\]
where $X$ is a random variable on $[0,+\infty)$ with density $\frac{k(n-k)}{\sqrt{2\pi}  x^{3/2} }e^{-\frac{k^2(n-k)^2}{2x}}$. Therefore, the conclusion of Theorem \ref{winding section} still holds in the case of $k \ge n-k$.

\chapter{Horizontal Brownian motions of the hyperbolic Stiefel fibration}

\section{Brownian motion on  hyperbolic complex Grassmannian manifolds}\label{Sec2}

\subsection{The hyperbolic complex Grassmannian manifold and affine coordinates}\label{sec:notations}

Let $n\in \mathbb{N}$, $n \ge 2$, and $k\in\{1,\dots, n-1\}$. We consider the non-compact Lie group of matrices
\[
\mathbf{U}(n-k,k)=\left\{ M \in GL(n,\mathbb{C}), M^* \begin{pmatrix} I_{n-k}   & 0 \\ 0 & -I_{k} \end{pmatrix} M =  \begin{pmatrix} I_{n-k}   & 0 \\ 0 & -I_{k} \end{pmatrix}\right\}.
\]
which belongs to the family of classical Lie groups of matrices (see \cite{Hec-Sch}).
The Lie algebra of $\mathbf{U}(n-k,k)$ is given by
\[
\mathfrak{u}(n-k,k)=\left\{ A \in \C^{n \times n}, A^* \begin{pmatrix} I_{n-k}   & 0 \\ 0 & -I_{k} \end{pmatrix} + \begin{pmatrix} I_{n-k}   & 0 \\ 0 & -I_{k} \end{pmatrix}A= 0 \right\}.
\]
Note that the real dimension of $\mathfrak{u}(n-k,k)$ and hence $\mathbf{U}(n-k,k)$ is $n^2$. 

We consider the complex hyperbolic Stiefel manifold $HV_{n,k}$ defined as the set
\begin{equation}\label{eq-HV}
HV_{n,k}=\left\{ \begin{pmatrix} X  \\ Z \end{pmatrix}  \in \C^{n\times k} , X \in  \C^{(n-k)\times k}, Z \in \C^{k\times k} | X^* X-Z^*Z=-I_k \right\}.
\end{equation}
It is a non-compact algebraic real sub-manifold of $\C^{n\times k}$ with real dimension $2nk-k^2$.
We note that $\mathbf{U}(n-k,k)$ acts transitively on $HV_{n,k}$, the action being  defined by $(g,M) \to gM$, $g \in \mathbf{U}(n-k,k)$, $M \in HV_{n,k}$. For this action, the isotropy group of $\begin{pmatrix} 0 \\ I_k \end{pmatrix} \in HV_{n,k}$ is simply the unitary group $\mathbf{U}(n-k)$, where $\mathbf{U}(n-k)$ is identified with the normal subgroup:
\[
\left\{   \begin{pmatrix}  Y  & 0 \\ 0 & I_{k} \end{pmatrix} , Y \in \mathbf{U}(n-k) \right\}.
\]
Therefore $HV_{n,k}$ can be identified with the coset space $\mathbf{U}(n-k,k) / \mathbf{U}(n-k)$. 

In addition, there is a right  action of the  unitary group $\mathbf{U}(k)$ on $HV_{n,k}$, which is given by the right matrix multiplication: $M  g$, $M \in HV_{n,k}$,  $g \in \mathbf{U}(k)$. The quotient space by this action 
$HG_{n,k}:=HV_{n,k} / \mathbf{U}(k)$ is called the \textit{complex hyperbolic Grassmannian manifold} and we shall denote by $\pi$ the projection $HV_{n,k}\to HG_{n,k}$. 
Besides, $HG_{n,k}$ is a non-compact complex manifold with complex dimension $k(n-k)$. Actually, as a  manifold, $HG_{n,k}$ is diffeomorphic to $\C^{(n-k)\times k}$. Indeed, consider the smooth map $p : HV_{n,k} \to \C^{(n-k)\times k}$ given by $p \begin{pmatrix} X  \\ Z \end{pmatrix}  = X Z^{-1}.$ It is clear that for every $g \in \mathbf{U}(k)$ and $M \in HV_{n,k}$, $p( M  g)= p(M)$. Since $p$ is a submersion from $HV_{n,k}$ onto its image $ p ( HV_{n,k}) = \C^{(n-k)\times k} $ we easily deduce that there exists a unique  diffeomorphism $\Psi$ from   $HG_{n,k}$ onto $\C^{(n-k)\times k}$ such that 
\begin{align}\label{inh-coord}
 \Psi \circ \pi= p.
 \end{align}
 The map $\Psi$ induces a global coordinate chart on $HG_{n,k}$ that we call homogeneous and through those global coordinates, we will often simply identify $HG_{n,k} \simeq \C^{(n-k)\times k}$.
 
 For the purpose of constructing and studying the Brownian motion on $HG_{n,k}$ we need to equip $HG_{n,k}$ with a Riemannian metric. To proceed, set 
  \[
 I_{n-k,k} :=\begin{pmatrix} I_{n-k}   & 0 \\ 0 & -I_{k} \end{pmatrix}.
 \]
and note that $A \in \mathfrak{u}(n-k,k)$ if and only if $I_{n-k,k}A$ is skew-Hermitian. Therefore, 
 $$\langle A, B\rangle_{\mathfrak{u}(n-k,k)}=-\frac12 \tr(I_{n-k,k}AI_{n-k,k}B),$$
is a non-degenerate $Ad$-invariant inner product on $\mathfrak{u}(n-k,k)$ where $Ad$ is the adjoint action of $U(n-k,k)$ on $\mathfrak{u}(n-k,k)$.  Moreover, an orthonormal basis of $\mathfrak{u}(n-k,k)$ with respect to this inner product is given by
\begin{align*}
 &\{ T_\ell, 1 \le \ell \le n \} \\
&\cup \{E_{\ell j}-E_{j \ell}, i(E_{\ell j}+E_{j\ell}),  1\le \ell <j \le n-k, n-k+1\le \ell <j \le n\}\\
&\cup\{E_{\ell j}+E_{j \ell}, i(E_{\ell j}-E_{j\ell}), 1\le \ell\le n-k, n-k+1\le j\le n\}
\end{align*}
where $E_{j\ell}=(\delta_{j\ell}(k,h))_{1\le k,h\le n}$ and $T_j=\sqrt{2}i E_{jj}$. In this way, we obtain a bi-invariant Riemannian structure on $\mathbf{U}(n-k, k)$. 

The inner product above also equips the coset space $HV_{n,k} \simeq \mathbf{U}(n-k,k) / \mathbf{U}(n-k)$ with the unique Riemannian structure such that  the map $ M \in \mathbf{U}(n-k, k) \to M \begin{pmatrix} 0 \\ I_k \end{pmatrix} \in HV_{n,k} $ is a Riemannian submersion. Finally, we can equip the hyperbolic Grassmannian $HG_{n,k} $  with the Riemannian metric such that the map 
\begin{equation*}
p: \begin{pmatrix} X  \\ Z \end{pmatrix}  \in HV_{n,k} \mapsto X Z^{-1} \in HG_{n,k}
\end{equation*}
 is a Riemannian submersion with totally geodesic fibers isometric to $\mathbf{U}(k)$. Equipped with that metric, $HG_{n,k} $ is a complex Riemannian symmetric space of non compact type and rank $\min (k,n-k) $. It is the dual symmetric space of the complex Grassmannian space $G_{n,k} $ that was considered in the previous chapter.

We therefore have a fibration
\[
\mathbf{U}(k) \to HV_{n,k} \to HG_{n,k}
\]
which we will  referred to as the hyperbolic Stiefel fibration. Note that this fibration is a special case of Example \ref{BB fibration} with $G=\mathbf{U}(n-k,k)$, $H=\mathbf{U}(n-k)\mathbf{U}(k)$, $K=\mathbf{U}(n-k)$.

For $k=1$, $HV_{n,1}$ is isometric to the complex anti-de Sitter space  $\mathbf{AdS}^{2n-1}$ equipped with its Riemannian metric, $HG_{n,1}$ is the complex hyperbolic space $\mathbb{C} H^{n-1}$ and the above fibration is the anti-de Sitter fibration considered in Chapter 4.  

On the other hand, due to the symmetric space representation $HG_{n,k} \simeq \mathbf{U}(n-k,k)/ \mathbf{U}(n-k)\mathbf{U}(k)$, one has the duality $HG_{n,k} \simeq HG_{n,n-k}$ which leads to the fibration
\[
\mathbf{U}(n-k) \to HV_{n,n-k} \to HG_{n,n-k}.
\]
Therefore without loss of generality and unless specified otherwise, we will always assume throughout the chapter that $k \le n-k$ as we did in the study of the Stiefel fibration. As such, the rank of $HG_{n,k}$ as a symmetric space is $k$.

 \subsection{Brownian motion on $HG_{n,k}$}
 
 In this section, we study the Brownian motion on $HG_{n,k}$ and show how it can be constructed from a Brownian motion on the indefinite Lie group $\mathbf{U}(n-k,k)$.
 
  A Brownian motion $(A(t))_{t\ge0}$ on the Lie algebra $\mathfrak{u}(n-k,k)$ is  given by
\begin{align}\label{eq-A(t)}
A(t) &=\sum_{1\le \ell<j\le n-k}(E_{\ell j}-E_{j\ell})B^{\ell j}_t+i(E_{\ell j}+E_{j\ell})\tilde{B}^{\ell j}_t+\sum_{j=1}^k T_{j}\hat{B}^J(t)\notag\\
&+\sum_{n-k+1\le \ell<j\le n}(E_{\ell j}-E_{j\ell})B^{\ell j}_t+i(E_{\ell j}+E_{j\ell})\tilde{B}^{\ell j}_t+\sum_{j=n-k+1}^n T_{j}\hat{B}^J(t)\\
&+\sum_{1\le \ell \le n-k}\sum_{n-k+1\le j\le n}(E_{\ell j}+E_{j\ell})B^{\ell j}_t+i(E_{\ell j}-E_{j\ell})\tilde{B}^{\ell j}_t\notag
\end{align}
where $B^{\ell j}$, $\tilde{B}^{\ell j}$, $\hat{B}^j$ are  independent standard real Brownian motions.
In the following, we will  use the block notations as below: For any $U\in \mathbf{U}(n-k,k)$ 
\[
U=\begin{pmatrix}  Y & X \\ W & Z \end{pmatrix},\quad
\]
where $X\in \C^{(n-k)\times k}$, $Y\in \C^{(n-k)\times (n-k)}$, $Z\in \C^{k\times k}$, $W\in \C^{k\times (n-k)}$. We note then  that  $\begin{pmatrix} X  \\ Z \end{pmatrix}  \in HV_{n,k}$.
The block notation for $A\in\mathfrak{u}(n-k,k)$ is 
\[
A=\begin{pmatrix} \ep & \gamma \\  \beta  &  \alpha \end{pmatrix},\quad
\]
where $\ep\in \C^{(n-k)\times (n-k)}$, $\beta\in \C^{k\times (n-k)}$, $\gamma\in \C^{(n-k)\times k}$, $\alpha\in \C^{k\times k}$. Clearly $\ep$ and $\alpha$ are skew-Hermitian and 
\[
\beta^*=\gamma.
\]
Let $(U(t))_{t\ge0}$ be a $\mathbf{U}(n-k,k)$-valued stochastic process that satisfies the Stratonovich differential equation
\begin{equation}\label{eq-bm-sde}
\begin{cases}
dU(t)=U(t)\circ dA(t) ,\\
U_0=\begin{pmatrix}  Y_0 & X_0 \\ W_0 & Z_0 \end{pmatrix}.
\end{cases}
\end{equation}
\begin{definition}
The process $(U(t))_{t \ge 0}$ is a Brownian motion  on $\mathbf{U}(n-k,k)$ starting from $U_0$.
\end{definition}
%
%

 The main theorem of the section is the following:

\begin{theorem}\label{main:s1}
Let $U(t)=\begin{pmatrix}  Y(t) & X(t) \\ w (t) & Z(t) \end{pmatrix}, t\ge0$ be a Brownian motion on the Lie group $\mathbf{U}(n-k,k)$.  Then, the process $(w (t))_{t \ge 0}:=\left( X(t)  Z(t)^{-1}\right)_{t \ge 0} $ is a diffusion process  and its  generator is the diffusion operator $\frac12\Delta_{HG_{n,k}}$, which in affine coordinates is given by
\begin{align*}
\Delta_{HG_{n,k}}&={{4}}\sum_{1\le i, i'\le n-k, 1\le j, j'\le k}(I_{n-k}-ww^*)_{ii'} (I_k-w^*w)_{j'j}\frac{\partial^2}{\partial w_{ij}{\partial \overline{w}_{i'j'}}}.
\end{align*}
Moreover, $(w (t))_{t \ge 0}$ is a Brownian motion in $HG_{n,k}$, i.e. $\Delta_{HG_{n,k}}$ is the Laplace-Beltrami operator on $HG_{n,k}$ expressed in affine coordinates.
\end{theorem}
\begin{proof}
First, we note that both the maps 
\begin{equation*}
p_1: \begin{pmatrix}  Y & X \\ W & Z \end{pmatrix} \in \mathbf{U}(n-k,k) \to \begin{pmatrix} X  \\ Z \end{pmatrix}  \in HV_{n,k}
\end{equation*}
and $p: HV_{n,k} \rightarrow \mathbb{C}^{(n-k) \times k} \sim HG_{n,k}$ are Riemannian submersions with totally geodesic fibers. Thus, they transform Brownian motions into Brownian motions so that $w (t)=(p \circ p_1)(U(t))$ is indeed a Brownian motion on $HG_{n,k}$. It remains to compute explicitly its generator.

From \eqref{eq-bm-sde} we have
\begin{align}\label{eq-H-matrix-sde}
dX &=X \circ d\alpha +Y \circ d\gamma =X  d\alpha +Y d\gamma +\frac12 (dX  d\alpha +dY  d\gamma )\nonumber \\
dY &=X \circ d\beta +Y \circ d\ep =X  d\beta +Y  d\ep +\frac12 (dX  d\beta +dY  d\ep)\nonumber\\
dZ &= Z \circ d\alpha +W \circ d\gamma =Z  d\alpha +W d\gamma +\frac12 (dZ  d\alpha +dW  d\gamma )\\
dW &=Z \circ d\beta +W \circ d\ep =Z  d\beta +W  d\ep +\frac12 (dZ  d\beta +dW  d\ep ).\nonumber
\end{align}
Moreover, from \eqref{eq-A(t)} we have that
\[
d\alpha d\alpha=-2kI_kdt,\quad d\ep d\ep=-2(n-k)I_{n-k}dt
\]
and 
\[
d\beta d\gamma= d\beta d\beta^* = 2(n-k)I_kdt,\quad d\gamma d\beta=2k I_{n-k}dt.
\]
We then have
\begin{align*}
dX &=X \circ d\alpha +Y \circ d\gamma =X  d\alpha +Y  d\gamma +(n-2k) X dt \\
dZ &= Z \circ d\alpha +W \circ d\gamma =Z  d\alpha +W  d\gamma +(n-2k) Z dt.
\end{align*}

Now the process $w=XZ^{-1}$ satisfies:
\[
dw =dX  Z^{-1}+X dZ^{-1}+dXdZ^{-1}.
\]
But $Z dZ^{-1}=-dZ  Z^{-1}-dZdZ^{-1}$ whence 
we have
\begin{align}\label{eq-dw}
&dw = dX  Z ^{-1}-w dZ  Z^{-1}-w dZdZ^{-1}+dX dZ^{-1}\nonumber \\
&=(X d\alpha+Y d\gamma+(n-2k)X dt) Z^{-1}-w (Z d\alpha+W d\gamma+(n-2k)Zdt) Z^{-1}-w dZdZ^{-1}+dXdZ^{-1}\nonumber  \\
&=Y d\gamma Z^{-1}-w W d\gamma Z^{-1}-w dZdZ^{-1}+dX dZ^{-1}.
\end{align}
Note that for the finite variation part of $dw$ we have
\begin{align*}
&-w dZdZ^{-1}+dX dZ ^{-1}=w dZ Z^{-1}dZ Z^{-1}-dX\, Z^{-1}dZ\, Z^{-1}\\
&\qquad=w(Zd\alpha+Wd\gamma)Z^{-1}(Zd\alpha+Wd\gamma)Z^{-1}-(X d\alpha+Y d\gamma)Z^{-1}\, (Z d\alpha+W d\gamma)Z^{-1}\\
&\qquad=w Z d\alpha d\alpha Z^{-1}-X(d\alpha d\alpha)Z^{-1}=0.
\end{align*}
Hence we have
\[
dw=Y d\gamma Z^{-1}-w W d\gamma Z^{-1}.
\]
We can now prove that $w$ is a matrix diffusion process using the above formula. Since for $1\le i\le n-k$, $1\le j\le k$,
\[
dw_{ij}=\sum_{\ell=1}^{n-k}(Y-w W )_{i\ell} (d\gamma Z^{-1})_{\ell j},
\]
we have
\[
dw_{ij}d\overline{w}_{i'j'}=\sum_{\ell,m=1}^{n-k}(Y-w W )_{i\ell}(\overline{Y}-\overline{w W} )_{i'm} (d\gamma Z^{-1})_{\ell j}(d\overline{\gamma} \overline{Z}^{-1})_{m j'}.
\]
But
\[
(d\gamma Z^{-1})_{\ell j}(d\overline{\gamma} \overline{Z}^{-1})_{m j'}
=\sum_{p,q=1}^k (d\gamma)_{\ell p}(Z^{-1})_{p j}(d\overline{\gamma})_{mq} (\overline{Z}^{-1})_{qj'}=2\delta_{m\ell} dt\sum_{p=1}^k(Z^{-1})_{p j} (\overline{Z}^{-1})_{pj'},
\]
whence
\begin{align}\label{eq-dwdw-mid}
dw_{ij} d\overline{w}_{i'j'}
&=2 ((Y-w W )\overline{(Y-w W)}^T )_{ii'}((ZZ^*)^{-1})_{j'j}dt.
\end{align}
Now, if $M \in U(n-k, k)$ if and only if:
$$ \begin{pmatrix} Y^*   & W^* \\ X^* & Z^* \end{pmatrix} \begin{pmatrix} I_{n-k}   & 0 \\ 0 & -I_{k} \end{pmatrix}\begin{pmatrix}  Y & X \\ W & Z \end{pmatrix}=
\begin{pmatrix}  Y & X \\ W & Z \end{pmatrix} \begin{pmatrix} I_{n-k}   & 0 \\ 0 & -I_{k} \end{pmatrix}\begin{pmatrix} Y^*   & W^* \\ X^* & Z^* \end{pmatrix}=\begin{pmatrix} I_{n-k}   & 0 \\ 0 & -I_{k} \end{pmatrix},$$ 
Equivalently,

\begin{equation}\label{eq-relations1}
X^*X-Z^*Z=-I_{k}, \quad X^*Y-Z^*W=0,\quad Y^*Y-W^*W=I_{n-k}
\end{equation}
and
\begin{equation}\label{eq-relations}
YY^*-XX^*=I_{n-k}, \quad ZX^*-WY^*=0,\quad WW^*-ZZ^*=-I_{k}.
\end{equation}

From \eqref{eq-relations}, we get the following relations
\[
wWY^* = XZ^{-1}WY^*=XX^*, \quad I_k-w^*w = I_k - (Z^{-1})^*X^*XZ^{-1} = (Z^{-1})^*Z^{-1}, 
\]
which give after substitution in \eqref{eq-dwdw-mid}:
\begin{align}\label{eq-dwdw}
dw_{ij}d\overline{w}_{i'j'}&=2(I_{n-k}-ww^*)_{ii'}((ZZ^*)^{-1})_{j'j}dt\nonumber \\
&=2(I_{n-k}-ww^*)_{ii'}(I_k-w^*w)_{j'j}dt.
\end{align}
Therefore, we conclude that $(w (t))_{t\ge0}$ is a diffusion whose generator is given by $\frac12\Delta_{HG_{n,k}}$.
\end{proof}

\begin{remark}
When $k=1$, the diffusion operator $\Delta_{HG_{n,1}}$ coincides with the Laplacian of $\mathbb{C}H^n$ in affine coordinates whose formula was already obtained in Theorem \ref{Laplace CHn}:
\[
\Delta_{\mathbb{C}H^{n-1}}=4(1-|w|^2)\sum_{k=1}^{n-1} \frac{\partial^2}{\partial w_k \partial\overline{w}_k}- 4(1-|w|^2)\mathcal{R} \overline{\mathcal{R}}
\]
where
\[
\mathcal{R} :=\sum_{j=1}^{n-1} w_j \frac{\partial}{\partial w_j}.
\] 
\end{remark}

\subsection{Invariant measure}
Since $HG_{n,k}$ is a non-compact symmetric space, its Riemannian volume measure has infinite mass. Our goal in this section will be to compute explicitly its density measure in affine coordinates up to a scalar multiple. We will take advantage of the explicit formula for $\Delta_{HG_{n,k}}$ obtained in Theorem \ref{main:s1}. Consider in affine coordinates the following measure
\[
d\mu:=\det(I_k-w^*w)^{-n} dm := \rho(w) dm
\]
where $m$ denotes the Lebesgue measure on $HG_{n,k  }\simeq \C^{(n-k)\times k}$ and we omit the dependence on $n$ for sake of simplicity. First note that $\det(I_k-w^*w)^{-n}$ is a well-defined density function. This is because from \eqref{eq-relations1} we know that $\det(I_k-w^*w)^{-n}=\det(ZZ^*)^{n}$ and that $\det(ZZ^*)=\det(Z^*Z)=\det(X^*X+I_k)>0$.

In the proposition below we show that $\mu$ is a symmetric and invariant measure for $\Delta_{HG_{n,k}}$.
\begin{proposition}\label{invariant1}
The measure $\mu$ is invariant and symmetric for the operator $(1/2)\Delta_{HG_{n,k}}$. Namely for every smooth and compactly supported functions $f,g$ on $\C^{(n-k)\times k}$, we have: 
\[
\int (\Delta_{HG_{n,k}}f) g\, d\mu = \int f(\Delta_{HG_{n,k}}g)\, d\mu=-\int \Gamma_{\Delta_{HG_{n,k}}}(f, g)\, d\mu,
\]
where the carr\'e du champ operator 
\[
\Gamma_{\Delta_{HG_{n,k}}}(f,g):=\frac{1}{2} \left( \Delta_{HG_{n,k}}(fg) - (\Delta_{HG_{n,k}}f) g - (\Delta_{HG_{n,k}}g) f \right)
\]
 is explicitly given by
\begin{equation}\label{eq-Gamma}
\Gamma_{\Delta_{HG_{n,k}}}(f,g)=2\sum_{1\le i, i'\le n-k, 1\le j, j'\le k}(I_{n-k}-ww^*)_{ii'}(I_k-w^*w)_{j'j}\bigg(\frac{\partial f}{\partial w_{ij}}\frac{\partial g}{\partial \bw_{i'j'}}+\frac{\partial g}{\partial w_{ij}}\frac{\partial f}{\partial \bw_{i'j'}}\bigg).
\end{equation}
\end{proposition}
\begin{proof}
For ease of notations, we  set:  
\begin{equation*}
\partial_{ij} :=\frac{\partial }{\partial w_{ij}}, \quad  \bpartial_{i j} :=\frac{\partial }{\partial \overline{w}_{ij}}, \quad A_{ii'jj'} := (I_{n-k}-ww^*)_{ii'}(I_k-w^*w)_{j'j}
\end{equation*}
so that 
\[
\Gamma_{\Delta_{HG_{n,k}}}(f,g) = 2\sum_{1\le i, i'\le n-k, 1\le j, j'\le k}A_{ii'jj'}\bigg((\partial_{ij}f)(\bpartial_{i'j'}g)+(\partial_{ij}g)(\bpartial_{i'j'}f)\bigg).
\]
By integration by parts, we have
\begin{align*}
&-\frac{1}{2}\int (\Delta_{HG_{n,k}}f) g\, d\mu \\
=& \sum_{1\le i, i'\le n-k, 1\le j, j'\le k} \int (\partial_{ij}f)\, \bpartial_{i'j'} (A_{ii'jj'}g\rho)dm+\int (\bpartial_{i'j'}f)\, \partial_{ij} (A_{ii'jj'}g\rho)dm\\
=&\frac{1}{2}\int \Gamma_{\Delta_{HG_{n,k}}}(f, g)\, d\mu + R 
\end{align*}
where 
\begin{multline*}
R := \sum_{1\le i, i'\le n-k, 1\le j, j'\le k} \bigg(\int [(\partial_{ij}f)\,(\bpartial_{i'j'} A_{ii'jj'})+(\bpartial_{i'j'}f)\,(\partial_{ij} A_{ii'jj'})]g\rho dm \\
+\int [(\partial_{ij}f)\,(\bpartial_{i'j'}\rho)+(\bpartial_{i'j'}f)\,(\partial_{ij}\rho)]gA_{ii'jj'} dm\bigg).
\end{multline*}
Hence, it remains to prove that $R=0$. To this end, we use the relations: 
\begin{align*}
\bpartial_{i'j'} A_{ii'jj'}=-w_{ij'}(\delta_{j'j}-(w^*w)_{j'j})-(\delta_{ii'}-(ww^*)_{ii'})w_{i'j},
\end{align*}
and
\[
\partial_{ij} A_{ii'jj'}=-\bw_{i'j}(\delta_{j'j}-(w^*w)_{j'j})-(\delta_{ii'}-(ww^*)_{ii'})\bw_{ij'}
\]
to get
\[
\sum_{1\le i'\le n-k, 1\le j'\le k} \bpartial_{i'j'} A_{ii'jj'}=-n({w}(I_k-J))_{ij}
\]
and
\[
\sum_{1\le i\le n-k, 1\le j\le k} \partial_{ij} A_{ii'jj'}=-n({\bw}(I_k-\bar{J}))_{i'j'}.
\]
Moreover, since
\begin{align*}
\bpartial_{i'j'}\det(I_k-J)&=\det(I_k-J)\sum_{1\le p,q\le k}\big((I_k-J)^{-1}\big)_{qp}\bpartial_{i'j'} (I_k-J)_{pq}\\
&=-\det(I_k-J)\bigg(w(I_k-J)^{-1} \bigg)_{i'j'}
\end{align*}
and
\[
\partial_{ij}\det(I_k-J)=-\det(I_k-J)\bigg(\bw(I_k-\bar{J})^{-1} \bigg)_{ij},
\]
then
\begin{align*}
\bpartial_{i'j'} \rho
&=n  \rho\bigg(w(I_k-J)^{-1}\bigg)_{i'j'}, \quad
\partial_{ij} \rho=n  \rho\bigg(\bw(I_k-\bar{J})^{-1} \bigg)_{ij}.
\end{align*}
As a result,
\begin{align*}
\sum_{1\le i'\le n-k, 1\le j'\le k}(\bpartial_{i'j'} A_{ii'jj'})g\rho + (\bpartial_{i'j'}\rho)gA_{ii'jj'}=0
\end{align*}
and
\[
\sum_{1\le i\le n-k, 1\le j\le k}(\partial_{ij} A_{ii'jj'})g\rho + (\partial_{ij}\rho)gA_{ii'jj'}=0,
\]
that is $R=0$ as claimed.
\end{proof}

\subsection{Eigenvalue process}\label{sec-J process}

Let $(w (t))_{t \ge 0}=(X(t) Z(t)^{-1})_{t \ge 0}$ be a Brownian motion on $HG_{n,k}$ as in Theorem \ref{main:s1}. Let  $J(t)= w (t)^*w (t) \in \mathbb{C}^{k \times k}$ for $t \ge 0$.  We wish to study the eigenvalues process of $J$. The first goal is therefore to write a stochastic differential equation for $J$.

\begin{proposition}
Let $(J(t))_{t \ge 0}$ be given as above. Let $t_0:=\inf\{ u\ge0, J_u \text{ is not invertible}\}$. Then for $t\in(0,t_0)$ there exists a Brownian motion $R$ in $\mathbb{C}^{ k \times k}$ such that:
\begin{align}\label{SDE-Matrix}
dJ(t)=& \sqrt{I_k-J(t)} dR\sqrt{I_{k} - J(t)} \sqrt{J(t)} + \sqrt{J(t)} \sqrt{I_{k} - J(t)}dr(t)^{*}\sqrt{I_k-J(t)} \\
&+ 2((n-k) - \tr(J(t)))(I_k - J(t)) dt. \notag
\end{align}
\end{proposition}
\begin{proof}
Since $J=w^*w$ we have
\begin{equation*}
dJ = (dw^*)w + w^*(dw) + (dw^*)(dw). 
\end{equation*}
From \eqref{eq-dwdw}, we readily derive: 
\begin{equation*}
(dw^*)(dw) = 2((n-k) - \tr(J))(I_k - J). 
\end{equation*}
As to the local martingale part of $dJ$, recall that that $I_{k} - w^*w = (Z^{-1})^*Z^{-1}$ is invertible and assume that $J$ is invertible up to time $t_0 > 0$. Then $w$ has a unique polar decomposition $w = Q\sqrt{J}$ where $Q \in \mathbb{C}^{(n-k)\times k}$ satisfies: 
\begin{equation*}
QQ^* = I_{n-k}, \quad Q^*Q = I_k.
\end{equation*}
In particular, $I_{n-k} - ww^*$ is invertible as well since 
\begin{equation*}
I_{n-k} - ww^* = 0 \quad \Leftrightarrow \quad Q(I_k-J)Q^* = 0. 
\end{equation*}
As a matter fact, \eqref{eq-dwdw} implies the existence of a complex Brownian matrix $B \in \mathbb{C}^{(n-k )\times k}$ such that 
\begin{equation*}
dw = \sqrt{I_{n-k} - ww^*} dB \sqrt{I_k-w^*w},
\end{equation*}
whence
\begin{equation}\label{eq-VdV}
(dw^*)w + w^*(dw) = \sqrt{I_k-w^*w} dB^* \sqrt{I_{n-k} - ww^*}w  + w^*\sqrt{I_{n-k} - ww^*} dB \sqrt{I_k-w^*w}. 
\end{equation}
Furthermore,
\begin{align*}
w^*\sqrt{I_{n-k} - ww^*} dB \sqrt{I_k-w^*w} & = \sqrt{J} Q^*\sqrt{I_{n-k} - QJQ^*}dB \sqrt{I_k-w^*w}
\\& = \sqrt{J} \sqrt{I_{k} - J}Q^*dB\sqrt{I_k-J} , 
\end{align*}
and similarly 
\begin{align*}
\sqrt{I_k-w^*w} dB^* \sqrt{I_{n-k} - ww^*}w =  \sqrt{I_k-J} dB^* Q\sqrt{I_{k} - J} \sqrt{J}. 
\end{align*}
Finally, the identity $QQ^* = I_{n-k}$ shows that the matrix-valued process $dR := dB^*Q$ is a Brownian motion process in $\C^{k\times k}$. Altogether, we end up with the following autonomous SDE:
\begin{equation}\label{SDE-Matrix2}
dJ= \sqrt{I_k-J(t)} dR\sqrt{I_{k} - J} \sqrt{J} + \sqrt{J} \sqrt{I_{k} - J}dR^{*}\sqrt{I_k-J} + 2((n-k) - \tr(J))(I_k - J) dt,
\end{equation}
as desired. 
\end{proof}
Using Theorem 4 in \cite{Gra-Mal}, we immediately obtain the SDE satisfied by the eigenvalues process $(\lambda_j)_{j=1}^k$ of $J$.
\begin{corollary}
Assume $\lambda_1(0) > \dots > \lambda_k(0) > 0$ and let 
\begin{equation*}
\tau := \inf\{t > 0, \lambda_l(t) = \lambda_j(t)\, \textrm{for some} \, (l,j) \},
\end{equation*}
be the first collision time. Then, for any $1 \leq j \leq k$:
\begin{multline}\label{eq-SDE-lambda}
d\lambda_j = 2\sqrt{\lambda_j}(1-\lambda_j) dN_j + 2\left[(n-k) - \sum_{l=1}^k\lambda_l\right]  {{(1-\lambda_j)}} dt \\ 
+ 2\sum_{l\neq j} \frac{(1-\lambda_l)\lambda_j(1-\lambda_j) + (1-\lambda_j)\lambda_l(1-\lambda_l)}{\lambda_j - \lambda_l} dt
\end{multline}
where $(N_j)_{j=1}^k$ is a $k$-dimensional Brownian motion. 
\end{corollary}
%
%
%

Using the rational transformation,

\[
\rho_j:=\frac{1+\lambda_j}{1-\lambda_j} >1 \quad \Leftrightarrow \quad \lambda_j = \frac{\rho_j-1}{\rho_j+1},
\]
we can prove that: 
\begin{proposition}\label{Collision}
The first collision time $\tau$ is infinite almost surely. 
\end{proposition}
\begin{proof}
A straightforward application of It\^o's formula yields: 
\begin{align*}
d\rho_j &= 4\frac{\sqrt{\lambda_j}}{(1-\lambda_j)} dN_j + 4\left[(n-1) - \lambda_j\right] \frac{dt}{1-\lambda_j}  + 8 \sum_{l\neq j} \frac{\lambda_l}{\lambda_j - \lambda_l} dt + 8 \frac{\lambda_j}{1-\lambda_j} dt
\\& = 2\sqrt{\rho_j^2-1}dN_j + 2\left[n\rho_j + (n-2) \right] dt + 4(\rho_j+1)\sum_{l\neq j} \frac{(\rho_l-1)}{(\rho_j-\rho_l)} dt
\\& = 2\sqrt{\rho_j^2-1}dN_j + 2\left[(n+2-2k)\rho_j + (n-2k) \right] dt  +4(\rho_j^2-1)\sum_{l\neq j} \frac{1}{(\rho_j-\rho_l)} dt.
\end{align*}
It follows that the generator of the diffusion process $(\rho_j)_{j=1}^k$ acts on smooth functions as: 
\begin{equation*}
L^{(k,n)}_{\rho_1, \dots, \rho_k} = \sum_{j=1}^k\left\{2(\rho_j^2-1)\partial_{j}^2 + 2\left[(n+2-2k)\rho_j + n-2k \right]\partial_j\right\} + 4 \sum_{l,j = 1, l \neq j}^k \frac{(\rho_j^2-1)}{(\rho_j-\rho_l)}\partial_j.
\end{equation*}
Up to a constant, this operator is the Vandermonde transform of $k$ independent diffusions with generator 
\begin{equation*}
\mathscr{L}^{(k,n)}_u := 2(u^2-1)\partial_{u}^2 + 2\left[(n+2-2k)u + n-2k \right]\partial_u.
\end{equation*}
Indeed, the Vandermonde function 
\begin{equation*}
V(\rho_1, \dots, \rho_k) = \prod_{1 \leq l < j \leq k} (\rho_l - \rho_j),  
\end{equation*} 
is positive on the Weyl chamber $\{\rho_1 > \rho_2 > \dots > \rho_k\}$ and satisfies (see Appendix in \cite{Dou}): 
\begin{equation*}
\sum_{j=1}^k\mathscr{L}^{(k,n)}_{\rho_j}V = \frac{k(k-1)(3n+2-4k)}{3}V.
\end{equation*}
Besides, one readily checks that 
\begin{equation}\label{DoobV}
L^{(k,n)}_{\rho_1, \dots, \rho_k}  = \frac{1}{V}\sum_{j=1}^k\mathscr{L}^{(k,n)}_{\rho_j}(V\cdot) - \frac{k(k-1)(3n+2-4k)}{3}.
\end{equation}
In particular, the process 
\begin{equation*}
t \mapsto \frac{1}{V(\rho_1(t), \dots, \rho_k(t))} + \frac{k(k-1)(3n+2-4k)}{3}, \quad t < \tau,  
\end{equation*}
starting at non-colliding particles is a continuous local martingale which blows-up as $t \rightarrow \tau$. Since it is a time-changed Brownian motion, the result follows (this is the classical McKean's argument). 
\end{proof}

We can further relate $(\rho_j)_{j=1}^k$ to a special instance of the so-called radial Heckman-Opdam process associated with the root system of type $BC_k$ (\cite{Sch}). More precisely, the process $(\zeta_j(t), t \geq 0)_{j=1}^k$ defined by: 
\begin{equation*}
\zeta_j(t) = \cosh^{-1}(\rho_j(t/4)), \quad 1 \leq j \leq k,
 \end{equation*}
satisfies
\begin{equation}\label{eq-SDE-coth}
d\zeta_j = dN_j + \frac{1}{2}\coth(\zeta_j)dt + \frac{(n-2k)}{2} \coth\left(\frac{\zeta_j}{2}\right)dt + \frac{1}{2}\sum_{i \neq j}\left[\coth\left(\frac{\zeta_i -\zeta_j}{2}\right) + \coth\left(\frac{\zeta_i + \zeta_j}{2}\right)\right]dt. 
\end{equation}
As a matter of fact, $(\rho_j)_{j=1}^k$ is the unique strong solution of the SDE it satisfies for any starting point $\rho(0) \in [1, \infty)^k$ (\cite{Sch}, Proposition 4.1). By the virtue  of Proposition \ref{Collision}, we deduce that 
\begin{corollary}
The SDE \eqref{eq-SDE-lambda} admits a unique strong solution for all $t \geq 0$ and any starting point $\lambda_1(0) \geq \lambda_2(0) \geq \dots \geq \lambda_k(0) \geq 0$.
\end{corollary}

\section{Long-time behavior and distribution of the functional $\int_0^t \mathop{\tr} (J(s))ds$} 
In this section, we derive the limit as $t \rightarrow \infty$ of the functional: 
\begin{equation*}
\int_0^t \mathop{\tr} ({J(s)})ds,
\end{equation*}
and compute its Laplace transform. Indeed, the asymptotics of this functional may be readily derived from \eqref{eq-SDE-coth}. 

\begin{proposition}\label{Lemma1}
As $t\to+\infty$, almost surely we have
\begin{equation}\label{eq-tr-J}
\lim_{t\to+\infty}\frac1t\int_0^t\tr({J(s)})ds=k.
\end{equation}
\end{proposition}

\begin{proof}
Recall the SDE \eqref{eq-SDE-coth}:
\[
d\zeta_j = dN_j + \left(\coth(\zeta_j) + \frac{(n-2k)}{2} \coth\left(\frac{\zeta_j}{2}\right)\right)dt + \frac{1}{2}\sum_{i\not=j}\left[\coth\left(\frac{\zeta_i -\zeta_j}{2}\right) + \coth\left(\frac{\zeta_i + \zeta_j}{2}\right)\right]dt. 
\]
Under the standing assumption of $n-k\ge k$ and using a comparison argument (see e.g. Proposition 4.2 in \cite{Sch}), we infer that $\zeta_j\to +\infty$ almost surely as $t\to+\infty$. Therefore 
\[
\rho_j(t)=\cosh(\zeta_j(4t))\to +\infty,\quad j=1,\dots, k,
\]
and in turn it holds almost surely that 
\[
\lambda_j(t)=\frac{\rho_j-1}{\rho_j+1}\to 1,\quad j=1,\dots, k.
\]
The conclusion follows.
\end{proof}

Now, we shall get more insight into the distribution of the above functional and give an expression of its Laplace transform relying on Girsanov Theorem and Karlin-McGregor formula. Moreover, as in \cite{Dem-Windings}, we shall further point out an interesting connection with the generalized Maass Laplacian in the complex hyperbolic space (\cite{Aya-Int}). 

Let $(w (t))_{t \ge 0}=(X(t) Z(t)^{-1})_{t \ge 0}$ be a Brownian motion on $HG_{n,k}$ as in Theorem \ref{main:s1} and recall $J=w^*w$ as well its the eigenvalues process $(\lambda_j)_{j=1}^k$ of $J$. Recall also
\begin{equation*}
\rho_j:=\frac{1+\lambda_j}{1-\lambda_j}.
\end{equation*}
We assume that the $\lambda_j(0)$'s and therefore the $\rho_j(0)$'s are pairwise distinct. This is not a loss of generality and our results extend by continuity to non necessarily pairwise distinct eigenvalues.

\subsection{An auxiliary lemma}
Let 
$$\Delta_k=\{ \rho \in \mathbb{R}^k, \rho_1 > \rho_2 > \dots > \rho_k>1\}$$ and for any $\alpha \geq 0$, introduce the following diffusion operator: 
\begin{equation}\label{Gen1}
\mathscr{L}_u^{(n,k,\alpha)} := 2(u^2-1)\partial_{u}^2 + 2\left[(n+2-2k+2\alpha)u + n-2k -2\alpha \right]\partial_u, \quad u \geq 1. 
\end{equation}
 Performing the variable change $u = \cosh(2r), r \geq 0$, $\mathscr{L}_u^{(n,k,\alpha)}$ is mapped into the following hyperbolic Jacobi operator (\cite{Koor}): 
 \begin{equation*}
\mathscr{H}_r^{(n,k, \alpha)} =  \frac{1}{2}\left\{\partial_r^2 + \left[(2(n-2k)+1)\coth\left(r\right) + (4\alpha + 1)\tanh\left(r\right)\right]\partial_r\right\}. 
 \end{equation*}
In particular, $\mathscr{L}_u^{(n,k,0)} = \mathscr{L}_u^{(n,k)}$ is mapped into the radial part of the Laplace-Beltrami operator on the complex hyperbolic space of $H_{n-2k+1} \sim \mathbf{SU}(n-2k+1,1)/\mathbf{SU}(n-2k+1)$. 

Let $q_t^{(n,k,\alpha)}$ denote the heat kernel (with respect to Lebesgue measure) of $\mathscr{L}_u^{(n,k,\alpha)}$ with Neumann boundary condition at $u=1$. This kernel does not admits in general a simple expression as can be seen from the Jacobi-Fourier inversion formula (Theorem 2.3. in \cite{Koor}). Nonetheless, this formula simplifies when $\alpha = 0$ (see also \cite{Ner}, \cite{Sch}) and yields: 
\begin{multline*}
q_t^{(n,k,0)}(u_1, u_2) = \frac{(u_2-1)^{n-2k}}{\pi 2^{n-2k},}\int_0^{\infty}e^{-2t(\mu^2+\kappa^2)}
F_{-\mu}(-v_1)F_{\mu}(-v_2) \left|\frac{\Gamma(\kappa+i\mu)\Gamma(\kappa-2\alpha + i\mu)}{\Gamma(2i\mu)}\right|^2 d\mu  
\end{multline*}
 where $v_i = (u_i-1)/2, i \in \{1,2\}$, $\kappa:= (n-2k+1)/2$ and 
\begin{equation*}
F_{\mu}(-v_i) = {}_2F_1\left(\kappa+ i\mu, \kappa-i\mu, n-2k+1; \frac{1-u_i}{2}\right), \quad i \in \{1,2\}, \quad \mu \in \mathbb{R},
\end{equation*}
are Jacobi functions. In this case, we have: 
\begin{lemma}\label{Estimate}
For any $u_1 \geq 1, t > 0, \alpha \geq 0$,  
 \begin{equation*}
\int_1^{\infty} (u_2)^{\alpha} q_t^{(n,k,0)}(u_1, u_2) du_2 < \infty. 
\end{equation*}
Consequently, 
\begin{equation*}
\mathbb{E}\left[ \sup_{0 \le s \le t} \det(I_k-J(s))^{-\alpha}\right] < \infty.
\end{equation*}
\end{lemma}

\begin{proof}
As to the first assertion, we use the comparison principle for stochastic differential equations together with the obvious inequality $\tanh(r) < \coth(r)$ to see that the diffusion associated to $\mathscr{H}_r^{(n,k,0)}$ is dominated by the (unique strong) solution 
$(H(t))_{t \geq 0}$ of the SDE: 
\begin{equation*}
H(t) = H(0) + \gamma(t) + (n-2k+1)\int_0^t \coth(H(s)) ds,
\end{equation*}
where $\gamma$ is a real Brownian motion. The diffusion $(\cosh(H(t)))_{t \geq 0}$ is studied in detail for instance in \cite{Jak-Wis} from which it is seen that it has moments of all orders. Keeping in mind the correspondence between $\mathscr{L}_u^{(n,k,\alpha)}$ and 
$\mathscr{H}_r^{(n,k, \alpha)}$, the first part of the lemma follows.

For the second part of the lemma, without loss of generality we assume $\rho(0) \in \Delta_k$. We then first note that from the non-collision property of $\rho$.
\begin{align*}
\det(I_k-J(t))^{-\alpha} & =2^{-\alpha k} \ \prod_{j=1}^k(1+\rho_j(t))^{\alpha} \\
 & \le 2^{-\alpha k} \prod_{j=1}^k(1+\rho_1(t))^{\alpha}  \\
 & \le 2^{-\alpha k} (1+\rho_1(t))^{k \alpha}
\end{align*}
It is therefore enough to prove that for every $\alpha \ge 0$ and $t \ge 0$ we have
$
\mathbb{E} \left( \sup_{0\le s \le t} \rho_1(s)^\alpha \right) < \infty
$
The SDE satisfied by $\rho_1$ has a non-negative drift and therefore the process $\rho_1(t)$ is a sub-martingale. From Doob's maximal inequality, it is therefore enough to prove that  for every $\alpha \ge 0$ and $t \ge 0$ we have
$
\mathbb{E} \left( \rho_1(t)^\alpha \right) < \infty.
$

From  \eqref{DoobV} and Karlin-McGregor formula the semigroup density of $(\rho_j(t))_{j=1}^k$ may be written as: 
\begin{equation*}
e^{-k(k-1)(3n+2-4k)t/3} \frac{V(\rho)}{V(\rho(0))} \det \left(q_t^{(n,k,0)}(\rho_j(0), \rho_a)\right)_{j,a = 1}^k. 
\end{equation*}
The conclusion follows therefore from the first part of the lemma.
 \end{proof}

\subsection{Laplace transform of $\int_0^t \mathop{\tr} ({J(s)})ds$}
Now, we are ready to prove the following theorem:
\begin{theorem}\label{theo-Laplace}
Assume $\rho(0) \in \Delta_k$. For every $\alpha >0$ and $t >0$,
\begin{align*}
 & \mathbb{E}\left\{\exp-\left(2\alpha^2 \int_0^t \mathop{\tr}(J(s)) ds \right)\right\} \\
 = & \frac{e^{[6\alpha (n-2k+1)-(k-1)(3n+2-4k)]kt/3}}{V(\rho(0))} \left[ \prod_{j=1}^k(1+\rho_j(0))^{\alpha} \right]
\det\left(\int_{1}^{+\infty} \frac{u^{a-1}}{(1+u)^{\alpha}}q_t^{(n,k,\alpha)}(\rho_j(0), u) du\right)_{a,j=1}^k. 
\end{align*}
\end{theorem}

\begin{proof}
The main ingredient is a matrix Girsanov transform.  Recall first the SDE \eqref{eq-SDE-lambda} for $(\lambda_j)_{j=1}^k$. It can be simplified to
\begin{equation*}
d\lambda_j =  2\sqrt{\lambda_j}(1-\lambda_j) dN_j + 2\left[(n-1) - \lambda_j\right] (1-\lambda_j) dt  + 4(1-\lambda_j)^2 \sum_{l\neq j} \frac{\lambda_l}{\lambda_j - \lambda_l} dt, 
\end{equation*}
for $1 \leq j \leq k$. Then 
\begin{equation*}
-d\log(1-\lambda_j) =  2\sqrt{\lambda_j} dN_j + 2(n-1) dt  + 4(1-\lambda_j) \sum_{l\neq j} \frac{\lambda_l}{\lambda_j - \lambda_l} dt, 
\end{equation*}
and in turn 
\begin{align*}
- d \log\det(I_k-J) & = 2\sum_{j=1}^k\sqrt{\lambda_j} dN_j  + 2(n-1)kdt + 4\sum_{j=1}^k \sum_{l\neq j} (1-\lambda_j)  \frac{\lambda_l}{\lambda_j - \lambda_l} dt 
\\& = 2\sum_{j=1}^k\sqrt{\lambda_j} dN_j  + 2(n-1)kdt + 4 \sum_{j=1}^k \sum_{l\neq j}   \frac{\lambda_l}{\lambda_j - \lambda_l} dt 
\\& = 2\sum_{j=1}^k\sqrt{\lambda_j} dN_j  + 2k(n-k) dt. 
\end{align*}
Consequently, for any $\alpha > 0$, the exponential local martingale 
\begin{equation*}
M_t^{(\alpha)}:= \exp\left(2\alpha\int_0^t \sum_{j=1}^k\sqrt{\lambda_j}(s) dN_j(s) - 2\alpha^2 \int_0^t \mathop{\tr}(J(s)) ds \right), \ t\ge0
\end{equation*}
may be written as 
\begin{equation*}
M_t^{(\alpha)} = e^{-2\alpha k(n-k)t} \left[\frac{\det(I_k-J(0))}{\det(I_k-J(t))}\right]^{\alpha} \exp-\left(2\alpha^2 \int_0^t \mathop{\tr}(J(s)) ds \right), \ t\ge0.
\end{equation*}
$(M_t^{(\alpha)})_{t \geq 0}$ is also a martingale since 
\begin{equation*}
M_t^{(\alpha)} \leq \left[\det(I_k-J(t))\right]^{-\alpha}
\end{equation*}
and by the virtue of Lemma \ref{Estimate}. 
We can therefore define for any fixed time $t > 0$,  a new probability measure 
\begin{equation*}
{\mathbb{P}}^{(\alpha)}_{|\mathscr{F}_t} = M_t^{(\alpha)} {\mathbb{P}}_{|\mathscr{F}_t}
\end{equation*}
and denote $\mathbb{E}^{(\alpha)}$ the corresponding expectation. Then under $\mathbb{P}^{(\alpha)}$
\begin{equation*}
\tilde{N}_j(t) := N_j(t) - 2\alpha\int_0^t\sqrt{\lambda_j(s)}ds, \quad 1 \leq j \leq k,
\end{equation*}
defines a $k$-dimensional Brownian motion so that 
\begin{equation*}
d\lambda_j =  2\sqrt{\lambda_j}(1-\lambda_j) d\tilde{N}_j + 2\left[(n-1) - (1-2\alpha)\lambda_j\right] (1-\lambda_j) dt  + 4(1-\lambda_j)^2 \sum_{l\neq j} \frac{\lambda_l}{\lambda_j - \lambda_l} dt. 
\end{equation*}
Recalling 
\[
\rho_j=\frac{1+\lambda_j}{1-\lambda_j},
\]
then It\^o's formula yields: 
\begin{align*}
d\rho_j &= 
 2\sqrt{\rho_j^2-1}d\tilde{N}_j + 2\left[(n+2-2k+2\alpha)\rho_j + n-2\alpha-2k \right] dt  +4(\rho_j^2-1)\sum_{l\neq j} \frac{1}{(\rho_j-\rho_l)} dt.
\end{align*}
Up to the multiplication operator by the constant 
\begin{equation*}
\frac{k(k-1)(3n+2-4k+6\alpha)}{3}, 
\end{equation*}
this process is again a Vandermonde transform of $k$ independent copies of the diffusion whose generator is given by: 
\begin{equation}\label{Gen2}
2(u^2-1)\partial_{u}^2 + 2\left[(n+2-2k+2\alpha)u + n-2k -2\alpha \right]\partial_u, \quad u \geq 1,
\end{equation}
with Neumann boundary condition at $u=1$. As a result, Karlin-McGregor's formula entails (see e.g. \cite{AOW} and references therein):  
\begin{align*}
\mathbb{E}\left\{\exp-\left(2\alpha^2 \int_0^t \mathop{\tr}(J(s)) ds \right)\right\} & = e^{2\alpha k(n-k)t} \mathbb{E}^{(\alpha)}\left[\frac{\det(I_k-J(t))}{\det(I_k-J(0))}\right]^{\alpha} 
\\& = e^{2\alpha k(n-k)t}\prod_{j=1}^k(1+\rho_j(0))^{\alpha}  \mathbb{E}^{(\alpha)}\left[\prod_{j=1}^k\frac{1}{(1+\rho_j(t))^{\alpha}}\right]
\\& = e^{[6\alpha (n-2k+1)-(k-1)(3n+2-4k)]kt/3}\prod_{j=1}^k(1+\rho_j(0))^{\alpha} 
\\& \int_{\Delta_k} \prod_{j=1}^k\frac{d\rho_j}{(1+\rho_j)^{\alpha}} \frac{V(\rho)}{V(\rho(0))} \det(q_t^{(n,k,\alpha)}(\rho_a(0), \rho_j))_{a,j=1}^k
\end{align*}
where $q_t^{(\alpha,k)}(u, v)$ is the heat semi-group of the infinitesimal generator \eqref{Gen1}. Equivalently, the Andr\'eief identity entails (\cite{Dei-Gio}, p. 37) 
\begin{multline*}
\mathbb{E}\left\{\exp-\left(2\alpha^2 \int_0^t \mathop{\tr}(J(s)) ds \right)\right\} = \frac{e^{[6\alpha (n-2k+1)-(k-1)(3n+2-4k)]kt/3}}{V(\rho(0))}\prod_{j=1}^k(1+\rho_j(0))^{\alpha} \\ 
\det\left(\int \frac{u^{a-1}}{(1+u)^{\alpha}}q_t^{(\alpha,k)}(\rho_j(0), u) du\right)_{a,j=1}^k. 
\end{multline*}
\end{proof}

\subsection{Connection to the Maass Laplacian in the complex hyperbolic space}
In this paragraph, we present the connection of $q_t^{(n,k,\alpha)}$ to heat semi-group of the Maass Laplacian on the complex hyperbolic space $\mathbb{C} H^{n-2k+1}$ realized in the unit ball (\cite{Aya-Int}). For $k=1$, such connection was already pointed out in \cite{Dem-Windings} and was the key ingredient to derive the density of the corresponding stochastic area. For higher ranks, the computations become tedious. Nonetheless, as in \cite{Dem-Windings}, the new expression we obtain below makes transparent the limiting behavior proved in \ref{Lemma1} through the exponential factor $e^{-2\alpha^2 t}$ and is somehow more explicit than the one displayed in Theorem \ref{theo-Laplace} since the heat semi-group of the Maass Laplacian admits a more compact form than $q_t^{(n,k,\alpha)}$ (\cite{Aya-Int}, Theorem 2.2).

Firstly, we perform the variable change $u = \cosh(2r)$ and use the identity $2\cosh^2(r) = 1+\cosh(2r)$ to get: 
\begin{multline}\label{eq-int-trace}
\mathbb{E}\left\{\exp-\left(2\alpha^2 \int_0^t \mathop{\tr}(J(s)) ds \right)\right\} = \frac{e^{[6\alpha (n-2k+1)-(k-1)(3n+2-4k)]kt/3}}{V(\rho(0))} \prod_{j=1}^k(1+\rho_j(0))^{\alpha} \\ 
 \det\left(\frac{[\cosh(2r)]^{a-1}}{2^{\alpha}[\cosh(r)]^{2\alpha}}q_t^{(n,k,\alpha)}(\rho_j(0), \cosh(2r)) d(\cosh(2r))\right)_{a,j=1}^k. 
\end{multline}
On the other hand we infer that the hyperbolic Jacobi operator $\mathscr{H}_r^{(n,k)}$ is intertwined via the map: 
\begin{equation*}
f \mapsto \frac{1}{\cosh^{2\alpha}(r)}f
\end{equation*}
with the radial part of the shifted Maass Laplacian $\mathcal{L}$ in the complex hyperbolic space $\mathbb{C} H^{n-2k+1}$. More precisely, it holds that
\[
\mathscr{H}_r^{(n,k)}\left(r \mapsto \frac{1}{\cosh^{2\alpha}r}f(r) \right) + \frac{\left(2\alpha+{n-2k+1}\right)^2}{2\cosh^{2\alpha}r}f(r) =\frac{1}{\cosh^{2\alpha}(r)}\mathcal{L}(f)(r)
\]
where
\begin{equation}\label{Radial}
\mathcal{L}=\frac{1}{2}\partial_{r}^2 + \left[\left(n-2k+\frac{1}{2}\right)\coth(r) + \frac{1}{2}\tanh(r) \right]\partial_{r} + \frac{2\alpha^2}{\cosh^2(r)} + \frac{(n-2k+1)^2}{2}.
 \end{equation}
Consequently, the following identity holds: 
\begin{multline}\label{eq-kernel-identity}
\frac{e^{2\alpha (n-2k+1)t}}{[\cosh(r)]^{2\alpha}} q_t^{(n,k,\alpha)}(\rho_j(0), \cosh(2r)) d(\cosh(2r)) = \frac{2^{\alpha} e^{-2\alpha^2t}}{(1+\rho_j(0))^{\alpha}} \\ 
V(t)^{(n-2k+1, \alpha)}\left(\frac{1}{2}\cosh^{-1}(\rho_j(0)),r\right)dr,
\end{multline}
where $V(t)^{(n-2k+1, \alpha)}$ is the heat kernel of $\mathcal{L} - (n-2k+1)^2/2$ in \eqref{Radial} with respect to the radial volume element : 
\begin{equation*}
\frac{2\pi^{n-2k+1}}{\Gamma(n-2k+1)}(\sinh(r))^{2(n-2k+1)-1}\cosh(r).
\end{equation*}
By plugging \eqref{eq-kernel-identity} into \eqref{eq-int-trace}, we arrive at: 
\begin{proposition}\label{thm-NewLaplace}
The Laplace transform derived in Theorem \ref{theo-Laplace} can be rewritten as: 
\begin{multline}\label{eq-NewLaplace}
\mathbb{E}\left\{\exp-\left(2\alpha^2 \int_0^t \mathop{\tr}(J(s)) ds \right)\right\} = \frac{e^{-2\alpha^2kt}}{V(\rho(0))} e^{-(k-1)(3n+2-4k)kt/3}\frac{2\pi^{n-2k+1}}{\Gamma(n-2k+1)} \\
\det\left(\int [\cosh(2\zeta)]^{a-1}V(t)^{(n-2k+1, \alpha)}\left(\frac{1}{2}\cosh^{-1}(\rho_j(0)),r\right) (\sinh(r))^{2(n-2k+1)-1}\cosh(r) dr\right)_{a,j=1}^k. 
\end{multline}
\end{proposition}

\begin{remark}
The heat kernel $V(t)^{(n-2k+1, \alpha)}(0,\cosh\zeta)$ is given by (\cite{Aya-Int}, Theorem 2.2.): 
\begin{multline*}
 V(t)^{(n-2k+1, \alpha)}\left(0, r\right) = \frac{4\pi^{n-2k+1}}{\Gamma(n-2k+1)} \int_{r}^{\infty} \frac{dx \sinh(x)}{\sqrt{\cosh^2(x) - \cosh^2(r)}} s_{t,2(n-2k+1)+1}(\cosh(x))
 \\ {}_2F_1\left(-2\alpha, 2\alpha, \frac{1}{2}; \frac{\cosh(r) - \cosh(x)}{2\cosh(r)}\right)dx 
\end{multline*}
where ${}_2F_1$ is the Gauss hypergeometric function and 
\begin{align*}
s_{t,2(n-2k+1)+1}(\cosh(x)) = \frac{e^{-(n-2k+1)^2t/2}}{(2\pi)^{n-2k+1}\sqrt{2\pi t}}\left(-\frac{1}{\sinh(x)}\frac{d}{dx}\right)^{n-2k+1}e^{-x^2/(2t)}, \quad x > 0,
\end{align*}
is the heat kernel with respect to the volume measure of the $2(n-2k+1)+1$-dimensional real hyperbolic space $H^{2n+1}$. 

More generally, if $r_0 \geq 0$ then the heat kernel $V(t)^{(n-2k+1, \alpha)}(r_0,r)$ comes with the following additional term resulting from the integration over the sphere $S^{2n-1}$ of an automorphy factor (see \cite{Aya-Int}, Theorem 2.2):
\begin{equation*}
e^{-2i\alpha \arg(1-\langle z,w\rangle)}, \quad |z| = \tanh(r_0),\quad |w| = \tanh(r), \quad w,z \in \mathbb{C} H^{n-2k+1},
\end{equation*}
with respect to the angular part of $w$. 
\end{remark}

\section{Skew-product decompositions, generalized stochastic areas and asymptotic windings}

\subsection{Skew-product decomposition}

Recall from Section \ref{sec:notations} the hyperbolic Stiefel fibration 
\[
\mathbf{U}(k) \to HV_{n,k} \to HG_{n,k},
\]
from which we can view $HV_{n,k}$ as a $\mathbf{U}(k)$-principal bundle over $HG_{n,k}$. Our goal in this section is to decompose the Brownian motion in $HV_{n,k}$ as a skew-product with respect to this fibration.

We first note that a computation similar to the computation done in Lemma \ref{connection form} shows that the connection form of this bundle is given by the following $\mathfrak{u}(k)$-valued one-form on $HV_{n,k}$:
\begin{align}\label{eq-contact-form}
\omega:&=\frac{1}{2} \left( (X^* \,  Z^*)\begin{pmatrix} I_{n-k}   & 0 \\ 0 & -I_{k} \end{pmatrix}d\begin{pmatrix}  X \\ Z  \end{pmatrix}-d(X^* \, Z^*)\begin{pmatrix} I_{n-k}   & 0 \\ 0 & -I_{k} \end{pmatrix}\begin{pmatrix}  X \\ Z  \end{pmatrix}\right) \\
&=\frac{1}{2} \left(X^*dX-dX^*X -(Z^*dZ-dZ^*Z)\right).
\end{align}

In the homogeneous coordinate $w:=XZ^{-1}$ introduced in Section \ref{sec:notations}, we consider then the following $\mathfrak{u}(k)$ valued one-form defined on $HG_{n,k}$ 
 \begin{align}\label{eq-def-eta}
 \eta:=& \frac12 \left( (I_k-w^* w)^{-1/2} (w^*dw-dw^* \, w)(I_k-w^* w)^{-1/2}  \right. \\
  & \left. - (I_k-w^* w)^{-1/2}\, d(I_k-w^* w)^{1/2}+d(I_k-w^* w)^{1/2} \, (I_k-w^* w)^{-1/2} \right) \notag
 \end{align}

We  are now in position to prove the following skew-product decomposition of the Brownian motion on ${HV}_{n,k}$.
\begin{theorem}\label{skew}
Let $(w (t))_{t \ge 0}$ be a Brownian motion on ${HG}_{n,k}$ started at $w_0=X_0Z_0^{-1} \in {HG}_{n,k}$ and $(\Theta(t))_{t \ge 0}$ the $\mathbf{U}(k)$-valued  solution of the Stratonovich stochastic differential equation
\begin{align*}
\begin{cases}
d\Theta(t) = \circ d\mathfrak A(t) \, \Theta(t) \\
\Theta_0=(Z_0 Z^*_0)^{-1/2}Z_0,
\end{cases}
\end{align*}
where  $\mathfrak{a}_t= \int_{w[0,t]} \eta$. Then the process
\[
\widetilde{w}_t:=\begin{pmatrix} w (t)  \\  I_k  \end{pmatrix}(I_k-w (t)^*w (t))^{-1/2}\Theta(t)
\]
is the horizontal lift at $\begin{pmatrix} X_0  \\ Z_0 \end{pmatrix}$ of $(w (t))_{t \ge 0}$ to ${HV}_{n,k}$. Moreover, if we denote by $(\Omega(t))_{t \ge 0}$ a Brownian motion on the unitary group $\mathbf{U}(k)$ independent from $(w (t))_{t \ge 0}$, then the process $$\begin{pmatrix} w (t)  \\  I_k  \end{pmatrix}(I_k-w (t)^*w (t))^{-1/2}\Theta(t) \,  \Omega(t)$$
is a Brownian motion on $HV_{n,k}$ started at $\begin{pmatrix} X_0  \\ Z_0 \end{pmatrix}$.
\end{theorem}
\begin{proof}
Note the fact that on $\widetilde{w}[0,t]$,
\begin{align*}
\omega&=\frac12\left(\circ d\Theta^*\Theta-\Theta^*\circ d\Theta \right)+\Theta^*\circ \eta\Theta\\
&=\Theta^*\circ\left(-d\mathfrak{a}+\eta\right) \Theta=0,
\end{align*}
where $\Theta=(I_k-w^*w)^{1/2}Z$.
The first assertion then follows from the definition of horizontal stochastic lift, namely 
$$\int_{\widetilde{w}[0,t]} \omega=0.$$

By taking into account that the pseudo-Riemannian submersion $HV_{n,k}\to {HG}_{n,k}$ is totally geodesic, the second assertion then follows from a similar argument as in the proof of Theorem \ref{skew}.
\end{proof}

\subsection{Limit theorem for the generalized stochastic area process in the hyperbolic complex Grassmannian}\label{Limit hyp stiefel kahl}

Let $(w (t))_{t \ge 0}$ be a Brownian motion on $HG_{n,k}$ as in Theorem \ref{main:s1}. From \eqref{eq-def-eta} we have that
\begin{align}\label{eq-int-tr-eta}
-\int_{w[0,t]} \mathrm{tr} (\eta) & =\frac12 \mathrm{tr} \left[ \int_0^t   (I_k-J)^{-1/2} (\circ dw^* \, w-w^*\circ dw)(I_k-J)^{-1/2}\right] \notag\\
 & =\frac12 \mathrm{tr} \left[ \int_0^t   (I_k-J)^{-1/2} ( dw^* \, w-w^* dw)(I_k-J)^{-1/2}\right].
\end{align}
We will now prove that:
\begin{align*}
d \mathrm{tr} (\eta)  & = \mathrm{tr}\left[(I_k -w^*w)^{-1}\partial w^*(I_{n-k} - ww^*)^{-1} \partial w\right] 
\\& =-\partial \overline{\partial} \ln \det (I_k -w^*w) 
\end{align*}
where 
\begin{equation*}
\partial: = \sum_{1 \leq a \leq n-k, 1 \leq b\leq k} dw_{a,b}\frac{\partial}{\partial{w_{a,b}}}, \quad \overline{\partial}: = \sum_{1 \leq a \leq n-k, 1 \leq b\leq k} d\overline{w}_{a,b}\frac{\partial}{\partial{\overline{w}_{a,b}}}
\end{equation*}
are the Dolbeaut operators and are such that $d = \partial + \overline{\partial}$.  Actually, the function $K(w)=-\ln\det(I_k-w^*w)$ is a K\"ahler potential and 
\[
\mathrm{tr} (\eta)=-\frac{1}{2} (\partial -\bar{\partial})K.
\]

We shall use the analogue of Lemma 1 in \cite{BZ} for complex Hermitian invertible matrices which asserts that: 
\begin{equation*}
\partial_{M_{qj}} M^{-1}_{pl}  = -M^{-1}_{pq} M^{-1}_{jl}, \quad \partial_{M_{qj}} \log[\det(M)] = M^{-1}_{jq}.
\end{equation*}
Now, recall that 
\begin{equation*}
-\mathrm{tr} (\eta) = \frac12 \mathrm{tr} \left[(I_k-w^*w)^{-1} ( dw^* \, w-w^* dw)\right], 
\end{equation*}
and take $M := (I_k-w^*w)$. Then the chain rule yields the following identities: 
\begin{eqnarray*}
\partial M^{-1} (dw)^* w & = & M^{-1} w^* (dw) M^{-1} (dw)^* w \\ 
\overline{\partial} M^{-1} (dw^*) w & = & M^{-1} (dw)^* w M^{-1} (dw)^* w \\ 
\partial M^{-1} w^* (dw)  & = & M^{-1} w^* (dw) M^{-1} w^* (dw) \\  
\overline{\partial} M^{-1} w^* (dw)  & = & M^{-1} (dw)^* w M^{-1} w^* (dw).
\end{eqnarray*}   
Next, we use the fact that the exterior product of one-forms is alternating to see that: 
\begin{align*}
\mathrm{tr} \left[\overline{\partial} M^{-1} (dw)^* w\right] = \mathrm{tr} \left[\partial M^{-1} w^* (dw)\right] = 0. 
\end{align*}
As a result, we get: 
\begin{align*}
-2d\mathrm{tr} (\eta) & =  \mathrm{tr}\left[M^{-1} w^* (dw) M^{-1} (dw)^* w\right] - \mathrm{tr}\left[M^{-1} (dw)^* wM^{-1} w^* (dw)\right] - 2\mathrm{tr}\left[M^{-1} (dw)^* (dw)\right]
\\& = 2\mathrm{tr}\left[M^{-1} w^* (dw) M^{-1} (dw)^* w\right] +  2\mathrm{tr}\left[(dw) M^{-1} (dw)^*\right]
\\& = 2 \mathrm{tr}\left[(I_{n-k} + wM^{-1}w^*) (dw) M^{-1} (dw)^*\right]. 
\end{align*}
Moreover, we note that 
\begin{align*}
(I_{n-k} + wM^{-1}w^*)(I_{n-k} - ww^*) &= I_{n-k} - ww^* + wM^{-1}w^* - wM^{-1}(w^*w)w^* 
\\& = I_{n-k} - ww^* + wM^{-1}w^* - wM^{-1}(w^*w - I_k + I_k)w^* = I_{n-k}, 
\end{align*}
so that $I_{n-k} + wM^{-1}w^* = (I_{n-k} - ww^*)^{-1}$, and in turn: 
\begin{align*}
-\mathrm{tr} (\eta) & = \mathrm{tr}\left[(I_{n-k} - ww^*)^{-1}(dw) (I_k-w^*w)^{-1} (dw)^*\right]
\\& = - \mathrm{tr}\left[(I_k-w^*w)^{-1} (dw)^*(I_{n-k} - ww^*)^{-1}(dw) \right],
\end{align*}
as claimed.

We proceed now to the proof of 
\begin{equation*}
d \mathrm{tr} (\eta) = -\partial \overline{\partial} \ln \det (I_k -w^*w) = -\partial \overline{\partial} \ln \det(M).
\end{equation*}
To this end, we derive: 
\begin{equation*}
\partial_{qm}\ln \det(M) = \sum_{l} (w^*)_{lq}M^{-1}_{ml}dw_{qm} 
\end{equation*}
whence we deduce 
\begin{align*}
\overline{\partial_{rj}}\partial_{qm}\ln \det(M) = -\delta_{rq}M^{-1}_{mj} d\overline{w}_{rj} \wedge dw_{qm} - \sum_{l,s}(w^*)_{lq}M^{-1}_{mj}M^{-1}_{sl}d\overline{w}_{rj} \wedge dw_{qm}.
\end{align*}
Summing over $r,j,q,m$, we get 
\begin{align*}
\partial \overline{\partial} \ln \det(M) = \mathrm{tr}\left[(dw)M^{-1}(dw)^*\right] + \mathrm{tr}\left[M^{-1} w^* (dw) M^{-1} (dw)^* w\right] = - d\mathrm{tr} (\eta).
\end{align*}
 It follows that $i \mathrm{tr} (d\eta)$ is the K\"ahler form on the complex manifold $HG_{n,k}$. Therefore, in some sense, the functional $\int_{w[0,t]} \mathrm{tr} (\eta)$ can be interpreted as  a generalized  stochastic area process on $HG_{n,k}$ as in the case of the Stiefel fibration, and as a nice consequence of \eqref{eq-int-tr-eta}, we obtain  in the theorem below the limiting theorem for this generalized stochastic area process.

\begin{theorem}\label{LimitTheorem}
When $t\to+\infty$, we have in distribution 
\[
\frac{1}{i\sqrt{t}} \int_{w[0,t]} \mathrm{tr} (\eta) \to \mathcal{N}(0,k),
\]
\end{theorem}
\begin{proof}
From the proof of Proposition \ref{SDE-Matrix}, we readily have: 
\begin{equation*}
dw^*w - w^*dw =\sqrt{I_k-J(t)} dR\sqrt{I_{k} - J} \sqrt{J} - \sqrt{J} \sqrt{I_{k} - J}dR^*\sqrt{I_k-J}, 
\end{equation*}
where $(r(t))_{t\ge0}$ is a $k\times k$ complex matrix-valued Brownian motion. Hence
\begin{align*}
 (I_k-J)^{-1/2} ( dw^* \, w-w^* dw)(I_k-J)^{-1/2}&=dR\sqrt{J}-\sqrt{J} dR^*\\
 &=dRU\sqrt{\Lambda}U^*-U\sqrt{\Lambda}U^*dR^*
\end{align*}
where the second quality comes from the diagonalization of $J=U\Lambda U^*$ where $U\in \mathbf{U}(k)$ and $\Lambda=\mathrm{diag}\{\lambda_1,\dots, \lambda_k\}$. Therefore
\begin{align*}
-\int_{w[0,t]} \mathrm{tr} (\eta) 
 & =\frac12  \int_0^t  \mathrm{tr} \left[ dR\sqrt{\Lambda}-\sqrt{\Lambda}dR^*\right]\\
 &\stackrel{D}{=}i\mathcal{B}_{\int_0^t \mathrm{tr}(J)ds}.
\end{align*}
where $\mathcal{B}$ is a one-dimensional Brownian motion independent of $J$.  We then obtain the desired conclusion  from \eqref{eq-tr-J}.
\end{proof}

\begin{remark}
When $k=1$, the limiting result stated in Theorem \ref{LimitTheorem} coincides with the one for the stochastic area process on the anti-de Sitter space. 
\end{remark}

\begin{remark}
The previous proof has shown that
\begin{align*}
\int_{w[0,t]} \mathrm{tr} (\eta) 
 &=i\mathcal{B}_{\int_0^t \mathrm{tr}(J)ds}.
\end{align*}
where $\mathcal{B}$ is a one-dimensional Brownian motion independent of $J$, therefore
\[
\mathbb{E} \left( e^{i \alpha \int_{w[0,t]} \mathrm{tr} (\eta) } \right)=\mathbb{E} \left( e^{- \alpha \mathcal{B}_{\int_0^t \mathrm{tr}(J)ds}} \right)=\mathbb{E} \left( e^{- \frac{\alpha^2}{2} \int_0^t \mathrm{tr}(J)ds} \right).
\]
A formula for the Laplace transform of the generalized area functional $\frac{1}{i} \int_{w[0,t]} \mathrm{tr} (\eta) $ therefore follows from  Theorem \ref{theo-Laplace}.
\end{remark}

\subsection{Limit theorem for the windings of the block determinants of the Brownian motion in $\mathbf{U}(n-k,k)$} 

In this section, we give an application of Theorems \ref{skew} and  \ref{LimitTheorem} to the study of the Brownian windings in the Lie group $\mathbf{U}(n-k,k)$.

\begin{theorem}
Let $U(t)=\begin{pmatrix}  Y(t) & X(t) \\ w (t) & Z(t) \end{pmatrix}$ be a Brownian motion on the Lie group $\mathbf{U}(n-k,k)$ with $1 \le k \le n-k$.  One has the polar decomposition
\[
\det (Z(t))=\varrho(t) e^{i\theta(t)}
\]
where $\varrho(t)  \ge 1$ is a continuous semimartingale  and $\theta(t)$ is a continuous martingale such that the following convergence holds in distribution when $t \to +\infty$
\[
\frac{\theta(t)}{\sqrt{t}} \to \mathcal{N}(0,2k) .
\]

\end{theorem}

 \begin{proof}
 We first note that from Theorem \ref{skew}, we have the identity in law
\[
\det (Z(t)) =\det (I_k-w (t)^*w (t))^{-1/2} \det \Theta(t) \, \det  \Omega(t).
\]
From Lemma \ref{trace lemma} we therefore obtain
\[
\det (Z(t)) =\varrho(t) e^{i \theta(t)}
\]
with
\[
\varrho(t) =\det(I_k-J(t))^{-1/2}, \, \, i \theta(t) = i\theta_0+ \mathrm{tr}(D_t)+ \int_{w[0,t]} \mathrm{tr} (\eta) 
\]
where $D_t$ is a Brownian motion on $\mathfrak{u}(k)$ independent from $w$ and $\theta_0$ is such that 
\begin{equation*}
e^{i\theta_0}=\frac{\det Z_0}{| \det Z_0|}.
\end{equation*}
The result follows then from Theorem \ref{LimitTheorem} and from the scaling property of $\mathrm{tr}(D_t)$: 
\begin{equation*}
\mathrm{tr}(D_t) \overset{d}{=} i\sqrt{t} \mathcal{N}(0,k).
\end{equation*}
\end{proof}

\section{Appendix: the full expansion of the Laplace transform in the rank-one case}
 For $k=1$, take $\rho(0) = 1$, then the Laplace transform \eqref{eq-NewLaplace} reduces to: 
\begin{multline}\label{Laplace1}
\mathbb{E}\left\{\exp-\left(2\alpha^2 \int_0^t J(s) ds \right)\right\} = \frac{2\pi^{n-1}}{\Gamma(n-1)} e^{-2\alpha^2t}  \int_0^{\infty} V(t)^{(n-1, \alpha)}\left(0,r\right) (\sinh(r))^{2n-3}\cosh(r) dr, 
\end{multline}
where (\cite{Aya-Int}, see also the proof of Theorem 1 in \cite{Dem-Windings}):
\begin{align*}
 V(t)^{(n-1, \alpha)}\left(0,r\right) = & 2 \int_{r}^{\infty} \frac{dx \sinh(x)}{\sqrt{\cosh^2(x) - \cosh^2(r)}} s_{t,2(n-1)+1}(\cosh(x))
   \\
   & {}_2F_1\left(-2\alpha, 2\alpha, \frac{1}{2}; \frac{\cosh(r) - \cosh(x)}{2\cosh(r)}\right)dx. 
\end{align*}
For ease of writing and in order to match our notations with those used in \cite{Dem-Windings}, we shall shift $n \geq 2 \rightarrow n+1, n \geq 1$. After all, the result of Proposition \ref{Lemma1} does not depend on $n$. 
In order to derive the limit of \eqref{Laplace1} as $t \rightarrow \infty$ (after rescaling), we firstly apply the quadratic transformation: 
\begin{equation*}
{}_2F_1(a,b,(a+b+1)/2; z) = {}_2F_1(a/2,b/2,(a+b+1)/2; 4z(1-z))
\end{equation*}
followed by Euler's transformation (\cite{Erd}, (22), p.64): 
\begin{equation*}
{}_2F_1(a,b,c; z) = (1-z)^{-a}{}_2F_1\left(a,c-b,c; \frac{z}{z-1}\right),
\end{equation*}
valid whenever both sides are analytic. Doing so, the heat semi-group $V(t)^{(n, \alpha)}\left(0,r\right)$ may be written as: 
\begin{multline*}
V(t)^{(n, \alpha)}\left(0,r\right) = \frac{4\pi^{n}}{\Gamma(n)} \int_{r}^{\infty} \frac{dx \sinh(x)}{\sqrt{\cosh^2(x) - \cosh^2(r)}} \frac{\cosh^{2\alpha}(r)}{\cosh^{2\alpha}(x)} 
\\ {}_2F_1\left(\frac{1}{2}+\alpha, \alpha, \frac{1}{2}; 1- \frac{\cosh^2(r)}{\cosh^2(x)}\right)  s_{t,2n+1}(\cosh(x))dx.
\end{multline*}
This expression has the merit to involve the Gauss hypergeometric series rather than its analytic continuation to the slit plane $\mathbb{C} \setminus [1,\infty)$: 
\begin{equation*}
{}_2F_1\left(\frac{1}{2}+\alpha, \alpha, \frac{1}{2}; 1- \frac{\cosh^2(r)}{\cosh^2(x)}\right) = \sum_{j\geq 0} \frac{(\alpha)_j(\alpha+1/2)_j}{(1/2)_j j!} \left(1- \frac{\cosh^2(r)}{\cosh^2(x)}\right)^j, \quad \zeta \leq x.
\end{equation*}
Consequently, the generalized binomial Theorem yields following expansion:
\begin{multline}\label{eq-2F1}
\left(\frac{\cosh^{2}(r)}{\cosh^{2}(x)}\right)^{\alpha} {}_2F_1\left(\frac{1}{2}+\alpha, \alpha, \frac{1}{2}; 1- \frac{\cosh^2(r)}{\cosh^2(x)}\right)  = 
\left(1-\left(1-\frac{\cosh^{2}(r)}{\cosh^{2}(x)}\right)\right)^{\alpha} \\ {}_2F_1\left(\frac{1}{2}+\alpha, \alpha, \frac{1}{2}; 1- \frac{\cosh^2(r)}{\cosh^2(x)}\right) \nonumber
  = 1+ \sum_{j\geq 1} c_j(\alpha) \left(1- \frac{\cosh^2(\zeta)}{\cosh^2(x)}\right)^j, 
\end{multline}
where 
\begin{equation*}
c_j(\alpha) = \frac{1}{j!}\sum_{m=0}^j \binom{j}{m} \frac{(\alpha+1/2)_m(\alpha)_m(-\alpha)_{j-m}}{(1/2)_m}, \quad j \geq 1. 
\end{equation*}
Now, assume $\alpha$ is small enough. Then for any $j-m \geq 1$, one has: 
\begin{equation*}
(-\alpha)_{j-m} = (-\alpha)(1-\alpha)\cdots (j-m-1-\alpha) < 0
\end{equation*}
so that 
\begin{equation*}
c_j(\alpha) < \frac{(\alpha+1/2)_j(\alpha)_j}{j!(1/2)_j}, \quad j \geq 1. 
\end{equation*}
Moreover, the fact that $m \mapsto (\alpha + 1/2)_m/(1/2)_m$ is increasing together with the previous observation show that 
\begin{align*}
c_j(\alpha) \geq \frac{(\alpha + 1/2)_j}{(1/2)_j}\left\{(\alpha)_j + \sum_{m=0}^{j-1}\binom{j}{m}(\alpha)_m(-\alpha)_{j-m}\right\} = 0,
\end{align*}
where we use the fact that the sequence of Pochhammer symbols (rising factorials) is of binomial-type. Using the fact that $s_{t,2n+1}(\cosh(x))$ is a probability density with respect to the volume element 
\begin{equation*}
\frac{2\pi^{n+1/2}}{\Gamma(n+1/2)} \sinh^{2n}(x)
\end{equation*} 
we readily get
\begin{equation*}
\frac{4\pi^{n}}{\Gamma(n)}\int_0^{\infty} (\sinh(r))^{2n-1}\cosh(r) dr \int_{r}^{\infty} \frac{dx \sinh(x)}{\sqrt{\cosh^2(x) - \cosh^2(r)}}s_{t,2n+1}(\cosh(x))  = 1.
\end{equation*}

On the other hand, changing the order of integration, we have: 
\[
\mathbb{E}\left\{\exp-\left(2\alpha^2 \int_0^t J(s) ds \right)\right\} = e^{-2\alpha^2t}\left(1+\sum_{j\ge1}c_j(\alpha)I_j\right).
\]
where we set for any $j\ge1$:
\begin{align*}
I_j:=  \frac{4\pi^{n}}{\Gamma(n)} & \int_0^{\infty} (\sinh(r))^{2n-1}\cosh(r) \\
&  \int_{r}^{\infty} \frac{dx \sinh(x)}{\sqrt{\cosh^2(x) - \cosh^2(r)}}s_{t,2n+1}(\cosh(x)) \left(\frac{\cosh^2x-\cosh^2r}{\cosh^2x}\right)^j dr.
\end{align*}
Using standard variables changes, we derive: 
\[
\int_0^x\frac{ (\sinh(\zeta))^{2n-1}\cosh(r)}{\sqrt{\cosh^2(x) - \cosh^2(r)}}  \left(\frac{\cosh^2x-\cosh^2r}{\cosh^2x}\right)^jdr
=\frac{(\sinh (x))^{2j+2n-1}}{(\cosh(x))^{2j}}\frac{\Gamma(j+\frac12)\Gamma(n)}{2\Gamma(j+n+\frac12)}
\]
whence
\begin{align*}
I_j&=\frac{4\pi^{n}\Gamma(j+\frac12)}{2\Gamma(j+n+\frac12)}\int_0^{\infty} (\sinh(x))^{2n}(\tanh(x))^{2j}  s_{t,2n+1}(\cosh(x)) dx\\
&=\frac{\Gamma(n+\frac12)\Gamma(j+\frac12)}{\sqrt{\pi}\Gamma(n+j+\frac12)}\mathbb E\left( (\tanh(d(0, B(t))))^{2j}\right),
\end{align*} 
where $d(\cdot, \cdot)$ is the Riemannian distance and $B(t)$ is the Brownian motion on $H^{2n+1}$. Altogether, we get for small $\alpha$:

\begin{equation*}
\mathbb{E}\left\{\exp-\left(2\alpha^2 \int_0^t J(s) ds \right)\right\} =  e^{-2\alpha^2t} + e^{-2\alpha^2t} \mathbb{E}\left(\sum_{j\ge1}\frac{(1/2)_j}{(n+1/2)_j} c_j(\alpha)(\tanh(d(0, B(t)))^{2j}\right).
\end{equation*}
But
\begin{align*}
 & \sum_{j\ge1}\frac{(1/2)_j}{(n+1/2)_j} c_j(\alpha)(\tanh(d(0, B(t)))^{2j} \\
  \leq  &\sum_{j\ge1}\frac{(\alpha+1/2)_j(\alpha)_j}{j!(n+1/2)_j} (\tanh(d(0, B(t)))^{2j}) \\
 \leq & \frac{\alpha}{n+1/2} \left(\alpha+\frac{1}{2}\right)  {}_2F_1\left(\alpha+1, \alpha + \frac{3}{2}, n+\frac{3}{2}; \tanh(d(0, B(t)))^{2}\right) \\
 \leq &  \frac{\alpha(2\alpha+1)}{2n+1} {}_2F_1\left(\alpha+1, \alpha + \frac{3}{2}, n+\frac{3}{2}; 1\right) \\
 = & \frac{\alpha(2\alpha+1)}{2n+1}\frac{\Gamma(n+3/2)\Gamma(n-1-\alpha)}{\Gamma(n-\alpha)\Gamma(n-\alpha+1/2)}
\end{align*}
where the last line follows from Gauss hypergeometric Theorem. Consequently, the dominated convergence theorem entails 
\begin{align*}
\lim_{\alpha \rightarrow 0^+} \mathbb{E}\left(\sum_{j\ge1}\frac{(1/2)_j}{(n+1/2)_j} c_j(\alpha)(\tanh(d(0, B(t)))^{2j}\right) = 0 
\end{align*}
which in turn allows to recover our limit theorem in Proposition \ref{Lemma1}.

\chapter{Appendix 1:  A primer on stochastic calculus}

In this Appendix, we review briefly the basic theory of stochastic
processes, diffusion processes and stochastic integrals as is used in this monograph. The stochastic integration is a natural, easy and fruitful integration theory
which is due to  It\^o. For a much more complete exposition we refer to \cite{MR3236242}, \cite{Ikeda}, \cite{MR0345224},   \cite{Pro}, \cite{Re-Yo} and \cite{Rogers1,Rogers2}. 
%
%
\section*{Stochastic processes and Brownian motion}

Let $\left( \Omega , (\mathcal{F}_t)_{t \geq 0} , \mathcal{F},
\mathbb{P} \right)$ be a filtered probability space which
satisfies the usual conditions, that is:
\begin{enumerate}
\item $(\mathcal{F}_t)_{t \geq 0}$ is a filtration, i.e. an increasing family of
sub-$\sigma$-fields of $\mathcal{F}$;
\item for any $t \geq 0$, $\mathcal{F}_t$ is complete with respect
to $\mathbb{P}$, i.e. every subset of a set of measure zero is
contained in $\mathcal{F}_t$;
\item $(\mathcal{F}_t)_{t \geq 0}$ is right continuous, i.e. for
any $t \geq 0$,
\[
\mathcal{F}_t =\bigcap_{s>t} \mathcal{F}_s.
\]
\end{enumerate}
A stochastic process $(X(t))_{t \geq 0}$ on $\left( \Omega ,
(\mathcal{F}_t)_{t \geq 0} , \mathcal{F}, \mathbb{P} \right)$ is
an application
\[
\begin{array}{lll}
\mathbb{R}_{\geq 0} \times \Omega & \rightarrow & \mathbb{R} \\
(t,\omega) & \rightarrow & X(t) (\omega),
\end{array}
\]
which is measurable with respect to $\mathcal{B} (
\mathbb{R}_{\geq 0}) \otimes \mathcal{F}$. The process $(X(t))_{t
\geq 0}$ is said to be adapted (with respect to the filtration
$(\mathcal{F}_t)_{t \geq 0}$) if for every $t \geq 0$, $X(t)$ is
$\mathcal{F}_t$-measurable. It is said to be continuous if for
almost every $\omega \in \Omega$, the function $t \rightarrow X(t)
(\omega)$ is continuous. A stochastic process $(\tilde{X}(t))_{t
\geq 0}$ is called a modification of $(X(t))_{t \geq 0}$ if for
every $t \geq 0$,
\[
\mathbb{P} \left( X(t) = \tilde{X}(t) \right)=1.
\]
The following theorem known as Kolmogorov's continuity
criterion\index{Kolmogorov's continuity criterion} is widely used to work with continuous modifications.
\begin{theorem}
Let $\alpha, \varepsilon, c >0$. If a process $(X(t))_{t \geq 0}$
satisfies for every $s,t \geq 0$,
\[
\mathbb{E} \left( \mid X(t) - X(s) \mid^{\alpha} \right) \leq c \mid
t-s \mid^{1+\varepsilon},
\]
then there exists a modification of $(X(t))_{t \geq 0}$ which is a
continuous process.
\end{theorem}
One of the most important examples of stochastic process is the
Brownian motion. A process $(B(t))_{t \geq 0}$ defined on $\left(
\Omega , (\mathcal{F}_t)_{t \geq 0} , \mathbb{P} \right)$ is said
to be a standard Brownian motion\index{Brownian motion} with
respect to the filtration $(\mathcal{F}_t)_{t \geq 0}$ if:
\begin{enumerate}
\item $B(0)=0$ a.s.;
\item $B(1) \overset{d}{=} \mathcal{N} (0,1)$;
\item $(B(t))_{t \geq 0}$ is $\mathcal{F}$-adapted;
\item For any $t>s\geq 0$,
\[
B(t) - B(s) \overset{d}{=} B(t-s);
\]
\item For any $t>s\geq 0$,
$B(t) -B(s)$ is independent of $\mathcal{F}_s$.
\end{enumerate}
If a process is a Brownian motion with respect to its own
filtration, we simply say that it is a standard Brownian motion,
without mentioning the underlying filtration. A $d$-dimensional
process $(B(t))_{t \geq 0}=(B^1(t),...,B^d(t))_{t \geq 0}$ is said to
be a $d$-dimensional standard Brownian motion if the processes
$(B^1(t))_{t \geq 0},...,(B^d(t))_{t \geq 0}$ are independent
standard Brownian motions.

A standard Brownian motion $(B(t))_{t \geq 0}$ is a Gaussian
process, that is for every $n \in \mathbb{N}^*$, and $0\leq t_1
\leq \cdots \leq t_n$, the random variable
\[
\left( B(t_1),...,B(t_n) \right)
\]
is Gaussian. The mean function of $(B(t))_{t \geq 0}$ is
\[
\mathbb{E} (B(t))=0,
\]
and its covariance function is
\[
\mathbb{E} (B(s) B(t))=\inf (s,t).
\]

It is possible to show, thanks to Kolmogorov's continuity criterion, that it is always possible to find a continuous modification of a standard Brownian motion. Of course, we always work with this
continuous modification. The law of a standard Brownian motion,
which therefore lives on the space of continuous functions
$\mathbb{R}_{\geq 0} \rightarrow \mathbb{R}$ is called the Wiener
measure\index{Wiener measure}. If $(B(t))_{t \geq 0}$ is a standard
Brownian motion, then the following properties hold:
\begin{enumerate}
\item For every $c >0$, $(B(ct))_{t \geq 0}\stackrel{\mbox{law}}{=} (\sqrt{c} B(t))_{t \geq
0}$;
\item $(t B(\frac{1}{t}))_{t \geq 0}\overset{d}{=}( B(t))_{t \geq
0}$;
\item For almost every $\omega \in \Omega$, the path $t
\rightarrow B(t) (\omega)$ is nowhere differentiable and locally
H\"older continuous of order $\alpha$ for every $\alpha
<\frac{1}{2}$;
\item For every subdivision $0 = t_0 \leq \cdots \leq t_n =t$ whose mesh tends to 0, almost surely we have
\[
\lim_{n \rightarrow +\infty} \sum_{i=0}^{n-1} \left( B(t_{i+1})
-B(t_i) \right)^2 =t.
\]
\end{enumerate}

\section*{Martingales}
Consider an adapted and continuous process $(M(t))_{t \geq 0}$
defined on a filtered probability space $\left( \Omega ,
(\mathcal{F}_t)_{t \geq 0} , \mathcal{F}, \mathbb{P} \right)$.
\begin{definition}
The process $(M(t))_{t \geq 0}$ is said to be a
martingale (with respect to the filtration
$(\mathcal{F}_t)_{t \geq 0}$) if:
\begin{enumerate}
\item For $t \geq 0$, $\mathbb{E} \left( \mid M(t) \mid \right) < +
\infty$;
\item For $0 \leq s \leq t$, $\mathbb{E} \left( M(t) \mid
\mathcal{F}_s \right)=M(s)$.
\end{enumerate}
\end{definition}
For instance a standard Brownian motion is a martingale. A random variable $T: \Omega \to \mathbb{R}_{\ge 0} \cup \{ +\infty \}$ is called a stopping time of the filtration $(\mathcal{F}_t)_{t \geq 0}$ if for every $t \ge 0$,
\[
\{ T \le t \} \in \mathcal{F}_t.
\]

For martingales, we have the following proposition, which is known as the stopping theorem.
\begin{proposition}
The following properties are equivalent:
\begin{enumerate}
\item The process $(M(t))_{t \geq 0}$ is a martingale;
\item For any bounded stopping time $T$, $\mathbb{E}
(M(T))=\mathbb{E} (M(0))$;
\item For any pair of bounded stopping times $S$ and $T$, with $S
\leq T$, $\mathbb{E} \left( M(T) \mid \mathcal{F}_S \right)=M(S)$.
\end{enumerate}
\end{proposition}
Actually, if $T$ is a bounded stopping time, then the process
$(M(t \wedge T))_{t \geq 0}$ is also a martingale.

The notion of uniform integrability is a  crucial tool in the study of continuous time martingales.

\begin{definition}
Let $(X_i)_{i \in \mathcal{I}}$ be a family of random variables. We say that the family $(X_i)_{i \in \mathcal{I}}$ 
is uniformly integrable if for every $\varepsilon >0$, there exists $K \ge 0$ such that
\[
\forall i \in \mathcal{I},\quad \mathbb{E}(| X_i | \mathbf{1}_{\mid X_i \mid >K})<\varepsilon.
\]
\end{definition}
We have the following properties:
\begin{itemize}
\item \hspace{.1in} A finite family of integrable random variables  is uniformly integrable ;
\item  \hspace{.1in} If the family $(X_i)_{i \in \mathcal{I}}$ is uniformly integrable then it is bounded in $L^1$, that is $\sup_\mathcal{I} \mathbb{E}(\mid X_i \mid ) <+\infty$;
\item  \hspace{.1in}  If the family $(X_i)_{i \in \mathcal{I}}$ is bounded in  $L^p$ with $p>1$, that is  $\sup_\mathcal{I} \mathbb{E}(\mid X_i \mid^p ) <+\infty$, then it is uniformly integrable.
\end{itemize}

\begin{theorem} (Doob's maximal inequalities)

Let $(\mathcal{F}_t)_{t \ge 0}$ be a filtration on probability space $(\Omega, \mathcal{F},\mathbb{P})$ and let
$(M(t))_{t \ge 0}$be a continuous martingale with respect to the filtration $(\mathcal{F}_t)_{t \ge 0}$.
\begin{enumerate}
\item Let $p \ge 1$ and $T>0$. If  $\mathbb{E} (\mid M(T)
\mid^p)<+\infty$, then for every $\lambda > 0$,
\[
\mathbb{P} \left( \sup_{0 \le t \le T} \mid M(t) \mid \ge \lambda
\right) \le \frac{\mathbb{E} \left( \mid M(T) \mid^p
\right)}{\lambda^p}.
\]
\item Let $p > 1$ and $T>0$. If $\mathbb{E} (\mid M(T)
\mid^p)<+\infty$ then,
\[
\mathbb{E}\left( \left( \sup_{0 \le t \le T} \mid M(t) \mid \right)^p
\right) \le \left( \frac{p}{p-1} \right)^p \mathbb{E} (\mid M(T)
\mid^p).
\]
\end{enumerate}
\end{theorem}

A martingale $(M(t))_{t \geq 0}$ is said to be square integrable
if for $t \geq 0$, $\mathbb{E} \left( M(t)^2 \right) < + \infty$.
In that case, from Jensen's inequality the function $t \rightarrow
\mathbb{E} \left( M(t)^2 \right)$ is increasing. We also have the
so-called Doob's inequality
\[
\mathbb{E} \left( \sup_{t \geq 0} M(t)^2 \right) \leq 4 \sup_{t
\geq 0} \mathbb{E} \left( M(t)^2 \right),
\]
which is one of the cornerstone of the stochastic integration theory.
Observe therefore that if $\sup_{t \geq 0} \mathbb{E} \left( M(t)^2
\right)< +\infty$, the martingale $(M(t))_{t \geq 0}$ is uniformly
integrable and converges in $L^2$ to a square integrable random
variable $M(\infty)$ which satisfies
\[
\mathbb{E} \left( M_{\infty} \mid \mathcal{F}_t \right)=M(t),
\text{ } t \geq 0.
\]
If $(M(t))_{t \geq 0}$ is a square integrable martingale, there
exists a unique increasing process denoted $(\langle M
\rangle_t)_{t \geq 0}$ which satisfies:
\begin{enumerate}
\item $\langle M \rangle_0=0$;
\item The process $(M(t)^2 - \langle M \rangle_t)_{t \geq 0}$ is a
martingale.
\end{enumerate}
This increasing process $(\langle M \rangle_t)_{t \geq 0}$ is
called the quadratic variation of the martingale $(M(t))_{t \geq
0}$. This terminology comes from the following property. If $0 =
t_0 \leq t_1 \leq \cdots \leq t_n =t$ is a subdivision of the time
interval $[0,t]$ whose mesh tends to 0, then in probability
\[
\lim_{n \rightarrow +\infty} \sum_{i=0}^{n-1} \left( M(t_{i+1})
-M(t_i) \right)^2 = \langle M \rangle_t.
\]
For technical reasons (localization procedures), we often have to
consider a wider class than martingales.
\begin{definition}
The process $(M(t))_{t \geq 0}$ is said to be a local
martingale (with respect to the filtration
$(\mathcal{F}_t)_{t \geq 0}$) if there exists a sequence $(T_n)_{n
\geq 0}$ of stopping times such that:
\begin{enumerate}
\item The sequence $(T_n)_{n\geq 0}$ is increasing and $\lim_{n \rightarrow +\infty} T_n =+\infty$ almost surely;
\item For every $n \geq 1$, the process $(M(t \wedge T_n))_{t \geq 0}$ is a uniformly integrable
 martingale with respect to the filtration $(\mathcal{F}_t)_{t \geq 0}$.
\end{enumerate}
\end{definition}
A martingale is always a local martingale but the converse is not
true. Nevertheless, a local martingale $(M(t))_{t \geq 0}$ such
that for every $t \geq 0$,
\[
\mathbb{E} \left( \sup_{s \leq t} \mid M(s) \mid \right) < +\infty,
\]
is a martingale. If $(M(t))_{t \geq 0}$ is a local martingale,
there still exists a unique increasing process denoted $(\langle M
\rangle_t)_{t \geq 0}$ which satisfies:
\begin{enumerate}
\item $\langle M \rangle_0=0$;
\item The process $(M(t)^2 - \langle M \rangle_t)_{t \geq 0}$ is a
local martingale.
\end{enumerate}
This increasing process $(\langle M \rangle_t)_{t \geq 0}$ is
called the quadratic variation of the local martingale $(M(t))_{t
\geq 0}$. By polarization, it is easily seen that, more generally,
if $(M(t))_{t\geq 0}$ and $(N(t))_{t\geq 0}$ are two continuous
local martingales, then there exists a unique continuous process
denoted $\left( \langle M,N \rangle_t \right)_{t \geq 0}$ and
called the quadratic covariation of $(M(t))_{t\geq 0}$ and
$(N(t))_{t\geq 0}$ which satisfies:
\begin{enumerate}
\item $\langle M,N \rangle_0=0$;
\item The process $(M(t) N(t) - \langle M ,N\rangle_t)_{t \geq 0}$ is a
local martingale.
\end{enumerate}
Before we turn to the theory of stochastic integration, we
conclude this section with L\'evy's characterization of Brownian
motion.
\begin{proposition}[L\'evy characterization theorem]\label{carac Levy}
Let $(M(t))_{t \geq 0}$ be a $d$-dimensional continuous local
martingale such that $M(0) = 0$ and
\[
\langle M^i , M^i \rangle_t = t, \text{ } \langle M^i , M^j
\rangle_t=0 \text{ if }i\neq j.
\]
Then $(M(t))_{t \geq 0}$ is a standard $d$-dimensional Brownian motion.
\end{proposition}

\section*{Stochastic integration}
Let $\left( \Omega , (\mathcal{F}_t)_{t \geq 0} , \mathcal{F},
\mathbb{P} \right)$ be a filtered probability space which
satisfies the usual conditions specified before. We aim now at
defining an integral $\int_0^t H(s) dM(s)$\index{It\^o's integral}
where $(M(t))_{t \geq 0}$ is an adapted continuous square
integrable martingale such that $\sup_{t \geq 0} \mathbb{E} (
M(t)^2 )< +\infty$ and $(H(t))_{t \geq 0}$ an adapted process
which shall be in a \textit{good} class. Observe that such an
integral could not be defined trivially since the Young's
integration theory does not cover the integration against paths
which are less than $\frac{1}{2}$-H\"older continuous. First, we
define the class of integrands. The predictable $\sigma$-field
$\mathcal{P}$ associated with the filtration $(\mathcal{F}_t)_{t
\geq 0}$ is the $\sigma$-field generated on $\mathbb{R}_{\geq 0}
\times \Omega$ by the space of indicator functions
$\mathbf{1}_{]S,T]}$, where $S$ and $T$ are two stopping times
such that $S \leq T$. An adapted stochastic process $(H(t))_{t \geq
0}$ is said to be predictable if the application $(t,\omega)
\rightarrow H(t) (\omega)$ is measurable with respect to
$\mathcal{P}$. Observe that a continuous adapted process is
predictable.

Let us first assume that $(H(t))_{t \geq 0}$ is a predictable
elementary process that can be written as
\begin{equation}\label{eq-elem-proc}
H(t) =\sum_{i=1}^{n-1} H_i \mathbf{1}_{(T_i , T_{i+1} ]} (t)
\end{equation}
where $H_i$ is a $\mathcal{F}_{T_i}$ measurable bounded random
variable, and where $(T_i)_{1 \leq i \leq n}$ is a finite
increasing sequence of stopping times. In that case, a natural
definition for $\int_0^t H(s) dM(s)$ is
\[
\int_0^t H(s) dM(s)=\sum_{i=1}^{n-1} H_i \left(M(t \wedge T_{i+1})
- M(t \wedge T_{i})\right).
\]
Then, we observe that the process $\left(\int_0^t H(s) dM(s)
\right)_{t \geq 0}$ is a bounded martingale. Moreover from  Doob's inequality it satisfies that
\[
\mathbb{E} \left( \sup_{t \geq 0} \left( \int_0^t H(s) dM(s)
\right)^2 \right) \leq 4 \parallel H \parallel_{\infty}^2
\mathbb{E} (M(\infty)^2).
\]
We also note that
\[
\mathbb{E} \left( \left( \int_0^t H(s) dM(s) \right)^2
\right)=\mathbb{E} \left( \sum_{i=1}^{n-1} H_i \left(M(t \wedge
T_{i+1}) - M(t \wedge T_{i})\right)^2 \right).
\]
Assume now that $(H(t))_{t \geq 0}$ is a bounded continuous adapted
process. To define $\int_0^t H(s) dM(s)$, the idea is of course to
approximate $(H(t))_{t \geq 0}$ with elementary processes
$(H^p(t))_{t \geq 0}$ as in \eqref{eq-elem-proc} and to check the convergence of $\left(
\int_0^t H^p(s) dM(s)\right)_{t \geq 0}$ with respect to a suitable
norm. Precisely, let us define for any $p \in \mathbb{N}^*$, the
following stopping times:
\[
T_0^p=0
\]
\[
T_1^p = \inf \left\lbrace t>0, \mid H(t) \mid \geq \frac{1}{p}
\right\rbrace,
\]
and by iteration
\[
T_{n+1}^p = \inf \left\lbrace t >T_n^p, \mid H(t) -H(T_n^p) \mid
\geq \frac{1}{p} \right\rbrace.
\]
We now define
\[
H^p(t)=\sum_{i=1}^{n-1} H(T_i^p) \mathbf{1}_{(T_i^p , T_{i+1}^p ]}
(t).
\]
For this sequence of processes $(H^p(t))_{t \geq 0}$, it easy to
show that
\[
\mathbb{E} \left( \sup_{t \geq 0} \left( \int_0^t (H^p(s) -H^q(s)
)dM(s) \right)^2 \right) \rightarrow_{p,q \rightarrow +\infty} 0.
\]
From this, we can deduce that there exists a continuous martingale,
denoted by $\left(\int_0^t H(s) dM(s) \right)_{t \geq 0}$, such that
\[
\int_0^t H(s) dM(s) =\lim_{p\rightarrow +\infty} \int_0^t H^p(s) dM(s)
\]
uniformly for $t$ on compact sets. We furthermore have
\[
\mathbb{E} \left( \sup_{t \geq 0} \left( \int_0^t H(s) dM(s)
\right)^2 \right) \leq 4
\parallel H \parallel_{\infty}^2 \mathbb{E} (M(\infty)^2).
\]
and
\[
\mathbb{E} \left( \left( \int_0^t H(s) dM(s) \right)^2
\right)=\mathbb{E} \left( \int_0^t H(s)^2 d \langle M \rangle_s
\right).
\]
Now, by localization, it is not difficult to extend naturally the
definition of $\int_0^t H(s) dM(s)$ in the general case where:
\begin{enumerate}
\item $(M(t))_{t \geq 0}$ is a semimartingale,
that is, $(M(t))_{t \geq 0}$ can be written as
\[
M(t)=A(t)+N(t),
\]
where $(A(t))_{t \geq 0}$ is a bounded variation process and
$(N(t))_{t \geq 0}$ is a local martingale;
\item $(H(t))_{t \geq 0}$ is a locally bounded predictable process.
\end{enumerate}
In this setting, we have:
\[
\int_0^t H(s) dM(s) =\int_0^t H(s) dA(s)+\int_0^t H(s) dN(s).
\]
Observe that, since $(A(t))_{t \geq 0}$ is a bounded variation
process, the integral $\int_0^t H(s) dA(s)$ is simply a
Riemann-Stieltjes integral. The process $\left( \int_0^t H(s) dN(s)
\right)_{t \geq 0}$ is a local martingale.

The class of semimartingales appears then as
a good class of integrators in the theory of stochastic
integration. It can be shown that this is actually the widest
possible class if we wish to obtain a \textit{natural} integration
theory (see Protter \cite{Pro}
for a precise statement). The decomposition of a semimartingale
$(M(t))_{t \geq 0}$ under the form
\[
M(t)=A(t)+N(t),
\]
is essentially unique under the condition $N(0) =0$. The process
$(A(t))_{t \geq 0}$ is called the bounded variation part of
$(M(t))_{t \geq 0}$. The process $(N(t))_{t \geq 0}$ is called the
local martingale part of $(M(t))_{t \geq 0}$. If $(M^1(t))_{t \geq
0}$ and $(M^2(t))_{t \geq 0}$ are two semimartingales, then we
define the quadratic covariation $\left( \langle M^1,M^2 \rangle_t
\right)_{t \geq 0}$ of $(M^1(t))_{t \geq 0}$ and $(M^2(t))_{t \geq
0}$ as being $\left( \langle N^1,N^2 \rangle_t \right)_{t \geq 0}$
where $(N^1(t))_{t \geq 0}$ and $(N^2(t))_{t \geq 0}$ are the local
martingale parts.

Throughout this monograph we shall only deal with continuous processes
so that we shall almost always omit to precise the
continuity of the processes which will be considered. 

%
\section*{It\^o's formula}
The It\^o's formula\index{It\^o's formula} is certainly the most
important formula of stochastic calculus.
\begin{theorem}
Let $(X(t))_{t \geq 0}=\left( X_1(t) , \dots , X_n(t) \right)_{t
\geq 0}$ be a $n$- dimensional continuous semimartingale. Let now
 $f:\mathbb{R}^n \rightarrow \mathbb{R}$ be a $C^2$
function. We have
\begin{align*}
f(X(t)) & =f(X_0)+\sum_{i=1}^n \int_0^t \frac{\partial f}{\partial
x_i} (X(s)) dX_i(s) +\frac{1}{2} \sum_{i,j=1}^n \int_0^t
\frac{\partial^2 f}{\partial x_i \partial x_j} (X(s)) d\langle X_i,
X_j \rangle(s).
\end{align*}
\end{theorem}

\section*{Girsanov's theorem}
We give here the most important theorem concerning the impact of a
change of probability measure on the class of semimartingales. Let
${C} ( \mathbb{R}_{\geq 0} , \mathbb{R}^d)$ denote the
space of continuous functions $\mathbb{R}_{\geq 0} \rightarrow
\mathbb{R}^d $. Let $X(s)$, $s \geq 0$, denote the coordinate
mappings, that is
\[
X(s) (\omega) = \omega (s), \text{ }\omega \in {C} (
\mathbb{R}_{\geq 0} , \mathbb{R}^d).
\]
Set now $\mathcal{B}^0_t =\sigma( X(s) ,s\leq t)$ and consider on
${C} ( \mathbb{R}_{\geq 0} , \mathbb{R}^d)$ the Wiener
measure $\mathbb{P}$, that is, the law of a $d$-dimensional
standard Brownian motion. Let  $\mathcal{B}_t$ be the usual
$\mathbb{P}$-augmentation of $\mathcal{B}^0_t$ and
\[
\mathcal{B}_{\infty}=\vee_{t \geq 0} \mathcal{B}_t.
\]
The filtered probability space
\[
\left( {C} ( \mathbb{R}_{\geq 0} , \mathbb{R}^d),
(\mathcal{B}_t)_{t \geq 0}, \mathcal{B}_{\infty}, \mathbb{P}
\right)
\]
is called the Wiener space\index{Wiener space}. Observe that, by
definition of $\mathbb{P}$, the process $(X(t))_{t \geq 0}$ under $\mathbb{P}$ is
 a $d$-dimensional standard Brownian motion. We
have the following theorem, referred to as the Girsanov's
theorem\index{Girsanov's theorem}.
\begin{theorem}
\begin{enumerate}
\item Let $\mathbb{Q}$ be a probability measure on $\left( {C} ( \mathbb{R}_{\geq 0} , \mathbb{R}^d),
\mathcal{B}_{\infty} \right)$ which is equivalent to $\mathbb{P}$.
Let us denote by $D$ the density of $\mathbb{Q}$ with respect to
$\mathbb{P}$. Then there exists an adapted continuous
$\mathbb{R}^d$-valued process $\left( \Theta(t) \right)_{t \geq 0}$
such that
\[
\mathbb{E} \left( D \mid \mathcal{B}(t) \right)=\exp \left(
\int_0^t \Theta(s) dX(s) - \frac{1}{2} \int_0^t | \Theta(s)
|^2 ds \right),
\]
and under $\mathbb{Q}$,
\[
\tilde{X}(t):=X(t) - \int_0^t \Theta(s) ds
\]
is a standard Brownian motion.
\item Let $\left( \Theta(t) \right)_{t \geq 0}$ be an adapted
continuous $\mathbb{R}^d$-valued process such that the process
\[
Z(t)=\exp \left( \int_0^t \Theta(s) dX(s) - \frac{1}{2} \int_0^t
 |\Theta(s) |^2 ds \right), \text{ }t \geq 0,
\]
is a uniformly integrable martingale under $\mathbb{P}$. Define a
probability measure on $\left( {C} ( \mathbb{R}_{\geq 0} ,
\mathbb{R}^d), \mathcal{B}_{\infty} \right)$ by:
\[
d\mathbb{Q} = Z_{\infty} d\mathbb{P}.
\]
Then, under $\mathbb{Q}$, the process
\[
X(t) - \int_0^t \Theta(s) ds
\]
is a standard Brownian motion.
\end{enumerate}
\end{theorem}
Observe that the process
\[
Z(t)=\exp \left( \int_0^t \Theta(s) dX(s) - \frac{1}{2} \int_0^t
| \Theta(s) |^2 ds \right), \text{ }t \geq 0,
\]
satisfies the following equation
\[
Z(t) =1+\int_0^t Z(s) \Theta(s) dX(s).
\]
It can be shown that sufficient conditions that ensure $(Z(t))_{t \geq 0}$ to be  a uniformly integrable martingale are the
following:
\begin{enumerate}
\item For any $t \geq 0$,
\[
\mathbb{E} (Z(t))=1;
\]
\item For any $t \geq 0$,
\[
\mathbb{E} \left( \exp \left( \frac{1}{2} \int_0^t
|\Theta(s)|^2 ds \right)\right)< + \infty.
\]
\end{enumerate}

\section*{Stochastic differential equations and diffusion processes}

\begin{definition}
A transition function\index{Transition function} $\{P_t,t \geq 0 \}$ on $\mathbb{R}^n$  is a family of kernels
\[
P_t : \mathbb{R} \times \mathcal{B}
(\mathbb{R}^n) \rightarrow [0,1]
\]
such that:
\begin{enumerate}
\item For $t \geq 0$ and $x \in \mathbb{R}^n$, $P_t (x,\cdot)$ is a probability measure on $\mathbb{R}^n$;
\item For $t \geq 0$ and any  Borel set $A$  in $\mathbb{R}^n$, the application
 $x \rightarrow P_t (x,A)$ is measurable;
\item For $s,t \geq 0$, $x \in \mathbb{R}^n $ and  a  Borel set $A$ in $\mathbb{R}^n$,
\begin{equation}\label{CK}
 P_{t+s} (x,A)=\int_{\mathbb{R}^n} P_t(y,A) P_s (x,dy).
\end{equation}
\end{enumerate}
\end{definition}
The relation (\ref{CK}) is called the Chapman-Kolmogorov relation\index{Chapman-Kolmogorov relation}. Given a transition function, we can define a one parameter family of linear operators $(P_t)_{t \geq 0}$ from the space of bounded Borel functions into itself as follows:
\[
(P_t f)(x)=\int_{\mathbb{R}} f(y) P_t (x,dy).
\]

If  $\{P_t,t \geq 0 \}$ is a transition function, then the following properties are satisfied:
\begin{itemize}
\item  \hspace{.1in} $P_t 1=1$;
\item  \hspace{.1in} For every $t\ge 0$, $P_t$ is a positivity preserving operator, in the sense that if $f$ is non negative, so is  $P_t f$;
\item  \hspace{.1in} For every $t\ge 0$, $P_t$ is a contraction from the space of bounded Borel functions into itself (that is, it is a continuous operator with a norm smaller than 1);
\item  \hspace{.1in} The semigroup property holds: For every $s,t \ge 0$, 
\[
P_{t+s} =P_t P_s, s,t \ge 0.
\]
\end{itemize}

\begin{definition}\label{definition Markov}
A stochastic process $(X(t))_{t \ge 0}$ defined on a probability space $(\Omega, \mathcal{F}, \mathbb{P})$ is called  a  Markov process if there exists a transition function  $\{P_t,t \geq 0 \}$ on $\mathbb{R}^n$ such that for every bounded and Borel function $f:\mathbb{R}^n \rightarrow \mathbb{R}$,
\[
\mathbb{E} \left( f(X(t+s)) \mid \mathcal{F}_s^X \right)=P_t f (X(s)), \quad s ,t \ge 0,
\]
where $\mathcal{F}^X$ denotes the natural filtration\footnote{This means that $\mathcal{F}_s^X$ is the smallest $\sigma$-algebra  that makes measurable all the random variables $(X(t_1),\dots,X(t_m))$, $0\le t_1 \le \cdots \le t_m \le s$.} of the process $(X(t))_{t \ge 0}$. The family of operators $(P_t)_{t \ge 0}$ is called the semigroup of the Markov process. 
\end{definition}

\begin{definition}\label{D:do}
A differential operator $L$ on $\mathbb{R}^n$,  is called a diffusion operator if it can be written as
\[
L=\sum_{i,j=1}^n \sigma_{ij} (x) \frac{\partial^2}{ \partial x_i \partial x_j} +\sum_{i=1}^n b_i (x)\frac{\partial}{\partial x_i},
\]
where $b_i$ and $\sigma_{ij}$ are continuous functions on $\mathbb{R}^n$, and if for every $x \in \mathbb{R}^n$ the matrix $\sigma(x) \overset{\mbox{def}}{=} (\sigma_{ij}(x))_{1\le i,j\le n}$ is symmetric and nonnegative. 
\end{definition}

\begin{definition}
Let $(X(t))_{t \ge 0}$ be a continuous Markov process and $L$ be a diffusion operator. The process $(X(t))_{t \ge 0}$ is called a diffusion process with generator $L$ if for every smooth function $f$ the process
\[
f(X(t))-\int_0^t Lf (X(s)) ds
\]
is a local martingale.
\end{definition}

The easiest way to construct diffusions is to solve stochastic differential equations.

\begin{theorem}
Let $\sigma=(\sigma_{ij})_{1 \le i , j \le n}$ be a matrix of Lipschitz functions $\sigma_{ij}: \mathbb{R} \to \mathbb{R}$ and $b: \mathbb{R}^n \to \mathbb{R}^n$ be a Lipschitz function.  Let $(B(t))_{t \ge 0}$ be a Brownian motion in $\mathbb{R}^n$. 
Given, $x_0 \in \mathbb{R}^n$, there exists a unique (up to modification) stochastic process $(X(t))_{t \ge 0}$ satisfying
\[
X(t)=x_0+\int_0^t b(X(s)) ds +\int_0^t \sigma(X(s)) dB(s).
\]
Moreover, this process $(X(t))_{t \ge 0}$ is a diffusion process with generator
\[
L=\frac{1}{2}\sum_{i,j=1}^n a_{ij} (x) \frac{\partial^2}{ \partial x_i \partial x_j} +\sum_{i=1}^n b_i (x)\frac{\partial}{\partial x_i},
\]
where $a=\sigma \sigma^*$.
\end{theorem}

If we only assume that $\sigma$ and $b$ are locally Lipschitz, then the process $(X(t))_{t \ge 0}$  exists and is a diffusion process  only up to a possibly finite explosion time, see for instance Chapter 1 in \cite{Hsu}.

\section*{Stratonovich  stochastic differential equations and the language of vector fields}\label{section strato}

 It is often useful to use the language of Stratonovich's integration to study stochastic differential equations because the It\^o's formula takes a  better and intrinsic  form within this context. If $(N(s))_{0 \leq s \leq t}$, $t>0$, is an $\mathcal{F}$-adapted real valued local martingale and if $(\Theta(s))_{0 \leq s \leq t}$ is an $\mathcal{F}$-adapted continuous semimartingale  then  the Stratonovich integral\index{Stratonovich integral} of $(\Theta(s))_{0 \leq s \leq t}$ with respect to $(N(s))_{s \ge 0}$ is given by
\[
\int_0^t \Theta(s) \circ d  N(s) =\int_0^t \Theta(t) dN(s)
+\frac{1}{2} \langle \Theta, N \rangle_t,
\]
where:
\begin{enumerate}
\item $\int_0^t \Theta(s) d  N(s)$ is the It\^o integral of $(\Theta(s))_{0 \leq s \leq t}$ against $(N(s))_{0 \leq s \leq t}$;
\item $ \langle \Theta, N \rangle_t$ is the quadratic covariation at time $t$ between $(\Theta(s))_{0 \leq s \leq t}$ and $(N(s))_{0 \leq s \leq t}$.
\end{enumerate}

By using Stratonovich integral instead of It\^o's,  we can see that the D\"oblin-It\^o formula reduces to the classical change of variable formula.
\begin{theorem}[It\^o's formula]
Let $(X(t))_{t \geq 0}=\left( X_1(t) , \dots , X_n(t) \right)_{t \geq 0}$ be a $n$- dimensional continuous semimartingale. Let now
 $f:\mathbb{R}^n \rightarrow \mathbb{R}$ be a $C^2$ function. We have
\begin{align*}
f(X(t))  =f(X_0)+\sum_{i=1}^n \int_0^t \frac{\partial f}{\partial x_i} (X(s)) \circ dX_i(s), \quad t \ge 0.
\end{align*}
\end{theorem}

Let $\mathcal{O} \subset \mathbb{R}^n$ be a non empty open set. A smooth vector field\index{Vector field} $V$ on $\mathcal{O}$ is  a smooth map
\[
\begin{array}{llll}
V: & \mathcal{O} & \rightarrow  & \mathbb{R}^{n} \\
& x & \rightarrow  & (v_{1}(x),...,v_{n}(x)).
\end{array}
\]
The vector field $V$ defines a differential operator acting on  smooth functions $f: \mathcal{O} \rightarrow \mathbb{R}$ as
follows:
\[
 Vf(x)=\sum_{i=1}^n v_i (x) \frac{\partial f}{\partial x_i}.
\]
We note that $V$ is a derivation, that is a map on ${C}^{\infty} (\mathcal{O} , \mathbb{R} )$, linear over $\mathbb{R}$, satisfying for $f,g \in {C}^{\infty} (\mathcal{O} , \mathbb{R} )$,
\[
V(fg)=(Vf)g +f (Vg).
\]
Conversely, any derivation on ${C}^{\infty} (\mathcal{O} , \mathbb{R} )$ is a vector field.

Let now 
\[
(B(t))_{t \geq 0}=(B_1(t),...,B_d(t))_{t \geq 0}
\]
be a $d$-dimensional Brownian motion and  consider $d+1$  vector fields $V_i : \mathbb{R}^n \rightarrow \mathbb{R}^n
$, $n \geq 1$, $i=0,...,d$. By using the language of vector fields and Stratonovich integrals, the fundamental theorem for the existence and the uniqueness of solutions for stochastic
differential equations is the following:
\begin{theorem}
Assume that  $V_0,V_1,\dots,V_d$ are $C^2$ vector fields with bounded derivatives up to order 2. Let $x_0 \in \mathbb{R}^n$. On $\left( \Omega , (\mathcal{F}_t)_{t \geq 0} , \mathbb{P} \right)$, there exists a unique continuous and adapted process $(X(t))_{t \geq 0}$  such that for $t \geq 0$,
\begin{equation}
X(t)=x_0 + \sum_{i=0}^d \int_0^t V_i (X(s)) \circ dB_i(s),
\end{equation}
with the notation that $B_0(t)=t$.
\end{theorem}
Thanks to It\^o's formula, the  It\^o's formulation of a Stratonovich equation is
\begin{equation*}
X(t) =x_0 + \frac{1}{2} \sum_{i=1}^d \int_0^t \nabla_{V_i}
V_i (X(s)) ds +\sum_{i=0}^d \int_0^t V_i (X(s)) dB_i(s),
\end{equation*}
where for $1 \leq i \leq d$, $\nabla_{V_i} V_i$ is the vector field given by
\[
\nabla_{V_i} V_i (x)=\sum_{j=1}^n \left( \sum_{k=1}^n v^i_k (x)
\frac{\partial v_j^i}{\partial x_k}(x)\right)\frac{\partial}{\partial x_j}, \text{ }x \in
\mathbb{R}^n.
\]
If $f:\mathbb{R}^n \rightarrow \mathbb{R}$ is a $C^2$ function, from It\^o's formula, we have for $t \geq 0$,
\[
f(X(t))=f(x_0) + \sum_{i=0}^d \int_0^t (V_i f) (X(s)) \circ dB_i(s),
\]
and  the process
\[
\left( f(X(t))-\int_0^t (Lf)(X(s))ds
\right)_{t \geq 0}
\]
is a local  martingale where $L$ is the second order differential operator
\[
L = V_0+\frac{1}{2} \sum_{i=1}^d V_i^2.
\]

\section*{H\"ormander's theorem}

Consider the following stochastic differential equation
\begin{equation}
\label{SDE18} X(t) =x_0 +  \int_0^tV_0(X(s))ds+\sum_{i=1}^n \int_0^t V_i
(X(s)) \circ dB_i(s), \text{ } t \geq 0,
\end{equation}
where $x_0 \in \mathbb{R}^n$, $V_0,V_1,\cdots,V_n$ are $C^{\infty}$
bounded vector fields on $\mathbb{R}^n$ and $(B(t))_{t \geq 0}$ is
a $n$-dimensional standard Brownian motion.

As we mentioned previously, $(X^{x_0}(t))_{t \ge 0}$ is a diffusion process with semigroup

\begin{equation}
\label{representation18} \mathbb{E} \left( f(X^{x_0}(t))
\right)=\left( P_tf \right)(x_0),
\end{equation}
and infinitesimal generator
\[
L=V_0+\frac{1}{2} \sum_{i=1}^n V_i^2.
\]
We are interested in the existence of smooth
densities for the random variables $X^{x_0}(t)$, $t>0$, $x_0 \in
\mathbb{R}^n$. According to formula \eqref{representation18},
this question is therefore equivalent to the question of the
existence of a smooth transition kernel with respect to the
Lebesgue measure for the operators $P_t$. 

Let us
recall the following definition which comes from functional
analysis.
\begin{definition}
A differential operator $\mathcal{G}$ defined on an open set
$\mathcal{O} \subset \mathbb{R}^n$ is called hypoelliptic if,
whenever $u$ is a distribution on $\mathcal{O}$, $u$ is a smooth
function on any open set $\mathcal{O}' \subset \mathcal{O}$ on
which $\mathcal{G} u$ is smooth.
\end{definition}
It is possible to show that the existence of a smooth transition
kernel with respect to the Lebesgue measure for $P_t$ is equivalent to the hypoellipticity of $L$. Therefore, our initial question about the existence
of smooth densities for the random variables $X^{x_0}(t)$, $t>0$,
$x_0 \in \mathbb{R}^n$ is equivalent  to the
hypoellipticity of $L$.
One of the most useful results  to prove hypoellipticity is H\"ormander's theorem. It is stated using the notion of Lie brackets of vector fields that we recall below.

We have already mentioned that a vector field $V$ may be seen as derivation, that is a map on
${C}^{\infty} (\mathcal{O} , \mathbb{R} )$, linear over
$\mathbb{R}$, satisfying for $f,g \in {C}^{\infty}
(\mathcal{O} , \mathbb{R} )$,
\[
V(fg)=(Vf)g +f (Vg).
\]
And, conversely, any derivation on ${C}^{\infty} (\mathcal{O} , \mathbb{R} )$ is a vector field. If $V'$ is another
smooth vector field on $\mathcal{O}$, then it is easily seen that
the operator $VV'-V'V$ is a derivation. It therefore defines a
smooth vector field on $\mathcal{O}$ which is called the Lie
bracket\index{Lie bracket of vector fields} of $V$ and $V'$ and
denoted $[V,V']$. A straightforward computation shows that for $x
\in \mathcal{O}$,
\[
[ V, V' ](x)=\sum_{i=1}^n \left( \sum_{j=1}^n v_j (x)
\frac{\partial v'_i}{\partial x_j}(x)- v'_j (x) \frac{\partial
v_i}{\partial x_j}(x)\right)\frac{\partial}{\partial x_i}.
\]
Observe that the Lie bracket obviously satisfies  $[V,V']=-[V',V]$
and the so-called Jacobi identity, that is:
\[
[V,[V',V'']]+[V',[V'',V]]+[V'',[V,V']]=0.
\]

Given the set of vector fields $V_1,\dots,V_n$ we define the Lie algebra $\mathfrak L$  generated by those vector fields as the smallest linear space of vector fields that contains $V_1,\dots,V_n$ and which is closed under Lie brackets.

\begin{theorem}[H\"ormander theorem]\label{Hormander theorem}
Assume that for every $x_0 \in \mathbb{R}^n$,
\[
\left\{ V (x_0) , V \in \mathfrak L\right\} =\mathbb{R}^n,
\]
then the operator $L$ is hypoelliptic.
\end{theorem}

The original proof by H\"ormander was rather complicated and has since then
been considerably simplified  by using the theory of
pseudo-differential operators. The probabilistic counterpart of
the theorem, namely the existence of a smooth density for the
random variable $X^{x_0}(t)$, $t>0$, has first been pointed out by
Malliavin. In order to reprove the theorem under weaker
assumptions, the author has developed a stochastic calculus of
variations which is now known as the Malliavin calculus, see \cite{Nualart}.

\section*{Additive functionals of diffusion processes and Yor transforms}

In this section we present the general idea of the powerful method due to Marc Yor to explicitly compute the characteristic functions of additive functionals of diffusion processes. This  method is extensively used in the present monograph. Rather than focusing on the exact assumptions under which it can be performed, we prefer to show the structure of the argument.

Let
\[
L=\frac{1}{2} \sum_{i,j=1}^n a_{ij}(x) \frac{\partial^2}{\partial x_i \partial x_j} +\sum_{i=1}^n b_i(x) \frac{\partial}{\partial x_i}
\]
be a diffusion operator in $\mathbb R^n$. Under suitable assumptions this operator is the infinitesimal generator of the diffusion process $(X(t))_{t \ge 0}$ that solves the It\^o's stochastic differential equation
\[
dX(t)=b(X(t))dt +\sigma (X(t))dB(t)
\]
where $(B(t))_{t \ge 0}$ is a $n$-dimensional Brownian motion and $\sigma$ satisfies $a=\sigma \sigma^*$.

For a given function $f:\mathbb{R}^n \to \mathbb{R}$ assume that we are interested to compute 
\[
\mathbb{E}_x\left[ \exp \left(- \int_0^t f(X(s)) ds\right) \mid X(t)=y \right].
\]
The first step is to find a function $g$  that solves the non linear partial differential equation
\[
Lg+\frac{1}{2} \Gamma (g,g)=f,
\]
where $\Gamma$ is the carr\'e du champ operator associated with $L$  defined by:
\[
\Gamma (f_1,f_2)= \sum_{i,j=1}^n a_{ij}(x) \frac{\partial f_1}{\partial x_i} \frac{\partial f_2}{\partial x_j}. 
\]

Then, It\^o calculus shows that the process
\[
M^f_t=\exp \left( g(X(t))- g(X_0) -\int_0^t f(X(s)) ds \right)
\]
is a local martingale. Let us assume that it is a uniformly integrable martingale. We can then consider a new probability measure defined by
\[
\mathbb{P}^f_{/ \mathcal F_t}= M^f_t \mathbb{P}_{/ \mathcal F_t}
\]
where $\mathcal F_t$ is the natural filtration of the underlying Brownian motion $(B(t))_{t \ge 0}$. From Girsanov's theorem, under the new probability measure $\mathbb{P}^f$, $(X(t))_{t \ge 0}$ is then a diffusion process with generator:
\[
L+\Gamma(g,\cdot).
\]
Denote then by $P_t^f(x,dy)$ the transition kernel of this diffusion and by $P_t(x,dy)$ the transition kernel of the original diffusion $(X(t))_{t \ge 0}$ .

We have then the following result: 

\begin{theorem}\label{Yor transform section}
\[
\mathbb{E}_x\left[ \exp \left(- \int_0^t f(X(s)) ds\right) \mid X(t)=y \right]= \exp(g(x)-g(y)) \, \frac{P_t^f(x,dy)}{P_t(x,dy)}
\]
\end{theorem}

\begin{proof}
Denote by $\mathbb{E}^f$ the expectation under the probability $\mathbb{P}^f$. Let $h$ be a bounded Borel function. On one hand
\begin{align*}
\mathbb{E}_x\left[ h(X(t)) \exp \left(- \int_0^t f(X(s)) ds\right)\right] &=\mathbb{E}^f_x\left[ \frac{1}{M_t^f} h(X(t)) \exp \left(- \int_0^t f(X(s)) ds\right)\right] \\
&=\mathbb{E}^f_x\left[\exp(g(x)-g(X(t))) h(X(t))\right]  \\
&=\int_{\mathbb R^n} \exp(g(x)-g(y)) h(y) P_t^f(x,dy).
\end{align*}
On the other hand
\begin{align*}
&\mathbb{E}_x\left[ h(X(t)) \exp \left(- \int_0^t f(X(s)) ds\right)\right] \\
&\quad\quad=\int_{\mathbb R^n} \mathbb{E}_x\left[ \exp \left(- \int_0^t f(X(s)) ds\right) \mid X(t)=y \right] h(y) P_t(x,dy).
\end{align*}
Since the above equality holds for any function $h$, the conclusion follows.
\end{proof}

Of course the above theorem is only useful if we are able to compute explicitly a function $g$ such that
\begin{align}\label{Burgers}
Lg+\frac{1}{2} \Gamma (g,g)=f
\end{align}
and such that the transition kernel of the diffusion operator

\[
L^f=L+\Gamma(g,\cdot)
\]
 
 is also explicit. A common way to solve the Burgers type equation \eqref{Burgers} is to linearize it by means of the Hopf-Cole transform $g=\ln h$. The equation \eqref{Burgers} is transformed into the more tractable linear equation
 \[
 Lh = f h
 \]
 and the operator $L^f$ is transformed into $L+\frac{1}{h}\Gamma(h,\cdot)$.
 
 \begin{example}\label{Hartman-Watson}
 For $d \ge 2$, consider the generator
\[
L=\frac{1}{2} \frac{\partial^2}{\partial r^2}+\frac{d-1}{2r} \frac{\partial}{\partial r},
\]
which is associated with the so-called Bessel process of dimension $d$, see Section \ref{Bessel section}. Let $(R(t))_{t \ge 0}$ be the associated process. We want to compute
\[
\mathbb{E}_x\left[ \exp \left(- \int_0^t f(R(s)) ds\right) \mid R(t)=y \right]
\]
where $f(r)=\frac{\lambda}{r^2}$ with $\lambda >0$. For this, we consider the equation
\[
L h = \frac{\lambda}{r^2} h.
\]
A particular solution is given by $h(r)=r^{\alpha}$ where
\[
\alpha= \frac{-(d-2)+\sqrt{(d-2)^2+8\lambda}}{2}.
\]
The operator $L^f$ is then given by
\[
L^f=\frac{1}{2} \frac{\partial^2}{\partial r^2}+\frac{d-1+2\alpha}{2r} \frac{\partial}{\partial r},
\]
which is the generator of a Bessel diffusion with dimension $2+\sqrt{(d-2)^2+8\lambda}$. Using the formulas for the densities of Bessel diffusions, see Section \ref{Bessel section}, one concludes
\[
\mathbb{E}_x\left[ \exp \left(- \lambda\int_0^t  \frac{ds}{R(s)^2}\right) \mid R(t)=y \right]=\frac{ I_{\sqrt{\nu^2+2\lambda}}\left( \frac{xy}{t} \right)}{I_{\nu}\left( \frac{xy}{t} \right)},
\]
where $\nu=\frac{d}{2}-1$.
 \end{example}

\chapter{Appendix 2: Special diffusions}\label{sec-appendix-2}

\section{Bessel diffusions}\label{Bessel section}

The Bessel diffusion  with dimension $d$ is the diffusion on $[0, +\infty)$ with generator
\[
\mathcal{L}^{d}=\frac{1}{2} \frac{\partial^2}{\partial r^2}+\frac{d-1}{2r} \frac{\partial}{\partial r}, \quad d > 0
\]
defined up to the first time it hits the boundary $\{ 0 \}$.

The point $0$ is:
\begin{itemize}
\item  \hspace{.1in} A regular point for $0 <d <1 $;
\item  \hspace{.1in} An entrance point for $d \ge 1$.
\end{itemize}

If $d \ge 2$, the Bessel diffusion  with dimension $d$ never hits $0$ and its transition density with respect to the Lebesgue measure is given by:
\begin{align*}
p_t(x,y)=
 \begin{cases}
\displaystyle 2^{-\nu} t^{-(\nu+1)} \Gamma (\nu +1)^{-1} y^{2\nu+1} \exp \left(-\frac{y^2}{2t} \right), x=0, y>0 \\
\displaystyle
\frac{1}{t} \left( \frac{y}{x}\right)^\nu y \exp \left(-\frac{x^2+y^2}{2t} \right) I_\nu \left( \frac{ xy }{t} \right), x>0, y>0,
\end{cases}
\end{align*}
where $\nu=\frac{d}{2}-1$ and 
\[
I_{\nu} (x)=\sum_{m=0}^{+\infty} \frac{1}{m! \Gamma(m+\nu+1)} \left( \frac{x}{2} \right)^{2m+\nu}
\]

    is the modified Bessel function of the first kind of index $\nu$. The parameter $\nu$ is sometimes called the index of the Bessel diffusion.

\section{Spherical Jacobi diffusions and heat kernels on spheres}\label{Jacobi diffusions}

The Jacobi diffusion is the diffusion on $[0,\pi/2]$ with generator
\[
\mathcal{L}^{\alpha,\beta}=\frac{1}{2} \frac{\partial^2}{\partial r^2}+\left(\left(\alpha+\frac{1}{2}\right)\cot r-\left(\beta+\frac{1}{2}\right) \tan r\right)\frac{\partial}{\partial r}, \quad \alpha,\beta >-1
\]
defined up to the first time it hits the boundary $\{ 0, \pi /2 \}$.

The point $0$ is:
\begin{itemize}
\item  \hspace{.1in}  A regular point for $-1<\alpha <0$;
\item \hspace{.1in}  An entrance point for $\alpha \ge 0$.
\end{itemize}

Similarly, the point $\pi /2$ is:
\begin{itemize}
\item  \hspace{.1in} A regular point for $-1<\beta <0$;
\item \hspace{.1in}  An entrance point for $\beta \ge 0$.
\end{itemize}

If $r$ is a Jacobi diffusion with generator $\mathcal{L}^{\alpha,\beta}$, then it is easily seen that $\rho=\cos 2r$ is a diffusion with generator $2\mathcal{G}^{\alpha,\beta}$ where,

\begin{equation}\label{eq-G-ab}
\mathcal{G}^{\alpha,\beta}=(1-\rho^2)\frac{\partial^2}{\partial \rho^2}-\left( (\alpha+\beta+2)\rho +\alpha -\beta \right)\frac{\partial}{\partial \rho}.
\end{equation}

The spectrum and eigenfunctions of $\mathcal{G}^{\alpha,\beta}$ are known. Let us denote by $P_m^{\alpha,\beta}(x)$, $m\in \mathbb{Z}_{\ge0}$ the  Jacobi polynomials given by
\[
P_m^{\alpha,\beta}(x)=\frac{(-1)^m}{2^mm!(1-x)^{\alpha}(1+x)^\beta}\frac{d^m}{dx^m}((1-x)^{\alpha+m}(1+x)^{\beta+m}).
\]
It is known that $\{P_m^{\alpha,\beta}(x)\}_{m\ge0}$ is orthonormal in $L^2([-1,1], 2^{-\alpha-\beta-1}(1+x)^{\beta}(1-x)^{\alpha}dx)$ and satisfies 
\[
\mathcal{G}^{\alpha,\beta}P_m^{\alpha,\beta}(x)=-m(m+\alpha+\beta+1)P_m^{\alpha,\beta}(x).
\]
If we denote by $p^{\alpha,\beta}_t(x,y)$ the transition density with respect to the Lebesgue measure of the diffusion $\rho$ starting from $x \in (-1,1)$, then we have

\begin{align*}
  p^{\alpha,\beta}_t(x,y)
 =2^{-\alpha-\beta-1}(1+y)^{\beta}(1-y)^{\alpha}\sum_{m=0}^{+\infty}&\bigg[ (2m+\alpha+\beta+1)\frac{\Gamma(m+\alpha+\beta+1)\Gamma(m+1)}{\Gamma(m+\alpha+1)\Gamma(m+\beta+1)}  \\
& \cdot  e^{-2m(m+\alpha+\beta+1)t}P_m^{\alpha,\beta}(x)P_m^{\alpha,\beta}(y)\bigg].
\end{align*}

In particular, when 1 is an entrance point, that is $\alpha \ge 0$,  we obtain

\begin{align*}
  p^{\alpha,\beta}_t(1,y)=2^{-\alpha-\beta-1}(1+y)^{\beta}(1-y)^{\alpha} \sum_{m=0}^{+\infty} \bigg[ 
  &(2m+\alpha+\beta+1)\frac{\Gamma(m+\alpha+\beta+1)}{\Gamma(m+\beta+1)\Gamma(\alpha+1)} \\
  &\cdot e^{-2m(m+\alpha+\beta+1)t}P_m^{\alpha,\beta}(y)\bigg].
\end{align*}

Denoting $q_t^{\alpha,\beta}$ the transition density of $r$, we then obtain for $\alpha,\beta \ge 0$,

\begin{align*}
  q^{\alpha,\beta}_t(r_0,r)=2(\cos r)^{2\beta +1} (\sin r)^{2\alpha +1}\sum_{m=0}^{+\infty} \bigg[ 
 &(2m+\alpha+\beta+1)\frac{\Gamma(m+\alpha+\beta+1)\Gamma(m+1)}{\Gamma(m+\alpha+1)\Gamma(m+\beta+1)} \\ 
 &\cdot e^{-2m(m+\alpha+\beta+1)t} P_m^{\alpha,\beta}(\cos 2r_0)P_m^{\alpha,\beta}(\cos 2r)\bigg],
\end{align*}

and

\begin{align}\label{eq-qt-ab-0}
  q^{\alpha,\beta}_t(0,r) 
 =2(\cos r)^{2\beta +1} (\sin r)^{2\alpha +1} 
 & \sum_{m=0}^{+\infty} \bigg[ (2m+\alpha+\beta+1)\frac{\Gamma(m+\alpha+\beta+1)}{\Gamma(m+\beta+1)\Gamma(\alpha+1)} \notag \\
& \cdot e^{-2m(m+\alpha+\beta+1)t}
 P_m^{\alpha,\beta}(\cos 2r)\bigg].
\end{align}

Spherical Jacobi diffusions can be used to compute the heat kernel on
spheres. Indeed, the radial part of the Laplace-Beltrami operator on the $n-1$ dimensional Euclidean sphere is given by the diffusion operator 
\[
\frac{d^2}{d \rho^2}+(n-2)\cot \rho \frac{d}{d\rho}
\]
 which, after the change of variable $\rho=2r$, is a Jacobi operator with parameters $\alpha=\beta=\frac{n-3}{2}$.
 
 In particular, if $p(t,x,y)$ denotes the heat kernel on $\mathbb{S}^{n-1}$ then we have for $t>0$, $x,y \in \mathbb{S}^{n-1}$,
\[
p(t,x,y)=\frac{\Gamma \left( \frac{n}{2}-1\right)}{ 4\pi^{\frac{n}{2}} }\sum_{m \ge 0} \left(2m+n-2 \right)e^{-m(m+n-2)t } C_m^{\frac{n}{2}-1}(\langle x, y \rangle),
\]

where the Gegenbauer polynomials $C_m^\nu$, $\nu >\frac{1}{2}$, $m \in \mathbb{N}$  are defined through the Rodrigues' formula:
\[
C_m^\nu (t)=\frac{(-1)^m}{2^m} \frac{\Gamma(m+2\nu) \Gamma\left( \nu+ \frac{1}{2} \right)}{\Gamma(2\nu)\Gamma(m+1)\Gamma\left(\nu+m+\frac{1}{2} \right)} \frac{1}{(1-t^2)^{\nu-\frac{1}{2}} } \left( \frac{d}{dt} \right)^m (1-t^2)^{\nu+m-\frac{1}{2}}.
\]
The first few Gegenbauer polynomials are
\begin{align*}
C_0^\nu (t)= &1 \\
C_1^\nu (t)= &2 \nu t \\
C_2^\nu (t)= &-\nu+2\nu(1+\nu)t^2.
\end{align*}

An induction on $m$ shows that $C_m^\nu$ is given by the formula
\[
C_m^\nu (t)=\sum_{l=0}^{[m/2]} (-1)^l \frac{\Gamma(m-l+\nu) }{\Gamma(\nu) \Gamma(m-2l+1) \Gamma(l+1) } (2t)^{m-2l}.
\]
This formula, together with the expansion formula
\[
\frac{1}{(1-z)^\nu}=\sum_{j=0}^{+\infty} \frac{\Gamma(j+\nu)}{\Gamma(j+1) \Gamma(\nu)} z^j,
\]
applied with $z=r(2t-r)$, gives the following generating function formula for the Gegenbauer polynomials:
\begin{align}\label{generating_gegenbauer}
\frac{1}{( 1-2t r +r^2)^\nu}=\sum_{k=0}^{+\infty} C_k^\nu (t) r^k, \quad |r| <1,  | t | \le 1.
\end{align}
We may observe that applying the above with $t=1$ leads to
\[
\frac{1}{( 1-r)^{2\nu}}=\sum_{k=0}^{+\infty} C_k^\nu (1) r^k,
\]
so that
\[
C_m^\nu(1)=\frac{\Gamma (m+2\nu)}{\Gamma( m+1) \Gamma(2 \nu) }.
\]

\begin{remark}
Some special cases of the Gegenbauer polynomials are worth pointing out explicitly:
\begin{itemize}
\item  \hspace{.1in} The Gegenbauer polynomials corresponding to $\nu=\frac{1}{2}$ are the so-called Legendre polynomials.
\item  \hspace{.1in} The Gegenbauer polynomials corresponding to $\nu=1$ are the so-called Chebyshev polynomials of the second kind and often are denoted by $U_m$. They usually are defined by the property that
\[
U_m(\cos \theta)=\frac{\sin (m+1) \theta}{\sin \theta}.
\] 
\end{itemize}
\end{remark}

Though the spectral expansion formula for the heat kernel  $p(t,x,y)$ is nice and perfectly captures the behavior of $p(t,x,y)$ when $ t \to \infty$, it is difficult to be used to derive the asymptotics  in small times, $t \to 0$. So, it may be interesting to find another formula for $p(t,x,y)$. For $\theta \in [0,\pi]$, let us denote
\begin{align}\label{radial_kernel_sphere}
h_{n-1}(t, \theta)=\frac{\Gamma \left( \frac{n}{2}-1\right)}{ 4\pi^{\frac{n}{2}} }\sum_{m \ge 0} \left(2m+n-2 \right)e^{-m(m+n-2)t } C_m^{\frac{n}{2}-1}(\cos \theta),
\end{align}
so that 
\[
p(t,x,y)=h_{n-1}(t,d(x,y)),
\]
here $d(x,y)=\arccos \langle x, y\rangle$ is the Riemannian distance between $x$ and $y$ on $\bS^{n-1}$.
Observe that we stressed the dependence in $n$ by using a subscript. For instance, we have
\[
h_1(t,\theta)=\frac{1}{\sqrt{4\pi t}} \sum_{k \in \mathbb{Z}} e^{-\frac{(\theta-2k\pi)^2}{4t} },
\]
which comes from the formula for the heat kernel on $\mathbb{S}^1$.
We have the following recursion formula:

\begin{proposition}\label{Recursion heat kernel}
\[
h_{n+2}(t,\theta)=-\frac{e^{nt}}{2\pi \sin \theta} \frac{\partial h_n}{\partial \theta}(t,\theta).
\]
\end{proposition}

\begin{proof}
We proceed with a direct computation. We have
\[
h_n(t, \theta)=\frac{\Gamma \left( \frac{n}{2}-1\right)}{ 4\pi^{\frac{n}{2}} }\sum_{m \ge 0} \left(2m+n-2 \right)e^{-m(m+n-2)t } C_m^{\frac{n}{2}-1}(\cos \theta),
\]
therefore
\[
-\frac{1}{\sin \theta} \frac{\partial h_n}{\partial \theta}(t,\theta)=\frac{\Gamma \left( \frac{n}{2}-1\right)}{ 4\pi^{\frac{n}{2}} }\sum_{m \ge 1} \left(2m+n-2 \right)e^{-m(m+n-2)t } \frac{dC_m^{\frac{n}{2}-1}}{dt} (\cos \theta).
\]
Now, as a consequence of (\ref{generating_gegenbauer}), it is easy to check that the following recursion formula holds
\[
\frac{dC_m^{\frac{n}{2}-1}}{dt} =(n-2)C_{m-1}^{\frac{n}{2}}.
\]
Therefore, we obtain
\begin{align*}
-\frac{1}{\sin \theta} \frac{\partial h_n}{\partial \theta}(t,\theta) & =\frac{\Gamma \left( \frac{n}{2}-1\right)}{ 4\pi^{\frac{n}{2}} }\sum_{m \ge 1} \left(2m+n-2 \right)e^{-m(m+n-2)t }2(n-2)C_{m-1}^{\frac{n}{2}} (\cos \theta) \\
 &=2\left(\frac{n}{2}-1\right)\frac{\Gamma \left( \frac{n}{2}-1\right)}{ 4\pi^{\frac{n}{2}} } \sum_{m \ge 0} \left(2m+n \right)e^{-(m+1)(m+n-1)t }C_{m-1}^{\frac{n}{2}} (\cos \theta).
\end{align*}
The claimed result then easily stems from the equality
\[
(m+1)(m+n-1)=m(m+n)+n.
\]
\end{proof}
We already know that
\[
h_1(t,\theta)=\frac{1}{\sqrt{4\pi t}} \sum_{k \in \mathbb{Z}} e^{-\frac{(\theta-2k\pi)^2}{4t} }.
\]
A simple recursion thus leads to
\begin{equation}\label{heat_kernel_odd}
h_{2n+1}= e^{n^2t} \left( -\frac{1}{2\pi \sin \theta} \frac{\partial}{\partial \theta} \right)^n V
\end{equation}
where 
\begin{equation*}
V(t,\theta)=\frac{1}{\sqrt{4\pi t}} \sum_{k \in \mathbb{Z}} e^{-\frac{(\theta-2k\pi)^2}{4t} }.    
\end{equation*}
This leads, for instance to the following asymptotics
\[
h_{2n+1}(t,\theta) \sim_{t \to 0} \frac{1}{(4\pi t)^{n+\frac{1}{2}}} \left( \frac{\theta}{\sin \theta} \right)^n e^{-\frac{\theta^2}{4t} }, \quad \theta \in [0, \pi).
\]

The analysis of the heat kernel asymptotics on even dimensional spheres is slightly more involved. In order to find a formula like (\ref{heat_kernel_odd}) by using Proposition \ref{Recursion heat kernel} we first need to understand the heat kernel on the two-dimensional sphere.  It turns out that the  heat kernel on the two-dimensional sphere may be expressed from  the heat kernel on the three-dimensional sphere through a simple integral transform.

To express this integral transform, it is more convenient to introduce
\[
q_n(t,\tau)=h_n(t,{\arccos} \tau).
\]
It is seen that  $q_n$ is  the Neumann heat kernel issued from $0$ of the diffusion operator
\[
(1-\tau^2)\frac{d^2}{d\tau^2}-n\frac{d}{d\tau}.
\]
Keep in mind that according to the definition of $h_n$, $q_n(t,\cos d(x,y))$ is the heat kernel on the $n$-dimensional sphere.

Thanks to Proposition \ref{Recursion heat kernel}, the following recursion formula holds
\[
q_{n+2}(t,\tau)=\frac{e^{nt}}{2 \pi} \frac{\partial q_n}{\partial \tau} (t,\tau).
\]

We have the following intertwining that may be used to study the heat kernel on even dimensional spheres.

\begin{proposition}\label{heat kernel entrelacement}
For $\theta \in [0,\pi]$,
\[
q_{2} (t ,\cos 2\theta)=\frac{1}{\pi} \int_0^\pi q_3 \left(\frac{t}{4}, \cos \theta \cos \phi \right) d\phi.
\]
\end{proposition}

\begin{proof}
If $f:[-1,1] \to \mathbb{R}$ is a smooth function, we define $\mathcal{T}f$ by
\[
\mathcal{T}f(\cos 2\theta) =\frac{1}{\pi} \int_0^\pi f(\cos \theta \cos \phi) d\phi, \quad \theta \in [0,\pi].
\]
It is easily checked by direct computation that
\[
\mathcal{T}\mathcal{L}_{3} f=4 \mathcal{L}_{2}\mathcal{T} f,
\]
where
\[
\mathcal{L}_{n}=(1-\tau^2)\frac{d^2}{d\tau^2}-n\frac{d}{d\tau}.
\]
This yields the Mehler-Dirichlet formula for the Gegenbauer polynomials:
\[
C^{1/2}_n(\cos 2 \theta)= \frac{1}{\pi} \int_0^\pi C^1_{2n+1} (\cos \theta \cos \psi) d\psi.
\]
But, thanks to the formula (\ref{radial_kernel_sphere}),
\[
q_{2}(t, \tau)=\frac{\sqrt{\pi}}{ 4\pi^{\frac{3}{2}} }\sum_{m \ge 0} \left(2m+1 \right)e^{-m(m+1)t } C_m^{\frac{1}{2}}( \tau),
\]
 and
 \[
q_{3}(t, \tau)=\frac{1}{ 4\pi^{2} }\sum_{m \ge 0} \left(2m+2 \right)e^{-m(m+2)t } C_m^{1}(\tau).
\]
The result then easily easily follows.
\end{proof}

\begin{remark}
By using Proposition \ref{Recursion heat kernel}, Proposition \ref{heat kernel entrelacement} and the Laplace method, one can prove that
\[
h_{2n}(t,\theta) \sim_{t \to 0} \frac{1}{(4\pi t)^{n}} \left( \frac{\theta}{\sin \theta} \right)^{n-\frac{1}{2}} e^{-\frac{\theta^2}{4t} }, \quad \theta \in [0, \pi).
\]
\end{remark}

\section{Hyperbolic Jacobi diffusions and heat kernels on hyperbolic spaces}\label{section hyp Jacobi}
The hyperbolic Jacobi operator is the diffusion operator defined by: 
\begin{equation*}
\mathscr{L}^{\alpha, \beta} = \frac{1}{2} \frac{\partial^2}{\partial r^2}+\left(\left(\alpha+\frac{1}{2}\right)\coth r  + \left(\beta+\frac{1}{2}\right) \tanh r\right)\frac{\partial}{\partial r}, \quad \alpha >-1,\ \beta \in \mathbb{R},\ r \geq 0.    
\end{equation*}
For special values of $\alpha$ and $\beta$, it appears as the radial part of the Laplace-Beltrami operators on  symmetric spaces of noncompact type as well as of Damek-Ricci (in general non symmetric) spaces. We refer the reader to \cite{Koo} for harmonic analysis on Riemannian non compact symmetric spaces, to \cite{Hec-Sch} for non necessarily Riemannian symmetric spaces and with a different flavour, and to \cite{ADY} for Damek-Ricci spaces. 

When $\beta = -1/2$, $\mathscr{L}^{\alpha, -1/2}$ is referred to as the hyperbolic Bessel operator. Since when $\alpha = n/2, n \in \mathbb{N}\setminus \{0\}$, it reduces  to the radial Laplacian on real hyperbolic spaces. The heat kernel of the latter admits several representations among which we find Millson's induction formulas distinguishing odd and even dimensions: let 
\begin{equation}\label{eq-kernel-hyper-jacobi}
H_n = \mathbb{R}^{n-1} \times \mathbb{R}_{\ge 0},\quad n \geq 2,
\end{equation}
be the the half-space model of the real hyperbolic space in $n$ dimensions, then with appropriate normalization  (see Appendix 3) of the volume density measure: 
\begin{itemize}
\item  \hspace{.1in} If $n = 2m+1, m \geq 0,$, then 
\begin{equation*}
q_{t,2m+1}(x,y) = q_{t,2m+1}(r) = \frac{e^{-m^2t/2}}{(2\pi)^m\sqrt{2\pi t}} \left(-\frac{1}{\sinh(r)}\frac{d}{dr}\right)^me^{-r^2/(2t)},
\end{equation*}
where $r = d_n(x,y)$ is the hyperbolic distance between $x, y \in H_n$. 
\item  \hspace{.1in} If $n = 2m+2, m \geq 0,$ is even, then  
\begin{align*}
 q_{t,2m+2}(r)& = \frac{e^{-(2m+1)^2t/8}\sqrt{2}}{(2\pi t)^{3/2} (2\pi)^m} \left(-\frac{1}{\sinh(r)}\frac{d}{dr}\right)^m \int_r^{\infty} \frac{\theta e^{-\theta^2/(2t)}}{(\cosh(\theta) - \cosh(r))^{1/2}} d\theta \\
 &=\frac{e^{-(2m+1)^2t/8}\sqrt{2}}{(2\pi t)^{1/2}(2\pi)^{m+1}  }  \int_r^{\infty} \frac{1 }{(\cosh(\theta) - \cosh(r))^{1/2}} \left(-\frac{d}{d\theta}\right)\left(-\frac{1}{\sinh(\theta)}\frac{d}{d\theta}\right)^m e^{-\theta^2/(2t)} d\theta .
\end{align*}
\end{itemize}
The Millson recursion formula then writes
\begin{equation}\label{Milson recursion hyp}
q_{t,n+2}(r)=-\frac{1}{2\pi}\frac{e^{-\frac{nt}{2}}}{\sinh (r)}\frac{d}{d r}q_{t,n}(r).
\end{equation}
It is the analogue for the hyperbolic space of the recursion formula in Proposition \ref{Recursion heat kernel}.
Another less-known formula for the hyperbolic heat semi-group was derived in \cite{Gruet} and does not distinguish odd and even dimensions. It is given by the following oscillatory integral:
\begin{equation}\label{GruetFor}
q_{t,n}(r) = \frac{e^{-(n-1)^2t/8}}{\pi(2\pi)^{n/2}\sqrt{t}}\Gamma\left(\frac{n+1}{2}\right)\int_0^{\infty}\frac{e^{(\pi^2-\rho^2)/(2t)}\sinh(\rho)\sin(\pi\rho/t)}{[\cosh(\rho) + \cosh(r)]^{(n+1)/2}} d\rho,
\end{equation}

For arbitrary $\alpha > -1/2$, the heat kernel of $\mathscr{L}^{\alpha, -1/2}$ admits an integral representation stemming from the harmonic analysis related to Jacobi functions and another one extending Gruet's formula alluded to above to real values of $\alpha$. The reader may find in \cite{Dem-Heat} a good account on these formulas and on their different proofs, as well as a new formula for the heat kernels of Damek-Ricci spaces including those corresponding to complex and quaternionic hyperbolic spaces.  We also refer to \cite{MR2018351} for formulas and estimates for the heat kernel on symmetric spaces of non-compact type which are based on non-commutative Fourier transform techniques.

In particular, on the  complex hyperbolic space $\mathbb{C}H^n$ with $n\ge 2$, the heat kernel is the heat kernel of $\mathscr{L}^{n-1, 0 }$  and is given by
\begin{align}\label{heat kernel CHn}
\frac{2 e^{-n^2t/2}}{(\pi t)^{1/2}(2\pi)^{n}  }  \int_r^{\infty} \frac{1 }{(\cosh(2\theta) - \cosh(2r))^{1/2}} \left(-\frac{d}{d\theta}\right)\left(-\frac{1}{\sinh(\theta)}\frac{d}{d\theta}\right)^{n-1} e^{-\theta^2/(2t)} d\theta.
\end{align}
From \cite{Dem-Heat}, an alternative formula for that kernel is 
\[
\frac{2e^{-n^2t/2}}{(2\pi t)^{1/2} (2\pi)^n}  \int_r^{\infty}\frac{\sinh(\theta)}{(\cosh^2(\theta) - \cosh^2(r))^{1/2}} \left(-\frac{1}{\sinh(\theta)}\frac{d}{d\theta}\right)^ne^{-\theta^2/(2t)} d\theta.
\]
 On the  quaternionic hyperbolic space $\mathbb{H}H^n$ with $n\ge 2$, the heat kernel is the heat kernel of $\mathscr{L}^{2n-1, 1 }$  and is given by
\begin{align}\label{heat kernel HHn}
\frac{4 e^{-(2n+1)^2t/2}}{(\pi t)^{1/2}(2\pi)^{2n}  }  \int_r^{\infty} & \frac{1 }{(\cosh(2\theta) - \cosh(2r))^{1/2}} \notag \\
 &\left(-\frac{d}{d\theta}\right)\left(-\frac{1}{\sinh(2\theta)}\frac{d}{d\theta}\right)\left(-\frac{1}{\sinh(\theta)}\frac{d}{d\theta}\right)^{2(n-1)} e^{-\theta^2/(2t)} d\theta.
\end{align}
From \cite{Dem-Heat}, an alternative formula for that kernel is 
\[
C_n \frac{e^{-(2n+1)^2t/2}}{\cosh(r)} \left( \int_0^{\infty} y^{2n-1}u(t,y) K_1(y\cosh(r)) dy \right)
\]
where $C_n$ is a normalizing constant, $K_1$ is the modified Bessel function of the second kind, and $u(t,y)$ is the density of the Hartman-Watson distribution at time $t$.

\section{Matrix Jacobi diffusions}
The real symmetric Jacobi diffusion was introduced by Y. Doumerc in his Ph.D. thesis as an extension of the Jacobi diffusion valued in $[0,1]$ (see \cite{Dou}). 
Its construction is inspired from the connection between the matrix Beta distribution and the radial part of the left-upper corner of a Haar unitary matrix due to B. Collins (see \cite{Collins}). The real symmetric Jacobi diffusion was then built out of a Brownian motion in the orthogonal group. 

The complex Hermitian analogue of this diffusion was introduced and studied in a series of paper by the secondly-named author, replacing the Brownian motion in the orthogonal group by its complex analogue. It may be written as 
\begin{equation*}
J(t) \oplus 0:= PY(t)QY(t)^*P, \quad t \geq 0,
\end{equation*}
where 
\begin{itemize}
    \item  \hspace{.1in} $(Y(t))_{t \geq 0}$ is a unitary Brownian motion. 
    \item  \hspace{.1in} $P,Q$ are diagonal projections. 
    \end{itemize}
The eigenvalues processes of the real symmetric and the complex Hermitian Jacobi diffusions satisfy the following differential system:  
\begin{multline}\label{JE}
d\lambda_i(t) = 2\sqrt{(\lambda_i(t)(1-\lambda_i(t))}d\nu_i(t) + \beta [(p - (p+q)\lambda_i(t))]dt + \\
 \beta \sum_{j \neq i}\frac{\lambda_i(t)(1-\lambda_j(t)) + \lambda_j(t)(1-\lambda_i(t))}{\lambda_i(t) - \lambda_j(t)} dt
\end{multline}
where $\beta = 1,2,$ respectively and $(\nu_i)_{i=1}^m$ are real independent Brownian motions. Moreover, if $m \geq 1$ is the matrix size and if $p\wedge q > (m-1) + 1/\beta$, then \eqref{JE} admits a unique strong solution for all time $t > 0$ and for any starting eigenvalue vector $(\lambda_j(0))_{j=1}^m$. We refer the reader to \cite{Dem-beta} for a detailed study of \eqref{JE} with arbitrary $\beta > 0$. In particular, the semi-group density of the latter admits the following expansion (see \cite{Dem-beta} and references therein for more details): 
\begin{equation}
\label{densite}
p_t^{r,s,\beta}(\theta, \lambda) := 
\sum_{\tau=(\tau_1 \geq \cdots \geq \tau_m \geq 0)}
e^{-2r_{\tau}^{\beta}t}\jam(\theta)\jam(\lambda) W_m^{r,s,\beta}(\lambda)
\end{equation}
where 
\begin{itemize}
\item  \hspace{.1in} $p \wedge q > (m-1) +1/\beta, $. 
\item  \hspace{.1in} $\beta (p - (m-1)) := 2(r+1),\,\beta(q-(m-1)) := 2(s+1)$.
\item  \hspace{.1in} The spectral value $r_{\tau}^{\beta}$ is given by:
\begin{equation*}
r_{\tau}^{\beta} := \sum_{i=1}^m\tau_i(\tau_i +r+s+1 + \beta(m-i)).
\end{equation*}
\item  \hspace{.1in} $\jam$ is the so-called multivariate Jacobi polynomial, also referred to as the Heckman-Opdam polynomial of type $BC$. For special values of the parameters $(r,s,\beta)$, it is a spherical function of the compact Grassmannians over division algebras. In particular, $P_{\tau}^{r,s,2}$ admits the following determinantal representation: 
\begin{equation*}
P_{\tau}^{r,s,2}(\lambda) = \frac{\det[P_{\tau_i+m-i}^{r,s}(\lambda_j)]_{i,j}}{V(\lambda)},
\end{equation*}
where $P_{\tau_i+m-i}^{r,s}$ is a real Jacobi polynomial and $V$ stands for the Vandermonde polynomial.
\item  \hspace{.1in} $W_m^{r,s,\beta}(\lambda)$ is the multivariate Beta distribution: 
\begin{equation*}
W_m^{r,s,\beta}(\lambda) = [V(\lambda)]^{\beta}\prod_{j=1}^m\lambda_j^r(1-\lambda_j)^s.
\end{equation*}
    
\end{itemize}

\chapter{Appendix 3: Radial part of the Laplacian in rank-one symmetric spaces}

If $\M$ is rank-one Riemannian symmetric space, there exists an operator $\Delta_r$ called the radial Laplacian such that 
\[
\Delta (f \circ d(o,\cdot))= (\Delta_r f )\circ d(o,\cdot)
\]
where $\Delta$ is the Laplace-Beltrami operator and $d(o,\cdot)$ the distance from an arbitrary fixed point $o$. An invariant and symmetric measure for $\Delta_r$ is called a radial measure.

\begin{table}[H]
\centering
\scalebox{0.8}{
\begin{tabular}{|p{1.5cm}||p{8cm}|c|>{\centering\arraybackslash}p{1.8cm}|c|  }
  \hline
 $\M$ &   Radial Laplacian   &   Radial measure      \\
  \hline
 \hline
$\mathbb{R}^m$ & $ \frac{\partial^2}{\partial r^2} + \frac{m-1}{r} \frac{\partial}{\partial r} $ & $ 2 \frac{\pi^{m/2}}{\Gamma (m/2)} \, r^{m-1} dr$  \\
$\mathbb{S}^m$ & $ \frac{\partial^2}{\partial r^2} + (m-1)\cot r\frac{\partial}{\partial r} $ & $ 2 \frac{\pi^{m/2}}{\Gamma (m/2)} \, (\sin r)^{m-1} dr$  \\
$H^m$ & $ \frac{\partial^2}{\partial r^2} + (m-1)\coth r\frac{\partial}{\partial r} $ & $ 2 \frac{\pi^{m/2}}{\Gamma (m/2)} \, (\sinh r)^{m-1} dr$  \\
 $\mathbb{C}P^m$ &$\frac{\partial^2}{\partial r^2}+((2m-2)\cot r +2 \cot 2 r)\frac{\partial}{\partial r} $ & $ \frac{\pi^m}{(m-1)!} \, (\sin r)^{2m-2} \sin (2r) \,dr$    \\
$\mathbb{C}H^m$  & $ \frac{\partial^2}{\partial r^2}+((2m-2)\coth r+2  \coth 2r)\frac{\partial}{\partial r}$ & $ \frac{\pi^m}{(m-1)!} \, (\sinh r)^{2m-2} \sinh (2r) \,dr$   \\
$\mathbb{H}P^m$ &$\frac{\partial^2}{\partial r^2}+((4m-4)\cot r+6\cot2 r)\frac{\partial}{\partial r}$ & $ \frac{\pi^{2m}}{4(2m-1)!} \, (\sin r)^{4m-4} \sin (2r)^3 \,dr$    \\
$\mathbb{H}H^m$  & $ \frac{\partial^2}{\partial r^2}+((4m-4)\coth r+6  \coth 2r)\frac{\partial}{\partial r}$ & $ \frac{\pi^{2m}}{4(2m-1)!} \, (\sinh r)^{4m-4} \sinh (2r)^3 \,dr$   \\
$\mathbb{O}P^1$ &$\frac{{\partial}^2}{\partial {r}^2}+(7\cot r-7\tan r)\frac{{\partial}}{\partial {r}}$ & $\frac{\pi^4}{3\cdot 2^7}  \sin(2r)^7 dr$\\
$\mathbb{O}H^1$ &$\frac{{\partial}^2}{\partial {r}^2}+(7\coth r+7\tanh r)\frac{{\partial}}{\partial {r}}$ & $\frac{\pi^4}{3\cdot 2^7}  \sinh(2r)^7 dr$\\
  \hline
\end{tabular}}
\caption{Radial Laplacians and radial measures in  rank-one symmetric spaces.}
\label{Table 1}
\end{table}

In the formulas, the normalization constant for the radial measure $\mu_r$ is such that 
\[
\int_\M f( d(o,x))  d\mu = \int_0^D f(r) d\mu_r
\]
where $D$ is the diameter of $\M$, $\mu$ the Riemannian volume measure. In particular, in the case where $\M$ is compact,
\[
\mathrm{Vol} (\M)= \mu_r ([0,D])
\]
where $\mathrm{Vol}(\M)$ is the Riemannian volume of $\M$.

Standard references for the Riemannian geometry of symmetric spaces and the study of their Laplacians are the monographs by S. Helgason \cite{Helgason,MR754767}.

\backmatter
\bibliographystyle{amsplain}
\bibliography{reference}

\end{document}